\documentclass[10pt,a4paper]{article}

\usepackage{amsmath}
\usepackage{amsfonts}
\usepackage{amssymb}
\usepackage{amsmath}
\usepackage{amsthm}
\usepackage{latexsym}
\usepackage{graphicx,cite}
\DeclareGraphicsExtensions{.jpg}
\usepackage{a4wide}
\usepackage{float}
\usepackage{mathrsfs}
\usepackage[titletoc,title]{appendix}
\usepackage{nicefrac}
\usepackage[utf8]{inputenc}
\usepackage[T1]{fontenc}
\usepackage[nottoc,numbib]{tocbibind}

\usepackage{palatino}

\usepackage[dvipsnames]{xcolor}
\usepackage[colorlinks, citecolor=violet,linkcolor=red]{hyperref}

\usepackage{mathtools}
\mathtoolsset{showonlyrefs}

\newtheorem{thm}{Theorem}[section]
\newtheorem{lem}[thm]{Lemma}
\newtheorem{prop}[thm]{Proposition}
\newtheorem{cor}[thm]{Corollary}

\theoremstyle{definition}
\newtheorem{defn}[thm]{Definition}
\newtheorem{conge}[thm]{Conjecture}

\newtheorem{oprob}[thm]{Open Problem}
\newtheorem{ackn}{Acknowledgement\!}

\theoremstyle{remark}
\newtheorem{rem}[thm]{Remark}

\numberwithin{equation}{section}
\numberwithin{figure}{section}

\def\dert{\partial_t}
\def\ders{\partial_s}
\def\comp{\circ}

\def\eps{\varepsilon}

\def\limup{\operatornamewithlimits{\overline{\lim}}}
\def\loc{_{\operatorname{\mathrm{loc}}}}

\def\HH{\mathcal H}

\def\H{{{\mathcal H}}}

\def\R{\mathbb R}
\def\SS{\mathbb S}
\def\R{{{\mathbb R}}}
\def\SS{{{\mathbb S}}}

\def\NN{\mathbb N}
\def\TTT{{\mathbb{T}}}
\def\TT{\mathbb T}
\def\TTT{\mathbb T}

\def\ssss{\mathfrak s}
\def\tt{\mathfrak t}

\def\pol{\mathfrak p}
\def\qol{\mathfrak q}
\def\sss{\xi}
\def\res {
 \begin{picture}(9,8)
 \put (1,0){\line(0,1){8}}
 \put (1,0){\line(1,0){5}}
 \end{picture}}

\newcommand{\Om} {\Omega}

\newcommand{\intbar}{\etaathop{\int\etaakebox(-13.5,0){\rule[4pt]{.7em}{0.3pt}}
\kern-6pt}\nolimits}

\newcommand{\be}{\begin{equation}}
\newcommand{\ee}{\end{equation}}
\newcommand{\bea}{\begin{equation*}}
\newcommand{\eea}{\end{equation*}}
\newcommand{\op}{\langle}
\newcommand{\cl}{\rangle}

\newcommand{\epsi}{\mathbb{E}}

\renewcommand{\l }{\lambda }
\renewcommand{\t }{\tau }

\newcommand\bff{\operatorname{Bf}}
\newcommand\ff{\operatorname{Ff}}
\newcommand{\lf}{\operatorname{Lf}}
\newcommand{\rf}{\operatorname{Rf}}

\newcommand{\vertiii}[1]{{\left\vert\kern-0.25ex\left\vert\kern-0.25ex\left\vert #1 
		\right\vert\kern-0.25ex\right\vert\kern-0.25ex\right\vert}}
\newcommand{\vertiiii}[1]{{\left\vert\kern-0.25ex\left\vert\kern-0.25ex\left\vert #1 
		\right\vert\kern-0.25ex\right\vert\kern-0.25ex\right\vert}}

\usepackage{tikz,fp,ifthen,fullpage}
\usetikzlibrary{backgrounds}
\usetikzlibrary{decorations.pathmorphing,backgrounds,fit,calc,through}
\usetikzlibrary{arrows}
\usetikzlibrary{shapes,decorations,shadows}
\usetikzlibrary{fadings}
\usetikzlibrary{patterns}
\usetikzlibrary{mindmap}
\usetikzlibrary{decorations.text}
\usetikzlibrary{decorations.shapes}
\usetikzlibrary{decorations.pathreplacing}

\definecolor{darkgreen}{rgb}{0,0.5,0}
\usepackage{varwidth}

\title{Evolution of networks with multiple junctions}
\author{Carlo Mantegazza \footnote{Dipartimento di Matematica e Applicazioni, Universit\`a di Napoli Federico II and Scuola Superiore Meridionale, Napoli, Italy} \and 
Matteo Novaga \footnote{Dipartimento di Matematica, Universit\`a di Pisa, Italy} 
\and Alessandra Pluda\footnotemark[2] \and Felix Schulze\footnote{Department of Mathematics, University of Warwick, United Kingdom}}

\begin{document}

\maketitle

\begin{abstract}
\noindent We consider the motion by curvature of a network of curves in the plane and we
discuss existence, uniqueness, singularity formation, and asymptotic behavior of the flow. 
\end{abstract}

\setcounter{tocdepth}{2}
\tableofcontents

\section{Introduction}

In this work we give an overview of the state--of--the--art of the motion by curvature of planar networks of curves, collecting known results and showing several new ones.
\begin{figure}[H]
\begin{center}\begin{tikzpicture}
\draw[color=black,thick,scale=1,domain=-3.141: 3.141,
smooth,variable=\t,shift={(-1,0)},rotate=0]plot({4.5*sin(\t r)},
{2*cos(\t r)});
\draw
(0.3,0)to[out=-30, in=160, looseness=1](0.6,-0.4)
(1,-0.5)to[out=120, in=45, looseness=1](0.6,-0.4)
(0.6,-0.4)to[out=-60, in=80, looseness=1](-1.3,-1.4)
(0.3,0)to[out=60, in=-160, looseness=1](1.3,0.5)
(1.3,0.5)to[out=60, in=-100, looseness=1](2.3,1.38)
(1.3,0.5)to[out=-100, in=80, looseness=1](2,0)
(1,-0.5)to[out=10, in=-120, looseness=1](2,0)
(1,-0.5)to[out=-60, in=120, looseness=1](2.3,-1.37)
(-3,0)to[out=100, in=-90, looseness=1](-4,1.5)
(2,0)to[out=-40, in=-170, looseness=1](3.45,0.2)
(0.2,1.3)to[out=40,in=-50, looseness=1](1.3,1.7)
(0.2,1.3)to[out=160, in=-60, looseness=1](-3,1.8)
(-0.5,1)to[out= 40,in=-80, looseness=1] (0.2,1.3)
(-4.5,0)to[out= 50,in=180, looseness=1] (-3,0)
to[out= 60,in=160, looseness=1] (-0.5,1)
(-0.5,1)to[out= -80,in=180, looseness=1] (0.3,0)
(-3,0)to[out=50, in=120, looseness=1](-1,-1.99)
(-3,0)to[out=-70, in=60, looseness=1](-3.7,-1.6)
(-4.5,0)to[out= 170,in=-20, looseness=1] (-5.2,0.7)
(-4.5,0)to[out= -70,in=30, looseness=1] (-5,-0.9);
\end{tikzpicture}
\end{center}
\begin{caption}{A planar network of curves in a convex domain.}
\end{caption}
\end{figure}
The problem, proposed by Mullins~\cite{hermul} and discussed first
in~\cite{hermul,brakke,bronsard,gurtin2,kinderliu}, attracted the interest of many authors in 
recent years~\cite{EsOt:14,Ilnevsch,MMN13,haettenschweiler,balhausman,BeNo,chenguo,
mannovplu,mannovtor,mannovtor2,mazsae,pluda,
schn-schu,schnurerlens,tonwic,kimton,garkoh,saez1,saez2}.
One strong motivation to study this flow is the analysis of models of two--dimensional 
multiphase systems, where the problem of the dynamics of the
interfaces between different phases arises naturally.
As an example, the model where the energy of a configuration is
simply given by the total length of the interfaces has proven useful to describe the growth of grain boundaries in a polycrystalline material (see~\cite{hermul,gurtin2,kinderliu} and {\em http:/\!/mimp.materials.cmu.edu}).\\
A second motivation is more theoretical: the evolution by curvature of such a network of curves is the simplest example of mean curvature flow of a set which is {\em essentially} singular.
To consider such flow not only for smooth submanifolds but also for non--regular sets,
several generalized (weak) definitions of the flow have been introduced 
in the literature~\cite{altawa,brakke,degio4,es,ilman1,soner1}. Anyway, while the smooth case was largely studied and understood (even if still not completely), the evolution of generalized submanifolds, possibly singular (for instance {\em varifolds}), has not been analyzed in great detail.\\
In his seminal paper, K.~Brakke~\cite{brakke} proved the existence of a global (very) weak solution, in a {\em geometric measure theory} context, called ``Brakke flow''.
Recently, the work of Brakke has been improved by L.~Kim and Y.~Tonegawa~\cite{kimton} (see also~\cite{SalvoToni}) in the case of the evolution of grain boundaries in $\mathbb{R}^n$ 
(which reduces to the evolution of networks when $n=2$).
They proved a global existence theorem and also showed that there exists a finite family of open sets moving continuously with respect to the Lebesgue measure, whose boundaries coincide with
the space--time support of the flow (for further results, see also the papers by K.~Kasai and Y.~Tonegawa~\cite{kaston} and Y.~Tonegawa and N.~Wickramasekera~\cite{tonwic}). Finally, in~\cite{KimTonegawa2}, Kim and Tonegawa also proved a regularity result for $1$--dimensional Brakke flows, showing that for almost all times, the evolving network consists of a finite number of embedded curves of class $W^{2,2}$, meeting at junctions with angles of $60$ or $120$ degrees or with a common tangent.\\
For another global existence result in any codimension and with special regularity properties, obtained adapting the elliptic regularization scheme of T.~Ilmanen~\cite{Il:93g,ilman1}, we refer to the work of the last author and B.~White~\cite{schulzewhite}.
Despite these recent improvements, Brakke's definition is anyway apparently too weak (possibly too general) if one is interested in a detailed description of the flow.\\
A completely different definition of evolution is instead based on the so-called {\em minimizing movements}: an implicit time--discrete variational scheme introduced in~\cite{altawa,luckstur} 
(see also~\cite{bellettinikho,caraballo1,degio7}).
In this context, another discretization scheme was developed and studied by 
S.~Esedoglu and F.~Otto~\cite{EsOt:14}, T.~Laux and F.~Otto~\cite{lauxotto,lauxotto2} (we motion also the more recent development~\cite{FiHeLaSi}).\\
Finally, we mention the ``level set'' approach to motion by curvature by L.~C.~Evans and J.~Spruck~\cite{es} or, alternatively, Y.~G.~Chen, Y.~Giga, and S.~Goto~\cite{cgg}, 
unfortunately not suitable for the motion of networks since if at least a multi--point is present then an interior region immediately develops (the so-called ``fattening'' phenomenon).

Even if all these approaches provide a globally defined evolution,
the possible conclusions on the structure and regularity of the moving networks are
actually quite weak.
To obtain a detailed description
of the evolution and of the singularity formation,
we tried to work in the smooth setting as much as possible.
The definition of the flow is then the first problem one has to face, due
to the contrast between such desire and the intrinsic singular geometric nature of a network. 
Consider for instance the network described by two curves
crossing each other, forming a $4$--point. There are actually several
possible candidates for the flow:
one cannot easily decide how the angles must behave, moreover, it could also be
allowed the four concurrent curves to separate into two pairs of
curves moving independently of each other and/or we could take into account
the possible ``birth'' of new multi--points from such a single one
(all these choices are possible with Brakke's definition). Actually, one would like a good/robust definition of curvature flow giving uniqueness of the motion (at least for ``generic'' initial networks) and forcing the evolving network, by an ``instantaneous regularization'' effect, with the possible exception of some discrete set of times, to have only triple junctions with the three angles between the concurring curves of $120$ degrees. This last property (which was experimentally observed for the growth of grain boundaries) is usually called {\em Herring condition}. These expectations are sustained also by the variational nature of the problem since this evolution can be considered as the ``gradient flow'' in the ``space of networks'' of the {\em Length} functional, which is the sum of the lengths of all the
curves of the network (see~\cite{brakke}). It must anyway be said that such a space does not share a natural linear structure and such a ``gradient'' is not actually a well-defined ``velocity'' vector driving the motion at the multiple junctions, in general. However, it follows that every point of a network different from its multi--points must move with a velocity whose normal component is the curvature vector of the curve it belongs, in order to decrease the {\em Energy} of the network (that is, the total length here) ``most efficiently'' (see~\cite{brakke}).
From this ``energetic'' point of view, it is then natural to expect also that configurations with multi--points of order greater than three or $3$--points with angles different from $120$ degrees, being {\em unstable} for the length functional, should be present only in the initial network or that they should appear only at some discrete set of times, during the flow.
This property is suggested also by numerical simulations and physical experiments, see~\cite{hermul,bronsard,gurtin2,kinderliu} and the {\em grain growth} movies at~{\em http://facstaff.susqu.edu/brakke}. One may hope that some sort of ``parabolic regularization'' could play a role here: for instance, if a multi--point has only two concurrent curves, it can be easily shown 
(see~\cite{altgra,angen2,angen3,gray1}) that the two curves become instantaneously a single smooth curve moving by curvature.\\
We mention that actually, it is always possible to find a Brakke flow sharing such property at almost every time (see~\cite{brakke}), by the variational spirit of its definition which is the closest to the ``gradient flow'' point of view. However, as uniqueness does not hold in this class, there are also Brakke flows starting from the same initial network which keeps their multi--points, or loose the connectedness of the network: for instance, a $4$--point can ``open'' as in the right side of Figure~\ref{figstandard}, or separate in two no more concurring curves, or it can ``persists'' to be a $4$--point where the two ``crossing'' curves move independently. Anyway, as we said, Brakke's definition is too ``weak'' if one is interested in a detailed description of the flow.

By this discussion it is then natural, due to their expected relevance, to call {\em regular} the networks with only $3$--points and where the three concurrent curves form angles of $120$ degrees. Then, following the ``energetic'' and experimental motivations mentioned above, we simply {\em impose} such regularity condition in the definition of a {\em smooth} curvature flow, for every positive time (at the initial time it could fail). If the initial network is regular and smooth enough, we will see that this definition leads to an almost satisfactory (in a way ``classical'') short--time existence theorem of a flow by curvature. Trying instead to let evolve an initial non--regular network, various complications arise related to the presence of multi--points or of $3$--points not satisfying the Herring condition. Notice also that, even starting with an initial regular network, we cannot avoid to deal also with non--regular networks when we analyze the global behavior of the flow. Indeed, during the flow, some of the triple junctions could ``collide'' along a ``vanishing'' curve of the network, when the length of the latter goes to zero (hence, modifying the topological structure of the network). In this case one has to ``restart'' the evolution with a different set of curves, possibly describing a non--regular network, typically with multi--points of order higher than three (consider, for instance, two $3$--points collapsing along a single curve connecting them) or even with ``bad'' $3$--points 
with angles between the concurring curves, not all equal to $120$ degrees
(think of three $3$--points collapsing together with the ``triangular'' region delimited by three curves connecting them). A suitable short--time existence (hence, ``restarting'') result for this situation has been worked out in~\cite{Ilnevsch} by T.~Ilmanen, A.~Neves and the fourth author and in~\cite{LiMazPlSa} by J.~Lira, R.~Mazzeo, M.~Saez and the third author. In these papers, it is indeed shown that starting from any non--regular network (with a natural technical hypothesis), there exists a ``satisfactory'' flow of networks by curvature which is immediately regular and smooth, for every positive time. Section~\ref{smtm3} is devoted to this topic.

The existence problem of a curvature flow for a regular network with only one $3$--point and fixed end--points, called {\em triod} (see Definition~\ref{triod}), was first considered by L.~Bronsard and F.~Reitich in~\cite{bronsard}. To be precise, they consider as initial datum any regular $C^{2+2\alpha}$ triod satisfying some compatibility conditions at the triple junctions and show short--time existence and uniqueness in the parabolic class $C^{2+2\alpha,1+\alpha}$. In~\cite{kinderliu} D.~Kinderlehrer and C.~Liu proved the global existence and convergence of a smooth solution if the initial regular triod is sufficiently close to a minimal (Steiner) configuration. 

After introducing regular networks, their flow by curvature, and some basic properties (Sections~\ref{notation} and~\ref{basiccomp}), we extend, in Section~\ref{smtm}, the above well--posedness theorem to general regular networks (Theorem~\ref{2compexist0}). Moreover, we also show an analogous result in suitable Sobolev spaces (Theorem~\ref{wellposednessSobolev}).

In Section~\ref{kestimates} we generalize to any regular network the integral estimates proved in~\cite{mannovtor} for a triod, 
which are needed to prove Theorem~\ref{c2shorttime} and will be actually used throughout the whole paper. A consequence of such estimates is the fact that if the lengths of the curves are bounded away from zero, as $t$ goes to the maximal time $T$ of existence of the flow, the maximum of the modulus of the curvature must go to $+\infty$ (Corollary~\ref{kexplod} and Theorem~\ref{curvexplod-general}).

The uniqueness of the flow is quite delicate. Indeed, by Theorem~\ref{2compexist0}, we only have that, for initial regular networks of class $C^{2+2\alpha}$ having the sum of the curvatures of the three concurring curves at every triple junction equal to zero, there is uniqueness in the parabolic class $C^{2+2\alpha,1+\alpha}$. In Section~\ref{smtm2}, by combining Theorems~\ref{wellposednessSobolev} and~\ref{2compexist0} (the first mainly for the uniqueness, the second for the existence) we then show a result of existence/geometric uniqueness for short time of the flow of an initial network of class $C^2$ (Theorem~\ref{c2shorttime}), in a subclass of the curvature flows which are simply $C^2$ in space and $C^1$ in time. In the same section, we will also see that the classical property of parabolic equations of instantaneous regularization of solutions for positive times also holds for the motion by curvature of networks, in a suitable sense.

The rest of the paper is devoted to the long-time behavior of the flow. For the sake of simplicity, in the following overview, we will restrict ourselves only to the behavior in the interior of a convex domain of a network flowing by curvature with fixed end--points on the boundary of such set, while in the whole paper also the behavior at the boundary (hence, at the end--points of the network) is analyzed in the same detail.

In Section~\ref{monotonsec} we recall Huisken's monotonicity formula for mean curvature flow which holds also for the evolution of a network and we introduce the rescaling procedures to get blow--up limit networks (discussed in Section~\ref{geosec}) at the maximal time of smooth existence. Then, to ``describe'' the singularities of the flow one needs to classify such possible blow--up limits. In some cases, arguing by contradiction with geometric arguments, this ``description'' can be used to exclude at all the formation of singularities. Key references for this method in the situation of a single smooth closed curve are~\cite{altsch,hamilton4,huisk3,huisk2}.
The most relevant difference in dealing with networks is the
difficulty in using the maximum principle, 
which in the case of closed curves is the main tool for getting pointwise 
estimates on the geometric quantities during the flow. 
For this reason, some crucial estimates which are straightforward in such case 
are here much more difficult to obtain and we had to resort to the integral estimates of Section~\ref{kestimates} (see also Section~\ref{van}), which are similar to the ones in~\cite{altsch,angen2,angen3,huisk1}, but require some extra work to deal with the triple junctions.

One can reasonably expect that an embedded regular network 
does not develop singularities during the flow if its ``topological structure'' does not change 
(for instance, in the case of a ``collision'' of two or more $3$--points). 
Our analysis in Sections~\ref{geosec},~\ref{locreg} and~\ref{behavsing} will show that if no ``multiplicities'' larger than one occur in the blow--up limit networks, this expectation is indeed true.
Under the assumption that the lengths of the curves are bounded away from zero 
the only possible blow--up limits (with multiplicity one by hypothesis)
are either a straight line, a halfline, or a flat unbounded regular triod (called ``standard triod'') composed of three halflines through the origin of $\mathbb{R}^2$ forming angles of $120$ degrees (see Proposition~\ref{resclimit} and Section~\ref{behavsing}). 
Then, a local regularity theorem for the flow (shown in~\cite{Ilnevsch}) together with such 
classification excludes the presence of singularities. 
This result, which is in the spirit of White's local regularity theorem for mean curvature flow in~\cite{white1}, is presented in detail in Section~\ref{locreg}.

Thus, again in Section~\ref{behavsing}, we try to understand what happens at the maximal time, knowing that some lengths of the curves composing the network cannot be uniformly bounded away from zero, hence at least two $3$--points get closer and closer.\\
First of all, we prove that under the hypothesis of multiplicity one of the blow--up limits,
if more than two triple junctions go to collide, then necessarily an entire
region (the interior of a ``loop'' of the network) vanishes, which implies that the curvature is necessarily unbounded getting close to the singular time. Hence, if the curvature stays bounded it must happen that (locally) we are in the case of two triple junctions (only) going to collide along a vanishing curve, forming a $4$--point in the limit. Vice versa, we are then able to show that in such a situation the curvature remains bounded. As a consequence, we conclude that the curvature is uniformly bounded along the flow if and only if no region is collapsing and that in such case only local vanishing of single curves can happen, with a formation of a $4$--point in the limit. This is clearly particularly relevant if the evolving network is a {\em tree}, that is, regions are not present at all. More in detail, we first show that in such case, as $t$ goes to the maximal time $T$, the networks $\SS_t$ converge in $C^1$--norm (up to reparametrization) to a unique limit set $\SS_T$ which is a {\em degenerate} (collapsed) regular network 
(see Definition~\ref{degnet}), that is, a smooth network possibly with multi--points of order higher than three and some collapsed parts ``hidden'' in its vertices. Then, we show that $\SS_T$ can have only $3$--points with angles of $120$ degrees or $4$--points with angles of $120/60$ degrees, like in the left side of Figure~\ref{Pcollapse}.\\
In the other situation, when the curvature is not bounded and a region collapses (Section~\ref{van}), we are able to obtain a weaker conclusion. Assuming the uniqueness of the blow--up limit along any sequence of rescalings (which can be instead proved in the above case), we can show that, as $t\to T$, the network $\SS_t$ converges to some degenerate (see above) regular network, whose ``non--collapsed'' part $\SS_T$ is a $C^1$, possibly non--regular, network which is smooth outside its multi--points and whose curvature is of order $o(1/r)$, where $r$ is the distance from its non--regular multi--points.

In several steps of the previous analysis the assumption of multiplicity one of the blow--up limits
is fundamental, we actually conjecture (Conjecture~\ref{ooo9}) that it holds in general, but 
up to now we can prove it only in some special cases. Indeed, in Section~\ref{dsuL} we discuss a scaling invariant, 
geometric quantity associated with a network, first proposed in~\cite{hamilton3} (see also~\cite{huisk2}) and later extended in~\cite{mannovtor, BeNo, pluda}, consisting in a sort of ``embeddedness measure'' which is positive when no self--intersections are present. By a monotonicity argument, we show that this quantity is uniformly positively bounded below along the flow, under the assumption that the number of $3$--points of the network is at most two. 
As a consequence, in such case every possible $C^1\loc$--limit of rescalings of the networks of the flow is an embedded network with multiplicity one. 
We underline that it is not clear to us how to obtain a similar conclusion for a general network with several triple junctions, since the analogous quantity, if there are more than two $3$--points, does not satisfy a monotonicity property. 

In Section~\ref{smtm3} we state a short--time existence result for possibly non--regular initial networks (that is, with multi--points 
of order greater than $3$ and/or non--regular $3$--points), giving a flow that is immediately regular and smooth for every positive time. This result, which clearly also provides a ``restarting theorem'', was worked out independently in~\cite{Ilnevsch} by T.~Ilmanen, A.~Neves and the fourth author (Theorem~\ref{evolnonreg}) and in~\cite{LiMazPlSa} by J.~Lira, R.~Mazzeo, M.~Saez and the third author (Theorem~\ref{evononreg2}), here we only give an outline of the arguments in the proofs (which are quite technical). The idea in Theorem~\ref{evolnonreg} is to locally desingularize the multi--points and the non--regular $3$--points via regular self--similarly expanding solutions. The argument hinges on a new monotonicity formula, which shows that such expanding solutions are dynamically stable, using the fact that the evolution of curves and networks in the plane are special cases of the Lagrangian mean curvature flow (these ideas have already been exploited by A.~Neves in the papers~\cite{neves1,neves2,nevestian}). Theorem~\ref{evononreg2} relies instead on blow--up arguments from geometric micro--local analysis. In this case, the same regular self--similarly expanding solutions naturally arise from the underlying geometric structure of the problem.

In Section~\ref{restart} it is explained how to combine Theorem~\ref{evolnonreg} with the previous analysis of the singularities in order to continue the flow after a singular time. Then, we analyze the preserved geometric quantities and the possible changes in the topology of a network in passing through a singularity. This is applied in Section~\ref{llong} to study the long-time behavior of the flow, indeed, the restarting procedure allows us to define an ``extended'' curvature flow with singularities at an increasing sequence of times. An important open question is whether the maximal time interval of existence of such flow is finite or not, where the main problem is the possible ``accumulation'' of the singular times (if they are not finite, which actually we do not know). We mention that in the special case of symmetric networks with only two triple junctions, it can be shown that the set of singular times is necessarily finite, see~\cite{NoSc23}. Clearly, if such ``extended'' flow can be defined for every time (as the Brakke flow obtained by L.~Kim and Y.~Tonegawa in~\cite{kimton}), we ask ourselves if the network converges, as $t\to+\infty$, to a stationary network for the length functional (a Steiner network). In Proposition~\ref{prolong} we prove the convergence {\em up
to a subsequence} of the family of the evolving networks to a possibly degenerate one (some curves could disappear in the limit), as $t\to+\infty$. If we then assume that such limit network is not degenerate, with the help 
of {\em {\L}ojasiewicz--Simon gradient inequality}, we are actually able to prove the full convergence of the flow, in Theorem~\ref{prop:Convergence}. We finally conclude Section~\ref{llong} presenting a stability result: if a network is sufficiently close in $W^{2,2}$--norm to a regular network $\SS_\ast$ composed of straight segments only, its motion by curvature exists for all times and smoothly converges to a regular network still composed of straight segments and with the same length of $\SS_\ast$. 

Up to now, the study of the behavior of the flow at the first singularity (and immediately after) 
is essentially complete when the network has at most two triple junctions, see~\cite{mannovtor,MMN13, pluda, mannovplu}, holding in this very special case the above mentioned {\em multiplicity one conjecture}, as it is shown in Section~\ref{dsuL}. In Section~\ref{globsec} we will describe, up to the best of our knowledge, the global evolution of such ``simple'' networks, which are actually interesting since most of the relevant phenomena of the general case are already present. In particular, we will see that the evolution of a tree--like network with only one $3$--point and three fixed end--points (called {\em triod}) is smooth and asymptotically converges to a Steiner network, if the lengths of the three curves stay uniformly bounded away from zero. 

The last section of the paper is devoted to collecting and presenting the main open problems.
Moreover, by courtesy of T.~Ilmanen, we include an appendix with pictures and computations of several examples of regular shrinkers, due to him and J.~H\"attenschweiler.

We conclude this introduction by mentioning that there are several interesting variants and generalizations of the problem of the motion by curvature of networks whose study is only at the beginning. For instance, one can consider the anisotropic version of the flow, as in~\cite{kronovpoz,gronovpoz,BeChKh} and/or take into account the mismatch of
the orientation of the grain in the model~\cite{applicato4,applicato2,applicato3}.\\
The analogous problem in higher dimensions (and codimensions) is still widely open. 
Besides the papers~\cite{kimton, schulzewhite}, where a global weak solution in the Brakke sense is constructed, the short--time existence of a smooth and regular solution in three dimensions has been established in~\cite{degako} in some special cases and in~\cite[Section~7]{schulzewhite} for the motion of a network in $\mathbb{R}^n$ with only triple junctions.
In these cases, the analysis of singularities and the subsequent possible restarting procedure are still open problems.\\ 
We also mention the works~\cite{freire2,freire3} where a graph evolving by mean curvature and meeting a horizontal hyperplane with a fixed angle of $60$ degrees is studied. 
By considering the union of such graph with its reflection through the hyperplane, one gets an evolving symmetric {\em lens--shaped} domain. We remark that in this particular case, the analysis is simpler since the maximum principle can be applied.

\begin{ackn}
The authors want to warmly thank Tom Ilmanen for several discussions and suggestions and for the courtesy of providing us the figures and the numerical computations in the Appendix, done with the contribution of J\"org H\"attenschweiler.\\
C.M., M.N. and A.P. are members of the INDAM--GNAMPA.\\
The work of C.M. is partially supported by PRIN Project 2022E9CF89 ``GEPSO - Geometric Evolution Problems and Shape Optimization''.\\
The work of M.N. is partially supported by PRIN Project 2022E9CF89 -- PNRR Italia Domani, funded by EU Program NextGenerationEU
and by the MUR Excellence Department Project awarded to the Department of Mathematics of the University of Pisa.\\
The work of A.P. is partially supported by PRIN project 2022R537CS ``$\rm{NO}^3$ - Nodal Optimization, NOnlinear elliptic equations, NOnlocal geometric problems, with a focus on regularity'', PRA project ``GEODOM - Geometric evolution problems and PDEs on variable domains'', BIHO project ``NEWS - NEtWorks  and Surfaces evolving by curvature'' and by the MUR Excellence Department Project awarded to the Department of Mathematics of the University
of Pisa.\\
\end{ackn}

\section{Notation, definitions and basic computations}\label{notation}

Given a $C^1$ curve $\sigma:[0,1]\to\R^2$ we say that it is {\em regular} if
$\sigma_x=\frac{d\sigma}{dx}$ is never zero. It is then well defined
its {\em unit tangent vector} $\tau=\sigma_{x}/|\sigma_{x}|$. 
We define its {\em unit normal vector} as
$\nu=\mathrm{R}\tau=\mathrm{R}\sigma_{x}/|\sigma_{x}|$
where $\mathrm{R}:\mathbb{R}^{2}\to\mathbb{R}^{2}$ is the counterclockwise
rotation centered in the origin of $\mathbb{R}^{2}$ of angle
${\pi}/{2}$.\\
If the curve $\sigma$ is of class $C^2$ and regular its {\em curvature vector} is well
defined as 
$$
\underline{k}=\tau_x/|\sigma_{x}|=\frac{1}{|\,\sigma_{x}|}\frac{d\tau}{dx}\,.
$$
The {\em arclength parameter} of a curve $\sigma$ is given by
$$
s=s(x)=\int_0^x\vert\sigma_x(\xi)\vert\,d\xi\,.
$$
Notice that $\partial_s=\vert\sigma_x\vert^{-1}\partial_x$ 
then $\tau=\partial_s\sigma$ and $\underline{k}=\partial_s\tau$, 
hence the curvature of $\sigma$ is given by $k=\langle\underline{k}\,\vert\,\nu\rangle$, 
as $\underline{k}=k\nu$.
We remind here that 
in the whole paper, we will use the word ``curve'' both for the parametrization and for the
set (image of the parametrization in $\mathbb{R}^2$).\\

Let $T>0$ and $\gamma:[0,1]\times [0,T)$
a time--dependent family regular $C^2$ curve.
Again, 
we let $\tau=\tau\left(x,t\right)$ be the unit
tangent vector to the curve $\gamma$, $\nu=\nu\left(x,t\right)
=\mathrm{R}\tau\left(x,t\right)$ be the unit
normal vector and
$\underline{k}=\underline{k}\left(x,t\right)=k\left(x,t\right)\nu\left(x,t\right)$
its curvature vector, as previously defined.\\ 

Here and in the sequel we will denote by $\partial_xf$, $\partial_sf$ and $\partial_tf$
the derivatives of a function $f$ along a curve $\gamma$ with respect to the $x$ variable, 
the arclength parameter $s$ on such curve (defined by
$s(x,t)=\int_0^x\vert\gamma_x(\xi,t)\vert\,d\xi$) and the time, respectively;
$\partial^n_xf$, $\partial^n_sf$, $\partial^n_tf$ are
the higher order partial derivatives which often we will also write as
$f_x,f_{xx}\dots$, $f_{s}, f_{ss},\dots$ and $f_t, f_{tt},\dots$.\\

We will call $\underline{v}=\gamma_t=V\nu+\lambda\tau$,
$\underline{\lambda}=\lambda\tau$ and
$\underline{V}=V\nu$
respectively the {\em
velocity}, the 
{\em normal velocity} 
and the {\em tangential velocity}
of the curve $\gamma$.
The scalar $V$ and $\lambda$
are the normal and tangential components of the velocity.
It is easy to see that $\underline{v}=\underline{V}+\underline{\lambda}$ and
$|\underline{v}|^2=|\underline{V}|^2+|\underline{\lambda}|^2=(V)^2+(\lambda)^2$.

\subsection{Networks}\label{netdef999}

\begin{defn}\label{Cinfty}
Let $\Omega$ be a smooth, convex, open set in $\mathbb{R}^{2}$. A
{\em network} $\mathbb{S}=\bigcup_{i=1}^{n}\sigma^{i}([0,1])$ in
$\Omega$ is a connected set in the plane described by a finite family
of $C^1$, regular curves $\sigma^{i}:[0,1]\to\overline{\Omega}$ such
that
\begin{enumerate}
\item the ``interior'' of every curve $\sigma^{i}$, that is 
$\sigma^i(0,1)$, is embedded (hence, it has no self--intersections); 
a curve can self--intersect itself only possibly ``closing'' at its end--points, that is 
$\sigma^i(0)=\sigma^i(1)$;
\item two different curves can intersect each other only at their end--points;
\item if a curve of the network touches the boundary of $\Omega$ at a
 point $P$, no other end--point of a curve can coincide with that point.
\end{enumerate}

If we interpret $\mathbb{S}$ as a planar graph,
we call {\em multi--points} of the network the vertices
$O^{1}, O^{2},\dots, O^{m}\in\Omega$ where the order is greater than one.
We call {\em end--points} of the network the vertices
$P^{1}, P^{2},\dots, P^{l}\in\overline{\Omega}$ of $\mathbb{S}$ (on the
boundary or not) with order one.\\

We say that a network is of class $C^{k}$ or $C^\infty$ if all the
$n$ curves are respectively of class $C^{k}$ or $C^\infty$.
\end{defn}

\begin{rem}
We require Condition~3 for the sake of simplicity.
It implies that the multi--points can be 
only inside $\Omega$ and not on the boundary. The end--points can
be both inside or on $\partial\Omega$.
\begin{figure}[H]
\begin{center}
\begin{tikzpicture}[scale=0.75]
\draw[color=black,scale=1,domain=-3.141: 3.141,
smooth,variable=\t,shift={(7,0)},rotate=0]plot({3.25*sin(\t r)},
{2.5*cos(\t r)});
\draw 
(3.75,0) node[left] {$P^1$} to[out=-30,in=110, looseness=1] (4.97,0) 
to[out= -70,in=180, looseness=1] (6.19,0)
to[out=60, in=-85, looseness=1] (7, 0.81)
to[out=35,in=150, looseness=1] (8.7,0.85)
to[out=-30,in=160, looseness=1.5] (10.2,0.5) node[right]{$P^3$}
(9.95,-1) node[right]{$P^2$}
to[out=130,in=20, looseness=1] (8.62,-0.81) 
to[out=-160,in=-60, looseness=0.5] (6.19,0)
(7, 0.81)
to[out=155,in=30, looseness=1](5, 1.20) node[above]{$\sigma^4$}
to[out=-150,in=30, looseness=1]
(3.75,0);
\path[font= \Large]
(4.25,-1.8) node[below] {$\Omega$};
\path[font=\large]
(4.52,0.1) node[below] {$\sigma^1$}
(8.5,-1.1) node[right] {$\sigma^2$}
(9,0.85) node[above]{$\sigma^3$}
(6,-0.1) node[below] {$O^1$}
(7.4,1.0) node[below] {$O^2$}; 
\end{tikzpicture}
\end{center}
\begin{caption}{An example of ``violation'' of Condition~3 in the definition of network.}
\end{caption}
\end{figure} 
\end{rem}

The curves $\sigma^{i}$ have (non--zero) finite lengths
$L^i=\int_0^1|\sigma_x^i(\xi)|\,d\xi$.

\begin{defn}
Let $\mathbb{S}=\bigcup_{i=1}^{n}\sigma^{i}$ be a network composed of $n$ curves.
We denote by 
$$
L = L^1+ \dots + L^n
$$
the {\em global length} of the network. 
\end{defn}

\begin{defn}\label{Cinftyopen}
An {\em open network}
$\mathbb{S}=\bigcup_{i=1}^{n}\sigma^{i}(I)$ in
$\R^2$ is a connected set in the plane composed of a finite family
of $C^1$, regular curves $\sigma^{i}:I\to\R^2$, where $I$ can be the interval
$[0,1]$ or $[0,1)$, such that
\begin{enumerate}
\item every ``open'' curve $\sigma^i:[0,1)\to\R^2$ is 
$C^1$--asymptotic to a half-line in $\R^2$ as $x\to 1$;
\item the ``interior'' of every curve $\sigma^{i}$ 
is embedded (hence, it has no self--intersections). Only the bounded curves
$\sigma^i:[0,1]\to\R^2$ can possibly self--intersect by 
``closing'' at their end--points;
\item two different curves can intersect each other only at their end--points;
\item considering $\SS$ as a planar graph, 
every end--point of a curve belongs to some multi--point of 
the network with order at least two;
\end{enumerate}
As before we say that an open network is of class $C^{k}$ or $C^\infty$ if all its curves are respectively of class $C^{k}$ or $C^\infty$.
\end{defn}

\begin{rem}\label{rem2.5}
Since we called these unbounded networks ``open'', we will adopt the
word ``closed'' for the previous networks in Definition~\ref{Cinfty}
which are bounded and possibly have some end--points.
\end{rem} 

Given a network composed of $n$ curves with $l$ end--points $P^1, P^2,\dots, P^l\in\overline{\Omega}$ (if present) and $m$ multi--points $O^1, O^2,\dots O^m\in\Omega$, we will denote by $\sigma^{pi}$ the curves of this network concurring at the
multi--point $O^p$, with the index $i$ varying from one to the order of the multi--point $O^p$. 
This is clearly redundant as some curves coincide, but it is a useful notation for the computations. A network of $n$ curves with $m$ triple junctions only (without higher multiplicity junctions) will then be described by the family (with possible repetitions) of curves $\sigma^{pi}$ where $p\in{\{1,2,\dots,m\}}$ and $i\in{\{1,2,3\}}$.

We now define a special class of networks that will play a key
role in the analysis.

\begin{defn}\label{regularnetwork}
We call a network (open or not) {\em regular} if
all its multi--points $O^1, O^2,\dots O^m\in\Omega$ have order three 
and at each of them the three concurring curves $\{\sigma^{pi}\}_{i=1,2,3}$ 
meet in such a way that the external unit tangents $\tau^{pi}$
satisfy $\tau^{p1}+\tau^{p2}+\tau^{p3}=0$, which means that the three
curves form three angles of 120 degrees at $O^p$ ({\em Herring condition}).
 
\smallskip

We call a network {\em non--regular} if at least a multi--point has order
different from three or if it has order three but the external unit
tangents of the three concurring curves $\{\sigma^{pi}\}_{i=1,2,3}$ do
not satisfy $\tau^{p1}+\tau^{p2}+\tau^{p3}=0$. 
We will call such a point a {\em non--regular} multi--point.
\end{defn}

\begin{figure}[H]
\begin{center}\begin{tikzpicture}
\draw[color=black,scale=1,domain=-3.141: 3.141,
smooth,variable=\t,shift={(-1,0)},rotate=0]plot({3.75*sin(\t r)},
{3*cos(\t r)});
\draw (-3.5,0) node[above]{$O^1$}
to[out= 50,in=180, looseness=1] (-2,0) node[right]{$\, O^2$}
to[out= 60,in=160, looseness=1] (-0.5,1) node[above]{$O^5$}
(-2,0)
to[out= -60,in=170, looseness=1] (-0.8,-1.3) node[above]{$O^3$}
(-3.5,0)
to[out= 170,in=-70, looseness=1] (-4.5,1) node[left]{$P^1$}
(-3.5,0)
to[out= -70,in=30, looseness=1] (-4.2,-1.6) node[left]{$P^2$}
(-0.5,1)
to[out= 40,in=-80, looseness=1] (0.2,1.7) node[above]{$O^6$}
(-0.5,1)
to[out= -80,in=180, looseness=1] (0.3,0) node[right]{$\, O^4$}
(0.2,1.7)
to[out=40,in=-50, looseness=1](1.3,2.35) node[above]{$\,\,\,\,\,\, P^7$}
(0.2,1.7)
to[out=160, in=-60, looseness=1](-3,2.5) node[above]{$P^8$}
(-0.8,-1.3)
to[out=50, in=-150, looseness=1](2,-1.75) node[right]{$P^4$}
(-0.8,-1.3)
to[out=-70, in=60, looseness=1](-2.7,-2.7) node[below]{$P^3$}
(0.3,0)
to[out=60, in=-100, looseness=1](2.3,1.4) node[above]{$\,\,\,\,\,\,\,\, P^6$}
(0.3,0)
to[out=-60, in=-120, looseness=1](2.675,-0.5) node[right]{$P^5$};
\end{tikzpicture}
\end{center}
\begin{caption}{A regular network.}
\end{caption}
\end{figure}

We will simply 
omit the indices of the curves of the network anytime there is no need
to make them explicit.\\
Moreover,
given $\mathbb{S}_t=\bigcup_{i=1}^n\gamma^i([0,1],t)$
a time--dependent family of regular $C^2$ network of curves,
we will adopt the following convention for integrals,
$$
\int_{{\mathbb{S}_t}} f(t,\gamma,\tau,\nu,k,k_s,\dots,\lambda,\lambda_s\dots)\,ds =
\sum_{i=1}^n \int_0^1
f(t,\gamma^i,\tau^i,\nu^i,k^i,k^i_s,\dots,\lambda^i,\lambda^i_s\dots)\,\vert
\gamma^i_x\vert\,dx
$$
as the arclength measure 
on every curve $\gamma^i$
is given by $ds=\vert\gamma^i_x\vert\,dx$.\\

Sometimes we will also use the following notation for 
a time--dependent family of networks 
$$
\mathbb{S}_t=\bigcup_{i=1}^n\gamma^i([0,1],t)
$$ 
with $t\in [0,T)$
in $\Omega\subseteq\R^2$.
We let $\SS\subseteq\R^2$ be a ``reference'' 
network 
and suppose that for every $t\in (0,T)$ the network 
$\SS_t$ is homeomorphic to $\SS$.
We consider a map $F:\SS\times(0,T)\to\mathbb{R}^2$ given by the ``union'' of the maps
$\gamma^i:I_i\times(0,T)\to\overline{\Omega}$ 
describing the time--dependent family of networks in the time
interval $(0,T)$, that is $\SS_t=F(\SS,t)$.

\subsection{Motion by curvature}
We are now ready to define the evolution by curvature of a $C^2$ regular
network, assuming that
either it is open or all its end--points (if present)
coincide with some points $P^1, P^2,\dots, P^{l}$ on the
boundary of $\Omega$.
As we have already said, in the ``closed'' case by Condition~3 in Definition~\ref{Cinfty} at most one curve of the network can arrive at the point $P^r$. We require the network to be regular during the flow and we ask that the end--points
$P^{r}\in\partial\Omega$ stay fixed ({\em Dirichlet boundary
conditions}). A similar problem is given by letting 
such end--points ``free'' to move on the boundary of $\Omega$, but
asking that the curves intersect orthogonally $\partial\Omega$ 
({\em Neumann boundary conditions}).\\
In the ``closed case'', the motion by curvature is the {\em geometric gradient flow} of the
{\em length functional}, that is, 
the sum of the lengths of all the curves of the network. Roughly
speaking, a (solution to the) {\em flow by curvature} of a network is a smooth
family of embedded, planar networks, such that the normal component of
the velocity under the evolution law, at every point of every curve of the
evolving network is given by the curvature vector of the curve at the
point.

\begin{defn}\label{d1}
We say that a family of homeomorphic, 
regular networks $\mathbb{S}_t$, 
each one composed of $n$ curves $\gamma^i(\cdot, t):I_i\to\overline{\Omega}$
(where $I_i$ is the interval $[0,1]$ or $[0,1)$ in case of an open
network), in a smooth convex, open set $\Omega\subseteq\mathbb{R}^{2}$,
{\em moves by curvature} in the time interval $(0,T)$
if the functions $\gamma^i:I_i\times(0,T)\to\overline{\Omega}$
are at least of class $C^2$ in space and $C^1$ in time and for every $x\in I_i$,
$t\in(0,T)$, $i\in\{1,2,\dots, n\}$, they satisfy
\begin{equation}\label{evoleq}
\gamma^i_t(x,t)=k^i(x,t)\nu^i(x,t)+\lambda^i(x,t)\tau^i(x,t)=
\frac{\langle\gamma_{xx}^i(x,t)\,\vert\,\nu^i(x,t)\rangle} 
{{\vert{\gamma_x^i(x,t)}\vert}^2}\,\nu^i(x,t)
+\lambda^i(x,t)\tau^i(x,t)
\end{equation}
for some continuous functions $\lambda^i$. 
\end{defn}

\begin{rem}
Notice that the {\em normal velocity} is given by the curvature vector
of the curve $\gamma^i$ at every point.
\end{rem}

\begin{rem}\label{0notstandard}
Another equivalent way to state evolution equation~\eqref{evoleq} is clearly
\begin{equation}\label{evoleqbis}
\gamma^i_t(x,t)^\perp=k^i(x,t)\nu^i(x,t)=\underline{k}^i(x,t)
=\frac{\langle\gamma_{xx}^i(x,t)\,\vert\,\nu^i(x,t)\rangle} 
{{\vert{\gamma_x^i(x,t)}\vert}^2}\,\nu^i(x,t)\,.
\end{equation}
\end{rem}

\begin{rem}\label{notstandard}
We spend some words on the above definition of motion by curvature.
The evolution equation~\eqref{evoleq} 
is not the usual way to describe the motion by curvature of a smooth curve. 
Indeed, ``classically'' it is written as
\begin{equation}\label{stdevol}
\gamma^i_t=\underline{k}^i=k^i\nu^i=\frac{\langle\gamma_{xx}^i\,\vert\,\nu^i\rangle}
{{\vert{\gamma_x^i}\vert}^2}\,\nu^i\,.
\end{equation}
Both motions are driven by a system of quasilinear partial differential equations, 
in our definition ``admitting a correction'' by a tangential term. 
The two velocities differ only by a tangential component
 $\underline{\lambda}^i=\lambda^i\tau^i$. 
In the curvature evolution of a smooth curve, it is well--known that any tangential contribution 
to the velocity affects only the ``inner motion'' 
of the ``single points'' (Lagrangian point of view), 
but it does not affect the motion of a curve as a whole subset of $\R^2$ 
forgetting its parametrization (Eulerian point of view). 
It can be shown that a flow of a closed curve satisfying equation~\eqref{evoleq} 
can be globally reparametrized (dynamically in time) 
in order it satisfies equation~\eqref{stdevol}.
However, in our situation, such a global reparametrization is not possible 
due to the presence of the $3$--points. It is necessary to consider 
such extra tangential terms to allow the motion of the $3$--points also. 
Indeed, if the velocity would be in normal
direction at every point of the three curves concurring at a $3$--point,
this latter should move in a direction which is normal to all of them, 
then the only possibility would be that it does not move at all (see also the discussions
and examples in~\cite{brakke,bronsard,kinderliu, mannovplunotes}).
\end{rem}

\begin{defn}\label{probdef} Given an initial, regular, $C^2$ network $\SS_0$, composed of
$n$ curves $\sigma^i:[0,1]\to\overline{\Omega}$, with $m$ triple junctions $O^1, O^2,\dots
O^m\in\Omega$ and (if present) $l$ end--points $P^1, P^2,\dots, P^l\in\partial\Omega$ 
in a smooth convex, open set
$\Omega\subseteq\mathbb{R}^{2}$, we say that a family of homeomorphic 
networks $\mathbb{S}_t$ described by the
family of time--dependent curves $\gamma^i(\cdot, t)$ is a {\em flow by curvature} of $\SS_0$ with fixed end--points in the time interval $[0,T)$, if the functions $\gamma^i:[0,1]\times[0,T)\to\overline{\Omega}$ are continuous, there holds $\gamma^i(x,0)=\sigma^i(x)$ 
for every $x\in[0,1]$ and $i\in\{1,2,\dots, n\}$ (initial data), they 
are at least $C^2$ in space and $C^1$ in time in $[0,1]\times(0,T)$ and satisfy
the following system of conditions for every $x\in[0,1]$, $t\in(0,T)$, $i\in\{1,2,\dots, n\}$,
\begin{equation}\label{problema}
\begin{cases}
\begin{array}{lll}
\gamma^i_t=k^i\nu^i+\lambda^i\tau^i\quad&
\text{with}\, \lambda^i\,\text{continuous functions}
\quad &\text{ motion by curvature}\\
\gamma_x^i(x,t)\not=0
\quad &&\text{ regularity}\\
\gamma^r(1,t)=P^r\quad&\text{with}\, 0\leqslant r \leqslant l\,
\quad &\text{ fixed end--points condition}\\
\sum_{j=1}^3\tau^{pj}(O^p,t)=0\quad&\text{at every $3$--point $O^p$}
\quad &\text{ angles of $120$ degrees}
\end{array}
\end{cases}
\end{equation}
where we assumed conventionally (possibly reordering the family of
curves and ``inverting'' their parametrization) that the end--point
$P^r$ of the network is given by $\gamma^r(1,t)$ 
(by Condition~3 in Definition~\ref{Cinfty} this can be always done).\\
Moreover, in the third equation, we abused a little the notation, denoting with $\tau^{pj}(O^p,t)$
the respective exterior unit tangent vectors at $O^p$ of the three curves
$\gamma^{pj}(\cdot,t)$ in the family $\{\gamma^i(\cdot,t)\}$ concurring at the
$3$--point $O^p$.
\end{defn}

We also state the same problem for regular, open networks.

\begin{defn}\label{probdef-open} Given an initial, regular, $C^2$ open network $\SS_0$, composed of $n$ curves $\sigma^i:I_i\to\R^2$, we say that a family of homeomorphic 
open networks $\mathbb{S}_t$ with the same structure as $\SS_0$ (in
particular, same asymptotic half-lines at infinity) described by the
family of time--dependent curves $\gamma^i(\cdot, t)$ is a {\em flow by curvature} of $\SS_0$ in the time interval $[0,T)$, if the functions $\gamma^i:I_i\times[0,T)\to\R^2$
are continuous, there holds $\gamma^i(x,0)=\sigma^i(x)$ for every
$x\in I_i$ and $i\in\{1,2,\dots, n\}$ (initial data), they are of class
at least $C^2$ in space and $C^1$ in time in $I_i\times(0,T)$ (here $I_i$ denotes the interval
$[0,1]$ or $[0,1)$ depending on whether the curve is unbounded or not) and satisfy
the following system, for every $x\in I_i$, $t\in(0,T)$, $i\in\{1,2,\dots, n\}$,
\begin{equation}\label{problema-open}
\begin{cases}
\begin{array}{lll}
\gamma^i_t=k^i\nu^i+\lambda^i\tau^i\quad&
\text{with}\, \lambda^i\,\text{continuous functions}
\quad &\text{ motion by curvature}\\
\gamma_x^i(x,t)\not=0
\quad &&\text{ regularity}\\
\sum_{j=1}^3\tau^{pj}(O^p,t)=0\quad&\text{at every $3$--point $O^p$}
\quad &\text{ angles of $120$ degrees}
\end{array}
\end{cases}
\end{equation}
where, in the second equation, we used the same notation as in Definition~\ref{probdef}.
\end{defn}

\begin{rem}\label{geocomprem0}
In Definitions~\ref{probdef} and~\ref{probdef-open} the
evolution equation~\eqref{evoleq} must be satisfied till the borders
of the intervals $[0,1]$ and $[0,1)$, that is, at the $3$--points and
the at end--points for every positive time. 
This is not the usual way to state boundary conditions for parabolic problems 
(the parabolic nature of this evolution problem is clear by Definition~\ref{d1}
 -- see also Remark~\ref{notstandard} 
and it will be even clearer in Section~\ref{smtm})
where usually only continuity at the boundary is required. 
Anyway as is common in parabolic problems, at every positive time 
such boundary conditions are satisfied by any ``natural solution''.\\
This property of regularity at the boundary implies that 
$$
\text{$(k\nu+\lambda\tau)(P^r)=0$,
{ for every} $r\in\{1,2,\dots,l\}$}
$$
and
$$
\text{
$(k^{pi}\nu^{pi}+\lambda^{pi}\tau^{pi})(O^p)=(k^{pj}\nu^{pj}+\lambda^{pj}\tau^{pj})(O^p)$,
{ for every} $i,j\in\{1,2,3\},\, p\in\{1,2,\dots m\}$}
$$
(where we abused a little the notation), obtained by simply requiring
that the velocity is zero at every end--point and it is the same for
any three curves at their concurrency $3$--point.

Moreover, notice that in Definitions~\ref{probdef} and~\ref{probdef-open} the
evolution equation~\eqref{evoleq} must be satisfied only for $t>0$. If
we want that the maps $\gamma^i$ are $C^2$ in space and $C^1$ 
in time till the whole {\em parabolic boundary} (given by $[0,1]\times\{0\}\cup\{0,1\}\times[0,T)$
in Definition~\ref{probdef} and
$[0,1]\times\{0\}\cup\{0,1\}\times[0,T)$ or
$[0,1)\times\{0\}\cup\{0\}\times[0,T)$ in
Definition~\ref{probdef-open}), the above conditions must be satisfied
also by the initial regular network $\SS_0$, for some functions
$\lambda_0$ extending continuously the functions $\lambda$ which are defined only for $t>0$.
\end{rem}

For the moment we focus on regular networks. 
Several difficulties arise when we study problems~\eqref{problema} and~\eqref{problema-open} with 
non--regular networks as initial data.
The issues are related to the presence of multi--points: 
if there are multi--points $O^{p}$ of order greater than three, there 
can be several possible candidates for the flow. Considering for
example the case of a network composed of two curves crossing each
other (presence of 4--point); one cannot easily decide how the angle at the meeting point must
behave, indeed one can allow the four concurrent curves to separate
in two pairs of curves, moving independently of each other and could even
be taken into account the creation of new multi--points from a single one.\\
If there are several multi--points during the flow some of them can
collapse together and the length of at least one curve of the network
can go to zero.\\
In these cases, one must possibly restart the evolution with a different
set of curves and the topology of the network changes dramatically, forcing
to change the ``structure'' of the system of equations governing the
evolution. Anyway, a very natural conjecture is that the curvature flow of a general network 
(under a suitably good definition) should be non--regular only for a discrete set of times. 
We will get back to this in the following sections.

\begin{rem}\label{specevol} 
One can clearly obtain solutions to system~\eqref{evoleq}
by requiring each curve to fulfill 
the quasilinear partial differential equation:
\begin{equation}\label{evoleq0}
\gamma^i_t=\frac{\gamma_{xx}^i}{{\vert{\gamma_x^i}\vert}^2}\,.
\end{equation}
In this case
$$
\begin{array}{ll} 
\underline{v}^i=\underline{v}^i(x,t)= \frac{\gamma_{xx}^i}{{\vert{\gamma_x^i}\vert}^2} 
&\qquad\text{ velocity of the point} \,\gamma^i(x,t)\,,\\ 
\lambda^i=\lambda^i(x,t)= \frac{\langle\gamma_{xx}^i\,\vert\,\tau^i\rangle} 
{{\vert{\gamma_x^i}\vert}^2}=\frac{\langle\gamma_{xx}^i\,\vert\,\gamma^i_x\rangle} 
{{\vert{\gamma_x^i}\vert}^3}=-\partial_x\frac{1}{\vert \gamma^i_x\vert} 
&\qquad\text{ tangential velocity of the point} \,\gamma^i(x,t)\,,\\ 
k^i=k^i(x,t)= \frac{\langle\gamma_{xx}^i\,\vert\,\nu^i\rangle} 
{{\vert{\gamma_x^i}\vert}^2}= 
\langle\partial_s\tau^i\,\vert\,\nu^i\rangle= 
-\langle\partial_s \nu^i\,\vert\,\tau^i\rangle 
&\qquad\text{ curvature at the point} \,\gamma^i(x,t)\,.
\end{array} 
$$
\end{rem}

\begin{defn}\label{special}
A curvature flow $\gamma^i$ for the initial, regular $C^2$ network 
$\SS_0=\bigcup_{i=1}^n\sigma^i([0,1])$ which satisfies 
$\gamma^i_t=\frac{\gamma_{xx}^{i}}
{|\gamma_{x}^{i}|^{2}}$ for every $t>0$ will be called a {\em special curvature flow} of $\SS_0$.
In this case, then we pass from the general system~\eqref{problema} to the following:
\begin{equation}\label{problema-nogauge-general}
\begin{cases}
\begin{array}{lll}
\gamma^i_t(x,t)=\frac{\gamma_{xx}^{i}\left(x,t\right)}{\left|\gamma_{x}^{i}\left(x,t\right)\right|^{2}}
\qquad &&\text{ special motion by curvature}\\
\gamma_x^i(x,t)\not=0
\quad &&\text{ regularity}\\
\gamma^r(1,t)=P^r\quad&\text{with $0\leqslant r \leqslant l$}\,
\quad &\text{ fixed end--points condition}\\
\sum_{j=1}^{3}\frac{\gamma_{x}^{pj}\left(O^p,t\right)}{\left| 
\gamma_{x}^{pj}\left(O^p,t\right)\right|}=0\quad&\text{at every $3$--point $O^p$}
\quad&\text{ angles of $120$ degrees}\\
\gamma^i(x,0)=\sigma^i(x)
\qquad &&\text{ initial data}
\end{array}
\end{cases}
\end{equation}
\end{defn}

\begin{rem}\label{gremh}
There are classes of networks, whose topological structure is particularly simple, whose evolution by curvature has
been extensively studied in the literature.
\begin{itemize}
\item When the network consists of a single closed embedded curve, its motion by curvature was widely 
studied~\cite{angen1,angen2,angen3,gage,gage0,gaha1,gray1}: the
curve evolves smoothly, becomes convex, and shrinks to a point
in finite time, becoming rounder and rounder. 
Curves with an angle or a cusps (where the
cusp is the most ``delicate'' situation) can be dealt with by means
of the works of Angenent~\cite{angen1,angen2,angen3}: 
the curve becomes immediately smooth, flowing by curvature.
\item The case in which two curves concur at a $2$--point of the network 
forming an angle (or a cusp, if they have the same tangent)
can be analyzed as we said above: consider them as a
single curve with a ``singular'' point (the angle) that vanishes immediately under
the flow.
\item If a network is composed of a single embedded curve with fixed
end--points, its evolution by curvature is discussed in~\cite{huisk2,stahl1,stahl2}.
The curve converges to the straight segment connecting 
the two fixed end--points in infinite time.
\begin{figure}[H]
\begin{center}
\begin{tikzpicture}[scale=0.78]
\draw
(-4.25,0) to[out=90,in=-100,looseness=1] (-2,1)
(-2,0.7) node[below] {$\sigma$} (-2,1)
to[out=80,in=90,looseness=1](0,0)
to[out=270,in=60,looseness=1](-2.5,-1)
to[out=-120,in=270,looseness=1](-4.25,0);
\draw [shift={(-2.5,0)}]
(3.75,0) node[left] {$P$}
to[out=30,in=110, looseness=1] (4.97,0) 
(4.5,0.3) node[below] {$\sigma^1$} (4.97,0)
to[out= -70,in=180, looseness=1] (6.19,0)
node[below] {$O$}
to[out=60, in=-145, looseness=1] (7, 0.81)
(7.2,0.9) node[below] {$\sigma^2$} (7,0.81)
to[out=35,in=150, looseness=1] (8.62,0.81)
node[right] {$Q$};
\end{tikzpicture}\quad
\begin{tikzpicture}[scale=0.53]
\draw[color=black,scale=1,domain=-3.141: 3.141,
smooth,variable=\t,shift={(-1,0)},rotate=0]plot({3.25*sin(\t r)},
{2.5*cos(\t r)});
\draw 
(-1.81,0)
to[out=90, in=-145, looseness=1] (-1, 0.81)
to[out=35,in=150, looseness=1] (0.62,0.81)
to[out=-30,in=150, looseness=1.5] (1.95,1)
node[right]{$\, P^2$}
(1.62,-1.5) node[right]{$P^1$}
to[out=130,in=20, looseness=1] (0.62,-0.81) 
to[out=-160,in=-90, looseness=0.5] (-1.81,0);
\path[font= \large]
(-3.75,-1.8) node[below] {$\Omega$};
\path (-2.3,0.5) node[right] {$\sigma$};
\end{tikzpicture}
\end{center}
\begin{caption}{Three special cases: a single closed curve, 
two curves forming an angle at their junction
and a single curve with two end--points on the boundary of $\Omega$.}
\end{caption}
\end{figure}
\end{itemize}
\end{rem}

\subsection{Basic computations}\label{basiccomp}
We work out some basic relations and formulas holding 
for a regular network evolving by curvature, 
assuming that all the derivatives of the functions $\gamma^i$ and $\lambda^i$ 
that appear to exist.

\begin{lem} If $\gamma$ is a curve moving by
$$
\gamma_t=k\nu+\lambda\tau\,,
$$
then the following commutation rule holds
\begin{equation}\label{commut}
\dert\ders=\ders\dert +
(k^2 -\lambda_s)\ders\,.
\end{equation}
\end{lem}
\begin{proof}
Let $f:[0,1]\times[0,T)\to\R$ be a smooth function, then
\begin{align*}
\dert\ders f - \ders\dert f =&\, \frac{f_{tx}}{\vert\gamma_x\vert} -
\frac{\langle \gamma_x\,\vert\,\gamma_{xt}\rangle f_x}
{\vert\gamma_x\vert^3}
- \frac{f_{tx}}{\vert\gamma_x\vert} = - {\langle
 \tau\,\vert\,\partial_s\gamma_t\rangle}\partial_sf\\
=&\, - {\langle\tau\,\vert\,\partial_s(\lambda\tau+k\nu)\rangle}\partial_sf=
 (k^2 - \lambda_s)\ders f
\end{align*}
and the formula is proved.
\end{proof}
Then, thanks to the commutation rule of the previous lemma for an evolving curve we can compute
\begin{align}
\dert\tau=&\,
\dert\ders\gamma=\ders\dert\gamma+(k^2-\lambda_s)\ders\gamma =
\ders(\lambda\tau+k\nu)+(k^2-\lambda_s)\tau =
(k_s+k\lambda)\nu\label{derttau}\,,\\
\dert\nu=&\, \dert({\mathrm R}\tau)={\mathrm
R}\,\dert\tau=-(k_s+k\lambda)\tau\label{dertdinu}\,,\\
\dert k=&\, \dert\langle \ders\tau\,|\, \nu\rangle=
\langle\dert\ders\tau\,|\, \nu\rangle\label{dertdik}
= \langle\ders\dert\tau\,|\, \nu\rangle +
(k^2-\lambda_s)\langle\ders\tau\,|\, \nu\rangle\\
=&\, \ders\langle\dert\tau\,|\, \nu\rangle + k^3-k\lambda_s =
\ders(k_s+k\lambda) + k^3-k\lambda_s\nonumber\\
=&\, k_{ss}+k_s\lambda + k^3\nonumber\,.
\end{align}

Moreover, as anticipated in Remark~\ref{specevol}, when the tangential velocity is
$\lambda=\frac{\langle\gamma_{xx}\,\vert\,\gamma_x\rangle}{{\vert\gamma_x\vert}^3}$,
the curve $\gamma$ evolves according to
$$
\gamma_t=\frac{\gamma_{xx}}{{\vert\gamma_x\vert}^2}=k\nu+\lambda\tau\,,
$$
so we can also compute
\begin{align}
\dert\lambda =&\, -\dert\partial_x\frac{1}{\vert\gamma_x\vert}=
\partial_x \frac{\langle\gamma_x\,\vert\,\gamma_{tx}\rangle}
{\vert\gamma_x\vert^3}=
\partial_x \frac{\langle\tau\,\vert\,\ders (\lambda\tau+k\nu)\rangle}
{\vert\gamma_x\vert}=\partial_x \frac{(\lambda_s - k^2)}
{\vert\gamma_x\vert}\label{dertdilamb}\\
=&\, \ders(\lambda_s - k^2) -\lambda(\lambda_s -
k^2)=\lambda_{ss} -\lambda\lambda_s - 2kk_s +\lambda k^2\,.\nonumber
\end{align}

We consider the curvature flow given by a family of regular, $C^\infty$ 
networks $\mathbb{S}_t$, composed of $n$ moving curves
$\gamma^{i}$ with $m$ triple junctions $O^1, O^2,\dots, O^m$ and 
$l$ end--points $P^1, P^2,\dots, P^l$.\\
As we said, we parametrize the curves of the evolving network
so that $\gamma^i(1,t)=P^i$ whenever 
$P^i$ is an end--point where a curve $\gamma^i$ arrives.
Consider instead a triple junction, say ${O}^p$,
where three distinct curves $\gamma^{p1}$, $\gamma^{p2}$ and $\gamma^{p3}$
meet. In general, we cannot always impose that 
\begin{equation}\label{eqokcar}
\gamma^{p1}(0,t)=\gamma^{p2}(0,t)=\gamma^{p3}(0,t)={O}^p(t)
\end{equation}
for all $p\in \{1,\ldots, m\}$, since (for instance) both the end--points of a curve could belong to the same triple junction, or simply for combinatorial reasons (see the networks in the following figure).

\begin{figure}[H]
\begin{center}
\begin{tikzpicture}
\draw (-4.25,0) node[left] {$P^1$} to[out=30,in=110, looseness=1] (-3.03,0)
to[out= -50,in=180, looseness=1] (-1.81,0)
to[out=60, in=35, looseness=1] (-1, 0.81)
to[out=35,in=150, looseness=1] (0.62,0.81)
to[out=-30,in=90, looseness=1.5] (1.44,0)
to[out=-90,in=20, looseness=1] (0.62,-0.81)
to[out=-160,in=-60, looseness=0.5] (-1.81,0);
\draw[white]
(0,-2)--(0,-2.1);
\path[font=\small]
(-3.8,0.6) node[right] {$\gamma^1$}
(0.5,-1) node[right] {$\gamma^2$}
(-1.9,0) node[below] {$\gamma^1$};
\end{tikzpicture}\qquad\qquad\qquad
\begin{tikzpicture}[scale=1]
\draw
 (-3.73,0) 
to[out= 50,in=180, looseness=1] (-2.3,0.7) 
to[out= 60,in=180, looseness=1.5] (-0.45,1.55) 
(-2.3,0.7)
to[out= -60,in=130, looseness=0.9] (-1,-0.3)
to[out= 10,in=100, looseness=0.9](0.1,-0.8)
(-1,-0.3)
to[out=-110,in=50, looseness=0.9](-2.7,-1.7);
\path[font=\small]
(-3.73,0) node[left]{$P^2$}
(-2.9,-1.7)node[below]{$P^3$}
(0.1,-0.8)node[right]{$P^4$} 
(-0.40,1.6) node[right]{$P^1$}
(-3,0.6) node[below] {$\gamma^2$}
(-1.5,1.3) node[right] {$\gamma^1$}
(-1.1,-1.2)[left] node{$\gamma^3$}
(0,-0.8)[left] node{$\gamma^4$}
(-1.3,0.5)[left] node{$\gamma^5$}
(-2.45,1.3) node[below] {${O}^1$}
(-1.4,-0.1) node[below] {${O}^2$}; 
\end{tikzpicture}
\end{center}
\begin{caption}{Examples of networks.}
\end{caption}
\end{figure}
Actually, in general, there holds
\begin{equation}\label{complicatedconcurrency}
\gamma^{p1}(x_1,t)=\gamma^{p2}(x_2,t)=\gamma^{p3}(x_3,t)={O}^p(t)\,,
\end{equation}
for every $p\in \{1,\ldots, m\}$ and some $x_1,x_2,x_3\in\{0,1\}$. Then, the fact that $x_1,x_2,x_3$ could be either $0$ or $1$ affects how the $120$ degrees angle condition at $O^p$ reads, that is,
\begin{equation}\label{complicatedanglecond}
(-1)^{x_1}\tau^{p1}(x_1,t)+(-1)^{x_2}\tau^{p2}(x_2,t)+(-1)^{x_3}\tau^{p3}(x_3,t)=0\,.
\end{equation}
For the sake of presentation and clarity, in the following analysis of the conditions holding at any $3$--point $O^p$, with $p\in\{1,2,\dots,m\}$, we will suppose that the three curves 
are parametrized in such a way that they all concur at $O^p$ for $x_1=x_2=x_3=0$, hence formula~\eqref{eqokcar} holds.

Differentiating in time the concurrency condition
$$
\gamma^{pi}\left(0,t\right)=\gamma^{pj}\left(0,t\right)\qquad\text{for every $i$ and $j$,}
$$
where $\gamma^{pi}$ denotes the $i$--th curve concurrent at the
$3$--point $O^p$, we get
$$
\lambda^{pi}\tau^{pi}+k^{pi}\nu^{pi}=
\lambda^{pj}\tau^{pj}+k^{pj}\nu^{pj}\,,
$$
at every $3$--point $O^p$, with
$p\in\{1,2,\dots,m\}$ for every $i,j\in\{1,2,3\}$.\\
Multiplying these vector identities
by $\tau^{pl}$ and $\nu^{pl}$ and varying $i, j, l$, thanks to the
conditions 
$$
\sum_{i=1}^{3}\tau^{pi}=\sum_{i=1}^{3}\nu^{pi}=0\,,
$$ we get
the relations
\begin{gather*}
\lambda^{pi}=-\lambda^{p(i+1)}/2-\sqrt{3}k^{p(i+1)}/2\\
\lambda^{pi}=-\lambda^{p(i-1)}/2+\sqrt{3}k^{p(i-1)}/2\\
k^{pi}=-k^{p(i+1)}/2+\sqrt{3}\lambda^{p(i+1)}/2\\
k^{pi}=-k^{p(i-1)}/2-\sqrt{3}\lambda^{p(i-1)}/2
\end{gather*}
with the convention that the second superscripts are to be 
considered ``modulus $3$''. Solving this system we get
\begin{gather*}
\lambda^{pi}=\frac{k^{p(i-1)}-k^{p(i+1)}}{\sqrt{3}}\\
k^{pi}=\frac{\lambda^{p(i+1)}-\lambda^{p(i-1)}}{\sqrt{3}}
\end{gather*}
which implies
\begin{equation}\label{eq:cond2} 
\sum_{i=1}^3 k^{pi}=\sum_{i=1}^3\lambda^{pi}=0
\end{equation}
at any $3$--point $O^p$ of the network $\SS_t$.\\
Moreover considering ${\mathrm K}^p=(k^{p1},k^{p2},k^{p3})$ and
${\mathrm\Lambda}^p=(\lambda^{p1},\lambda^{p2},\lambda^{p3})$ as vectors in
$\R^3$, we have seen that ${\mathrm K}^p$ and ${\mathrm\Lambda}^p$ belong
to the plane orthogonal to the vector $(1,1,1)$ and 
\begin{equation}\label{kl}
{\mathrm K}^p={\mathrm \Lambda}^p\wedge (1,1,1)/\sqrt{3}\,,\qquad\qquad\qquad 
{\mathrm \Lambda}^p=-{\mathrm K}^p\wedge (1,1,1)/\sqrt{3}\,,
\end{equation}
that is, ${\mathrm K}^p={\mathrm S}{\mathrm \Lambda}^p$ and 
${\mathrm \Lambda}^p=-{\mathrm S}{\mathrm K}^p$ where ${\mathrm S}$ 
is the rotation in $\R^3$ of an angle of $\pi/2$ around the axis
${\mathrm I}=\langle(1,1,1)\rangle$. Hence it also follows that
\begin{equation}\label{stimaklambdaneitripunti}
\sum_{i=1}^3 (k^{pi})^2=\sum_{i=1}^3 (\lambda^{pi})^2
\qquad\text{ { and} }\qquad \sum_{i=1}^3 k^{pi}\lambda^{pi}=0\,. 
\end{equation}
at any $3$--point $O^p$ of the network $\SS_t$.

Now we differentiate in time the angular condition
$\sum_{i=1}^3\tau^{pi}=0$ at every $3$--point $O^p$, with
$p\in\{1,2,\dots,m\}$, by equation~\eqref{derttau} for every pair $i, j$ we get
\begin{equation*}
k_{s}^{pi}+\lambda^{pi}k^{pi}=k_{s}^{pj}+\lambda^{pj}k^{pj}\,.
\end{equation*}
In terms of vectors in $\R^3$, as before, we can write
$$
{\mathrm K}_s^p+{\mathrm {\Lambda}}^p{\mathrm K}^p=(k^{p1}_s+\lambda^{p1}
 k^{p1},k^{p2}_s+\lambda^{p2} k^{p2},k^{p3}_s+\lambda^{p3} k^{p3})\in {\mathrm I}\,.
$$
Differentiating repeatedly in time all these vector relations we have
\begin{gather}
\partial_t^l{\mathrm K}^p\,,\,\partial_t^l{\mathrm \Lambda}^p\perp{\mathrm
 I}\quad \text{ and }\quad \partial_t^l\langle{\mathrm K}^p\,\vert\,
{\mathrm \Lambda}^p\rangle=0\,,\nonumber\\
\dert^l{\mathrm \Lambda}^p=-\dert^l{\mathrm S}{\mathrm K}^p=-{\mathrm
 S}\dert^l{\mathrm K}^p\,,\label{lambdakappa}\\
\partial_t^m({\mathrm K}_s^p+{\mathrm {\Lambda}}^p{\mathrm K}^p)\in 
{\mathrm I}\nonumber\,,
\end{gather}
which, making explicit the indices, give the following identities at every $3$--point $O^p$, with $p\in\{1,2,\dots,m\}$,
\begin{equation} 
\dert^l\sum_{i=1}^3 k^{pi}=\sum_{i=1}^3\dert^l k^{pi}
=\dert^l\sum_{i=1}^3 \lambda^{pi}=\sum_{i=1}^3\dert^l\lambda^{pi}
=\dert\sum_{i=1}^3k^{pi}\lambda^{pi}=0\,,\nonumber\\ 
\end{equation}
\begin{equation}
\sum_{i=1}^3 (\dert^l k^{pi})^2=\sum_{i=1}^3 (\dert^l\lambda^{pi})^2\nonumber\,\,
\text{ { for every} $l\in\NN$,}\\ 
\end{equation} 
\begin{equation} 
\dert^m(k^{pi}_s+\lambda^{pi} k^{pi})=
\dert^m(k^{pj}_s+\lambda^{pj} k^{pj}) \,\,\text { for every pair $i, j$ and $m\in\NN$.}\nonumber\\ 
\end{equation} 
Moreover by the orthogonality relations with respect to the axis ${\mathrm I}$ we
get also
$$
\partial^l_t{\mathrm K}^p\partial_t^m({\mathrm K}^p_s
+{\mathrm {\Lambda}}^p{\mathrm K}^p)=\partial^l_t{\mathrm 	
\Lambda}^p\partial_t^m({\mathrm K}_s^p
+{\mathrm {\Lambda}}^p{\mathrm K}^p)=0\,,
$$
that is,
\begin{equation}
 \sum_{i=1}^3 \dert^lk^{pi}\,\dert^m(k^{pi}_s+\lambda^{pi} k^{pi}) 
=\sum_{i=1}^3 \dert^l\lambda^{pi}\,\dert^m(k^{pi}_s+\lambda^{pi} k^{pi})=0 \,\,
\text { for every $l, m\in\NN$.}\label{eq:orto}\\ 
\end{equation}

\begin{rem}\label{geocomprem1}
By the previous computations, for every solution in Definitions~\ref{probdef} 
or~\ref{probdef-open} at $t>0$ the curvature at the end--points and the sum of the three 
curvatures at every $3$--point has to be zero and the same holds for the functions $\lambda$.\\
Then, a necessary condition for the maps $\gamma^i$ to be $C^2$ in
space and $C^1$ in time till the whole {\em parabolic boundary} (given
by $[0,1]\times\{0\}\cup\{0,1\}\times[0,T)$ in Definition~\ref{probdef} and
$[0,1]\times\{0\}\cup\{0,1\}\times[0,T)$ or
$[0,1)\times\{0\}\cup\{0\}\times[0,T)$ in
Definition~\ref{probdef-open}) is that these conditions are satisfied
also by the initial regular network $\SS_0$, for some functions
$\lambda_0$ (see Remark~\ref{geocomprem0}) extending continuously 
the functions $\lambda$ which are defined only for $t>0$. 
That is, for the initial regular network $\SS_0$, there must hold 
$$
\text{$(k\nu+\lambda_0\tau)(P^r)=0$,
{ for every} $r\in\{1,2,\dots,l\}$}
$$
and
$$
\text{
$(k^{pi}\nu^{pi}+\lambda^{pi}_0\tau^{pi})(O^p)=(k^{pj}\nu^{pj}+\lambda^{pj}_0\tau^{pj})(O^p)$,
{ for every} $i,j\in\{1,2,3\}$.}
$$
In particular for the initial network $\SS_0=\bigcup_{i=1}^{n}\sigma^{i}(I_i)$ 
the curvature at the end--points and the sum of the three curvatures 
at every $3$--point has to be zero.\\
These conditions on the curvatures of $\SS_0=\bigcup_{i=1}^{n}\sigma^{i}(I_i)$ 
are clearly {\em geometric}, that is independent of the parametrizations 
of the curves $\sigma^i$ but intrinsic to the {\em set} $\SS_0$ 
and they are not satisfied by a generic regular, $C^2$ network
\end{rem}

\section{Short time existence I}\label{smtm}

We want to study existence and uniqueness of the flow by curvature of an initial regular network with fixed end--points on the boundary of a smooth, convex, open set $\Omega\subseteq\R^2$, as in Definition~\ref{probdef}.\\
First of all, we need to discuss what we mean by {\em uniqueness of the flow} in our geometric context. 
If we consider an evolving network $\SS_t$, composed by curves $\gamma^i$ solutions of system~\eqref{problema}, that is, satisfying $\gamma^i_t=\underline{k}^i+\underline{\lambda}^i$ and dynamically reparametrize each curve $\gamma^i_t$ with sufficiently regular maps $\varphi^i:[0,1]\times[0,T)\to[0,1]$ (for instance, $C^2$ in space and $C^1$ in time) such that $\varphi^i(0,t)=0$, $\varphi^i(1,t)=1$, $\varphi^i(x,0)=x$ and $\varphi_x^i(x,t)\not=0$ for every $(x,t)\in[0,1]\times[0,T)$, we get another solution of system~\eqref{problema} (see Remark~\ref{notstandard}). This fact is related to the geometric nature of the problem: if $\widetilde{\gamma}^i(x,t)={\gamma}^i(\varphi^i(x,t),t)$, we have indeed
\begin{align*}
\widetilde{\gamma}^i_t(x,t)=&\,\partial_t[{\gamma}^i(\varphi^i(x,t),t)]\\
=&\,{\gamma}^i_x(\varphi^i(x,t),t)\varphi^i_t(x,t)+{\gamma}^i_t(\varphi^i(x,t),t)\\
=&\,{\gamma}^i_x(\varphi^i(x,t),t)\varphi^i_t(x,t)+\underline{{k}}^i(\varphi^i(x,t),t)
+\underline{{\lambda}}^i(\varphi^i(x,t),t)\\
=&\,\underline{{k}}^i(\varphi^i(x,t),t)+\underline{{\lambda}}^i(\varphi^i(x,t),t)+
\widetilde{\gamma}^i_x(x,t)\varphi^i_t(x,t)/\varphi^i_x(x,t)\\
=&\,\underline{\widetilde{k}}^i(x,t)+\underline{\widetilde{\lambda}}^i(x,t)\,,
\end{align*}
with 
$$
\underline{\widetilde{\lambda}}^i(x,t)=\underline{{\lambda}}^i(\varphi^i(x,t),t)+
\widetilde{\gamma}^i_x(x,t)\varphi^i_t(x,t)/\varphi^i_x(x,t)\,.
$$
Hence, being $\widetilde{\gamma}^i(x,0)=\gamma^i(x,0)=\sigma^i(x)$, the flow of the networks $\widetilde{\SS}_t$ given by the curves $\widetilde{\gamma}^i$ is another curvature flow for the initial network $\SS_0=\bigcup_{i=1}^n\sigma^i([0,1])$.

For this reason, the natural notion of uniqueness of the curvature flow is ``{\em up to dynamic reparametrizations}''. It is then also clear that we could have considered our networks simply as {\em sets} and their curvature flows as flows of {\em sets} that could be parametrized in order to satisfy Definition~\ref{probdef}. In~\cite{mannovplunotes} it actually followed this possible alternative point of view.

\begin{defn}\label{uniqdef}
We say that the curvature flow $\SS_t$ of an initial network 
$\SS_0=\bigcup_{i=1}^n\sigma^i([0,1])$ is {\em geometrically unique} 
in some regularity class $\mathbb{E}$, if all the curvature flows in such class, solutions of system~\eqref{problema}, with the same initial network, can be obtained from each other using time--dependent reparametrizations.\\
More precisely, if $\SS_t$ and $\widetilde{\SS}_t$ are two curvature flows of 
$\SS_0$, described by some maps $\gamma^i\in\mathbb{E}$ and $\widetilde{\gamma}^i\in\mathbb{E}$, 
there exists a family of sufficiently regular
maps $\varphi^i:[0,1]\times[0,T)\to[0,1]$ 
such that $\varphi^i(0,t)=0$, $\varphi^i(1,t)=1$, $\varphi^i(x,0)=x$, 
$\varphi_x^i(x,t)\not=0$, $\widetilde{\gamma}^i(x,t)={\gamma}^i(\varphi^i(x,t),t)$ for every 
$(x,t)\in[0,1]\times[0,T)$.\\
If geometric uniqueness holds, any solution to the flow clearly describes a unique evolving network, seen as a subset of $\R^2$, for every time $t\in[0,T)$.
\end{defn}

One of the difficulties in getting existence and uniqueness of solutions in the sense of Definition~\ref{probdef} is the lack of the {\em maximum principle}, due to the presence of the $3$--points which behave as ``boundary'' points (whereas, by the Herring condition, from a ``distributional point of view'' they behave more like ``interior'' points). This means, in particular, that 
differently from  the case of the motion by curvature of a smooth curve (or more in general, for the mean curvature flow of a smooth hypersurface -- see~\cite{Manlib})
we do not have a (geometric) {\em comparison principle} for solutions, the usual tool to show the uniqueness of the flow. 
This is the reason why we will have to resort to {\em integral} a priori estimates, instead of pointwise ones (see Section~\ref{kestimates}),  the most ``natural'' ones in the smooth cases.

The ``natural'' initial regular networks are composed of curves of class $C^2$ and the ``natural'' regularity of their flow is $C^1$ in time and $C^2$ in space. Unfortunately, without additional requirements on the initial data, there is no hope of having a solution with curves in $C^{2,1}([0,1]\times[0,T))$. The problem is due to the way the evolving networks approach the initial one since they become immediately smooth (up to reparametrization) for every positive time, by a ``parabolic regularization'' effect (that we will discuss in Section~\ref{smtm2}) and satisfy some extra geometric properties which are stable under the $C^2$ convergence as $t\to 0$ (see Remark~\ref{ght1} and the related discussion in Section~\ref{wellposedHol}). Weakening such convergence at time zero of the flow, as we actually did in defining in great generality the flow of an initial regular network in Definition~\ref{probdef}, asking only for the continuity of the curves $\gamma$ as $t\to0$, could possibly result in the loss of uniqueness. We actually conjecture that uniqueness does not hold even if we ask for the continuity of the maps $\gamma_x$ (or of the unit tangent vectors to the curves) up to time zero.\\
In Section~\ref{smtm2}, by means of the results of this section, we will then show a quite satisfactory theorem of existence/geometric uniqueness for a short time of the flow of a regular $C^2$ initial network (Theorem~\ref{c2shorttime}) in a space of solutions which can be considered ``natural'' for the analytic/geometric peculiarities of the problem. It is well known that from a PDE's perspective, working directly with $C^2$ initial data and looking for solutions of class $C^1$ in time and $C^2$ in space is not a good choice, hence in this section we start showing existence and uniqueness in suitable Sobolev and H\"{o}lder spaces. Then, by means of these two results (the first mainly for the uniqueness, the second for the existence problem) and the estimates of the next section, we will show such Theorem~\ref{c2shorttime}. Indeed, roughly speaking, the space of flows $C^1$ in time and $C^2$ in space are in a way ``in the middle'' between the flows in Sobolev and H\"{o}lder spaces: if the initial datum of class only $C^2$, hence not necessarily in the H\"{o}lder space $C^{2+\alpha}$, either one uses the existence theorem in the Sobolev setting, or obtain a flow approximating such initial datum in $C^{2+\alpha}$. Then, in the first case, one obtains a Sobolev flow which could lack the property to be of class $C^{2,1}$, in the second case, because of the approximation procedure, one cannot use the uniqueness in the H\"{o}lder setting to conclude. Moreover, as we said, in the same Section~\ref{smtm2} we will also see that the ``classical'' property of parabolic equations of ``instantaneous regularization'' of the solutions for every positive time, also holds for the motion by curvature of networks. 

\medskip

The strategy of the proof is exactly the same for both the Sobolev and the H\"{o}lder case, so we briefly describe it below without specifying the spaces of the initial data and of the solutions, which we will simply denote by $\mathcal{I}$ and $\mathbb{E}_T$, respectively. Then, in the next sections, we will enter more into the details of both cases, in particular where they differ a little bit.

We will first prove existence and (standard) uniqueness for system~\eqref{problema-nogauge-general} in such spaces, giving the {\em special} curvature flow of an initial network, then we will show the existence and geometric uniqueness for the curvature flow Problem~\eqref{problema} in Definition~\ref{probdef} in the same spaces (``dropping'' the continuity requirement on the tangential velocity functions $\lambda^i$ and allowing initial networks less smooth that $C^2$, in the Sobolev setting). For simplicity, we will deal in detail with the case of the simplest possible network, a {\em triod}, and then we will explain how to adapt the arguments to the case of a general regular network.

\begin{defn}\label{triod}
A {\em triod} $\mathbb{T}=\bigcup_{i=1}^{3}\sigma{}^{i}([0,1])$ is a network 
composed of only three $C^1$ regular curves
$\sigma^{i}:[0,1]\to\overline{\Omega}$ where $\Omega$
is a smooth, convex, open subset of $\mathbb{R}^{2}$. 
These three curves intersect at a single $3$--point $O$
and have the other three end--points coinciding with three distinct
points $P^{i}=\sigma^{i}(1)\in\overline{\Omega}$.\\
A triod is {\em regular} if the unit tangents of the three curves
form angles of $120$ degrees at the $3$--point $O$.
\end{defn}

\begin{figure}[H]
\begin{center}
\begin{tikzpicture}
\draw[shift={(-2,0)}]
(-3.73,0) node[left]{$P^1$}
to[out= 50,in=180, looseness=1] (-2,0)
to[out= 60,in=180, looseness=1.5] (-0.45,1.55)
(-2,0)
to[out= -60,in=180, looseness=0.9] (-0.75,-1.75);
\draw[color=black,scale=1,domain=-3.141: 3.141,
smooth,variable=\t,shift={(-3.72,0)},rotate=0]plot({2.*sin(\t r)},
{2.*cos(\t r)});
\path[font=\large,shift={(-2,0)}]
(-3,0.8) node[below] {$\sigma^1$}
(-1.5,1) node[right] {$\sigma^3$}
(-0.8,-1)[left] node{$\sigma^2$}
(-2.2,0) node[below] {$O$}
(-0.21,1.35)node[above]{$\,\,\,\,\,\, P^3$}
(-0.55,-1.65) node[below] {$\,\,\,\, P^2$};
\end{tikzpicture}
\end{center}
\begin{caption}{A regular triod.}
\end{caption}
\end{figure}

For the reader's convenience, we state Problem~\eqref{problema} in the case of a triod (without the continuity requirement on the functions $\lambda^i$).

\begin{defn}
The one--parameter family of triods $\mathbb{T}=\left(\gamma^1,\gamma^2,\gamma^3\right)$ is
a flow by curvature in the time interval $\left[0,T\right]$ of the initial regular triod 
$\mathbb{T}_0=\left(\sigma^1,\sigma^2,\sigma^3\right)\in\mathcal{I}$ in a smooth convex, open set
$\Omega\subseteq\mathbb{R}^{2}$, if the three maps 
$\gamma^{i}\in\mathbb{E}_T$ satisfy the following system of conditions for every $x\in[0,1]$, $t\in [0,T]$, $i\in\{1,2,3\}$,
\begin{equation}\label{problematriodbis}
\begin{cases}
\gamma^i_t=k^i\nu^i+\lambda^i\tau^i\quad
&\text{ motion by curvature}\\
\gamma_x^i(x,t)\not=0
\quad &\text{ regularity}\\
\gamma^i(1,t)=P^i
\quad &\text{ fixed end--points condition}\\
\gamma^1(0,t)=\gamma^2(0,t)=\gamma^3(0,t)
\qquad &\text{ concurrency condition}\\
\sum_{i=1}^{3}\tau^i(0,t)=0
\quad &\text{ angles of $120$ degrees}
\end{cases}
\end{equation}
and there holds $\gamma^i(x,0)=\sigma^i(x)$ 
for every $x\in[0,1]$.
\end{defn}

Then, to show the existence of a solution of this problem, we consider system~\eqref{problema-nogauge-general} in the case of a triod, where we simply substitute $k^i\nu^i+\lambda^i\tau^i$ with $\frac{\gamma_{xx}^{i}}{\left|\gamma_x^i\right|^2}$, as the two velocities differ only by a tangential component. As we said in Remark~\ref{specevol}, this {\em a priori} choice of the tangential velocity makes the problem a system of {\em non--degenerate} quasilinear parabolic PDE's.

\begin{defn}[Special flow of triods]
The map $\gamma=(\gamma^1,\gamma^2,\gamma^3)$ is a solution of the {\em special flow} in $[0,T]$ with initial datum $\sigma=(\sigma^1,\sigma^2,\sigma^3)\in \mathcal{I}$
if it belongs to the space $\mathbb{E}_T$ and satisfies the following system, for every $x\in[0,1],t\in[0,T)$ and
$i\in\{1,2,3\}$
\begin{equation}\label{problema-nogauge}
\begin{cases}
\gamma^i_t(x,t)=\frac{\gamma_{xx}^{i}(x,t)}{|\gamma_{x}^{i}(x,t)|^{2}}
\qquad &\text{ special motion by curvature}\\
\gamma_x^i(x,t)\not=0
\qquad &\text{ regularity}\\
\gamma^i(1,t)=P^i
\qquad &\text{ fixed end--points condition}\\
\gamma^1(0,t)=\gamma^2(0,t)=\gamma^3(0,t)
\qquad &\text{ concurrency condition}\\
\sum_{i=1}^{3}\frac{\gamma_{x}^{i}(0,t)}{|\gamma_{x}^{i}(0,t)|}=0
\qquad &\text{ angles of $120$ degrees}\\
\gamma^i(x,0)=\sigma^i(x)
\qquad &\text{ initial data}
\end{cases}
\end{equation}
\end{defn}

Noticing that we can write the equations of motion as
\begin{equation}\label{linmotion}
\gamma^i_t-\frac{\gamma^i_{xx}}{\vert\sigma^i_x\vert^2}
=\bigg(\frac{1}{\vert\gamma^i_x\vert^2}-\frac{1}{\vert\sigma^i_x\vert^2}\bigg)\gamma^i_{xx} 
=\overline{f}^i[\gamma^i_{xx},\gamma^i_{x}]\,,
\end{equation}
for $i\in\{1,2,3\}$ and the angle condition at the triple junction as (here $\sigma^i_x=\sigma^i_x(0)$ and $\gamma^i_x=\gamma^i_x(0,t)$)
\begin{equation}\label{linboundary}
-\sum_{i=1}^3 \frac{\gamma^i_x}{\vert \sigma^i_x\vert}
-\frac{\sigma^i_x\langle \gamma^i_x\,|\,\sigma^i_x\rangle}{\vert \sigma^i_x\vert^3}=\sum_{i=1}^3 \bigg[\bigg(\frac{1}{\vert \gamma^i_x\vert}
 -\frac{1}{\vert\sigma^i_x\vert}\bigg)\gamma^i_x +
 \frac{\sigma^i_x\langle \gamma^i_x\,|\,\sigma^i_x\rangle}{\vert \sigma^i_x\vert^3}\bigg]
=\overline{b}[\gamma_x]\,,
\end{equation}
aiming at showing the existence and uniqueness of the solutions of system~\eqref{problema-nogauge}, we are led to deal with the following {\em linearization} of such system, with right-hand side data $(f,\eta,b,\psi)$ in suitable spaces:
\begin{equation}\label{linsys}
\begin{cases}
\gamma^i_{t}(x,t)-\frac{\gamma^i_{xx}(x,t)}{|\sigma^i_{x}(x)|^{2}}=f^i(x,t) &\qquad t\in[0,T),\,x\in[0,1],\,i\in\{1,2,3\} \\
\gamma^i(1,t)=\eta^i(t) &\qquad t\in[0,T],\,i\in\{1,2,3\} \\
\gamma^1(0,t)-\gamma^{2}(0,t)=0 &\qquad t\in[0,T] \\
\gamma^2(0,t)-\gamma^{3}(0,t)=0 &\qquad t\in[0,T] \\
-\sum_{i=1}^3 \Big(\frac{\gamma^i_x(0,t)}{\vert \sigma^i_x(0)\vert}
-\frac{\sigma^i_x(0)\langle\gamma^i_x(0,t)\,|\,\sigma^i_x(0)\rangle}{\vert \sigma^i_x(0)\vert^3}\Big)=b(t) &\qquad t\in[0,T] \\
\gamma^i(x,0)=\psi^i(x) &\qquad x\in[0,1],\,i\in\{1,2,3\}
\end{cases}
\end{equation}

Then, to apply Solonnikov's theory in~\cite{solonnikov1} (see also~\cite{eidelman2} and~\cite{lasolura}), precisely Theorem~5.4 for the Sobolev case and Theorem~4.9 for the H\"older case, respectively, we have to show that this system satisfies the so--called {\em complementary conditions} (see~\cite[Page~11]{solonnikov1} or~\cite[Chapter~I]{eidelman2} where they are also called {\em Lopatinskii--Shapiro condition}), which are a sort of ``algebraic'' relations between the evolution equation and the ``boundary'' constraints at the $3$--point and at the end--points of the triod (see~\cite[Section~3]{bronsard}). It is in general not so easy to show them, but in our case, the ones related only to the parabolic operator are almost immediate since it is uncoupled, while the remaining ones follow by applying the argument at pages~10--12, Lemma~I.1 in~\cite[Section~I.2]{eidelman2}. Indeed, for this particular system, by such argument, they hold if at the triple junction, for every $\lambda\in\mathbb{C}$ with $ \Re(\lambda)>0$, every solution $z=(z^1,z^2,z^3)\in C^2([0,+\infty),\mathbb{C}^3)$ of the second order ODE's system
\begin{equation}\label{LopatinskiiShapirosystem}
\begin{cases}
\lambda z^i(s)-\frac{\ddot{z}^i(s)}{\vert\sigma^i_x(0)
\vert^2}=0&\text{for every $s\in[0,+\infty)$ and $i\in\{1,2,3\}$}\\
z^{1}(0)=z^{2}(0)=z^{3}(0)&\\
\sum_{i=1}^3\Big(\frac{\dot{z}^i(0)}{|\sigma_{x}^{i}(0)|} -\frac{\sigma^i_x(0)\langle \dot{z}^i(0)\,\vert\,\sigma^i_x(0)\rangle}{\vert \sigma^i_x(0)\vert^3}\Big)
=0&
\end{cases}
\end{equation}
which satisfies $\lim_{s\to+\infty}\lvert z^i(s)\rvert=0$ is the trivial solution and similarly, at the end--points, every solution $z=(z^1,z^2,z^3)\in C^2([0,+\infty),\mathbb{C}^3)$ of
\begin{equation}\label{LopatinskiiShapiroinone}
\begin{cases}
\lambda z^i(s)-\frac{\ddot{z}^i(s)}{\vert\sigma^i_x(0) \vert^2}=0&\text{for every $s\in[0,+\infty)$ and $i\in\{1,2,3\}$}\\
z^{i}(0)=0&\text{for every $i\in\{1,2,3\}$}
\end{cases}
\end{equation}
which satisfies $\lim_{s\to+\infty}\lvert z^i(s)\rvert=0$ is the trivial solution.\\
These two conditions are clearly immediate to be checked, by directly writing the solutions to the above ODE's.\\
Then, holding such complementary conditions, by Solonnikov's theory, the linearized system has actually a unique solution for $(f,\eta,b,\psi)$ in suitable spaces if the initial datum $\psi\in\mathcal I$ satisfies some ``compatibility conditions'' which are different in the Sobolev and H\"older cases. We will discuss them precisely in the next sections.

Introducing the spaces
\begin{align*}
\widetilde{\mathbb{E}}_T=&\,\big\lbrace\gamma\in \mathbb{E}_T\,\,\big\vert\,\,\text{$\gamma^1(0,t)=\gamma^2(0,t)=\gamma^3(0,t)$, for $i\in\{1,2,3\},t\in[0,T]$}\,\big\rbrace\subseteq{\mathbb{E}}_T\\
\mathbb{F}_T=&\,\big\lbrace\text{$(f,\eta,b,\psi)$ in suitable spaces and $\psi\in \mathcal{I}$ satisfies the compatibility conditions}\big\rbrace
\end{align*}
the existence and uniqueness of solutions of system~\eqref{linsys} is then equivalent to the fact that the linear map $L_T:\widetilde{\mathbb{E}}_T\to \mathbb{F}_T$, defined as
\begin{equation}\label{Lop}
L_{T}(\gamma)=
\begin{pmatrix}
\gamma^i_t-\frac{\gamma^i_{xx}}{\vert\sigma^i_x\vert^2}\\
\gamma^i|_{x=1}\\
-\sum_{i=1}^3 \Big(\frac{\gamma^i_x}{\vert \sigma^i_x\vert}
-\frac{\sigma^i_x\langle\gamma^i_x\,|\,\sigma^i_x\rangle}{\vert \sigma^i_x\vert^3}\Big)\Big|_{x=0}\\
\gamma^i|_{t=0}
\end{pmatrix}_{i\in\{1,2,3\}}
\end{equation}
is a continuous isomorphism.

To ``get back'' to the solutions of the special flow system~\eqref{problema-nogauge}, we then need ``contraction'' estimates in order to apply a fixed point argument.\\
We define the space 
$$
\mathbb{E}^{\varphi,P}_T=\big\{\gamma\in 
\widetilde{\mathbb{E}}_T\,\,\,\big\vert\,\,\gamma\vert_{ t=0}=\varphi\;\text{and}\,
\gamma^i(1,t)=P^i,\;\text{for}\;i\in\{1,2,3\}\big\}
$$
and an operator 
$N_T:\mathbb{E}^{\varphi,P}_T \to \mathbb{F}_T$ that ``contains all the information"
about the non--linearity of our problem, given by 
\begin{equation}\label{NT}
N_T(\gamma)=\big(N^1_T(\gamma),\gamma|_{x=1},0,0,N^2_T(\gamma),\gamma|_{t=0}\,\big)
\end{equation}
where
\begin{align}
N^{1}_T(\gamma)^i=&\,\,\overline{f}^i[\gamma^i_{xx},\gamma^i_{x}]=\bigg(\frac{1}{|\gamma^i_{x}(x,t)|^{2}}-\frac{1}{|\sigma^i_x(x)|^{2}}\bigg)\gamma^i_{xx}(x,t),\label{NT1}\\
\intertext{for $i\in\{1,2,3\}$ and}
N^{2}_T(\gamma)=&\,\,\overline{b}[\gamma_x]=\sum_{i=1}^3 \bigg[\bigg(\frac{1}{\vert \gamma^i_x(0,t)\vert}
-\frac{1}{\vert\sigma^i_x(0)\vert}\bigg)\gamma^i_x(0,t) +
\frac{\sigma^i_x(0)\langle \gamma^i_x(0,t)\,|\,\sigma^i_x(0)\rangle}{\vert \sigma^i_x(0)\vert^3}\,\bigg]\label{NT2}
\end{align}
are the functions at the right hand sides of equations~\eqref{linmotion} and~\eqref{linboundary}, respectively.\\
We then introduce the operator $K_T:\mathbb{E}^{\varphi,P}_T\to \mathbb{E}^{\varphi,P}_T$
defined by $K_T(\gamma)=L_T^{-1}N_T(\gamma)$, where $L_T$ is the map above. Hence, $\gamma$ is a solution for system~\eqref{problema-nogauge} if and only if $\gamma\in \mathbb{E}^{\varphi,P}_T$ and 
\begin{equation*}
L_T(\gamma)=N_T(\gamma)\qquad \Longleftrightarrow
\qquad \gamma=L_T^{-1}N_T(\gamma)= K_T(\gamma)\,.
\end{equation*}
Thus, there exists a unique solution to system~\eqref{problema-nogauge} 
if and only if $K_T:\mathbb{E}^{\varphi,P}_T\to \mathbb{E}^{\varphi,P}_T$ has a unique fixed point and to get this, it is enough to show that $K_T$ is a contraction.

This clearly solves the existence problem of a curvature flow, Problem~\eqref{problematriodbis} in the space $\mathbb{E}_T$, when the initial data belongs to $\mathcal{I}$ (as we said, if the solution is not $C^2$ at least -- like it will happen in the Sobolev case -- we must ``drop'' the requirement that the ``tangential'' part of the velocity is continuous).\\
Finally,we will have to deal with the geometric uniqueness of the flow, that is, if $\TT_t$ and $\widetilde{\TT}_t$ are two solutions in such spaces, at every time one is a reparametrization of the other.
To conclude, we will extend all the results to the case of a general regular network.

The next two sections will be devoted to exhibiting the details of this strategy of proof in suitable Sobolev and H\"{o}lder spaces, respectively obtaining Theorems~\ref{wellposednessSobolev} and~\ref{2compexist0}.

\subsection{Well--posedness in Sobolev spaces}\label{wellposedSob}

We are going to show the existence and the geometric uniqueness of the solutions when the initial datum is a regular network in the fractional Sobolev space $W^{2-{{2}/{p}},p}$ (notice that here we are allowing non--$C^2$ initial regular networks).

\begin{defn}\label{geosolutionSob}
Let $p\in (3,+\infty)$. Given an initial, regular, $W^{2-{{2}/{p}},p}$ network $\SS_0$, composed of $n$ curves $\sigma^i:[0,1]\to\overline{\Omega}$, with $m$ triple junctions $O^1, O^2,\dots O^m\in\Omega$ and (if present) $l$ end--points $P^1, P^2,\dots, P^l\in\partial\Omega$ in a smooth convex, open set $\Omega\subseteq\mathbb{R}^{2}$, we say that a family of homeomorphic networks $\mathbb{S}_t$, described by the family of time--dependent curves $\gamma^i(\cdot, t)$, is a {\em Sobolev--solution} of the motion by curvature problem with fixed end--points for $\SS_0$, in the time interval $[0,T)$, if (with a little abuse of notation, switching the variables $t$ and $x$ inside $\gamma$)
$$
\gamma^i\in W^{1,p}([0,T);L^p([0,1];\overline{\Omega}))\cap 
L^p([0,T);W^{2,p}([0,1];\overline{\Omega}))\,,
$$
there hold $\gamma^i(x,0)=\sigma^i(x)$ (in the sense of traces), for every $x\in[0,1]$ and $i\in\{1,2,\dots, n\}$ (initial data) and the following system is (weakly) satisfied for every $x\in [0,1]$, $t\in[0,T)$, $i\in\{1,2,\dots, n\}$,
\begin{equation}\label{problemasob}
\begin{cases}
\begin{array}{lll}
\gamma^i_t=k^i\nu^i+\lambda^i\tau^i\quad&
&\text{ motion by curvature}\\
\gamma_x^i(x,t)\not=0
\quad &&\text{ regularity}\\
\gamma^r(1,t)=P^r\quad&\text{with}\, 0\leqslant r \leqslant l\,
\quad &\text{ fixed end--points condition}\\
\sum_{j=1}^3\tau^{pj}(O^p,t)=0\quad&\text{at every $3$--point $O^p$}
\quad &\text{ angles of $120$ degrees}
\end{array}
\end{cases}
\end{equation}
where we used the same notation of Definition~\ref{probdef}.
\end{defn}

The goal of this section is to prove the following theorem.

\begin{thm}\label{wellposednessSobolev}
Let $p\in (3,+\infty)$ and let $\SS_0$ be a regular initial network of class $W^{2-{{2}/{p}},p}$. Then, there exists a {\em geometrically} unique Sobolev--solution $\SS_t$ of the motion by curvature problem for $\SS_0$, as in the definition above, in a maximal time interval $[0,T)$.
\end{thm}

We let $p\in (3,+\infty)$ and we define the solutions space 
\begin{equation*}
\mathbb{E}_T=W_p^{1,2}([0,T)\times[0,1])=W^{1,p}([0,T);L^p([0,1]))\cap L^p([0,T);W^{2,p}([0,1]))
\end{equation*}
endowed with the norm $\lVert\cdot\rVert_{\mathbb{E}_T}=\lVert\cdot\rVert_{W_p^{1,2}([0,T)\times[0,1])}$.\\
To keep the notation simple, here and in the following we avoid writing the ``target'' spaces of the vector-valued functions, that is, for instance $W_p^{1,2}([0,T)\times[0,1]);\R^k)$ will be simply denoted with $W_p^{1,2}([0,T)\times[0,1])$, as the dimension of such target vector space is clear from the context.\\
The space $\mathbb{E}_T$ is then the intersection of two Sobolev spaces of functions with values in a Banach space.

Let $m\in\mathbb{N}$, $I\subseteq\mathbb{R}$ be an interval and $X$ be a Banach space. For $1\leqslant p\leqslant+\infty$, the Sobolev space of order $m\in\mathbb{N}$ is defined as 
\begin{equation*}
W^{m,p}(I;X)=\{f\in L^p(I;X)\,\,|\,\,\text{$\partial^k_x f\in L^p(I;X)$ for all $1\leqslant k\leqslant m$}\}\,,
\end{equation*}
which is a Banach space with the norm
\begin{equation}
\lVert f\rVert_{W^{m,p}(I;X)}=
\bigg(\sum_{0\leqslant k\leqslant m}^{}\lVert\partial^k_x f\rVert_{L^p(I;X)}^p\bigg)^{{{1}/{p}}}\,.\label{Sobolev Norm}
\end{equation}
Elements in the solutions space $\mathbb{E}_T$ are thus functions $f\in L^p([0,T); L^p([0,1]))$ that have one distributional derivative with respect to time $\partial_t f\in L^p([0,T);L^p([0,1]))$. Furthermore, for almost every $t\in[0,T)$, the function $f(t)$ lies in $W^{2,p}([0,1])$ and thus has two space derivatives $\partial_x f(t)$, $\partial_x ^2 f(t)\in L^p([0,1])$. One then easily sees that the functions $t\mapsto \partial_x^k f(t)$ belong to $L^p([0,T);L^p([0,1]))$, for $k\in\{1,2\}$.

The space $\mathcal{I}$ of initial data is the time--trace of $\mathbb{E}_T$, given by the {\em fractional Sobolev space} $W^{2-{{2}/{p}},p}([0,1])$. In general, if $d\in\mathbb{N}$, $p\in [1,+\infty)$ and $\theta\in[0,1]$ the {\em Gagliardo semi--norm} of an element $f\in L^p([0,1])$ is defined as
\begin{equation*}
[f]_{\theta,p}=\bigg(\int_{0}^{1}\int_{0}^{1}\frac{\lvert f(x)-f(y)\rvert^p}{\lvert x-y\rvert^{\theta p+1}}\,{d}x\,{d}y\bigg)^{{{1}/{p}}}\,,
\end{equation*}
then, if $s\in(0,+\infty)$ is not integer, the fractional Sobolev space $W^{s,p}([0,1])$ is given by
\begin{equation*}
W^{s,p}([0,1])=\big\{f\in W^{\lfloor s\rfloor,p}\big([0,1]\big)\,\,\big|\,\,\big[\partial_x^{\lfloor s\rfloor}f\big]_{s-\lfloor s\rfloor,p}<+\infty\big\}\,,
\end{equation*}
with the norm 
$$
\Vert f\Vert_{W^{s,p}([0,1])}=\Vert f\Vert_{W^{\lfloor s\rfloor,p}}+\big[\partial_x^{\lfloor s\rfloor}f\big]_{s-\lfloor s\rfloor,p}\,.
$$

For $p\in(3,+\infty)$ and $\alpha\in(0,1-{{3}/{p}}\,]$, the Sobolev embedding 
theorem~\cite[Theorem~4.6.1~(e)]{Triebel} implies
\begin{equation*}
W^{2-{{2}/{p}},p}([0,1])\hookrightarrow C^{1+\alpha}([0,1])\,,
\end{equation*}
thus, we have the continuous embeddings
\begin{equation}\label{regsob}
W_p^{1,2}([0,T)\times[0,1])\hookrightarrow C([0,T];W^{2-{{2}/{p}},p}([0,1]))\hookrightarrow C([0,T];C^{1+\alpha}([0,1]))\,.
\end{equation}
In particular, any initial network in $W^{2-{{2}/{p}},p}$ is of class $C^1$, hence the angle condition at every triple junction is pointwise well--defined (classical). Similarly, we specify the spaces of boundary values, as for $p\in [1,+\infty)$, the operators 
\begin{align*}
f\mapsto f(\cdot,0)\text{ and } f\mapsto f(\cdot,1) \qquad\qquad&\text{ from $W_p^{1,2}([0,T)\times[0,1])$ to $W^{1-1/2p,p}([0,T))$}\\
f\mapsto 
f_x(\cdot,0) \qquad\quad\qquad\qquad&\text{ from $W_p^{1,2}([0,T)\times[0,1])$ to $W^{1/2-1/2p,p}([0,T))$}
\end{align*}
are linear and continuous (Theorem~5.1 in~\cite{solonnikov1}).

\medskip

Now, to show Theorem~\ref{wellposednessSobolev}, we ``specialize'' the line of proof illustrated in the previous section to this Sobolev case, adding the missing details. As we said, we will deal with a triod and then we will explain how all the conclusions extend to general networks.

\subsubsection{Well--posedness of the linearized system~\eqref{linsys} and of the special flow~\eqref{problema-nogauge}}\label{wellposedlinsys}

The first point to be made precise is what are the ``compatibility conditions'' that the initial datum must satisfy so that the linearized system~\eqref{linsys} has a unique solution.

\begin{defn}[Linear compatibility conditions]\label{linearcompcond}
A function $\psi=(\psi^1,\psi^2,\psi^3)\in\mathcal{I}$ satisfies the {\em linear compatibility conditions} for system~\eqref{linsys} with respect to the functions $\eta=(\eta^1,\eta^2,\eta^3)$ and $b$ if, for $i,j\in\{1,2,3\}$, there holds $\psi^i(0)=\psi^j(0)$, $\psi^i(1)=\eta^i(0)$ and 
\begin{equation}\label{compcondtriod}
-\sum_{i=1}^3 \bigg(\frac{\psi^i_x(0)}{\vert \sigma^i_x(0)\vert}
-\frac{\sigma^i_x(0)\langle\psi^i_x(0)\,|\,\sigma^i_x(0)\rangle}{\vert \sigma^i_x(0)\vert^3}\bigg)=b(0)\,.
\end{equation}
\end{defn}

Then, the following proposition is a consequence of Theorem~5.4 in the book of Solonnikov~\cite{solonnikov1} (see also~\cite{lasolura} and~\cite{eidelman2}) keeping in mind that we know that system~\eqref{linsys} satisfies the complementary conditions.

\begin{prop}\label{exlintriod}
Let $p\in(3,+\infty)$. For every $T>0$, system~\eqref{linsys} has a unique solution $\gamma\in\mathbb{E}_T$ provided that $f\in L^p([0,T);L^p([0,1])$, $\eta\in W^{1-{{1}/{2p}},p}([0,T))$, $b\in W^{{{1}/{2}}-{{1}/{2p}},p}([0,T))$ and $\psi\in W^{2-{{2}/{p}},p}([0,1])$ fulfills the linear compatibility conditions stated in Definition~\ref{linearcompcond}, with respect to $\eta$ and $b$.\\
Moreover, there exists a constant $C=C(T)>0$ such that the following estimate holds:
\begin{equation}\label{estimate}
\Vert \gamma\Vert_{\mathbb{E}_T} \leqslant C( 
 \Vert f\Vert_{L^p([0,T);L^p([0,1]))}+\Vert\eta\Vert_{W^{1-{{1}/{2p}},p}([0,T))}+
\Vert b\Vert_{W^{{{1}/{2}}-{{1}/{2p}},p}([0,T))}
+ \Vert \psi \Vert_{W^{2-{{2}/{p}},p}([0,1])})\,.
\end{equation} 
\end{prop}

This proposition can be restated by saying that the linear operator 
$L_{T}:\widetilde{\mathbb{E}}_T\to \mathbb{F}_T$ defined as
\begin{equation}
L_{T}(\gamma)=
\begin{pmatrix}
\gamma^i_t-\frac{\gamma^i_{xx}}{\vert\sigma^i_x\vert^2}\\
\gamma^i|_{x=1}\\
-\sum_{i=1}^3 \Big(\frac{\gamma^i_x}{\vert \sigma^i_x\vert}
-\frac{\sigma^i_x\langle\gamma^i_x\,|\,\sigma^i_x\rangle}{\vert \sigma^i_x\vert^3}\Big)\Big|_{x=0}\\
\gamma^i|_{t=0}
\end{pmatrix}_{i\in\{1,2,3\}}
\end{equation}
is a continuous isomorphism between the spaces
\begin{align*}
\widetilde{\mathbb{E}}_T=&\,\big\lbrace\gamma=(\gamma^1,\gamma^2,\gamma^3)\in \mathbb{E}_T\,\,\big\vert\,\,\text{$\gamma^1(0,t)=\gamma^2(0,t)=\gamma^3(0,t)$, for $i\in\{1,2,3\}$ and $t\in[0,T)$}\,\big\rbrace\subseteq{\mathbb{E}}_T\\
\mathbb{F}_T=&\,
\begin{cases}\begin{rcases}
(f,\eta,b,\psi)\in L^p([0,T);L^p([0,1]))\times W^{1-{{1}/{2p}},p}([0,T))\times W^{{{1}/{2}}-{{1}/{2p}},p}([0,T))\times W^{2-{{2}/{p}},p}([0,1])\ \\[.5ex]
\,\text{$\psi$ satisfies the linear compatibility conditions of Definition~\ref{linearcompcond} with respect to $\eta$ and $b$}
\end{rcases}\end{cases}
\end{align*}
Moreover, it is possible to prove (Lemma~3.6 in~\cite{GoMePl}) that for every $T_0>0$, there exists a constant $C(T_0,p)$ such that
\begin{equation}\label{Lbound}
\sup_{T\in (0,T_0]}\vertiii{L_T^{-1}}_{\mathscr{L}(\mathbb{F}_T,\widetilde{\mathbb{E}}_T)}\leqslant C(T_0,p)\,.
\end{equation}

\medskip

As we said in the previous section, the well--posedness of the linearized system implies the same for the special flow, by means of contraction estimates involving the operator 
$N_T:\mathbb{E}^{\varphi,P}_T \to \mathbb{F}_T$, given by
\begin{equation}\label{NT0}
N_T(\gamma)=\big(N^1_T(\gamma),\gamma|_{x=1},0,0,N^2_T(\gamma),\gamma|_{t=0}\,\big)
\end{equation}
where $N^1_T$ and $N^2_T$ are defined by formulas~\eqref{NT1} and~\eqref{NT2}, respectively and
$$
\mathbb{E}^{\varphi,P}_T=\big\{\gamma\in 
\widetilde{\mathbb{E}}_T\,\,\,\big\vert\,\,\gamma\vert_{ t=0}=\varphi\;\text{and}\,
\gamma^i(1,t)=P^i,\;\text{for}\;i\in\{1,2,3\}\big\}\,.
$$
The following result is proved in~\cite[Theorem 3.7]{GoMePl}, it gives the existence and uniqueness for the special flow of a regular initial triod in the Sobolev setting.

\begin{thm}\label{short--time existence} 
Let $p\in(3,+\infty)$ and let $\sigma=(\sigma^1,\sigma^2,\sigma^3)\in W^{2-{{2}/{p}},p}([0,1])$ describes a regular triod. In particular,
$$
\mathcal{L}\sigma=L_1^{-1}(0,\sigma(1),0,\sigma)
$$
is well defined, as $\sigma$ satisfies the linear compatibility conditions in Definition~\ref{linearcompcond} with respect to the functions $t\mapsto\sigma(1)$ and zero.\\
Then, there exists a positive time $\widetilde{T}=\widetilde{T}(\sigma)$, depending on $\min_{i\in\{1,2,3\},\,x\in[0,1]}\vert\sigma^i_x(x)\vert$ and $\lVert\sigma\rVert_{W^{2-{{2}/{p}},p}([0,1])}$, such that for all $T\in (0,\widetilde{T})$, the system~\eqref{problema-nogauge} has a solution $\mathcal{E}\sigma$ in $\widetilde{\mathbb{E}}_T$ which is unique in
$$
\overline{B}_M=\{\gamma\in\widetilde{\mathbb{E}}_T\,\,|\,\,\Vert\gamma\Vert_{\mathbb{E}_T}\leqslant M\},
$$
with
\begin{equation*}
M=2\max\,\Big\{\sup_{T\in(0,1]}\vertiii{L_T^{-1}}_{\mathscr{L}(\mathbb{F}_T,
\widetilde{\mathbb{E}}_T)},1\Big\}\,
\max\,\big\{\Vert{\mathcal{L}\sigma}\Vert_{\mathbb{E}_1},\Vert(N_{1}^1(\mathcal{L}\sigma),\sigma(1),N_{1}^2(\mathcal{L}\sigma),\sigma)\Vert_{\mathbb{F}_1}\big\}\,.
\end{equation*}
\end{thm}

\subsubsection{Existence and geometric uniqueness}\label{exuniqSob}

Once we have obtained the existence and uniqueness of solutions to the special flow~\eqref{problema-nogauge}, we can come back to the geometric problem. The following theorem gives the ``existence part'' of Theorem~\ref{wellposednessSobolev}.

\begin{thm}\label{existencegeopro}
Let $p\in(3,+\infty)$ and $\mathbb{T}_0$ a regular initial triod parametrized by $\sigma=(\sigma^1,\sigma^2,\sigma^3)\in W^{2-{{2}/{p}},p}([0,1])$. 
Then, for some $T>0$, there exists a Sobolev--solution of the motion by curvature problem in Definition~\ref{geosolutionSob} with initial datum $\mathbb{T}_0$, in the time interval $[0,T)$.
\end{thm}
\begin{proof}
Proposition~\ref{short--time existence} implies that there exists $T>0$ and a solution
$\mathcal{E}\sigma\in W_p^{1,2}([0,T)\times[0,1])$ to the special flow system~\eqref{problema-nogauge} in $[0,T]$ with $\mathcal{E}\sigma(0)=\sigma$. Then, setting $\gamma(x,t)=\mathcal{E}\sigma(t)(x)$, we have that $\TT_t=\bigcup_{i=1}^3\gamma^i([0,1],t)$ is a Sobolev--solution to the motion by curvature with initial triod $\TT_0$ in $[0,T)$.
\end{proof}

Now we deal with the geometric uniqueness of the solution given by the previous theorem.

\begin{thm}\label{geouniquenesslocal}
Let $p\in(3,+\infty)$ and $\mathbb{T}_0$ a regular initial triod parametrized by $\sigma=(\sigma^1,\sigma^2,\sigma^3)\in W^{2-{{2}/{p}},p}([0,1])$. If $\TT_t$, $\widetilde{\TT}_t$ are two Sobolev--solutions to the motion by curvature problem in Definition~\ref{geosolutionSob} with initial datum $\mathbb{T}_0$, in the time intervals $[0,T)$ and $[0,\widetilde{T})$, respectively, then $\TT_t$ and $\widetilde{\TT}_t$ coincides up to reparametrization, for all $t\in [0, \min\{T,\widetilde{T}\})$. In particular, $\TT_t$ is geometrically unique.
\end{thm}
\begin{proof}
By Proposition~\ref{short--time existence}, we have a Sobolev--solution $\gamma=\mathcal{E}\sigma$ of system~\eqref{problema-nogauge} with initial datum $\sigma$, which is unique in $\overline{B}_M$, with $M$ as in such proposition. In particular, it gives a Sobolev--solution $\mathbb{T}_t$ to the motion by curvature in $[0,T)$ with initial datum $\mathbb{T}_0$. 

Suppose that there is another Sobolev--solution $\widetilde{\mathbb{T}}_t$ with initial datum $\mathbb{T}_0$ in $[0,\widetilde{T})$, parametrized by $\widetilde{\gamma}\in\mathbb{E}_{\widetilde{T}}$. We then want to show that there exists a family of time--dependent diffeomorphisms $\varphi^i(\cdot,t):[0,1]\to[0,1]$ with $t\in[0,\widehat{T})$ for some $\widehat{T}\leqslant\min\{T,\widetilde{T}\}$, such that $\varphi^i(\cdot,0)$ is the identity and the equality
\begin{equation*}
\widetilde{\gamma}^i(\varphi^i(x,t),t)=\gamma^i(x,t)
\end{equation*}
holds in the space $\mathbb{E}_{\widehat{T}}$, for every $i\in\{1,2,3\}$. In order to make use of the uniqueness conclusion in Proposition~\ref{short--time existence}, we construct the reparametrizations $\varphi=(\varphi^1,\varphi^2,\varphi^3)$ in such a way that the functions $(x,t)\mapsto \widetilde{\gamma}^i(\varphi^i(x,t),t)$ are a solution to the special flow in $\mathbb{E}_{\widehat{T}}$ with initial datum $\sigma$.\\ 
Then, formal differentiation shows that the reparametrizations $\varphi^i$ need to satisfy the following boundary value problem:
\begin{equation}\label{systemrepara}
\begin{cases}
\displaystyle{\varphi^i_t(x,t)\,=\frac{\varphi_{xx}^{i}(x,t)}{\left|\widetilde{\gamma}_{x}^{i}(\varphi^i(x,t),t)\right|^{2}\varphi^i_x(x,t)^2}
-\bigg\langle \widetilde{\gamma}_t^i(\varphi^i(x,t),t)-\frac{\widetilde{\gamma}_{xx}^i(\varphi^i(x,t),t)}{\vert\,\widetilde{\gamma}^i_x(\varphi^i(x,t),t)\vert^2}\,\bigg\vert\,\frac{\widetilde{\gamma}^i_x(\varphi^i(x,t),t)}{\vert\widetilde{\gamma}^i_x(\varphi^i(x,t),t)\vert^2}\bigg\rangle}\\[1em]
\varphi^i(0,t)\,=0\\
\varphi^i(1,t)\,=1\\
\varphi^i(x,0)=x
\end{cases}
\end{equation}
We observe that the right-hand side of the motion equation in system~\eqref{systemrepara} contains terms of the form $Q^i(\varphi^i(x,t),t)$. To remove this dependence it is convenient to consider the associated problem for the inverse diffeomorphisms $\xi=(\xi^1,\xi^2,\xi^3)$ given by $\xi^i(\cdot,t)=\varphi^i(\cdot,t)^{-1}$, for every fixed $t\in[0,\widehat{T})$. Indeed, suppose that $\varphi\in W^{1,2}_p([0, \widetilde{T})\times [0,1];[0,1]^3)$ is a solution of system~\eqref{systemrepara} with $\varphi^i(\cdot,t):[0,1]\to[0,1]$ a $C^1$--diffeomorphism, then it is easy to show that also $\xi$ is of class $W^{1,2}_p([0, \widetilde{T})\times [0,1];[0,1]^3)$ (and viceversa) and the formulas
\begin{align*}
\xi^i_y(y,t)&=\varphi^i_x(\xi^i(y,t),t)^{-1}\\
\xi^i_{yy}(y,t)&=-\xi^i_y(y,t)^3\varphi^i_{xx}(\xi^i(y,t),t)
\end{align*}
yield the evolution equation
\begin{align*}
\xi^i_t(y,t)=&-\varphi^i_t(\xi^i(y,t),t)\xi_y^i(y,t)\\
=&-\frac{\varphi_{xx}^{i}(\xi^i(y,t),t)}{\left|\widetilde{\gamma}_{x}^{i}(y,t)\right|^{2}}\,\xi^i_y(y,t)^3+\bigg\langle \widetilde{\gamma}_t^i(y,t)-\frac{\widetilde{\gamma}_{xx}^i(y,t)}{\vert\,\widetilde{\gamma}^i_x(y,t)\vert^2}\,\bigg\vert\,
\frac{\widetilde{\gamma}_x^i(y,t)}{\left|\widetilde{\gamma}_{x}^{i}(y,t)\right|^2}\bigg\rangle\,\xi^i_y(y,t)\\
=&\,\frac{\xi^i_{yy}(y,t)}{\left|\widetilde{\gamma}_{x}^{i}(y,t)\right|^{2}}
+\bigg\langle \widetilde{\gamma}_t^i(y,t)-\frac{\widetilde{\gamma}_{xx}^i(y,t)}{\vert\,\widetilde{\gamma}^i_x(y,t)\vert^2}\,\bigg\vert\,
\frac{\widetilde{\gamma}_x^i(y,t)}{\left|\widetilde{\gamma}_{x}^{i}(y,t)\right|^2}\bigg\rangle\,\xi^i_y(y,t)\,.
\end{align*}
Hence, we have the following {\em linear} system for $\xi$,
\begin{equation}\label{systemreparainverse}
\begin{cases}
\displaystyle{\xi^i_t(y,t)\,=\frac{\xi^i_{yy}(y,t)}{\left|\widetilde{\gamma}_{x}^{i}(y,t)\right|^{2}}
+\bigg\langle \widetilde{\gamma}_t^i(y,t)-\frac{\widetilde{\gamma}_{xx}^i(y,t)}{\vert\,\widetilde{\gamma}^i_x(y,t)\vert^2}\,\bigg\vert\,
\frac{\widetilde{\gamma}_x^i(y,t)}{\left|\widetilde{\gamma}_{x}^{i}(y,t)\right|^2}\bigg\rangle\,\xi^i_y(y,t)}\\[1em]
\xi^i(0,t)\,=0\\
\xi^i(1,t)\,=1\\
\xi^i(y,0)=y
\end{cases}
\end{equation}
for all $t\in [0,\widetilde{T})$, $y\in[0,1]$ and $i\in\{1,2,3\}$.\\

We observe that this linear boundary value problem has a very similar structure to the linearization of special flow system~\eqref{linsys}, with a perturbation in the evolution equation of lower order. Then, checking that it satisfies the complementary conditions is analogous and the compatibility conditions for the initial data are simply $\psi^i(0)=0$ and $\psi^i(1)=1$, which are clearly satisfied by $\xi^i(y,0)=y$. Hence, again by Solonnikov's theory (Theorem~5.4 in~\cite{solonnikov1}), we have a solution $\xi^i\in W^{1,2}_p([0,\widehat{T})\times [0,1])$, for some $\widehat{T}\leqslant\widetilde{T}$, such that for every $t\in[0,\widehat{T}]$ the map $\xi^i(\cdot,t):[0,1]\to[0,1]$ is a $C^1$--diffeomorphism. Then, the inverse functions $\varphi^i(\cdot,t)=\xi^i(\cdot,t)^{-1}$ also belong to $W^{1,2}_p([0,\widehat{T})\times [0,1])$ and solve system~\eqref{systemrepara}.
It is not difficult to show (see~\cite[Lemma~3.17]{GoMePl}) that the composition $(x,t)\mapsto\widetilde{\gamma}^i(\varphi^i(x,t),t)$ lies in $\mathbb{E}_{\widehat{T}}$ and by construction, it is a solution to the special flow system~\eqref{problema-nogauge} with initial datum $\sigma$. We may now choose a possibly smaller $\widehat{T}$ such that $(x,t)\mapsto\widetilde{\gamma}^i(\varphi^i(x,t),t)$ belongs to $\overline{B}_M$, hence it must coincide with $\gamma$ restricted to the time interval $[0,\widehat{T})$.\\
Let now $\overline{T}\leqslant \min\{T,\widetilde{T}\}$ be the infimum of the times in which $\widetilde{\TT}_t$ is not a reparametrization of $\TT_t$ and suppose $\overline{T}< \min\{T,\widetilde{T}\}$. Then, $\widetilde{\TT}_{\overline{T}}$ is obtained via a reparametrization $\varphi$ of $\TT_{\overline{T}}$ and if we consider the flow obtained reparametrizing all the networks $\TT_t$, for $t\geqslant \overline{T}$, with the same fixed ``static'' reparametrization $\varphi$, we obtain a Sobolev--solution with initial datum $\widetilde{\TT}_{\overline{T}}$ on some time interval $[\overline{T},\overline{T}+\delta)$. Then, by the previous discussion about uniqueness, it must coincide with the flow $\widetilde{\TT}_t$ for $t\in[\overline{T},\overline{T}+\delta')$, for some $\delta'>0$. This clearly shows that for $t\in[\overline{T},\overline{T}+\delta')$, all the networks $\widetilde{\TT}_t$ are reparametrizations of $\TT_t$, in contradiction with the infimum property of $\overline{T}$ and we are done.
\end{proof}

Putting together these two theorems, we obtain Theorem~\ref{wellposednessSobolev} in the special case of a triod.

\subsubsection{Extension to general regular networks}\label{strutture-complicate}

We explain here how to generalize the previous analysis for a triod to general networks. 

We consider an initial regular network $\SS_0$ composed of $n$ curves, with $l$ end--points $\gamma^k(t,1)=P^k\in\partial\Omega$, for $k\in\{1,\ldots,l\}$ and $m$ triple junctions $O^1, O^2,\dots O^m\in\Omega$. As in Section~\ref{netdef999} (recall the discussion just after Remark~\ref{rem2.5}), we will denote by $\sigma^{pj}$, for $j\in\{1,2,3\}$, the curves of this network concurring at $O^p$, for every $p\in\{1,\dots,m\}$.

The equations of motion for the special flow system~\eqref{problema-nogauge-general} for $\SS_0$ and its linearization do not differ from the version for a triod: formula~\eqref{linmotion} must hold for each curve $\gamma^i$ of the network,
\begin{equation}\label{deffm}
\gamma^i_{t}(x,t)-\frac{\gamma^i_{xx}(x,t)}{\left|\sigma^i_{x}(x)\right|^{2}}=
\bigg(\frac{1}{\left|\gamma^i_{x}(x,t)\right|^{2}}-\frac{1}{\left|\sigma^i_x(x)\right|^{2}}\bigg)
\gamma^i_{xx}(x,t) \,,
\end{equation}
for every $i\in\{1,\dots,n\}$ and we have formula~\eqref{linboundary} at {\em each} triple junction, that is, assuming that $O^p(t)=\gamma^{p1}(0,t)=\gamma^{p2}(0,t)=\gamma^{p3}(0,t)$ and $O^p(0)=\sigma^{p1}(0)=\sigma^{p2}(0)=\sigma^{p3}(0)$,
$$
\-\sum_{j=1}^3 \frac{\gamma^{pj}_x}{\vert \sigma^{pj}_x\vert}
-\frac{\sigma^{pj}_x\langle \gamma^{pj}_x\,|\,\sigma^{pj}_x\rangle}{\vert \sigma^{pj}_x\vert^3}=\sum_{j=1}^3 \bigg[\bigg(\frac{1}{\vert \gamma^{pj}_x\vert}
 -\frac{1}{\vert\sigma^{pj}_x\vert}\bigg)\gamma^{pj}_x +
 \frac{\sigma^{pj}_x\langle \gamma_x\,|\,\sigma^{pj}_x\rangle}{\vert \sigma^{pj}_x\vert^3}\bigg]\,,
$$
where $\sigma^{pj}_x=\sigma^{pj}_x(0)$ and $\gamma^{pj}_x=\gamma^{pj}_x(0,t)$, for every $p\in\{1,\dots,m\}$.

The analogous of the linearized system~\eqref{linsys} is then the following,
\begin{equation}\label{sysnew}
\begin{cases}
\gamma^i_{t}(x,t)-\frac{\gamma^i_{xx}(x,t)}{\left|\sigma^i_{x}(x)\right|^{2}}=f^i(x,t)&\qquad t\in[0,T),\, x\in[0,1],\,i\in\{1,\dots,n\}\\
\gamma^k(1,t)=\eta^k(t)&\qquad t\in[0,T],\,k\in\{1,\ldots,l\}\\
\gamma^{p1}(0,t)-\gamma^{p2}(0,t)=0&\qquad t\in[0,T],\,p\in\{1,\dots,m\}\\
\gamma^{p2}(0,t)-\gamma^{p3}(0,t)=0&\qquad t\in[0,T],\,p\in\{1,\dots,m\}\,\\
-\sum_{j=1}^3\Big(\frac{\gamma^{pj}_x(0,t)}{\vert \sigma^{pj}_x(0)\vert}
-\frac{\sigma^{pj}_x(0)\langle\gamma^{pj}_x(0,t)\,\vert\,\sigma^{pj}_x(0)\rangle}{\vert \sigma^{pj}_x(0)\vert^3}\Big)=b^p(t)&\qquad t\in[0,T],\,p\in\{1,\dots,m\}\,\\
\gamma^i(x,0)=\psi^i(x) &\qquad x\in[0,1],\,i\in\{1,\dots,n\}
\end{cases}
\end{equation}
for a general right hand side $(f,\eta,b,\psi)$, with $\eta=(\eta^1,\dots,\eta^l)$ and $b=(b^1,\dots,b^m)$.

Hence, in order to apply again Solonnikov's theory to get the well--posedness of this linearized system, the necessary complementary conditions are simply the same that we have seen for a single triple junction and only three end--points, repeated for each $3$--point and end--point in this case and we can check all of them exactly in the same way we did for a triod.

Then, the generalization of Definition~\ref{linearcompcond} is as follows, which is simply asking that equation~\eqref{compcondtriod} holds at every $3$--point.

\begin{defn}\label{linearcompcondgeneral}
Let $p\in (3,+\infty)$. A function $\psi=(\psi^1,\ldots,\psi^n)$ of class $W^{2-{{2}/{p}},p}([0,1])$
satisfies the {\em linear compatibility conditions} for system~\eqref{sysnew}, with respect to given functions $\eta=(\eta^1,\dots,\eta^l)\in W^{1-{{1}/{2p}},p}([0,T)$ and $b$ and $b=(b^1,\dots,b^m)\in W^{{{1}/{2}}-{{1}/{2p}},p}([0,T))$ if, for every $k\in\{1,\ldots,l\}$ and $p\in\{1,\ldots,m\}$, there holds $\psi^k(1)=\eta^k(0)$ $\psi^{p1}(0)=\psi^{p2}(0)=\psi^{p3}(0)$ and 
\begin{equation*}
-\sum_{j=1}^3\bigg(\frac{\psi^{{pj}}_x(0)}{\vert \sigma^{{pj}}_x(0)\vert}
-\frac{\sigma^{{pj}}_x(0)\left\langle\psi^{{pj}}_x(0)\,\vert\,\sigma^{{pj}}_x(0)\right\rangle}{\vert \sigma^{{pj}}_x(0)\vert^3}\bigg)=b^p(0)\,.
\end{equation*}
\end{defn}

The rest of the proof leading to Theorem~\ref{wellposednessSobolev} then follows analogously to the case of a triod, in particular the version of Theorem~\ref{short--time existence} for general initial regular networks. All this discussion concludes the proof of Theorem~\ref{wellposednessSobolev}.

\begin{rem} 
We mention that a different argument to extend the conclusions from the case of a triod to the one of a general network is to add some extra ``fake boundary points'' in the middle of every curve ``separating'' it in two new curves so that each curve of the resulting new family always connects one triple junction and one boundary point. Then, imposing ``artificial'' boundary conditions on such ``fake boundary points'' forbidding two of the new curves concurring there to form an angle, we have a new system which is ``equivalent'' to system~\eqref{sysnew} and easier (in terms of notation) to be dealt with. Applying Solonnikov's theory to such a system, one then gets the same conclusion that we obtained above. This line was pursued in~\cite{vB}, where the author carries on this procedure in full detail.
\end{rem}

\subsection{Well--posedness in H\"older spaces}\label{wellposedHol}

We want to show the existence and the geometric uniqueness of the flow, Problem~\eqref{problema} in Definition~\ref{probdef}, when all the curves of the initial regular network belong to the H\"older space $C^{2+2\alpha}$, with $\alpha\in(0,1/2)$ and satisfy some extra conditions. We underline that this section is based on the results of Bronsard and Reitich in~\cite{bronsard} (see also~\cite{mannovtor}).

We do not need a particular definition for these flows, that we are going to call {\em H\"older--solutions} or {\em H\"older--curvature flows} , similarly as we did with Definition~\ref{geosolutionSob} for the Sobolev case, since the initial data space $\mathcal{I}$ will be the H\"older space $C^{2+2\alpha}([0,1])$, which is a subspace of the ``natural'' space of initial $C^2$ regular networks. Omitting, as before, the target vector space for simplicity of notation, we have
$$
\mathcal{I}=C^{2+2\alpha}([0,1])
$$
and the solutions space,
$$
\mathbb{E}_T=C^{2+2\alpha,1+\alpha}([0,1]\times[0,T))\,,
$$
with $\alpha\in (0,1/2)$, endowed the norm $\lVert\cdot\rVert_{\mathbb{E}_T}=\lVert\cdot\rVert_{C^{2+2\alpha,1+\alpha}([0,1]\times[0,T))}$.\\ 
For the reader's convenience, we recall the definition and some properties of these parabolic H\"older spaces (see~\cite[Sections~11 and~13]{solonnikov1}). For a function $u:[0,1]\times[0,T]\to\mathbb{R}$, we define the H\"older semi--norms
$$
[u]_{\beta,0}=\sup_{x,y\in[0,1],\, t\in[0,T]}\frac{\vert u(x,t)-u(y,t)\vert}{\vert x-y\vert^\beta}\,,
$$
and
$$
[u]_{0,\theta}=\sup_{x\in[0,1]\, t,\tau\in[0,T]}\frac{\vert u(x,t)-u(x,\tau)\vert}{\vert t-\tau\vert^\theta}\,,
$$
then $C^{2+2\alpha,1+\alpha}([0,1]\times[0,T])$ is the space of the functions $u:[0,1]\times[0,T]\to\mathbb{R}$ having continuous derivatives $\partial_t^i\partial_x^ju$, for every $i,j\in\mathbb{N}$ with $2i+j\leqslant 2$ and such that the norm
\begin{equation*}
\left\lVert u\right\rVert_{C^{2+2\alpha,1+\alpha}([0,1]\times[0,T])}=\sum_{2i+j=0}^2\left\lVert\partial_t^i\partial_x^ju\right\rVert_\infty
+\sum_{2i+j=2}\left[\partial_t^i\partial_x^ju\right]_{2\alpha,0}+\sum_{2i+j=2}\left[\partial_t^i\partial_x^ju\right]_{0,\alpha}
\end{equation*}
is finite. 

As we did for the Sobolev case in the previous section, we now ``specialize'' the strategy of proof illustrated at the beginning to the H\"older case. Again, we first deal with a triod and then we extend all the results to general networks.

\subsubsection{Well--posedness of the linearized system~\eqref{linsys} and of the special flow~\eqref{problema-nogauge}}\label{wellposedlinsys2}

Differently from the Sobolev case, to get well--posedness of system~\eqref{problema-nogauge} in the above H\"older spaces, the initial datum cannot merely be a regular triod, but suitable ``extra conditions'' are necessary. 

\begin{defn}\label{compatibilitytriod}
We say that the {\em compatibility conditions of order $2$} for system~\eqref{problema-nogauge} are satisfied by the (initial) $C^2$ regular triod $\mathbb{T}_0=\bigcup_{i=1}^{3}\sigma^{i}\left([0,1]\right)$, if at the end--points and at the $3$--point, there hold all the relations on the space derivatives, up to second order, of the 
functions $\sigma^i$ given by the boundary conditions and their time derivatives, assuming that the evolution equation holds also at such points.\\
Explicitly, the compatibility conditions of order $0$ at the $3$--point are
\begin{equation*}
\sigma^i(0)=\sigma^j(0)\qquad\text{ for every $i,j\in\{1,2,3\}$}
\end{equation*}
and
\begin{equation*}
\sigma^i(1)=P^i\qquad\text{ for every $i\in\{1,2,3\}$,}
\end{equation*}
that is, simply the concurrency and fixed end--points conditions.\\
The compatibility condition of order $1$ is given by 
\begin{equation*}
\sum_{i=1}^3\frac{\sigma_x^i(0)}{|\sigma_x^i(0)|}=0\,,
\end{equation*}
that is, the 120 degrees condition at the $3$--point.\\
To get the second order conditions, one has to differentiate in time the first ones, getting
\begin{equation*}
\frac{\sigma_{xx}^i(0)}{|\sigma_x^i(0)|^2}=\frac{\sigma_{xx}^j(0)}{|\sigma_x^j(0)|^2}\;
\qquad\text{for every $i,j\in\{1,2,3\}$}
\end{equation*}
and
\begin{equation*}
\frac{\sigma_{xx}^i(1)}{|\sigma_x^i(1)|^2}=0\qquad \text{ for every $i\in\{1,2,3\}$}\,.
\end{equation*}
\end{defn}

As in the Sobolev case, we consider the linearized system~\eqref{linsys}, which also needs more conditions on the initial data in order to be well--posed.

\begin{defn}\label{linearcompcond-ordertwo}
A function $\psi=(\psi^1,\psi^2,\psi^3)\in\mathcal{I}$ satisfies the {\em linear compatibility conditions of order $2$} for system~\eqref{linsys} with respect to the functions $f=(f^1,f^2,f^3)$, $\eta=(\eta^1,\eta^2,\eta^3)$ and $b$, if $\psi$ satisfies the linear compatibility conditions as in Definition~\ref{linearcompcond} and, in addition,
$$
\frac{\psi^i_{xx}(0)}{|\sigma^i_{x}(0)|^{2}}+f^i(0,0)
=\frac{\psi^j_{xx}(0)}{|\sigma^j_{x}(0)|^{2}}+f^j(0,0)\qquad\text{ for every $i,j\in\{1,2,3\}$}
$$
and
$$
\frac{\psi^i_{xx}(1)}{|\sigma^i_{x}(1)|^{2}}+f^i(1,0)=\eta^i_t(0)\qquad\text{ for every $i\in\{1,2,3\}$}\,.
$$
\end{defn}

Then, the following proposition (analogous to Proposition~\ref{exlintriod}) is a consequence of Theorem~4.9 in the book of Solonnikov~\cite{solonnikov1} (see also~\cite{lasolura} and~\cite{eidelman2}), as we know that system~\eqref{linsys} satisfies the complementary conditions.

\begin{prop}\label{exlintriod-holder}
Let $\alpha\in(0,1/2)$. For every $T>0$, system~\eqref{linsys} has a unique solution $\gamma\in\mathbb{E}_T$
provided that $f\in C^{2\alpha,\alpha}([0,1]\times[0,T])$, $\eta\in C^{1+\alpha}([0,T])$, 
$b\in C^{1/2+\alpha}([0,T])$ and $\psi\in C^{2+2\alpha}([0,1])$ fulfills the linear compatibility conditions of order $2$ stated in Definition~\ref{linearcompcond-ordertwo}. Moreover, there exists a constant $C=C(T)>0$ such that the following estimate holds:
\begin{equation}\label{estimate2}
\Vert \gamma\Vert_{\mathbb{E}_T} \leqslant C\big( 
 \Vert f\Vert_{C^{2\alpha,\alpha}([0,1]\times[0,T])}+\Vert\eta\Vert_{C^{1+\alpha}([0,T])}+
\Vert b\Vert_{C^{1/2+\alpha}([0,T])}
+ \Vert \psi \Vert_{C^{2+2\alpha}([0,1])}\big)\,.
\end{equation} 
\end{prop}

Arguing as in the Sobolev case, by means of contraction estimates, the work of Bronsard and Reitich~\cite{bronsard} then shows the well--posedness of the special curvature flow system~\eqref{problema-nogauge} in the H\"older setting.

\begin{thm}\label{2smoothexist0-triod} 
For any initial, regular $C^{2+2\alpha}$ triod $\TTT_0=\bigcup_{i=1}^3\sigma^i([0,1])$, 
with $\alpha\in(0,1/2)$, satisfying the compatibility conditions of order $2$, there exists a positive time $T$ such that system~\eqref{problema-nogauge} has a unique solution in $C^{2+2\alpha,1+\alpha}([0,1]\times[0,T])$. Moreover, every triod $\TTT_t=\bigcup_{i=1}^3\gamma^i([0,1],t)$ satisfies the compatibility conditions of order $2$.
\end{thm}

\begin{rem} 
In~\cite{bronsard} the authors do not consider exactly system~\eqref{problema-nogauge}, 
but the analogous ``Neumann problem''. That is, they require that the end--points of the three curves meet the boundary
of $\Omega$ orthogonally.
\end{rem}

\subsubsection{Existence and geometric uniqueness}\label{3.2.2}

Clearly, a solution of system~\eqref{problema-nogauge} provides a H\"older--solution to Problem~\eqref{problematriodbis}.

\begin{thm}\label{2compexist0-triod} 
For any initial, regular $C^{2+2\alpha}$ triod $\TTT_0=\bigcup_{i=1}^3\sigma^i([0,1])$, with
$\alpha\in(0,1/2)$, in a smooth, convex, open set $\Omega\subseteq\R^2$, satisfying the 
compatibility conditions of order $2$, there exists a H\"older--curvature flow of $\TTT_0$ of class $C^{2+2\alpha,1+\alpha}([0,1]\times[0,T))$ in a maximal positive time interval $[0,T)$. Moreover, every triod 
$\TTT_t=\bigcup_{i=1}^3\gamma^i([0,1],t)$ satisfies the compatibility conditions of order $2$.
\end{thm}
\begin{proof}
If $\gamma^i\in C^{2+2\alpha,1+\alpha}([0,1]\times[0,T))$ is a solution of 
system~\eqref{problema-nogauge}, then it solves Problem~\eqref{problematriodbis} with
$$
\lambda^i(x,t)=\frac{\langle\gamma_{xx}^i(x,t)\,\vert\,\tau^i(x,t)\rangle}{\left|
\gamma_{x}^{i}\left(x,t\right)\right|^{2}}=
\frac{\langle\gamma_{xx}^i(x,t)\,\vert\,\gamma_x^i(x,t)\rangle}{\left|
\gamma_{x}^{i}\left(x,t\right)\right|^{3}}\,.
$$
Indeed, it follows immediately by the regularity properties of this flow 
that the relative functions $\lambda^i$ belong to the parabolic H\"older space 
$C^{2\alpha,\alpha}([0,1]\times[0,T))$ (hence, in $C^\alpha([0,1]\times[0,T))$, thus continuous) 
and all the triods $\TTT_t$ are in $C^{2+2\alpha}$, 
satisfying the compatibility conditions of order $2$.\\
The property that these evolving triods are regular follows by the standard fact that the maps 
$\gamma^i_x$ are continuous, belonging to $C^{1+2\alpha,1/2+\alpha}([0,1]\times[0,T])$ 
(see~\cite[Section~8.8]{krylov1}), hence, being $\sigma^i$ regular curves, 
$\gamma_x^i(x,t)\not=0$ still holds for every $x\in[0,1]$ 
and for some positive interval of time.\\
The fact that a curve cannot self--intersect or two curves cannot intersect each 
other can be ruled out by noticing that such an intersection cannot happen at the $3$--point 
by geometric reasons, as the curvature is locally bounded and the curves are regular, 
then it is well known for the motion by curvature that strong maximum principle 
prevents such intersections for the flow of two embedded curves 
(or two distinct parts of the same curve). 
A similar argument and again the strong maximum principle also prevent 
a curve from ``hitting'' the boundary of $\Omega$ at a point different 
from a fixed end--point of the triod.
\end{proof}

\begin{rem}\label{ght1}Since every curve $\gamma^i$ of a special curvature flow $\TT_t$ satisfies $\gamma^i_t=\frac{\gamma_{xx}^{i}}{|\gamma_{x}^{i}|^{2}}$ for every $t>0$, by the very Definition~\ref{compatibilitytriod}, every triod $\TT_t$ is $2$--compatible.\\ 
If instead we have simply a $C^{2,1}$ curvature flow $\TT_t$, 
it is not necessarily $2$--compatible for every time. 
It only has to satisfy $k\nu+\lambda\tau=0$ at every end--point and 
$$
\text{
$(k^{i}\nu^{i}+\lambda^{i}\tau^{i})(O)=(k^{j}\nu^{j}+\lambda^{j}\tau^{j})(O)$,
{ for } $i,j\in\{1,2,3\}$}\,.
$$
These relations imply anyway that for every evolving triod $\TT_t$ the curvature is zero 
at the end--points and the sum of the three curvatures at the $3$--point is zero. We are going to see that this implies that by reparametrizing $\TT_t$ by a $C^\infty$ map we obtain a 
2--compatible network.
\end{rem}

The observations in this remark can be clearly extended to general networks, as well as Definition~\ref{compatibilitytriod}.

\begin{defn}\label{2compcond}
We say that a regular $C^2$ network $\SS_0=\bigcup_{i=1}^{n}\sigma^{i}([0,1])$ 
is {\em 2--compatible} if the maps $\sigma^i$ satisfy the 
{\em compatibility conditions of order 2} for system~\eqref{problema-nogauge-general}, 
that is $\sigma_{xx}^i=0$ at every end--point and 
\begin{equation*}
\frac{\sigma_{xx}^{pi}(O^p)}{|\sigma_x^{pi}(O^p)|^2}
=\frac{\sigma_{xx}^{pj}(O^p)}{|\sigma_x^{pj}(O^p)|^2}
\end{equation*}
for every pair of curves $\sigma^{pi}$ and $\sigma^{pj}$ concurring 
at any $3$--point $O^p$ (where we abused a little the notation like in Definition~\ref{probdef}).
\end{defn}

\begin{defn}\label{geom-2-comp}
We say that a regular $C^2$ network $\SS_0=\bigcup_{i=1}^n\sigma^i([0,1])$ is 
{\em geometrically $2$--compatible} if the curvature is zero at 
every end--point and the sum of the three curvatures at every $3$--point is zero.
\end{defn}

By this definition, to be geometrically $2$--compatible is a property invariant 
by reparametrization of the curves of a network (it involves only the curvature,
a geometric quantity invariant under reparametrization). Arguing as in Remark~\ref{ght1}, we immediately have the following proposition.

\begin{prop}\label{c2geocomp} Given a curvature flow $\SS_t$ of an initial regular $C^2$ 
network $\SS_0=\bigcup_{i=1}^n\sigma^i([0,1])$ all the networks $\SS_t$, for $t>0$, 
are geometrically $2$--compatible.
\end{prop}

There is a clear relation between {\em geometrically} $2$--compatible and $2$--compatible networks that we give in the following lemma.

\begin{lem}\label{repar0} 
Let $\SS_0=\bigcup_{i=1}^n\sigma^i([0,1])$ be a geometrically $2$--compatible network. Then,
it admits a regular reparametrization by a $C^\infty$ map such that it becomes $2$--compatible.
\end{lem}
\begin{proof}
We look for some $C^\infty$ maps $\theta^i:[0,1]\to[0,1]$, with $\theta_x^i(x)\not=0$ for every 
$x\in [0,1]$ and $\theta^i(0)=0$, $\theta^i(1)=1$ such that the reparametrized curves 
$\widetilde{\sigma}^i=\sigma^i\comp\,\theta^i$ satisfy 
\begin{equation*}
\frac{\widetilde{\sigma}_{xx}^i}{|\widetilde{\sigma}_x^i|^2}
=\frac{\widetilde{\sigma}_{xx}^j}{|\widetilde{\sigma}_x^j|^2}
\end{equation*}
for every pair of concurring curves $\widetilde{\sigma}^i$ and $\widetilde{\sigma}^j$ at any 
$3$--point and ${\widetilde{\sigma}_{xx}^i}=0$ at every end--point of the network. Setting 
$\widetilde{\lambda}_0^i=\frac{\langle\widetilde{\sigma}_{xx}^i\,
\vert\widetilde{\sigma}^i_x\rangle}{|\widetilde{\sigma}_x^i|^3}$ this means 
$$
\widetilde{k}^i\widetilde{\nu}^i+\widetilde{\lambda}^i_0\widetilde{\tau}^i
=\widetilde{k}^j\widetilde{\nu}^j+\widetilde{\lambda}^j_0\widetilde{\tau}^j
$$
for every pair of concurring curves $\widetilde{\sigma}^i$ and $\widetilde{\sigma}^j$ at any 
$3$--point and $\widetilde{k}^i\widetilde{\nu}^i+\widetilde{\lambda}^i_0
\widetilde{\tau}^i=0$ at every end--point of the network. Since the curvature is invariant by 
reparametrization, using computations of Section~\ref{basiccomp} and the hypotheses 
on the curvature, these two conditions are satisfied if and only if 
$\widetilde{\lambda}^i_0=0$ at every end--point of the network and 
$$
\widetilde{\lambda}^i_0=\frac{k^{i-1}-k^{i+1}}{\sqrt{3}}
$$
at every $3$--point of the network, for $i\in\{1,2,3\}$ (modulus 3).\\
Hence, we only need to find $C^\infty$ reparametrizations $\theta^i$ such that at the borders 
of $[0,1]$ the values of $\widetilde{\lambda}_0^i=\frac{\langle\widetilde{\sigma}_{xx}^i\,
\vert\widetilde{\sigma}^i_x\rangle}{|\widetilde{\sigma}_x^i|^3}$ are given by these relations. 
This can be easily done since at the borders of the interval $[0,1]$ we have $\theta^i(0)=0$ 
and $\theta^i(1)=1$, hence 
$$
\widetilde{\lambda}_0^i=\frac{\langle\widetilde{\sigma}_{xx}^i\,
\vert\widetilde{\sigma}^i_x\rangle}{|\widetilde{\sigma}_x^i|^3}=-\partial_x\frac{1}
{\vert\widetilde{\sigma}_x^i\vert}
=-\partial_x\frac{1}{\vert\sigma_x^i\comp\,\theta^i\vert\theta_x^i}
=\frac{\langle{\sigma}_{xx}^i\,\vert{\sigma}^i_x\rangle}{|{\sigma}_x^i|^3}+\frac{\theta_{xx}^i}
{\vert\sigma_x^i\vert\vert\theta_x^i\vert^2}
=\lambda_0^i+\frac{\theta_{xx}^i}{\vert\sigma_x^i\vert\vert\theta_x^i\vert^2}
$$ 
where $\lambda_0^i=\frac{\langle{\sigma}_{xx}^i\,\vert{\sigma}^i_x\rangle}{|{\sigma}_x^i|^3}$, 
then we can simply choose any $C^\infty$ functions $\theta^i$ with 
$\theta_x^i(0)=\theta_x^i(1)=1$, $\theta_{xx}^i=-\lambda^i_0|\sigma_x^i||\theta_x^i|^2$ at 
every end--point and 
$$
\theta_{xx}^i=\left(\frac{k^{i-1}-k^{i+1}}{\sqrt{3}}-\lambda_0^i\right)\,
\vert\sigma_x^i\vert\vert\theta_x^i\vert^2
$$
at every $3$--point of the network (for instance, one can use a polynomial function). It follows 
that the reparametrized network $\widetilde{\SS}_0=\bigcup_{i=1}^n(\sigma^i\comp\,\theta^i)
([0,1])$ is $2$--compatible.
\end{proof}

We are then ready to deal with networks with general topological structure, having as a goal the following final conclusion.

\begin{thm}\label{2compexist0} 
For any initial, regular $C^{2+2\alpha}$ network $\SS_0=\bigcup_{i=1}^n\sigma^i([0,1])$, with $\alpha\in(0,1/2)$, in a smooth, convex, open set $\Omega\subseteq\R^2$, which is geometrically $2$--compatible, there exists a geometrically unique H\"older-- $C^{2+2\alpha,1+\alpha}([0,1]\times[0,T))$ curvature flow $\SS_t$ (in the sense of Definition~\ref{uniqdef}) in $C^{2+2\alpha,1+\alpha}([0,1]\times[0,T))$, in a maximal time interval $[0,T)$. Moreover, all the networks $\SS_t$ are geometrically $2$--compatible.
\end{thm}

We first extend the short--time existence Theorem~\ref{2compexist0-triod} to regular, $C^{2+2\alpha}$ initial networks which are geometrically $2$--compatible, hence showing the ``existence part'' of Theorem~\ref{2compexist0}.

\begin{prop}\label{2geocomp} 
For any initial regular $C^{2+2\alpha}$ network $\SS_0=\bigcup_{i=1}^n\sigma^i([0,1])$ 
which is geometrically $2$--compatible, with $\alpha\in(0,1/2)$, in a smooth, convex, open set $\Omega\subseteq\R^2$, there exists a H\"older--curvature flow of class $C^{2+2\alpha,1+\alpha}([0,1]\times[0,T))$ for a maximal positive time interval $[0,T)$.
\end{prop}
\begin{proof}
By Lemma~\ref{repar0}, we can reparametrize the network $\SS_0$ with some $C^\infty$ maps 
$\theta^i$ to make it $2$--compatible. If the network $\SS_0$ belongs to $C^{2+2\alpha}$ the 
reparametrized one $\widetilde{\SS}_0$ is still in $C^{2+2\alpha}$, then we can argue step--by--step exactly as we did in Section~\ref{strutture-complicate} for the Sobolev setting, in order to extend Theorem~\ref{2smoothexist0-triod} to general regular networks, getting the unique special curvature flow $\widetilde{\gamma}^i$ for $\widetilde{\SS}_0=\bigcup_{i=1}^n\widetilde{\sigma}^i([0,1])=\bigcup_{i=1}^n(\sigma^i\comp\,\theta^i)([0,1])$ which is in $C^{2+2\alpha,1+\alpha}
([0,1]\times[0,T))$ for a maximal positive time interval $[0,T)$. Moreover, every network 
$\SS_t=\bigcup_{i=1}^n\gamma^i([0,1],t)$ is $2$--compatible.

If now we consider the maps $\gamma^i$ given by 
$\gamma^i(x,t)=\widetilde{\gamma}^i([\theta^i]^{-1}(x),t)$, we have that they still belong to 
$C^{2+2\alpha,1+\alpha}([0,1]\times[0,T))$ (as the maps $[\theta^i]^{-1}$ are in $C^\infty$), 
$\gamma^i(\cdot,0)=\sigma^i$ and 
\begin{align*}
\gamma^i_t(x,t)=&\,\partial_t[\widetilde{\gamma}^i([\theta^i]^{-1}(x),t)]\\
=&\,\widetilde{\gamma}^i_t([\theta^i]^{-1}(x),t)\\
=&\,\underline{\widetilde{k}}^i([\theta^i]^{-1}(x),t)
+{\widetilde{\lambda}}^i([\theta^i]^{-1}(x),t)\widetilde{\tau}^i([\theta^i]^{-1}(x),t)\\
=&\,\underline{{k}}^i(x,t)
+\underline{\lambda}^i(x,t)\,,
\end{align*}
with $\underline{\lambda}^i(x,t)={\widetilde{\lambda}}^i([\theta^i]^{-1}
(x),t)\widetilde{\tau}^i([\theta^i]^{-1}(x),t)$. Hence, $\gamma^i$ is a flow by curvature of the 
network ${\SS}_0$ in $C^{2+2\alpha,1+\alpha}([0,1]\times[0,T))$
\end{proof}

Finally, we address the geometric uniqueness of the flow in H\"{o}lder space, obtaining Theorem~\ref{2compexist0}. 

\begin{proof}[Proof of Theorem~\ref{2compexist0}] 
By Proposition~\ref{2geocomp}, we have a H\"{o}lder--curvature flow $\SS_t$ of $\SS_0$, given by the family of moving curves $\gamma^i$. We first show that if $\SS_0=\bigcup_{i=1}^n\sigma^i([0,1])$ satisfies the compatibility conditions of order $2$ then the solution given by Theorem~\ref{2compexist0-triod} (which is the special flow given by the extension of Theorem~\ref{2smoothexist0-triod}, as in the proof of the previous proposition) is geometrically unique among the curvature flows in the class $C^{2+2\alpha,1+\alpha}([0,1]\times[0,T))$. 

Suppose that $\widetilde{\gamma}^i:[0,1]\times[0,\widetilde{T})\to\overline{\Omega}$ is
another maximal solution in $C^{2+2\alpha,1+\alpha}([0,1]\times[0,\widetilde{T}))$ satisfying
$\widetilde{\gamma}^i_t=\widetilde{k}^i\widetilde{\nu}^i+\widetilde{\lambda}^i\widetilde{\tau}^i$ 
for some functions $\widetilde{\lambda}^i$ in $C^{2\alpha}([0,1]\times[0,\widetilde{T}))$,
we want to see that it coincides with $\gamma^i$ up to a
reparametrization of the curves $\widetilde{\gamma}^i(\cdot,t)$ for
every $t\in[0,\min\{T,\widetilde{T}\})$.\\
If we consider functions $\varphi^i:[0,1]\times[0,\min\{T,\widetilde{T}\})\to[0,1]$ 
belonging to $C^{2+2\alpha,1+\alpha}([0,1]\times[0,\min\{T,\widetilde{T}\}))$ and the
reparametrizations $\overline{\gamma}^i(x,t)=\widetilde{\gamma}^i(\varphi^i(x,t),t)$, 
we have that $\overline{\gamma}^i\in C^{2+2\alpha,1+\alpha}([0,1]\times[0,\min\{T,\widetilde{T}\}))$ and 
\begin{align*}
\overline{\gamma}^i_t(x,t)=&\,\partial_t[\widetilde{\gamma}^i(\varphi^i(x,t),t)]\\
=&\,\widetilde{\gamma}^i_x(\varphi^i(x,t),t)\varphi^i_t(x,t)
+\widetilde{\gamma}^i_t(\varphi^i(x,t),t)\\
=&\,\widetilde{\gamma}^i_x(\varphi^i(x,t),t)\varphi^i_t(x,t)
+\underline{\widetilde{k}}^i(\varphi^i(x,t),t)
+\underline{\widetilde{\lambda}}^i(\varphi^i(x,t),t)\\
=&\,\widetilde{\gamma}^i_x(\varphi^i(x,t),t)\varphi^i_t(x,t)
+\frac{\left\langle\widetilde{\gamma}_{xx}^{i}\left(\varphi^i(x,t),t\right)\,\vert\,
\widetilde{\nu}^i(\varphi^i(x,t),t)\right\rangle}
{\left|\widetilde{\gamma}_{x}^{i}\left(\varphi^i(x,t),t\right)\right|
^{2}}\widetilde{\nu}^i(\varphi^i(x,t),t)\\
&\,+\widetilde{\lambda}^i(\varphi^i(x,t),t)
\frac{\widetilde{\gamma}_x^i(\varphi^i(x,t),t)}{\left|
\widetilde{\gamma}_{x}^{i}\left(\varphi^i(x,t),t\right)\right|}\,.
\end{align*}
We choose now maps $\varphi^i\in C^{2+2\alpha,1+\alpha}([0,1]\times[0,\widehat{T}))$ which are solutions for some positive interval of time $[0,\widehat{T})$ of the following quasilinear PDE's
\begin{equation}\label{reparphi}
\varphi^i_t(x,t)=\frac{\left\langle\widetilde{\gamma}_{xx}^{i}\left(\varphi^i(x,t),t\right)\,\vert\,
\widetilde{\gamma}_x^i(\varphi^i(x,t),t)\right\rangle}
{\left|\widetilde{\gamma}_{x}^{i}\left(\varphi^i(x,t),t\right)\right|^{4}}
-\frac{\widetilde{\lambda}^i(\varphi^i(x,t),t)}{\left|
\widetilde{\gamma}_{x}^{i}\left(\varphi^i(x,t),t\right)\right|}
+\frac{\varphi^i_{xx}(x,t)}{\left|\widetilde{\gamma}_{x}^{i}\left(\varphi^i(x,t),t\right)\right|^2
\left|\varphi^i_x(x,t)\right|^2}
\end{equation}
with $\varphi^i(0,t)=0$, $\varphi^i(1,t)=1$ and $\varphi^i(x,0)=x$ (hence, 
$\overline{\gamma}^i(x,0)=\gamma^i(x,0)=\sigma^i(x)$).\\ 
To find such reparametrizations $\varphi$, we consider, as in Section~\ref{exuniqSob}, the associated problem for the inverse diffeomorphisms $\xi=(\xi^1,\xi^2,\xi^3)$ given by $\xi^i(\cdot,t)=\varphi^i(\cdot,t)^{-1}$, for every fixed $t\in[0,\widehat{T})$.
\begin{equation}\label{systemreparainverseH}
\begin{cases}
\displaystyle{\xi^i_t(y,t)\,=\frac{\xi^i_{yy}(y,t)}{\left|\widetilde{\gamma}_{x}^{i}(y,t)\right|^{2}}
+\bigg\langle \widetilde{\gamma}_t^i(y,t)-\frac{\widetilde{\gamma}_{xx}^i(y,t)}{\vert\,\widetilde{\gamma}^i_x(y,t)\vert^2}\,\bigg\vert\,
\frac{\widetilde{\gamma}_x^i(y,t)}{\left|\widetilde{\gamma}_{x}^{i}(y,t)\right|^2}\bigg\rangle\,\xi^i_y(y,t)}\\[1em]
\xi^i(0,t)\,=0\\
\xi^i(1,t)\,=1\\
\xi^i(y,0)=y
\end{cases}
\end{equation}
for all $t\in [0,\widetilde{T})$, $y\in[0,1]$ and $i\in\{1,2,3\}$.\\
We already know, from Section~\ref{exuniqSob}, that this linear system satisfies the complementary conditions, hence for the existence of a solution $\xi\in C^{2+2\alpha,1+\alpha}([0,1]\times[0,\min\{T,\widetilde{T}\}))$, we only have to check that the compatibility conditions of order $2$ (as in Definition~\ref{linearcompcond-ordertwo}) for the initial data holds. By simplicity, we show it for a triod: in such case, they are $\psi^i(0)=0$ and $\psi^i(1)=1$, which are clearly satisfied by $\xi^i(y,0)=y$ and 
\begin{gather}
\frac{\psi^i_{yy}(0)}{\left|\widetilde{\gamma}_{x}^{i}(0,0)\right|^{2}}
+\bigg\langle \widetilde{\gamma}_t^i(0,0)-\frac{\widetilde{\gamma}_{xx}^i(0,0)}{\vert\,\widetilde{\gamma}^i_x(0,0)\vert^2}\,\bigg\vert\,\frac{\widetilde{\gamma}_x^i(0,0)}{\left|\widetilde{\gamma}_{x}^{i}(0,0)\right|^2}\bigg\rangle\,\psi^i_y(0)=0\\
\frac{\psi^i_{yy}(0)}{\left|\widetilde{\gamma}_{x}^{i}(1,0)\right|^{2}}
+\bigg\langle \widetilde{\gamma}_t^i(1,0)-\frac{\widetilde{\gamma}_{xx}^i(1,0)}{\vert\,\widetilde{\gamma}^i_x(1,0)\vert^2}\,\bigg\vert\,\frac{\widetilde{\gamma}_x^i(1,0)}{\left|\widetilde{\gamma}_{x}^{i}(1,0)\right|^2}\bigg\rangle\,\psi^i_y(0)=0
\end{gather}
where, putting $\psi^i(y)=\xi^i(y,0)=y$, we get the equations
\begin{equation}\label{eqcar2323}
\bigg\langle \widetilde{\gamma}_t^i(0,0)-\frac{\widetilde{\gamma}_{xx}^i(0,0)}{\vert\,\widetilde{\gamma}^i_x(0,0)\vert^2}\,\bigg\vert\,\frac{\widetilde{\gamma}_x^i(0,0)}{\left|\widetilde{\gamma}_{x}^{i}(0,0)\right|^2}\bigg\rangle\,=
\bigg\langle \widetilde{\gamma}_t^i(0,0)-\frac{\sigma_{xx}^i(0)}{\vert\,\sigma^i_x(0)\vert^2}\,\bigg\vert\,\frac{\sigma_x^i(0)}{\left|\sigma_{x}^{i}(0)\right|^2}\bigg\rangle\,=\,0
\end{equation}
$$
\bigg\langle \widetilde{\gamma}_t^i(1,0)-\frac{\widetilde{\gamma}_{xx}^i(1,0)}{\vert\,\widetilde{\gamma}^i_x(1,0)\vert^2}\,\bigg\vert\,\frac{\widetilde{\gamma}_x^i(1,0)}{\left|\widetilde{\gamma}_{x}^{i}(1,0)\right|^2}\bigg\rangle\,=\,-\,\bigg\langle \,\frac{\sigma_{xx}^i(1)}{\vert\,\sigma^i_x(1)\vert^2}\,\bigg\vert\,\frac{\sigma_x^i(1)}{\left|\sigma_{x}^{i}(1)\right|^2}\bigg\rangle\,=\,0\,.
$$
Since $\sigma$ satisfies the compatibility conditions of order $2$ for system~\eqref{problema-nogauge}, we have (Definition~\ref{compatibilitytriod})
\begin{equation*}
\frac{\sigma_{xx}^i(0)}{|\sigma_x^i(0)|^2}=\frac{\sigma_{xx}^j(0)}{|\sigma_x^j(0)|^2}\;
\qquad\text{ and }\qquad
\frac{\sigma_{xx}^i(1)}{|\sigma_x^i(1)|^2}=0\,,
\end{equation*}
for every $i,j\in\{1,2,3\}$, hence the second equation above is immediately verified and the vector $v=\widetilde{\gamma}_t^i(0,0)-\frac{\sigma_{xx}^i(0)}{\vert\,\sigma^i_x(0)\vert^2}$ is independent of $i\in\{1,2,3\}$. It follows,
$$
\langle v\,\vert\,\widetilde{\nu}^i(0,0)\rangle=\bigg\langle \widetilde{\gamma}_t^i(0,0)-\frac{\sigma_{xx}^i(0)}{\vert\,\sigma^i_x(0)\vert^2}\,\bigg\vert\,\widetilde{\nu}^i(0,0)\bigg\rangle=\widetilde{k}^i(0,0)-\bigg\langle\frac{\sigma_{xx}^i(0)}{\vert\,\sigma^i_x(0)\vert^2}\,\bigg\vert\,\widetilde{\nu}^i(0,0)\bigg\rangle=0
$$
for every $i\in\{1,2,3\}$, which implies $v=0$, thus equation~\eqref{eqcar2323} is also satisfied. In the case of a general network, the above argument must simply be repeated for every triple junction and every end--point (by means of Definition~\ref{2compcond}).\\
Then, again by Solonnikov's theory (Theorem~4.9 in~\cite{solonnikov1}), we have a solution $\xi\in C^{2+2\alpha,1+\alpha}([0,1]\times[0,\min\{T,\widetilde{T}\}))$, for some $\widehat{T}\leqslant\widetilde{T}$, such that for every $t\in[0,\widehat{T}]$ the map $\xi^i(\cdot,t):[0,1]\to[0,1]$ is a $C^1$--diffeomorphism. Hence, the inverse functions $\varphi^i(\cdot,t)=\xi^i(\cdot,t)^{-1}$ also belong to $C^{2+2\alpha,1+\alpha}([0,1]\times[0,\min\{T,\widetilde{T}\}))$ and are solutions of system~\eqref{reparphi}. Moreover, by arguing as in the last part of the proof of Theorem~\ref{geouniquenesslocal}, we can show that $\widehat{T}$ can be taken equal to $\min\{T,\widetilde{T}\}$.\\
It follows that the reparametrizations $\overline{\gamma}^i(x,t)=\widetilde{\gamma}^i(\varphi^i(x,t),t)$ satisfy the special flow system~\eqref{problema-nogauge}:
\begin{align*}
\overline{\gamma}^i_t(x,t)=
&\,\frac{\left\langle\widetilde{\gamma}_{xx}^{i}\left(\varphi^i(x,t),t\right)\,\vert\,
\widetilde{\gamma}_x^i(\varphi^i(x,t),t)\right\rangle}
{\left|\widetilde{\gamma}_{x}^{i}\left(\varphi^i(x,t),t\right)\right|
^{4}}\widetilde{\gamma}_x^i(\varphi^i(x,t),t)
+\frac{\varphi^i_{xx}(x,t)\widetilde{\gamma}_{x}^{i}\left(\varphi^i(x,t),t\right)}
{\left|\widetilde{\gamma}_{x}^{i}\left(\varphi^i(x,t),t\right)\right|^2
\left|\varphi^i_x(x,t)\right|^2}\\
&\,+\frac{\left\langle\widetilde{\gamma}_{xx}^{i}\left(\varphi^i(x,t),t\right)\,\vert\,
\widetilde{\nu}^i(\varphi^i(x,t),t)\right\rangle}
{\left|\widetilde{\gamma}_{x}^{i}\left(\varphi^i(x,t),t\right)\right|
^{2}}\widetilde{\nu}^i(\varphi^i(x,t),t)\\
=&\,\frac{\left\langle\widetilde{\gamma}_{xx}^{i}\left(\varphi^i(x,t),t\right)\,\vert\,
\widetilde{\tau}^i(\varphi^i(x,t),t)\right\rangle}
{\left|\widetilde{\gamma}_{x}^{i}\left(\varphi^i(x,t),t\right)\right|
^{2}}\widetilde{\tau}^i(\varphi^i(x,t),t)
+\frac{\varphi^i_{xx}(x,t)\widetilde{\gamma}_{x}^{i}\left(\varphi^i(x,t),t\right)}
{\left|\widetilde{\gamma}_{x}^{i}\left(\varphi^i(x,t),t\right)\right|^2
\left|\varphi^i_x(x,t)\right|^2}\\
&\,+\frac{\left\langle\widetilde{\gamma}_{xx}^{i}\left(\varphi^i(x,t),t\right)\,\vert\,
\widetilde{\nu}^i(\varphi^i(x,t),t)\right\rangle}
{\left|\widetilde{\gamma}_{x}^{i}\left(\varphi^i(x,t),t\right)\right|
^{2}}\widetilde{\nu}^i(\varphi^i(x,t),t)\\
=&\,\frac{\widetilde{\gamma}_{xx}^{i}\left(\varphi^i(x,t),t\right)}
{\left|\widetilde{\gamma}_{x}^{i}\left(\varphi^i(x,t),t\right)\right|^{2}}
+\frac{\varphi^i_{xx}(x,t)\widetilde{\gamma}_{x}^{i}\left(\varphi^i(x,t),t\right)}
{\left|\widetilde{\gamma}_{x}^{i}\left(\varphi^i(x,t),t\right)\right|^2
\left|\varphi^i_x(x,t)\right|^2}\\
=&\,\frac{\overline{\gamma}_{xx}^{i}(x,t)}
{\vert\overline{\gamma}_{x}^{i}(x,t)|^{2}}\,.
\end{align*}
We can then conclude that by the uniqueness part of (the extension to general networks of) Theorem~\ref{2smoothexist0-triod} that $\overline{\gamma}^i=\gamma^i$ for every $i\in\{1,2,\dots, n\}$, hence ${\gamma}^i(x,t)=\widetilde{\gamma}^i(\varphi^i(x,t),t)$ in the time interval $[0,\min\{T,\widetilde{T}\})$ and since this ``reparametrization relation'' between any two maximal solutions of Problem~\eqref{problema} is symmetric (by means of the maps $\xi^i$), we have $\widetilde{T}=T$ and we are done.

Assume now that the network $\SS_0$ is only geometrically $2$--compatible, then the proof of Proposition~\ref{2geocomp} gives a special solution $\gamma^i$ given by $\gamma^i(x,t)=\widetilde{\gamma}^i([\theta^i]^{-1}(x),t)$ 
where $\theta^i$ are smooth maps and $\widetilde{\gamma}^i$ is a special solution as above for the $2$--compatible network 
$\widetilde{\SS}_0=\bigcup_{i=1}^n\widetilde{\sigma}^i([0,1])=\bigcup_{i=1}^n(\sigma^i\comp\,\theta^i)([0,1])$ which is in 
$C^{2+2\alpha,1+\alpha}([0,1]\times[0,T))$ for a maximal positive time interval $[0,T)$.

Suppose that $\overline{\gamma}^i:[0,1]\times[0,\widetilde{T})\to\overline{\Omega}$ is another maximal curvature 
flow for $\SS_0$ in $C^{2+2\alpha,1+\alpha}([0,1]\times[0,\widetilde{T}))$, satisfying 
$\overline{\gamma}^i_t=\overline{k}^i\overline{\nu}^i+\overline{\lambda}^i\overline{\tau}^i$ for 
some functions $\overline{\lambda}^i$ in $C^{2\alpha}([0,1]\times[0,\widetilde{T}))$. 
If we consider the maps $\widehat{\gamma}^i(x,t)=
\overline{\gamma}^i(\theta^i(x),t)$, they give a $C^{2+2\alpha,1+\alpha}([0,1]\times[0,\widetilde{T}))$ 
curvature flow of the initial network $\widetilde{\SS}_0$ which satisfies the compatibility 
conditions of order $2$, hence (by the above argument) $\widetilde{T}=T$ and the maps 
$\widehat{\gamma}^i$ and $\widetilde{\gamma}^i$ only differ by reparametrizations given by some maps $\varphi^i\in C^{2+2\alpha,1+\alpha}([0,1]\times[0,T))$ with $\varphi^i(x,0)=x$, 
that is,
$$
\widehat{\gamma}^i(x,t)=\widetilde{\gamma}^i(\varphi^i(x,t),t)\,.
$$
It follows that
$$
\overline{\gamma}^i(x,t)=\widehat{\gamma}^i([\theta^i]^{-1}(x),t)=
\widetilde{\gamma}^i(\varphi^i([\theta^i]^{-1}(x),t),t)
=\gamma^i(\theta^i(\varphi^i([\theta^i]^{-1}(x),t)),t)
$$
which shows that the two flows $\overline{\gamma}^i$ and $\gamma^i$ 
of the initial network $\SS_0$ coincide up to the time--dependent 
reparametrizations $(x,t)\mapsto (\theta^i(\varphi^i([\theta^i]^{-1}(x),t)),t)$.\\
The last assertion follows by Proposition~\ref{c2geocomp}.
\end{proof}

\subsection{Initial data with higher regularity}\label{smoothdata}
We discuss the higher regularity of the flow when the initial network is of class $C^\infty$.

\begin{defn}\label{ncompcond}
We say that the {\em compatibility conditions of 
every order} for system~\eqref{problema-nogauge-general} are satisfied by an (initial) regular $C^\infty$ network $\SS_0=\bigcup_{i=1}^{n}\sigma^{i}\left([0,1]\right)$ and we call such a network {\em smooth}, 
if at every end--points and every $3$--point there hold all the relations 
on the space derivatives of the functions $\sigma^i$, 
obtained repeatedly differentiating in time the
boundary conditions and using the evolution equation $\gamma^i_t(x,t)
=\frac{\gamma_{xx}^{i}\left(x,t\right)}{\left|\gamma_{x}^{i}\left(x,t\right)\right|^{2}}$ 
to substitute time derivatives with space derivatives.\\
We say that a $C^\infty$ flow by curvature $\SS_t$ is {\em smooth} 
if all the networks $\SS_t$ are {\em smooth}.
\end{defn}

It is immediate by this definition that every network $\SS_t$ of a $C^\infty$ 
special curvature flow of an initial regular network $\SS_0$ is smooth for every $t>0$.

\begin{rem}\label{ght2}
We underline that being a {\em smooth} network implies being {\em regular} and $C^\infty$ (composed of $C^\infty$ curves), but it is way more restrictive than that. Analogously, a {\em smooth} curvature flow of networks is not simply $C^\infty$ up to the parabolic boundary (see Remark~\ref{ght1}). Anyway, similarly as before (Proposition~\ref{c2geocomp}), 
every network of a $C^\infty$ curvature flow can be reparametrized to be smooth.
\end{rem}

If we assume that the initial regular network is smooth, 
we have the following higher regularity result.

\begin{thm}\label{smoothexist-prob} 
For any initial smooth network $\SS_0$ in a smooth, convex, open set 
$\Omega\subseteq\R^2$ there exists a unique $C^\infty$ solution of 
system~\eqref{problema-nogauge-general} in a maximal time interval $[0,T)$. 
\end{thm}
\begin{proof}
Since the initial network $\SS_0$ satisfies the compatibility condition at every order, 
the method of the previous section actually 
provides, for every $n\in\NN$, a unique solution in 
$C^{2n+2\alpha,n+\alpha}([0,1]\times[0,T_n])$ of system~\eqref{problema-nogauge-general} 
satisfying the compatibility 
conditions of order $0,1, \dots, 2n$ at every time.\\
So, if we have a solution $\gamma^{i}\in C^{2n+2\alpha,n+\alpha}([0,1]\times[0,T_n])$
for $n\geqslant1$, then the functions $\gamma_x^i$ belong to $C^{2n-1+2\alpha,n-1/2+\alpha}
([0,1]\times[0,T_n])$ (see~\cite[Section~8.8]{krylov1}). Considering the parabolic system 
satisfied by $v^i(x,t)=\gamma_t^i(x,t)$ (see~\cite[Page~250]{mannovtor}), by Solonnikov results 
in~\cite{solonnikov1} $v^i=\gamma^i_t$ belongs to $C^{2n+2\alpha,n+\alpha}([0,1]\times[0,T_n])$. Since $\gamma^i_{xx}=\gamma^i_t\,\vert \gamma^i_x\vert^2$ 
with $ \vert \gamma^i_x\vert^2\in C^{2n-1+2\alpha,n-1/2+\alpha}([0,1]\times[0,T_n])$, 
we get also 
$$
\gamma^i_{xx}\in C^{2n-1+2\alpha,n-1/2+\alpha}([0,1]\times[0,T_n])\,.
$$
Following~\cite{lusiw}, we can then conclude that 
$\gamma^i\in C^{2n+1+2\alpha,n+1/2+\alpha}([0,1]\times[0,T_n])$. \\
Iterating this argument, we see that $\gamma^i\in C^\infty([0,1]\times[0,T_n])$.
Moreover, since for every $n\in\NN$ the solution obtained is unique, 
it must coincide with $\gamma^i$ and we can choose all the $T_n$ 
to be the same positive value $T$. \\
It follows that the solution is in $C^\infty$ till the parabolic boundary, hence, 
all the compatibility conditions are satisfied at every time $t\in[0,T)$.
\end{proof}

As a consequence, we have the following theorem.

\begin{thm}\label{smoothexist} 
For any initial smooth network $\SS_0$ in a smooth, convex, open set 
$\Omega\subseteq\R^2$ there exists a smooth curvature flow of $\SS_0$ in a maximal 
positive time interval $[0,T)$. 
\end{thm}

For $C^\infty$ networks we then introduce the concept of {\em geometrically smoothness}. 

\begin{defn}\label{geosmoothdef}
We say that a network $\SS_0=\bigcup_{i=1}^{n}\sigma^{i}\left([0,1]\right)$
of class $C^\infty$ is {geometrically smooth} if it can be reparametrized to be smooth. 
\end{defn}

\begin{rem}\label{ght3}
By arguments similar to the ones of Lemma~\ref{repar0}, it can be shown that, like for geometrical $2$--compatibility, this property 
depends only on (some relations on) the curvature and its derivatives at the end--points and at 
the $3$--points of a $C^\infty$ network (see~\cite{mannovtor} for more details), that is, 
geometrical smoothness is again a geometric property (obviously invariant by $C^\infty$ 
reparametrizations, by the definition).\\
Moreover, as before (see Proposition~\ref{c2geocomp}), every $C^\infty$ curvature flow of an
initial regular network $\SS_0$ is actually composed of geometrically smooth networks
$\SS_t$ for every $t>0$.
\end{rem}

The following short--time existence theorem holds, essentially with the same proof of Proposition~\ref{2geocomp}.

\begin{thm}\label{geosmoothexist}
For any initial geometrically smooth network $\SS_0$ in a smooth, convex, open set 
$\Omega\subseteq\R^2$ there exists a $C^\infty$ curvature flow of $\SS_0$ in a maximal 
positive time interval $[0,T)$. 
\end{thm}

An immediate consequence is the following corollary.

\begin{cor}\label{sunique} 
For any initial geometrically smooth network $\SS_0=\bigcup_{i=1}^n\sigma^i([0,1])$ 
in a smooth, convex, open set $\Omega\subseteq\R^2$, there exists a geometrically 
unique solution of Problem~\eqref{problema} in the class of maps 
$C^{2+2\alpha,1+\alpha}([0,1]\times[0,T))$ in a maximal positive time interval $[0,T)$. 
Moreover, such a solution is $C^\infty$ and if the initial network is actually smooth, it can be chosen to be a {\em special} curvature flow.
\end{cor}

\begin{rem}\label{fffggg}
Notice that it follows that any curvature flow as in the hypotheses of the above theorem and 
corollary is a reparametrization (of class $C^{2+2\alpha,1+\alpha}$ in the first case and 
$C^\infty$ in the latter) of the special curvature flow 
(which is $C^\infty$ under the hypotheses of this corollary, by 
Theorem~\ref{smoothexist-prob}).\\
This corollary implies the geometric uniqueness of this flow in the class of smooth maps.
\end{rem}

\section{Integral estimates}\label{kestimates}

In this section, we work out some integral estimates for a special flow by
curvature of a smooth regular network. These estimates were previously proved for the case of the special curvature flow of a regular smooth triod with fixed end--points, in~\cite{mannovtor}. We now extend 
them to the case of a smooth network with ``controlled'' behavior of its end--points. An outline for such estimates with controlled behavior of the end--points, for a general curvature flow, appeared in~\cite[Section 7]{Ilnevsch}. We advise the reader that when the computations are exactly the same we will refer directly to~\cite[Section~3]{mannovtor}, where it is possible to find every detail.

In all this section we will assume that the special flow by curvature is given by a $C^\infty$ solution $\gamma^i$ of system~\eqref{problema-nogauge-general}, that is, there holds
$$
\gamma^i_t(x,t)=\frac{\gamma_{xx}^{i}\left(x,t\right)}{\left|\gamma_{x}^{i}\left(x,t\right)\right|^{2}}\,,
$$
(see Remark~\ref{specevol} and Definition~\ref{special} for the case of an initial $C^2$ network). The estimates, which only involve geometric quantities and do not involve the tangential velocities $\lambda_i$, hold also for any smooth flow (the ones where we do not use the special form of the functions $\lambda^i$ given by this equation). To use these estimates for a general smooth flow, because of geometric uniqueness (see Corollary~\ref{sunique} and Remark~\ref{fffggg}), one must reparametrize such a flow, preserving the boundary condition~\eqref{endsmooth} below, so it becomes special, then carry back the geometric (invariant by reparametrization) estimates to the original flow. Alternatively, one can also directly prove these estimates without reparametrizing first to a special flow, see~\cite[Section 7]{Ilnevsch}.

We will see that such a special flow of a regular smooth network with ``controlled'' end--points exists smoothly as long as
the curvature stays bounded and none of the lengths of the curves goes to zero (Theorem~\ref{curvexplod}). 

We suppose that the special solution maps $\gamma^i$ above exist and are of class $C^\infty$ in the time interval $[0,T)$ and that they describe the flow of a regular $C^\infty$ network $\mathbb{S}_t$ in $\Omega$, composed of $n$ curves
$\gamma^{i}(\cdot,t):[0,1]\to\overline{\Omega}$ with $m$ $3$--points
$O^1, O^2,\dots, O^m$ and $l$ end--points $P^1, P^2,\dots, P^l$. We 
will assume that either such end--points are fixed or that there 
exist uniform (in time) constants $C_j$, for every $j\in\NN$, such that 
\begin{equation}\label{endsmooth}
\vert\partial_s^jk(P^r,t)\vert+\vert\partial_s^j\lambda(P^r,t)\vert\leqslant C_j\,,
\end{equation}
for every $t\in[0,T)$ and $r\in{1,2,\dots,l}$.

The first computation we are going to show is the evolution in time of the total length of a
network under the curvature flow.

\begin{prop}\label{equality1000} The time derivative of the measure $ds$ on any curve $\gamma^i$ of the network is
given by the measure $(\lambda_s^i-(k^i)^2)\,ds$. As a consequence, we have
$$
\frac{dL^i(t)}{dt}=\lambda^i(1,t)-\lambda^i(0,t)-\int_{\gamma^i(\cdot,t)}(k^i)^2\,ds
$$
and
$$
\frac{dL(t)}{dt}=\sum_{r=1}^l\lambda(P^r,t)-\int_{\SS_t}k^2\,ds\,,
$$
where, with a little abuse of notation, $\lambda(P^r,t)$ is the tangential velocity at the end--point $P^r$ of the curve of the network getting at such point, for any $r\in\{1,2,\dots,l\}$.\\
In particular, if the end--points $P^r$ of the network are fixed during the evolution, we have 
\begin{equation}\label{evolength}
\frac{dL(t)}{dt}=-\int_{\SS_t}k^2\,ds\,,
\end{equation}
thus, in such case, the total length $L(t)$ is decreasing in time and uniformly bounded above by the length of the initial network $\SS_0$.
\end{prop}

\begin{proof}
The formula for the time derivative of the measure $ds$ follows easily by the commutation formula~\eqref{commut}. Then,
$$
\frac{dL^i(t)}{dt}= \frac{d\,}{dt}\int_{\gamma^i(\cdot,t)}1\, ds
=\int_{\gamma^i(\cdot,t)}(\lambda^i_s-(k^i)^2)\,ds
=\lambda^i(1,t)-\lambda^i(0,t)-\int_{\gamma^i(\cdot,t)}(k^i)^2\,ds\,.
$$
Adding these relations for all the curves, the contributions of $\lambda^{pi}$ at every $3$--point $O^p$ vanish, by relation~\eqref{eq:cond2}, 
and the formula of the statement follows. If the end--points are fixed all the terms $\lambda(P^r,t)$ are zero and the last formula follows. 
\end{proof}

The following notation will be very useful for the next computations in this section.
\begin{defn}
We will denote with $\pol_\sigma(\ders^j\lambda,\ders^h k)$ a 
polynomial with constant coefficients in 
$\lambda,\dots,\ders^j\lambda$ and
$k,\dots,\ders^h k$ such that every monomial it contains is of the
form
$$
C\prod_{l=0}^j (\ders^l\lambda)^{\alpha_l} \cdot \prod_{l=0}^h
(\ders^lk)^{\beta_l}\,
\text{ { with} $\,\, \sum_{l=0}^j (l+1)\alpha_l + \sum_{l=0}^h
(l+1)\beta_l = \sigma$,}
$$
we will call $\sigma$ the {\em geometric order} of $\pol_\sigma$.\\
Moreover, if one of the two arguments of $\pol_\sigma$ 
does not appear, it means that the polynomial does not contain it, for
instance, $\pol_\sigma(\ders^h k)$
does not contain neither $\lambda$ nor its derivatives.\\
We will denote with
$\qol_\sigma(\dert^j\lambda,\ders^h k)$ a polynomial as before in
$\lambda,\dots,\dert^j\lambda$ and $k,\dots,\ders^h k$ such that all its
monomials are of the form
$$
C\prod_{l=0}^j (\dert^l\lambda)^{\alpha_l} \cdot \prod_{l=0}^h
(\ders^lk)^{\beta_l}\,
\text{ { with} $\,\, \sum_{l=0}^j (2l+1)\alpha_l + \sum_{l=0}^h
 (l+1)\beta_l = \sigma$.}
$$
Finally, when we will write $\pol_\sigma(\vert\ders^j\lambda\vert,
\vert\ders^h k\vert)$ (or $\qol_\sigma(\vert\dert^j\lambda\vert,
\vert\ders^h k\vert)$) we will mean a finite sum of terms like
$$
C\prod_{l=0}^j \vert\ders^l\lambda\vert^{\alpha_l} \cdot \prod_{l=0}^h
\vert\ders^lk\vert^{\beta_l}\,
\text{ { with} $\,\, \sum_{l=0}^j (l+1)\alpha_l + \sum_{l=0}^h
 (l+1)\beta_l = \sigma$,}
$$
where $C$ is a positive constant and the exponents $\alpha_l,\beta_l$
are non negative {\em real} values (analogously for $\qol_\sigma$).\\
Clearly we have $\pol_\sigma(\ders^j\lambda,\ders^h k)\leqslant
\pol_\sigma(\vert\ders^j\lambda\vert,\vert\ders^h k\vert)$. 
\end{defn}

By means of the commutation rule~\eqref{commut}, the relations in the
next lemma are easily proved by induction (Lemmas~3.7 and~3.8 
in~\cite{mannovtor}), starting from the relations in Section~\ref{basiccomp}.

\begin{lem}\label{kexpr}
The following formulas hold for every curve of the evolving network $\SS_t$:
\begin{equation*}
\begin{array}{ll}\dert\ders^jk=\ders^{j+2}k+\lambda\ders^{j+1}k +
\pol_{j+3}(\ders^{j}k)\qquad& \text{{ for every} $j\in \NN$,}\\
\ders^jk=\dert^{j/2}k+\qol_{j+1}(\dert^{j/2-1}\lambda,\ders^{j-1}k)\qquad
&\text{{ if} $j\geqslant2$ { is even,}}\\
\ders^jk=\dert^{(j-1)/2}k_s+\qol_{j+1}(\dert^{(j-3)/2}\lambda,\ders^{j-1}k)\qquad
&\text{{ if} $j\geqslant1$ { is odd,}}\\
\dert\ders^j\lambda=\ders^{j+2}\lambda-\lambda\ders^{j+1}\lambda -2k\ders^{j+1}k+
\pol_{j+3}(\ders^j\lambda,\ders^jk)
\qquad&\text{{ for every} $j\in \NN$,}\\
\ders^j\lambda=\dert^{j/2}\lambda+\pol_{j+1}(\ders^{j-1}\lambda,\ders^{j-1}k)
\qquad&\text{{ if} $j\geqslant2$ { is even,}}\\
\ders^j\lambda=\dert^{(j-1)/2}\lambda_s+\pol_{j+1}(\ders^{j-1}\lambda,\ders^{j-1}k)
\qquad&\text{{ if} $j\geqslant1$ { is odd.}}
\end{array}
\end{equation*}
\end{lem}

\begin{rem}\label{qolpol}
Notice that, by relations~\eqref{lambdakappa} 
at any $3$--point $O^p$ of the network there holds $\dert^j\lambda^{pi}=({\mathrm S}\dert^j {\mathrm K})^{pi}$,
that is, the time derivatives of $\lambda^{pi}$ are expressible as time
derivatives of the functions $k^{pi}$. Then, by using repeatedly such
relation and the first formula of Lemma~\ref{kexpr}, we can express
these latter as space derivatives of $k^{pi}$. Hence, we will have the
relation
$$
\sum_{i=1}^3\qol_\sigma(\dert^j\lambda^{pi},\ders^h
k^{pi})\,\biggr\vert_{\text{{ at the $3$--point $O^p$}}}=
\pol_\sigma(\ders^{\max\{2j, h\}}
{\mathrm K}^p)\,\biggr\vert_{\text{{ at
 the $3$--point $O^p$}}}
$$
with the meaning that this last polynomial contains also a product of
derivatives of different $k^{pi}$'s, because of the action of the linear
operator ${\mathrm S}$.\\
We will often make use of this identity in the computations of the
sequel in the following form,
$$
\sum_{i=1}^3\qol_\sigma(\dert^j\lambda^{pi},\ders^h
k^{pi})\,\biggr\vert_{\text{{ at the $3$--point $O^p$}}}\leqslant
\Vert \pol_\sigma(\vert\ders^{\max\{2j, h\}}k\vert)\Vert_{L^\infty}\,.
$$
\end{rem}

\begin{rem}\label{qolpol2}
We state the following {\em calculus rules} which will be used
extensively in the sequel,
\begin{align*}
\pol_\alpha(\ders^j\lambda,\ders^hk)\cdot
\pol_\beta(\ders^l\lambda,\ders^mk)=&\,
\pol_{\alpha+\beta}(\ders^{\max\{j,l\}}\lambda,\ders^{\max\{h,m\}}k)\,,\\
\qol_\alpha(\dert^j\lambda,\ders^hk)\cdot
\qol_\beta(\dert^l\lambda,\ders^mk)=&\,
\qol_{\alpha+\beta}(\dert^{\max\{j,l\}}\lambda,\ders^{\max\{h,m\}}k)\,.
\end{align*}
We already saw that the time derivatives of $k$ and $\lambda$ can be expressed in terms of space derivatives of $k$ at any $3$--point, the same holds for the space derivatives of $\lambda$, arguing by induction using the last two formulas in Lemma~\ref{kexpr}. Hence, it follows that 
\begin{align*}
\ders^l\pol_\alpha(\ders^j\lambda,\ders^hk)=
\pol_{\alpha+l}(\ders^{j+l}\lambda,\ders^{h+l}k)\,,\qquad&
\dert^l\pol_\alpha(\ders^j\lambda,\ders^hk)=
\pol_{\alpha+2l}(\ders^{j+2l}\lambda,\ders^{h+2l}k)\\
\dert^l\qol_\alpha(\dert^j\lambda,\ders^hk)=
\qol_{\alpha+2l}(\dert^{j+l}\lambda,\ders^{h+2l}k)\,,\qquad&
\qol_\alpha(\dert^j\lambda,\ders^hk)=\pol_{\alpha}
(\ders^{2j}\lambda,\ders^{\max\{h,2j-1\}}k)\,.
\end{align*}
\end{rem}

We are now ready to compute, for $j\in\NN$, 

\begin{align}
\frac{d\,}{dt} \int_{{\mathbb{S}_t}} |\ders^j k|^2\,ds
= &\, 2\int_{{\mathbb{S}_t}} \ders^j k\, \dert\ders^jk\,ds +
\int_{{\mathbb{S}_t}} |\ders^j k|^2(\lambda_s -k^2)\,ds\nonumber\\ 
=&\, 2\int_{{\mathbb{S}_t}} \ders^j k\, \ders^{j+2}k +\lambda \ders^{j+1} k\,\ders^{j} 
k +\pol_{j+3}(\ders^jk)\,\ders^jk\,ds 
+\int_{{\mathbb{S}_t}} |\ders^j k|^2(\lambda_s -k^2)\,ds\nonumber\\ 
=&\, -2\int_{{\mathbb{S}_t}} \vert\ders^{j+1} k\vert^2\,ds 
+ \int_{{\mathbb{S}_t}} \ders(\lambda|\ders^j k|^2)\,ds 
+\int_{{\mathbb{S}_t}}\pol_{2j+4}(\ders^jk)\,ds\nonumber\\ &\,
- 2\sum_{p=1}^m\sum_{i=1}^3 \ders^j k^{pi}\,\ders^{j+1}k^{pi}\, \biggr\vert_{\text{{ at the $3$--point $O^p$}}} 
+ 2\sum_{r=1}^l \ders^j k\,\ders^{j+1}k\, \biggr\vert_{\text{{ at the end--point} $P^r$}}\nonumber\\ 
\leqslant&\, -2\int_{{\mathbb{S}_t}} \vert\ders^{j+1}k\vert^2\,ds 
+ \int_{\mathbb{S}_t} \pol_{2j+4}(\ders^{j}k)\,ds + lC_jC_{j+1}\nonumber\\ 
&\,- \sum_{p=1}^m\sum_{i=1}^3 2\ders^j k^{pi}\,\ders^{j+1}k^{pi}+\lambda^{pi}\vert\ders^{j}k^{pi}\vert^2\, 
\biggr\vert_{\text{{ at the $3$--point $O^p$}}}\label{evolint000} 
\end{align} 
where we integrated by parts the first term on the second line and we
estimated the contributions given by the end--points $P^r$ by means of
assumption~\eqref{endsmooth}.

In the case that we consider the end--points $P^1, P^2,\dots, P^l$ to
be fixed, we can assume that the terms $C_jC_{j+1}$ are all zero in the
above conclusion, by the following lemma.

\begin{lem}\label{evenly}
If the end--points $P^r$ of the network are fixed, then there holds
$\ders^jk=\ders^j\lambda=0$, for every even $j\in\NN$.
\end{lem}
\begin{proof}
The first case $j=0$ simply follows from the fact that the velocity
$\underline{v}=\lambda\tau+k\nu$ is always zero at the fixed end--points $P^r$.\\
We argue by induction, we suppose that
for every even natural $l\leqslant j-2$ we have
$\ders^lk=\ders^l\lambda=0$, then, by using the first equation in
Lemma~\ref{kexpr}, we get
$$
\ders^{j}k=\dert\ders^{j-2}k-\lambda\ders^{j-1}k -\pol_{j+1}(\ders^{j-2}k)
$$
at every end--point $P^r$.\\
We already know that $\lambda=0$ and by the
inductive hypothesis $\ders^{j-2}k=0$, thus
$\dert\ders^{j-2}k=0$. 
Since $\pol_{j+1}(\ders^{j-2}k)$ is a sum of terms like $C\prod_{l=0}^{j-2}
(\ders^lk)^{\alpha_l}$ with $\sum_{l=0}^{j-2}(l+1)\alpha_l=j+1$ which is
odd, at least one of the terms of this sum has to be odd, hence at
least for one index $l$, the product $(l+1)\alpha_l$ is odd. It
follows that at least for one even $l$ the exponent $\alpha_l$ is
nonzero. Hence, at least one even derivatives is present in every monomial
of $\pol_{j+1}(\ders^{j-2}k)$, which contains only derivatives up to the
order $(j-2)$.\\
Again, by the inductive hypothesis, we then conclude that at
the end--points $\ders^{j}k=0$.\\
We can deal with $\lambda$ similarly, by
means of the relations in Lemma~\ref{kexpr}.
\end{proof}

In the very special case $j=0$ we get explicitly 
\begin{equation*}
 \frac{d\,}{dt} \int_{{\mathbb{S}_t}} k^2\,ds 
\leqslant -2\int_{{\SS_t}} \vert k_s\vert^2\,ds 
+ \int_{\SS_t} k^4\,ds 
- \sum_{p=1}^m\sum_{i=1}^3 2 k^{pi}k^{pi}_s+\lambda^{pi}|k^{pi}|^2\,
\biggr\vert_{\text{{ at the $3$--point $O^p$}}}+ lC_0C_1 
\end{equation*}
where the two constants $C_0$ and $C_1$ come from assumption~\eqref{endsmooth}.

Then, recalling relation~\eqref{eq:orto}, we have 
$\sum_{i=1}^3 k^{pi}k^{pi}_s+\lambda^{pi}|k^{pi}|^2\,\bigr\vert_{\text{{at the $3$--point $O^p$}}}=0$, 
and substituting the above, 
\begin{equation}\label{ksoltanto} 
\frac{d\,}{dt} \int_{{\SS_t}} k^2\,ds 
\leqslant -2\int_{{\SS_t}} \vert k_s\vert^2\,ds 
+ \int_{\SS_t} k^4\,ds 
+\sum_{p=1}^m\sum_{i=1}^3 \lambda^{pi}|k^{pi}|^2\,\biggr\vert_{\text{{
 at the $3$--point $O^p$}}} + lC_0C_1\,,
\end{equation} 
hence, we lowered the maximum order of the space derivatives of the curvature in the $3$--point terms, 
particular now it is lower than the one of the ``nice'' negative integral.

As we have just seen for the case $j=0$, also for the general case
we want to simplify the term $\sum_{i=1}^3
2\partial_{s}^{j}k^{pi}\partial_{s}^{j+1}k^{pi}+\lambda^{pi}|\partial_{s}^{j}k^{pi}|^{2}\,\bigr\vert_{\text{{at
 the $3$--point $O^p$}}}$, 
in order to control it.\\

Using formulas in Lemma~\ref{kexpr}, we have (see~\cite[Pages~258--259]{mannovtor}, for details)
\begin{align*} 
2\ders^jk &\,\,\ders^{j+1} k + \lambda |\ders^j k|^2\\ 
=&\, 2\dert^{j/2}k\cdot\dert^{j/2}(k_s+k\lambda)+ \qol_{j+1}(\dert^{j/2-1}\lambda,\ders^{j-1}k)
\cdot\dert^{j/2}k_s + \qol_{2j+3}(\dert^{j/2}\lambda,\ders^{j}k)\,. 
\end{align*}

We now examine the term 
$\qol_{j+1}(\dert^{j/2-1}\lambda,\ders^{j-1}k)\cdot\dert^{j/2}k_s$, which, by using Lemma~\ref{kexpr}, 
can be written as 
$\dert\qol_{2j+1}(\dert^{j/2-1}\lambda,\ders^{j-1}k) +
\qol_{2j+3}(\dert^{j/2}\lambda,\ders^{j}k)$ (see~\cite[Pages~258--259]{mannovtor}, for details).
It follows that 
\begin{align*} 
\sum_{p=1}^m\sum_{i=1}^3 2\ders^j k^{pi}\,\ders^{j+1}k^{pi}&\,
+\lambda^{pi}|\ders^j k^{pi}|^2\lambda\, \biggr\vert_{\text{{ at the $3$--point $O^p$}}}\\ 
=&\,\sum_{p=1}^m\sum_{i=1}^3 \dert\qol_{2j+1}(\dert^{j/2-1}\lambda^{pi},\ders^{j-1}k^{pi}) 
+\qol_{2j+3}(\dert^{j/2}\lambda^{pi},\ders^{j}k^{pi}) \,\biggr\vert_{\text{{ at the $3$--point $O^p$}}} 
\end{align*}

Resuming, if $j\geqslant2$ is even, we have 
\begin{align*}
\frac{d\,}{dt} \int_{{\SS_t}} |\ders^j k|^2\,ds\leqslant&\, 
-2\int_{{\SS_t}} \vert\ders^{j+1}k\vert^2\,ds 
+ \int_{\SS_t}\pol_{2j+4}(\ders^{j}k)\,ds+ lC_jC_{j+1}\nonumber \\ 
&\, 
+\sum_{p=1}^m \sum_{i=1}^3 \dert\qol_{2j+1}(\dert^{j/2-1}\lambda^{pi},\ders^{j-1}k^{pi}) 
+\qol_{2j+3}(\dert^{j/2}\lambda^{pi},\ders^{j}k^{pi}) \,\biggr\vert_{\text{{ at the $3$--point $O^p$}}}\,.\label{pippo100} 
\end{align*} 
Now, the key tool to estimate the terms 
$\int_{\SS_t}\pol_{2j+4}(\ders^{j}k)\,ds$ and
$\sum_{i=1}^3\qol_{2j+3}(\dert^{j/2}\lambda^{pi},\ders^{j}k^{pi})\,\bigr \vert_{\text{{ at the $3$--point $O^p$}}}
$
are the following Gagliardo--Nirenberg interpolation inequalities
(see~\cite[Section~3, Pages~257--263]{nirenberg1}).

\begin{prop}\label{gl}
Let $\gamma$ be a $C^\infty$, regular curve in $\R^2$ with finite length ${\mathrm L}$. 
If $u$ is a $C^\infty$ function defined on $\gamma$ and $m\geqslant1$, $p\in[2,+\infty]$, 
we have the estimates 
\begin{equation}\label{int1} 
{\Vert\partial_s^n u\Vert}_{L^p} \leqslant C_{n,m,p} 
{\Vert\partial_s^m u\Vert}_{L^2}^{\sigma} 
{\Vert u\Vert}_{L^2}^{1-\sigma}+ 
\frac{B_{n,m,p}}{{\mathrm L}^{m\sigma}}{\Vert u\Vert}_{L^2} 
\end{equation} 
for every $n\in\{0,\dots, m-1\}$ where $$ \sigma=\frac{n+1/2-1/p}{m} $$ 
and the constants $C_{n,m,p}$ and $B_{n,m,p}$ are independent of $\gamma$.
In particular, if $p=+\infty$, 
\begin{equation}\label{int2} 
{\Vert\partial_s^n u\Vert}_{L^\infty} \leqslant C_{n,m} 
{\Vert\partial_s^m u\Vert}_{L^2}^{\sigma} 
{\Vert u\Vert}_{L^2}^{1-\sigma}+ 
 \frac{B_{n,m}}{{\mathrm L}^{m\sigma}}{\Vert 
u\Vert}_{L^2}\qquad\text{ { with} }\quad \text{ $\sigma=\frac{n+1/2}{m}$.} 
\end{equation}
\end{prop}

After estimating the integral of every monomial 
of $\pol_{2j+4}(\ders^{j}k)$ by mean of the H\"{o}lder inequality,
one uses the Gagliardo--Nirenberg estimates on the result, concluding that
$$
\int_{{\SS_t}} \pol_{2j+4}(\ders^jk)\,ds\leqslant 
1/4\int_{{\SS_t}} \vert\ders^{j+1}k\vert^2\,ds
+ C\bigg(\int_{{\SS_t}} k^2\,ds\bigg)^{2j+3} + C\,,
$$
where the constant $C$ depends only on $j\in\NN$ and the lengths of the
curves of the network (see~\cite[Pages~260--262]{mannovtor}, for details).\\
Any term $\sum_{i=1}^3\qol_{2j+3}(\dert^{j/2}\lambda^{pi},\ders^{j}k^{pi})\,\bigr
\vert_{\text{{ at the $3$--point $O^p$}}}$ can be estimated similarly.

Hence, for every even $j\geqslant2$ we can finally write 
\begin{align}\label{final1} 
\frac{d\,}{dt} \int_{{\SS_t}} |\ders^j k|^2\,ds\leqslant &\,-\int_{{\SS_t}} \vert\ders^{j+1}k\vert^2\,ds 
+ C\bigg(\int_{{\SS_t}} k^2\,ds\bigg)^{2j+3} + C+ lC_jC_{j+1}\\ 
&\, + \dert\sum_{p=1}^m\sum_{i=1}^3 \qol_{2j+1}(\dert^{j/2-1}\lambda^{pi},\ders^{j-1}k^{pi}) \,
\biggr\vert_{\text{{ at the $3$--point $O^p$}}}\nonumber\\\leqslant &\,
C\bigg(\int_{{\SS_t}} k^2\,ds\bigg)^{2j+3} + 
\dert\sum_{p=1}^m\sum_{i=1}^3 \qol_{2j+1}(\dert^{j/2-1}\lambda^{pi},\ders^{j-1}k^{pi}) \,
\biggr\vert_{\text{{ at the $3$--point $O^p$}}} + C +lC_jC_{j+1}\,.\nonumber 
\end{align} 

Recalling the computation in the special case $j=0$, 
this argument gives the same final estimate without the contributions coming from the $3$--points:
\begin{equation}\label{evolint999} 
\frac{d\,}{dt} \int_{{\SS_t}} k^2\,ds\leqslant C\left(\int_{{\SS_t}} k^2\,ds\right)^{3} + C +lC_0C_1\,. 
\end{equation}

Integrating~\eqref{final1} in time on $[0,t]$ and estimating we get 
\begin{align*} 
\int_{\SS_t} |\ders^j k|^2\,ds\leqslant&\, 
\int_{\SS_0} |\ders^j k|^2\,ds 
+ C\int_{0}^t\bigg(\int_{\SS_\xi} k^2\,ds\bigg)^{2j+3}\,d\xi 
+ Ct+ lC_jC_{j+1}t\\ &\, 
+\sum_{p=1}^m \sum_{i=1}^3 \qol_{2j+1}(\dert^{j/2-1}\lambda^{pi}(0,t),\ders^{j-1}k^{pi}(0,t))\\&\, 
-\qol_{2j+1}(\dert^{j/2-1}\lambda^{pi}(0,0),
\ders^{j-1}k^{pi}(0,0)) \\ \leqslant&\,
C\int_{0}^t\bigg(\int_{\SS_\xi} k^2\,ds\biggr)^{2j+3}\,d\xi 
+ \Vert\pol_{2j+1}(\vert\ders^{j-1}k\vert)\Vert_{L^\infty} 
+ Ct + lC_jC_{j+1}t+C \,,
\end{align*} 
where in the last passage we used Remark~\ref{qolpol}. 
The constant $C$ depends only on $j\in\NN$ and on the network $\SS_0$.\\ 
Interpolating again by means of inequalities~\eqref{int2}, one gets 
$$
\Vert\pol_{2j+1}(\vert\ders^{j-1}k\vert)\Vert_{L^\infty}\leqslant 1/2\Vert\partial_s^{j}k\Vert^2_{L^2}+ 
C\Vert k \Vert_{L^2}^{4j+2}\,. 
$$
Hence, putting all together, for every even $j\in\NN$, we conclude
$$
\int_{\SS_t} |\ders^j k|^2\,ds \leqslant C\int_{0}^t\bigg(\int_{\SS_\xi} k^2\,ds\bigg)^{2j+3}\,d\xi 
+ C\bigg(\int_{\SS_t} k^2\,ds\bigg)^{2j+1} + Ct + lC_jC_{j+1}t+ C\,. 
$$ 
Passing from integral to $L^\infty$ estimates, by using inequalities~\eqref{int2}, we have the following proposition.

\begin{prop}\label{pluto1000} 
If assumption~\eqref{endsmooth} holds, the lengths of all the curves
are uniformly positively bounded from below and the $L^2$ norm of $k$
is uniformly bounded on $[0,T)$, then the curvature of $\SS_t$ and
all its space derivatives are uniformly bounded in the same time
interval by some constants depending only on the $L^2$ integrals 
of the space derivatives of $k$ on the initial network $\SS_0$. 
\end{prop}

By using the relations in Lemma~\ref{kexpr}, one then gets also estimates for every time and space derivatives of $\lambda$ which finally imply estimates on all the derivatives of the maps $\gamma^i$, stated in the next Proposition~\ref{unif222} 
(see~\cite[Pages~263--266]{mannovtor} for details). We discuss here explicitly how, in the hypotheses of this proposition, we deal with $\lambda$ and the ``velocity'' $\underline{v}=\gamma_t=k\nu+\lambda\tau$ of the flow.\\
At every $3$--point $O^p$ we have $\sum_{i=1}^3(\lambda^{pi})^2=\sum_{i=1}^3(k^{pi})^2$, by relations~\eqref{stimaklambdaneitripunti}, hence the squared modulus of the velocity $v^2=\vert \underline{v}\vert^2$ is uniformly bounded at every $3$--point, being $k^2$ uniformly bounded by some constant $C$.\\
Then, since $v^2$ is also uniformly bounded at the end--points of $\SS_t$, by 
assumption~\eqref{endsmooth}, applying the maximum principle to the equation for $v^2$, given by
\begin{equation*}
\dert v^2 =(v^2)_{ss} -2\lambda^2_s -2k^2_s -\lambda(v^2)_s + 2v^2 k^2\,,
\end{equation*}
which follows from equation~\eqref{dertdilamb}
\begin{equation*}
\dert\lambda =\lambda_{ss} -\lambda\lambda_s - 2kk_s +\lambda k^2
\end{equation*}
and equation~\eqref{dertdik}, we see that if $v^2$ gets larger than some fixed constant (independent of time), then its maximum is taken in the {\em interior} of
some curve of $\SS_t$ and
\begin{equation*}
\dert v^2_{\max} \leqslant 2v^2_{\max} k^2\leqslant 2Cv^2_{\max}\,.
\end{equation*}
Hence, integrating this {\em linear} differential inequality, we obtain that
$\underline{v}$ and hence $\lambda$ are also uniformly bounded as $k$
and its derivatives in the time interval $[0,T)$.

\begin{rem}\label{klstimarem} Notice that the conclusion that $v^2$ is uniformly bounded follows simply knowing that the curvature is uniformly bounded and assumption~\eqref{endsmooth} holds. In particular, for the case of an evolving network $\SS_t$ with fixed end--points and uniformly bounded curvature in an interval $[0,T)$
\end{rem} 

\begin{prop}\label{unif222} 
If $\SS_t$ is a $C^\infty$ special flow of the initial network 
$\SS_0=\bigcup_{i=1}^n\sigma^i$, satisfying assumption~\eqref{endsmooth}, such that the lengths of the $n$ curves are uniformly bounded away from zero 
and the $L^2$ norm of the curvature is uniformly bounded by some
constant in the time interval $[0,T)$, then 
\begin{itemize} 
\item all the derivatives in space and time of $k$ and $\lambda$ are
 uniformly bounded in $[0,1]\times[0,T)$, 
\item all the derivatives in space and time of the curves $\gamma^i(x,t)$ 
are uniformly bounded in $[0,1]\times[0,T)$, 
\item the quantities $\vert\gamma^i_x(x,t)\vert$ are uniformly bounded 
from above and away from zero in $[0,1]\times[0,T)$. 
\end{itemize} 
All the bounds depend only on the uniform controls on the $L^2$ norm of $k$, on the
lengths of the curves of the network from below, on the constants
$C_j$ in assumption~\eqref{endsmooth}, on the $L^\infty$ norms of
the derivatives of the curves $\sigma^i$ 
and on the bound from above and below on $\vert\sigma^i_x(x,t)\vert$, for the curves describing 
the initial network $\SS_0$. 
\end{prop}

Now, we work out a second set of estimates 
where everything is controlled -- still under the
assumption~\eqref{endsmooth} -- only by the $L^2$ norm of the
curvature and the inverses of the lengths of the curves at time
zero.
 
As before we consider the $C^\infty$ special curvature flow $\SS_t$ of a smooth network
$\SS_0$ in the time interval $[0,T)$, composed of $n$ curves $\gamma^{i}(\cdot,t):[0,1]\to\overline{\Omega}$ with $m$ $3$--points $O^1, O^2,\dots, O^m$ and $l$ end--points $P^1, P^2,\dots, P^l$, satisfying assumption~\eqref{endsmooth}.

\begin{prop}\label{stimaL} 
For every $M>0$ there exists a time $T_M\in (0,T)$, depending only on the {\em structure} of the network and the constants $C_0$ and $C_1$ in assumption~\eqref{endsmooth}, such that if the square of the $L^2$ norm of the curvature and the inverses of the lengths of the curves of $\SS_0$ are bounded by $M$, then the square of the $L^2$ norm of $k$ and the inverses of the lengths of the curves of $\SS_t$ are smaller than $2(n+1)M+1$, for every time $t\in[0,T_M]$. 
\end{prop}
\begin{proof} 
The evolution equations for the lengths of the $n$ curves are given by 
$$
\frac{dL^i(t)}{dt}=\lambda^i(1,t)-\lambda^i(0,t)-\int_{\gamma^i(\cdot,t)}k^2\,ds\,,
$$
then, recalling computation~\eqref{ksoltanto}, 
we have 
\begin{align*} 
\frac{d\,}{dt} \bigg(\int_{\SS_t} k^2\,ds 
+ \sum_{i=1}^n\frac{1}{L^i}\bigg)\leqslant &\, 
-2\int_{{\SS_t}} k_{s}^2\,ds 
+ \int_{{\SS_t}} k^4\,ds 
+6m\Vert k\Vert_{L^\infty}^3 +lC_0C_1
- \sum_{i=1}^n\frac{1}{(L^i)^2}\frac{dL^i}{dt}\\
=&\, -2\int_{{\SS_t}} k_{s}^2\,ds 
+ \int_{{\SS_t}} k^4\,ds 
+6m\Vert k\Vert^3_{L^\infty}+lC_0C_1\\
&\,-\sum_{i=1}^n\frac{\lambda^i(1,0)-\lambda^i(0,t)
+\int_{\gamma^i(\cdot,t)}k^2\,ds}{(L^i)^2}\\
\leqslant&\,-2\int_{{\SS_t}} k_{s}^2\,ds 
+ \int_{{\SS_t}} k^4\,ds +6m\Vert k\Vert^3_{L^\infty}+lC_0C_1\\
&\,+2\sum_{i=1}^n\frac{\Vert k\Vert_{L^\infty}+C_0}{(L^i)^2} +\sum_{i=1}^n\frac{\int_{\SS_t}k^2\,ds}{(L^i)^2}\\
\leqslant&\,-2\int_{{\SS_t}} k_{s}^2\,ds 
+ \int_{{\SS_t}} k^4\,ds 
+ (6m+2n/3)\Vert k\Vert_{L^\infty}^3+lC_0C_1+2nC_0^3/3\\ 
&\,+ \frac{n}{3}\bigg(\int_{\SS_t}k^2\,ds\bigg)^3 
+ \frac{2}{3}\sum_{i=1}^n\frac{1}{(L^i)^3} 
\end{align*} 
where we used Young inequality in the last passage.\\ 
Interpolating as before (and applying again Young inequality) 
but keeping now in evidence the terms depending on $L^i$ in inequalities~\eqref{int1}, 
we obtain 
\begin{align*} 
\frac{d\,}{dt} \left(\int_{\SS_t} k^2\,ds 
+ \sum_{i=1}^n\frac{1}{L^i}\right) \leqslant&\,
-\int_{\SS_t} k_{s}^2\,ds 
+ C\left(\int_{\SS_t}k^2\,ds\right)^3 
+ C\sum_{i=1}^n\frac{\left(\int_{\SS_t}k^2\,ds\right)^2} {L^i}\\ 
&\,+C\sum_{i=1}^n\frac{\left(\int_{\SS_t}k^2\,ds\right)^{3/2}} {(L^i)^{3/2}}+ C\sum_{i=1}^n\frac{1}{(L^i)^3} + C\\ 
\leqslant&\, C\left(\int_{\SS_t}k^2\,ds\right)^3 
+ C\sum_{i=1}^n\frac{1}{(L^i)^3}+C\\ 
\leqslant&\, C\left(\int_{\SS_t}k^2\,ds 
+\sum_{i=1}^n\frac{1}{L^i}+1\right)^3\,,
\end{align*} 
with a constant $C$ depending only on the {\em structure} of the network and on the constants $C_0$ and $C_1$ in 
assumption~\eqref{endsmooth}.\\ 
This means that the positive function 
$f(t)=\int_{\SS_t} k^2\,ds 
+\sum_{i=1}^n\frac{1}{L^i(t)}+1$ 
satisfies the differential inequality $f^\prime\leqslant Cf^3$, 
hence, after integration
$$
f^2(t)\leqslant \frac{f^2(0)}{1-2Ctf^2(0)}\leqslant \frac{f^2(0)}{1-2Ct[(n+1)M+1]}
$$

then, if $t\leqslant T_M=\frac{3}{8C[(n+1)M+1]}$, we get
$f(t)\leqslant2f(0)$. Hence,
$$
\int_{\SS_t} k^2\,ds +\sum_{i=1}^n\frac{1}{L^i(t)}\leqslant 2\int_{\SS_0} k^2\,ds +2\sum_{i=1}^n\frac{1}{L^i(0)}+1\leqslant 2[(n+1)M]+1\,.
$$ 
\end{proof}

By means of this proposition, we can strengthen the conclusion of
Proposition~\ref{unif222}.

\begin{cor}\label{topolino7} In the hypothesis of the previous proposition, in the time
 interval $[0,T_M]$ all the bounds in Proposition~\ref{unif222} depend only on the 
 $L^2$ norm of $k$ on $\SS_0$, on the constants $C_j$ in assumption~\eqref{endsmooth}, 
 on the $L^\infty$ norms of the derivatives of the curves $\sigma^i$, 
 on the bound from above and below on $\vert\sigma^i_x(x,t)\vert$ and on the inverses of the lengths of the curves of 
 the initial network $\SS_0$.
\end{cor}

From now on we assume that the $L^2$ norm of the curvature 
and the inverses of the lengths of the curves are bounded in the interval $[0,T_M]$.

Considering $j\in\NN$ even, if we differentiate the function 
$$ \int_{{\SS_t}} k^2+ tk_s^2
+ \frac{t^2 k_{ss}^2}{2!}+\dots
+ \frac{t^j\vert \partial_s^jk\vert^2}{j!}\,ds\,, $$ 
and we estimate with interpolation inequalities as before (see~\cite[Pages~268--269]{mannovtor}, for details), we obtain 
\begin{align}\label{ciaociao}
\frac{d\,}{dt}\int_{{\SS_t}}&\, k^2+ tk_s^2
+\frac{t^2 k_{ss}^2}{2!}+\dots+ \frac{t^j\vert \partial_s^jk\vert^2}{j!}\,ds\\ \leqslant
&\,-\varepsilon\int_{{\SS_t}} k_s^2 
+ t k_{ss}^2 + t^2 k_{sss}^2+\dots+ t^j\vert \partial_s^{j+1}k\vert^2\,ds + C\nonumber\\ &\, 
+ \dert\sum_{p=1}^m\sum_{i=1}^3 t^2\qol_5(\lambda^{pi}, k_s^{pi}) 
+ t^4\qol_9(\dert\lambda^{pi},k_{sss}^{pi}) 
+\dots+ t^j\qol_{2j+1}(\dert^{j/2-1}\lambda^{pi},\ders^{j-1}k^{pi}) \,\biggr\vert_
{\text{{ at the $3$--point $O^p$}}}\nonumber\\ &\, 
+C\sum_{p=1}^m\sum_{i=1}^3 t k_s^{pi} k^{pi}_{ss}
+ t^3 k^{pi}_{sss} k_{ssss}^{pi} +
\dots + t^{j-1}\ders^{j-1} k^{pi}\,\ders^{j}k^{pi}\, 
\biggr\vert_{\text{{ at the $3$--point $O^p$}}}\nonumber 
\end{align} 
in the time interval $[0,T_M]$, 
where $\varepsilon>0$ and $C$ are two constants depending only on the $L^2$ norm of the curvature, the constants in 
assumption~\eqref{endsmooth} and the inverses of the lengths of the $n$ curves of $\SS_0$. \\
We proceed as we did before for the computation of 
$\frac{d}{dt}\int_{\SS_t}|\partial_s^jk|^2\,ds$\,.\\
First, we deal with the last line, 
$$
\sum_{i=1}^3 t k_s^{pi} k^{pi}_{ss}
+ t^3 k^{pi}_{sss} k_{ssss}^{pi} +
\dots + t^{j-1}\ders^{j-1} k^{pi}\,\ders^{j}k^{pi}\, 
\biggr\vert_{\text{{ at the $3$--point}}}\,.
$$
By formulas in Lemma~\ref{kexpr} and by Remark~\ref{qolpol}, we can write, for any term $\sum_{i=1}^3t^{h-1}\ders^{h-1}
k^i\ders^{h}k^i\, \biggr\vert_{\text{{ at the $3$--point}}}$,
\begin{align*}
\sum_{i=1}^3t^{h-1}\ders^{h-1}
k^i\ders^{h}k^i\, \biggr\vert_{\text{{ at the $3$--point}}}= 
&\,\sum_{i=1}^3t^{h-1}\qol_{2h+1}(\dert^{h/2-1}\lambda^i,\ders^{h-1}k^i)\\
 &\,\phantom{\sum_{i=1}^3} + t^{h-1}
\ders^{h}k^i\cdot \qol_{h}(\dert^{h/2-1}\lambda^i,\ders^{h-2}k^i)\,
\biggr\vert_{\text{{ at
 the $3$--point}}}\\
\leqslant &\,t^{h-1}\Vert \pol_{2h+1}(\vert\ders^{h-1}k\vert)\Vert_{L^\infty}
+t^{h-1}\Vert\ders^{h}k\Vert_{L^\infty}\Vert\pol_{h}(\vert\ders^{h-2}k\vert)\Vert_{L^\infty}
\end{align*}
(see~\cite[Page~270]{mannovtor}, for details).\\
The term $t^{h-1}\Vert \pol_{2h+1}(\vert\ders^{h-1}k\vert)\Vert_{L^\infty}$
is controlled as before by a small fraction of the term
$t^{h-1}\int_{\SS_t}\vert \ders^h k\vert^2\,ds$ and a possibly large
multiple of $t^{h-1}$ times some power of the $L^2$ norm of $k$ 
(which is bounded), whereas 
$t^{h-1}\Vert\ders^{h}k\Vert_{L^\infty}\Vert\pol_{h}(\vert\ders^{h-2}k\vert)\Vert_{L^\infty}$
is the critical term.\\
Again by means of interpolation inequalities~\eqref{int2} one estimates
$\Vert \ders^h k\Vert_{L^\infty}\,,\Vert \pol_{h}(\ders^{h-2}k)\Vert_{L^\infty}$
and $\Vert \ders^h k\Vert_{L^2}$
with the $L^2$ norm of $k$ and its derivatives.
After some computation (see~\cite[Pages~270--271]{mannovtor}, for details), one gets
$$
\sum_{i=1}^3t^{h-1}\ders^{h-1} k^i\ders^{h}k^i\,
\biggr\vert_{\text{{ at
 the $3$--point}}}
\leqslant \varepsilon_h/2 \left (t^h\int_{\SS_t} \vert\ders^{h+1}k\vert^2\,ds
+ t^{h-1}\int_{\SS_t} \vert\ders^{h}k\vert^2\,ds + Ct^h\right) +C/t^{\theta_h}
$$
with $\theta_h<1$ and some small $\varepsilon_h>0$.

We apply this argument for every even $h$ from 2 to 
$j$, choosing accurately small values $\varepsilon_j$.\\
Hence, we can continue estimate~\eqref{ciaociao} as follows,
\begin{align*}
\frac{d\,}{dt}\int_{{\SS_t}}&\, k^2+ tk_s^2+ \frac{t^2 k_{ss}^2}{2!}+\dots+
\frac{t^j\vert \partial_s^jk\vert^2}{j!}\,ds\\
\leqslant&\,-\varepsilon/2\int_{{\SS_t}} k_s^2 + t k_{ss}^2 + t^2 k_{sss}^2+\dots+
t^j\vert \partial_s^{j+1}k\vert^2\,ds + C + C/t^{\theta_2}+\dots+C/t^{\theta_j}\\
&\, + \dert\sum_{i=1}^3 t^2\qol_5(\lambda^i, k_s^i)
+ t^4\qol_9(\dert\lambda^i,k_{sss}^i)
+\dots+ t^j\qol_{2j+1}(\dert^{j/2-1}\lambda^i,\ders^{j-1}k^i)
\,\biggr\vert_{\text{{ at the $3$--point}}}\\
\leqslant &\, C + C/t^\theta + \dert\sum_{i=1}^3 t^2\qol_5(\lambda^i, k_s^i)
+ t^4\qol_9(\dert\lambda^i,k_{sss}^i)
+\dots+ t^j\qol_{2j+1}(\dert^{j/2-1}\lambda^i,\ders^{j-1}k^i)
\,\biggr\vert_{\text{{ at the $3$--point}}}
\end{align*}
for some $\theta<1$.\\
Integrating this inequality in time on $[0,t]$ with $t\leqslant T_M$ and taking
into account Remark~\ref{qolpol}, we get
\begin{align*}
\int_{{\SS_t}} k^2&\,+ tk_s^2+ \frac{t^2 k_{ss}^2}{2!}+\dots+
\frac{t^j\vert \partial_s^jk\vert^2}{j!}\,ds\\
\leqslant&\, \int_{{\SS_0}} k^2\,ds + CT_M+ CT_M^{(1-\theta)}\\
&\,+ \sum_{i=1}^3 t^2\qol_5(\lambda^i, k_s^i)
+ t^4\qol_9(\dert\lambda^i,k_{sss}^i)
+\dots+ t^j\qol_{2j+1}(\dert^{j/2-1}\lambda^i,\ders^{j-1}k^i)
\,\biggr\vert_{\text{{ at the $3$--point}}}\\
\leqslant&\, \int_{{\SS_0}} k^2\,ds + C
+ t^2\Vert \pol_5(\vert k_s\vert)\Vert_{L^\infty}
+ t^4\Vert \pol_9(\vert k_{sss}\vert)\Vert_{L^\infty}
+\dots+ t^j\Vert \pol_{2j+1}(\vert \ders^{j-1}k\vert) \Vert_{L^\infty}\,.
\end{align*}
Now we absorb all the polynomial terms, after interpolating each
one of them between the corresponding ``good'' integral in the left
member and some power of the $L^2$ norm of $k$, 
as we did in showing Proposition~\ref{pluto1000}, hence we finally
obtain for every even $j\in\NN$,
\begin{equation}\label{keyestimate}
\int_{{\SS_t}} k^2+ tk_s^2
+ \frac{t^2 k_{ss}^2}{2!}
+\dots+ \frac{t^j\vert \partial_s^jk\vert^2}{j!}\,ds \leqslant \overline{C}_j
\end{equation}
with $t\in[0,T_M]$ and a constant $\overline{C}_j$ 
depending only on the constants in assumption~\eqref{endsmooth} and the bounds on $\int_{\SS_0}k^2\,ds$ 
and on the inverses of the lengths of the curves of the initial network $\SS_0$.\\
This family of inequalities clearly implies 
\begin{equation}\label{stimak} 
\int_{{\SS_t}} \vert \partial_s^jk\vert^2 \,ds\leqslant \frac{C_j j!}{t^j}\qquad\text{ { for every even} $j\in\NN$.} 
\end{equation}
Then, passing as before from integral to $L^\infty$ estimates by means of inequalities~\eqref{int2}, 
we have the following proposition.

\begin{prop}\label{topolino5} 
For every $\mu>0$ the curvature and 
all its space derivatives of $\SS_t$ are uniformly bounded in the time interval $[\mu,T_M]$ 
(where $T_M$ is given by 
Proposition~\ref{stimaL}) by some constants depending only on $\mu$, the constants in assumption~\eqref{endsmooth} 
and the bounds on $\int_{\SS_0}k^2\,ds$ 
and on the inverses of the lengths of the curves of the initial network $\SS_0$.
\end{prop}

By means of these a priori estimates, we can now work out some results
about the smooth flow of an initial regular geometrically smooth
network $\SS_0$. Notice that these are examples of how to use the
previous estimates on special smooth flows to get the conclusion
on general flows or even only $C^\infty$ flows, as we mentioned in the
beginning of this section.

 \begin{thm}\label{curvexplod} 
If $[0,T)$, with $T<+\infty$, is the maximal time interval of existence of a $C^\infty$ curvature flow 
of an initial geometrically smooth network $\SS_0$, then 
\begin{enumerate} 
\item either the inferior limit of the length of at least one curve of $\SS_t$ is zero, as $t\to T$, 
\item or $\limup_{t\to T}\int_{\SS_t}k^2\,ds=+\infty$. 
\end{enumerate} 
Moreover, if the lengths of the $n$ curves are uniformly positively bounded from below, 
then this superior limit is actually a limit and 
there exists a positive constant $C$ such that
$$
\int_{{\SS_t}} k^2\,ds \geqslant \frac{C}{\sqrt{T-t}}\,,
$$
for every $t\in[0, T)$.
\end{thm}
\begin{proof}
We can $C^\infty$ reparametrize the flow $\SS_t$ in order that it becomes a special smooth flow $\widetilde{\SS}_t$ in $[0,T)$.\\ 
If the lengths of the curves of $\SS_t$ are uniformly bounded away from zero and the $L^2$ norm of $k$ is bounded, the same holds for the networks $\widetilde{\SS}_t$, then, by Proposition~\ref{unif222} and
Ascoli--Arzel\`a Theorem, the network $\widetilde{\SS}_t$ converges in $C^\infty$ to a smooth network $\widetilde{\SS}_T$ as $t\to
T$. Then, applying Theorem~\ref{smoothexist} to $\widetilde{\SS}_T$
we could restart the flow obtaining a $C^\infty$ special curvature flow in a longer time interval. Reparametrizing back this last flow, we get a $C^\infty$ ``extension'' in time of the flow $\SS_t$, hence contradicting the maximality of the interval
$[0,T)$.\\
Now, considering again the flow $\widetilde{\SS}_t$, by means of differential inequality~\eqref{evolint999}, we have
$$
\frac{d\,}{dt} \int_{{\widetilde{\SS}_t}} \widetilde{k}^2\,ds
\leqslant C \left(\int_{{\widetilde{\SS}_t}} \widetilde{k}^2\,ds\right)^{3} + C
\leqslant C \left(1+\int_{{\widetilde{\SS}_t}} \widetilde{k}^2\,ds\right)^{3}\,,
$$
which, after integration between $t,r\in[0,T)$ with $t<r$, gives
$$
\frac{1}{\left(1+\int_{{\widetilde{\SS}_t}} \widetilde{k}^2\,ds\right)^{2}}
-\frac{1}{\left(1+\int_{{\widetilde{\SS}_r}} \widetilde{k}^2\,ds\right)^{2}}\leqslant C(r-t)\,.
$$
Then, if case $(1)$ does not hold, we can choose 
a sequence of times $r_j\to T$ such that 
$\int_{\widetilde{\SS}_{r_j}} \widetilde{k}^2\,ds\to+\infty$. Putting $r=r_j$ in the inequality above and passing
to the limit, as $j\to\infty$, we get
$$
\frac{1}{\left(1+\int_{{\widetilde{\SS}_t}} \widetilde{k}^2\,ds\right)^{2}}\leqslant C(T-t)\,,
$$
hence, for every $t\in[0, T)$,
$$
\int_{{\widetilde{\SS}_t}} \widetilde{k}^2\,ds \geqslant \frac{{C}}{\sqrt{T-t}}-1\geqslant
\frac{C}{\sqrt{T-t}}\,,
$$
for some positive constant $C$ and $\lim_{t\to T}\int_{{\widetilde{\SS}_t}}k^2\,ds=+\infty$.\\
By the invariance of the curvature by reparametrization, this last estimate implies the same estimate for the flow $\SS_t$.
\end{proof}

This theorem obviously implies the following corollary.

\begin{cor}\label{kexplod} If $[0,T)$, with $T<+\infty$, is the maximal time interval of existence of a $C^\infty$ curvature flow of an initial geometrically smooth network $\SS_0$ and the lengths of the curves are uniformly bounded away from zero, then
\begin{equation}\label{krate}
\max_{\SS_t}k^2\geqslant\frac{C}{\sqrt{T-t}}\to+\infty\,,
\end{equation}
as $t\to T$.
\end{cor}

\begin{rem}\label{kratrem}
In the case of the evolution $\gamma_t$ 
of a single closed curve in the plane 
there exists a constant $C>0$ such that if at time $T>0$ a 
singularity develops, then 
$$
\max_{{\gamma}_t} k^2\geqslant\frac{C}{{T-t}}
$$
for every $t\in[0,T)$ (see~\cite{huisk3}).\\ 
If this lower bound on the rate of blowing up of the curvature (which
is clearly stronger than the one in inequality~\eqref{krate}) holds
also in the case of the evolution of a network is an open problem (even if the network is a triod).
\end{rem}

We conclude this section with the following estimate from below on the maximal time of smooth existence.

\begin{prop}\label{unif333} For every $M>0$ there exists a
 positive time $T_M$ such that if the $L^2$ norm of the curvature and
 the inverses of the lengths of the geometrically smooth network $\SS_0$ are bounded by $M$, then the maximal time of existence $T>0$ of a $C^\infty$ curvature flow of $\SS_0$ is larger than $T_M$.
\end{prop}
\begin{proof}
As before, considering again the reparametrized special curvature flow $\widetilde{\SS}_t$, by Proposition~\ref{stimaL} in the interval 
$[0,\min\{T_M,T\})$ the
$L^2$ norm of $\widetilde{k}$ and the inverses of the lengths of the curves
of $\widetilde{\SS}_t$ are bounded by $2M^2+6M$.\\
Then, by Theorem~\ref{curvexplod}, the value $\min\{T_M,T\}$ cannot
coincide with the maximal time of existence of $\widetilde{\SS}_t$ (hence of $\SS_t$), so it must be $T>T_M$.
\end{proof}

\section{Short--time existence II}\label{smtm2}

In this section, we are going to prove the short--time existence and geometric uniqueness of a curvature flow for a regular initial network $\SS_0$ which is only $C^2$ in a ``natural subclass'' of the curvature flows which are simply $C^2$ in space and $C^1$ in time. Before doing that, we discuss the property of {\em parabolic regularization} for the flow.

\medskip

Let $\SS_t=\bigcup_{i=1}^n\gamma^i([0,1],t)$ be a $C^\infty$ flow by curvature, we discuss what happens if we reparametrize every curve of the network proportionally to arclength.

If we consider smooth functions $\varphi^i:[0,1]\times[0,T)\to[0,1]$ 
and the reparametrizations $\widetilde{\gamma}^i(x,t)={\gamma}^i(\varphi^i(x,t),t)$, 
imposing that $\vert\widetilde{\gamma}_x^i\vert$ is constant, 
we must have that $\vert\gamma_x^i(\varphi^i(x,t),t)\vert\varphi_x^i(x,t)=L^i(t)$
where $L^i(t)$ is the length of the curve $\gamma^i$ at time $t$. It follows that $\varphi^i(x,t)$ can be obtained by integrating the ODE
$$
\varphi_x^i(x,t)=L^i(t)/\vert\gamma_x^i(\varphi^i(x,t),t)\vert
$$
with initial data $\varphi^i(0,t)=0$ and that it is $C^\infty$ 
as $L^i$ and $\gamma^i$ are $C^\infty$.

Being a reparametrization, $\widetilde{\gamma}^i$ is still a $C^\infty$ curvature flow, 
that is, $\widetilde{\gamma}^i_t=\widetilde{k}^i\widetilde{\nu}^i
+\widetilde{\lambda}^i\widetilde{\tau}^i$, we want to determine the functions $\widetilde{\lambda}^i=\langle\widetilde{\gamma}^i_t\,\vert\,\widetilde{\tau}^i\rangle$.
Differentiating this equation in arclength and keeping into account that 
$\widetilde{\gamma}_x(x,t)=L^i(t)\widetilde{\tau}^i(x,t)$, we get
\begin{equation*}
\widetilde{\lambda}^i_s=
\frac{\langle\widetilde{\gamma}^i_{tx}\,\vert\,\widetilde{\tau}^i\rangle}
{\vert\widetilde{\gamma}^i_x\vert}+
\langle\widetilde{\gamma}^i_t\,\vert\,\partial_s\widetilde{\tau}^i\rangle
=
\frac{\langle\partial_t(L^i\widetilde{\tau}^i)\,\vert\,\widetilde{\tau}^i\rangle}{L^i}+\langle 
\widetilde{k}^i\widetilde{\nu}^i
+\widetilde{\lambda}^i\widetilde{\tau}^i\,\vert\,\widetilde{k}^i\widetilde{\nu}^i\rangle
=
\frac{\partial_tL^i}{L^i}+(\widetilde{k}^i)^2\,.
\end{equation*}
This equation immediately says that $\widetilde{\lambda}^i_s- (\widetilde{k}^i)^2$ 
is constant in space. Moreover, by Proposition~\ref{equality1000}, 
$$
\partial_tL^i(t)=\widetilde{\lambda}^i(1,t)-
\widetilde{\lambda}^i(0,t) -\int_{\gamma^i(\cdot,t)}(\widetilde{k}^i)^2\,ds
$$ 
and that the values of $\widetilde{\lambda}^i$ at the end--points or
$3$--points of the network are (uniformly) linearly related to (hence also
bounded by) the values of $\widetilde{k}^i$.
Hence, we can conclude that $\widetilde{\lambda}^i_s$ is bounded by an expression involving $L^i(t)$ and $\Vert\widetilde{k}(\cdot,t)\Vert_{L^\infty}$.

\medskip

We show now that the geometrically unique solution obtained starting from an initial $C^{2+2\alpha}$ network which is geometrically $2$--compatible (which exists, as we proved in Theorem~\ref{2compexist0}) can be actually reparametrized to be a $C^\infty$ curvature flow for every positive time (so that the geometric estimates of Section~\ref{kestimates} can be applied). This clearly can be seen as a (geometric) parabolic regularization property.

\begin{thm}[Existence, uniqueness and smoothness in H\"older spaces]\label{parareg}
For any initial, regular $C^{2+2\alpha}$ network $\SS_0=\bigcup_{i=1}^n\sigma^i([0,1])$, with $\alpha\in(0,1/2)$, which is geometrically $2$--compatible, the geometrically unique solution $\gamma^i$ found in Theorem~\ref{2compexist0} can be reparametrized to be a $C^\infty$ curvature flow on $(0,T)$, that is, the networks $\SS_t=\bigcup_{i=1}^n\gamma^i([0,1],t)$ are geometrically smooth for every positive time (see Definition~\ref{geosmoothdef}).
\end{thm}
\begin{proof}
We first assume that $\SS_0$ satisfies the compatibility conditions of order $2$ for the special flow (namely, it is $2$--compatible).\\
By analyzing the proof of Theorem~\ref{2smoothexist0-triod} given in~\cite{bronsard}, one can see that the solution to system~\eqref{problema-nogauge-general} given by such theorem actually depends continuously in $C^{2+2\alpha,1+\alpha}$ on the initial data $\sigma^i$ 
in the $C^{2+2\alpha}$ norm. Then, we approximate the network 
$\SS_0=\bigcup_{i=1}^n\sigma^i([0,1])$ in $C^{2+2\alpha}$ 
with a family of smooth networks $\SS_j$ with the same end--points, composed of $C^\infty$ 
curves $\sigma_j^i\to\sigma^i$, as $j\to\infty$. Hence, for every $\varepsilon>0$, the smooth 
solutions of system~\eqref{problema-nogauge-general} for these approximating initial 
networks, given by the curves $\gamma_j^i(x,t):
[0,1]\times[0,T-\varepsilon]\to\overline{\Omega}$, converge as $j\to\infty$ in 
$C^{2+2\alpha,1+\alpha}([0,1]\times[0,T-\varepsilon])$ to the solution $\gamma^i$ for the initial network $\SS_0$. By the $C^{2+2\alpha}$--convergence, the inverses of the
lengths of the initial curves, the integrals $\int_{\SS_j} k_j^2\,ds$ and 
$\vert\partial_x\sigma_j^i(x)\vert$ (from above and away from zero) for all the approximating networks are equibounded, thus Proposition~\ref{topolino5} gives uniform
estimates on the $L^\infty$ norms of the curvature and of all its derivatives in every ``rectangle'' $[0,1]\times[\mu,T_M)$, with $\mu>0$ and $T_M\leqslant T$.\\
We now reparametrize every curve $\gamma^i_j(\cdot,t)$ and $\gamma^i(\cdot, t)$ 
proportionally to arclength by some maps $\varphi^i_j$ and $\varphi^i$ as above.
Notice that, since $\gamma^i_j$ and $\gamma^i$ are uniformly bounded in 
$C^{2+2\alpha,1+\alpha}$, we have that the maps $\partial_x \gamma_j^i$ and $\partial_x \gamma^i$ are uniformly bounded in $C^{1+2\alpha,1/2+\alpha}$. 
Hence, by a standard ODE's argument, the reparametrizing maps $\varphi^i_j$ and $\varphi^i$ 
above are also uniformly bounded in $C^{1+2\alpha,1/2+\alpha}$, in particular they are 
uniformly H\"older continuous in space and time. 
This means that the reparametrized maps $\widetilde{\gamma}^i_j$ converge uniformly to 
$\widetilde{\gamma}^i$ which is a (possibly only continuous in $t$) reparametrization of the original flow. It is easy to see that these latter gives a curvature flow of the arclength reparametrized network $\widetilde{\SS}_0=\bigcup_{i=1}^n(\sigma^i\comp\varphi^i(\cdot,0))[0,1]$ which then still belongs to $C^{2+2\alpha}$.\\ 
As the curvature and all its arclength derivatives are invariant under reparametrization and the 
equibounded lengths of the curves, the above uniform estimates hold also for the 
reparametrized maps $\widetilde{\gamma}_j^i$ in every ``rectangle'' $[0,1]\times[\mu,T_M)$. 
Moreover, by the discussion about reparametrizing these curves proportional to arclength, it 
follows that we have uniform estimates also on $\widetilde{\lambda}_j^i$ and all their arclength 
derivatives for these flows in every ``rectangle'' $[0,1]\times[\mu,T_M)$. Hence, the curves 
$\widetilde{\gamma}_j^i$, possibly passing to a subsequence, actually converge in 
$C^\infty([0,1]\times[\mu,T_M))$, for every $\mu>0$, to the limit flow $\widetilde{\gamma}^i$ 
which then belongs to $C^\infty([0,1]\times(0,T))\cap C^0([0,1]\times[0,T))$.\\
If $\SS_0$ is only geometrically $2$--compatible, this procedure can be applied for the 
flow of its $2$--compatible reparametrization, giving the same resulting flow, as the arclength 
reparametrized flow is the same for any two flows differing only for a reparametrization (the 
fact that the flow of a $C^{2+2\alpha}$ geometrically $2$--compatible initial network is a 
reparametrization of the flow of a $2$--compatible $C^{2+2\alpha}$ initial network is stated in 
Remark~\ref{fffggg}).\\
The last step is to find extensions $\theta^i:[0,1]\times[0,T)\to[0,1]$ of the arclength 
reparametrizing maps $\varphi^i(\cdot,0)\in C^{2+2\alpha}$ which are in 
$C^\infty([0,1]\times(0,T))$ and satisfy $\theta^i(x,0)=\varphi^i(x,0)$, $\theta^i(0,t)=0$, 
$\theta^i(1,t)=1$ and $\theta_x^i(x,t)\not=0$ for every $x$ and $t$. This can be done, for 
instance, by means of time--dependent convolutions with smooth kernels. Then, the maps 
$\overline{\gamma}^i(\cdot, t)=\widetilde{\gamma}^i([\theta^i(\cdot,t)]^{-1},t)$ give a curvature flow of the network $\SS_0=\bigcup_{i=1}^n\sigma^i([0,1])$ which becomes immediately $C^\infty$ for every positive time $t>0$.
\end{proof}

As for every positive time, the flow obtained by this theorem is $C^\infty$ 
and hence every network $\SS_t$ is geometrically smooth, again by Remark~\ref{fffggg} this flow can be reparametrized, from any positive time on, to be a $C^\infty$ special smooth flow.\\
This argument can clearly be applied to any $C^{2+2\alpha,1+\alpha}$ curvature flow $\SS_t$ 
in a time interval $(0,T)$, being every network of this flow geometrically $2$--compatible 
(Proposition~\ref{c2geocomp}), simply considering as initial network 
any $\SS_{t_0}$ with $t_0>0$. 

\begin{cor}\label{parareg0}
Given any $C^{2+2\alpha,1+\alpha}$ curvature flow in an interval of time $(0,T)$, for every 
$\mu>0$, the restricted flow $\SS_t$ for $t\in[\mu,T)$ can be reparametrized to be a 
$C^\infty$ special curvature flow in $[\mu,T)$.\\
In particular, this applies to any $C^{2+2\alpha,1+\alpha}$ curvature flow of an initial, regular 
$C^{2+2\alpha}$ geometrically $2$--compatible network $\SS_0=\bigcup_{i=1}^n\sigma^i([0,1])$.
\end{cor}

The parabolic regularization property of the flow also holds when the initial data is of class $W^{2-{2}/{p},p}$. We have the following result for the special flow, whose proof can be found in~\cite[Section~4]{GoMePl}.

\begin{prop}\label{smoothnessthm}
Let $\gamma\in W^{1,2}_p([0,T)\times[0,1])$ be a Sobolev--solution to the special flow in $[0,T)$ with $T>0$ and initial network in $W^{2-{{2}/{p}},p}([0,1])$. Then, $\SS_t=\bigcup_{i=1}^{n}\gamma^{i}([0,1],t)$ are geometrically smooth for all positive times.
\end{prop}

\begin{rem}
The proof is based on the so called ``parameter trick'' of Angenent~\cite{angen3}, which has been generalized to several 
situations~\cite{lunardi1,lunardi2,Prusssimonett}. However, these works do not deal with fully non--linear boundary conditions like
\begin{equation*}
\sum_{i=1}^{3} \frac{\gamma^i_x(0,t)}{\vert\gamma^i_x(0,t)\vert}=0
\end{equation*}
as in the special flow of networks. An adaptation of such ``parameter trick'', allowing also the treatment of fully non--linear boundary conditions, is presented in~\cite[Section~6.6]{goesswein2019Dissertation} and then modified for the application in the Sobolev setting in~\cite[Section~4]{GoMePl}, to get the above result.
\end{rem}

Thanks to the above proposition, we have a complete short--time existence, uniqueness and parabolic smoothing result for Sobolev--solutions. Indeed, combining Theorem~\ref{wellposednessSobolev} and Proposition~\ref{smoothnessthm} we have the following theorem.

\begin{thm}[Existence, uniqueness and smoothness in Sobolev spaces]\label{complete-Sob}
Let $p\in (3,+\infty)$ and $\SS_0$ be a regular network of class $W^{2-{{2}/{p}},p}$. Then there exists a maximal Sobolev--solution 
$\SS_{t\in [0, T_{\max})}$ to the motion by curvature with initial datum $\SS_0$ in the maximal time interval $[0,T)$ which is geometrically unique. Furthermore, the networks $\SS_t=\bigcup_{i=1}^{n}\gamma^{i}([0,1],t)$ are geometrically smooth for all positive times.
\end{thm}

We finally consider a general curvature flow. If we have a curvature flow $\SS_t$ in $[0,T)$ which is $C^2$ in space and $C^1$ in time in $[0,1]\times(0,T)$, then for every positive time $\mu$, the flow is of class $C^{2,1}([0,1]\times[\mu,T))$, in particular, it belongs to $W^{1,2}_p([\mu,T)\times[0,1])$, thus, it must coincide with the unique flow given by the previous theorem of the initial network $\SS_{\mu}$. In particular, by parabolic regularization, it must be a geometrically smooth flow. Being $\mu>0$ is arbitrary, this must hold for such flow on $(0,T)$, hence the flow is smooth for every positive time.

This argument extends Theorem~\ref{parareg} to every curvature flow.

\begin{thm}\label{parareg3}
Every curvature flow as in Definition~\ref{probdef} is geometrically smooth for every positive time.
\end{thm}

A consequence of this ``geometric'' parabolic smoothing theorem is the extension of Theorem~\ref{curvexplod} and Corollary~\ref{kexplod} to any curvature flow. As before, we apply such results to the reparametrized $C^\infty$ special curvature flow given by Corollary~\ref{sunique} (or Corollary~\ref{parareg0}). The conclusions also hold for the original flow since they are concerned only with the curvature and the lengths of the curves, which are invariant by reparametrization.

\begin{thm}\label{curvexplod-general} 
Let $T<+\infty$ be the maximal time interval of existence of a curvature flow $\SS_t$ which is $C^2$ in space and $C^1$ in time in $[0,1]\times(0,T)$, then 
\begin{enumerate} 
\item either the inferior limit of the length of at least one curve of $\SS_t$ is zero, as $t\to T$, 
\item or $\limup_{t\to T}\int_{\SS_t}k^2\,ds=+\infty$, hence
 the curvature is not bounded as $t\to T$.
\end{enumerate} 
Moreover, if the lengths of the $n$ curves are uniformly positively bounded from below, 
then this superior limit is actually a limit and 
there exists a positive constant $C$ such that 
$$
\int_{{\SS_t}} k^2\,ds \geqslant \frac{C}{\sqrt{T-t}}\,\,\text{ and }\,\,
\max_{\SS_t}k^2\geqslant\frac{C}{\sqrt{T-t}}
$$
for every $t\in[0, T)$.
\end{thm}

We can finally show the existence and geometric uniqueness of a curvature flow for a regular initial network $\SS_0$ of class $C^2$, in a ``quite natural'' subclass of of the flows which are $C^2$ in space and $C^1$ in time. 
The parabolic regularization allows us to use the integral estimates of Section~\ref{kestimates} to prove the existence of a solution to the motion by curvature when the initial datum is a regular network of class $C^2$, without requiring any extra condition at the triple junctions and at the end--points. 
Geometric uniqueness is then obtained from the well--posedness in Sobolev spaces.

\begin{thm}\label{c2shorttime}
For any initial $C^2$ regular network $\SS_0=\bigcup_{i=1}^n\sigma^i([0,1])$ there exists a solution $\gamma^i$
of Problem~\eqref{problema} in a maximal time interval $[0,T)$, 
which is continuous in $[0,1]\times[0,T)$ and such that
\begin{itemize}
\item the flow $\SS_t=\bigcup_{i=1}^n\gamma^i([0,1],t)$ is a smooth flow for every $t>0$,
\item the unit tangents $\tau^i$ are continuous in $[0,1]\times[0,T)$, 
\item the functions $k(\cdot,t)$ converge weakly in $L^2$ to $k(\cdot,0)$, as $t\to 0$,
\item the function $t\mapsto\int_{\SS_t}k^2\,ds$ is continuous on $[0,T)$.
\end{itemize}
Moreover, such flow is geometrically unique in the class $\mathcal{N}$ of the curvature flows of $\SS_0$ which are $C^2$ in space and $C^1$ in time, for $t>0$ and such that
\begin{itemize}
\item the unit tangents $\tau^i$ are continuous in $[0,1]\times[0,T)$,
\item the integral $\int_{\SS_t}k^2\,ds$ is locally bounded for $t\in[0,T)$.
\end{itemize}
\end{thm}

\begin{proof}
We can approximate in $W^{2,2}(0,1)$ (hence in $C^1([0,1])$)
the network $\SS_0=\bigcup_{i=1}^n\sigma^i([0,1])$ with a family of smooth networks
$\SS_j$, composed of $C^\infty$ curves $\sigma_j^i\to\sigma^i$, as
$j\to\infty$ with the same end--points and satisfying 
$\partial_x \sigma_j^i(0) =\partial_x \sigma^i(0)$, 
$\partial_x \sigma_j^i(1)=\partial_x \sigma^i(1)$.\\
By the convergence in $W^{2,2}$ and in $C^1$, the inverses of the
lengths of the initial curves, the integrals $\int_{\SS_j} k^2\,ds$ and 
$\vert\partial_x\sigma_j^i(x)\vert$ (from
above and away from zero) for all the approximating networks are
equibounded, thus Proposition~\ref{unif333} assures the existence
of a uniform interval $[0,T)$ of existence of smooth
evolutions given by the curves
$\gamma_j^i(x,t):[0,1]\times[0,T)\to\overline{\Omega}$.\\
Now, for the same reason, Proposition~\ref{topolino5} gives uniform
estimates on the $L^\infty$ norms of the curvature and of all its
derivatives in every rectangle $[0,1]\times[\mu,T_M)$, with $\mu>0$.\\
This means that if we reparametrize at every time all the curves
$\gamma_j^i$ proportional to their arclength, by means of a diagonal
argument, we can find a subsequence of the family of reparametrized
flows $\widetilde{\gamma}^i_j$ which converges in
$C^\infty\loc ([0,1]\times(0,T))$ to some flow, parametrized
proportional to its arclength, $\widetilde{\gamma}^i$ in the time
interval $(0,T)$. Moreover, by the hypotheses, the curves of the
initial networks $\widetilde{\sigma}^i_j$ converge in $W^{2,2}(0,1)$
to $\widetilde{\sigma}^i$ which are the reparametrizations,
proportional to their arclength, of the curves $\sigma^i$ of the
initial network $\SS_0$. If we show that the maps
$\widetilde{\gamma}^i$ are continuous up to the time $t=0$ we have a
curvature flow for the network
$\widetilde{\SS}_0=\bigcup_{i=1}^n\widetilde{\sigma}^i([0,1])$ which then
gives a curvature flow for the original network $\SS_0$ in
$C^\infty([0,1]\times(0,T))$, reparametrizing it back with some family
of continuous maps $\theta^i:[0,1]\times[0,T)\to[0,1]$ with
$\theta_x^i\not=0$ everywhere, $\theta^i\in
C^\infty([0,1]\times(0,T))$ and
$\widetilde{\sigma}^i(\theta^i(\cdot,0))=\sigma^i$ (this can be easily
done as the maps $\theta^i(\cdot,0)$ are of class $C^2$, since in
general, the arclength reparametrization maps have the same regularity
of the network).\\
Hence, we deal with the continuity up to $t=0$ of the maps
$\widetilde{\gamma}^i$. By the uniform $L^2$ bound on the curvature
and the parametrization proportional to the arclength, the theorem of
Ascoli--Arzel\`a implies that for every sequence of times $t_l\to0$,
the curves $\widetilde{\gamma}^i(\cdot,t_l)$ have a converging
subsequence in $C^1([0,1])$ to some family of limit curves
${\zeta}^i:[0,1]\to\overline{\Omega}$, still parametrized proportionally
to arclength, by the $C^1$--convergence. Moreover, we can also assume that $k(\cdot,t_l)$ 
converge weakly in $L^2(ds)$ to the curvature function associated with the family of curves 
$\zeta^i$. We want to see that actually
$\zeta^i=\widetilde{\sigma}^i$, hence showing that the flow
$\widetilde{\gamma}^i:[0,1]\times[0,T)\to\overline{\Omega}$ is
continuous and that the unit tangent vector
$\tau:[0,1]\times[0,T)\to\R^2$ is a continuous map up to the time
$t=0$ (this property is stable under the above reparametrization so it
then will hold also for the final curvature flow $\gamma^i$).\\
We consider a function $\varphi\in C^\infty(\R^2)$ and the time
derivative of its integral on the evolving networks $\widetilde{\gamma}^i_j$, that is,
\begin{align*}
\frac{d\,}{dt}\int_{\widetilde{\SS}_j(t)}\varphi\,ds=
&\,\int_{\widetilde{\SS}_j(t)}\varphi(\widetilde{\lambda}_s-\widetilde{k}^2)\,ds
+\int_{\widetilde{\SS}_j(t)}\langle\nabla\varphi\,\vert\,\widetilde{\underline{k}}+\widetilde{\underline{\lambda}}\rangle\,ds\\
=&\,-\int_{\widetilde{\SS}_j(t)}\varphi\widetilde{k}^2\,ds
-\int_{\widetilde{\SS}_j(t)}\langle\nabla\varphi\,\vert\,\widetilde{\tau}\rangle\widetilde{\lambda}\,ds
+\int_{\widetilde{\SS}_j(t)}\langle\nabla\varphi\,\vert\,\widetilde{\underline{k}}+\widetilde{\underline{\lambda}}\rangle\,ds\\
=&\,-\int_{\widetilde{\SS}_j(t)}\varphi\widetilde{k}^2\,ds
+\int_{\widetilde{\SS}_j(t)}\langle\nabla\varphi\,\vert\,\widetilde{\underline{k}}\rangle\,ds\,,
\end{align*}
where we integrated by parts, passing from first to second line.\\
Let us consider now any sequence of times $t_l$ converging to zero as above, such that the curves $\widetilde{\gamma}^i(\cdot,t_l)$ converge in $C^1([0,1])$ to some family of limit curves
${\zeta}^i:[0,1]\to\overline{\Omega}$ (still parametrized proportionally
to arclength) as above, describing some regular network $\overline{\SS}$ and $k(\cdot,t_l)$ converge weakly in $L^2(ds)$ to the curvature function associated to the family of curves $\zeta^i$.
Integrating this equality in the time interval $[0,t_l]$ we get
$$
\int_{\widetilde{\SS}_j(t_l)}\varphi\,ds-
\int_{\widetilde{\SS}_j(0)}\varphi\,ds
=-\int^{t_l}_0\int_{\widetilde{\SS}_j(t)}\varphi\widetilde{k}^2\,ds\,dt
+\int^{t_l}_0\int_{\widetilde{\SS}_j(t)}\langle\nabla
\varphi\,\vert\,\widetilde{\underline{k}}\rangle\,ds\,dt
$$
which clearly passes to the limit as $j\to\infty$, by the smooth
convergence of the flows $\widetilde{\gamma}^i_j$ to the flow
$\widetilde{\gamma}^i$ (and the uniform bound on
$\int_{\widetilde{\SS}_j(t)}\widetilde{k}^2\,ds$) and of the initial networks
$\widetilde{\SS}_j(0)=\bigcup_{i=1}^n\widetilde{\sigma}^i_j([0,1])$ to
$\widetilde{\SS}_0=\bigcup_{i=1}^n\widetilde{\sigma}^i([0,1])$, hence,
$$
\int_{\widetilde{\SS}_{t_l}}\varphi\,ds-
\int_{\widetilde{\SS}_0}\varphi\,ds
=-\int^{t_l}_0\int_{\widetilde{\SS}_t}\varphi\widetilde{k}^2\,ds\,dt
+\int^{t_l}_0\int_{\widetilde{\SS}_t}\langle\nabla
\varphi\,\vert\,\widetilde{\underline{k}}\rangle\,ds\,dt
$$
By the uniform bound on the $L^2$ norm of the curvature, we then get
$$
\biggl\vert\int_{\widetilde{\SS}_{t_l}}\varphi(\widetilde{\gamma}(\cdot,t_l))\,ds-
\int_{\widetilde{\SS}_0}\varphi(\widetilde{\sigma})\,ds\,\biggr\vert\leqslant Ct_l\,,
$$
where we made explicit the integrands, for the sake of clarity. Sending
$l\to\infty$ we finally obtain
$$
\biggl\vert\int_{\overline{\SS}}\varphi(\zeta)\,ds-
\int_{\widetilde{\SS}_0}\varphi(\widetilde{\sigma})\,ds\,\biggr\vert=0\,,
$$
that is,
$$
\int_{\overline{\SS}}\varphi\,ds=
\int_{\widetilde{\SS}_0}\varphi\,ds
$$
for every function $\varphi\in C^\infty(\R^2)$.\\
Since, both the networks $\widetilde{\SS}_0=\bigcup_{i=1}^n\widetilde{\sigma}^i([0,1])$ and $\overline{\SS}=\bigcup_{i=1}^n\zeta^i([0,1])$ are
$C^1$, regular and parametrized proportionally to their arclength, this
equality for every $\varphi\in C^\infty(\R^2)$ implies that 
$\widetilde{\sigma}^i=\zeta^i$, which is what we
wanted.

Notice that, the continuity of $\gamma^i$ and $\tau$ also implies that the measures
$\HH^1\res\SS_t$ weakly$^{{\displaystyle{\star}}}$ converge to $\HH^1\res\SS_0$, where
$\HH^1$ is the one--dimensional Hausdorff measure, as $t\to 0$.

Finally, integrating on $[0,t)$ inequality~\eqref{evolint999} 
for the approximating flows $\widetilde{\gamma}_j^i$, and passing to the limit as $j\to\infty$, we see that
\[
\limsup_{t\to 0^+}\int_{\widetilde{\SS}_t}k^2\,ds\le \int_{\widetilde{\SS}_0}k^2\,ds.
\]
Since the function $t\to \int_{\widetilde{\SS}_t}k^2\,ds$ is lower semicontinuous,
we then get that such function is indeed {\em continuous} on $[0,T)$ (also at $t=0$). Being such integral invariant by
reparametrization, this also holds for the flow $\gamma^i$. The same for the weak convergence in $L^2(ds)$ of the functions $k(\cdot,t)$ to $k(\cdot,0)$ as $t\to 0$.

Let now $\SS_t$ be any curvature flow of $\SS_0$ in $[0,T)$, belonging to the class $\mathcal{N}$ of flows as in the statement. By estimates~\eqref{keyestimate} (with $j=2$) we have
\begin{equation}\label{eqcar32324}
\int_{{\widetilde{\SS}_t}} k^2+ tk_s^2\,ds\leqslant\int_{{\widetilde{\SS}_t}} k^2+ tk_s^2+ t^2k_{ss}^2\,ds \leqslant C\qquad\text { hence }\qquad\Vert k_s\Vert_{L^2}\leqslant C /t^{1/2}
\end{equation}
for every $t\in[0,T)$, with a constant $C$ depending only on the inverses of the lengths of the curves of the initial network $\SS_0$ and on $\int_{\SS_0}k^2\,ds$. Taking into account Proposition~\ref{stimaL} uniformly bounding from below the lenghts of the curves of the evolving network in a time interval $[0,\widehat{T}]$ (with $\widehat{T}$ depending only on the initial network), by means of Gagliardo--Nirenberg interpolation inequalities in Proposition~\ref{gl}, we have the estimate
\begin{equation}
{\Vert k \Vert}_{L^\infty} \leqslant C{\Vert k_s\Vert}_{L^2}^{1/2}{\Vert k\Vert}_{L^2}^{1/2}+C{\Vert 
k\Vert}_{L^2} \leqslant C{\Vert k_s\Vert}_{L^2}^{1/2}+C
\end{equation}
where the constant is independent of $t\in[0,\widehat{T}]$. Hence, by inequality~\eqref{eqcar32324},
\begin{equation}
{\Vert k(\cdot,t)\Vert}_{L^\infty} \leqslant C/t^{1/4}+C\qquad\text{ and }\qquad \int_0^{\widehat{T}}\int_{{\widetilde{\SS}_t}} k^{7/2}\,ds\,dt\leqslant C\int_0^{\widehat{T}}C/t^{7/8}+C\,dt\leqslant C
\end{equation}
meaning that $k\in L^{7/2}([0,1]\times[0,\widehat{T}])$. Reparametrizing the flow as at the beginning of this section so that every curve becomes parametrized proportionally to its arclength, we have a new flow $\widetilde{\gamma}^i(x,t)={\gamma}^i(\varphi^i(x,t),t)$ with $\vert\widetilde{\gamma}_x^i\vert$ constantly equal to $L^i(t)$, the length of the curve $\gamma^i$ at time $t\in[0,T)$, by means of reparametrizations $\widetilde{\gamma}^i(x,t)={\gamma}^i(\varphi^i(x,t),t)$ solving the ODE's
$$
\varphi_x^i(x,t)=L^i(t)/\vert\gamma_x^i(\varphi^i(x,t),t)\vert
$$
with initial data $\varphi^i(0,t)=0$. Moreover, we have seen that letting $\widetilde{\gamma}^i_t=\widetilde{k}^i\widetilde{\nu}^i
+\widetilde{\lambda}^i\widetilde{\tau}^i$, we have 
\begin{equation*}
\widetilde{\lambda}^i_s=\frac{\partial_tL^i}{L^i}+(\widetilde{k}^i)^2\,.
\end{equation*}
This equation immediately says that $\widetilde{\lambda}^i_s- (\widetilde{k}^i)^2$ 
is constant in space, then integrating 
\begin{equation}\label{eqcar9898}
|\widetilde{\lambda}(s,t)|\leqslant |\widetilde{\lambda}(0,t)|+|\partial_t L^i(t)| +\int_{\gamma^i(\cdot,t)}(\widetilde{k}^i)^2\,ds\leqslant C\Vert \widetilde{k}(\cdot,t)\Vert_{L^\infty}+C,
\end{equation}
for every $s\in[0,L^i(t)]$, as $|\widetilde{\lambda}(\cdot,t)|$ at the borders of any curve is estimated by $C\Vert \widetilde{k}(\cdot,t)\Vert_{L^\infty}$, $\int_{\gamma^i(\cdot,t)}(\widetilde{k}^i)^2\,ds$ is invariant by reparametrization and bounded by hypotheses and
$$
|\partial_t L^i(t)|=\bigg|\lambda^i(1,t)-\lambda^i(0,t)-\int_{\gamma^i(\cdot,t)}(k^i)^2\,ds\,\bigg|\leqslant \Vert \widetilde{k}(\cdot,t)\Vert_{L^\infty}+C.
$$
It follows 
$$
\Vert\widetilde{\lambda}(\cdot,t)\Vert_{L^\infty}\leqslant C\Vert \widetilde{k}(\cdot,t)\Vert_{L^\infty}+C=C\Vert {k}(\cdot,t)\Vert_{L^\infty}+C\leqslant C/t^{1/4}+C,
$$
hence, $\widetilde{k},\widetilde{\lambda}\in L^{7/2}([0,1]\times[0,\widehat{T}])$. As a consequence, $\widetilde{\gamma}^i_t=\widetilde{k}^i\widetilde{\nu}^i+\widetilde{\lambda}^i\widetilde{\tau}^i\in L^{7/2}([0,1]\times[0,\widehat{T}])$ and being 
$\widetilde{\gamma}^i_{xx}=\widetilde{k}\widetilde{\nu}/(L^i)^2$, also $\widetilde{\gamma}^i_{xx}$ belongs to $L^{7/2}([0,1]\times[0,\widehat{T}])$, hence this flow $\widetilde{\gamma}$ belongs to $W^{1,2}_{7/2}([0,\widehat{T})\times[0,1])$, thus it is geometrically uniquely determined, by Theorem~\ref{wellposednessSobolev} (or~\ref{complete-Sob}).\\
This argument shows that any two curvature flows in the class $\mathcal{N}$ can be reparametrized one to the other, that is, we have geometric uniqueness in this class and we are done.
\end{proof}

\begin{rem}\label{oprob2}\hspace{.5truecm}
\begin{enumerate}
\item We underline that the initial network is not required to satisfy any compatibility condition, but only to have angles of $120$ degrees between the concurring curves at every $3$--point, that is, to be regular and $C^2$. In particular, it is not necessary that the sum of the three curvatures at the $3$--points is zero.
\item As for every positive time the flow obtained by this theorem is $C^\infty$, hence every network $\SS_t$ is geometrically smooth, arguing as before (by means of Remark~\ref{fffggg}), Corollary~\ref{parareg0} applies: this flow can be reparametrized, from any positive time on, to be a $C^\infty$ special smooth flow. 
\item It should be noticed that if the initial curves $\sigma^i$ are $C^\infty$, the flow $\SS_t$ is smooth till $t=0$ far from the $3$--points, that is, in any closed ``rectangle'' included in $(0,1)\times[0,T)$ we can locally reparametrize the curves $\gamma^i$ to get a smooth flow up to $t=0$. This follows from the local estimates for the motion by curvature (see~\cite{eckhui2}).
\item A natural question is whether uniqueness of the curvature flow of an initial regular $C^2$ network holds also ``outside'' of the subclass $\mathcal{N}$, in the general class of curvature flows as in Definition~\ref{probdef} (or possibly asking only the continuity of the tangent vectors as $t\to 0$).  
At the moment this is still an open problem.
\end{enumerate}
\end{rem}

Now that we have gained the short--time existence for an initial regular $C^2$ network, the next important question is what can be said if the initial network does not satisfy the $120$ degrees condition, that is, it is non--regular (even if all its curves are 
$C^\infty$). We will face this question in Section~\ref{smtm3} below.
Clearly, the unit tangent vectors of any curvature flow having as an initial network a configuration that does not satisfy the $120$ degrees condition cannot be continuous up to time $t=0$, being a curvature flow $C^2$ and regular for positive time. 
Anyway, notice that in the definition of curvature flow, we require only that the maps $\gamma^i$ are continuous in $[0,1]\times(0,T)$ for some positive time $T$, hence one could hope to be able to find a curvature flow such that the $120$ degrees condition is satisfied instantaneously, at every positive time $t>0$, as it happens for the geometrical smoothness in Theorem~\ref{parareg}.

In Section~\ref{smtm3} we will also treat the problem of the evolution of a non--regular network with multi--points of order greater than three. In this case, the continuity condition at $t=0$ has to be suitably stated, since, if we want the curvature flow to be regular for every positive time, 
the collection of maps describing the network, as well as the topological structure of the network, must change.

\section{Smooth flows are Brakke flows} 

To continue the flow when at some time a curve collapses and possibly some multi--points appear in the (limit) network, one can consider a more general ({\em weak}) definition of curvature flow.

As mentioned in the introduction, there exist several weak definitions of motion by curvature of a subset of $\R^n$. Among the existing notions,
the most suitable to our point of view is the one introduced by Brakke in~\cite{brakke}, which in general lacks uniqueness but at least
maintains the (Hausdorff) dimension of the evolving sets.

We introduce now the concept of {\em Brakke flow} (with equality) of a network.
\begin{defn}\label{brk}
A {\em regular Brakke flow} is a family of $W^{2,2}\loc$ networks $\SS_t$ in $\Omega$, satisfying the inequality
\begin{equation}\label{brakkeqqqineq}
\frac{\overline{d\,\,}}{dt}\int_{\SS_t}\varphi(\gamma,t)\,ds\leqslant
-\int_{\SS_t}\varphi(\gamma,t) k^2\,ds
+\int_{\SS_t}\langle\nabla\varphi(\gamma,t)\,\vert\,\underline{k}\rangle\,ds
+ \int_{\SS_t}\varphi_t(\gamma,t)\,ds\,,
\end{equation}
for every {\em non negative} smooth function with compact support
$\varphi:\Omega\times[0,T)\to\R$ and $t\in[0,T)$, where 
$\frac{\overline{d\,\,}}{dt}$ is the {\em upper} derivative (the
$\limup$ of the incremental ratios).

If the time derivative at the left-hand side exists and the inequality is equality, for every smooth function with compact support
$\varphi:\Omega\times[0,T)\to\R$ and $t\in[0,T)$, that is,
\begin{equation}\label{brakkeqqq}
\frac{d\,\,}{dt}\int_{\SS_t}\varphi(\gamma,t)\,ds=
-\int_{\SS_t}\varphi(\gamma,t) k^2\,ds
+\int_{\SS_t}\langle\nabla\varphi(\gamma,t)\,\vert\,\underline{k}\rangle\,ds
+ \int_{\SS_t}\varphi_t(\gamma,t)\,ds\,,
\end{equation}
we say that $\SS_t$ is a regular Brakke flow {\em with equality}.
\end{defn}

\begin{rem}
The original definition of Brakke flow given
in~\cite[Section~3.3]{brakke} (in any dimension and codimension)
allows the networks $\SS_t$ to be simply one--dimensional countably 
rectifiable subsets of $\R^2$, with possible integer multiplicity $\theta_t:\SS_t\to\NN$ and with a distributional notion of tangent space and (mean) curvature, called {\em rectifiable varifolds} (see~\cite{simon}). With such a general definition, the networks are identified with the associated Radon measures $\mu_t=\theta_t\HH^1\res\SS_t$.\\
More precisely, the inequality 
\begin{align}\label{braorig}
\frac{\overline{d\,\,}}{dt}\int_{\SS_t}\varphi(x,t)\theta_t(x)\,d\HH^1(x)\leqslant
&\,-\int_{\SS_t}\varphi(x,t) k^2(x,t)\theta_t(x)\,d\HH^1(x)
+\int_{\SS_t}\langle\nabla\varphi(x,t)
\,\vert\,\underline{k}(x,t)\rangle\theta_t(x)\,d\HH^1(x)\\
&\,+ \int_{\SS_t}\varphi_t(x,t)\theta_t(x)\,d\HH^1(x)\,,\nonumber
\end{align}
must hold for every non-negative smooth function with compact support
$\varphi:\Omega\times[0,T)\to\R$ and $t\in[0,T)$, where $\HH^1$ 
is the Hausdorff one--dimensional measure in $\R^2$.\\
These weak conditions were introduced by Brakke
in order to prove an existence result~\cite[Section~4.13]{brakke} for
a family of initial sets much wider than networks of curves, but, on the
other hand, it opens the possibility of instantaneous vanishing of
some parts of the sets during the evolution.
\end{rem}

A big difference between Brakke flows and the evolutions obtained as
solutions of Problem~\eqref{problema} is that the former networks are
simply considered as {\em subsets} of $\R^2$ without any mention to their
parametrization (that clearly is not unique). This means that actually
a Brakke flow can be a family of networks given by the maps
$\gamma^i(x,t)$ which are $C^2$ in space, but possibly do not have
absolutely any regularity with respect to the time variable $t$.

An open question is whether any Brakke flows with equality, possibly
under some extra hypotheses, admits a reparametrization such that it
becomes a solution of Problem~\eqref{problema}.\\
This problem is also related to the uniqueness of the Brakke
flows with equality (maybe further restricting the candidates to a
special class with extra geometric properties).

\begin{prop}\label{equality1001} Any solution of
 Problem~\eqref{problema} in 
$C^{2,1}([0,1]\times[0,T))$ is a regular Brakke flow with equality.\\
In particular, for every curve $\gamma^i(\cdot,t)$ and for every time
$t\in[0,T)$ we have
\begin{equation}\label{levol}
\frac{dL^i(t)}{dt}=\lambda^i(1,t)-\lambda^i(0,t)-\int_{\gamma^i(\cdot,t)}k^2\,ds
\end{equation}
and
$$
\frac{dL(t)}{dt}=-\int_{\SS_t}k^2\,ds\,,
$$
that is, the total length $L(t)$ is decreasing in time and it is uniformly
bounded by the length of the initial network $\SS_0$.
\end{prop}
\begin{proof}
If the flow $\gamma^i$ is in $C^\infty([0,1]\times[0,T))$, we have
\begin{align*}
\frac{dL^i(t)}{dt}=&\,\frac{d\,}{dt}\int_0^1 \vert \gamma_x^i\vert \,dx\\
=&\,\int_0^1 \frac{\langle \gamma_{xt}^i\,\vert\,\gamma_x^i\rangle}{\vert\gamma_x^i\vert} \,dx\\
=&\,\int_0^1 \biggl\langle\partial_x\gamma_t^i\,\biggr\vert\biggl.\, \frac{\gamma_x^i}{\vert\gamma_x^i\vert}\biggr\rangle \,dx\\
=&\,\int_0^1 \langle\partial_x\gamma_t^i\,\vert\,\tau^i\rangle \,dx\\
=&\,\langle\gamma_t^i(1,t)\,\vert\,\tau^i(1,t)\rangle 
- \langle\gamma_t^i(0,t)\,\vert\,\tau^i(0,t)\rangle
- \int_0^1 \langle\gamma_t^i\,\vert\,\partial_x\tau^i\rangle\,dx\,.
\end{align*}
Then, approximating the maps $\gamma^i$ with a family of maps 
$\gamma^{i\varepsilon}\in C^\infty$ such that $\gamma^{i\varepsilon}\to \gamma^i$
in $C^1$ and $\gamma^{i\varepsilon}_{xx}\to \gamma_{xx}^i$ in $C^0$,
as $\varepsilon \to 0$, we see that we can pass to the limit in this
formula and conclude that it holds for the original flow which is only
in $C^{2,1}([0,1]\times[0,T))$. Finally, since $\partial_x\tau^i=k^i\nu^i\vert\gamma^i_x\vert$, we
get
$$
\frac{dL^i(t)}{dt}=\lambda^i(1,t) - \lambda^i(0,t) -
\int_{\gamma^i(\cdot,t)}k^2\,ds
$$
as $\gamma^i_t=k^i\nu^i+\lambda^i\tau^i$.\\
The formula for the derivative of the total length of the evolving
network then follows by the zero--sum property of the functions
$\lambda^i$ at every $3$--point at the fact that all the $\lambda^i$ are
zero at the end--points.

A similar argument shows that formula~\eqref{brakkeqqq} 
defining a regular Brakke flow with equality also holds.
\end{proof}

\begin{thm}\label{brakkeevolution} If $\SS_t$ is a curvature flow of a $C^2$ initial network such that 
\begin{itemize}
\item the unit tangents $\tau^i$ are
continuous in $[0,1]\times[0,T)$,
\item the functions $k(\cdot,t)$ converge weakly in $L^2$ to $k(\cdot,0)$, as $t\to 0$,
\item the function $t\mapsto\int_{\SS_t}k^2\,ds$ is continuous on $[0,T)$, 
\end{itemize}
then $\SS_t$ is a regular Brakke flow with equality.
\end{thm}
\begin{proof}
By the previous Theorem~\ref{equality1001}, we only need to check Brakke equality~\eqref{brakkeqqq} at $t=0$.

For every positive time and for every smooth test function
$\varphi:\overline{\Omega}\times[0,T)\to\R$, we have
$$
\frac{d\,}{dt}\int_{\SS_t}\varphi\,ds=
-\int_{\SS_t}\varphi\,k^2\,ds
+\int_{\SS_t}\langle\nabla\varphi\,\vert\,\underline{k}\rangle\,ds\,d
+\int_{\SS_t}\varphi_t\,ds\,,
$$
hence, it suffices to show that the right member is continuous at $t=0$. By the hypotheses, the only term that really need to be checked is $\int_{\SS_t}\varphi\,k^2\,ds$, we separate it as the sum of $\int_{\SS_t}\varphi^+\,k^2\,ds$ and $\int_{\SS_t}\varphi^-\,k^2\,ds$ and we show the continuity of these two terms separately (here $\varphi^+=\varphi\wedge 0$ and $\varphi^-=\varphi\lor0$).
Thus, we assume that $0\leqslant\varphi\leqslant1$, then, by the weak convergence in $L^2(ds)$ of $k(\cdot,t)$ to $k(\cdot,0)$, the integral $\int_{\SS_t}\varphi\,k^2\,ds$ is lower semicontinuous in $t$, that is, 
$\int_{\SS_0}\varphi\,k^2\,ds\leqslant\liminf_{t_l\to
 0}\int_{\SS_t}\varphi\,k^2\,ds$ for every $t_l\to 0$, but if this is
not an equality for some sequence of times, it cannot happen that
$\int_{\SS_t}k^2\,ds$ is continuous at $t=0$, indeed, we would have
\begin{align*}
\lim_{t_l\to0}\int_{\SS_t}k^2\,ds
\geqslant&\,\liminf_{t_l\to0}\int_{\SS_t}\varphi\,k^2\,ds
+\liminf_{t_l\to0}\int_{\SS_t}(1-\varphi)\,k^2\,ds\\
>&\,\int_{\SS_0}\varphi\,k^2\,ds
+\int_{\SS_0}(1-\varphi)\,k^2\,ds\, =
\int_{\SS_t}k^2\,ds\,.
\end{align*}
This concludes the proof.
\end{proof}

\begin{cor}
The curvature flows whose short--time existence is proved in Theorems~\ref{2compexist0} and~\ref{geosmoothexist} are regular Brakke flows with equality. The curvature flow of an initial $C^2$ regular network obtained in Theorem~\ref{c2shorttime} is also a regular Brakke flow with equality. Any curvature flow of a regular network is a regular Brakke flow with equality for every positive time.
\end{cor}

We conclude this section with the following property of Brakke flows.

\begin{prop}\label{brakprop} For any regular Brakke flow with equality (hence, for every curvature flow of a regular network) such that the curvature is uniformly bounded in a time interval $[0,T)$, the lengths of the curves of the network $L^i(t)$ converge to some limit, as $t\to T$.\\
In particular, if the flow satisfies the conclusions of Theorem~\ref{curvexplod-general} at the maximal time of existence
$T$, there must be at least one curve such that $L^i(t)\to 0$, as $t\to T$.
\end{prop}
\begin{proof}
If the curvature is bounded, by formula~\eqref{levol}, any function
$L^i$ as a uniformly bounded derivative, as $k$ controls $\lambda$ at the
end--points and $3$--points of the network, thus the conclusion follows.
\end{proof}

\section{The monotonicity formula and the rescaling procedures}\label{monotonsec}

Let $F : \SS\times [0,T)\to\mathbb{R}^2$ be the curvature flow of a
regular network in its maximal time interval of existence. As before,
with a little abuse of notation, we will write $\tau(P^r,t)$ and
$\lambda(P^r,t)$ respectively for the unit tangent vector and the
tangential velocity at the end--point $P^r$ of the curve of the
network getting at such point, for any $r\in\{1,2,\dots,l\}$.

A modified form of Huisken's monotonicity formula for smooth
hypersurfaces moving by mean curvature (see~\cite{huisk3}), holds. It
can be proved to start by formula~\eqref{brakkeqqq} and with a slight
modification of the computation in the proof of Lemma~6.3
in~\cite{mannovtor}.\\
Let $x_0\in\R^2, t_0 \in (0,+\infty)$ and
$\rho_{x_0,t_0}:\R^2\times[-\infty,t_0)$ be the one--dimensional {\em
backward heat kernel} in $\R^2$ relative to $(x_0,t_0)$, that is, 
$$
\rho_{x_0,t_0}(x,t)=\frac{e^{-\frac{\vert x-x_0\vert^2}{4(t_0-t)}}}{\sqrt{4\pi(t_0-t)}}\,.
$$
We will often write $\rho_{x_0}(x,t)$ to denote $\rho_{x_0,T}(x,t)$
(or $\rho_{x_0}$ to denote $\rho_{x_0,T}$), when $T$ is 
the maximal (singular) time of existence of a smooth curvature flow.

\begin{prop}[Monotonicity formula]\label{promono}
Assume $t_0>0$. For every $x_0\in\R^2$ and $t\in [0,\min\{t_0,T\})$ the following identity holds 
\begin{align}
\frac{d\,}{dt}\int_{{\SS_t}} \rho_{x_0,t_0}(x,t)\,ds= &\,
-\int_{{\SS_t}} \left\vert\,\underline{k}
+\frac{(x-x_0)^{\perp}}{2(t_0-t)}\right\vert^2 \rho_{x_0,t_0}(x,t)\,ds\label{eqmonfor}\\ &\,
+\sum_{r=1}^l\biggl[\biggl\langle\,\frac{P^r-x_0}{2(t_0-t)}\,\biggr\vert\,
\tau(P^r,t)\,\biggr\rangle -\lambda(P^r,t)\,\biggl]\,
\rho_{x_0,t_0}(P^r,t)\,.\nonumber 
\end{align} 
Integrating between $t_1$ and $t_2$ with $0\leqslant t_1\leqslant t_2<\min\{t_0,T\}$ we get 
\begin{align}\label{eqmonfor2} 
\int_{t_1}^{t_2}\int_{{\SS_t}} \left\vert\,
\underline{k}+\frac{(x-x_0)^{\perp}}{2(t_0-t)}\right\vert^2 &\rho_{x_0,t_0}(x,t)\,ds\,dt 
= \,\int_{{\SS_{t_1}}} \rho_{x_0,t_0}(x,t_1)\,ds - 
\int_{{\SS_{t_2}}} \rho_{x_0,t_0}(x,t_2)\,ds\\ 
&\,+\sum_{r=1}^l\int_{t_1}^{t_2} \biggl[\biggl\langle\,\frac{P^r-x_0}{2(t_0-t)}\,\biggr\vert\,
\tau(P^r,t)\,\biggr\rangle -\lambda(P^r,t)\,\biggl]\,
\rho_{x_0,t_0}(P^r,t)\,dt\,.\nonumber
\end{align} 
\end{prop}

We need the following lemma to estimate the end--points contributions in this formula (its proof is analogous to the one of Lemma~6.5 in~\cite{mannovtor}).

\begin{lem}\label{stimadib} 
If $t_0\in(0,T]$, for every $r\in\{1,2,\dots,l\}$ and $x_0\in\R^2$, the following estimate holds 
\begin{equation*} 
\int_{t}^{t_0}\biggl\vert\biggl\langle\,\frac{P^r-x_0}{2(t_0-\xi)}\,\biggr\vert\,
\tau(P^r,\xi)\,\biggr\rangle -\lambda(P^r,\xi)\,\biggl\vert\,
\rho_{x_0,t_0}(P^r,\xi)\,d\xi\,\leqslant C\,, 
\end{equation*} 
for every $t\in[0,t_0)$, where $C$ is a constant depending only on the constant $C_0$ in assumption~\eqref{endsmooth} (independent of $t_0$ and $t$). It follows that the integral
\begin{equation*} 
\int_{t}^{t_0}\biggl[\biggl\langle\,\frac{P^r-x_0}{2(t_0-\xi)}\,\biggr\vert\,
\tau(P^r,\xi)\,\biggr\rangle -\lambda(P^r,\xi)\,\biggl]\,
\rho_{x_0,t_0}(P^r,\xi)\,d\xi 
\end{equation*} 
exists and it is finite, for every $t_0\in(0,T]$ and $t\in[0,t_0)$.\\
As a consequence, for every point $x_0\in\R^2$ and $t_0\in(0,T]$, we have 
\begin{equation*} 
\lim_{t\to t_0}\sum_{r=1}^l\int_{t}^{t_0}\biggl[\biggl\langle\,\frac{P^r-x_0}{2(t_0-\xi)}\,\biggr\vert\,
\tau(P^r,\xi)\,\biggr\rangle -\lambda(P^r,\xi)\,\biggl]\,
\rho_{x_0,t_0}(P^r,\xi)\,d\xi=0\,. 
\end{equation*} 
\end{lem}
By formula~\eqref{eqmonfor2} and this lemma, we can then write
\begin{align*}
\int_{{\SS_{t}}} \rho_{x_0,t_0}(x,t)\,ds=&\,\int_{{\SS_{0}}} \rho_{x_0,t_0}(x,0)\,ds-\int_{0}^{t}\int_{{\SS_\xi}} \left\vert\,
\underline{k}+\frac{(x-x_0)^{\perp}}{2(t_0-\xi)}\right\vert^2\rho_{x_0,t_0}(x,\xi)\,ds\,d\xi\\ 
&\,+\sum_{r=1}^l\int_{0}^{t} \biggl[\biggl\langle\,\frac{P^r-x_0}{2(t_0-\xi)}\,\biggr\vert\,
\tau(P^r,\xi)\,\biggr\rangle -\lambda(P^r,\xi)\,\biggl]\,
\rho_{x_0,t_0}(P^r,\xi)\,d\xi\,,\\
=&\,\int_{{\SS_{0}}} \rho_{x_0,t_0}(x,0)\,ds-\int_{0}^{t}\int_{{\SS_\xi}} \left\vert\,
\underline{k}+\frac{(x-x_0)^{\perp}}{2(t_0-\xi)}\right\vert^2\rho_{x_0,t_0}(x,\xi)\,ds\,d\xi\\ 
&\,+\sum_{r=1}^l\int_{0}^{t_0} \biggl[\biggl\langle\,\frac{P^r-x_0}{2(t_0-\xi)}\,\biggr\vert\,
\tau(P^r,\xi)\,\biggr\rangle -\lambda(P^r,\xi)\,\biggl]\,
\rho_{x_0,t_0}(P^r,\xi)\,d\xi\,,\\
&\,-\sum_{r=1}^l\int_{t}^{t_0} \biggl[\biggl\langle\,\frac{P^r-x_0}{2(t_0-\xi)}\,\biggr\vert\,
\tau(P^r,\xi)\,\biggr\rangle -\lambda(P^r,\xi)\,\biggl]\,
\rho_{x_0,t_0}(P^r,\xi)\,d\xi\,,
\end{align*} 
for every $t_0\in(0,T]$ and $t\in[0,t_0)$. Now we notice that the first line on the right side of this formula is a monotone non increasing function in $t\in[0,t_0)$, the second line is a constant and the third line converges to zero as $t\to t_0$, by Lemma~\ref{stimadib}. Hence, the non negative function $t\mapsto \int_{{\SS_{t}}} \rho_{x_0,t_0}(x,t)\,ds$ converges to some limit as $t\to t_0$. Then, the following definition is well posed.

\begin{defn}[Gaussian densities]\label{Gaussiandensities}
For every $x_0\in\R^2, t_0 \in(0,+\infty)$ we define the {\em Gaussian
 density} function $\Theta_{x_0,t_0}:[0,\min\{t_0,T\})\to\R$ as 
$$
\Theta_{x_0,t_0}(t)=\int_{\SS_t}\rho_{x_0,t_0}(\cdot,t)\,ds
$$
and, provided $t_0\leqslant T$, the {\em limit Gaussian density} function $\widehat{\Theta}:\R^2\times (0,+\infty)\to\R$ as
$$
\widehat{\Theta}(x_0,t_0)=\lim_{t\to t_0}\Theta_{x_0,t_0}(t)\,.
$$
which exists finite and non negative, for every $(x_0,t_0)\in\R^2\times(0,T]$, by the above argument (under assumption~\eqref{endsmooth}, or simply if the end--points $P^r$ of the network $\SS_t$ are fixed, hence $\lambda(P^r,\cdot)=0$).\\
We will often write $\Theta_{x_0}(t)$ to denote $\Theta_{x_0,T}(t)$ and $\widehat{\Theta}(x_0)$ for $\widehat{\Theta}(x_0,T)$.
\end{defn}
Notice that the map $\widehat{\Theta}:\mathbb{R}^2\to\mathbb{R}$ is upper semicontinuous (see~\cite[Proposition~2.12]{MMN13}), being given by the monotone limit of continuous functions ``perturbed'' by a sequence of functions pointwise converging to zero.

\subsection{Parabolic rescaling of the flow}\label{pararesc}
For a fixed $\mu > 0$ the standard parabolic rescaling of a
curvature flow is given by the map $F$ above, around a
space--time point $(x_0,t_0)$, is defined as the family of maps
\begin{equation}
 \label{eq:parrescaling}
 F^\mu_\tt = \mu \big(F(\cdot, \mu^{-2}\tt + t_0) - x_0
 \big)\,, 
\end{equation}
where $\tt \in [-\mu^2 t_0, \mu^2(T-t_0))$. Notice that this is
again a curvature flow in the domain $\mu (\Omega -x_0)$ with new time
parameter $\tt$.

Given a sequence $\mu_i\nearrow +\infty$ and a space--time point
$(x_0,t_0)$, where $0<t_0\leqslant T$, we then consider the sequence of
curvature flows $F^{\mu_i}_\tt$ in the whole $\R^2$ 
that we denote with $\SS^{\mu_i}_\tt$.\\
Recall that the monotonicity formula implies
\begin{align*}
\Theta_{x_0,t_0}(t)-\widehat{\Theta}(x_0,t_0)=&\,\int\limits_t^{t_0}\int\limits_{\SS_{\sigma}}
\Big|\underline{k}+
\frac{(x-x_0)^\perp}{2(t_0-\sigma)}\Big|^2\rho_{x_0,t_0}(\cdot,\sigma)\,
ds\,d\sigma\\
&\,-\sum_{r=1}^l\int_{t}^{t_0} \biggl[\biggl\langle\,\frac{P^r-x_0}{2(t_0-\sigma)}\,\biggr\vert\,
\tau(P^r,\sigma)\,\biggr\rangle -\lambda(P^r,\sigma)\,\biggl]\,
\rho_{x_0,t_0}(P^r,\sigma)\,d\sigma\,.
\end{align*}
Changing variables according to the parabolic rescaling, we obtain
\begin{align*}
\Theta_{x_0,t_0}(t_0+\mu_i^{-2}\tt)-\widehat\Theta(x_0,t_0)=&\,\int\limits_{\tt}^{0}\int
\limits_{\SS^{\mu_i}_\ssss}
\Big|\underline{k}^i-\frac{x^\perp}{2\ssss}\Big|^2\rho_{0,0}(\cdot,\ssss)\,
ds\,d\ssss\\ 
&\,+\sum_{r=1}^l\int\limits_{\tt}^{0} \biggl[\biggl\langle\,\frac{P_i^r}{2\ssss}\,\biggr\vert\,
\tau(P_i^r,\ssss)\,\biggr\rangle +\lambda^i(P_i^r,\ssss)\,\biggl]\,
\rho_{0,0}(P_i^r,\ssss)\,d\ssss\,,
\end{align*}
where $P^r_i = \mu_i(P^r-x_0)$ and $k^i$ and $\lambda^i$ are the rescaled curvatures and tangential velocities.

Hence, sending $i\to\infty$, by Lemma~\ref{stimadib}, for every $\tt\in(-\infty, 0)$ we get
\begin{equation*}
\lim_{i\to\infty}\int\limits_{\tt}^{0}\int
\limits_{\SS^{\mu_i}_\ssss}
\Big|\underline{k}^i- \frac{x^\perp}{2\ssss}\Big|^2\rho_{0,0}(\cdot,\ssss)\,
ds\,d\ssss=0\,.
\end{equation*}

\subsection{Huisken's dynamical rescaling} 
We introduce the rescaling procedure of Huisken in~\cite{huisk3} at the maximal time $T$.\\ 
Fixed $x_0\in\R^2$, 
let $\widetilde{F}_{x_0}:\SS\times [-1/2\log{T},+\infty)\to\R^2$ 
be the map 
$$
\widetilde{F}_{x_0}(p,\tt)
=\frac{F(p,t)-x_0}{\sqrt{2(T-t)}}\qquad \tt(t)
=-\frac{1}{2}\log{(T-t)} $$ 
then, the rescaled networks are given by 
\begin{equation}\label{huiskeqdef}
\widetilde{\SS}_{x_0,\tt}=\frac{\SS_t-x_0}{\sqrt{2(T-t)}}
\end{equation}
and they evolve according to the equation 
$$
\frac{\partial\,}{\partial \tt}\widetilde{F}_{x_0}(p,\tt)
=\widetilde{\underline{v}}(p,\tt)+\widetilde{F}_{x_0}(p,\tt) $$
where 
$$ \widetilde{\underline{v}}(p,\tt)=\sqrt{2(T-t(\tt))}\cdot\underline{v}(p,t(\tt))=
 \widetilde{\underline{k}}+\widetilde{\underline{\lambda}}= 
\widetilde{k}\nu+\widetilde{\lambda}\tau\qquad 
\text{ and }\qquad t(\tt)=T-e^{-2\tt}\,. 
$$ 
Notice that we did not put the sign $\widetilde{}$ over the unit tangent and normal, 
since they remain the same after the rescaling.\\ 
We will write $\widetilde{O}^p(\tt)=\widetilde{F}_{x_0}(O^p,\tt)$ 
for the $3$--points of the rescaled network $\widetilde{\SS}_{x_0,\tt}$ 
and $\widetilde{P}^r(\tt)=\widetilde{F}_{x_0}(P^r,\tt)$ for the end--points, 
when there is no ambiguity on the point $x_0$.\\ 
The rescaled curvature evolves according to the following equation, 
\begin{equation*}\label{evolriscforf} 
{\partial_\tt} \widetilde{k}= \widetilde{k}_{\ssss\ssss}+\widetilde{k}_\ssss\widetilde{\lambda}
+ \widetilde{k}^3 -\widetilde{k} 
\end{equation*} 
which can be obtained by means of the commutation law 
\begin{equation*}\label{commutforf} 
{\partial_\tt}{\partial_\ssss}={\partial_\ssss}{\partial_\tt} 
+ (\widetilde{k}^2 -\widetilde{\lambda}_\ssss-1){\partial_\ssss}\,, 
\end{equation*} 
where we denoted with $\ssss$ the arclength parameter for $\widetilde{\SS}_{x_0,\tt}$.

\begin{rem}\label{relaz}
It is easy to see that the relations between the two rescaling procedures are given by
$$
\SS^\mu_\tt=\sqrt{-2\tt}\,\widetilde{\SS}_{x_0,\log{(\mu/\sqrt{-\tt})}}\quad\text{ and }\quad\widetilde{\SS}_{x_0,\tt}=
\frac{e^\tt}{\mu\sqrt{2}}\,\SS^\mu_{-\mu^2e^{-2\tt}}\,,
$$
in particular,
$$
\SS^\mu_{-1/2}=\widetilde{\SS}_{x_0,\log{(\mu\sqrt{2})}}\,.
$$
\end{rem}

By a straightforward computation (see~\cite{huisk3})
we have the following rescaled version of the monotonicity formula.

\begin{prop}[Rescaled monotonicity formula] 
Let $x_0\in\R^2$ and set 
$$ \widetilde{\rho}(x)=e^{-\frac{\vert x\vert^2}{2}} $$ 
For every $\tt\in[-1/2\log{T},+\infty)$ the following identity holds 
\begin{equation*} 
\frac{d\,}{d\tt}\int_{\widetilde{\SS}_{x_0,\tt}} \widetilde{\rho}(x)\,d\ssss= 
-\int_{\widetilde{\SS}_{x_0,\tt}}\vert \,\widetilde{\underline{k}}+x^\perp\vert^2\widetilde{\rho}(x)\,d\ssss 
+\sum_{r=1}^l\Bigl[\Bigl\langle\,{\widetilde{P}^r(\tt)} \,\Bigl\vert\,
{\tau}(P^r,t(\tt))\Bigr\rangle-\widetilde{\lambda}(P^r,t(\tt))\Bigl]\,\widetilde{\rho}(\widetilde{P}^r(\tt))\label{reseqmonfor} 
\end{equation*} 
where $\widetilde{P}^r(\tt)=\frac{P^r-x_0}{\sqrt{2(T-t(\tt))}}$.\\ 
Integrating between $\tt_1$ and $\tt_2$ with $-1/2\log{T}\leqslant \tt_1\leqslant \tt_2<+\infty$ we get 
\begin{align} 
\int_{\tt_1}^{\tt_2}\int_{\widetilde{\SS}_{x_0,\tt}}\vert \,\widetilde{\underline{k}}
+x^\perp\vert^2\widetilde{\rho}(x)\,d\ssss\,d\tt= &\, 
\int_{\widetilde{\SS}_{x_0,\tt_1}}\widetilde{\rho}(x)\,d\ssss 
-\int_{\widetilde{\SS}_{x_0,\tt_2}}\widetilde{\rho}(x)\,d\ssss\label{reseqmonfor-int}\\ &\,
+\sum_{r=1}^l\int_{\tt_1}^{\tt_2}\Bigl[\Bigl\langle\,\widetilde{P}^r(\tt)\,
\Bigl\vert\,{\tau}(P^r,t(\tt))\Bigr\rangle-\widetilde{\lambda}(P^r,t(\tt))\Bigl]\,\widetilde{\rho}(\widetilde{P}^r(\tt)\,d\tt\,.\nonumber 
\end{align} 
\end{prop}

We have also the analog of Lemma~\ref{stimadib} 
(see Lemma~6.7 in~\cite{mannovtor}).

\begin{lem}\label{rescstimadib} 
For every $r\in\{1,2,\dots,l\}$ and $x_0\in\R^2$,
the following estimate holds for all $\tt\in \bigl[-\frac{1}{2}\log{T},+\infty\bigr)$,
\begin{equation*} 
\int_{\tt}^{+\infty}\Bigl\vert\Bigl\langle\,\widetilde{P}^r(\xi)\,
\Bigl\vert\,{\tau}(P^r,t(\xi))\Bigr\rangle-\widetilde{\lambda}(P^r,t(\xi))\Bigl\vert\,d\xi\,\leqslant C\,,
\end{equation*}
where $C$ is a constant depending only on the constants $C_0$ in assumption~\eqref{endsmooth} (independent of $\tt$).\\
As a consequence, for every point $x_0\in\R^2$, we have
\begin{equation*} 
\lim_{\tt\to +\infty}\sum_{r=1}^l\int_{\tt}^{+\infty}\Bigl[\Bigl\langle\,\widetilde{P}^r(\xi)\,
\Bigl\vert\,{\tau}(P^r,t(\xi))\Bigr\rangle-\widetilde{\lambda}(P^r,t(\xi))\Bigl]\,d\xi=0\,.
\end{equation*}
\end{lem}

\section{Classification of possible blow--up limits}\label{geosec}

In this section we want to discuss the possible limits of an evolving
network at the maximal time of existence. When the curvature does not remain bounded, 
we are interested in the possible blow--up limit
networks after parabolic or Huisken's rescaling procedure, using
the rescaled monotonicity formula (see Section~\ref{monotonsec}). 
In some cases, such limit {\em sets} are no more regular networks, so we introduce the following definition.

\begin{defn}[Degenerate regular network]\label{degnet} 
Consider a tuple $(G,\mathbb{S})$ with the following properties:

\begin{itemize}
\item $G=\bigcup_{i=1}^{n}E^i$ is an oriented graph with possible unbounded edges $E^i$, such that every vertex has only one or three concurring edges (we call end--points of $G$ the vertices with order one);

\item given a family of $C^1$ curves $\sigma^i:I^i\to\R^2$, 
where $I^i$ is the interval $(0,1)$, $[0,1)$, $(0,1]$ or $[0,1]$,
and orientation preserving homeomorphisms $\varphi^i:E^i\to I^i$, then $\mathbb{S}$ is the union of the images of $I^i$ through the curves $\sigma^i$, that is, $\SS=\bigcup_{i=1}^{n}\sigma^{i}(I^i)$
(notice that the interval $(0,1)$ can only appear if it is associated with an unbounded edge $E^i$ without vertices, which is clearly a single connected component of $G$);

\item in the case that $I^i$ is $(0,1)$, $[0,1)$ or $(0,1]$, 
the map $\sigma^i$ is a regular $C^1$ curve with unit tangent vector field $\tau^i$;

\item in the case that $I^i=[0,1]$, the map $\sigma^i$ is either a regular $C^1$ curve with unit tangent vector field $\tau^i$, 
or a constant map and in this case it is ``assigned'' also a {\em constant} unit vector $\tau^i:I^i\to\R^2$, 
that we still call unit tangent vector of $\sigma^i$ (we call these maps $\sigma^i$ ``degenerate curves'');

\item for every degenerate curve $\sigma^i:I^i\to\R^2$ with assigned unit vector $\tau^i:I^i\to\R^2$, we call ``assigned exterior unit tangents'' of the curve $\sigma^i$ at the points $0$ and $1$ of $I^i$, respectively the unit vectors $-\tau^i$ and $\tau^i$.

\item the map $\Gamma:G\to\R^2$ given by the union $\Gamma=\bigcup_{i=1}^n(\sigma^i\comp\varphi^i)$ 
is well-defined and continuous;

\item for every $3$--point of the graph $G$, where the edges $E^i$, $E^j$, $E^k$ concur, the exterior unit tangent vectors (real or ``assigned'') at the relative borders of the intervals $I^i$, $I^j$, $I^k$ of the concurring curves $\sigma^i$, $\sigma^j$ $\sigma^k$ have zero sum ({\em ``degenerate $120$ degrees condition''}).
\end{itemize}
Then, we call $\SS=\bigcup_{i=1}^{n}\sigma^{i}(I^i)$ a {\em degenerate regular network}.

If one or several edges $E^i$ of $G$ are mapped under the map $\Gamma:G\to\R^2$ to a single point $p\in\R^2$, we call this
sub--network given by the union $G^\prime$ of such edges $E^i$, the {\em core} of $\SS$ at $p$. 

We call multi--points of the degenerate regular network $\SS$, the images of the vertices of multiplicity three of the graph $G$, by the map $\Gamma$.

We call end--points of the degenerate regular network $\SS$, the images of the vertices of multiplicity one of the graph $G$, by the map $\Gamma$.
\end{defn}

\begin{rem}\ \label{remreg0}
\begin{itemize}
\item A regular network is clearly a degenerate regular network.

\item This definition will be useful to deal with the limit sets 
when at some time a curve of the network collapses, 
namely, its length goes to zero (later on in Section~\ref{behavsing}).

\item Seen as a subset in $\R^2$,
a degenerate regular network $\SS$ with underlying graph $G$, 
is a $C^1$ network, not necessarily regular, that can have
end--points and/or unbounded curves. Moreover self--intersections
and curves with integer multiplicities can be present. Anyway by
the degenerate $120$ degrees condition at the last point of the
definition, at every image of a multi--point of $G$ the sum (possibly with
multiplicities) of the exterior unit tangents (the ``assigned'' ones
cancel each other in pairs) is zero. Notice that this implies that
every multiplicity--one $3$--point must satisfy the $120$ degrees
condition.
\end{itemize}
\end{rem}

\begin{lem}\label{lemreg} Let $\SS=\bigcup_{i=1}^{n}\sigma^{i}(I_i)$ be a
degenerate regular network in $\Omega$ and $X:\R^2\to\R^2$ 
be a smooth vector field with compact support. Then, there holds
$$
\int_\SS \partial_s\langle X(\sigma)\,\,\vert\tau\rangle\,d\overline{\HH}^1
=-\sum_{r=1}^l\langle X(P^r)\,\,\vert\tau(P^r)\rangle\,,
$$
where $P^1, P^2,\dots, P^l$ are the end--points of $\SS$, 
$\tau(P^1),\tau(P^2),\dots,\tau(P^l)$ are the exterior
unit tangents at $P^r$ and $\overline{\HH}^1$ is the one--dimensional
Hausdorff measure, counting multiplicities.
\end{lem}
\begin{proof} This is a consequence of the degenerate $120$ degrees condition, 
implying that the sum of all the contributions at a multi--point given by the boundary terms after 
the integration on every single curve is zero 
(as the sum of the exterior unit tangents of the concurring curves). 
Thus the only remaining terms are due to the end--points of the degenerate regular network.
\end{proof}

\begin{defn}\label{degconv} We say that a sequence of regular networks 
$\SS_k=\bigcup_{i=1}^{n}\sigma^{i}_k(I^i_k)$ converges in $C^1\loc$ 
to a degenerate regular network $\SS=\bigcup_{j=1}^{l}\sigma^{j}_\infty(I^j_\infty)$ 
with underlying graph $G=\bigcup_{j=1}^{l}E^j$ if:
\begin{itemize}
\item letting $O^1, O^2,\dots, O^m$ be the multi--points of $\SS$, for every open set 
$\Omega\subseteq\R^2$ with compact closure in $\R^2\setminus\{O^1, O^2,\dots, O^m\}$
the networks $\SS_k$ restricted to $\Omega$, for $k$ large enough, are described by families 
of regular curves which converge up to reparametrization to the family of regular 
curves given by the restriction of $\SS$ to $\Omega$;
\item for every multi--point $O^p$ of $\SS$, image of one or more vertices of the graph $G$ 
(if a core is present), there is a sufficiently small $R>0$ and 
a graph $\widetilde{G}=\bigcup_{r=1}^{s}F^r$, 
with edges $F^r$ associated to intervals $J^r$, such that:
\begin{itemize}
\item the restriction of $\SS$ to $B_R(O^p)$ is a regular degenerate network described by a 
family of curves $\widetilde{\sigma}^{r}_\infty:J^r\to\R^2$ with (possibly ``assigned'', if 
the curve is degenerate) unit tangent $\widetilde{\tau}^r_\infty$,
\item for $k$ sufficiently large the restriction of $\SS_k$ to $B_R(O^p)$ is a regular network 
with underlying graph $\widetilde{G}$, described by the family of regular curves 
$\widetilde{\sigma}^{r}_k:J^r\to\R^2$,
\item for every $j$, possibly after reparametrization of the curves, the sequence of maps 
$J^r\ni x\mapsto\bigl(\widetilde{\sigma}^r_k(x),\widetilde{\tau}^r_k(x)\bigr)$ 
converge in $C^0\loc$ to the maps $J^r\ni x\mapsto\bigl(\widetilde{\sigma}^r_\infty(x),
\widetilde{\tau}^r_\infty(x)\bigr)$ for every $r\in\{1,2,\dots,s\}$.
\end{itemize}
\end{itemize}
We will say that $\SS_k$ converges to $\SS$ in $C^1\loc\cap E$, where $E$ is some function 
space, if the above curves also converge in the topology of $E$.
\end{defn}

\begin{rem}\ \label{remreg}
\begin{itemize}
\item If the limit regular network $\SS$ is non--degenerate, the above convergence of a sequence of regular networks $\SS_k$ to $\SS$ is simply the $C^1\loc$--convergence of the curves of $\SS_k$ to the relative ones of $\SS$. Anyway, in general, if $\SS$ is a degenerate regular network $\SS$, the above definition of $C^1\loc$--convergence for a sequence of regular networks $\SS_k$ to $\SS$, is clearly stronger than that, by the last request at the second point. Asking only the $C^1\loc$--convergence of the curves of a sequence of regular networks $\SS_k$ would not guarantee that the limit degenerate network $\SS$ is {\em regular}, as the last point in Definition~\ref{degnet} could possibly not being satisfied by $\SS$.
\item It is easy to see that if a sequence of regular networks $\SS_k$ converges in $C^1\loc$ 
to a degenerate regular network $\SS$, the associated one--dimensional
Hausdorff measures, counting multiplicities, 
weakly--converge (as measures) to the one--dimensional
Hausdorff measure associated with the set $\SS$ seen as a subset of $\R^2$.
\item If a degenerate regular network $\SS$ is the limit of a sequence of regular networks as 
above, being these embedded, it clearly can have only {\em tangent} self--intersections but not 
a ``crossing'' of two of its curves.
\item If $\SS$ is the limit of a sequence of ``rescalings'' of the networks of a curvature flow 
$\SS_t$ with fixed end--points, it can have only one end--point at the origin of $\R^2$ 
and only if the center of the rescalings coincides with an end--point of $\SS_t$, otherwise, it 
has no end--points at all (they go to $\infty$ in the rescaling).
\end{itemize}
\end{rem}

\subsection{Self--similarly shrinking networks}\label{sshnet}
\begin{defn}\label{shrinkers} 
A regular $C^2$ open network $\SS=\bigcup_{i=1}^n\sigma^i(I_i)$ 
is called a {\em regular shrinker} if at every point $x\in\SS$ there holds
\begin{equation}\label{shrinkeq}
\underline{k} + x^\perp=0. 
\end{equation}
This relation is called the {\em shrinkers equation}.
\end{defn}

The name comes from the fact that if $\SS=\bigcup_{i=1}^n\sigma^i(I_i)$ is a shrinker, 
then the evolution given by $\SS_t=\bigcup_{i=1}^n\gamma^i(I_i,t)$ 
where $\gamma^i(x,t)=\sqrt{-2t}\, \sigma^i(x)$ is a self--similarly shrinking curvature flow
in the time interval $(-\infty,0)$ with $\SS=\SS_{-1/2}$.
Viceversa, if $\SS_t$ is a self--similarly shrinking curvature flow
 in the maximal time interval $(-\infty,0)$, then $\SS_{-1/2}$ is a shrinker.

\begin{figure}[H]
\begin{center}
\begin{tikzpicture}[scale=0.6]
\draw[shift={(-7.5,0)}]
(-2,-2) to [out=45, in=-135,looseness=1] (2,2);
\draw[dashed,shift={(-7.5,0)}]
(-2.5,-2.5) to [out=45, in=-135,looseness=1] (-2,-2)
(2,2)to [out=45, in=-135,looseness=1](2.6,2.6);
\draw
(0,0) to [out=90, in=-90,looseness=1] (0,2)
(0,0) to [out=210, in=30,looseness=1] (-1.73,-1)
(0,0) to [out=-30, in=150,looseness=1](1.73,-1);
\draw[dashed]
(0,2) to [out=90, in=-90,looseness=1] (0,3)
(-1.73,-1) to [out=210, in=30,looseness=1] (-2.59,-1.5)
(1.73,-1) to [out=-30, in=150,looseness=1](2.59,-1.5);
\draw[color=black,scale=1,domain=-3.141: 3.141,
smooth,variable=\t,shift={(7.5,0)},rotate=0]plot({2.*sin(\t r)},
{2.*cos(\t r)});
\fill(0,0) circle (2pt);
\fill(-7.5,0) circle (2pt);
\fill(7.5,0) circle (2pt);
\path[font=\small]
(-7.5,.2) node[left]{$O$}
(7.65,.2) node[left]{$O$}
(.1,.2) node[left]{$O$};
\end{tikzpicture}
\end{center}
\begin{caption}{Examples of regular shrinkers with zero or one triple junction:
a line through the origin, an unbounded triod composed of three halflines from the origin meeting at $120$ degrees, that we call {\em standard triod} 
and the unit circle $\SS^1$.}
\end{caption}
\end{figure}
\begin{figure}[H]
\begin{center}
\begin{tikzpicture}[scale=0.4]
\draw[color=black]
(-6.3,0)to[out=0,in=180,looseness=1](-1.535,0)
(-1.535,0)to[out=60,in=180,looseness=1] (3.7,3)
(3.7,3)to[out=0,in=90,looseness=1] (6.93,0)
(-1.535,0)to[out=-60,in=180,looseness=1] (3.7,-3)
(3.7,-3)to[out=0,in=-90,looseness=1] (6.93,0);
\draw[color=black,dashed](-8,0)to[out=0,in=180,looseness=1](-6.5,0);
\fill(1,0) circle (3pt);
\path[font=\large]
 (.3,-.35) node[above]{$O$};
\end{tikzpicture}
\end{center}
\begin{caption}{Another example of a regular shrinker with one triple junction: a {\em Brakke spoon}.\label{brakspfig}}
\end{caption}
\end{figure}
\noindent
In these figures, there are drawn {\em all} the regular shrinkers 
with at most one triple junction (see~\cite{haettenschweiler}).
In particular by the work of Abresch and Langer~\cite{ablang1} 
it follows that the only regular shrinkers {\em without} triple junctions (simply curves) 
are the lines through the origin and the unit circle. In the case of
complete, embedded, regular shrinker
with two triple junctions it is not difficult to show
that there are only two possible topological shapes: the ``lens/fish'' shape and the
Greek ``Theta'' letter (or ``double cell''),
as depicted in the next figure (see also~\cite{balhausman2}).
\begin{figure}[H]
\begin{center}
\begin{tikzpicture}[scale=1]
\draw[shift={(0,0)}] 
(-3.73,0)
to[out=50,in=180,looseness=1] (-2.8,0) 
to[out=60,in=150,looseness=1.5] (-1.5,1) 
(-2.8,0)
to[out=-60,in=180,looseness=0.9] (-1.25,-0.75)
(-1.5,1)
to[out=-30,in=90,looseness=0.9] (-1,0)
to[out=-90,in=60,looseness=0.9] (-1.25,-0.75)
to[out=-60,in=180,looseness=0.9](-0.3,-1.3);
\draw[dashed,shift={(0,0)}] 
(-3.73,0)to[out=50,in=180,looseness=1] (-4,-.6);
\draw[dashed,shift={(0,0)}] 
(-0.3,-1.3)to[out=0,in=150,looseness=.8] (0.53,-.8);
\path[font=\small,shift={(0,0)}]
(-2.4,0.3) node[below] {$O^1$}
(-1.2,-0.7)node[right]{$O^2$}
(-1.5,-1)[left] node{$\gamma^2$}
(-3.2,.4) node[left] {$\gamma^1$}
(-0.5,0.65)[left] node{$\gamma^4$}
(-0.35,-1.30)[below] node{$\gamma^3$};
\draw[shift={(7,0)}] 
(-1.73,-1.8) 
to[out=180,in=180,looseness=1] (-2.8,0) 
to[out=60,in=150,looseness=1.5] (-1.5,1) 
(-2.8,0)
to[out=-60,in=180,looseness=0.9] (-1.25,-0.75)
(-1.5,1)
to[out=-30,in=90,looseness=0.9] (-1,0)
to[out=-90,in=60,looseness=0.9] (-1.25,-0.75)
to[out=-60,in=0,looseness=0.9](-1.73,-1.8);
\path[font=\small, shift={(7,0)}] 
(-1.2,-0.8)node[right]{$O^2$}
 (-1.5,-0.3)[left] node{$\gamma^2$}
 (-0.6,.9)[left] node{$\gamma^1$}
 (-0.6,-1.45)[left] node{$\gamma^3$}
 (-3,0.5) node[below] {$O^1$}; 
\end{tikzpicture}
\end{center}
\begin{caption}{A lens/fish--shaped and a $\Theta$--shaped network.\label{fishshape}}
\end{caption}
\end{figure}
It is well known that there exist unique (up to rotations)
lens--shaped or fish--shaped, embedded, regular shrinkers that are
symmetric with respect to a line through the origin of $\R^2$
(see~\cite{chenguo,schnurerlens}). 
Instead, there are no regular $\Theta$--shaped shrinkers (see~\cite{balhausman}).

\begin{figure}[H]
\begin{center}
\begin{tikzpicture}[scale=0.3]
\draw[color=black]
(-3.035,0)to[out=60,in=180,looseness=1] (2.2,2.8)
(2.2,2.8)to[out=0,in=120,looseness=1] (7.435,0)
(2.2,-2.7)to[out=0,in=-120,looseness=1] (7.435,0)
(-3.035,0)to[out=-60,in=180,looseness=1] (2.2,-2.7);
\draw[color=black]
(-7,0)to[out=0,in=180,looseness=1](-3.035,0)
(7.435,0)to[out=0,in=180,looseness=1](11.4,0);
\draw[color=black,dashed]
(-9,0)to[out=0,in=180,looseness=1](-7,0)
(11.4,0)to[out=0,in=180,looseness=1](13.4,0);
\fill(2.2,0) circle (3pt);
\path[font=\small]
(1.5,-.35) node[above]{$O$};
\draw[color=black,scale=4,shift={(7,0)}]
(-0.47,0)to[out=20,in=180,looseness=1](1.5,0.65)
(1.5,0.65)to[out=0,in=90,looseness=1] (2.37,0)
(-0.47,0)to[out=-20,in=180,looseness=1](1.5,-0.65)
(1.5,-0.65)to[out=0,in=-90,looseness=1] (2.37,0);
\draw[white, very thick,scale=4,shift={(7,0)}]
(-0.47,0)--(-.150,-0.13)
(-0.47,0)--(-.150,0.13);
\draw[color=black,scale=4,shift={(7,0)}]
(-.150,0.13)to[out=-101,in=90,looseness=1](-.18,0)
(-.18,0)to[out=-90,in=101,looseness=1](-.150,-0.13);
\draw[black,scale=4,shift={(7,0)}]
(-.150,0.13)--(-1.13,0.98);
\draw[black,scale=4,shift={(7,0)}]
(-.150,-0.13)--(-1.13,-0.98);
\draw[black, dashed,scale=4,shift={(7,0)}]
(-1.13,0.98)--(-1.50,1.31);
\draw[black, dashed,scale=4,shift={(7,0)}]
(-1.13,-0.98)--(-1.50,-1.31);
\fill(28.05,0) circle (3pt);
\path[font=\small]
(29,-.6) node[above]{$O$};
\end{tikzpicture}
\end{center}
\begin{caption}{The {\em shrinking lens} and the {\em shrinking fish} (up to rotations).\label{fishfig}}
\end{caption}
\end{figure}
\noindent A ``gallery'' with these and other more complicated regular shrinkers 
can be found in the Appendix.

\begin{defn}[Degenerate shrinkers]\label{dshrinkers} 
We call a degenerate regular network $\SS=\bigcup_{i=1}^n\sigma^i(I_i)$ a {\em degenerate regular shrinker} 
if at every point $x\in\SS$ there holds
$$
\underline{k} + x^\perp=0\,.
$$
\end{defn}
Clearly, a regular shrinker is a degenerate regular shrinker and, as before, the maps $\gamma^i(x,t)=\sqrt{-2t}\,\sigma^i(x)$ 
describe the self--similarly shrinking evolution of a degenerate regular network $\SS_t$ 
in the time interval $(-\infty,0)$, with $\SS=\SS_{-1/2}$.

\begin{defn} 
A {\em standard cross} is a degenerate regular network given the union of two straight lines
intersecting at the origin of $\R^2$ and forming angles of $120$ and $60$ degrees, with an underlying graph $G$ as in the following figure. Its core consists of the degenerate curve mapping the ``central'' curve of $G$ constantly to the origin.
The ``assigned'' tangent vector to the degenerate curve is one of the two unit vectors that generates the 
bisector line of the $120$ degrees angles.
\end{defn}
\begin{figure}[H]
\begin{center}
\begin{tikzpicture}[scale=0.40]
\draw[color=black]
(0,0)to[out=120,in=-60,looseness=1] (-2,3.46)
(0,0)to[out=-120,in=60,looseness=1] (-2,-3.46)
(0,0)to[out=60,in=-120,looseness=1] (2,3.46)
(0,0)to[out=-60,in=120,looseness=1] (2,-3.46);
\draw[color=black,dashed]
(-2,3.46)to[out=120,in=-60,looseness=1] (-3,5.19)
(-2,-3.46)to[out=-120,in=60,looseness=1] (-3,-5.19)
(2,3.46)to[out=60,in=-120,looseness=1] (3,5.19)
(2,-3.46)to[out=-60,in=120,looseness=1] (3,-5.19);
\fill(0,0) circle (4pt);
\path[font=\large]
(-0,-4) node[below] {$\mathbb{C}$}
(1,-.35) node[above]{$O$};
\end{tikzpicture}\qquad\qquad\qquad
\begin{tikzpicture}[scale=0.40]
\draw[color=black]
(-1,0)to[out=0,in=180,looseness=1](1,0)
(-1,0)to[out=120,in=-60,looseness=1] (-3,3.46)
(-1,0)to[out=-120,in=60,looseness=1] (-3,-3.46)
(1,0)to[out=60,in=-120,looseness=1] (3,3.46)
(1,0)to[out=-60,in=120,looseness=1] (3,-3.46);
\draw[color=black,dashed]
(-3,3.46)to[out=120,in=-60,looseness=1] (-4,5.19)
(-3,-3.46)to[out=-120,in=60,looseness=1] (-4,-5.19)
(3,3.46)to[out=60,in=-120,looseness=1] (4,5.19)
(3,-3.46)to[out=-60,in=120,looseness=1] (4,-5.19);
\fill(1,0) circle (4pt);
\fill(-1,0) circle (4pt);
\path[font=]
(-0,-4) node[below] {$G$};
\end{tikzpicture}
\end{center}
\begin{caption}{A {\em standard cross} with angles of $60/120$ 
degrees and its underlying graph $G$.\label{crossfig}}
\end{caption}
\end{figure}
\begin{rem}\label{abla}
As every non--degenerate curve of a degenerate regular shrinker (or
simply of a regular shrinker) satisfies the equation $\underline{k} +
x^\perp=0$, it must be a piece of a line through the origin or of the
so called {\em Abresch--Langer curves}. Their classification results
in~\cite{ablang1} imply that any of these non straight pieces are
compact, hence any unbounded curve of a shrinker must be a line or an
halfline ``pointing'' towards the origin. Moreover, it also follows that if a
curve contains the origin, then it is a straight line through the
origin (if it is in the interior) or a halfline from the origin (if it
is an end--point of the curve).
\end{rem}

For a degenerate regular shrinker $\SS$, in analogy with Definition~\ref{Gaussiandensities},
we denote with
$$
\Theta_\SS=\Theta_{0,0}(-1/2)=\int_\SS\rho_{0,0}(\cdot,-1/2)\, d\overline{s}
$$
its {\em Gaussian density} (here $d\overline{s}$ denotes the integration with respect to the canonical measure on $\SS$, counting multiplicities). 
Notice that the integral $\Theta_{0,0}(t)=\int_{\SS_t}\rho_{0,0}(\cdot,t)\,d\overline{s}$ is constant for
$t\in(-\infty,0)$, hence equal to $\widehat{\Theta}(0)$ for the
self--similarly shrinking curvature flow $\SS_t=\sqrt{-2t}\,\SS$ generated by $\SS$, as above.

The Gaussian density of a straight line through the origin is 1, 
of a halfline from the origin is 1/2, of a standard triod $\TTT$ is $3/2$, of a standard cross ${\mathbb{C}}$ is $2$. 
The Gaussian density of the unit circle $\SS^1$ can be easily computed to be
\begin{equation}\label{thetas1}
\Theta_{\SS^1}=\sqrt{\frac{2\pi}{e}}\approx 1,\!5203\,.
\end{equation}
Notice that $\Theta_\TTT=3/2<\Theta_{\SS^1}<2$.\\
The Gaussian densities of several other regular shrinkers can be found in the Appendix.

We have the following two classification results for degenerate regular shrinkers, see Lemma~8.3 and~8.4 in~\cite{Ilnevsch}.

\begin{lem}\label{lemmatree}
Let $\SS=\bigcup_{i=1}^n\sigma^i(I_i)$ be a degenerate regular shrinker which is a $C^1\loc\cap W^{2,2}\loc$--limit of regular networks $\SS_i$ homeomorphic to the underlying graph $G$ of\, $\SS$ (as in Definition~\ref{degnet}) and assume that $G$ is a tree without end--points. 
Then $\SS$ consists of halflines from the origin, with possibly a core at the origin.\\
Moreover, if $G$ is connected, without end--points and $\SS$ is a network with unit multiplicity, this latter can only be
\begin{itemize}
\item a line (no cores),
\item a standard triod (no cores),
\item two lines intersecting at the origin forming angles of $120/60$
 degrees (the core is a collapsed segment in the origin with ``assigned'' unit tangent vector bisecting the angles of $120$ degrees), that is, a standard cross (see Figure~\ref{crossfig}).
\end{itemize}
\end{lem}
\begin{proof} 
We assume that $G$ is connected, otherwise, we argue on every single connected component. By the hypothesis of approximation with regular (embedded) networks, $G$ is a planar graph.

As we said in Remark~\ref{abla}, if a non--degenerate 
curve contains the origin, then it is a piece of a 
straight line. Otherwise, it is contained in a
compact subset of $\R^2$ and has a constant winding direction with
respect to the origin. Aside from the circle, any other solution has a
countable, non--vanishing number of self--intersections (all these facts were shown in~\cite{ablang1}).

We underline that the length of some curves of $\SS_i$ can go to zero in the limit, then any {\em core} of the limit network is the union of some of these vanishing curves. 
Suppose that the network $\SS$ has a core at some point $P\in\SS$, then, at least an edge of $G$ is mapped into $P$ and the length of
at least one curve, let us say $\gamma_i$, goes to zero in the limit. Being the graph $G$ a tree, if $N\geqslant 2$ triple junctions are contained in the core, then $N+2$
curves (counted with multiplicity) with strictly positive length concur at $P$.
This fact can be easily proved by induction: if $N=2$, then two triple junctions
are present in the core and hence the length of the curve connecting the two junctions has gone to zero in the limit,
but the other four curves emanating from the two different junctions have still positive lengths. We suppose now that the statement holds for $N=\widetilde{N}$ and we show it for $N=\widetilde{N}+1$. With respect to the situation in which $\widetilde{N}$ triple junctions are in the core, we 
add an extra triple junction $\mathcal{O}$ to the core, but
to do so one of the original $\widetilde{N}+2$ curves emanating from the core
has to go to zero. However, the other two concurring curves to $\mathcal{O}$ have length bounded from below away from zero and now concur to $P$, thus there are 
$(\widetilde{N}+2)-1+2=\widetilde{N}+3$ curves with strictly positive length concurring at $P$
and the claim is proved.

We can suppose (up to reparametrization) that for every $i\in\mathbb{N}$, any curve $\gamma_i:[0,1]\to\mathbb{R}^2$ of $\SS_i$ is parametrized with constant modulus of its velocity, equal to its length.
Then we get
\begin{equation*}
\lim_{i\to\infty}\,\sup_{x,y\in [0,1]} \left\lvert 
{\tau_i(x)-\tau_i(y)}\right\rvert=0\,,
\end{equation*}
indeed, given $x,y\in[0,1]$, there holds
$$
 \left\lvert 
{\tau_i(x)-\tau_i(y)}\right\rvert= \bigg| \int_{s(x)}^{s(y)} \partial_s\tau_i\,{d}s\,\bigg|\leqslant \int_{\gamma_i} |\underline{k}_i| \, {d}s\leqslant \left( \int_{\gamma_i} |{k}_i|^2 \, {d}s \right)^{1/2} L(\gamma_i)^{1/2}
$$
and we obtain the conclusion, by passing to the limit.
Hence, the vanishing curves of $\SS_i$ are straighter and straighter, 
as $i\to\infty$ and for $i\in\NN$ large enough, so we can assume in the next argument that the unit tangent vectors are constant on each of such curves.\\
We describe the structure of the core. Let $i\in\mathbb{N}$ be sufficiently large and 
consider the longest simple ``path'' of curves of $\SS_i$ that go to the core 
of $\SS$ at $P$. We then orient the path and follow its edges. The ``assigned'' unit tangent vectors (possibly changed of sign on some edges in order to coincide with the orientation of the path) cannot ``turn'' of an angle of $60$ degrees in the same ``direction'' for two consecutive times along the path, otherwise, since $G$ is a tree with only triple junctions, without external vertices and with non--compact branches, the approximating networks must have a self--intersection (see Figure~\ref{fig6} below). 
\begin{figure}[H]
\begin{center}
\begin{tikzpicture}[scale=0.3,rotate=90]
\draw[color=black]
(-1,0)to[out=0,in=180,looseness=1](1,0)
(3,-3.46)to[out=0,in=180,looseness=1](5,-3.46)
(-3,-3.46)to[out=0,in=180,looseness=1](-5,-3.46)
(-1,0)to[out=120,in=-60,looseness=1] (-3,3.46)
(-1,0)to[out=-120,in=60,looseness=1] (-3,-3.46)
(1,0)to[out=60,in=-120,looseness=1] (3,3.46)
(1,0)to[out=-60,in=120,looseness=1] (3,-3.46)
(-3,-3.46)to[out=-60,in=120,looseness=1] (-0.12,-8.45)
(0.1,-8.823)to[out=-60,in=120,looseness=1] (1,-10.38)
(3,-3.46)to[out=-120,in=60,looseness=1] (-1,-10.38);
\draw[color=black,dashed]
(5,-3.46)to[out=0,in=180,looseness=1](6.4,-3.46)
(-5,-3.46)to[out=0,in=180,looseness=1](-6.2,-3.46)
(-3,3.46)to[out=120,in=-60,looseness=1] (-4,5.19)
(3,3.46)to[out=60,in=-120,looseness=1] (4,5.19);
\draw[very thick,shift={(1,0)},scale=2,color=black,rotate=210]
(0,0)to[out=90,in=-90,looseness=1](0,0.75)
(0,0.75)to[out=-45,in=135,looseness=1](0.15,0.6)
(0,0.75)to[out=-135,in=45,looseness=1](-0.15,0.6);
\draw[very thick,shift={(-1,0)},scale=2,color=black,rotate=-90]
(0,0)to[out=90,in=-90,looseness=1](0,0.75)
(0,0.75)to[out=-45,in=135,looseness=1](0.15,0.6)
(0,0.75)to[out=-135,in=45,looseness=1](-0.15,0.6);
\draw[very thick, shift={(-3,-3.46)},scale=2,color=black,rotate=330]
(0,0)to[out=90,in=-90,looseness=1](0,0.75)
(0,0.75)to[out=-45,in=135,looseness=1](0.15,0.6)
(0,0.75)to[out=-135,in=45,looseness=1](-0.15,0.6);
\fill(1,0) circle (4pt);
\fill(-1,0) circle (4pt);
\fill(3,-3.46)circle (4pt);
\fill(-3,-3.46)circle (4pt);
\path[font=]
(-5.4,-8) node[below] {$G$};
\end{tikzpicture}\qquad\qquad
\begin{tikzpicture}[scale=0.3,rotate=90]
\draw[color=black]
(0,0)to[out=0,in=180,looseness=1](3,0)
(0,0)to[out=0,in=180,looseness=1](-3,0)
(0,0)to[out=120,in=-60,looseness=1] (-2,3.46)
(0,0)to[out=60,in=-120,looseness=1] (2,3.46)
(0,0)to[out=-60,in=120,looseness=1](2,-3.46)
(0,0)to[out=-120,in=60,looseness=1] (-2,-3.46);
\draw[color=black,dashed]
(3,0)to[out=0,in=180,looseness=1](4.5,0)
(-3,0)to[out=0,in=180,looseness=1](-4.5,0)
(-2,3.46)to[out=120,in=-60,looseness=1] (-3,5.19)
(2,3.46)to[out=60,in=-120,looseness=1] (3,5.19)
(-2,-3.46)to[out=-120,in=60,looseness=1] (-3,-5.19)
(2,-3.46)to[out=-60,in=120,looseness=1] (3,-5.19);
\fill(0,0) circle (4pt);
\path[font=]
(-5.4,-8) node[below] {$\mathbb{S}$};
\end{tikzpicture}\qquad
\begin{tikzpicture}[scale=0.3,rotate=90]
\draw[color=black]
(-1,0)to[out=0,in=180,looseness=1](1,0)
(-1,0)to[out=-120,in=60,looseness=1] (-3,-3.46)
(1,0)to[out=-60,in=120,looseness=1] (3,-3.46);
\draw[very thick,shift={(1,0)},scale=2,color=black,rotate=210]
(0,0)to[out=90,in=-90,looseness=1](0,0.75)
(0,0.75)to[out=-45,in=135,looseness=1](0.15,0.6)
(0,0.75)to[out=-135,in=45,looseness=1](-0.15,0.6);
\draw[very thick,shift={(-1,0)},scale=2,color=black,rotate=-90]
(0,0)to[out=90,in=-90,looseness=1](0,0.75)
(0,0.75)to[out=-45,in=135,looseness=1](0.15,0.6)
(0,0.75)to[out=-135,in=45,looseness=1](-0.15,0.6);
\draw[very thick, shift={(-3,-3.46)},scale=2,color=black,rotate=330]
(0,0)to[out=90,in=-90,looseness=1](0,0.75)
(0,0.75)to[out=-45,in=135,looseness=1](0.15,0.6)
(0,0.75)to[out=-135,in=45,looseness=1](-0.15,0.6);
\path[font=\footnotesize]
(-4,-5) node[below] {The core of $\mathbb{S}$};
\path[white, font=]
(-5.4,-8) node[below] {$\mathbb{S}$};
\end{tikzpicture}\qquad
\end{center}
\begin{caption}{If the assigned unit tangent vector ``turns'' of an angle of $60$ degrees in the same direction for two consecutive times, $G$ has self--intersections.
An example of such a pair $(G,\mathbb{S})$.\label{fig6}}
\end{caption}
\end{figure}
Hence, if the assigned unit tangent vector ``turns'' of an angle of $60$ degrees then it must ``turn'' back, in passing from an edge to another along such longest path. This means that at the initial/final point of such path, either the two assigned unit tangent vectors are the same (when the number of edges is odd) or they differ of $60$ degrees (when the number of edges is even). By a simple check, we can then see that, in the first case the four curves images of the four non--collapsed edges exiting from such initial/final points of the path, have four different exterior unit tangent vectors at $P$ (opposite in pairs), in the second case, they have three exterior unit tangent vectors at $P$ which are non--proportional each other.
\begin{figure}[H]
\begin{center}
\begin{tikzpicture}[scale=0.30]
\draw[color=black]
(2,1.73)to[out=0,in=180,looseness=1](4,1.73)
(2,-1.73)to[out=0,in=180,looseness=1](4,-1.73)
(2,1.73)to[out=120,in=-60,looseness=1] (1,3.46)
(2,-1.73)to[out=-120,in=60,looseness=1] (1,-3.46)
(1,0)to[out=60,in=-120,looseness=1] (2,1.73)
(1,0)to[out=-60,in=120,looseness=1] (2,-1.73);
\draw[color=black,dashed]
(4,1.73)to[out=0,in=180,looseness=1](5,1.73)
(4,-1.73)to[out=0,in=180,looseness=1](5,-1.73)
(1,3.46)to[out=120,in=-60,looseness=1] (0,5.19)
(1,-3.46)to[out=-120,in=60,looseness=1] (0,-5.19);
\fill(1,0) circle (3pt);
\fill(2,1.73) circle (3pt);
\fill(2,-1.73) circle (3pt);
\path[font=](3,-4) node[below] {$G$};
\draw[very thick, shift={(1,0)},scale=1,color=black,rotate=330]
(0,0)to[out=90,in=-90,looseness=1](0,0.75)
(0,0.75)to[out=-45,in=135,looseness=1](0.15,0.6)
(0,0.75)to[out=-135,in=45,looseness=1](-0.15,0.6);
\draw[very thick, shift={(2,-1.73)},scale=1,color=black,rotate=-330]
(0,0)to[out=90,in=-90,looseness=1](0,0.75)
(0,0.75)to[out=-45,in=135,looseness=1](0.15,0.6)
(0,0.75)to[out=-135,in=45,looseness=1](-0.15,0.6);
\end{tikzpicture}\qquad\quad
\begin{tikzpicture}[scale=0.30]
\draw[color=black]
(2,1.73)to[out=0,in=180,looseness=1](4,1.73)
(2,1.73)to[out=120,in=-60,looseness=1] (1,3.46)
(2,1.73)to[out=-120,in=60,looseness=1] (1,0);
\draw[color=black,dashed]
(4,1.73)to[out=0,in=180,looseness=1](5,1.73)
(1,3.46)to[out=120,in=-60,looseness=1] (0,5.19)
(1,0)to[out=-120,in=60,looseness=1] (0,-1.73);
\fill(2,1.73) circle (3pt);
\path[font=\footnotesize](3,1.53) node[above] {$2$};
\path[font=\footnotesize](0.8,3.46) node[below] {$1$};
\path[font=\footnotesize](0.8,0) node[above] {$1$};
\path[font=](3,-2.27) node[below] {$\mathbb{S}$};
\end{tikzpicture}\qquad\quad
\begin{tikzpicture}[scale=0.30]
\draw[color=black]
(2,1.73)to[out=0,in=180,looseness=1](4,1.73)
(2,1.73)to[out=120,in=-60,looseness=1] (1,3.46)
(2,-1.73)to[out=-120,in=60,looseness=1] (1,-3.46)
(1,0)to[out=60,in=-120,looseness=1] (2,1.73)
(1,0)to[out=-60,in=120,looseness=1] (2,-1.73);
\draw[color=black,dashed]
(4,1.73)to[out=0,in=180,looseness=1](5,1.73)
(1,3.46)to[out=120,in=-60,looseness=1] (0,5.19);
\draw
(1,-3.46)--(-1,-3.46)
(1,-3.46)--(2,-5.19);
\draw[dashed]
(2,-5.19)--(3,-6.92)
(-1,-3.46)--(-2,-3.46);
\fill(1,0) circle (3pt);
\fill(2,1.73) circle (3pt);
\fill(2,-1.73) circle (3pt);
\fill(1,-3.46) circle (3pt);
\path[font=](4,-5.73) node[below] {$G$};
\draw[very thick, shift={(1,0)},scale=1,color=black,rotate=330]
(0,0)to[out=90,in=-90,looseness=1](0,0.75)
(0,0.75)to[out=-45,in=135,looseness=1](0.15,0.6)
(0,0.75)to[out=-135,in=45,looseness=1](-0.15,0.6);
\draw[very thick, shift={(1,-3.46)},scale=1,color=black,rotate=330]
(0,0)to[out=90,in=-90,looseness=1](0,0.75)
(0,0.75)to[out=-45,in=135,looseness=1](0.15,0.6)
(0,0.75)to[out=-135,in=45,looseness=1](-0.15,0.6);
\draw[very thick, shift={(2,-1.73)},scale=1,color=black,rotate=-330]
(0,0)to[out=90,in=-90,looseness=1](0,0.75)
(0,0.75)to[out=-45,in=135,looseness=1](0.15,0.6)
(0,0.75)to[out=-135,in=45,looseness=1](-0.15,0.6);
\end{tikzpicture}\qquad\quad
\begin{tikzpicture}[scale=0.30]
\draw[color=black]
(0,1.73)to[out=0,in=180,looseness=1](2,1.73)
(2,1.73)to[out=0,in=180,looseness=1](4,1.73)
(2,1.73)to[out=120,in=-60,looseness=1] (1,3.46)
(2,1.73)to[out=-60,in=120,looseness=1] (3,0);
\draw[color=black,dashed]
(-1,1.73)to[out=0,in=180,looseness=1](0,1.73)
(4,1.73)to[out=0,in=180,looseness=1](5,1.73)
(1,3.46)to[out=120,in=-60,looseness=1] (0,5.19)
(3,0)to[out=-60,in=120,looseness=1] (4,-1.73);
\fill(2,1.73) circle (3pt);
\path[font=\footnotesize](3,1.53) node[above] {$1$};
\path[font=\footnotesize](1,1.90) node[below] {$1$};
\path[font=\footnotesize](0.6,3.66) node[below] {$1$};
\path[font=\footnotesize](3.4,-.2) node[above] {$1$};
\path[font=](3,-2.27) node[below] {$\mathbb{S}$};
\end{tikzpicture}\qquad\quad
\end{center}
\begin{caption}{Examples of the edges at the initial/final points of the longest simple path in $G$ and of the relative curves in $\SS$,
the numbers $1$ and $2$ denote their multiplicity.}
\end{caption}
\end{figure}
If then there is a $3$--point or a core at some point $P\not=0$, since at most two of the four directions in the first case above and at most one of the three directions in the second case, can belong to the straight line through $P$ and the origin, there are always at least two distinct non--straight Abresch--Langer curves arriving/starting at $P$. Clearly, this property holds also if there is no core, but $P$ is simply a $3$--point. 

Let us consider $\SS^\prime \subseteq \SS$, which consists of $\SS$ with the interior of all the pieces of straight lines removed and let $\sigma^i$ one of the two curves above. We follow $\sigma^i$ till its other end--point $Q$. At this end--point, even if there is a core at $Q$, there is
always another different non--straight curve $\sigma^{j}$ to continue moving in $\SS$ avoiding the pieces of straight lines (hence staying far from the origin). Actually, either the underlying intervals $I_i$ and $I_j$ are concurrent at the vertex corresponding to $Q$ in the graph $G$ or there is a path in $G$ (collapsed in the core at $Q$) joining $I_i$ and $I_j$. We then go on with this path on $\SS$ (and on $G$) till, looking at things on the graph $G$, we arrive at an already considered vertex, which happens since the number of vertices of $G$ is finite, obtaining a closed loop, hence, a contradiction. Thus, $\SS^\prime$ cannot contain $3$--points or cores outside the origin. If anyway $\SS$ contains a non--straight Abresch--Langer curve, we can repeat this argument getting again a contradiction, hence, we are done with the first part of the lemma, since then $\SS$ can only consist of halflines from the origin.

Now we assume that $G$ is connected and $\SS$ is a network with multiplicity one, composed of halflines from the origin.\\
If there is no core, $\SS$ is homeomorphic to $G$ and composed only by halflines for the origin, hence $G$ has at most one vertex, by connectedness. If $G$ has no vertices, then $\SS$ must be a line, if it has a $3$--point, $\SS$ is a standard triod.\\
If there is a core in the origin, by the definition of degenerate
regular network it follows that the halflines of $\SS$ can only have
six possible directions, by the $120$ degrees condition, hence, by the
unit multiplicity hypothesis, the graph $G$ is a tree in the plane
with at most six unbounded edges. Arguing as in the first part of the
lemma, if $N$ denotes the number (greater than one) of $3$--points
contained in the core, it follows that $N$ can only assume the values
$2, 3, 4$. Repeating the argument of the ``longest path'', we
immediately also exclude the case $N=3$, since there would be a pair
of coincident halflines in $\SS$, against the multiplicity--one
hypothesis, while for $N=4$ we have only two possible situations,
described at the bottom of the following figure.
\begin{figure}[H]
\begin{center}
\begin{tikzpicture}[scale=0.30]
\draw[color=black]
(-1,0)to[out=0,in=180,looseness=1](1,0)
(-1,0)to[out=120,in=-60,looseness=1] (-3,3.46)
(-1,0)to[out=-120,in=60,looseness=1] (-3,-3.46)
(1,0)to[out=60,in=-120,looseness=1] (3,3.46)
(1,0)to[out=-60,in=120,looseness=1] (3,-3.46);
\draw[color=black,dashed]
(-3,3.46)to[out=120,in=-60,looseness=1] (-4,5.19)
(-3,-3.46)to[out=-120,in=60,looseness=1] (-4,-5.19)
(3,3.46)to[out=60,in=-120,looseness=1] (4,5.19)
(3,-3.46)to[out=-60,in=120,looseness=1] (4,-5.19);
\fill(1,0) circle (4pt);
\fill(-1,0) circle (4pt);
\draw[very thick, shift={(1,0)},color=black,scale=1,rotate=90]
(0,0)to[out=90,in=-90,looseness=1](0,0.75)
(0,0.75)to[out=-45,in=135,looseness=1](0.15,0.6)
(0,0.75)to[out=-135,in=45,looseness=1](-0.15,0.6);
\path[font=]
(-0,-4) node[below] {$G$};
\end{tikzpicture}\quad
\begin{tikzpicture}[scale=0.30]
\draw[color=black]
(0,0)to[out=120,in=-60,looseness=1] (-2,3.46)
(0,0)to[out=-120,in=60,looseness=1] (-2,-3.46)
(0,0)to[out=60,in=-120,looseness=1] (2,3.46)
(0,0)to[out=-60,in=120,looseness=1] (2,-3.46);
\draw[color=black,dashed]
(-2,3.46)to[out=120,in=-60,looseness=1] (-3,5.19)
(-2,-3.46)to[out=-120,in=60,looseness=1] (-3,-5.19)
(2,3.46)to[out=60,in=-120,looseness=1] (3,5.19)
(2,-3.46)to[out=-60,in=120,looseness=1] (3,-5.19);
\fill(0,0) circle (4pt);
\path[font=]
(-0,-4) node[below] {$\mathbb{S}$};
\end{tikzpicture}\quad
\begin{tikzpicture}[scale=0.30]
\draw[color=black]
(-1,0)to[out=0,in=180,looseness=1](1,0);
\draw[very thick, shift={(1,0)},color=black,scale=1,rotate=90]
(0,0)to[out=90,in=-90,looseness=1](0,0.75)
(0,0.75)to[out=-45,in=135,looseness=1](0.15,0.6)
(0,0.75)to[out=-135,in=45,looseness=1](-0.15,0.6);
\path[font=\footnotesize]
(-0,-4) node[below] {The core of $\mathbb{S}$};
\end{tikzpicture}\qquad
\begin{tikzpicture}[scale=0.30]
\draw[color=black]
(-1,0)to[out=0,in=180,looseness=1](1,0)
(2,1.73)to[out=0,in=180,looseness=1](4,1.73)
(2,-1.73)to[out=0,in=180,looseness=1](4,-1.73)
(2,1.73)to[out=120,in=-60,looseness=1] (1,3.46)
(2,-1.73)to[out=-120,in=60,looseness=1] (1,-3.46)
(1,0)to[out=60,in=-120,looseness=1] (2,1.73)
(1,0)to[out=-60,in=120,looseness=1] (2,-1.73);
\draw[color=black,dashed]
(-2,0)to[out=0,in=180,looseness=1](-1,0)
(4,1.73)to[out=0,in=180,looseness=1](5,1.73)
(4,-1.73)to[out=0,in=180,looseness=1](5,-1.73)
(1,3.46)to[out=120,in=-60,looseness=1] (0,5.19)
(1,-3.46)to[out=-120,in=60,looseness=1] (0,-5.19);
\fill(1,0) circle (4pt);
\fill(2,1.73) circle (4pt);
\fill(2,-1.73) circle (4pt);
\path[font=](3,-4) node[below] {$G$};
\draw[very thick, shift={(1,0)},scale=1,color=black,rotate=330]
(0,0)to[out=90,in=-90,looseness=1](0,0.75)
(0,0.75)to[out=-45,in=135,looseness=1](0.15,0.6)
(0,0.75)to[out=-135,in=45,looseness=1](-0.15,0.6);
\draw[very thick, shift={(2,-1.73)},scale=1,color=black,rotate=-330]
(0,0)to[out=90,in=-90,looseness=1](0,0.75)
(0,0.75)to[out=-45,in=135,looseness=1](0.15,0.6)
(0,0.75)to[out=-135,in=45,looseness=1](-0.15,0.6);
\end{tikzpicture}\quad
\begin{tikzpicture}[scale=0.30]
\draw[color=black]
(0,1.73)to[out=0,in=180,looseness=1](2,1.73)
(2,1.73)to[out=0,in=180,looseness=1](4,1.73)
(2,1.73)to[out=120,in=-60,looseness=1] (1,3.46)
(2,1.73)to[out=-120,in=60,looseness=1] (1,0);
\draw[color=black,dashed]
(-1,1.73)to[out=0,in=180,looseness=1](0,1.73)
(4,1.73)to[out=0,in=180,looseness=1](5,1.73)
(1,3.46)to[out=120,in=-60,looseness=1] (0,5.19)
(1,0)to[out=-120,in=60,looseness=1] (0,-1.73);
\fill(2,1.73) circle (4pt);
\path[font=](3,-2.27) node[below] {$\mathbb{S}$};
\end{tikzpicture}\quad
\begin{tikzpicture}[scale=0.30]
\draw[color=black]
(1,0)to[out=60,in=-120,looseness=1] (2,1.73)
(1,0)to[out=-60,in=120,looseness=1] (2,-1.73);
\draw[very thick, shift={(1,0)},scale=1,color=black,rotate=330]
(0,0)to[out=90,in=-90,looseness=1](0,0.75)
(0,0.75)to[out=-45,in=135,looseness=1](0.15,0.6)
(0,0.75)to[out=-135,in=45,looseness=1](-0.15,0.6);
\draw[very thick, shift={(2,-1.73)},scale=1,color=black,rotate=-330]
(0,0)to[out=90,in=-90,looseness=1](0,0.75)
(0,0.75)to[out=-45,in=135,looseness=1](0.15,0.6)
(0,0.75)to[out=-135,in=45,looseness=1](-0.15,0.6);
\path[font=\footnotesize]
(2,-4) node[below] {The core of $\mathbb{S}$};
\end{tikzpicture}\qquad\quad
\begin{tikzpicture}[scale=0.30]
\draw[color=black]
(-1,0)to[out=0,in=180,looseness=1](1,0)
(2,1.73)to[out=0,in=180,looseness=1](4,1.73)
(2,-1.73)to[out=0,in=180,looseness=1](4,-1.73)
(2,1.73)to[out=120,in=-60,looseness=1] (1,3.46)
(2,-1.73)to[out=-120,in=60,looseness=1] (1,-3.46)
(1,0)to[out=60,in=-120,looseness=1] (2,1.73)
(1,0)to[out=-60,in=120,looseness=1] (2,-1.73);
\draw[color=black,dashed]
(-2,0)to[out=0,in=180,looseness=1](-1,0)
(4,1.73)to[out=0,in=180,looseness=1](5,1.73)
(4,-1.73)to[out=0,in=180,looseness=1](5,-1.73)
(1,3.46)to[out=120,in=-60,looseness=1] (0,5.19);
\draw
(1,-3.46)--(-1,-3.46)
(1,-3.46)--(2,-5.19);
\draw[dashed]
(2,-5.19)--(3,-6.92)
(-1,-3.46)--(-2,-3.46);
\fill(1,0) circle (3pt);
\fill(2,1.73) circle (3pt);
\fill(2,-1.73) circle (3pt);
\fill(1,-3.46) circle (3pt);
\path[font=](4,-5.73) node[below] {$G$};
\draw[very thick, shift={(1,0)},scale=1,color=black,rotate=330]
(0,0)to[out=90,in=-90,looseness=1](0,0.75)
(0,0.75)to[out=-45,in=135,looseness=1](0.15,0.6)
(0,0.75)to[out=-135,in=45,looseness=1](-0.15,0.6);
\draw[very thick, shift={(1,-3.46)},scale=1,color=black,rotate=330]
(0,0)to[out=90,in=-90,looseness=1](0,0.75)
(0,0.75)to[out=-45,in=135,looseness=1](0.15,0.6)
(0,0.75)to[out=-135,in=45,looseness=1](-0.15,0.6);
\draw[very thick, shift={(2,-1.73)},scale=1,color=black,rotate=-330]
(0,0)to[out=90,in=-90,looseness=1](0,0.75)
(0,0.75)to[out=-45,in=135,looseness=1](0.15,0.6)
(0,0.75)to[out=-135,in=45,looseness=1](-0.15,0.6);
\end{tikzpicture}\quad
\begin{tikzpicture}[scale=0.30]
\draw[color=black]
(0,1.73)to[out=0,in=180,looseness=1](2,1.73)
(2,1.73)to[out=0,in=180,looseness=1](4,1.73)
(2,1.73)to[out=120,in=-60,looseness=1] (1,3.46)
(2,1.73)to[out=-60,in=120,looseness=1] (3,0);
\draw[color=black,dashed]
(-1,1.73)to[out=0,in=180,looseness=1](0,1.73)
(4,1.73)to[out=0,in=180,looseness=1](5,1.73)
(1,3.46)to[out=120,in=-60,looseness=1] (0,5.19)
(3,0)to[out=-60,in=120,looseness=1] (4,-1.73);
\fill(2,1.73) circle (4pt);
\path[font=](3,-2.27) node[below] {$\mathbb{S}$};
\end{tikzpicture}\quad
\begin{tikzpicture}[scale=0.30]
\draw[color=black]
(2,-1.73)to[out=-120,in=60,looseness=1] (1,-3.46)
(1,0)to[out=60,in=-120,looseness=1] (2,1.73)
(1,0)to[out=-60,in=120,looseness=1] (2,-1.73);
\draw[very thick, shift={(1,0)},scale=1,color=black,rotate=330]
(0,0)to[out=90,in=-90,looseness=1](0,0.75)
(0,0.75)to[out=-45,in=135,looseness=1](0.15,0.6)
(0,0.75)to[out=-135,in=45,looseness=1](-0.15,0.6);
\draw[very thick, shift={(1,-3.46)},scale=1,color=black,rotate=330]
(0,0)to[out=90,in=-90,looseness=1](0,0.75)
(0,0.75)to[out=-45,in=135,looseness=1](0.15,0.6)
(0,0.75)to[out=-135,in=45,looseness=1](-0.15,0.6);
\draw[very thick, shift={(2,-1.73)},scale=1,color=black,rotate=-330]
(0,0)to[out=90,in=-90,looseness=1](0,0.75)
(0,0.75)to[out=-45,in=135,looseness=1](0.15,0.6)
(0,0.75)to[out=-135,in=45,looseness=1](-0.15,0.6);
\path[font=\footnotesize]
(2,-5.73) node[below] {The core of $\mathbb{S}$};
\end{tikzpicture}\qquad
\begin{tikzpicture}[scale=0.30]
\draw[color=black]
(-1,0)to[out=120,in=-60,looseness=1] (-2,1.73)
(-1,0)to[out=-120,in=60,looseness=1] (-2,-1.73)
(-1,0)to[out=0,in=180,looseness=1](1,0)
(2,1.73)to[out=0,in=180,looseness=1](4,1.73)
(2,-1.73)to[out=0,in=180,looseness=1](4,-1.73)
(2,1.73)to[out=120,in=-60,looseness=1] (1,3.46)
(2,-1.73)to[out=-120,in=60,looseness=1] (1,-3.46)
(1,0)to[out=60,in=-120,looseness=1] (2,1.73)
(1,0)to[out=-60,in=120,looseness=1] (2,-1.73);
\draw[color=black,dashed]
 (-2,1.73)to[out=120,in=-60,looseness=1](-3,3.46)
 (-2,-1.73)to[out=-120,in=60,looseness=1](-3,-3.46)
(4,1.73)to[out=0,in=180,looseness=1](5,1.73)
(4,-1.73)to[out=0,in=180,looseness=1](5,-1.73)
(1,3.46)to[out=120,in=-60,looseness=1] (0,5.19)
(1,-3.46)to[out=-120,in=60,looseness=1] (0,-5.19);
\fill(-1,0)circle (4pt);
\fill(1,0) circle (4pt);
\fill(2,1.73) circle (4pt);
\fill(2,-1.73) circle (4pt);
\path[font=](3,-4) node[below] {$G$};
\draw[very thick, shift={(1,0)},scale=1,color=black,rotate=330]
(0,0)to[out=90,in=-90,looseness=1](0,0.75)
(0,0.75)to[out=-45,in=135,looseness=1](0.15,0.6)
(0,0.75)to[out=-135,in=45,looseness=1](-0.15,0.6);
\draw[very thick, shift={(2,-1.73)},scale=1,color=black,rotate=-330]
(0,0)to[out=90,in=-90,looseness=1](0,0.75)
(0,0.75)to[out=-45,in=135,looseness=1](0.15,0.6)
(0,0.75)to[out=-135,in=45,looseness=1](-0.15,0.6);
\end{tikzpicture}\quad
\begin{tikzpicture}[scale=0.30]
\draw[color=black]
(2,1.73)to[out=0,in=180,looseness=1](4,1.73)
(2,1.73)to[out=120,in=-60,looseness=1] (1,3.46)
(2,1.73)to[out=-120,in=60,looseness=1] (1,0);
\draw[color=black,dashed]
(4,1.73)to[out=0,in=180,looseness=1](5,1.73)
(1,3.46)to[out=120,in=-60,looseness=1] (0,5.19)
(1,0)to[out=-120,in=60,looseness=1] (0,-1.73);
\fill(2,1.73) circle (4pt);
\path[font=](3,-2.27) node[below] {$\mathbb{S}$};
\end{tikzpicture}\quad
\begin{tikzpicture}[scale=0.30]
\draw[color=black]
(1,0)to[out=60,in=-120,looseness=1] (2,1.73)
(1,0)to[out=-60,in=120,looseness=1] (2,-1.73);
\draw[very thick, shift={(1,0)},scale=1,color=black,rotate=330]
(0,0)to[out=90,in=-90,looseness=1](0,0.75)
(0,0.75)to[out=-45,in=135,looseness=1](0.15,0.6)
(0,0.75)to[out=-135,in=45,looseness=1](-0.15,0.6);
\draw[very thick, shift={(2,-1.73)},scale=1,color=black,rotate=-330]
(0,0)to[out=90,in=-90,looseness=1](0,0.75)
(0,0.75)to[out=-45,in=135,looseness=1](0.15,0.6)
(0,0.75)to[out=-135,in=45,looseness=1](-0.15,0.6);
\path[font=\footnotesize]
(2,-3) node[below] {The longest}
(2,-4) node[below] {simple path}
(2,-5) node[below] {in the core of $\SS$};
\end{tikzpicture}\qquad
\end{center}
\begin{caption}{The possible local structure of the graphs $G$, with relative networks $\SS$ and cores, for $N=2, 3, 4$.\label{fig8}}
\end{caption}
\end{figure}
\noindent Hence, if $N=4$, in both two situations above there is in $\SS$ at least one halfline with multiplicity two, thus such case is also excluded.\\
Then, we conclude that the only possible network with a core is when $N=2$ and $\SS$ is given by two lines intersecting at the origin forming angles of $120/60$ degrees and the core consists of a collapsed segment which must have an ``assigned'' unit tangent vector bisecting the two angles of $120$ degrees formed by the four halflines.
\end{proof}

\begin{lem}\label{thm:densitybound}
Let $\SS=\bigcup_{i=1}^n\sigma^i(I_i)$ be a degenerate regular shrinker which is $C^1\loc$--limit of regular networks homeomorphic to the underlying graph $G$ of\, $\SS$ (as in Definition~\ref{degnet}) and assume that $\Theta_{\SS}<\Theta_{\SS^1}$. 
Then, the graph $G$ of $\SS$ is a tree. Thus, $\SS$ is either a multiplicity--one line or a standard triod. 
\end{lem}
\begin{proof}
By the hypotheses, we see that $G$ is a planar graph. We assume that $G$ is not a tree, that is, it contains a loop, then we can find a (possibly smaller) loop bounding a region. If such loop is in a core at some point $P$, it is easy to see, by the degenerate $120$ degrees condition, that such region has six edges and, arguing as in Lemma~\ref{lemmatree}, that there must always be at least two non--collapsed, non--straight Abresch--Langer curves arriving/starting at $P$ in different directions.

Then, if we assume that the complement of $\SS$ in $\R^2$ contains no bounded components, repeating the argument in the proof of the previous lemma, it follows that $\SS$ consists of a union of halflines for the origin and the loops of $G$ are all collapsed in the core. Then, by what we said above, there must be at least six halflines emanating from (the core at) the origin. This implies that $\Theta_{\SS}\geqslant 3$, which is a contradiction.

Let now $B$ be a bounded component of the complement of $\SS$ and $\gamma$ a connected component of the sub--network of $\SS$ which bounds $B$, counted
with unit multiplicity. Since $\gamma$ is an embedded, closed curve, smooth with corners and no triple junctions, we can evolve it by ``classical'' curve shortening flow $\gamma_t$, for $t\in[-1/2,t_0)$ where we set $\gamma_{-1/2}=\gamma$, until it shrinks at some $t_0>-1/2$ to a ``round'' point $x_0\in\R^2$ (by the works of Angenent, Gage, Grayson, Hamilton~\cite{angen1,angen2,angen3,gage,gage0,gaha1,gray1}, see Remark~\ref{gremh}).\\
By the monotonicity formula, we have
$$
\int_{\gamma}\rho_{x_0,t_0}(\cdot,-1/2)\, ds \geqslant \Theta_{\SS^1}
$$
and, by the work of Colding--Minicozzi~\cite[Section~7.2]{coldmin6},
there holds
\begin{equation}\label{eq:locreg.0.0}
\Theta_{\SS} = \int_\SS \rho_{0,0}(\cdot,-1/2)\, d\overline{s} = \sup_{x_0\in\R^2,
 t_0>-1/2} \int_\SS \rho_{x_0,t_0}(\cdot,-1/2)\, d\overline{s}\,.
\end{equation}
Then,
$$
\Theta_{\SS}\geqslant \int_\SS \rho_{x_0,t_0}(\cdot,-1/2)\, d\overline{s} \geqslant
\int_{\gamma}\rho_{x_0,t_0}(\cdot,-1/2)\, ds \geqslant \Theta_{\SS^1}\,,
$$
which is a contradiction and we are done.
\end{proof}

\subsection{Geometric properties of the flow}\label{geopropsub}
Before proceeding, we show some geometric properties of the curvature flow 
of a network that we will need in the sequel.

\begin{prop}\label{omegaok} 
Let $\SS_t$ be the curvature flow of a regular network in a smooth, convex,
bounded, open set $\Omega$, with fixed end--points on the boundary of
$\Omega$, for $t\in[0,T)$. Then for every time $t\in[0,T)$ the
network $\SS_t$ intersects the boundary of $\Omega$ only at the
end--points, and such intersections are transversal for every positive
time. Moreover, $\SS_t$ remains embedded.
\end{prop}
\begin{proof}
By continuity the $3$--points cannot hit the boundary of $\Omega$ at least 
for some time $T^\prime>0$. The convexity of $\Omega$ and 
the strong maximum principle (see~\cite{prowein}) 
imply that the network cannot intersect the boundary for the first time
at an inner regular point. 
As a consequence, if $t_0>0$ is the ``first time'' when
the $\SS_t$ intersects the boundary at an inner point, this latter has
to be a $3$--point. The minimality of $t_0$ is then easily
contradicted by the convexity of $\Omega$, the $120$ degrees condition
and the nonzero length of the curves of $\SS_{t_0}$.\\
Even if some of the curves of the initial network are tangent to
$\partial\Omega$ at the end--points, by the strong maximum
principle, as $\Omega$ is convex, the intersections become immediately
transversal and stay so for every subsequent time.\\
Finally, if the evolution $\SS_t$ loses embeddedness for the first time,
this cannot happen either 
at a boundary point, by the argument above, nor 
at a $3$--point, by the $120$ degrees condition.
Hence it must happen at interior regular points, 
but this contradicts the strong maximum principle.
\end{proof}

\begin{prop}\label{omegaok2} 
In the same hypotheses of the previous proposition, if the smooth, bounded, open set $\Omega$ is strictly convex, for every fixed end--point $P^r$ on the boundary of $\Omega$, for $r\in\{1,2,\dots,l\}$, there is a time $t_r\in(0,T)$ and an angle $\alpha_r$ smaller than $\pi/2$ such that the curve of the network arriving at $P^r$ form an angle less that $\alpha_r$ with the inner normal to the boundary of $\Omega$, for every time $t\in(t_r,T)$.
\end{prop}
\begin{proof}
We observe that the evolving network $\SS_t$ is contained in the convex set $\Om_t\subseteq \Om$,
obtained by letting $\partial \Om$ (which is a finite set of smooth curves with end--points $P^r$) move by curvature 
keeping fixed the end--points $P^r$ (see~\cite{huisk2,stahl1,stahl2}). By the strict convexity of $\Omega$ and strong maximum principle, for every positive $t>0$, the two curves of the boundary of $\Omega$ concurring at $P^r$ form an angle smaller than $\pi$ which is not increasing in time. Hence, the statement of the proposition follows.
\end{proof}

We briefly discuss now the behavior of the area of regions enclosed by
the evolving regular network $\SS_t$. Let us suppose that a (moving) region
${\mathcal{A}}(t)$ is bounded by some curves
$\gamma^1,\gamma^2,\dots,\gamma^m$ and let $A(t)$ its area. Possibly
reparametrizing these curves which form the loop
$\ell=\bigcup_{i=1}^m\gamma^i$ in the network, we can assume that
$\ell$ is parametrized counterclockwise, hence, the curvature $k$ is
positive at the convexity points of the boundary of
${\mathcal{A}}(t)$. Then, we have
$$
A'(t)=-\sum_{i=1}^m\int_{\gamma^i}\langle \gamma^i_t\,\vert\,\nu\rangle\,ds
=-\sum_{i=1}^m\int_{\gamma^i}\langle k\nu\,\vert\,\nu\rangle\,ds
=-\sum_{i=1}^m\int_{\gamma^i} k\,ds
=-\sum_{i=1}^m\Delta\theta_i
$$
where $\Delta\theta_i$ is the difference in the angle between the unit
tangent vector $\tau$ and the unit coordinate vector $e_1\in\R^2$ at the
final and initial point of the curve $\gamma^i$, indeed (supposing the
unit tangent vector of the curve $\gamma^i$ ``lives'' in the second quadrant
of $\R^2$ -- the other cases are analogous) there holds 
$$
\partial_s\theta_i=\partial_s\arccos \langle\tau\,\vert\,e_1\rangle
=-\frac{\langle\tau_s\,\vert\,e_1\rangle}{\sqrt{1-\langle\tau\,\vert\,e_1\rangle^2}}=k\,,
$$
so
$$
A'(t)=-\sum_{i=1}^m\int_{\gamma^i}\partial_s\theta_i\,ds=-\sum_{i=1}^m\Delta\theta_i
$$
Being $\ell$ a closed loop and considering that at all the end--points of
the curves $\gamma^i$ the angle of the unit tangent vector ``jumps''
of $120$ degrees, we have
\begin{equation}\label{gb}
m\pi/3+\sum_{i=1}^m\Delta\theta_i=m\pi/3+\sum_{i=1}^m\int_{\gamma^i} k\,ds=2\pi\,,
\end{equation}
hence
\begin{equation}\label{areaevolreg}
A'(t)=-(2-m/3)\pi
\end{equation}
(this is called {\em von~Neumann rule}, see~\cite{vn}).

An immediate consequence is that the area of every region fully bounded by the curves of the network evolves linearly and, more precisely, it
increases if the region has more than six edges, it is constant with six
edges and it decreases if the edges are less than six. Moreover, this
implies that if a region with less than six edges is present, 
with area $A_0$ at time $t=0$, the maximal time $T$ of existence of a smooth flow is finite and
\begin{equation}\label{stimaT0}
T\leqslant \frac{A_0}{(2-m/3)\pi}\leqslant\frac{3A_0}{\pi}\,. 
\end{equation}

\begin{rem}\label{schreg}
Since every bounded region contained in a shrinker must
decrease its area during the curvature flow of such shrinker (since it
is homothetically contracting), another consequence is that the only
compact regions that can be present in a 
regular shrinker are bounded by less than six curves 
(actually this conclusion also holds for the
``visible'' regions -- not the cores -- of any degenerate regular
shrinker).\\
Moreover, letting a shrinker evolve, since every bounded region must
collapse after a time interval of 1/2, 
the area of such a region is only dependent on the number $m$ of its
edges (less than 6), 
by equation~\eqref{areaevolreg}, indeed
$$
A(0)=A(0)-A(1/2)=-\int_0^{1/2} A'(t)\,dt=\int_0^{1/2} (2-m/3)\pi\,dt=(2-m/3)\pi/2\,.
$$
This implies that the possible structures (topology) of the shrinkers with equibounded diameter are finite.

It is actually conjectured in~\cite[Conjecture~3.26]{haettenschweiler}
that there is an upper bound for the possible number of bounded
regions of a shrinker. This would imply that the possible topological 
structures of shrinkers are finite.
\end{rem}

We explain now a geometric construction that we will use several times in the following.\\
We consider the curvature flow of network $\SS_t$ in a strictly convex set
$\Omega$, with fixed end--points on $\partial\Omega$
labeled by $\{P^1, P^2,\dots, P^l\}$, in a maximal time interval $[0,T)$.
\begin{figure}[H]
\begin{center}
\begin{tikzpicture}[rotate=25,scale=1.3]
\draw[color=black,scale=1,domain=-3.15: 3.15,
smooth,variable=\t,rotate=0]plot({1*sin(\t r)},
{1*cos(\t r)}); 
\draw[scale=0.5]
(-2,0) to [out=45, in=-160,looseness=1] (-0.85,0.25)
(-0.85,0.25) to [out=-40, in=150,looseness=1] (0.75,-0.35)
(-0.85,0.25)to [out=80, in=-90,looseness=1] (0,2)
(0.75,-0.35)to [out=30, in=-120,looseness=1](2,0)
(0.75,-0.35)to [out=-90, in=90,looseness=1](0,-2);
\draw[color=black,scale=1,domain=-3.15: 3.15,
smooth,variable=\t,rotate=0,shift={(2,0)}]plot({1*sin(\t r)},
{1*cos(\t r)}); 
\draw[scale=0.5]
(6,0) to [out=-135, in=20,looseness=1] (4.85,-0.25)
(4.85,-0.25) to [out=140, in=-30,looseness=1] (3.25,0.35)
(4.85,-0.25)to [out=-100, in=90,looseness=1] (4,-2)
(3.25,0.35)to [out=-150, in=90,looseness=1](2,0)
(3.25,0.35)to [out=90, in=-90,looseness=1](4,2);
\draw[color=black,scale=1,domain=-3.15: 3.15,
smooth,variable=\t,rotate=0,shift={(-2,0)}]plot({1*sin(\t r)},
{1*cos(\t r)}); 
\draw[scale=0.5]
(-2,0) to [out=-135, in=40,looseness=1] (-3.15,-0.25)
(-3.15,-0.25) to [out=140, in=-30,looseness=1] (-4.75,0.35)
(-3.15,-0.25)to [out=-100, in=90,looseness=1] (-4,-2)
(-4.75,0.35)to [out=-150, in=60,looseness=1](-6,0)
(-4.75,0.35)to [out=90, in=-90,looseness=1](-4,2);
\draw[color=black,scale=1,domain=-3.15: 3.15,
smooth,variable=\t,rotate=0,shift={(0,2)}]plot({1*sin(\t r)},
{1*cos(\t r)});
\draw[scale=0.5]
(2,4) to [out=-135, in=20,looseness=1] (0.85,3.75)
(0.85,3.75) to [out=140, in=-30,looseness=1] (-0.75,4.35)
(0.85,3.75)to [out=-100, in=90,looseness=1] (0,2)
(-0.75,4.35)to [out=-150, in=60,looseness=1](-2,4)
(-0.75,4.35)to [out=90, in=-90,looseness=1](0,6);
\draw[color=black,scale=1,domain=-3.15: 3.15,
smooth,variable=\t,rotate=0,shift={(0,-2)}]plot({1*sin(\t r)},
{1*cos(\t r)}); 
\draw[scale=0.5]
(2,-4) to [out=-135, in=20,looseness=1] (0.85,-4.25)
(0.85,-4.25) to [out=140, in=-30,looseness=1] (-0.75,-3.65)
(0.85,-4.25)to [out=-100, in=90,looseness=1] (0,-6)
(-0.75,-3.65)to [out=-150, in=60,looseness=1](-2,-4)
(-0.75,-3.65)to [out=90, in=-90,looseness=1](0,-2);
\path[font=\footnotesize,rotate=-25]
(1.35,-0.4) node[left]{$\SS_t$}
(-2.95,-0.4) node[left]{${\mathbb{H}}^1_t$}
(0.05,-2.77) node[left]{${\mathbb{H}}^2_t$}
(3.4,0.4) node[left]{${\mathbb{H}}^3_t$}
(0.1,3) node[left]{${\mathbb{H}}^4_t$}
(-0.88,-0.45) node[left]{$P^1$}
(0.43,-0.81) node[left]{$P^2$}
(.9,0.4) node[left]{$P^3$}
(-.06,1.13) node[left]{$P^4$}
(-.01,.07) node[left]{$O^1$}
(0.48,-0.15) node[left]{$O^2$};
\end{tikzpicture}
\end{center}
\begin{caption}{A network $\mathbb{S}_t$ with the associated networks ${\mathbb{H}}^i_t$.\label{reflectfig}}
\end{caption}
\end{figure}
\noindent We recall that as the curves composing the network are at least $C^2$ and the boundary 
points are fixed, at each $P^r$ both the velocity and the
curvature are zero, namely, the compatibility conditions of order $2$ (see 
Definition~\ref{geom-2-comp}) are satisfied.\\
For every end--point $P^i$, we define the ``symmetrized'' networks ${\mathbb{H}}^i_t$
each one obtained as the union of $\mathbb{S}_t$ 
with its ``reflection'' $\mathbb{S}^{R_i}_t$ with respect to $P^i$.
As the domain $\Omega$ is strictly convex and $\mathbb{S}_t$ is inside $\Omega$,
this operation clearly does not introduce self--intersections in the union 
${\mathbb{H}}^i_t=\mathbb{S}_t\cup\mathbb{S}^{R_i}_t$ and the
number of triple junctions of ${\mathbb{H}}^i_t$ is exactly twice the number of $\SS_t$.
Every network ${\mathbb{H}}^i_t$ is a regular network and the flow
is still in $C^{2,1}$, thanks to the compatibility conditions of order $2$ satisfied at $P^i$. The 
evolution is clearly symmetric with respect to $P^i$. If we have that the flow $\mathbb{S}_t$ is 
smooth then also all the flows ${\mathbb{H}}^i_t$ are smooth 
(see Definition~\ref{ncompcond}) and {\em viceversa}.

\subsection{Limits of rescaling procedures}\label{seclim}
Given a sequence $\mu_i\nearrow +\infty$ and a space--time point
$(x_0,t_0)$, where $0<t_0\leqslant T$, with $T$ the maximal time of smooth
existence, we consider as before in Section~\ref{pararesc}, the
sequence of parabolically rescaled curvature flows $F^{\mu_i}_\tt$ in the whole $\R^2$, 
that we denote with $\SS^{\mu_i}_\tt$. 

We know that, by rescaling the monotonicity formula (end of Section~\ref{pararesc}),
\begin{equation}\label{monresc}
\lim_{i\to\infty}\int\limits_{\tt}^{0}\int
\limits_{\SS^{\mu_i}_\ssss}
\Big|\underline{k}- \frac{x^\perp}{2\ssss}\Big|^2\rho_{0,0}(\cdot,\ssss)\,
ds\,d\ssss=0\,,
\end{equation}
for every $\tt\in(-\infty,0)$. We see now that this implies that
there exists a subsequence of parabolic rescalings 
which ``converges'' to a (possibly empty) degenerate, self--similarly shrinking network
flow.

\begin{defn}\label{ublr}
We say that a (possibly degenerate and with multiplicity) network
$\SS$ has {\em bounded length ratios} by the constant $C>0$, if 
\begin{equation}\label{blr}
\overline{\mathcal{H}}^1(\SS\cap B_R(\overline{x}))\leqslant C R\,, 
\end{equation}
for every $\overline{x}\in\R^2$ and $R>0$ ($\overline{\HH}^1$ is the one--dimensional 
Hausdorff measure counting multiplicities).
\end{defn}
Notice that this is a scaling invariant property, with the same constant $C$. The following technical lemma is due to Stone~\cite{stone1}.
\begin{lem}\label{rescestim2}
For any $\mu>0$, let $\SS^{\mu}_\tt$ be the parabolically rescaled flow around some $(x_0,t_0)\in\R^2\times(0,T)$, as defined in formula~\eqref{eq:parrescaling}.
\begin{enumerate}
\item There exists a constant $C=C(\SS_0)$ such that, for every $\overline{x}\in\R^2$, $t\in[0,T)$ and $R>0$ there holds
\begin{equation}\label{equ10biss}
{\mathcal H}^1(\SS_t\cap B_R(\overline{x}))\leqslant CR\,.
\end{equation}
That is, the family of networks $\SS_t$ has uniformly bounded length ratios by $C$.\\
It follows that for every $\overline{x}\in\R^2$, $\tt\in[-\mu^2t_0,0]$, $\mu>0$ and $R>0$, we have 
\begin{equation}\label{equ10bis}
{\mathcal H}^1(\SS^{\mu}_\tt\cap B_R(\overline{x}))\leqslant CR\,.
\end{equation}
\item For any $\varepsilon > 0$ there is a uniform radius $R= R(\varepsilon)$ such that
\begin{equation*}
\int_{\SS^{\mu}_\tt\setminus B_R(\overline{x})} e^{-|x|^2 /2}\,ds\leqslant \varepsilon\,,
\end{equation*}
that is, the family of measures $e^{-|x|^2 /2}\,{\mathcal H}^1\res\SS^{\mu}_\tt$ is {\em tight} (see~\cite{dellame}).
\end{enumerate}
\end{lem}
\begin{proof}
By Definition~\ref{Cinftyopen}, if $\SS_0$ is an open network, the number of unbounded curves ($C^1$--asymptotic to straight lines) is finite. Then, it is easy to see that, open or not, $\SS_0$ has bounded length ratios, that is, there exists a constant $C>0$ such that
\begin{equation}\label{ubr}
{\mathcal{H}}^1(\SS_0 \cap B_R(\overline{x})) \leqslant C'R\,,
\end{equation}
for all $\overline{x} \in\R^2$ and $R>0$. This implies that the {\em entropy} of $\SS_0$ (see~\cite{coldmin6,manmag}) is bounded, that is,
\begin{equation}\label{entropydef}
E(\SS_0)=\sup_{\overline{x} \in\R^2, \tau >0} \int_{\SS_0} \frac{e^{-\frac{|x-\overline{x}|^2}{4\tau}}}{\sqrt{4\pi \tau}} \,ds
=\sup_{\overline{x} \in\R^2, \tau >0} \int_{\SS_0} \rho_{\overline{x},\tau}(\cdot,0)\,ds\leqslant C''\,.
\end{equation}
Indeed, for any $\overline{x} \in\R^2$ and $\tau>0$, changing variable as $y=(x-\overline{x})/2\tau$, we have
\begin{align*}
\int_{\SS_0} \frac{e^{-\frac{|x-\overline{x}|^2}{4\tau}}}{\sqrt{4\pi \tau}} \,ds
=&\,\int_{\frac{\SS_0-\overline{x}}{2\tau}} \frac{e^{-\frac{|y|^2}{2}}}{\sqrt{2\pi}}\,ds\\
=&\,\sum_{n=0}^\infty \int_{\frac{\SS_0-\overline{x}}{2\tau}\cap(B_{n+1}(0)\setminus B_n(0))} \frac{e^{-\frac{|y|^2}{2}}}{\sqrt{2\pi}} \,ds\\
\leqslant &\,\frac{1}{\sqrt{2\pi}} \sum_{n=0}^\infty e^{-n^2/2}{\mathcal{H}}^1\Bigl(\frac{\SS_0-\overline{x}}{2\tau}\cap B_{n+1}(0)\Bigr)\\
=&\,\frac{1}{\sqrt{2\pi}}\sum_{n=0}^\infty e^{-n^2/2}{\mathcal{H}}^1\Bigl(\frac{1}{2\tau}\bigl(\SS_0\cap B_{2\tau(n+1)}(\overline{x})-\overline{x}\bigr)\Bigr)\\
=&\,\frac{1}{\sqrt{2\pi}}\sum_{n=0}^\infty e^{-n^2/2}{\mathcal{H}}^1\bigl(\SS_0\cap B_{2\tau(n+1)}(\overline{x})\bigr)\frac{1}{2\tau}\\
\leqslant &\,\frac{1}{\sqrt{2\pi}}\sum_{n=0}^\infty e^{-n^2/2}(n+1)C'\\
=&\,C'
\end{align*}
since the series converges (in the last inequality we applied estimate~\eqref{ubr}).\\
Then, by the monotonicity formula~\eqref{eqmonfor}, for any $\overline{x}\in\R^2$, $t\in[0,T)$ and $R>0$, by setting $\tau=t+R^2$, we have 
$$
\int_{\SS_t} \frac{e^{-\frac{|x-\overline{x}|^2}{4R^2}}}{\sqrt{4\pi}R}\,ds=\int_{\SS_t} \rho_{\overline{x},t+R^2}(\cdot,t)\,ds\leqslant \int_{\SS_0} \rho_{\overline{x},t+R^2}(\cdot,0)\,ds\leqslant C''\,,
$$
hence, 
$$
{\mathcal{H}}^1(\SS_t \cap B_R(\overline{x})) \leqslant \sqrt{4\pi}eR \int_{\SS_t\cap B_R(\overline{x})} \frac{e^{-\frac{|x-\overline{x}|^2}{4R^2}}}{\sqrt{4\pi}R}\,ds\leqslant \sqrt{4\pi}C''eR\,.
$$
Since this conclusion is scaling invariant, it also holds for all the rescaled networks $\SS^{\mu_i}_\tt$ and the first point of the lemma follows with $C=\sqrt{4\pi}C''e$. The second point is a consequence of the first one, indeed, we have
\begin{align*}
\int_{\SS_\tt^{\mu_i}\setminus B_R(\overline{x})} e^{-\frac{|x|^2}{2}}\,ds
=&\,\sum_{n=1}^\infty \int_{\SS_\tt^{\mu_i}\cap(B_{(n+1)R}(\overline{x})\setminus B_{nR}(\overline{x}))} e^{-\frac{|x|^2}{2}}\,ds\\
\leqslant &\,\sum_{n=1}^\infty e^{-n^2R^2/2}{\mathcal{H}}^1\bigl(\SS_\tt^{\mu_i}\cap B_{(n+1)R}(\overline{x})\bigr)\\
\leqslant &\,C\sum_{n=1}^\infty e^{-n^2R^2/2}(n+1)R\\
=&\,f(R)
\end{align*}
and the function $f$ satisfies $\lim_{R\to+\infty}f(R)=0$.
\end{proof}

\begin{prop}\label{thm:shrinkingnetworks.1} Given a sequence of parabolically rescaled curvature flows $\SS^{\mu_i}_\tt$, as above, there exists a subsequence $\mu_{i_j}$ and a (possibly empty) degenerate regular 
 self--similarly shrinking network flow $\SS^\infty_\tt$ such that for
 almost all $\tt\in (-\infty, 0)$ and for any $\alpha \in (0,1/2)$, 
$$
\SS^{\mu_{i_j}}_\tt\to\SS^\infty_\tt
$$
in $C^{1,\alpha}\loc \cap W^{2,2}\loc$. This convergence
also holds in the sense of Radon measures for all $\tt\in (-\infty, 0)$.\\
Moreover, for every continuous function with compact support in space--time $\varphi:\R^2\times(-\infty,0)\to\R$ there holds
\begin{equation}\label{gggg1bis}
\lim_{j\to\infty}\int_{(-\infty,0)}\int_{\SS^{\mu_{i_j}}_\tt}
\varphi(\cdot,\tt)\,ds\,d\ssss=
\int_{(-\infty,0)}\int_{\SS^\infty_\tt}
\varphi(\cdot,\tt)\,d\overline{s}\,d\ssss\,,
\end{equation}
where $d\overline{s}$ denotes the integration with respect to the
canonical measure on $\SS^\infty_\tt$, counting multiplicities and
\begin{equation}\label{gggg2bis}
\lim_{j\to\infty}\int_{\SS^{\mu_{i_j}}_\tt}
\rho_{0,0}(\cdot,\tt)\,ds=
\int_{\SS^\infty_\tt}
\rho_{0,0}(\cdot,\tt)\,d\overline{s}=\Theta_{\SS^\infty_{-1/2}}=\widehat{\Theta}(x_0,t_0)\,,
\end{equation}
for every $\tt\in(-\infty,0)$.
\end{prop}
\begin{proof}
We follow ideas in Ilmanen~\cite[Lemma~8]{ilman3} and~\cite[Section~7.1]{ilman1}.\\
By the first point of Lemma~\ref{rescestim2}, for every ball 
$B_R$ centered at the origin of $\R^2$, we have the uniform bound 
$\HH^1(\SS^{\mu_i}_\tt\cap B_R)\leqslant CR$, for some
constant $C$ independent of $i\in\NN$ and $\tt\in(-\infty,0)$. Hence, we can assume that the sequence of Radon measures defined by the left side of equation~\eqref{gggg1bis} are locally equibounded and converges to some limit measure in the space--time ambient $\R^2\times(-\infty,0)$\\
Considering the functions
$$
f_i(\tt)=\int\limits_{\SS^{\mu_{i}}_\tt}
\Big|\underline{k}- \frac{x^\perp}{2\tt}\Big|^2\rho_{0,0}(\cdot,\tt)\,
ds\,,
$$
the limit~\eqref{monresc} implies that $f_i\to 0$ in
$L^1\loc(-\infty, 0)$. Thus, there exists a (not relabeled) subsequence 
such that the sequence of functions $f_i$ converges pointwise almost everywhere to zero. We call $A\subseteq(-\infty,0)$ such a convergence set.\\
Then, for any $\tt\in A$, because of the uniform bound 
$\HH^1(\SS^{\mu_{i}}_\tt\cap B_R)\leqslant CR$, we have that for any $R>0$
$$
\int\limits_{\SS^{\mu_{i}}_\tt\cap B_R} k^2\,ds \leqslant C_R(\tt)\,, 
$$
for a constant $C_R(\tt)$ independent of $i$. 
Hence, if $\tt\in A$, reparametrizing the curves of the rescaled networks by arclength, we obtain curves in $W^{2,2}\loc$ with uniformly
bounded first derivatives, which implies that any subsequence of the networks $\SS^{\mu_{i}}_\tt$ admits a further subsequence converging weakly in $W^{2,2}\loc$, hence in $C^{1,\alpha}\loc$ to a degenerate regular network $\SS^\infty_\tt$. Moreover, such subsequence $\SS^{\mu_{i_j}}_\tt$ actually converges strongly in $W^{2,2}\loc$ by the weak convergence in $W^{2,2}\loc$ and the fact that $f_i(\tt) \to 0$ in $L^1\loc$. Finally, by the convergence in $C^{1,\alpha}\loc$, the associated Radon measures $\lambda_\tt^{i_j}=\HH^1\res\SS^{\mu_{i_j}}_\tt$ weakly converge to $\lambda_\tt^\infty= \overline{\HH}^1\res\SS^\infty_\tt$ (where $\overline{\HH}^1\res\SS^\infty_\tt$ is the one--dimensional Hausdorff measure restricted to $\SS^\infty_\tt$, counting multiplicities).\\
Since the integral functional 
$$
\SS\mapsto\int\limits_{\SS}
\Big|\underline{k}- \frac{x^\perp}{2\tt}\Big|^2\rho_{0,0}(\cdot,\tt)\,
ds\,.
$$
is lower semicontinuous with respect to this convergence
 (see~\cite{simon}, for instance), the limit $\SS^\infty_\tt$ satisfies
$$
\underline{k}-\frac{x^\perp}{2\tt}=0\,,
$$
in $W^{2,2}\loc$, hence, by a bootstrap argument, each non--degenerate curve of $\SS^\infty_\tt$ is actually smooth. Thus, for every $\tt\in A$ the network $\SS^\infty_\tt$ is a 
degenerate regular shrinker, up to a dilation factor.

By a standard diagonal argument we can assume that for $\tt$ in a dense countable subset $B_1\subseteq A$ the subsequence $\SS^{\mu_{i_j}}_\tt$ converges in $W^{2,2}\loc$ and $C^{1,\alpha}\loc$ to a limit degenerate regular shrinker $\SS^\infty_\tt$, with associated Radon measure $\lambda_\tt^\infty= \overline{\HH}^1\res\SS^\infty_\tt$, as above.\\
When $\tt\in A\setminus B_1$ we consider as $\SS^\infty_\tt$ the limit degenerate regular shrinker of an arbitrary converging subsequence of the networks $\SS^{\mu_{i_j}}_\tt$ and $\lambda_\tt^\infty= \overline{\HH}^1\res\SS^\infty_\tt$.\\
When $\tt\in(-\infty,0)\setminus A$ we instead consider as $\lambda_\tt^\infty$ the limit Radon measure of an arbitrary weakly--converging subsequence of the Radon measures $\lambda_\tt^{i_j}=\HH^1\res\SS^{\mu_{i_j}}_\tt$.\\
In this way we defined the limit network $\SS^\infty_\tt$ for every $\tt\in A$ and the limit Radon measures $\lambda_\tt^\infty$ for every $\tt\in(-\infty,0)$.

If ${\mathcal{F}}$ is a countable dense family of smooth functions in the cone of non negative functions in $C^0_c(\R^2)$, by the above convergence and the rescaled monotonicity formula, it follows that for every $\varphi\in{\mathcal {F}}$, there holds (by Proposition~\ref{equality1001} and formula~\eqref{brakkeqqq})
\begin{align*}
\frac{d\,}{d\tt}\int_{\SS^{\mu_{i_j}}_\tt}\varphi\,ds=&\,
-\int_{\SS^{\mu_{i_j}}_\tt}\varphi k^2\,ds
+\int_{\SS^{\mu_{i_j}}_\tt}\langle\nabla\varphi\,\vert\,\underline{k}\rangle\,ds\\
=&\,
-\int_{\SS^{\mu_{i_j}}_\tt}\varphi \left\vert\, \underline{k}-\frac{\nabla\varphi}{2\varphi}\,\right\vert^2\,ds
+\int_{\SS^{\mu_{i_j}}_\tt}\frac{\vert \nabla\varphi\vert^2}{4\varphi}\,ds\\
\leqslant&\,\frac{1}{4}\int_{\SS^{\mu_{i_j}}_\tt}\frac{\vert \nabla\varphi\vert^2}{\varphi}\,ds\\
\leqslant&\, (\max\,\vert\nabla^2\varphi\vert/2)\,\lambda_\tt^{i_j}(\{\varphi>0\})\\
\leqslant&\, C(\varphi,\nabla^2\varphi)\,,
\end{align*}
where we used the estimate $\vert \nabla\varphi\vert^2/\varphi\leqslant 2\max\, \vert\nabla^2\varphi\vert$, holding for every $\varphi\in C^2_c(\R^n)$ (where $\varphi>0)$, proved in~\cite[Lemma~6.6]{ilman1} and the uniform bound $\HH^1(\SS^{\mu_i}_\tt\cap B_R)\leqslant CR$, for some
constant $C$ independent of $i\in\NN$ and $\tt\in(-\infty,0)$.\\ 
Hence, fixing a single $\tt_0\in (-\infty,0)\setminus B_1$, the function
$$
\int_{\SS^{\mu_{i_j}}_\tt}\varphi\,ds- C(\varphi,\nabla^2\varphi)\tt
$$
is monotone non increasing once restricted to $B_1\cup\{\tt_0\}$. Passing to the limit (on the $\tt_0$--special subsequence such that $\lambda^{i_j}_{\tt_0}$ converges to $\lambda^\infty_{\tt_0}$) the same holds for the function 
$$
\tt\mapsto\int_{\R^2}\varphi\,d\lambda^\infty_\tt -C(\varphi,\nabla^2\varphi)\tt\,,
$$
restricted to $B_1\cup\{\tt_0\}$. By the arbitrariness of $\tt_0\in (-\infty,0)\setminus B_1$, we then conclude that such function is monotone non increasing on the whole $(-\infty,0)$. Thus, for every $\varphi\in{\mathcal{F}}$ the function
$\tt\mapsto\int_{\R^2}\varphi\,d\lambda_\tt^\infty$ has an at most countable set of (jump) discontinuities, that we call $B_\varphi$. Hence, we have that outside a countable subset $B=\bigcup_{\varphi\in{\mathcal{F}}} B_\varphi$ of $(-\infty,0)$, all the functions 
$$
\tt\mapsto\int_{\R^2}\varphi\,d\lambda^\infty_\tt
$$
are continuous, for every $\varphi\in{\mathcal{F}}$. This clearly implies that if $\tt\in(-\infty,0)\setminus B$, then the value of the integral $\int_{\R^2}\varphi\,d\lambda^\infty_\tt$ is uniquely determined and independent of the $\tt$--subsequence chosen to define $\lambda^\infty_\tt$, for every $\varphi\in{\mathcal{F}}$. An immediate consequence is that (by the density of ${\mathcal{F}}$),
\begin{itemize}
\item if $\tt\in (-\infty,0)\setminus B$, the Radon measure $\lambda_\tt^\infty$ is uniquely determined and the full sequence $\lambda^{i_j}_\tt$ converges to $\lambda^\infty_\tt$, 
\item if $\tt\in A$, the network $\SS^\infty_\tt$ is uniquely determined and the full sequence $\SS^{\mu_{i_j}}_\tt$ converges to $\SS^\infty_\tt$ in $W^{2,2}\loc$ and $C^{1,\alpha}\loc$,
\end{itemize}
as $j\to\infty$.\\
Then, we can conclude by a diagonal argument on the sequences of networks $\SS^{\mu_{i_j}}_\tt$ when $\tt\in B$, that we have a subsequence (not relabeled) of $\mu_{i_j}$ such that for every $\tt\in A$ the networks $\SS^{\mu_{i_j}}_\tt$ converge in $W^{2,2}\loc$ and $C^{1,\alpha}\loc$ and as Radon measures to $\SS^\infty_\tt$, as $j\to\infty$ and for every $\tt\in(-\infty,0)$ we have $\lambda^{i_j}_\tt\to\lambda^\infty_\tt$ as Radon measures.

By Proposition~\ref{equality1001}, every rescaled flow is a regular Brakke flow with equality, hence, the integrated version of equation~\eqref{brakkeqqq} holds, that is, 
\begin{equation*}
\int_{\R^2}\varphi(\cdot,\tt_1)\,d\lambda_{\tt_1}^{i_j}
-\int_{\R^2}\varphi(\cdot,\tt_2)\,d\lambda_{\tt_2}^{i_j}
=\int_{\tt_2}^{\tt_1}\biggl[
-\int_{\SS_\tt^{\mu_{i_j}}}\varphi(\gamma,t) k^2\,ds
+\int_{\SS_\tt^{\mu_{i_j}}}\langle\nabla\varphi(\gamma,t)\,\vert\,\underline{k}\rangle\,ds
+ \int_{\SS_\tt^{\mu_{i_j}}}\varphi_t(\gamma,t)\,ds\biggr]\,d\tt\,,
\end{equation*}
for every smooth function with compact support
$\varphi:\R^2\times(-\infty,0)\to\R$ and $\tt_1,\tt_2\in(-\infty,0)$.\\
By the $W^{2,2}\loc$--convergence almost everywhere (for $\tt$ in the set $A$) and the limit~\eqref{monresc} (which allows us to use the dominated convergence theorem) we can pass to the limit to get
\begin{equation*}
\int_{\R^2}\varphi(\cdot,\tt_1)\,d\lambda_{\tt_1}^\infty
-\int_{\R^2}\varphi(\cdot,\tt_2)\,d\lambda_{\tt_2}^\infty
=\int_{\tt_2}^{\tt_1}\biggl[
-\int_{\SS_\tt^\infty}\varphi(\gamma,t) k^2\,d\overline{s}
+\int_{\SS_\tt^\infty}\langle\nabla\varphi(\gamma,t)\,\vert\,\underline{k}\rangle\,d\overline{s}
+ \int_{\SS_\tt^\infty}\varphi_t(\gamma,t)\,d\overline{s}\biggr]\,d\tt\,,
\end{equation*}
where $d\overline{s}$ denotes the integration with respect to the canonical measure on $\SS^\infty_\tt$, counting multiplicities.\\
This shows that the function $\tt\mapsto\int_{\R^2}\varphi(\cdot,\tt)\,d\lambda^\infty_\tt$ is absolutely continuous on $(-\infty,0)$ and for almost every $\tt\in(-\infty,0)$, there holds
\begin{equation}\label{eqrrr}
\frac{d\,}{d\tt}\int_{\R^2}\varphi(\cdot,\tt)\,d\lambda_{\tt}^\infty
=-\int_{\SS_\tt^\infty}\varphi(\gamma,t) k^2\,d\overline{s}
+\int_{\SS_\tt^\infty}\langle\nabla\varphi(\gamma,t)\,\vert\,\underline{k}\rangle\,d\overline{s}
+ \int_{\SS_\tt^\infty}\varphi_t(\gamma,t)\,d\overline{s}\,.
\end{equation}

We then consider, for every $\tt\in(-\infty,0)$, the Radon measures defined by 
$$
\nu_\tt(D)=\lambda^\infty_\tt(\sqrt{-2\tt}\,D)/\sqrt{-2\tt}\,.
$$
It is easy to see that showing that $\lambda^\infty_\tt=\overline{\HH}^1\res (\sqrt{-2\tt}\,\SS^\infty_{-1/2})$ for every $\tt\in(-\infty,0)$, is equivalent to prove that the measures $\nu_\tt$ are all the same and this means that $\SS^\infty_\tt$ is a degenerate regular self--similarly shrinking network flow.

We have, for every smooth function with compact support
$\psi:\R^2\to\R$,
\begin{equation*}
\int_{\R^2}\psi(x)\,d\nu_\tt(x)=
\frac{1}{\sqrt{-2\tt}}\int_{\R^2}\psi\Bigl(\frac{x}{\sqrt{-2\tt}}\Bigr)\,d\lambda^\infty_\tt(x)\,,
\end{equation*}
hence, choosing $\varphi(x,\tt)=\psi\bigl(\frac{x}{\sqrt{-2\tt}}\bigr)$, at every time $\tt$ such that equality~\eqref{eqrrr} holds (almost every $\tt\in(-\infty,0)$), we have
\begin{align*}
\frac{d\,}{d\tt}\int_{\R^2}\psi(x)\,d\nu_\tt(x)
=&\,
\frac{1}{-2\tt\sqrt{-2\tt}}\int_{\SS_\tt^\infty}\psi\Bigl(\frac{\gamma}{\sqrt{-2\tt}}\Bigr)\,d\overline{s}
-\frac{1}{\sqrt{-2\tt}}\int_{\SS_\tt^\infty}\psi\Bigl(\frac{\gamma}{\sqrt{-2\tt}}\Bigr)k^2\,d\overline{s}\\
&\,+\frac{1}{-2\tt}\int_{\SS_\tt^\infty}\Bigl\langle\nabla\psi\Bigl(\frac{\gamma}{\sqrt{-2\tt}}\Bigr)\,\Bigr\vert\,\Bigl.\underline{k}\Bigr\rangle\,d\overline{s}
+ \int_{\SS_\tt^\infty}\Bigl\langle\nabla\psi\Bigl(\frac{\gamma}{\sqrt{-2\tt}}\Bigr)\,\Bigr\vert\Bigl.\,\frac{\gamma}{{4\tt^2}}\Bigr\rangle\,d\overline{s}\,.
\end{align*}
Substituting $\underline{k}={\gamma^\perp}/{2\tt}$, we obtain
\begin{align*}
\frac{d\,}{d\tt}\int_{\R^2}\psi(x)\,d\nu_\tt(x)
=&\,
\frac{1}{-2\tt\sqrt{-2\tt}}\int_{\SS_\tt^\infty}\psi\Bigl(\frac{\gamma}{\sqrt{-2\tt}}\Bigr)\,d\overline{s}
-\frac{1}{\sqrt{-2\tt}}\int_{\SS_\tt^\infty}\psi\Bigl(\frac{\gamma}{\sqrt{-2\tt}}\Bigr)\frac{\langle \,\underline{k}\,\vert\,\gamma^\perp\rangle}{2\tt}\,d\overline{s}\\
&\,-\int_{\SS_\tt^\infty}\Bigl\langle\nabla\psi\Bigl(\frac{\gamma}{\sqrt{-2\tt}}\Bigr)\,\Bigr\vert\,\Bigl. \frac{\gamma^\perp}{4\tt^2}\Bigr\rangle\,d\overline{s}
+ \int_{\SS_\tt^\infty}\Bigl\langle\nabla\psi\Bigl(\frac{\gamma}{\sqrt{-2\tt}}\Bigr)\,\Bigr\vert\Bigl.\,\frac{\gamma}{{4\tt^2}}\Bigr\rangle\,d\overline{s}\\
=&\,
\frac{1}{-2\tt\sqrt{-2\tt}}\int_{\SS_\tt^\infty}\psi\Bigl(\frac{\gamma}{\sqrt{-2\tt}}\Bigr)\,d\overline{s}
-\frac{1}{\sqrt{-2\tt}}\int_{\SS_\tt^\infty}\psi\Bigl(\frac{\gamma}{\sqrt{-2\tt}}\Bigr)\frac{\langle\,\underline{k}\,\vert\,\gamma^\perp\rangle}{2\tt}\,d\overline{s}\\
&\,+ \int_{\SS_\tt^\infty}\Bigl\langle\nabla\psi\Bigl(\frac{\gamma}{\sqrt{-2\tt}}\Bigr)\,\Bigr\vert\Bigl.\,\frac{\gamma^\top}{{4\tt^2}}\Bigr\rangle\,d\overline{s}\\
=&\,
\frac{1}{-2\tt\sqrt{-2\tt}}
\int_{\SS_\tt^\infty}\Bigl[\psi\Bigl(\frac{\gamma}{\sqrt{-2\tt}}\Bigr)
+\psi\Bigl(\frac{\gamma}{\sqrt{-2\tt}}\Bigr)\langle \,\underline{k}\,\vert\,\gamma\rangle
+\Bigl\langle\nabla\psi\Bigl(\frac{\gamma}{\sqrt{-2\tt}}\Bigr)\,\Bigr\vert\Bigl.\,\frac{\tau}{\sqrt{-2\tt}}\Bigr\rangle\,\langle \tau\,\vert\,\gamma\rangle\Bigr]\,d\overline{s}\,,
\end{align*}
where we denoted with $\gamma^\top$ the tangential component of the vector $\gamma\in\R^2$, that is, $\gamma^\top=\langle\tau\,\vert\,\gamma\rangle\tau$. Noticing now that
\begin{align*}
\partial_s\Bigl[\psi\Bigl(\frac{\gamma}{\sqrt{-2\tt}}\Bigr)\langle \tau\,\vert\,\gamma\rangle\Bigr]
=&\,
\Bigl\langle\nabla\psi\Bigl(\frac{\gamma}{\sqrt{-2\tt}}\Bigr)\,\Bigr\vert\Bigl.\,\frac{\tau}{\sqrt{-2\tt}}\Bigr\rangle\,\langle \tau\,\vert\,\gamma\rangle
+\psi\Bigl(\frac{\gamma}{\sqrt{-2\tt}}\Bigr)\langle\, \underline{k}\,\vert\,\gamma\rangle+
\psi\Bigl(\frac{\gamma}{\sqrt{-2\tt}}\Bigr)\langle \tau\,\vert\,\tau\rangle\\
=&\,
\Bigl\langle\nabla\psi\Bigl(\frac{\gamma}{\sqrt{-2\tt}}\Bigr)\,\Bigr\vert\Bigl.\,\frac{\tau}{\sqrt{-2\tt}}\Bigr\rangle\,\langle \tau\,\vert\,\gamma\rangle+
\psi\Bigl(\frac{\gamma}{\sqrt{-2\tt}}\Bigr)\langle\, \underline{k}\,\vert\,\gamma\rangle+\psi\Bigl(\frac{\gamma}{\sqrt{-2\tt}}\Bigr)\,,
\end{align*}
we conclude
\begin{equation*}
\frac{d\,}{d\tt}\int_{\R^2}\psi(x)\,d\nu_\tt(x)=
\frac{1}{-2\tt\sqrt{-2\tt}}
\int_{\SS_\tt^\infty}\partial_s\Bigl[\psi\Bigl(\frac{\gamma}{\sqrt{-2\tt}}\Bigr)\langle \tau\,\vert\,\gamma\rangle\Bigr]\,d\overline{s}
\end{equation*}
and this last integral is zero by Lemma~\ref{lemreg} and the last point of Remark~\ref{remreg}.\\
Since for every map $\varphi:\R^2\to\R$ the function $\tt\mapsto\int_{\R^2}\varphi(x)\,d\nu_\tt(x)$ is absolutely continuous on $(-\infty,0)$ with zero derivative almost everywhere, it is constant and we are done.

Equation~\eqref{gggg1bis} clearly follows by the convergence assumption on the sequence of Radon measures in $\R^2\times(-\infty,0)$ and this conclusion.

Finally, for every $\tt\in(-\infty,0)$, by the second point of Lemma~\ref{rescestim2}, we can pass to the limit in the Gaussian integral and we get
$$
\lim_{j\to\infty}\int_{\SS^{\mu_{i_j}}_\tt}
\rho_{0,0}(\cdot,\tt)\,ds=
\int_{\SS^\infty_\tt}
\rho_{0,0}(\cdot,\tt)\,d\overline{s}=\Theta_{\SS^\infty_{-1/2}}\,,
$$
since the right integral is constant in $\tt$, being $\SS^\infty_\tt$ a self--similarly shrinking flow.\\
Recalling that (see Section~\ref{pararesc})
$$
\int_{\SS^{\mu_{i_j}}_\tt}\rho_{0,0}(\cdot,\tt)\,ds=\Theta_{x_0,t_0}(t_0+\mu_{i_j}^{-2}\tt)\to\widehat{\Theta}(x_0,t_0)\,,
$$
as $j\to\infty$, equality~\eqref{gggg2bis} follows.
\end{proof}

\begin{rem} We underline that even if the limit flow is composed of homothetic rescalings of a single degenerate regular network, we cannot conclude that the convergence of $\SS^{\mu_{i_j}}_\tt$ to $\SS^\infty_\tt$ is in $W^{2,2}\loc$ and $C^{1,\alpha}\loc$ for {\em every} $\tt\in(-\infty,0)$ but only for almost every $\tt\in(-\infty,0)$. For the ``other'' times the convergence could be only as Radon measures.
\end{rem}

We deal now with the possible blow--up limits arising from Huisken's
dynamical procedure. We recall that
$$
\widetilde{\rho}(x)=e^{-\frac{\vert x\vert^2}{2}}\,.
$$

The following technical lemma is the exact analogue of Lemma~\ref{rescestim2} for Huisken's rescaling procedure. It follows in the same way by the first point of such lemma.
\begin{lem}\label{resclimit2H}
Let $\widetilde{\SS}_{x_0,\tt}$ be the family of rescaled networks, obtained via Huisken's dynamical procedure around some $x_0\in\R^2$, as defined in formula~\eqref{huiskeqdef}.
\begin{enumerate}
\item There exists a constant $C=C(\SS_0)$ such that, for every $\overline{x},x_0\in\R^2$, $\tt\in\bigl[-\frac{1}{2}\log T,+\infty\bigr)$ and $R>0$ there holds
\begin{equation}\label{equ10bissH}
{\mathcal H}^1(\widetilde{\SS}_{x_0,\tt}\cap B_R(\overline{x}))\leqslant CR\,.
\end{equation}
\item For any $\varepsilon > 0$ there is a uniform radius $R= R(\varepsilon)$ such that
\begin{equation*}
\int_{\widetilde{\SS}_{x_0,\tt}\setminus B_R(\overline{x})} e^{-|x|^2 /2}\,ds\leqslant \varepsilon\,,
\end{equation*}
that is, the family of measures $e^{-|x|^2 /2}\,{\mathcal H}^1\res\widetilde{\SS}_{x_0,\tt}$ is {\em tight} (see~\cite{dellame}).
\end{enumerate}
\end{lem}

\begin{prop}\label{resclimit-general}
Let $\SS_t=\bigcup_{i=1}^n\gamma^i([0,1],t)$ be a $C^{2,1}$ curvature flow 
of regular networks in the time interval $[0,T]$. Then
for every $x_0\in\R^2$ and for every subset $\mathcal I$ of $[-1/2\log
T,+\infty)$ with infinite Lebesgue measure
there exists a sequence of rescaled times
$\tt_j\to+\infty$, with $\tt_j\in{\mathcal I}$, such that the sequence
of rescaled networks $\widetilde{\SS}_{x_0,\tt_{j}}$ (obtained via Huisken's dynamical procedure) 
converges in $C^{1,\alpha}\loc\cap W^{2,2}\loc$, for any $\alpha \in (0,1/2)$,
 to a (possibly empty) limit network which is a degenerate regular shrinker $\widetilde{\SS}_\infty$ (possibly with multiplicity).
Moreover, we have
\begin{equation}\label{gggg2}
\lim_{j\to\infty}\frac{1}{\sqrt{2\pi}}\int_{\widetilde{\SS}_{x_0,\tt_j}}
\widetilde{\rho}\,d\sigma=\frac{1}{\sqrt{2\pi}}
\int_{\widetilde\SS_\infty}\widetilde{\rho}\,d\overline{\sigma}=\Theta_{\widetilde{\SS}_\infty}=\widehat{\Theta}(x_0)\,.
\end{equation}
where $d\overline{\sigma}$ denotes the integration with respect to the
canonical measure on
$\widetilde{\SS}_\infty$, counting multiplicities.
\end{prop}
\begin{proof}
Letting $\tt_1=-1/2\log T$ and $\tt_2\to +\infty$ in 
the rescaled monotonicity formula~\eqref{reseqmonfor-int} by
Lemma~\ref{rescstimadib} we get
$$
\int\limits_{-1/2\log{T}}^{+\infty}\int\limits_{\widetilde{\SS}_{x_0,\tt}}
\vert\,\widetilde{\underline{k}}+x^\perp\vert^2\widetilde{\rho}\,d\sigma\,d\tt<+\infty\,,
$$
which implies
$$
\int\limits_{{\mathcal{I}}}\int\limits_{\widetilde{\SS}_{x_0,\tt}}
\vert\,\widetilde{\underline{k}}+x^\perp\vert^2\widetilde{\rho}\,d\sigma\,d{\tt}<+\infty\,.
$$
Being the last integral finite and being the integrand a non negative
function on a set of infinite Lebesgue measure, we can extract within ${\mathcal I}$ a
sequence of times $\tt_{j}\to+\infty$, such that 
\begin{equation}\label{eqW22}
\lim_{j\to +\infty}\int\limits_{\widetilde{\SS}_{x_0,\tt_j}}
\vert\,\widetilde{\underline{k}}+x^\perp\vert^2\widetilde{\rho}\,d\sigma
=0\,.
\end{equation} 
It follows that for every ball $B_R$ of radius $R$ in $\R^2$, the networks 
$\widetilde{\SS}_{x_0,\tt_j}$ have
curvature uniformly bounded in $L^2(B_R)$. 
Moreover by the first point of Lemma~\ref{resclimit2H} for every ball 
$B_R$ centered at the origin of $\R^2$ we have the uniform bound 
$\HH^1({\widetilde{\SS}_{x_0,\tt_j}}\cap B_R)\leqslant C R$, for some
constant $C$ independent of $j\in\NN$. Then reparametrizing the
rescaled networks in arclength, we obtain curves with uniformly
bounded first derivatives and with second derivatives in $L^2\loc$. \\
By a standard compactness argument (see~\cite{huisk3,langer2})
the sequence ${\widetilde{\SS}_{x_0,\tt_{j}}}$ of reparametrized
networks admits a subsequence ${\widetilde{\SS}_{x_0,\tt_{j_l}}}$
which converges, weakly in $W^{2,2}\loc$ and strongly in $C^{1,\alpha}\loc$, 
to a (possibly empty) limit regular degenerate $C^1$ network
$\widetilde{\SS}_\infty$ (possibly with multiplicity). \\
Since the integral functional 
$$
\widetilde{\SS}\mapsto
\int\limits_{\widetilde{\SS}}\vert\,\widetilde{\underline{k}}+x^\perp\vert^2\widetilde{\rho}\,d\sigma
$$
is lower semicontinuous with respect to this convergence
 (see~\cite{simon} for instance), the limit ${\widetilde\SS}_\infty$ satisfies 
$\widetilde{\underline{k}}_\infty+x^\perp=0$ in the sense of distributions. \\
A priori the limit network is composed of curves in
$W^{2,2}\loc$ but from the relation
$\widetilde{\underline{k}}_\infty+x^\perp=0$ it follows that 
the curvature $\widetilde{\underline{k}}_\infty$ is continuous. 
By a bootstrap argument, it is then easy to see that $\widetilde\SS_\infty$ 
is actually composed of $C^\infty$ curves.

By means of the second point of Lemma~\ref{resclimit2H} 
we can pass to the limit in the Gaussian integral and we get
$$
\lim_{j\to\infty}\frac{1}{\sqrt{2\pi}}\int_{\widetilde{\SS}_{x_0,
\tt_j}}\widetilde{\rho}\,d\sigma=\frac{1}{\sqrt{2\pi}}
\int_{\widetilde{\SS}_\infty}\widetilde{\rho}\,d\overline{\sigma}=\Theta_{\widetilde{\SS}_\infty}\,.
$$
Recalling that
$$
\frac{1}{\sqrt{2\pi}}\int_{\widetilde{\SS}_{x_0,\tt_j}}\widetilde{\rho}\,d\sigma
=\int_{\SS_{t(\tt_j)}}\rho_{x_0}(\cdot,t(\tt_j))\,ds
=\Theta_{x_0}(t(\tt_j))\to\widehat{\Theta}(x_0)
$$
as $j\to\infty$, equality~\eqref{gggg2} follows.

The convergence in $W^{2,2}\loc$ is implied by the weak convergence
in $W^{2,2}\loc$ and equation~\eqref{eqW22}.
\end{proof}
 
\begin{rem}\label{T1}
A singularity in which the curvature is unbounded is called of {\em Type~I} if there exists a constant $C$ such that
\begin{equation}\label{typeI}
\max_{\SS_t} k^2\leqslant \frac{C}{T-t}
\end{equation}
for every $t\in[0,T)$. 
Otherwise, the singularity is called of {\em Type~II}.\\
If the singularity is of Type~I, then the proof of this proposition gets easier 
and we get a stronger convergence to the limit network. 
Indeed, thanks to the Type~I estimate~\eqref{typeI} one obtains a
uniform pointwise bound on the curvature (and consequently on its derivatives)
of the rescaled network (see~\cite[Section~6, Proposition~6.16]{mannovtor}, for instance).
Similarly, with the right choice of the sequence $\mu_{i_j}$, 
the same holds also for Proposition~\ref{thm:shrinkingnetworks.1}.
\end{rem}

\begin{rem}\label{equal}
Even if the two rescaling procedures are different 
(and actually one can use the more suitable for an argument) 
the family of blow--up limit shrinkers $\widetilde{\SS}_\infty$ arising from Huisken's one 
coincides with the family of shrinkers $\SS^\infty_{-1/2}$ where $\SS^\infty_\tt$ is any self--
similarly shrinking curvature flow coming from Proposition~\ref{thm:shrinkingnetworks.1}. 
This can be easily seen by Remark~\ref{relaz}, since if $\SS^{\mu_i}_{-1/2}\to\SS^\infty_{-1/2}$, then setting $\tt_i=\log{(\sqrt{2}\mu_i)}$ we have 
$\widetilde{\SS}_{x_0,\tt_i}\to\SS^\infty_{-1/2}$, as $i\to\infty$, hence $\SS^\infty_{-1/2}=\widetilde{\SS}_\infty$ for such sequence. Vice versa, if $\widetilde{\SS}_{x_0,\tt_i}\to\widetilde{\SS}_\infty$, setting $\mu_i=e^{\tt_i}/\sqrt{2}$, by means of Proposition~\ref{thm:shrinkingnetworks.1}, we have a converging (not relabeled) subsequence of rescaled curvature flows $\SS^{\mu_i}_\tt\to\SS^\infty_\tt$ such that $\SS^{\mu_i}_{-1/2}\to\widetilde{\SS}_\infty$, as $i\to\infty$, hence $\widetilde{\SS}_\infty=\SS^\infty_{-1/2}$. As a consequence, for every blow--up limit shrinker $\widetilde{\SS}_\infty$ and any self--similarly shrinking curvature flow $\SS^\infty_\tt$ there holds
$$
\Theta_{\widetilde{\SS}_\infty}=\Theta_{\SS^\infty_{-1/2}}=\widehat{\Theta}(x_0)\,,
$$
by formulas~\ref{gggg2bis} and~\ref{gggg2}.

Notice that in the first implication, for simplicity, we assumed the convergence at time 
$\tt=-1/2$ of the parabolically rescaled flows, which actually is not guaranteed by 
Proposition~\ref{thm:shrinkingnetworks.1}. To be precise one should argue by considering a 
time $\tt$, such that the sequence of networks $\SS^{\mu_i}_\tt$ converges to 
$\SS^\infty_\tt=\lambda\SS^\infty_{-1/2}$, for some factor $\lambda>0$.
\end{rem}

\begin{rem}\label{thetanet}
By means of Proposition~\ref{resclimit-general}, it is easy to see that, if $t_0<T$, hence the flow is smooth in $[0,t_0]$ and the curvature is bounded, we have $\widehat{\Theta}(x_0,t_0)=0$ if $x_0\not\in\SS_{t_0}$, since every blow--up limit is clearly empty and that $\widehat{\Theta}(x_0,t_0)=1$, if $x_0\in\SS_{t_0}$ and it is neither a $3$--point nor an end--point of $\SS_{t_0}$, as every blow--up limit must be a multiplicity--one line through the origin of $\R^2$ (see~\cite[Remark~3.2.15]{Manlib}). Then, by means of the ``reflection argument'' at the end of Section~\ref{geopropsub}, if $x_0$ is an end--point 
there holds $\widehat{\Theta}(x_0,t_0)=1/2$, being the Gaussian density of a halfline. Finally, if $x_0\in\SS_{t_0}$ is a triple junction, we see that $\widehat{\Theta}(x_0,t_0)=3/2$, indeed, if $O^i(t)$ is the $3$--point such that $O^i(t_0)=x_0$, since the curvature is bounded every blow--up limit shrinker must be non--degenerate, without end--points and have zero curvature, moreover, it is a tree locally around $x_0$ as no region collapses (the flow is smooth up to $t_0$). Being the modulus of the velocity $v^i(t)$ of $O^i(t)$ bounded by some constant $C$, for $t\in[0,t_0)$ we have
\begin{equation*}
\vert O^i(t)-x_0\vert=\vert O^i(t_0)-O^i(t)\vert=\biggl\vert\int_{t}^{t_0}v^i(\xi)\,d\xi\,\biggl\vert\leqslant
\int_{t}^{t_0}\vert v^i(\xi)\vert\,d\xi\leqslant C\vert t_0-t\vert\,,
\end{equation*}
which implies, after performing Huisken's rescaling procedure, that its image $\widetilde{O}^i(\tt)$ satisfy
$$
\vert\widetilde{O}^i(\tt)\vert=\frac{\vert O^i(t(\tt)-x_0\vert}{\sqrt{2(t_0-t(\tt))}}\leqslant\frac{C\vert t_0-t(\tt)\vert}{\sqrt{2(t_0-t(\tt))}}=C\sqrt{(t_0-t(\tt))/2}\,,
$$
which tends to zero, as $\tt\to+\infty$. In particular, the image of the $3$--point cannot ``disappear'' in the limit regular shrinker (for instance, going to infinity), then Lemma~\ref{lemmatree} tells us that the only possible blow--up limit shrinkers are standard triods $\TTT$ which have Gaussian density $\Theta_{\TTT}$ equal to $3/2$.
\end{rem}

The following lemma is helpful in strengthening the convergence in the previous proposition.

\begin{lem}\label{boh} 
Given a sequence of smooth curvature flows of networks $\SS_t^i$ in a time interval $(t_1,t_2)$ with uniformly bounded length ratios, if in a dense subset of times $t\in(t_1,t_2)$ the networks $\SS^i_t$ converge in a ball $B\subseteq\R^2$ in $C^1\loc$, as $i\to\infty$, to a multiplicity--one, embedded, $C^\infty$--curve $\gamma_t$ moving by curvature in $B'\supseteq\overline{B}$, for $t\in(t_1,t_2]$ (hence, the curvature of $\gamma_t$ is uniformly bounded), then for every $(x_0,t_0)\in B\times(t_1,t_2]$, the curvature of $\SS_t^i$ is uniformly bounded in a neighborhood of $(x_0,t_0)$ in space--time. It follows that, for every $(x_0,t_0)\in B\times(t_1,t_2]$, we have $\SS_t^i\to\gamma_t$ smoothly around $(x_0,t_0)$ in space--time (possibly, up to local reparametrizations of the networks $\SS_t^i$).
\end{lem}
\begin{proof} Being $\gamma_t$ a smooth flow of an embedded curve in $B$, we have $\widehat{\Theta}(x_0,t_0)=1$ (by Remark~\ref{thetanet}), hence, for $(x,t)$ in a suitably small neighborhood of $(x_0,t_0)\in B\times(t_1,t_2]$ we have that $\Theta_{x,t}(\tau)\leqslant 1+\varepsilon/2<3/2$, for every $\tau\in(\tau_0,t)$ and some $\tau_0>0$, where $\varepsilon>0$ is smaller than the ``universal'' constant given by White's local regularity theorem for mean curvature flow in~\cite{white1}. Then, in a possibly smaller space--time neighborhood of $(x_0,t_0)$, for a fixed time $\overline{\tau}\in(\tau_0,t)$ where the $C^1\loc$--convergence of the networks $\SS_{\overline{\tau}}^i\to\gamma_{\overline{\tau}}$ holds (such a subset of times is dense), for $i$ large enough, the Gaussian density functions of $\SS_{\overline{\tau}}^i$ satisfy $\Theta_{x,t}^i(\overline{\tau})<1+\varepsilon<3/2$
(the Gaussian density functions are clearly continuous under the $C^1\loc$ convergence with uniform length ratios estimate, by the exponential decay of backward heat kernel). Hence, by Proposition~\ref{promono}, Lemma~\ref{stimadib} and the subsequent discussion (possibly choosing a larger $\overline{\tau}$), this also holds for {\em every} $\tau\in(\overline{\tau},t)$. In other words, 
$\Theta_{x,t}^i(t-r^2)<1+\varepsilon<3/2$, for every $(x,t)$ in a space--time neighborhood of $(x_0,t_0)$, $0<r<r_0$ and $i>i_0$, for some $r_0>0$. By Remark~\ref{thetanet}, this ``forbids'' the presence of a $3$--point of $\SS_t^i$ in such space--time neighborhood, hence we are dealing simply with (classical) curvature flows of {\em curves}. Then, White's local regularity theorem gives a uniform, local (in space--time) estimate on the curvature of all $\SS_t^i$, which actually implies uniform bounds on all its higher derivatives (for instance, by Ecker and Huisken interior estimates in~\cite{eckhui2}), around $(x_0,t_0)$. Hence the statement of the lemma follows (see also~\cite[Theorem~7.3]{white1}).
\end{proof}
As a consequence, the convergence of $\SS^{\mu_{i_j}}_\tt$ to the limit degenerate regular self--similarly shrinking network flow $\SS^\infty_\tt$ in Proposition~\ref{thm:shrinkingnetworks.1} is smooth locally in space--time around every interior point of the multiplicity--one curves of the network $\SS^\infty_\tt$.\\
Moreover if $\SS^\infty_\tt$ is non--degenerate (no cores) and with only multiplicity--one curves, then actually $\SS^{\mu_{i_j}}_\tt\to\SS^\infty_\tt$ smoothly, locally in space--time (also around the $3$--points). This can be shown by following the argument of the proof of Lemma~8.6 in~\cite{Ilnevsch} (see anyway the proof in the special case of Lemma~\ref{thm:locreg.3}).\\
Analogously, also for Huisken's dynamical procedure it can be shown that the convergence of the rescaled networks
 $\widetilde{\SS}_{x_0,\tt_{j}}$ to $\widetilde{\SS}_\infty$ is
 locally smooth far from the cores and non multiplicity--one curves
 of $\widetilde{\SS}_\infty$.

Notice that the blow--up limit degenerate shrinker $\widetilde\SS_\infty$, 
obtained by Proposition~\ref{resclimit-general} a priori depends on the chosen
sequence of rescaled times $\tt_j\to+\infty$. If such a limit is a
multiplicity--one line (or a halfline, if $x_0$ is an end--point of the
network), we have $\widehat{\Theta}(x_0)=1$ 
($\widehat{\Theta}(x_0)=1/2$ in the case of a
halfline), then by White's result~\cite[Theorem~3.5]{white1}, locally around $x_0$ the curvature is
uniformly bounded in time and the flow is smooth up to time $T$ (using the ``reflection argument'' at the end of Section~\ref{geopropsub}, if $x_0$ is an end--point), hence, the limit is unique. In general, uniqueness 
of such a limit is actually unknown.

\begin{oprob}[Uniqueness of Blow--up Assumption -- {\bf{U}}]\label{ooo12}
The limit degenerate regular shrinker $\widetilde\SS_\infty$ is
independent of the chosen converging sequence of rescaled networks
$\widetilde{\SS}_{x_0,\tt_j}$ in Proposition~\ref{resclimit-general}. More precisely, the full 
family $\widetilde{\SS}_{x_0,\tt}$ converges in $C^1\loc$ to $\widetilde\SS_\infty$, as 
$\tt\to+\infty$.
\end{oprob}

In Section~\ref{behavsing} we will partially address this problem, concluding that it has a positive answer in the case of tree--like networks (see Remark~\ref{remUtree}). Moreover, some positive partial results were recently obtained in~\cite{PlPo22A-2}.

\begin{rem}\label{uequiv}
A similar (actually equivalent, in view of Remark~\ref{relaz}) problem 
can be stated for the limit degenerate regular self--similarly shrinking flow 
$\SS^\infty_\tt$ given by a converging subsequence $\SS^{\mu_{i_j}}_\tt$ of the family of the 
parabolically rescaled curvature flows $\SS^{\mu_i}_\tt$ in Proposition~\ref{thm:shrinkingnetworks.1}, 
about the independence of $\SS^\infty_\tt$ of the sequence $\mu_i$ and subsequence $\mu_{i_j}$. Namely, 
do we have the full convergence of the family of flows $\SS^{\mu}_\tt$ to $\SS^\infty_\tt$, as $\mu\to+\infty$?
\end{rem}

\begin{rem}\label{t1u} 
A regular shrinker is said to be {\em multiplicity--one} if it has no cores
and none of its curves has multiplicity higher than one.
In case the limit degenerate regular shrinker $\widetilde\SS_\infty$ 
is actually a multiplicity--one regular shrinker (or the same for the limit 
degenerate regular self--similarly shrinking flow $\SS^\infty_\tt$)
the above uniqueness assumption implies that the singularity is of Type~I 
(see the Remark~\ref{T1} above).
Indeed, by Lemma~\ref{boh} the convergence of the rescaled networks 
to $\widetilde\SS_\infty$ is smooth which implies that the curvature 
is locally uniformly bounded by $C/\sqrt{T-t}$.
\end{rem}

It is then natural in view of this remarks to state also the following open problems.

\begin{oprob}[Non--degeneracy of the blow--up]\label{ooo112}\ 
\begin{itemize}
\item Any blow--up limit shrinker $\widetilde\SS_\infty$ 
different from a standard cross (see Figure~\ref{crossfig} and Lemma~\ref{lemmatree}) is non--degenerate (the same for the limit self--similarly shrinking flow $\SS^\infty_\tt$)?
\item There can be curves with multiplicity larger than one?
\item If $\widetilde\SS_\infty$ is degenerate, there can be any cores outside the origin?
\end{itemize}
\end{oprob}

\begin{oprob}[Type~I Conjecture]\label{ooo113}
Every singularity is of Type~I (there exists a constant $C>0$ such that inequality~\eqref{typeI} 
is satisfied, for every $t\in[0,T)$).
\end{oprob}

\subsection{Blow--up limits under hypotheses on the lengths of the curves of the network} 
\begin{prop}\label{resclimit}
Let $\SS_t=\bigcup_{i=1}^n\gamma^i([0,1],t)$ be the curvature flow of a regular network 
with fixed end--points in a smooth, convex, bounded open set $\Omega\subseteq\R^2$ 
such that three end--points of the network are never aligned.
Assume that the lengths $L^i(t)$ of the curves of the networks satisfy
\begin{equation}\label{Lbasso}
\lim_{t\to T}\frac{L^i(t)}{\sqrt{T-t}}=+\infty\,,
\end{equation}
for every $i\in\{1,2,\dots, n\}$. 
Then any limit degenerate regular shrinker $\widetilde\SS_\infty$ 
obtained by Proposition~\ref{resclimit-general}, if
non--empty, 
is one of the following networks:\\
if the rescaling point belongs to $\Omega$
\begin{itemize}
\item a straight line through the origin with multiplicity $m\in\NN$ 
(in this case $\widehat\Theta(x_0)=m$);
\item a standard triod centered at the origin with multiplicity
 $1$ (in this case $\widehat\Theta(x_0)=3/2$);
\end{itemize}
if the rescaling point is a fixed end--point of the evolving network
(on the boundary of $\Omega$)
\begin{itemize}
\item a halfline from the origin with multiplicity $1$ (in this case $\widehat\Theta(x_0)=1/2$).
\end{itemize}
Moreover, we have
\begin{equation}\label{gggg}
\lim_{j\to\infty}\frac{1}{\sqrt{2\pi}}\int_{\widetilde{\SS}_{x_0,\tt_j}}
\widetilde{\rho}\,d\sigma=\frac{1}{\sqrt{2\pi}}
\int_{\widetilde\SS_\infty}\widetilde{\rho}\,d\overline{\sigma}
=\Theta_{\widetilde{\SS}_\infty}=\widehat{\Theta}(x_0)\,,
\end{equation}
and the $L^2$--norm of the curvature of $\widetilde{\SS}_{x_0,\tt_{j}}$ 
goes to zero in every ball $B_R\subseteq\R^2$, as $j\to\infty$. 
\end{prop}
\begin{proof}
We assume, by Proposition~\ref{resclimit-general}, that the sequence 
${\widetilde{\SS}_{x_0,\tt_{j}}}$ of reparametrized
networks converges in $C^1\loc\cap W^{2,2}\loc$ to the limit 
regular shrinker network $\widetilde{\SS}_\infty$
composed of $C^\infty$ curves (with possibly multiplicity), which are
actually non--degenerate as the bound from below on their lengths
prevents any collapsing along the rescaled sequence.\\
If the point $x_0\in\R^2$ is distinct from all the end--points $P^r$, 
then $\widetilde\SS_\infty$ has no end--points, since they go to infinity 
along the rescaled sequence. If $x_0=P^r$ for some $r$, the set $\widetilde\SS_\infty$ has a single end--point at the
origin of $\R^2$.\\ 
Moreover, from the lower bound on the length of the curves it follows
that all the curves of $\widetilde\SS_\infty$ have infinite
length, hence, by Remark~\ref{abla}, they must be pieces of straight
lines from the origin, because of the uniform bound $\HH^1(\SS^{\mu_i}_\tt\cap B_R)\leqslant C_R$, for every ball $B_R\subseteq\R^2$.\\
This implies that every connected component of the graph underlying 
$\widetilde\SS_\infty$ can contain at most one $3$--point 
and in such case such component must be mapped to a standard triod
(the $120$ degrees condition must satisfied) with multiplicity one
since the sequence of converging networks is all embedded 
(to get in the $C^1\loc$--limit a triod with multiplicity higher than one it is
necessary that the approximating networks have
self--intersections). Moreover, again since the converging networks
are all embedded, if a standard triod is present, a straight line or
another triod cannot be there, since they would intersect
transversally (see Remark~\ref{remreg}). Vice versa, if a straight line
is present, a triod cannot be present.

If an end--point is not present, that is, we are rescaling around a
point in $\Omega$ (not on its boundary) and a $3$--point is not
present, the only possibility is a straight line (possibly with
multiplicity) through the origin of $\R^2$.

If an end--point is present, we are rescaling around an end--point of
the evolving network, hence, by the convexity of $\Omega$ (which
contains all the networks) the limit $\widetilde{\SS}_\infty$ must be
contained in a halfplane with boundary a straight line $H$ for the
origin. This excludes the presence of a standard triod since it cannot
be contained in any halfplane. Another halfline is obviously excluded,
since they ``come'' only from end--points and they are all
distinct. In order to exclude the presence of a straight line, we observe that the argument of Proposition~\ref{omegaok2} implies that, if $\Om_t\subseteq \Om$ is the evolution by curvature of $\partial \Om$ keeping fixed the end--points $P^r$, the blow--up of $\Om_t$ at an end--point must be a cone spanning angle strictly less then $\pi$ 
(here we use the fact that three end--points are not aligned) and $\widetilde{\SS}_\infty$ is contained in such a cone. It follows that $\widetilde{\SS}_\infty$ cannot contain a straight line.

In every case the curvature of $\widetilde\SS_\infty$ is zero everywhere and the
last statement follows by the $W^{2,2}\loc$--convergence.

Finally, formula~\eqref{gggg} is a special case of
equation~\eqref{gggg2}.
\end{proof}

\begin{rem} If the two curves describing the boundary of $\Omega$ 
around an end--point $P^r$ are actually segments of the same line, 
namely the three end--points are $P^{r-1}, P^r, P^{r+1}$ aligned,
the argument of Proposition~\ref{omegaok2} does not work and 
we cannot conclude that taking a blow--up at $P^r$ we only get a halfline with unit multiplicity. 
It could also be possible that a straight line (possibly with multiplicity) 
through the origin is present, coinciding with $H$. 
Moreover in such special case, it forces
also the halfline to be contained in $H$, since the only way to get a
line, without self--intersections in the sequence of converging
networks contained in $\Omega$ is that the curves that are converging
to the straight line ``pushes'' the curve getting to the end--point of
the network, toward the boundary of $\Omega$.
\end{rem}

With the same arguments of the proof of Proposition~\ref{resclimit}, 
an analogous proposition holds for the self--similarly shrinking limit network flow 
obtained by the parabolic rescaling procedure.

\begin{prop}\label{resclimit2}
Under the hypotheses of Proposition~\ref{resclimit}, the degenerate regular 
 self--similarly shrinking network flow $\SS^\infty_\tt$, obtained in Proposition~\ref{thm:shrinkingnetworks.1} by parabolically rescaling 
 around the point $(x_0,T)$ in space--time, is (if non--empty) one of the following ``static'' flows.\\
 If the rescaling point belongs to $\Omega$:
\begin{itemize}
\item a straight line through the origin with multiplicity $m\in\NN$ (in this case $\widehat\Theta(x_0)=m$);
\item a standard triod centered at the origin with multiplicity
 $1$ (in this case $\widehat\Theta(x_0)=3/2$).
\end{itemize}
If the rescaling point is a fixed end--point of the evolving network
(on the boundary of $\Omega$):
\begin{itemize}
\item a halfline from the origin with multiplicity $1$ (in this case $\widehat\Theta(x_0)=1/2$).
\end{itemize}
\end{prop}

\begin{oprob}\label{ooo7}
Is it possible to classify in general all the possible limit degenerate shrinkers 
$\widetilde\SS_\infty$ or self--similarly shrinking flows $\SS^\infty_\tt$, obtained respectively 
by Huisken's dynamical procedure or by parabolic rescaling?
\end{oprob}

\begin{rem}
If the evolving network is a tree, every connected component 
of a limit degenerate regular shrinker (possibly with multiplicities) is still a tree. 
Hence by Lemma~\ref{lemmatree} and the same
argument of the proof of Proposition~\ref{resclimit}
such a network has zero curvature and it is a union of halflines from the
origin, possibly with multiplicity and a core.
\end{rem}

\begin{rem}\label{ooo8}
In Section~\ref{behavsing} we will discuss under what hypotheses 
the (unscaled) evolving networks $\SS_t$
converge to some limit (well--behaved) {\em set} $\SS_T\subseteq\R^2$,
as $t\to T$ and what are the relations between such $\SS_T$ and any
limit degenerate shrinker $\widetilde\SS_\infty$ or self--similarly shrinking flow 
$\SS^\infty_\tt$.
\end{rem}

\section{Local regularity}\label{locreg}

In this section, we first show that any smooth, curvature flow of regular networks which is only $C^1\loc$--close to the static flow given by a standard triod, is actually smoothly close. An important ingredient here is the estimates from
Proposition~\ref{stimaL}, under the hypotheses~\eqref{endsmooth}, which
make it possible to control the evolution of the $L^2$--norm of
$k$ locally. 

Then this result together with the classification of tangent flows from Lemma~\ref{thm:densitybound} yield a local regularity
theorem. As a consequence, locally (in space--time) around the points with limit Gaussian density not greater than $3/2$, the curvature of the evolving network $\SS_t$ is bounded and the flow is smooth, meaning that locally $\SS_t$ converges smoothly to a limit smooth network $\SS_T$, as $t\to T$.

\begin{lem}\label{thm:locreg.3} 
Let $\TTT$ be the static flow given by a standard triod centered at the origin
and let $\SS^i_t$ for $t\in(-1,0)$ be a sequence of smooth curvature flows of networks with uniformly bounded length ratios (see Definition~\ref{ublr}).
Suppose that the sequence $\SS^i_t$ converges to 
$\TTT$ in $C^1\loc$ for almost every $t\in(-1,0)$, as $i\to\infty$.
Then the convergence is smooth on any subset of the form
$B_R(0)\times [\widetilde{t},0)$ where $R>0$ and $-1<\widetilde{t}<0$.
\end{lem}

\begin{proof}
As the length ratios are uniformly bounded, the exponential decay of the backward heat kernels $\rho_{0,0}(\cdot, t)$ 
and the $C^1\loc$--convergence imply that for almost every $-1<t<0$ we have
$$
\int_{\SS^i_{t}} \rho_{0,0}(\cdot,t)\, ds \to \int_{{\mathbb{T}}} \rho_{0,0}(\cdot,t)\, ds = \frac{3}{2}<+\infty\,,
$$
hence by~\eqref{monresc} it follows that the sequence of functions
$$
f_i(t)=\int\limits_{\SS^i_t}
\Big|\underline{k}_i- \frac{x^\perp}{2t}\Big|^2\rho_{0,0}(\cdot,t)\,
ds\,,
$$
converges to zero in $L^1\loc(-1, 0)$. 

Arguing as in the proof of Proposition~\ref{thm:shrinkingnetworks.1}, we 
see that we can choose a further subsequence (not relabeled) such that
 $\SS^i_t \to\TTT$ in $C^{1,\alpha}\loc\cap W^{2,2}\loc$ for all
$t \in A$ where $A\subseteq (-1,0)$ is a set of full measure. Choose $R>0$, $\widetilde{t}\in(-1,0)$ and $t_0\in A$ such that $t_0<\widetilde{t}$. Lemma~\ref{boh}, with a compactness argument, implies that the curvature of the networks $\SS_t^i$ with all its derivatives are uniformly bounded and the convergence $\SS^i_t \to\TTT$ is smooth and uniform in $\bigl(B_{R+1}(0)\setminus B_R(0)\bigr)\times[t_0,0)$. We can thus introduce three ``artificial'' boundary points $P^r_i(t) \in\SS^i_t\cap
(B_{R+1}(0)\setminus B_R(0)),\ 
r=1,2,3$, for $t\in [t_0,0)$ along the three rays such that the
estimates~\eqref{endsmooth} are satisfied, more precisely, we can assume
that 
$$ \partial^j_s\lambda_i(P^r_i(t),t) = 0 \qquad \text{and} \qquad 
|\partial^j_sk_i(P^r_i(t),t)|\leqslant 1\,,
$$
for all $i\geqslant i_0$ and all $j\geqslant 0$.\\
Let $T_1>0$ be the constant from Proposition~\ref{stimaL} for $M=1$
and let $\delta = T_1/2$. Then, choose $t_l \in A$, for $l = 1,2,\dots, N= [\delta^{-1}]+1$, such that
$$
t_l<t_{l+1}\,,\qquad |t_N| \leqslant \delta/2\quad \text{ and }\quad |t_{l+1}-t_l| \leqslant \delta/2,
$$ 
for all $0\leqslant l \leqslant N-1$.\\
By increasing $i_0$, if necessary, we can assume that
$$
\int_{\SS^i_{t_l}\cap B_{R+1}(0)}k_i^2\, ds \leqslant 1
$$
and that $\SS^i_{t_l}$ is $1/100$--close in $C^{1,\alpha}$ to ${\mathbb{T}}$ on $B_{R+1}(0)$,
for all $l=0, \dots, N$ and $i>i_0$.\\
Proposition~\ref{topolino5} then implies uniform estimates
on $k_i$ and all its space derivatives on $B_R(0)\times [\widetilde{t},0)$, for all $i>i_0$. This clearly implies the convergence conclusion in the statement.
\end{proof} 

\begin{rem} With a similar argument it can be shown that if $\SS^i_t$ converge as above to a self--similarly shrinking regular network flow, non--degenerate and with unit multiplicity, then the convergence is smooth and uniform on any compact subset of $\R^2\times (-1,0)$ (Lemma~8.6 in~\cite{Ilnevsch}). 
\end{rem}

We now show a local regularity result in the spirit of the analogous White's theorem for mean curvature flow in~\cite{white1}, actually being an extension of such theorem to the network flow, roughly saying that (like in the case of the motion of smooth curves) the ``regular'' points are the ones with limit Gaussian density smaller than $\Theta_{\SS^1}$ (which is greater than $3/2$ and less than $2$, see formula~\eqref{thetas1}).\\
We follow here the alternative proof of Ecker~\cite[Theorem~5.6]{eck1}. 

\begin{thm}[Theorem~1.3 in~\cite{Ilnevsch}]\label{thm:locreg.2}
Let $\SS_t$ for ${t\in (T_0,T)}$ be a curvature flow of a smooth, regular 
network in $\R^2$ with uniformly bounded length ratios by some constant $L$ (see Definition~\ref{ublr}). Let $(x_0,t_0)\in\R^2\times(T_0,T)$ such that $x_0\in\SS_{t_0}$, then for every $\varepsilon, \eta>0$ there exists a constant $C = C(\varepsilon, \eta, L)$ such that if
\begin{equation}\label{eq:locreg.0.5}
{\Theta}_{x,t}(t-r^2) \leqslant \Theta_{\SS^1} -\varepsilon\,,
\end{equation}
for all $(x,t) \in B_\rho(x_0)\times (t_0-\rho^2,t_0)$ and
$0<r<\eta\rho$, for some $\rho >0$, where $T_0+(1+\eta)\rho^2\leqslant
t_0<T$, then
$$
k^2(x,t) \leqslant \frac{C}{\sigma^2\rho^2}\,,
$$
for all $\sigma \in (0,1)$ and every $(x,t)$ such that $t\in(t_0-(1-\sigma)^2\rho^2,t_0)$ and $x\in\SS_t\cap B_{(1-\sigma)\rho}(x_0)$. 
\end{thm}

\begin{proof} By translation and scaling we can assume that $x_0=0$, $t_0=0$ and $\rho = 1$. We can now follow more or less verbatim the proof of Theorem 5.6 in~\cite{eck1}.\\
We argue by contradiction. Supposing that the statement is not correct we can find a sequence of smooth curvature flows of regular open networks $\SS^j_t$, defined for $t \in [-1-\eta,0]$, satisfying the above conditions for every $(x,t)\in B_1(0)\times(-1,0)$, but with
 \begin{equation}
 \label{eq:locreg.0.1}
 \zeta_j^2 = \sup_{\sigma \in [0,1]}\Bigg(\sigma^2
 \sup_{t\in(-(1-\sigma)^2,0)}\sup_{\SS^j_t\cap B_{1-\sigma}} k^2_j
 \Bigg) \to +\infty
 \end{equation}
as $j \to \infty$.

Hence, we can find $\sigma_j \in (0,1]$ such that
$$
\zeta_j^2 = \sigma_j^2
 \sup_{t\in(-(1-\sigma_j)^2,0)}\sup_{\SS^i_t\cap B_{1-\sigma_j}} k^2_j
$$
and $y_j \in \SS^j_{\tau_j} \cap \overline{B}_{1-\sigma_j}$ at a time
$\tau_j \in [-(1-\sigma_j)^2,0]$ so that
 \begin{equation}
 \label{eq:locreg.0.2}
\zeta_j^2 = \sigma_j^2 k^2_j(y_j,\tau_j)\,.
\end{equation}
We now take 
$$
\lambda_j = |k_j(y_j,\tau_j)|
$$
(clearly $\lambda_j\to+\infty$ as $j\to\infty$) and define
$$
\widetilde{\SS}^j_\tt = \lambda_j
\Big(\SS^j_{\lambda_j^{-2}\tt+\tau_j} - y_j\Big)\,,
$$
for $\tt\in [-\lambda_j^{2}\sigma_j^2/4,0]$, following the proof of
Theorem 5.6 in~\cite{eck1}. We can then see that 
\begin{equation}\label{eq:locreg.0.2a}
0\in \widetilde{\SS}^j_0\,, \qquad \widetilde{k}_j^2(0,0)=1
\end{equation}
and
\begin{equation}\label{eq:locreg.0.2ab}
\sup_{\tt \in (-\lambda_j^{2}\sigma_j^2/4,0)}\sup_{\widetilde{\SS}^j_\tt
 \cap B_{\lambda_j\sigma_j/2}} {\widetilde{k}}_j^2 \leqslant 4
\end{equation}
for every $j \geqslant 1$. 
By direct computation, we
have
$$
{\widetilde{\Theta}}^j_{\overline{x},\overline{\tt}}(\tt)=\int_{\widetilde{\SS}^j_\tt}\rho_{\overline{x},\overline{\tt}}(\cdot,\tt)\,d{s}
=\int_{{\SS}^j_t}\rho_{y_j+\overline{x}\lambda_j^{-1},\tau_j+\overline{\tt}\lambda_j^{-2}}(\cdot,t)\,d{s}
={\Theta}^j_{y_j+\overline{x}\lambda_j^{-1},\tau_j+\overline{\tt}\lambda_j^{-2}}(t)
$$
where $t=t(\tt)=\tau_j+\tt\lambda_j^{-2}$ and $\Theta^j$ are the
Gaussian densities relative to the flows $\SS^j_t$. Since, by hypothesis, ${\Theta}^j_{y_j+\overline{x}\lambda_j^{-1},\tau_j+\overline{\tt}\lambda_j^{-2}}(t)\leqslant
\Theta_{\SS^1} -\varepsilon$ for every $j\in\NN$, $y_j+\overline{x}\lambda_j^{-1}\in B_1(0)$ and $\tau_j+\overline{\tt}\lambda_j^{-2}\in(-1,0)$, we conclude that ${\widetilde{\Theta}}^j_{\overline{x},\overline{\tt}}(\tt)\leqslant\Theta_{\SS^1} -\varepsilon$, for $j$ sufficiently large, for all $(\overline{x},\overline{\tt}) \in
\mathbb{R}^2\times (-\infty,0]$ and $-\lambda_j^2\sigma_j^2/4<\tt<\overline{\tt}$.\\
This implies that for every $\tt\in(-\lambda_j^2\sigma_j^2/4,0)$, we have
$$
\int_{\widetilde{\SS}^j_\tt\cap B_R(0)}\frac{e^{R^2/4\tt}}{\sqrt{-4\pi\tt}}\,ds
\leqslant \int_{\widetilde{\SS}^j_\tt\cap B_R(0)}\frac{e^{|x|^2/4\tt}}{\sqrt{-4\pi\tt}}\,ds
\leqslant \int_{\widetilde{\SS}^j_\tt}\rho_{0,0}(\cdot,\tt)\,ds
={\widetilde{\Theta}}^j_{0,0}(\tt)\leqslant\Theta_{\SS^1} -\varepsilon\,,
$$
hence, for $j$ sufficiently large, 
\begin{equation}\label{infest}
\mathcal{H}^1(\widetilde{\SS}^j_\tt\cap B_R(0))
\leqslant C_R(\tt)={e^{-R^2/4\tt}}{\sqrt{-4\pi\tt}}(\Theta_{\SS^1} -\varepsilon)\,.
\end{equation}
Moreover, the family of networks $\widetilde{\SS}^j_\tt$ has uniformly bounded length ratios by $L$, since this holds for the unscaled networks and such condition is scaling invariant.\\
Since $\lambda_j^{2}\sigma_j^2 = \zeta_j^2\to +\infty$, by the length estimate~\eqref{infest}, 
arguing as in Proposition~\ref{thm:shrinkingnetworks.1}, we see that up to a subsequence, labeled again the same, for every $\tt\in(-\infty,0)$, we have 
\begin{equation}\label{eq:locreg.0.3}
\widetilde{\SS}_\tt^j \to \widetilde{\SS}^\infty_\tt 
\end{equation} 
in $C^1\loc$ and weakly in $W^{2,\infty}\loc$, for almost every $\tt\in(0,-\infty)$, to a limit $C^{1,1}$--flow $\widetilde{\SS}^\infty_\tt$. Actually, the uniform bound on the curvature, everywhere in space--time, implies that such convergence holds for {\em every} $\tt\in(-\infty,0]$ and it is locally uniform in time. 
Such flow (which is not a priori a curvature flow) of networks is possibly degenerate, that is, cores and
higher density lines can develop, it moves with normal velocity bounded by $4$, by estimates~\eqref{eq:locreg.0.2ab} and it is not empty as
$0\in\widetilde{\SS}^j_0$ for every $j\in\NN$, hence $0\in\widetilde{\SS}^\infty_0$ also.

Because of the uniformly bounded length ratios of the family of networks $\widetilde{\SS}^j_\tt$ and the exponential decay of the backward heat kernels, we can pass to the limit in the Gaussian densities, as $j\to\infty$, that is,
$$
{\widetilde{\Theta}}^\infty_{\overline{x},\overline{\tt}}(\tt)=\lim_{j\to\infty}{\widetilde{\Theta}}^j_{\overline{x},\overline{\tt}}(\tt)
=\lim_{j\to\infty}{\Theta}^j_{y_j+\overline{x}\lambda_j^{-1},\tau_j+\overline{\tt}\lambda_j^{-2}}(t)\leqslant \Theta_{\SS^1} -\varepsilon
$$
for all $(\overline{x},\overline{\tt}) \in \mathbb{R}^2\times (-\infty,0]$ and $\tt<\overline{\tt}$, where we denoted with ${\widetilde{\Theta}}^j$ and ${\widetilde{\Theta}}^\infty$ the Gaussian density functions relative to the flows $\widetilde{\SS}^j_\tt$ and
$\widetilde{\SS}^\infty_\tt$, respectively.\\
Moreover, $0\in\widetilde{\SS}^j_0$ implies $\widehat{\Theta}^j(0,0)\geqslant 1$, hence ${\widetilde{\Theta}}^j_{0,0}(\tt)\geqslant \widehat{\Theta}^j(0,0)\geqslant 1$ for every $\tt<0$, by monotonicity. It follows that 
${\widetilde{\Theta}}^\infty_{0,0}(\tt)=\lim_{j\to\infty}{\widetilde{\Theta}}^j_{0,0}(\tt)\geqslant1$, thus,
\begin{equation}\label{nep}
\widehat{\Theta}^\infty(0,0)=\lim_{\tt\to0}{\widetilde{\Theta}}^\infty_{0,0}(\tt)=\lim_{\tt\to0}\lim_{j\to\infty}{\widetilde{\Theta}}^j_{0,0}(\tt)\geqslant1\,.
\end{equation}

We want now to show that $\widetilde{\SS}^\infty_\tt$ is actually a
static self--similarly shrinking flow given by either a
multiplicity--one line or a standard triod.

As in Section~\ref{pararesc}, we consider the rescaled monotonicity
formula for the curvature flows $\widetilde{\SS}^j_\tt$, that is, considered $\overline{x}\in\R^2$ we have
$$
{\widetilde{\Theta}}_{\overline{x},0}^j(\tt_1)-{\widetilde{\Theta}}_{\overline{x},0}^j(\tt_2)
=\int\limits_{\tt_1}^{\tt_2}\int\limits_{\widetilde{\SS}^j_\ssss}
\Big|\underline{\widetilde{k}}_j-\frac{x^\perp}{2\ssss}\Big|^2\rho_{\overline{x},0}(\cdot,\ssss)\,
d{s}\,d\ssss
$$
hence, passing to the limit, as $j\to\infty$, we get (here $d\overline{s}$ denotes the integration with respect to the canonical measure on $\widetilde{\SS}_\tt^\infty$, 
counting multiplicities)
\begin{equation}\label{shg}
{\widetilde{\Theta}}_{\overline{x},0}^\infty(\tt_1)-{\widetilde{\Theta}}_{\overline{x},0}^\infty(\tt_2)
=\lim_{j\to\infty}\int\limits_{\tt_1}^{\tt_2}\int
\limits_{\widetilde{\SS}^j_\ssss}
\Big|\underline{\widetilde{k}}_j-\frac{x^\perp}{2\ssss}\Big|^2\rho_{\overline{x},0}(\cdot,\ssss)\,
ds\,d\ssss
\geqslant\int\limits_{\tt_1}^{\tt_2}\int\limits_{\widetilde{\SS}^\infty_\ssss}
\Big|\underline{\widetilde{k}}_\infty-\frac{x^\perp}{2\ssss}\Big|^2\rho_{\overline{x},0}(\cdot,\ssss)\,
d\overline{s}\,d\ssss
\end{equation}
for every $\tt_1<\tt_2\leqslant 0$ and $\overline{x}\in\R^2$, by the lower semicontinuity of the
$L^2$--integral of the curvature under the $W^{2,\infty}\loc$--weak
convergence. It follows that the Gaussian density function 
${\widetilde{\Theta}}_{\overline{x},0}^\infty(\tt)$ is non increasing in $\tt\in(-\infty,0]$,
then, as we know that it is uniformly bounded above by $\Theta_{\SS^1}
-\varepsilon$, there exists the limit
$$
\widehat{\Theta}_{\overline{x},0}^\infty(-\infty)=\lim_{\tt\to-\infty}{\widetilde{\Theta}}_{\overline{x},0}^\infty(\tt)
\leqslant \Theta_{\SS^1} -\varepsilon\,.
$$
Notice that $\widehat{\Theta}_{0,0}^\infty(-\infty)\geqslant 1$, as we know
that ${\widetilde{\Theta}}^\infty_{0,0}(\tt)\geqslant1$, for every $\tt<0$.

Parabolically rescaling the flow $\widetilde{\SS}^\infty_\tt$ around the point $(\overline{x},0)$ (following the proof of Proposition~\ref{thm:shrinkingnetworks.1}) by means of
inequality~\eqref{shg}, the uniform bound on the curvature and the uniform bound on the length ratios, we obtain that the limit (which exists by the monotonicity of $\tt\mapsto\widetilde{\Theta}_{\overline{x},0}^\infty(\tt)$)
$$
\widehat{\Theta}^\infty({\overline{x},0})=\lim_{\tt\to0}\widetilde{\Theta}_{\overline{x},0}^\infty(\tt)\leqslant
\widehat{\Theta}_{\overline{x},0}^\infty(-\infty)\leqslant \Theta_{\SS^1} -\varepsilon
$$
coincides with the Gaussian density of a limit degenerate regular shrinker (possibly empty). Being such a limit bounded by
$\Theta_{\SS^1} -\varepsilon$, the only possibilities are $0$, $1$ and
$3/2$, by Lemma~\ref{thm:densitybound} (an empty limit, a line, or a standard triod).

Since $\widetilde{\SS}^\infty_0$ is not empty, we notice that if it contains a $3$--point, let us say at $\overline{x}\in\R^2$, then by the bound on the velocity, also all the networks $\widetilde{\SS}^\infty_\tt$ contain a $3$--point at distance less than $-5\tt$ from $\overline{x}$. This implies that parabolically rescaling as above around $\overline{x}$, we get a limit self--similarly shrinking network flow with zero curvature and with a $3$--point, then it must be a static standard triod and $\widehat{\Theta}^\infty({\overline{x},0})=3/2$. We then take a point $\overline{x}\in\R^2$ such that
$\widehat{\Theta}^\infty({\overline{x},0})$ is maximum, hence either $1$ or
$3/2$ by what we said above and we consider the sequence of translated and rescaled flows for $\tau\in(-\infty,0]$ defined as
$$
\overline{\SS}_\tau^n=\frac{1}{\sqrt{n}}\Bigl(\widetilde{\SS}^\infty_{n\tau}-\overline{x}\Bigr)\,,
$$
for $n\in\NN$.\\
This family of flows still has uniformly bounded length ratios (since this holds for the flows $\widetilde{\SS}_\tt^\infty$) and rescaling the monotonicity formula for the flows $\widetilde{\SS}^\infty_\tt$, for every $\tau_1<\tau_2<0$, there holds
$$
\int\limits_{\tau_1}^{\tau_2}\int
\limits_{\overline{\SS}^n_\sigma}
\Big|\underline{\overline{k}}_n-\frac{x^\perp}{2\sigma}\Big|^2\rho_{0,0}(\cdot,\sigma)\,
d\overline{s}\,d\sigma
\leqslant{\overline{\Theta}}_{0,0}^n(\tau_1)-{\overline{\Theta}}_{0,0}^n(\tau_2)
={\widetilde{\Theta}}_{\overline{x},0}^\infty(n\tau_1)-{\widetilde{\Theta}}_{\overline{x},0}^\infty(n\tau_2)\to0
$$
as $n\to\infty$, since $\lim_{\tt\to-\infty}{\widetilde{\Theta}}_{\overline{x},0}^\infty(\tt)\to\widehat{\Theta}_{\overline{x},0}^\infty(-\infty)$ as $\tt\to-\infty$ (here we denoted with ${\overline{\Theta}}^n$ the Gaussian density functions relative to the flows $\overline{\SS}^n_\tau$).\\
Then, repeating the argument of the proof of Proposition~\ref{thm:shrinkingnetworks.1}, we
can extract a subsequence, not relabeled, of the flows
$\overline{\SS}_\tau^n$ converging in $C^1\loc\cap W^{2,2}\loc$, for almost every $\tau\in(-\infty,0)$, to a limit
self--similarly shrinking flow $\overline{\SS}_\tau^\infty$, as $n\to\infty$, which is called ``tangent flow at $-\infty$'' to the flow $\widetilde{\SS}^\infty_\tt$.

Since, 
$$
\overline{\Theta}_{0,0}^n(\tau)=\int\limits_{\overline{\SS}^n_\tau}\rho_{0,0}(\cdot,\tau)\,d\overline{s}
=\int\limits_{\widetilde{\SS}^\infty_{n\tau}}\rho_{\overline{x},0}(\cdot,n\tau)\,d\overline{s}
=\widetilde{\Theta}_{\overline{x},0}^\infty(n\tau)\,,
$$
it follows that, passing to the limit as $n\to\infty$ (again because of the uniformly bounded length ratios and the exponential decay of the backward heat kernels), for almost every $\tau\in(-\infty,0)$, there holds
$$
\Theta_{\overline{\SS}^\infty_{-1/2}}=\overline{\Theta}_{0,0}^\infty(\tau)=\lim_{n\to\infty}\widetilde{\Theta}_{\overline{x},0}^\infty(n\tau)=
\widehat{\Theta}_{\overline{x},0}^\infty(-\infty)\leqslant \Theta_{\SS^1} -\varepsilon
$$
which implies that the limit flow $\overline{\SS}_\tau^\infty$ is not empty, as $\widehat{\Theta}_{\overline{x},0}^\infty(-\infty)\geqslant 1$ and 
it is a static self--similarly shrinking flow, given by either a
multiplicity--one line or a standard triod, by Lemma~\ref{thm:densitybound}.

If $\overline{\Theta}_{0,0}^\infty(\tau)=1$, then
$\widehat{\Theta}_{\overline{x},0}^\infty(-\infty)=1$ which forces 
$\widetilde{\Theta}_{\overline{x},0}^\infty(\tt)$ to be constant equal to one for
every $\tt\in(-\infty,0)$, since
$\widehat{\Theta}^\infty({\overline{x},0})$ must be equal to 1. 

If $\overline{\Theta}_{0,0}^\infty(\tau)=3/2$, being
$\overline{\SS}_\tau^\infty$ a standard triod, it follows that a
$3$--point is present in the flow $\widetilde{\SS}^\infty_\tt$,
hence also in $\widetilde{\SS}^\infty_0$. Then, if we choose $\overline{x}$ to coincide with such $3$--point, 
we would have $\widehat{\Theta}^\infty({\overline{x},0})=3/2$ and again the Gaussian
density $\widetilde{\Theta}_{\overline{x},0}^\infty(\tt)$ is constant equal to $3/2$, for $\tt\in(-\infty,0)$.

In both cases we conclude that $\widetilde{\SS}^\infty_\tt$ is a
self--similarly shrinking flow around the point $\overline{x}\in\R^2$, by formula~\ref{shg}, given by a
multiplicity--one line in the first case and a standard triod in the 
second one.

If $\widetilde{\SS}^\infty_\tt$ is a line for every $\tt\in(-\infty,0]$, hence with zero curvature, Lemma~\ref{boh} implies that the convergence of the flows $\widetilde{\SS}_\tt^j \to \widetilde{\SS}^\infty_\tt$ is locally smooth. This gives a contradiction since, by formula~\eqref{eq:locreg.0.2a}, it would follow that $0\in \widetilde{\SS}^\infty_0$ and $\widetilde{k}_\infty^2(0,0)=1$.

If $\widetilde{\SS}^\infty_\tt$ is a static standard triod, then Lemma~\ref{thm:locreg.3} gives a contradiction as before.
\end{proof}

\begin{rem}\label{rem:locreg.1}\ 
\begin{enumerate}
\item The result is still true if the
 flow is only defined on the ball $B_{2\rho}(x_0)$, by localizing Huisken's monotonicity
formula with a suitable cut--off function. This makes the result 
applicable for curvature flows of networks with fixed end--points on the boundary of a domain $\Omega\subseteq\R^2$, once assuming that there are no boundary points in $B_{2\rho}(x_0)\times (t_0-(1+\eta)\rho^2,t_0)$.
We refer the reader to~\cite[Section~10 ]{white5} and Remark~4.16
together with Proposition~4.17 in~\cite{eck1}.

\item By an easy contradiction argument, one can show that the bound on the curvature, together with the $120$ degrees 
condition and assumption~\eqref{eq:locreg.0.5}, imply that there is a constant
 $\ell=\ell(\varepsilon, \eta, \rho)>0$ such that for $t \in (t_0-(1-\sigma)^2\rho^2,t_0)$
 the length of each curve of $\SS_t$ which intersects
 $B_{(1-\sigma)\rho}(x_0)$ is
 bounded from below by $\ell\cdot \sigma \rho$. 
This implies, using Proposition~\ref{topolino5}, corresponding scaling invariant estimates on all the higher derivatives of the curvature.
\end{enumerate}
\end{rem}

The following corollary is then an extension of White's result~\cite[Theorem~3.5]{white1} to the curvature flow $\SS_t$ of a network in a smooth, convex, bounded open set $\Omega\subseteq\R^2$, with fixed end--points on $\partial\Omega$.
 
\begin{cor}\label{regcol} If at a point $x_0\in\Omega$ there holds $\widehat{\Theta}(x_0)\leqslant 3/2$, then the curvature is uniformly bounded along the flow $\SS_t$, for $t\in[0,T)$, in a neighborhood of $x_0$. Then, the flow is smooth in such a neighborhood, in the sense that $\SS_t$ converges smoothly to a limit smooth network $\SS_T$ there, as $t\to T$.
\end{cor}
\begin{proof}
First, by Lemma~\ref{rescestim2}, the family of networks $\SS_t$ has uniformly bounded length ratios. Then, as $\widehat{\Theta}(x_0)=\widehat{\Theta}(x_0,T)\leqslant3/2$, by Proposition~\ref{promono}, Lemma~\ref{stimadib} and the subsequent discussion about the behavior of $\Theta_{x_0,T}(t)$, there exists $\rho_1\in(0,1)$ such that $\Theta_{x_0,T}(T-\rho_1^2)<3/2+\delta/2$, for some small $\delta>0$. The function $(x,t)\mapsto\Theta_{x,t}(t-\rho_1^2)$ is continuous, hence, we can find $\rho<\rho_1$ such that if $(x,t)\in B_\rho(x_0)\times(T-\rho^2,T)$, then $\Theta_{x,t}(t-\rho^2_1)<3/2+\delta$, thus, again by by Proposition~\ref{promono}, Lemma~\ref{stimadib} and the subsequent discussion (possibly choosing smaller $\rho_1$ and $\rho$), also $\Theta_{x,t}(t-r^2)<3/2+\delta$, for any $r\in(0,\rho/2)$, as clearly $(t-r^2)>( t-\rho_1^2)$.\\
This implies that if $\delta>0$ is small enough such that $3/2+\delta<\Theta_{\SS^1}=\sqrt{2\pi/e}\approx 1,\!5203$ (see equation~\eqref{thetas1}), for any $t_0$ close enough to $T$ the hypotheses of Theorem~\ref{thm:locreg.2} (see the first point of Remark~\ref{rem:locreg.1}) are satisfied at $(x_0,t_0)$, for $\eta=3/4$ and $\varepsilon=\Theta_{\SS^1}-3/2-\delta>0$. Choosing $\sigma=1/2$, we conclude that 
$$
k^2(x,t) \leqslant \frac{4C(\varepsilon,3/4)}{\rho^2}
$$
for every $(x,t)$ such that $t\in(t_0-\rho^2/4,t_0)$ and $x\in\SS_t\cap B_{\rho/2}(x_0)$. Since this estimate on the curvature is independent of $t_0<T$, it must hold for every $t\in(T-\rho^2/4,T)$ and $x\in\SS_t\cap B_{\rho/2}(x_0)$ and we are done.

We now show the smoothness of the flow up to time $T$ in a neighborhood of $x_0$. Since the curvature of $\SS_t$ is bounded in $B_{\rho/2}(x_0)$, the modulus of the velocity $v^i(t)$ of any triple junction $O^i(t)$ in such ball is uniformly bounded by some constant $D$, hence, if for $t$ in an interval of time $[t_1,t_2]$, such triple junction belongs to the ball $B_{\rho/2}(x_0)$, there holds 
\begin{equation}\label{Oconv}
\vert O^i(t_2)-O^i(t_1)\vert=\biggl\vert\int_{t_1}^{t_2}v^i(\xi)\,d\xi\,\biggl\vert\leqslant
\int_{t_1}^{t_2}\vert v^i(\xi)\vert\,d\xi\leqslant D\vert t_2-t_1\vert\,.
\end{equation}
This implies that if for some $t_0$ close enough to $T$, the triple junction $O^i(t_0)$ belongs to the ball $B_{\rho/4}(x_0)$, then it can no more ``escape'' from the ball $B_{\rho/2}(x_0)$, hence such estimate holds for every $t\in[t_0,T)$ implying that $O^i(t)$ is a Cauchy sequence and $O^i(t)\to x_i$, for some $x_i\in B_{\rho/2}(x_0)$. As a consequence, since the family of the limit points $\{x_i\}$ of the triple junctions in $B_{\rho/4}(x_0)$ is finite, possibly taking a smaller $\rho$, we can assume that only $x_0$ (possibly) belongs to such family. Hence, for any $\delta\in(0,\rho/4)$, the annulus ${\mathcal A}_\delta=B_{\rho/4}(x_0)\setminus \overline{B}_{\delta}(x_0)$ does not contains triple junctions $O^i(t)$ for $t$ larger than some $\overline{t}\in[0,T)$. This clearly means that the ``restriction'' of the flow $\SS_t$ to the open set ${\mathcal A}_\delta$ is a smooth (classical) flow by curvature of curves in a domain of the plane with uniformly bounded curvature. By standard estimates (for instance, by Ecker and Huisken interior estimates in~\cite{eckhui2}) then $\SS_t\cap {\mathcal A}_\delta$ converges smoothly to some limit family of embedded and non--intersecting smooth curves in ${\mathcal A}_\delta$. Since this holds for every $\delta\in(0,\rho/4)$, we can conclude that $\SS_t$ converges (possibly after reparametrization) in $C^1$ to a degenerate regular network $\SS_T$ in $B_{\rho/4}(x_0)$ (with possibly a core only at $x_0$) and locally smoothly in $B_{\rho/4}(x_0)\setminus\{x_0\}$.\\
It is then easy to see, possibly considering a smaller $\rho$, that we can find $\overline\rho<\rho/8$ such that 
\begin{itemize}
\item the network $\SS_T\cap B_{\rho/4}(x_0)$ is connected;
\item the curves of the networks $\SS_t$ intersect transversally the circle $\partial B_{\overline{\rho}}(x_0)$.
\end{itemize}
Then, by the uniform bound on the velocity and the smooth convergence of $\SS_t$ to $\SS_T$ in $B_{\rho/4}(x_0)\setminus\{x_0\}$, possibly choosing a larger $\overline{t}$, we can conclude that for every $t\in[\overline{t},T)$,
\begin{itemize}
\item the ``topologic structure'' of $\SS_t$ in $B_{\overline{\rho}}(x_0)$ is ``stable'' and that the network $\SS_T\cap B_{\rho/4}(x_0)$ is connected, that is, no ``new'' $3$--points or pieces of curves can ``get into'' $B_{\overline{\rho}}(x_0)$;
\item the curves of the networks $\SS_t$ intersect transversally the circle $\partial B_{\overline{\rho}}(x_0)$.
\end{itemize}
The last property implies then that condition~\eqref{endsmooth} are satisfied (possibly after reparametrizing the networks in order to deal with $\lambda$ and its derivatives).\\
If now $\SS_t\cap B_{\overline{\rho}}(x_0)$ contains more than a triple junction, all of them must converge to $x_0$, as $t\to T$, by what we said above, moreover, by equation~\eqref{Oconv}, we have
$$
\vert O^i(t)-x_0\vert\leqslant D\vert T-t\vert\,,
$$
hence, they images $\widetilde{O}^i(\tt)$, after performing Huisken's rescaling procedure, satisfy
$$
\vert\widetilde{O}^i(\tt)\vert=\frac{\vert O^i(t(\tt)-x_0\vert}{\sqrt{2(T-t(\tt))}}\leqslant\frac{D\vert T-t(\tt)\vert}{\sqrt{2(T-t(\tt))}}=D\sqrt{(T-t(\tt))/2}\,,
$$
which tends to zero, as $\tt\to+\infty$, in particular they cannot ``disappear'' in the limit degenerate regular shrinker (going to infinity). This is in contradiction with the fact that, by Lemma~\ref{thm:densitybound}, since $\widehat{\Theta}(x_0)\leqslant 3/2$, the only possible blow--up limit shrinkers at $x_0$ are the empty set, a line or a standard triod, hence, with at most {\em one} triple junction. Containing then $\SS_t\cap B_{\overline{\rho}}(x_0)$ at most one $3$--point, possibly choosing smaller $\rho,\overline{\rho}$ and larger $\overline{t}$, if $\SS_t\cap B_{\overline{\rho}}(x_0)$ is not empty (when $\widehat{\Theta}(x_0)=0$), it follows that we are dealing, either with the (classical) motion with uniformly bounded curvature of a single smooth curve (case without triple junctions, $\widehat{\Theta}(x_0)=1$) or with the motion of a triod (when $\widehat{\Theta}(x_0)=3/2$) with uniformly bounded curvature and conditions~\eqref{endsmooth} satisfied. Moreover, in both cases the lengths of all the curves of $\SS_t\cap B_{\overline{\rho}}(x_0)$ are uniformly positively bounded below, by the construction (the choice of $\rho$).

Then, if $\SS_t\cap B_{\overline{\rho}}(x_0)$ is empty, there is nothing to show, in the case of the motion of a single curve the flow is locally smooth up to time $T$, since the curvature is locally bounded (again by using Ecker and Huisken interior estimates in~\cite{eckhui2}), while in the case of an evolving triod, the local smoothness of the flow up to time $T$ follows by the estimates on all the derivatives of the curvature given by Proposition~\ref{topolino5} (see the second point of Remark~\ref{rem:locreg.1}).
\end{proof}

This corollary can be extended to the points on the boundary of $\Omega$ by the ``reflection argument'' at the end of Section~\ref{geopropsub}.

\begin{cor}\label{regcolb} If at a point $x_0\in\partial\Omega$ there holds $\widehat{\Theta}(x_0)\leqslant 3/4$, then the curvature is uniformly bounded along the flow $\SS_t$, for $t\in[0,T)$, in a neighborhood of $x_0$. Then, the flow is smooth in such neighborhood, in the sense that $\SS_t$ converges smoothly to a limit smooth network $\SS_T$ there, as $t\to T$.
\end{cor}

\section{The behavior of the flow at a singular time}\label{behavsing}

By means of the tools of the previous sections
we want to discuss now the behavior of the network approaching a singular time.

Let $T<+\infty$ be the maximal time of existence of the curvature flow $\SS_t$ of an initial
regular $C^2$ network with fixed end--points in a smooth, strictly convex,
bounded open set $\Omega\subseteq\R^2$. Then, by Theorem~\ref{curvexplod-general}, as $t\to T$,
either the curvature is not bounded, or the inferior limit of the lengths $L^i(t)$ of at least one
curve of $\SS_t$ is zero.

Hence if all the lengths of the curves of the network are uniformly positively bounded
from below, the curvature is not bounded (actually again by Theorem~\ref{curvexplod-general}) the maximum of the absolute value
of the curvature goes to $+\infty$). By Proposition~\ref{brakprop} we also know that if the curvature is uniformly bounded, all the
lengths of the curves converge as $t\to T$, thus at least some 
$L^i(t)$ must go to zero.\\
We will then divide our analysis into the following three cases:

\begin{itemize}
\item all the lengths of the curves of the network are uniformly
 positively bounded from below and the maximum of the modulus of the curvature goes
 to $+\infty$, as $t\to T$;
\item the curvature is uniformly bounded along the flow and the length
 $L^i(t)$ of at least one curve of $\SS_t$ goes to zero when $t\to T$;
\item the curvature is not bounded and the length of at least one
 curve of the network is not positively bounded from below, as $t\to T$.
\end{itemize}

In all three cases, the possible blow--up limits will play a key
role, with the obvious consequence that the fewer possibilities we have,
the easier we can get conclusions. In particular, it is crucial to exclude the onset of
blow--up limits of multiplicity larger than one, in particular ``multiple lines'',
exactly as in the study of the evolution of a single smooth closed curve
(see~\cite{huisk2}, for instance).
In the case of 
curves this can be done by means of 
some ``embeddedness'' or ``non--collapsing'' quantities
(see~\cite{hamilton3,huisk2}) 
that actually inspired our results in Section~\ref{dsuL}.\\
Unfortunately, in the case of regular networks proving that any blow--up limit has multiplicity one without asking for any extra assumption is still an open problem, maybe the 
major one.

\begin{oprob}[Multiplicity--One Conjecture -- {\bf{M1}}]\label{ooo9}
Every blow--up limit shrinker arising from Huisken's rescaling 
procedure or limit of parabolic rescalings at a point $x_0\in\overline{\Omega}$ 
is an embedded network with multiplicity one.
\end{oprob}

This conjecture is implied by the two equivalent statements in the following open problem.

\begin{oprob}[Strong Multiplicity--One Conjecture -- 
{\bf{SM1}}/No Double--Line Conjecture -- {\bf{L1}}]\label{ooo10}\ 
\begin{itemize}
\item[{\bf{SM1}}:] Every possible $C^1\loc$--limit of rescalings 
of networks of the flow is an embedded network with multiplicity one.
\item[{\bf{L1}}:] A straight line with multiplicity larger than one cannot be obtained
as a $C^1\loc$--limit of rescalings of networks of the flow.
\end{itemize}
\end{oprob}

While it is obvious that the first statement implies both {\bf{M1}} and {\bf{L1}}, the fact that 
the second one implies the first can be seen as follows: 
if {\bf{SM1}} does not hold, since the networks of the flow are all
embedded, any limit of rescalings $\SS_i$
can lose embeddedness only if two curves in the limit network ``touch'' each other at some
point $x_0\in\R^2$ with a common tangent (or they locally coincide, if
they ``produce'' a piece of curve with multiplicity larger than one). Then, ``slowly''
dilating the networks $\SS_i$ around $x_0$, in order that the distance
between such two curves and $x_0$ still go to zero, we would get a multiplicity--two line, 
contradicting {\bf{L1}}.

We will see in Section~\ref{dsuL} some cases in which we are able to
show that the strong multiplicity--one conjecture holds:
\begin{itemize}
\item If during the flow the triple junctions stay uniformly far from each
 other, then {\bf{SM1}} is true.
\item If the initial network has at most two triple junctions, then
 {\bf{SM1}} is true.
\end{itemize}

\begin{rem}\label{enanrem}
If {\bf{M1}} holds, the flow $\SS_\tt^\infty$ in
Proposition~\ref{thm:shrinkingnetworks.1} is composed of embedded,
multiplicity--one network and the same holds for the limit network
$\widetilde{\SS}_\infty$ in Proposition~\ref{resclimit-general}. In
particular under the hypotheses of
Proposition~\ref{resclimit} any blow--up limit network at a point $x_0$
and singular time $T$, obtained by Huisken's procedure, 
or self--similarly shrinking network
flow, obtained by the parabolic rescaling procedure, 
is (if not empty) a ``static'' straight line
through the origin (then $\widehat{\Theta}(x_0)=1$) or a standard
triod (then $\widehat{\Theta}(x_0)=3/2$), if the rescaling
point belongs to $\Omega$. If the rescaling point is instead a
fixed end--point of the evolving network on the boundary of $\Omega$,
then such limit can only be a single halfline from the origin 
(and $\widehat{\Theta}(x_0)=1/2$).
\end{rem}

Before analyzing the three situations above, we set some
notation and we show some general properties of the flow at the
singular time. 

We let $F: \SS\times [0,T) \to \overline{\Omega}$, with $T <+\infty$,
represent the curvature flow $\SS_t$ of a regular network moving 
by curvature in its maximal time interval of smooth existence. We let
$O^1, O^2,\dots, O^m$ the $3$--points of $\SS$.

We define the set of {\em reachable points} of the flow as follows:
$$
\mathcal{R} = \bigl\{ x \in \R^2\,\bigl\vert\,\text{ there exist $p_i
\in \SS$ and $t_i \nearrow T$ such that $\lim_{i \to \infty}F(p_i, t_i) = x$}\bigr\}\,.
$$
Such a set is easily seen to be closed and contained in
$\overline{\Omega}$ (hence compact as $\Omega$ is bounded).
Moreover
the following lemma holds:

\begin{lem}
A point $x \in \R^2$ belongs to $\mathcal{R}$ if and only if for every
time $t \in [0,T)$ the closed ball with center $x$ and radius
$\sqrt{2(T-t)}$ intersects $\SS_t$. 
\end{lem}
\begin{proof}
One of the two implications is trivial. We have to prove that if $x
\in \mathcal{R}$, then $F(\SS, t) \cap \overline{B}_{\sqrt{2(T-t)}}(x)
\neq \emptyset$. If $x$ is one of
the end--points, the result is obvious, otherwise we define the function 
$$d_x(t) = \inf_{p \in \SS} |F(p,t) - x|\,,$$ 
where, due to the compactness of $\SS$ the infimum is actually a
minimum and as $t\to T$, let us say for $t>t_x$, it cannot be achieved at an end--point,
by the assumption $x\in\mathcal{R}$ and $x$ different from an end--point, such a maximum cannot be either achieved at a $3$--point,
by the $120$ degrees angle condition. Since the function $d_x: [0,T) \to \R$ is locally Lipschitz, 
we can then use Hamilton's trick (see~\cite{hamilton2} or~\cite[Lemma~2.1.3]{Manlib}), to compute its time derivative and get 
(for any point $q$, different from an end--point, where at time $t$ the minimum of $|F(p,t)
- x|$ is attained)
\begin{equation*}
\begin{split}
\partial_t d_x(t) & = \partial_t |F(q,t) - x| \geqslant \frac{\op k(q,t)\nu(q,t) + \l(q,t)
 \t(q,t), F(q,t) - x \cl}{|F(q,t) - x|} \\
& = \frac{\op k(q,t) \nu(q,t), F(q,t)- x \cl}{|F(q,t) - x|} \geqslant -\frac{1}{d_x(t)}\,,
\end{split}
\end{equation*}
since at a point of minimum distance the vector 
$ \frac{F(q,t) - x}{|F(q,t) - x|} $ 
is parallel to $\nu(q,t)$. Integrating
this inequality over time, we get
$$
d^2_x(t) - d^2_x(s) \leqslant 2(s-t) 
\qquad \quad{\rm for\ }s>t>t_x\,.
$$
We now use the hypothesis that $x$ is reachable ($\lim_{t_i \to T}
d_x(t_i) = 0$) and we conclude
$$
d^2_x(t) = \lim_{i \to \infty} [d^2_x(t) - d^2_x(t_i)] \leqslant 2 \lim_{i
 \to \infty} (t_i - t) = 2 (T-t)\,,
$$
for every $t>t_x$.
\end{proof}

As a consequence, when we consider the blow--up limits of the evolving networks by Huisken's
rescaling procedure around points of $\overline{\Omega }$, we have a dichotomy
among these latter. If the blow--up point belongs to $\mathcal{R}$, this lemma ensures that any
rescaled network contains at least one point of the closed unit ball of
$\R^2$, hence the limit of any sequence is not empty (and clearly {\em vice versa}). If the point does not belong to $\mathcal{R}$ any blow--up limit is empty, since the distance of the evolving network from
the point of blow--up is positively bounded below (by the very definition of ${\mathcal{R}}$) and rescaling, the whole dilated networks go to infinity. By Lemma~\ref{equal}, the same conclusion holds for the self--
similarly, shrinking curvature flows coming from the parabolic rescaling procedure.

\begin{lem}\label{Rlemma}
The family of blow--up limit shrinkers $\widetilde{\SS}_\infty$ arising from Proposition~\ref{resclimit-general} and the family of self--similarly shrinking curvature flows coming from Proposition~\ref{thm:shrinkingnetworks.1} are not empty, if and only if the blow--up point $x_0$ belongs to $\mathcal{R}$. It follows that the set of reachable points of the flow coincides with $\{x\in\overline{\Omega}~\vert~\widehat{\Theta}(x)>0\}$.
\end{lem}

We now show that, assuming the multiplicity--one conjecture, as $t\to
T$, all the $3$--points of the network $\SS_t$ converge.

\begin{lem}\label{linelemma} 
If {\bf{M1}} holds, there exists a radius $R=R(\SS_t,x_0)>0$, such that if a blow--up limit regular shrinker $\widetilde{\SS}_\infty$ (or $\SS^\infty_{-1/2}$) at the point $x_0$ has no triple junctions in the ball $B_R(0)$, then it is a line through the origin of $\R^2$ or the unit circle.
\end{lem}
\begin{proof}
Assume that the conclusion is false, then there is a sequence $R_i\to+\infty$ and blow--up limit regular shrinkers $\SS_i$ at $x_0$, all different from a line or circle, such that each $\SS_i$ has no triple junctions in $B_{R_i}(0)$, for every $i\in\NN$.

As we said in the discussion above, any shrinker $\SS_i$ must intersect the unit circle, hence, by the shrinkers equation~\eqref{shrinkeq}, we can extract a subsequence of $\SS_i$ locally converging in $C^1$ to a non empty limit shrinker $\overline{\SS}$ without triple junctions at all. By the work of Abresch and Langer~\cite{ablang1}, then $\overline{\SS}$ must be 
a line through the origin or the unit circle and this latter case is excluded, since, for $i$ large enough also $\SS_i$ would be a circle, which is a contradiction. If the limit $\overline{\SS}$ is a line, by the multiplicity--one conjecture, its multiplicity must be one, being any limit of blow--up limits of $\SS_t$ at the point $x_0$ again a blow--up limit at $x_0$.\\
Then, by the second point of Lemma~\ref{rescestim2}, the contribution of $\SS_i\setminus B_R(0)$ to the Gaussian density of the whole $\SS_i$ is small as we want, for every $i\in\NN$, by choosing a value $R$ large enough, while, for sufficiently large $i$, the contribution of $\SS_i\cap B_R(0)$ is smaller than one, as $\SS_i\to\overline{\SS}$, which is a multiplicity--one line. Hence, we conclude that the Gaussian density of $\SS_i$ is close to one for sufficiently large $i$, then Lemma~\ref{thm:densitybound} implies that $\SS_i$ is also a line through the origin, which is again a contradiction and we are done.
\end{proof}

\begin{rem} It is actually possible to find a uniform value of $R>0$ in this lemma, also independent of the flow $\SS_t$ (Tom Ilmanen, {\em personal communication}).
\end{rem}

\begin{lem}\label{remhot}
If {\bf{M1}} holds, there exist the limits $x_i=\lim_{t\to T}O^i(t)$, for
$i\in\{1,2,\dots,m\}$ and the set $\{x_i=\lim_{t\to T}O^i(t)~\vert~i=1,2,\dots,m\}$ is the union
of the set of the points $x$ in $\Omega$ where
$\widehat\Theta(x)>1$ with the set of the end--points of $\SS_t$ such that
the curve getting there collapses as $t\to T$.
\end{lem}
\begin{proof}
Let ${\mathcal{D}}=\{x\in\Omega~\vert~\widehat{\Theta}(x)>1\}$, 
${\mathcal{O}}(t)=\{O^1(t), O^2(t),\dots, O^m(t)\}$ and 
${\mathcal{P}}=\{P^1, P^2,\dots, P^l\}$. 
Let $R>0$ be given by the previous lemma and consider a finite subset
$\overline{\mathcal{D}}\subseteq{\mathcal{D}}$, supposing that the set
$$
{\mathcal I}_{\overline{\mathcal{D}}}\,=\,\bigl\{\,\tt\in[-1/2\log T,
+\infty)~\vert~\max_{x\in{\overline{\mathcal{D}}}}
d(x,{\mathcal{O}}(t(\tt)))
\geqslant R\sqrt{2(T-t(\tt))}\,\bigr\}
$$
has infinite Lebesgue measure, there must be $x_0\in\overline{\mathcal{D}}$ such that 
$$
{\mathcal I}_{x_0}\,=\,\bigl\{\,\tt\in[-1/2\log T, +\infty)~\vert~d(x_0,{\mathcal{O}}(t(\tt)))\geqslant R\sqrt{2(T-t(\tt))}\,\bigr\}
$$
has infinite Lebesgue measure.
Hence, by rescaling with Huisken's procedure around $x_0$, by
Proposition~\ref{resclimit-general}, we can extract a sequence of times $\tt_{j}\in {\mathcal I}_{x_0}$
 such that the rescaled networks ${\widetilde{\SS}_{x_0,\tt_j}}$
 converge in $C^1\loc$ to a line through the origin of $\R^2$, by Lemma~\ref{linelemma} (if the limit is the unit circle, the network is a closed curve and there is nothing to prove, as there are no $3$--points), since
 in any ball centered at the origin, there cannot be $3$--points, by the construction of
 ${\mathcal I}_{x_0}$ and holding {\bf{M1}}.
 This clearly implies that $\widehat{\Theta}(x_0)=1$, contradicting the
hypothesis $x_0\in\mathcal{D}$, hence, ${\mathcal
 I}_{\overline{\mathcal{D}}}$ must have finite Lebesgue measure.
It is then easy to see that this implies that the points of ${\overline{\mathcal{D}}}$ and thus of
${\mathcal{D}}$, cannot be more than the number $m$ of the $3$--points
of the evolving network $\SS_t$.\\
If now we consider a small $\delta>0$, as every point $x$ in the open set
$$
\Omega_\delta=\Omega\setminus\bigl\{x\in\Omega~\vert~d(x,{\mathcal{D}}\cup{\mathcal{P}})\leqslant\delta\bigr\}
$$
satisfies $\widehat{\Theta}(x)\leqslant 1$, by compactness and Corollary~\ref{regcol} (or White's local regularity
theorem in~\cite{white1}), it follows that the networks $\SS_{t_j}$ restricted to the set
$\Omega_\delta$ have uniformly bounded curvature and smoothly converge to a limit smooth network in $\Omega_\delta$ without $3$--points, otherwise at any of such $3$--points we would have a Gaussian density equal to $3/2$, larger than one.\\ This argument clearly implies that choosing $\delta$ small enough (as ${\mathcal{D}}\cup{\mathcal{P}}$ is finite), every $3$--point $O^i(t)$, for every $i\in\{1,2,\dots,m\}$, has to ``choose'' a point $x_i\in{\mathcal{D}}\cup{\mathcal{P}}$ to stay close and actually converges there.

Finally, if $x\in{\mathcal{D}}$, there must be a multi--point in any
blow--up limit shrinker, otherwise we can only have a line, by Lemma~\ref{linelemma} (the unit circle is excluded, as we said before), that would imply $\widehat{\Theta}(x)=1$, against the definition of ${\mathcal{D}}$. Hence, for some
$i\in\{1,2,\dots,m\}$ and $t_n\to T$ there must hold $O^i(t_n)\to x_i$
that forces $\lim_{t\to T}O^i(t)=x_i$, by the previous discussion..\\
If the curve of $\SS_t$ getting to an end--point $P^r$ collapses along a sequence of times $t_j\to T$, 
clearly, as before, for some $k\in\{1,2,\dots,m\}$ there must hold $O^k(t_j)\to
P^r=x_k$ and we have the same conclusion $\lim_{t\to T}O^k(t)=P^r=x_k$.
\end{proof}

\subsection{Regularity without vanishing of curves}\label{novan}
Let $T<+\infty$ be the maximal time of existence and 
assume that the lengths of all the curves of the network are
uniformly positively bounded from below, hence as $t\to T$ the maximum of the modulus of the curvature goes to $+\infty$. 
We are going to show that if {\bf{M1}} holds, $T$ cannot be a singular time, hence 
we conclude that this case simply cannot happen. 
This conclusion justifies the
title of this section: to have a singularity (assuming the
multiplicity--one conjecture) some curves must disappear.

Such result follows by the local regularity Theorem~\ref{thm:locreg.2} (precisely, by Corollary~\ref{regcol}, see also the first point of Remark~\ref{rem:locreg.1}), implying that the curvature is locally 
bounded around {\em every} point of $\overline{\Omega}$, as
$t\to T$. Indeed, performing a parabolic rescaling at any reachable,
interior point $x_0\in\Omega$ (at the other interior points of $\Omega$
the blow--up limits are empty, so $\widehat{\Theta}(x_0)=0$), since we assumed that the multiplicity--one
conjecture holds, by the discussion in Remark~\ref{enanrem}, 
we can obtain as blow--up limits only a straight lines with unit
multiplicity, so $\widehat{\Theta}(x_0)=1$, or standard triods, hence $\widehat{\Theta}(x_0)=3/2$. By Corollary~\ref{regcol}, we then conclude that 
the curvature is uniformly locally bounded along the flow, around
such point $x_0$.\\
If we instead rescale at an end--point $P^r$ we get a halfline and
this case can be treated as above by means of the ``reflection construction'' at the end of Section~\ref{geopropsub}. That is, for the flow ${\mathbb{H}}^r_t$ the point 
$P^r$ is no more an end--point and a blow--up there gives a
straight line, hence implying that the curvature is locally bounded
also around $P^r$, as before.\\
By the compactness of the set of reachable points ${\mathcal{R}}$, this argument clearly 
implies that the curvature of the whole $\SS_t$ 
is uniformly bounded, as $t\to T$, which is a contradiction.

\begin{prop}\label{regnocollapse}
Assuming {\bf{M1}}, if $T<+\infty$ is the maximal time of existence of the
curvature flow of a regular network with fixed end--points, then the inferior limit of the length of at
least one curve is zero, as $t\to T$.
\end{prop}

\begin{rem}\label{remlocglo}
Proposition~\ref{regnocollapse} can be seen as the global (in space) version of
the local regularity Theorem~\ref{thm:locreg.2} which deals with the situation of a single
$3$--point. Usually in analytic problems local and global (in space)
regularity coincides, actually in this case the tool to pass from one to the
other is the validity of the multiplicity--one conjecture.
\end{rem}

\bigskip

{\em In all the analysis of the following sections we will assume that {\bf{M1}} holds. Moreover, we assume that the bounded open set $\Omega$ is strictly convex.

We remark that, with minor modifications in the proofs, all the following results also hold for the flow of open networks in $\R^2$, ignoring the conclusions about the behavior at the end--points that are not present in such case.}

\bigskip

By the above discussion, we will have to analyze the behavior of the flow $\SS_t$ around the points $x_i=\lim_{t\to T}O^i(t)$, limits of the triple junctions in $\Omega$ (see Lemma~\ref{remhot}) where $\widehat\Theta(x)>3/2$ and the end--points of $\SS_t$ such that the curve getting there collapses, as $t\to T$. Notice that if a limit point $x_i$ is the limit of a single $3$--point $O^i(t)$, then the other ones must ``stay far'' and locally around $x_i$ there cannot be the collapse of a curve, then, by the same argument as above, we conclude that $\widehat\Theta(x_i)=3/2$. It follows that the only limit points $x_i\in\Omega$ we have to deal with are the ones which are limit of more than one triple junction, as $t\to T$.

\subsection{Limit networks with bounded curvature}\label{van0}
The analysis in this case consists in understanding the possible limit
networks that can arise, as $t\to T$, under the assumption that the
curvature is uniformly bounded along the flow. This to find
out how to continue the flow (if possible) 
as discussed in Section~\ref{resum}.

As we said, at least one curve of the network $\SS_t$ has to
``vanish'', approaching the singular time $T$. We 
show that in this case, as $t\to T$, assuming the multiplicity--one conjecture, $\SS_t$ converges to a {\em unique} limit degenerate regular network $\SS$, containing in the interior of $\Omega$ only regular triple junctions or $4$--points with four concurring curves whose exterior unit tangents form four angles of $120$, $60$, $120$ and $60$ degrees (any of them coming from two $3$--points going to ``collide'' each other along a single isolated collapsing curve) and at any end--point on $\partial\Omega$, either a regular single curve or two curves ``exiting'' from such end--point, forming an angle of $120$ degrees among them (coming from the single isolated collapse of the curve of the network getting there). The cores of such limit degenerate regular network are thus given only by the isolated collapsed curves.

We will see in Proposition~\ref{prop999} in the next section that {\em viceversa}, when locally only a single isolated curve collapses, the curvature stays bounded (see also the example in Proposition~\ref{type0prop}).

\begin{prop}\label{bdcurvcollapse} 
If {\bf{M1}} holds and $\SS_t=\bigcup_{i=1}^n\gamma^i([0,1],t)$ is the
curvature flow of a regular network in $\Omega$ with fixed end--points in a
maximal time interval $[0,T)$ such that the curvature is uniformly
bounded along the flow, then the networks $\SS_t$, up to
reparametrization proportional to arclength, converge in $C^1$ to some degenerate regular network $\widehat{\SS}_T=\bigcup_{i=1}^n\widehat{\gamma}^i_T([0,1])$ in $\Omega$, as $t\to T$.\\
The non--degenerate curves of $\widehat{\SS}_T$ belong to $C^1\cap W^{2,\infty}$ and they are smooth outside the multi--points. 
Moreover, denoting with $\SS_T$ the $C^1$ network described by the family of the non--degenerate curves of $\widehat{\SS}_T$, every multi--point of the $\SS_T$ is either a regular triple junction or an end--point of $\SS_t$ or 
\begin{itemize}
\item a $4$--point where the four concurring curves have opposite exterior unit tangent vectors in pairs and form angles of $120/60$ degrees between them --
{\em collapse of a curve in the ``interior'' of $\SS_t$},
\item a $2$--point at an end--point of the network $\SS_t$ where the
 two concurring curves form an angle of $120$
degrees among them -- {\em collapse of the curve getting to such
 end--point of $\SS_t$}.
\end{itemize}
\begin{figure}[H]
\begin{center}
\begin{tikzpicture}[rotate=90,scale=0.75]
\draw[color=black!40!white, shift={(0,-2.8)}]
(-0.05,2.65)to[out= -90,in=150, looseness=1] (0.17,2.3)
(0.17,2.3)to[out= -30,in=100, looseness=1] (-0.12,2)
(-0.12,2)to[out= -80,in=40, looseness=1] (0.15,1.7)
(0.15,1.7)to[out= -140,in=90, looseness=1.3](0,1.1)
(0,1.1)--(-.2,1.35)
(0,1.1)--(+.2,1.35);
\draw[color=black]
(-0.05,2.65)to[out= 30,in=180, looseness=1] (2,3)
(-0.05,2.65)to[out= -90,in=150, looseness=1] (0.17,2.3)
(-0.05,2.65)to[out= 150,in=-20, looseness=1] (-2,3.3)
(0.17,2.3)to[out= -30,in=100, looseness=1] (-0.12,2)
(-0.12,2)to[out= -80,in=40, looseness=1] (0.15,1.7)
(0.15,1.7)to[out= -140,in=90, looseness=1](0,1.25)
(0,1.25)to[out= -30,in=180, looseness=1] (1.9,0.7)
(0,1.25)to[out= -150,in=-15, looseness=1] (-1.9,1.2);
\draw[color=black,dashed]
(-2,3.3)to[out= 160,in=-20, looseness=1](-2.7,3.5)
 (-1.9,1.2)to[out= 165,in=-15, looseness=1](-2.7,1.3)
(1.9,0.7)to[out= 0,in=-160, looseness=1] (2.6,0.9)
(2,3)to[out= 0,in=160, looseness=1] (2.8,2.9);
\draw[color=black!40!white,shift={(0,-6)}]
(0,2.65)--(1.73,3.65)
(0,2.65)--(1.73,1.65)
(0,2.65)--(-1.73,3.65)
(0,2.65)--(-1.73,1.65);
\draw[color=black,shift={(0,-6)}]
(0,2.65)to[out= -30,in=180, looseness=1] (1.9,2)
(0,2.65)to[out= -150,in=-15, looseness=1] (-1.9,2.3)
(0,2.65)to[out= 30,in=180, looseness=1] (2.2,3.3)
(0,2.65)to[out= 150,in=-20, looseness=1] (-2.2,3.1);
\draw[color=black,dashed,shift={(0,-6)}]
(-2.2,3.1)to[out= 160,in=-20, looseness=1](-3,3.3)
 (-1.9,2.3)to[out= 165,in=-15, looseness=1](-2.7,2.4)
(1.9,2)to[out= 0,in=-160, looseness=1] (2.6,2.2)
(2.2,3.3)to[out= 0,in=160, looseness=1] (3,3.2);
\end{tikzpicture}\qquad\qquad\qquad
\begin{tikzpicture}[rotate=90,scale=0.75,shift={(15,0)}]
\draw[color=black!40!white, shift={(0,-3)}]
(-0.05,2.65)to[out= -90,in=150, looseness=1] (0.17,2.3)
(0.17,2.3)to[out= -30,in=100, looseness=1] (-0.12,2)
(-0.12,2)to[out= -80,in=40, looseness=1] (0.15,1.7)
(0.15,1.7)to[out= -140,in=90, looseness=1.3](0,1.1)
(0,1.1)--(-.2,1.35)
(0,1.1)--(+.2,1.35);
\draw[color=black]
(-0.05,2.65)to[out= 30,in=180, looseness=1] (2,3)
(-0.05,2.65)to[out= -90,in=150, looseness=1] (0.17,2.3)
(-0.05,2.65)to[out= 150,in=-20, looseness=1] (-2,3.3)
(0.17,2.3)to[out= -30,in=100, looseness=1] (-0.12,2)
(-0.12,2)to[out= -80,in=40, looseness=1] (0.15,1.7)
(0.15,1.7)to[out= -140,in=90, looseness=1](0,1.25);
\draw[color=black,dashed]
(-2,3.3)to[out= 160,in=-20, looseness=1](-2.7,3.5)
(2,3)to[out= 0,in=160, looseness=1] (2.8,2.9);
\draw[color=black!40!white]
(-3,1.7)to[out= -30,in=-180, looseness=1] (0,1.25)
(+3,1.7)to[out= -150,in=0, looseness=1] (0,1.25);
\path[font=\small]
(0,.2) node[left]{$P^r$};
\path[font=\Large]
(2.6,.5) node[left]{$\Omega$};
\draw[color=black!40!white,shift={(0,-6)}]
(0,2.65)--(1.73,3.65)
(0,2.65)--(-1.73,3.65);
\draw[color=black,shift={(0,-6)}]
(0,2.65)to[out= 30,in=180, looseness=1] (2.2,3.3)
(0,2.65)to[out= 150,in=-20, looseness=1] (-2,3.3);
\draw[color=black,dashed,shift={(0,-6)}]
(-2,3.3)to[out= 160,in=-20, looseness=1](-2.7,3.5)
(2.2,3.3)to[out=0,in=-60, looseness=1] (2.8,3.7);
\draw[color=black!40!white,shift={(0,-6)}]
(-3,3.1)to[out= -30,in=-180, looseness=1] (0,2.65)
(+3,3.1)to[out= -150,in=0, looseness=1] (0,2.65);
\fill(0,1.25) circle (2pt);
\fill(0,-3.35) circle (2pt);
\path[font=\small]
(0,-4.4) node[left]{$P^r$};
\path[font=\Large]
(2.6,-4.1) node[left]{$\Omega$};
\end{tikzpicture}
\end{center}
\begin{caption}{Collapse of a curve in the interior and at
 an end--point of $\SS_t$.\label{Pcollapse}}
\end{caption}
\end{figure}
\end{prop}

\begin{defn}\label{reg4} By their clear importance, we call {\em regular $4$--points} the ones like in this proposition.
\end{defn}

\begin{proof} By Proposition~\ref{brakprop}, since $\SS_t$ is the curvature 
flow of a regular network, there exist the limits of the lengths of the curves $L^i(T)=\lim_{t\to T}L^i(t)$, for every $i\in\{1,2,\dots, n\}$. 
Moreover every limit of $\SS_t$ is a connected, bounded subset of $\R^2$.\\
Recalling the third point of Remark~\ref{oprob2} (or directly Corollary~\ref{parareg0}), we reparametrize the networks so that the flow is a special smooth flow, then, by Remark~\ref{klstimarem}, all the velocities ${\gamma}_t^i$ are uniformly bounded in space and time by some constant $D$, hence we have, 
$$
\vert{\gamma}^i(x,t)-{\gamma}^i(x,\overline{t})\vert\leqslant
\int_t^{\overline{t}}\vert{\gamma}_t^i(x,\xi)\vert\,d\xi\leqslant D\vert t-\overline{t}\vert\,
$$
uniformly for any $x\in[0,1]$ and every pair $t,\overline{t}\in[0,T)$. This clearly implies that ${\gamma}^i(\cdot, t):[0,1]\to\R^2$ is a Cauchy sequence in $C^0([0,1])$, hence the network $\SS_t$ converges uniformly to a limit family of continuous curves ${\gamma}^i_T:[0,1]\to \R^2$, as $t\to T$, composing the set ${\SS}_T=\bigcup_{i=1}^n{\gamma}^i_T([0,1])$. As the curvature and the total length of all $\SS_t$ are uniformly bounded by some constant $C$, reparametrizing instead all the curves $\gamma^i(\cdot,t)$ proportionally to their 
arclength, getting the maps $\widehat{\gamma}^i(\cdot,t):[0,1]\to\R^2$, these latter are a
family of uniformly Lipschitz curves with curvature uniformly bounded in space and
time, hence relatively compact in $C^1$. It is then easy to see that for every $C^1$--converging subsequence, the curves $\widehat{\gamma}^i_T:[0,1]\to\R^2$ of any limit family $\widehat{\SS}_T=\bigcup_{i=1}^n\widehat{\gamma}^i_T([0,1])$ have the same supports of the curves ${\gamma}^i_T:[0,1]\to \R^2$ and either are {\em constant} (the limits of collapsing curves) or are also parametrized proportionally to arclength. Hence, this argument implies that the whole family of curves composing the networks $\SS_t$, reparametrized proportionally to arclength, converges in $C^1$, as $t\to T$, to the family $\widehat{\gamma}^i_T$ composing $\widehat{\SS}_T$. Clearly, by the uniform bound on the curvature, all the curves $\widehat{\gamma}^i_T$ belong to $W^{2,\infty}$ and, by Lemma~\ref{boh}, they are smooth outside the multi--points.\\
According to Definition~\ref{degconv}, we have to deal now with the convergence of the unit tangent vectors. We observe that if we denote with $s$ the arclength parameter, we have
\begin{equation}\label{taudeg}
\biggl\vert\frac{\partial \widehat{\tau}^i(x,t)}{\partial x}\biggr\vert=
\biggl\vert\frac{\partial \tau^i(s,t)}{\partial s}\biggr\vert L^i(t)=
\vert k^i(s,t)\vert L^i(t)\leqslant CL^i(t)\leqslant C^2\,,
\end{equation}
for some constant $C$, hence, every sequence of times $t_n\to T$ have a -- not relabeled --
subsequence such that the maps $\widehat{\tau}^i(\cdot,t_n)$
converge uniformly to some maps $\widehat{\tau}_T^i$.\\
If the curve $\widehat{\gamma}^i_T$ is a regular curve (that is, $L^i(t)$ does
not go to zero), it is easy to see that the limit maps $\widehat{\tau}_T^i$ must
coincide with the unit tangent vector field to
the curve $\widehat{\gamma}^i_T$, hence, the full sequence
$\widehat{\tau}^i(\cdot,t)$ converges to $\widehat{\tau}^i_T$.\\
If $L^i(t)$ converges to zero, as $t\to T$, by
inequality~\eqref{taudeg}, the maps $\widehat{\tau}^i(\cdot,t_n)$
converge to a constant unit vector $\widehat{\tau}_T^i$ which, if it is independent of
the subsequence $t_n$, it will be the ``assigned'' constant unit vector
to the degenerate constant curve $\widehat{\gamma}^i_T$ of the network $\widehat{\SS}_T$, as in Definition~\ref{degnet}, then it follows (see Remark~\ref{remreg}) that $\widehat{\SS}_T$ 
is a degenerate regular network and that $\SS_t$ converges in $C^1$ to $\widehat{\SS}_T$, 
as $t\to T$.\\
We start dealing with the behavior of the curves without end--points on the boundary of $\Omega$.\\
If a region is ``collapsing'', that is, its area is going to zero, as $t\to T$, being $\Omega$ strictly convex, we have that the region must be completely ``inside'' $\Omega$ (not bounded by curves getting to the end--points of the networks $P^r$ which are all distinct, hence a ``collapse'' on $\partial\Omega$ is impossible by the strong maximum principle) and, by the computations in Section~\ref{geopropsub}, it can have at most $m\leqslant 5$ bounding curves $\gamma^\ell(\cdot,t)$ and its area satisfies
$$
A(t)=(2-m/3)\pi(T-t)/2\,,
$$
by equation~\eqref{areaevolreg}. By Lemma~\ref{remhot} the $3$--points of the region converge to some limit points, as $t\to T$, if these limits are not all coincident with a single point $x_0\in\Omega$, the limit family of $C^1$ curves ${\gamma}^i_T$ must bound a ``region'' with zero area not given by a single point, hence there would be at least two non--degenerate (non--constant) curves with the same support, which is forbidden by the multiplicity--one conjecture ${\bf {M1}}$. Hence, we conclude that all such $3$--points converge to the point $x_0\in\Omega$ and the whole region vanishes at $x_0$, as $t\to T$. In particular, all the lengths of the bounding curves $\gamma^\ell(\cdot,t)$ also go to zero, as $t\to T$. Since, by equation~\eqref{gb} we have
$$
\sum_{\ell=1}^m\int_{\gamma^\ell_t} k\,ds=(2-m/3)\pi>\pi/3\,,
$$
it follows that we have a contradiction with the fact that the curvature is bounded and the perimeter of the region goes to zero. Hence, with bounded curvature, which is our case, no regions can collapse, which implies that around every point the network is locally a tree, as $t$ gets close to $T$. Recalling now Lemma~\ref{remhot}, we only have to check things locally around every point $x_0$ which belongs to the set of the limits of the triple junctions $\{O^j(t)\}$, as $t\to T$, since outside such (finite) set the network converges smoothly to $\widehat{\SS}_T$ (which is composed of regular smooth curves there), by Lemma~\ref{boh}. If the point $x_0$ is the limit point of a single triple junction $O^j(t)$, clearly locally around $x_0$ no curve is collapsing and the convergence of $\SS_t$ to $\widehat{\SS}_T$ is smooth (see the comments at the end of the previous section). Assuming then that the curve $\gamma^i(\cdot, t)$ (at least) collapses with its end--points going to $x_0$ and performing the Huisken's rescaling procedure at $x_0$, we can only get as blow--up limit degenerate shrinkers which are trees with zero curvature (being bounded, by the rescaling, the curvature converges uniformly to zero). Moreover, these shrinkers have unit multiplicity since we assumed {\bf{M1}}, hence they must be among the ones of Lemma~\ref{lemmatree}: a line, a standard triod or a standard cross. The first two cases are clearly excluded, since it would hold $\widehat{\Theta}(x_0)\leqslant 3/2$, then Corollary~\ref{regcol} would tell us that the flow is locally smooth and there is no collapse of curves. Hence, the only possibility is a standard cross (which has a core composed only of a collapsed curve), this actually implies that at $x_0$ there are no other collapsing curves other than $\gamma^i(\cdot,t)$ and only its end--points (among the triple junctions) are converging to $x_0$. Indeed, arguing as in Corollary~\ref{regcol}, since there holds
\begin{equation}\label{Oconv2}
\vert O^j(t_2)-O^j(t_1)\vert=\biggl\vert\int_{t_1}^{t_2}v^j(\xi)\,d\xi\,\biggl\vert\leqslant
\int_{t_1}^{t_2}\vert v^j(\xi)\vert\,d\xi\leqslant D\vert t_2-t_1\vert\,,
\end{equation}
for every triple junction $O^j(t)$ converging to $x_0$, for every $t_1,t_2\in[0,T)$, hence
$$
\vert O^j(t)-x_0\vert\leqslant D\vert T-t\vert
$$
for every $t\in[0,T)$, we have that its image $\widetilde{O}^j(\tt)$, after performing Huisken's rescaling procedure, satisfies
$$
\vert\widetilde{O}^j(\tt)\vert=\frac{\vert O^j(t(\tt)-x_0\vert}{\sqrt{2(T-t(\tt))}}\leqslant\frac{D\vert T-t(\tt)\vert}{\sqrt{2(T-t(\tt))}}=D\sqrt{(T-t(\tt))/2}\,,
$$
which tends to zero, as $\tt\to+\infty$, in particular, all the triple junctions converging to $x_0$ cannot ``disappear'' in the limit degenerate regular shrinker (going to infinity). As the core of the standard cross is a single line (the underlying graph has only two triple junctions), the above claim follows and the collapsing curve $\gamma^i(\cdot,t)$ is ``isolated''. As a consequence, around $x_0$ the curve $\gamma^i(\cdot,t)$ collapses and other four curves $\gamma^\ell(\cdot,t)$ converge in $C^1$ (smoothly outside $x_0$), as $t\to T$, to four regular curves $\gamma_T^\ell$ with an end--point at $x_0$, forming a $4$--point. By the $C^1$--convergence of the four curves and the $120$ degrees condition at the two converging triple junctions, if for a sequence $t_n\to T$ we have that $\widehat{\tau}^i(\cdot,t_n)$ converge to a constant unit vector $\widehat{\tau}_T^i$, this unit vector is uniquely determined by the (unique) exterior unit tangents at $x_0$ of the four concurring curves, hence we conclude that $\widehat{\tau}_T^i$ it is independent of the sequence $t_n\to T$, as we wanted to show. Then, it is the ``assigned'' constant unit vector to the degenerate constant curve $\widehat{\gamma}^i_T$ of the network $\widehat{\SS}_T$, as in Definition~\ref{degnet}. It follows (see Remark~\ref{remreg}) that $\widehat{\SS}_T$ is a degenerate regular network and that $\SS_t$ converges in $C^1$ to $\widehat{\SS}_T$, as $t\to T$.\\
By this argument, we can also conclude that $x_0$ is a $4$--point of $\widehat{\SS}_T$ (or of $\SS_T$) where the 
four concurring curves have opposite unit tangents in pairs and form 
angles of $120/60$ degrees between them, as in the statement.\\
Finally, in the case of a collapsing curve arriving at an end--point $P^r$ of
$\SS_t$, we get the statement of the proposition by considering the
network ${\mathbb{H}}_t^r$, obtained by the union of $\SS_t$ with its ``reflection'' with
respect to the point $P^r$ (see the end of Section~\ref{geopropsub}) and applying the previous conclusions to
such network.
\end{proof}

The next corollary follows from this proof.

\begin{cor}\label{curvcore}
Every core (there could be more than one) of $\widehat{\SS}_T$ is composed of a single collapsed curve. In the case of bounded curvature, only ``isolated'' curves can collapse.
\end{cor}

Moreover, during the proof, we also showed the following intuitive fact about a collapsing region, that is, with its area is going to zero, as $t\to T$. 

\begin{lem}\label{regcollap} If {\bf{M1}} holds and a region is collapsing as $t\to T$, then the curvature of the network cannot be bounded.
\end{lem}

\begin{rem}\label{angolostretto}
Notice that if at an end--point the two curves of the boundary of the
convex set $\Omega$ form an angle (or 
the whole network is contained in an angle whose vertex is such end--point)
with amplitude less than $120$ degrees, then the collapse situation described in Proposition~\ref{bdcurvcollapse} cannot
happen at such end--point. This is, for instance, the case of an initial triod contained in a triangle with angles less than $120$ degrees and fixed end--points in the vertices.\\
The same conclusion holds, by the argument in the proof of Proposition~\ref{omegaok2}, calling $\Om_t\subseteq \Om$ the evolution by curvature of $\partial \Om$, keeping fixed the end--points of $\SS_t$, if the angle formed by $\Om_t$ at such end--point becomes smaller than $120$ degrees.
\end{rem}

\begin{rem}
Notice that, even if $\SS_T$ is smooth outside its multi--points and $W^{2,\infty}$, we cannot say at the moment that its curves are of class $C^2$. This will be actually a consequence of the analysis of the next section, see Theorem~\ref{tree-bdcurb2} and Remark~\ref{rem555}. 
\end{rem}
 
All the previous arguments can be easily localized and we have the following conclusion.

\begin{prop}\label{noloop}
If {\bf {M1}} holds and the curvature of $\SS_t$ is locally uniformly bounded around a point $x_0\in\overline{\Omega}$, as $t\to T$, the networks $\SS_t$, up to
reparametrization, converge in $C^1\loc$ locally around $x_0$ to some degenerate regular network $\widehat{\SS}_T$ whose non--degenerate curves form a $C^1$ network $\SS_T$, having possibly some non--regular multi--points which are among the ones described in Proposition~\ref{bdcurvcollapse}.\\
Moreover, the curves of $\SS_T$ belong to $C^1\cap W^{2,\infty}$, in a neighborhood of $x_0$, 
and are smooth outside the multi--points. 
\end{prop}

\begin{rem}\label{type0} Referring to Remark~\ref{T1}, we can call these singularities with bounded curvature {\em Type~0} singularities. They are peculiar to the network flow, as they cannot appear in the motion by curvature of a single curve. 
\end{rem}

\subsection{Vanishing of curves with unbounded curvature}\label{van}
Suppose now that, as $t\to T$, the curvature is not bounded and the length of at least one curve of the flow $\SS_t$ is not positively bounded from below. This last case is the most delicate.\\
Performing, as before, any of the blow--up procedures, even assuming the
multiplicity--one conjecture, there can be several shrinkers as
possible blow--up limits given by Propositions~\ref{thm:shrinkingnetworks.1},~\ref{resclimit-general}
and we need to classify them in order to understand the behavior of the flow $\SS_t$ approaching the singular time $T$. 
In doing that, the (local) structure (topology) of the evolving network plays an
important role in the analysis since it restricts the family of possible
shrinkers obtained as blow--up limits of $\SS_t$. A very relevant case is when the evolving network has no loops, namely, it is a tree, studied in detail in~\cite{mannovplu2}).

\begin{prop}\label{prop999}
If {\bf{M1}} holds and the evolving regular network $\SS_t$ is a tree in a neighborhood of $x_0\in\overline{\Omega}$, for $t$ close enough to $T$, then the curvature of $\SS_t$ is locally uniformly bounded around $x_0$, during the flow. Hence, the conclusions of Proposition~\ref{noloop} apply.
\end{prop}
\begin{proof}
Let $\mathbb{S}_t$ be a smooth flow in the maximal time interval $[0,T)$
of the initial network $\mathbb{S}_0$. Let $x_0\in\overline{\Omega}$ be a reachable point for the flow and let $B$ be a ball containing $x_0$ where $\SS_t$ is a tree, for $t$ close enough to $T$ (we clearly only need to consider reachable points).\\
Let us consider a sequence of parabolically rescaled curvature flows $\SS^{\mu_i}_\tt$ around $(x_0,T)$, as in Proposition~\ref{thm:shrinkingnetworks.1}. Then, as $i\to\infty$, it converges to a degenerate regular self--similarly shrinking network flow $\SS^\infty_\tt$, in $C^{1,\alpha}\loc \cap W^{2,2}\loc$, for almost all $\tt\in (-\infty, 0)$ and for any $\alpha \in (0,1/2)$.\\
Thanks to the multiplicity--one hypothesis {\bf{M1}}
and to the topology of the network (locally a tree, see Lemma~\ref{lemmatree}),
if we suppose that $x_0\not\in\partial\Omega$, then $\SS^\infty_\tt$ can only be the ``static'' flow given by:
\begin{itemize}
\item a straight line;
\item a standard triod;
\item four concurring halflines with opposite unit tangent vectors in pairs, forming angles of $120/60$ degrees between them, that is, a standard cross.
\end{itemize}
By White's local regularity theorem in~\cite{white1},
if the sequence of rescaled curvature flows converges to a straight line, the curvature is
uniformly bounded for $t\in[0,T)$ in a ball around the point $x_0$.
Thanks to Theorem~\ref{thm:locreg.2} the same holds in the case of the standard triod. Hence, the only situation we have to deal with to complete the proof in this case is the collapse of two triple junctions at a point of $\Omega$, when the limit flow is given by the static degenerate regular network composed of four concurring halflines with opposite unit tangents in pairs forming angles of $120/60$ degrees between them, a standard cross. We claim that also in this case the curvature is locally uniformly 
bounded during the flow, around the point $x_0$ (the next proposition and lemmas are devoted to prove this fact).

If instead $x_0\in\partial\Omega$, the only two possibilities for $\SS^\infty_\tt$ are the static flows given by:
\begin{itemize}
\item a halfline;
\item two concurring halflines forming an angle of $120$ degrees.
\end{itemize}
For both these two situation the thesis is obtained by going back to the case in which $x_0\in\Omega$, with the ``reflection construction'' we described at the end of Section~\ref{geopropsub}.
\end{proof}

\begin{rem}\label{possibiliblowuptree}
Obviously, the conclusion of this proposition holds when $\SS_0$ is a tree (globally), since it remains so during the flow.
\end{rem}

\begin{prop}\label{cross}
Let $\mathbb{S}_t$ be a smooth flow in the maximal time interval $[0,T)$ for the initial network 
$\mathbb{S}_0$.
Let $x_0$ be a reachable point for the flow such that 
the sequence of rescaled curvature flows $\SS^{\mu_i}_\tt$ around $(x_0,T)$, as in Proposition~\ref{thm:shrinkingnetworks.1}, as $i\to\infty$, converges, in $C^{1,\alpha}\loc \cap W^{2,2}\loc$, for almost all $\tt\in (-\infty, 0)$ and for any $\alpha \in (0,1/2)$, to a limit degenerate static flow $\SS^\infty_\tt$ given by a standard cross. 
Then,
$$
\vert k(x,t)\vert\leqslant C<+\infty
$$
for all $t\in[0,T)$ and $x$ in a neighborhood of $x_0$.
\end{prop}

We briefly outline the proof of this proposition. First, in Lemma~\ref{kinfty} and~\ref{kappa2}, we show that for any tree, if we assume a uniform control on the motion of its end--points, the $L^2$--norm of its curvature is uniformly bounded in a time interval depending on its initial value. Moreover, we also bound the $L^\infty$--norm of the curvature in terms of its $L^2$--norm and of the $L^2$--norm of its derivative.\\
Then, we prove that for a special tree, composed of only five curves, two triple junctions and four end--points on the boundary of $\Omega$ open, convex and regular (see Figure~\ref{tree}), uniformly controlling, as before, its end--points and the lengths of the ``boundary curve'' from below, the $L^2$--norm of $k_s$ is bounded until $\Vert k\Vert_{L^2}$ stays bounded. The statement of the proposition will follow by localizing these estimates.

\begin{lem}\label{kinfty}
Let $\Omega$ be a convex open regular set and $\mathbb{S}_0$ a tree
with end--points $P^1, P^2,\dots, P^l$ (not necessarily fixed during its motion) on $\partial\Omega$. Let $\SS_t$ be a smooth evolution by curvature for $t\in [0,T)$ of the network $\SS_0$ such that the square of the curvature at the end--points of $\SS_t$ is uniformly bounded in time by some constant $C$. Then,
\begin{equation}\label{inkappa}
\|k\|_{L^\infty}^2 \leqslant 4^{n-1}C+D_n\|k\|_{L^2} \|k_s\|_{L^2}\,,
\end{equation}
where $n\in\NN$ is such that for every point $Q\in\SS_0$ there is a path to get from $Q$ to an end--point passing by at most $n$ curves (clearly, $n$ is smaller than the total number of curves of $\SS_0$) and the constant $D_n$ depends only on $n$. 
\end{lem}
\begin{proof}
Let us first consider a network $\SS_0$ with five curves, 
two triple junctions $O^1, O^2$ and
four end--points $P^1, P^2, P^3, P^4$. In this case $n$ is clearly equal to two.
We call $\gamma^i$, for $i\leqslant 4$, the curve connecting $P^i$ with one of the two triple junctions and $\gamma^5$ the curve connecting the two triple junctions (see the following Figure~\ref{tree}).
\begin{figure}[H]
\begin{center}
\begin{tikzpicture}[scale=1]
\draw
 (-3.73,0) 
to[out= 50,in=180, looseness=1] (-2.3,0.7) 
to[out= 60,in=180, looseness=1.5] (-0.45,1.55) 
(-2.3,0.7)
to[out= -60,in=130, looseness=0.9] (-1,-0.3)
to[out= 10,in=100, looseness=0.9](0.1,-0.8)
(-1,-0.3)
to[out=-110,in=50, looseness=0.9](-2.7,-1.7);
\draw[color=black,scale=1,domain=-3.141: 3.141,
smooth,variable=\t,shift={(-1.72,0)},rotate=0]plot({2.*sin(\t r)},
{2.*cos(\t r)});
\path[font=\small]
 (-3.73,0) node[left]{$P^1$}
 (-2.9,-1.7)node[below]{$P^2$}
 (0.1,-0.8)node[right]{$P^3$} 
 (-0.40,1.6) node[right]{$P^4$}
 (-3,0.6) node[below] {$\gamma^1$}
 (-1.5,1.3) node[right] {$\gamma^4$}
 (-1.1,-1.2)[left] node{$\gamma^2$}
 (0,-0.8)[left] node{$\gamma^3$}
 (-1.3,0.5)[left] node{$\gamma^5$}
 (-2.45,1.3) node[below] {$O^1$}
 (-1.4,-0.1) node[below] {$O^2$}; 
\end{tikzpicture}
\end{center}
\begin{caption}{A tree--like network with five curves.\label{tree}}
\end{caption}
\end{figure}
Fixed a time $t\in[0,T)$, let $Q\in\gamma^i\subseteq\SS_t$, for some $i\leqslant 4$. We compute
$$
[k^i(Q)]^2=[k^i(P^i)]^2+ 2 \int_{P^i}^Q k k_s ds \leqslant C + 2 \|k\|_{L^2} \|k_s\|_{L^2}\,,
$$
hence, for every $Q\in\SS_t\setminus\gamma^5$ we have
$$
[k^i(Q)]^2\leqslant C + 2 \|k\|_{L^2} \|k_s\|_{L^2}\,.
$$
Assume now instead that $Q\in \gamma^5$. Recalling that 
$\sum_{i=1}^3 k^i =0$ at each triple junction, by the previous argument we have $[k^i(O^1)]^2, [k^i(O^2)]^2\leqslant C + 2 \|k\|_{L^2} \|k_s\|_{L^2}$, for all $i\in\{1, 2, 3, 4\}$, then it follows that $[k^5(O^1)]^2, [k^5(O^2)]^2\leqslant 4C + 8\|k\|_{L^2} \|k_s\|_{L^2}$.
Hence, arguing as before, we get
$$
[k^5(Q)]^2 = [k^5(O^1)]^2 + 2 \int_{O^1}^Q k k_s\, ds 
\leqslant 4C + 8\|k\|_{L^2} \|k_s\|_{L^2} + 2 \int_{O^1}^Q k k_s\, ds\,,
$$
In conclusion, we get the uniform in time inequality for $\SS_t$
\begin{equation*}
\|k\|_{L^\infty}^2 \leqslant 4C + 10\|k\|_{L^2} \|k_s\|_{L^2}.
\end{equation*}
In the general case, since $\SS_t$ are all trees homeomorphic to $\SS_0$, we can argue similarly to get the conclusion by induction on $n$.
\end{proof}

\begin{lem}\label{kappa2}
Let $\Omega\subseteq\mathbb{R}^2$ be open, convex and regular, 
let $\mathbb{S}_0$ be a tree with end--points $P^1, P^2,\dots, P^l$ on $\partial\Omega$
that satisfy assumption~\eqref{endsmooth} 
and let $\SS_t$ for $t\in [0,T)$ be a smooth evolution by curvature of the network $\SS_0$.
Then $\Vert k\Vert_{L^2}^2$ is uniformly bounded on $[0,\widetilde{T})$ by $\sqrt{2}\bigl[\Vert k(\cdot,0)\Vert_{L^2}^2+1\bigr]$, where 
$$
\widetilde{T}=\min\,\Bigl\{ T, 1\big/
8C\,\bigl(\Vert k(\cdot,0)\Vert_{L^2}^2+1\bigr)^2\Bigr\}\,.
$$
Here the constant $C$ depends only on the number $n\in\NN$ of Lemma~\ref{kinfty} and the constants in assumption~\eqref{endsmooth}. 
\end{lem}
\begin{proof}
By inequality~\eqref{ksoltanto} we have
\begin{align}
\frac{d\,}{dt}\int_{\SS_t} k^2\, ds 
\leqslant&\, -2\int_{\SS_t} k_s^2\, ds + \int_{\SS_t} k^4\, ds\nonumber
+ \sum_{p=1}^m\sum_{i=1}^3\lambda^{pi}\left( k^{pi}\right) ^2\,\biggr\vert_{\text{{ at the $3$--point $O^p$}}}+C\nonumber\\
\le&\,-2\int_{\SS_t} k_s^2\,ds + \|k\|_{L^\infty}^2\int_{\SS_t}k^2\,ds
+ C \|k\|_{L^\infty}^3+C\,.\label{kappat}
\end{align}
By estimate~\eqref{inkappa} and the Young inequality, we then obtain
\begin{eqnarray*}
\|k\|_{L^\infty}^3 &\le& C_n + C_n \|k\|_{L^2}^\frac 32 \|k_s\|_{L^2}^\frac 32 \leqslant C_n+ \eps \|k_s\|_{L^2}^2 + C_{n,\eps} \|k\|_{L^2}^6\,,\\
\|k\|_{L^\infty}^2\|k\|_{L^2}^2 &\le& C_n \|k\|_{L^2}^2 + D_n\|k\|_{L^2}^3 \|k_s\|_{L^2} \leqslant C_n \|k\|_{L^2}^2+\eps \|k_s\|_{L^2}^2 + C_{n,\eps} \|k\|_{L^2}^6\,,
\end{eqnarray*}
for every small $\varepsilon>0$ and a suitable constant 
$C_{n,\varepsilon}$.\\
Plugging these estimates into inequality~\eqref{kappat} we get
\begin{align}
\frac{d\,}{dt}\int_{\SS_t} k^2 ds 
\leqslant &\,-2\Vert k_s\Vert^2+ \|k\|_{L^\infty}^2\Vert k\Vert^2+ C \|k\|_{L^\infty}^3+C\nonumber\\
\leqslant &\,-2\Vert k_s\Vert^2+ C_n \|k\|_{L^2}^2+
 \eps \|k_s\|_{L^2}^2 + C_{n,\eps} \|k\|_{L^2}^6
+C_n+ \eps \|k_s\|_{L^2}^2 + C_{n,\eps} \|k\|_{L^2}^6+C_n\nonumber\\
\leqslant&\, C\Bigl( \int_{\SS_t} k^2 ds\Bigr)^3+C\,,\label{kappatt}
\end{align}
Where we chose $\eps=1/2$ and the constant $C$ depends only on the number $n\in\NN$ of Lemma~\ref{kinfty} and the constants in condition~\eqref{endsmooth}. 

Calling $y(t)= \int_{\SS_t} k^2\, ds+1$, we can rewrite inequality~\eqref{kappatt} as the differential ODE 
$$
y'(t)\leqslant 2Cy^3(t)\,,
$$
hence, after integration, we get
$$
y(t)\leqslant \frac{1}{\sqrt{\frac{1}{y^2(0)}-4Ct}}
$$
and, choosing $\widetilde{T}$ as in the statement, the conclusion is straightforward.
\end{proof}

\begin{lem}\label{kappaesse2}
Let $\Omega\subseteq\mathbb{R}^2$ be open, convex and regular, 
let $\mathbb{S}_0$ be a tree with five curves, 
two triple junctions $O^1, O^2$ and
four end--points $P^1, P^2, P^3, P^4$ on $\partial\Omega$, as in Figure~\ref{tree}, satisfying assumption~\eqref{endsmooth} and assume that $\SS_t$, for $t\in [0,T)$, is a smooth evolution by curvature of the network $\SS_0$ such that $\Vert k\Vert_{L^2}$ is uniformly bounded on 
$[0,T)$.\\
If the lengths of the curves of the network arriving at the end--points are uniformly bounded below by some constant $L>0$, then $\Vert k_s\Vert_{L^2}$ is uniformly bounded on $[0,T)$.
\end{lem}

\begin{proof}
We first estimate $\Vert k_s\Vert^2_{L^\infty}$ in terms of $\|k_s\|_{L^2}$ and $\|k_{ss}\|_{L^2}$.\\
Fixed a time $t\in[0,T)$, let $Q\in\gamma^i\subseteq\SS_t$, for some $i\leqslant 4$. We compute
\begin{eqnarray*}
[k_s^i(Q)]^2 &=& [k_s^i(P^i)]^2 + 2 \int_{P^i}^Q k_s k_{ss}\,ds \leqslant C + 2 \|k_s\|_{L^2} \|k_{ss}\|_{L^2}\,,
\end{eqnarray*}
hence, in this case, 
$$
[k_s^i(Q)]^2 \leqslant C + 2 \|k_s\|_{L^2} \|k_{ss}\|_{L^2}\,,
$$
for every $Q\in\SS_t\setminus\gamma^5$.\\
Assume now instead that $Q\in \gamma^5$. Recalling that 
$k^i_s+\lambda^ik^i=k^j_s+\lambda^jk^j$ at each triple junction, we get 
$$
k^5_s(O^1)=k^i_s(O^1)+\lambda^i(O^1)k^i(O^1)-\lambda^5(O^1)k^5(O^1)\,,
$$
hence,
\begin{align*}
\vert k^5_s(O^1)\vert&\leqslant \vert k^i_s(O^1)\vert +C\Vert k\Vert^2_{L^\infty}\\
&\leqslant \vert k^i_s(O^1)\vert +C\Vert k\Vert_{L^2}\Vert k_s\Vert_{L^2}+C\\
&\leqslant \vert k^i_s(O^1)\vert +C\left(1+\Vert k_s\Vert_{L^2}\right)\,,
\end{align*}
by Lemma~\ref{kappa2}. Then,
\begin{equation}\label{kesse5}
[k^5_s(O^1)]^2\leqslant 2[k^i_s(O^1)]^2+C\left(1+\Vert k_s\Vert^2_{L^2}\right)
\end{equation}
and it follows
\begin{eqnarray*}
[k^5_s(Q)]^2 &=& [k_s^5(O^1)]^2 + 2 \int_{O^1}^Q k_s k_{ss}\,ds\\ 
&\le& 2[k_s^i(O^1)]^2 + C\left(1+\Vert k_s\Vert^2_{L^2}\right)+ 2 \int_{O^1}^Q k_s k_{ss}\,ds\\
&\le& C + C\Vert k_s\Vert^2_{L^2}+2 \|k_s\|_{L^2} \|k_{ss}\|_{L^2}\,,
\end{eqnarray*}
since, by the previous argument, we have $[k^i_s(O^1)]^2, [k^i_s(O^2)]^2\leqslant C + 2 \|k_s\|_{L^2} \|k_{ss}\|_{L^2}$, for all $i\in\{1, 2, 3, 4\}$. Hence, we conclude
\begin{equation}
\Vert k_s\Vert_{L^\infty}^2\leqslant C + C\Vert k_s\Vert^2_{L^2}+2 \|k_s\|_{L^2} \|k_{ss}\|_{L^2}\,.
\end{equation}

We now pass to estimate $\|k_s\|_{L^2}$.
Making computation~\eqref{evolint000} explicit for $j=1$, we have
\begin{equation}
\partial_t \int_{\SS_t} k_s^2\,ds\leqslant
-2\int_{\SS_t} k_{ss}^2\,ds + 7\int_{\SS_t} k^2k_s^2\,ds
-\sum_{p=1}^2\sum_{i=1}^3 2k^{pi}_s k^{pi}_{ss}+\lambda ^{pi} \left( k^{pi}_s\right) ^2\,
\biggr\vert_{\text{{ at the $3$--point $O^p$}}}+C\,.\label{kappast}
\end{equation}
Then, as in Section~\ref{kestimates} we work to lower the differentiation order of the boundary term $\sum_{i=1}^3k^i_sk^i_{ss}$ at each $3$--point.\\
We claim that the following equality holds at each $3$--point,
\begin{equation}\label{lowerboundaryterm}
3\sum_{i=1}^3\lambda^i k^i k^i_t = \partial_t \sum_{i=1}^3\lambda^i \left( k^i\right) ^2\,.
\end{equation}
Keeping in mind that, at every $3$--point, we have $\sum_{i=1}^3k^i=0$ and $\lambda^{i}=\frac{k^{i-1}-k^{i+1}}{\sqrt{3}}$, 
with the convention that the superscripts are considered ``modulus $3$'' (see Section~\ref{basiccomp}), we obtains
\begin{align*}
\sqrt{3}\sum_{i=1}^3\lambda^i k^i k^i_t &= \sum_{i=1}^3 \left(k^{i-1}-k^{i+1}\right) k^ik^i_t \\
&= \sum_{i=1}^3 k^{i+1}\left(k^{i+1}+k^{i-1}\right)k^i_t-k^{i-1}\left(k^{i+1}+k^{i-1}\right)k^i_t\\
&=\sum_{i=1}^3 \left[\left(k^{i+1}\right)^2-\left(k^{i-1}\right)^2\right]k^i_t\,,
\end{align*}
and
\begin{align*}
\sqrt{3} \partial_t \sum_{i=1}^3\lambda^i \left( k^i\right) ^2
&=\sqrt{3}\sum_{i=1}^3\lambda^i_t\left( k^i\right)^2+2\lambda^i k^ik^i_t\\
&=\sum_{i=1}^3 \left(k_t^{i-1}-k_t^{i+1}\right)\left(k^i\right)^2
 +2\sum_{i=1}^3\left(k^{i-1}-k^{i+1}\right)k^ik^i_t\\
&=\sum_{i=1}^3\left[\left(k^{i+1}\right)^2-\left(k^{i-1}\right)^2
+2k^ik^{i-1}-2k^ik^{i+1}\right]k_t^i\\
&=\sum_{i=1}^3\left[\left(k^{i+1}\right)^2-\left(k^{i-1}\right)^2
-2(k^{i-1}+k^{i+1})k^{i-1}+2(k^{i-1}+k^{i+1})k^{i+1}\right]k_t^i\\
&=3\sum_{i=1}^3\left[\left(k^{i+1}\right)^2-\left(k^{i-1}\right)^2
\right]k^i_t\,,
\end{align*}
thus, equality~\eqref{lowerboundaryterm} is proved.\\
Now we use such equality to lower the differentiation order of 
the term $\sum_{i=1}^3k^i_s k^i_{ss}$. Recalling the formula $\dert k=k_{ss}+k_s\lambda + k^3$ and that $\sum_{i=1}^3k^i_t=\partial_t\sum_{i=1}^3k^i=0$, we get 
\begin{eqnarray}
\sum_{i=1}^3k^i_s k^i_{ss} 
&=& \sum_{i=1}^3k^i_s\bigl[k^i_{t}-\lambda^i k^i_s-\left( k^i\right) ^3\bigr]\nonumber
\\
&=& \sum_{i=1}^3\bigl(k^i_s+\lambda^i k^i-\lambda^i k^i\bigr)k_t^i
-\sum_{i=1}^3\lambda^i \left( k^i_s\right)^2 + \left( k^i\right)^3 k^i_s\nonumber
\\
&=& \sum_{i=1}^3\bigl(k^i_s+\lambda^i k^i\bigr)k_t^i
-\sum_{i=1}^3\lambda^i k^i k^i_t
-\sum_{i=1}^3\lambda^i \left( k^i_s\right)^2 + \left( k^i\right)^3 k^i_s\nonumber
\\
&=& - \partial_t \sum_{i=1}^3\lambda^i \left( k^i\right) ^2\bigr/3 -\sum_{i=1}^3\lambda^i \left( k^i_s\right)^2 + \left( k^i\right) ^3 k^i_s\,,\label{kessekesseesse}
\end{eqnarray}
at the triple junctions $O^1$ and $O^2$, where we used the fact that $k^i_s+\lambda^i k^i$ is independent of $i\in\{1, 2, 3\}$.\\
Substituting this equality into estimate~\eqref{kappast}, we obtain
\begin{align}
\partial_t \int_{\SS_t} k_s^2\,ds
\le&\, -2\int_{\SS_t} k_{ss}^2\,ds + 7\int_{\SS_t} k^2k_s^2\,ds+ \sum_{p=1}^2\sum_{i=1}^3 2\left( k^{pi}\right) ^3 k^{pi}_s+\lambda^{pi} \left( k^{pi}_s\right)^2\,
\biggr\vert_{\text{{ at the $3$--point $O^p$}}}+C\nonumber\\
&\,+ 2\partial_t \sum_{p=1}^2\sum_{i=1}^3\lambda^{pi} \left( k^{pi}\right) ^2\bigr/3\,
\biggr\vert_{\text{{ at the $3$--point $O^p$}}}\nonumber\\
\le&\, 
-2\int_{\SS_t} k_{ss}^2\,ds +C\Vert k\Vert^2_{L^2}\Vert k_s\Vert^2_{L^\infty}+ \sum_{p=1}^2\sum_{i=1}^3 2\left( k^{pi}\right) ^3 k^{pi}_s+\lambda^{pi} \left( k^{pi}_s\right)^2\,
\biggr\vert_{\text{{ at the $3$--point $O^p$}}}\nonumber\\
&\,+ 2\partial_t \sum_{p=1}^2\sum_{i=1}^3\lambda^{pi}
\left(k^{pi}\right)^2\bigr/3\,
\biggr\vert_{\text{{ at the $3$--point $O^p$}}}+C\,.\label{kappast2}
\end{align}
Using the previous estimate on $\Vert k_s \Vert_{L^\infty}$,
the hypothesis of uniform boundedness of $\Vert k\Vert_{L^2}$ and Young inequality, we get
\begin{align*}
\Vert k\Vert^2_{L^2}\Vert k_s\Vert^2_{L^\infty}
&\leqslant C+C\Vert k_s\Vert_{L^2}^2+C\Vert k_s\Vert_{L^2}\Vert k_{ss}\Vert_{L^2}\\
&\leqslant C+C\Vert k_s\Vert_{L^2}^2+C_\varepsilon \Vert k_s\Vert_{L^2}^2+\varepsilon\Vert k_{ss}\Vert_{L^2}^2\\
&=C+C_\varepsilon \Vert k_s\Vert_{L^2}^2+\varepsilon\Vert k_{ss}\Vert_{L^2}^2\,,
\end{align*}
for any small value $\eps>0$ and a suitable constant $C_\eps$.\\
We deal now with the boundary term $\sum_{i=1}^3 2\left(k^i\right)^3 k^i_s+\lambda^i \left( k^i_s\right)^2$.\\
By the fact that $k_s^i+\lambda^ik^i=k_s^j+\lambda^jk^j$, for every pair $i, j$, it follows that
$\left(k_s+\lambda k\right)^2\sum_{i=1}^3\lambda^i=0$, hence,
$$
\sum_{i=1}^3\lambda^i \left( k^i_s\right)^2
=-\sum_{i=1}^3\left( \lambda^i\right)^3\left(k^i\right)^2+2\left(\lambda^i\right)^2 k^ik^i_s\,,$$
then, we can write
\begin{align}
\sum_{i=1}^3 2\left(k^i\right)^3 k^i_s+\lambda^i \left(k^i_s\right)^2
=&\,\sum_{i=1}^3 2\left(k^i\right)^3 k^i_s
-\left( \lambda^i\right)^3\left(k^i\right)^2-2\left(\lambda^i\right)^2 k^ik^i_s\nonumber\\
=&\,\sum_{i=1}^3 2\bigl[\left( k^i\right)^3-\left( \lambda^i\right)^2 k^i\bigr] k_s^i -\sum_{i=1}^3\left( \lambda^i\right)^3 \left( k^i\right) ^2\nonumber\\
=&\,2(k_s+\lambda k)\sum_{i=1}^3\left( k^i\right)^3-\left( \lambda^i\right)^2 k^i
+\sum_{i=1}^3 \left( \lambda^i\right)^3 \left(k^i\right)^2-2\lambda^i\left( k^i\right)^4\,.\label{Eq1}
\end{align}
At the triple junction $O^1$, where the curves $\gamma^1,\gamma^2$ and $\gamma^5$
concur, there exists $i\in\{1,2\}$ such that $\vert k^i(O^1)\vert\geqslant \frac{K}{2}$,
where $K=\max_{j\in\{1, 2, 3\}}\vert k^j(O^1)\vert$, hence at the $3$--point $O^1$
\begin{align*}
2(k_s+\lambda k)\sum_{i=1}^3\left( k^i\right)^3-\left( \lambda^i\right)^2 k^i+\sum_{i=1}^3 &\,\left( \lambda^i\right)^3 \left(k^i\right)^2-2\lambda^i\left( k^i\right)^4\\
& \leqslant C K^5+C\vert k_s^i(O^1)\vert K^3
\\
&\leqslant C\vert k^i(O^1)\vert^5+C\vert k_s^i(O^1)\vert \vert k^i(O^1)\vert^3
\\
&\leqslant C\Vert k^i\Vert_{L^\infty(\gamma^i)}^5+C\Vert k_s^i\Vert_{L^\infty(\gamma^i)}\Vert k^i\Vert_{L^\infty(\gamma^i)}^3\,.
\end{align*}
We estimate now $C\Vert k\Vert_{L^\infty(\gamma^i)}^5+C\Vert k_s\Vert_{L^\infty(\gamma^i)}\Vert k\Vert_{L^\infty(\gamma^i)}^3$ via the Gagliardo--Nirenberg interpolation inequalities in Proposition~\ref{gl}. Letting $u=k^i$, $p=+\infty$, $m=2$ and $n=0,1$ in formula~\eqref{int1}, we get 
\begin{align}
\Vert k^i\Vert_{L^\infty(\gamma^i)}&\leqslant C\Vert k^i_{ss}\Vert^{\frac14}_{L^2(\gamma^i)}\Vert k^i\Vert_{L^2(\gamma^i)}^{\frac34}
+\frac{B}{L^{\frac12}}\Vert k^i\Vert_{L^2(\gamma^i)}
\leqslant C\Vert k^i_{ss}\Vert_{L^2(\gamma^i)}^\frac14+C_L\label{GN3}\\
\Vert k^i_s\Vert_{L^\infty(\gamma^i)}&\leqslant C\Vert k^i_{ss}\Vert^{\frac34}_{L^2(\gamma^i)}\Vert k^i\Vert_{L^2(\gamma^i)}^{\frac14}
+\frac{B}{L^{\frac32}}\Vert k^i\Vert_{L^2(\gamma^i)}
\leqslant C\Vert k^i_{ss}\Vert_{L^2(\gamma^i)}^\frac34+C_L\label{GN1}\,,
\end{align} 
hence,
$$
C\Vert k^i\Vert_{L^\infty(\gamma^i)}^5+C\Vert k^i\Vert_{L^\infty(\gamma^i)}^3\Vert k^i_s\Vert_{L^\infty(\gamma^i)}
\leqslant
C\Vert k^i_{ss}\Vert_{L^2(\gamma^i)}^\frac54+
C\Vert k^i_{ss}\Vert_{L^2(\gamma^i)}^\frac32+ C_L
\leqslant
\varepsilon\Vert k^i_{ss}\Vert_{L^2(\gamma^i)}^{2}+C_{L,\varepsilon}\,.
$$
Thus, finally, 
$$
2(k_s+\lambda k)\sum_{i=1}^3\left( k^i\right)^3-\left( \lambda^i\right)^2 k^i+\sum_{i=1}^3\left( \lambda^i\right)^3 \left(k^i\right)^2-2\lambda^i\left( k^i\right)^4\leqslant 
\varepsilon\Vert k^i_{ss}\Vert_{L^2(\gamma^i)}^{2}+C_{L,\varepsilon}\leqslant 
\varepsilon\Vert k_{ss}\Vert_{L^2}^{2}+C_{L,\varepsilon}\,.
$$
Coming back to computation~\eqref{kappast2}, we have
\begin{align*}
\partial_t & \biggl(\int_{\SS_t} k_s^2\, ds-2 \sum_{p=1}^2\sum_{i=1}^3\lambda^{pi} \left(k^{pi}\right)^2
\,\big/3 \,\biggr\vert_{\text{{ at the $3$--point $O^p$}}}\biggr)\\
&\leqslant -2\int_{\SS_t} k_{ss}^2 ds+C\Vert k_s\Vert_{L^2}^2+\varepsilon\Vert k_{ss}\Vert_{L^2}^2+C_{L,\varepsilon}\\
&\leqslant -2\int_{\SS_t} k_{ss}^2 ds+C\Vert k_s\Vert_{L^2}^2+2\varepsilon\Vert k_{ss}\Vert_{L^2}^2-C_{L,\eps}\Vert k^i\Vert_{L^\infty(\gamma^i)}^3+C_{L,\varepsilon}\\
&\leqslant C_{L,\varepsilon}\biggl(\int_{\SS_t} k_s^2\, ds-2 \sum_{p=1}^2\sum_{i=1}^3\lambda^{pi} \left(k^{pi}\right)^2\,\bigl/3
\, \biggr\vert_{\text{{ at the $3$--point $O^p$}}}\biggr)+C_{L,\varepsilon}\,,
\end{align*}
where we chose $\varepsilon<1$.\\
By Gronwall's Lemma, it follows that $\|k_s\|_{L^2}^2-2 \sum_{p=1}^2\sum_{i=1}^3\lambda^{pi} \left(k^{pi}\right)^2
\,\bigl/3\, \Bigr\vert_{\text{{ at the $3$--point $O^p$}}}$ is uniformly bounded, for $t\in[0,T)$, by a constant depending on $L$ and its value on the initial network $\SS_0$. Then, applying Young inequality to estimate~\eqref{inkappa} of Lemma~\ref{kinfty}, there holds
$$
\|k\|_{L^\infty}^3 \leqslant C+C\|k\|_{L^2}^{3/2} \|k_s\|_{L^2}^{3/2}
\leqslant C+C_\eps\|k\|_{L^2}^{6} +\eps\|k_s\|_{L^2}^{2}\leqslant 
C_\eps+\eps \|k_s\|_{L^2}^{2}\,,
$$
as $\Vert k\Vert_ {L^2}$ is uniformly bounded in $[0,T)$. Choosing $\eps>0$ small enough, we conclude that also $\|k_s\|_{L^2}$ is uniformly bounded in $[0,T)$.
\end{proof}

\begin{proof}[Proof of Proposition~\ref{cross}]
By the hypotheses, we can assume that the sequence of rescaled networks $\SS^{\mu_i}_{-1/(2+\delta)}$ converges in $W^{2,2}\loc$, as $i\to\infty$, to a standard cross (which has zero curvature), for some $\delta>0$ as small as we want.\\
Arguing as in the proof of Lemma~\ref{thm:locreg.3}, by means of Lemma~\ref{boh}, we can also assume
that, for $R>0$ large enough, the sequence of rescaled flows $\SS_\tt^{\mu_i}$ converges smoothly and uniformly to the flow $\SS^\infty_\tt$, given by the four halflines, in $\bigl(B_{3R}(0)\setminus B_R(0)\bigr)\times[-1/2,0)$. Hence, there exists $i_0\in\NN$ such that for every $i\geqslant i_0$ the flow $\SS_t$ in the annulus $B_{3R/\mu_i}(x_0)\setminus B_{R/\mu_i}(x_0)$ has equibounded curvature, no $3$--points and a uniform bound from below on the lengths of the four curves, for $t\in [T-\mu_i^{-2}/(2+\delta),T)$. Setting $t_i=T-\mu_i^{-2}/(2+\delta)$, we have then a sequence of times $t_i\to T$ such that, when $i\geqslant i_0$, the above conclusion holds for the flow $\SS_t$ in the annulus $B_{3R\sqrt{2(T-t_i)}}(x_0)\setminus B_{R\sqrt{2(T-t_i)}}(x_0)$ and with $t\in[t_i,T)$, we can thus introduce four ``artificial'' moving boundary points $P^r(t)\in\SS_t$ with $\vert P^r(t)-x_0\vert=2R\sqrt{2(T-t_i)}$, with $r\in\{1,
2, 3, 4\}$ and $t\in [t_i,T)$, such that the estimates~\eqref{endsmooth} are satisfied, that is, the hypotheses
about the end--points $P^i(t)$ of Lemmas~\ref{kinfty},~\ref{kappa2} and~\ref{kappaesse2} hold.\\
As we the sequence of networks $\SS^{\mu_i}_{-1/(2+\delta)}$
converges in $W^{2,2}\loc$ to a limit network with zero curvature, as $i\to\infty$, we have
$$
\lim_{i\to\infty}\Vert
\widetilde{k}\Vert_{L^2(B_{3R}(0)\cap\,\SS^{\mu_i}_{-1/(2+\delta)})}=0\,,\qquad
\text{ that is, }\qquad 
\int_{B_{3R}(0)\cap\,\SS^{\mu_i}_{-1/(2+\delta)}}\widetilde{k}^2\,d\sigma\leqslant\varepsilon_i\,,
$$
for a sequence $\varepsilon_i\to 0$ as $i\to\infty$.
Rewriting this condition for the non--rescaled networks, we have
\begin{equation}\label{notrescaled}
\int_{B_{3R\sqrt{2(T-t_i)}}(x_0)\cap\mathbb{S}_{t_i}} k^2\,ds\leqslant \frac{\varepsilon_i}{\sqrt{2(T-t_i)}}\,.
\end{equation}
Applying now Lemma~\ref{kappa2} to the flow of networks $\SS_t$ in the ball $B_{2R\sqrt{2(T-t_i)}}(x_0)$ in the time interval
 $[t_i,T)$, we have that $\Vert k\Vert_{L^2(B_{2R\sqrt{2(T-t_i)}}(x_0)\cap\SS_t)}$ is uniformly bounded, up to
 time 
$$
T_i=t_i+\min\,\Bigl\{ T, 1\big/
8C\,\bigl(\Vert k
 \Vert^2_{L^2(B_{2R\sqrt{2(T-t_i)}}(x_0)\cap\SS_{t_i})}+1\bigr)^2\Bigr\}\,.
$$
We want to see that actually $T_i>T$ for $i$ large enough, hence, $\Vert k
 \Vert_{L^2(B_{2R}(x_0)\cap\SS_t)}$ is uniformly bounded for
 $t\in[t_i,T)$. If this is not true, we have
\begin{align*}
T_i=&\,t_i+\frac{1}{8C\,\bigl(\Vert k \Vert^2_{L^2(B_{2R\sqrt{2(T-t_i)}}(x_0)\cap\SS_{t_i})}+1\bigr)^2}\\
\geqslant&\,t_i+\frac{1}{8C\,\bigl(\eps_i/\sqrt{2(T-t_i)}+1\bigr)^2}\\
=&\,t_i+\frac{2(T-t_i)}{8C\,\bigl(\eps_i+\sqrt{2(T-t_i)}\,\bigr)^2}\\
=&\,T+(2(T-t_i))\biggl(\frac{2}{8C\,\bigl(\eps_i+\sqrt{2(T-t_i)}\,\bigr)^2}-1\biggr)\,,
\end{align*}
which is clearly larger than $T$, as $\eps_i\to0$, when
$i\to\infty$.

Choosing then $i_1\geqslant i_0$ large enough, since $\Vert
k\Vert_{L^2(B_{2R\sqrt{2(T-t_{i_1})}}(x_0)\cap\,\SS_t)}$ is 
uniformly bounded for all times $t\in[t_{i_1},T)$ and 
the length of the four curves that connect the junctions with the ``artificial'' boundary points $P^r(t)$
are bounded below by a uniform constant, Lemma~\ref{kappaesse2}
applies, hence, thanks to Lemma~\ref{kinfty}, we have a uniform bound 
on $\Vert k\Vert_{L^\infty(B_{2R\sqrt{2(T-t_{i_1})}}(x_0)\cap\,\SS_t)}$ for $t\in[0,T)$.
\end{proof}

As we proved Proposition~\ref{cross}, Proposition~\ref{prop999} follows.
An obvious consequence is that evolving trees do not develop this kind
of singularity, hence their curvature flow is smooth till a curve collapses with
uniformly bounded curvature. Moreover it is also easy to see that if no
region collapses, the network is locally a tree around every point of 
$\overline{\Omega}$, for $t$ close enough to $T$, so Proposition~\ref{prop999} applies globally.

\begin{cor}\label{tree-bdcurb} If {\bf{M1}} holds and $\SS_0$ is a
tree, the curvature of $\SS_t$ is uniformly bounded during the flow (hence
 we are in the case of Proposition~\ref{bdcurvcollapse} in the previous section).
\end{cor}

Combining Propositions~\ref{noloop} and~\ref{prop999}, 
we have the following local conclusion.

\begin{thm}\label{tree-bdcurb2}
If {\bf{M1}} holds and $\SS_t$ is a tree in a neighborhood of $x_0\in\overline{\Omega}$, for $t$ close enough to $T$, the curvature is uniformly locally bounded and either the flow $\SS_t$ is locally smooth or, up to reparametrization proportional to arclength, converge in $C^1$ locally around $x_0$, as $t\to T$, to some degenerate regular network
$\widehat{\SS}_T$ whose non--degenerate curves form a $C^2$ network $\SS_T$ with a possibly non--regular multi--point which is among the ones described in Proposition~\ref{bdcurvcollapse}, coming from the collapse of single ``isolated'' curve of $\SS_t$.\\
Moreover, the curves of $\SS_T$, in a neighborhood of $x_0$, are smooth outside the multi--point.\\
Obviously, the conclusion holds when $\SS_0$ is a tree.
\end{thm}
\begin{proof}
We only have to show that the curves of $\SS_T$ are actually $C^2$. By means of Lemma~\ref{kappaesse2}, $\Vert k_s\Vert_{L^2}$ is locally uniformly bounded on $[0,T)$, which implies that the convergence of the non--collapsing curves of $\SS_t$ to $\SS_T$, as $t\to T$, is actually in $C^2\loc$ and we are done. The smooth convergence outside the multi--points then follows by the interior estimates of Ecker--Huisken in~\cite{eckhui2}. 
\end{proof}

\begin{rem}\label{rem555}
We expect that, by extending the estimates of Lemmas~\ref{kinfty},~\ref{kappa2} and~\ref{kappaesse2} to the higher order derivatives of the curvature, one should actually get the smoothness of the curves of $\SS_T$ and of the convergence of the non--collapsing curves of $\SS_t$ to $\SS_T$. Moreover, the collapsing curve should converges in $C^\infty$ to a constant map, hence also the local convergence of $\SS_t$ to $\widehat{\SS}_T$ would be actually smooth.
\end{rem}

\begin{cor}\label{regioncor}
If {\bf{M1}} holds, the curvature is uniformly bounded along the flow for $t\in[0,T)$, if and only if no region collapses as $t\to T$. Equivalently, in every neighborhood $\SS_t$ is a tree, for $t$ close enough to $T$.
\end{cor}
\begin{proof}
By Lemma~\ref{regcollap} when the curvature is bounded, regions cannot collapse. Viceversa, if no region collapses the network is locally a tree around every point of $\overline{\Omega}$, hence by compactness and Proposition~\ref{prop999}, the curvature is uniformly bounded.
\end{proof}

\begin{rem}\label{rem777}
This corollary holds also locally.
\end{rem}

\begin{cor}\label{tree-bdcurb2cor}
If {\bf{M1}} holds and no region collapses as $t\to T$, the $C^2$ network $\SS_T$ has only multi--points like the ones described in Proposition~\ref{bdcurvcollapse}, coming from the collapse of a family of single ``isolated'' curves of $\SS_t$.
\end{cor}

Another consequence of the previous analysis is the existence of Type~0 singularities (see Remark~\ref{type0}).

\begin{prop}\label{type0prop} If {\bf{M1}} holds, Type~0 singularities actually exist. 
\end{prop}
\begin{proof} Let us consider an initial (regular) smooth network $\SS_0$, which is centrally symmetric, in the convex domain $\Omega$ (also centrally symmetric) as in the following figure:
\begin{figure}[H]
\begin{center}
\begin{tikzpicture}[scale=0.30]
\draw[color=black,scale=3.3,domain=-3.141: 3.141,
smooth,variable=\t,shift={(0,0)},rotate=0]plot({5*sin(\t r)},
{1.5*cos(\t r)});
\draw[color=black]
(0,0)to[out= 90,in=-90, looseness=1](0,2)
(0,2)to[out=150,in=-50, looseness=1] (-12,3.46)
(0,2)to[out= 30,in=-160, looseness=1] (12,3.46);
\draw[color=black, rotate=180]
(0,0)to[out= 90,in=-90, looseness=1](0,2)
(0,2)to[out=150,in=-50, looseness=1] (-12,3.46)
(0,2)to[out= 30,in=-160, looseness=1] (12,3.46);
\draw[color=black!30!white]
(-10,0)to[out= 0,in=180, looseness=1](10,0)
(-10,0)to[out= 120,in=-60, looseness=1] (-12,3.46)
(-10,0)to[out= -120,in=60, looseness=1] (-12,-3.46)
(10,0)to[out= 60,in=-120, looseness=1] (12,3.46)
(10,0)to[out=-60,in=120, looseness=1] (12,-3.46);
\path[font=]
(15,-2.8) node[below] {\Large{$\Omega$}}
(-1,4.5) node[below] {\large{$\SS_0$}}
(12.2,2) node[below] {\large{${\mathbb M}$}};
\end{tikzpicture}
\end{center}
\begin{caption}{The networks $\SS_0$ and ${\mathbb M}$.\label{type0fig}}
\end{caption}
\end{figure}
\noindent where in gray we drew the minimal network ${\mathbb M}$ connecting the four end--points of $\SS_0$ on the boundary of $\Omega$. Assuming that $\Omega$ is very ``long and thin'', it can be shown that ${\mathbb M}$ is the only ``stationary'' (regular and with zero curvature) network connecting the four end--points of $\SS_0$.\\
By Corollary~\ref{tree-bdcurb}, during the smooth curvature flow $\SS_t$ of $\SS_0$ (given by Theorem~\ref{smoothexist-prob}, maintaining the central symmetry) the curvature is bounded and either a singularity develops or the flow $\SS_t$ is smooth for every positive time. Then, it is easy to guess and actually it will be a consequence of Proposition~\ref{prolong} that, as $t\to+\infty$, the network $\SS_t$ converges in $C^1$ to ${\mathbb M}$, which is a contradiction because of their different structures. Hence, at some time $T<+\infty$ a Type~0 singularity must develop and the only possibility is the collapse of the ``central'' curve of $\SS_t$, by its symmetry.
\end{proof}

Bounded curvature is not actually the case if some loops are
present in $\SS_t$, indeed a region bounded by less than six
curves possibly collapses, then in such case the curvature cannot stay bounded, by Corollary~\ref{regioncor}.
\begin{figure}[H]
\begin{center}
\begin{tikzpicture}[rotate=90]
\draw[color=black!40!white, shift={(0,-7)}]
(-0.05,2.65)to[out= -90,in=150, looseness=1] (0.17,2.3)
(0.17,2.3)to[out= -30,in=100, looseness=1] (-0.12,2)
(-0.12,2)to[out= -80,in=40, looseness=1] (0.15,1.7)
(0.15,1.7)to[out= -140,in=90, looseness=1.3](0,1.1)
(0,1.1)--(-.2,1.35)
(0,1.1)--(+.2,1.35);
\draw[color=black]
(0.76,0)to[out= 0,in=180, looseness=1](2.5,0)
(0.23,-0.72)to[out=-72,in=108, looseness=1](0.77,-2.37)
(0.23,0.72)to[out=72,in=-108, looseness=1](0.77,2.37)
(-0.61,-0.44)to[out=-144,in=36, looseness=1](-2,-1.46)
(-0.61,0.44)to[out=144,in=-36, looseness=1](-2,1.46)
(0.76,0)to[out= 120,in=-48, looseness=1](0.23,0.72)
to[out= -162,in=24, looseness=1](-0.61,0.44)
to[out= -96,in=96, looseness=1](-0.61,-0.44)
(0.76,0)to[out= -120,in=48, looseness=1](0.23,-0.72)
to[out= 162,in=-24, looseness=1](-0.61,-0.44);
\draw[color=black,dashed]
(2.5,0)to[out= 0,in=180, looseness=1](3,0)
(0.77,-2.37)to[out=-72,in=108, looseness=1](0.92,-2.85)
(0.77,2.37)to[out=72,in=-108, looseness=1](0.92,2.85)
(-2,-1.46)to[out=-144,in=36, looseness=1](-2.42,-1.76)
(-2,1.46)to[out=144,in=-36, looseness=1](-2.42,1.76);
\path[font=\small]
(-1,-5.1) node[above]{$t\to T$}
(2.8,-1.2) node[below]{$\SS_t$}
(2.8,-11.2) node[below]{$\SS_T$};
\fill[color=black,shift={(0,-10)}]
(0,0) circle (1.5pt);
\draw[color=black,shift={(0,-10)}]
(0,0)to[out= 0,in=180, looseness=1](2.5,0)
(0,0)to[out=-72,in=108, looseness=1](0.77,-2.37)
(0,0)to[out=72,in=-108, looseness=1](0.77,2.37)
(0,0)to[out=-144,in=36, looseness=1](-2,-1.46)
(0,0)to[out=144,in=-36, looseness=1](-2,1.46);
\draw[color=black,dashed,shift={(0,-10)}]
(2.5,0)to[out= 0,in=180, looseness=1](3,0)
(0.77,-2.37)to[out=-72,in=108, looseness=1](0.92,-2.85)
(0.77,2.37)to[out=72,in=-108, looseness=1](0.92,2.85)
(-2,-1.46)to[out=-144,in=36, looseness=1](-2.42,-1.76)
(-2,1.46)to[out=144,in=-36, looseness=1](-2.42,1.76);
\end{tikzpicture}
\end{center}
\begin{caption}{Homothetic collapse of a (symmetric) pentagonal region of $\SS_t$ ({\em five--ray star}).\label{5rayfig}}
\end{caption}
\end{figure}
\noindent Determining what asymptotically happens in detail in the general case can be quite
complicated because of the difficulty in classifying all the regular shrinkers 
with loops. Anyway, some special cases with ``few'' triple
junctions can be fully understood. We will show an example of this
analysis in Section~\ref{globsec}, considering networks with at most {\em two} triple junctions. We underline that the interest
in these very special cases is because of the multiplicity--one conjecture
holds for such networks (Corollary~\ref{Elemma1}).

However, even if we cannot describe all the possible shrinkers
$\SS_{-1/2}^\infty$ or $\widetilde{\SS}_\infty$, arising respectively
from the parabolic or Huisken's rescaling procedure at the singular
time $T<+\infty$, we can get enough information in order to restart
the flow by means of Theorem~\ref{evolnonreg} in the next section (actually by its extension discussed in Remark~\ref{felix}). The point is to connect the information on the possible blow--up limit networks $\widetilde{\SS}_\infty$ to the existence
and the structure of a network $\SS_T$ which is the limit of $\SS_t$, as $t\to T$.

We recall that assuming the multiplicity--one conjecture, by Lemma~\ref{remhot}, there exist the limits $x_i=\lim_{t\to T}O^i(t)$, for
$i\in\{1,2,\dots,m\}$ and correspond to the (finitely many) points in $\Omega$ where
$\widehat\Theta(x_0)>1$ and to the end--points of $\SS_t$ such that
the curve getting there collapses as $t\to T$.

We first discuss what happens around an end--point $P^r$ of the network $\SS_t$ if $x_i=P^r$ for some (possibly more than one) $i\in\{1,2,\dots,m\}$. As before, we consider the network ${\mathbb{H}}_t^r$, obtained by the union of $\SS_t$ with its ``reflection'' with respect to the point $P^r$ (see the end of Section~\ref{geopropsub}). If $\Omega$ is strictly convex, by Proposition~\ref{omegaok2}, every blow--up limit network $\widetilde{\mathbb{H}}_\infty^r$, obtained rescaling around the end--point $P^r$, must be symmetric and contained in the union of two cones for the origin of $\R^2$. Then, by an argument similar to the one in the proof of Lemma~\ref{thm:densitybound}, either $\widetilde{\mathbb{H}}_\infty^r$ is a tree, or it contains a loop around the origin, which is clearly impossible by such property. Hence, we conclude that $\widetilde{\mathbb{H}}_\infty^r$ is a tree and the same the blow--up limit network $\widetilde{\SS}_\infty$, which means that we are in the previous case, considered in Proposition~\ref{prop999}, in particular, the curvature is locally bounded.

Then, by Proposition~\ref{noloop}, Theorem~\ref{tree-bdcurb2} and Remark~\ref{rem555}, we have a complete description of the behavior of $\SS_t$ locally around its end--point, as $t\to T$.

\begin{thm}\label{prop999b}
If {\bf{M1}} holds and the open set $\Omega$ is strictly convex, then in a neighborhood of its fixed end--points on $\partial\Omega$, the evolving regular network $\SS_t$ is a tree, for $t$ close enough to $T$ and its curvature is uniformly locally bounded during the flow. Hence, around any end--point $P^r$ either the flow is smooth, or the curve of $\SS_t$ getting to $P^r$ collapses and the network $\SS_t$ locally converges in $C^\infty$, as $t\to T$, to two concurring curves at such end--point forming an angle of $120$ degrees, as in the right side of Figure~\ref{Pcollapse}.
\end{thm}

\begin{rem}
We remark that the hypothesis of strict convexity of $\Omega$
can actually be weakened by asking that $\Omega$ is convex and that there do not exist three aligned end--points of the initial network $\mathbb{S}_0$ on $\partial\Omega$.
\end{rem}

We now deal with the situation of a point $x_0=\lim_{t\to T}O^i(t)$, for some $i\in\{1,2,\dots,m\}$, with $x_0\in\Omega$. Assuming that around $x_0\in{\Omega}$ the network is not definitively a tree for $t$ close enough to $T$ (which would imply that the curvature is locally bounded, by Proposition~\ref{prop999}), there must be at least one bounded region of $\SS_t$ collapsing to $x_0$ at the singular time. 
By the estimates in Section~\ref{geopropsub}, then the area
$A(t)$ of any such region must satisfy $A(t)=C(T-t)$, for some
constant $C$ depending on the number of its edges. Hence, all the
rescaled networks $\widetilde{\SS}_{x_0,\tt}$ must contain the
rescalings of such regions that will have a respective constant
area. These rescaled regions cannot ``go all to infinity'' and
disappear in the blow--up limit network $\widetilde{\SS}_\infty$, along any
converging sequence $\widetilde{\SS}_{x_0,\tt_j}\to\widetilde{\SS}_\infty$, otherwise Lemma~\ref{lemmatree} would apply and we could repeat the argument of the proof of Proposition~\ref{prop999}, concluding that the curvature is uniformly bounded around $x_0$.

We now suppose that the full rescaled family of networks $\widetilde{\SS}_{x_0,\tt}$ converges to $\widetilde\SS_\infty$, for instance, if the {\em uniqueness of blow--up assumption} {\bf{U}} in Problem~\ref{ooo12}, that we recall here below for the reader's convenience, holds (see also Remark~\ref{uequiv}):
\begin{itemize}
\item[{\bf{U}}:] In Proposition~\ref{resclimit-general}, the full family of rescaled regular networks $\widetilde{\SS}_{x_0,\tt}$ converges in $C^1\loc$ to the limit degenerate regular shrinker $\widetilde\SS_\infty$, as $\tt\to+\infty$.\\
Equivalently, the full family of parabolically rescaled curvature flows $\SS^{\mu}_\tt$ converges to the degenerate regular self--similarly shrinking flow $\SS^\infty_\tt$, as $\mu\to+\infty$, in Proposition~\ref{thm:shrinkingnetworks.1}.
\end{itemize}
Then, we can separate $\widetilde{\SS}_\infty$ in two parts:
\begin{itemize}
\item a compact subnetwork $\widetilde{\mathbb{M}}_\infty$ of $\widetilde{\SS}_\infty$, given by the union of the cores and the bounded curves (which are pieces of Abresch--Langer curves or straight segments passing by the origin of $\R^2$),
\item the union $\widetilde{\NN}_\infty=\widetilde{\SS}_\infty\setminus \widetilde{\mathbb{M}}_\infty$ of the unbounded curves of $\widetilde{\SS}_\infty$, which must be halflines ``pointing'' towards the origin (but not necessarily containing it), by Remark~\ref{abla}.
\end{itemize}

\begin{figure}[H]
\begin{center}

\begin{tikzpicture}[rotate=90]
\draw[color=black]
(-2.5,0)to[out= 0,in=180, looseness=1](-1.15,0)
(1.15,0)to[out= 0,in=180, looseness=1](2.5,0)
(0,2.5)to[out= -90,in=90, looseness=1](0,1.15)
(0,-2.5)to[out= 90,in=-90, looseness=1](0,-1.15);
\draw[color=black!40!white]
(-1.15,0)to[out= 60,in=-150, looseness=1](0,1.15)
(-1.15,0)to[out= -60,in=150, looseness=1](0,-1.15)
(1.15,0)to[out= 1200,in=-30, looseness=1](0,1.15)
(1.15,0)to[out= -120,in=30, looseness=1](0,-1.15);
\draw[color=black,dashed]
(-3,0)to[out= 0,in=180, looseness=1](-2.5,0)
(2.5,0)to[out= 0,in=180, looseness=1](3,0)
(0,3)to[out= -90,in=90, looseness=1](0,2.5)
(0,-3)to[out= 90,in=-90, looseness=1](0,-2.5);
\fill(0,0) circle (1.5pt);
\path[font=\small]
(2.8,-1.2) node[below]{$\widetilde{\SS}_\infty$}
(-.8,-1) node[below]{$\widetilde{{\mathbb{M}}}_\infty$}
(-.05,-.2) node[above]{$O$};
\end{tikzpicture}
\end{center}
\begin{caption}{The subnetwork $\widetilde{\mathbb{M}}_\infty$ (in gray) of a 4--symmetric regular shrinker $\widetilde{\SS}_\infty$ ({\em four--ray star}).}
\end{caption}
\end{figure}
\noindent Then, by rescaling--back (dynamically contracting) the flow
$\widetilde{\SS}_{x_0,\tt}\to\widetilde\SS_\infty$, by the uniqueness assumption, the subnetwork ${\mathbb{M}}_t$ of $\SS_t$ corresponding to the compact subnetwork of $\widetilde{\SS}_{x_0,\tt}$ converging to $\widetilde{\mathbb{M}}_\infty$, is contained in the ball $B_{C\sqrt{2(T-t)}/2}(x_0)$ for every $t\in[0,T)$, for some constant $C$ independent of $t$ (dependent on $\widetilde{{\mathbb{M}}}_\infty$). In particular, ${\mathbb{M}}_t$ completely collapses to the point $x_0$, ``disappearing'' in the limit, as $t\to T$.\\
We want now to describe the local behavior of the rest $\NN_t$ of the network $\SS_t$ (corresponding to the union of the curves of $\widetilde{\SS}_{x_0,\tt}$ neither collapsing, nor entirely going to infinity, converging to the halflines of $\widetilde{\SS}_\infty$), around the point $x_0$, as $t\to T$. 

\begin{rem}\label{remUtree} Notice that, inspecting the proof of Proposition~\ref{bdcurvcollapse}, it is easy to see that the uniqueness assumption {\bf{U}} holds at every point where the curvature is locally uniformly bounded. In particular, it holds in general if the network is a tree, by Corollary~\ref{tree-bdcurb}.
\end{rem}

\begin{prop}\label{loopsing}
If {\bf{M1}} and the above uniqueness assumption {\bf{U}} of the blow--up limit shrinker $\widetilde\SS_\infty$ hold, then, as $t\to T$, the family $\gamma^i_t$ of curves of $\NN_t$ converges in $C^1(U)$ and in $C^\infty(U\setminus\{x_0\})$, where $U$ is a neighborhood of $x_0$, as $t\to T$, to an embedded, possibly non--regular network $\SS_T$, composed of $C^1$ curves $\gamma^i_T$ concurring at $x_0$.\\
The directions of the halflines of $\widetilde\SS_\infty$ coincide with the inner unit tangent vectors of the limit curves $\gamma^i_T$ at $x_0$, hence, these latter are all distinct.\\
Moreover, the curvature of every curve $\gamma^i_T$ is of order $o(1/r)$, as $r\to 0$, where $r$ is the distance from the multi--point $x_0\in\SS_T$.
\end{prop}
\begin{proof}
Since rescaling the evolving networks $\SS_t$ the inner unit tangent vectors
at the end--points of the curves in $\NN_t$ do not change and $\widetilde{\NN}_{x_0,\tt}\to\widetilde{\NN}_\infty$, the inner unit
tangent vectors of the set of curves $\gamma^i_t$ converge to the unit vectors generating the halflines of $\widetilde\SS_\infty$. More precisely, if the sequence of rescalings $\widetilde{\gamma}_{x_0,\tt}^i$ of a curve $\gamma_t^i\in\NN_t$ converges in $C^1\loc$ to a halfline $H^i\subseteq\widetilde\NN_\infty$, the inner unit
tangent vectors at the end--point of $\gamma_t^i$ converge to the unit vector generating $H^i$, as $t\to T$.\\
As, by Lemma~\ref{remhot} and the collapse of the subnetwork ${\mathbb{M}}_t$, there is a neighborhood $U$ of $x_0$, such that for every $\rho>0$ in $U\setminus B_\rho(x_0)$, for $t$ close enough to $T$ there are no triple junctions, hence, by Lemma~\ref{boh}, the networks $\SS_t$ converge in $C^\infty\loc(U\setminus\{x_0\})$ to a smooth network $\SS_T$ composed of smooth curves $\gamma_T^i$ with an end--point at $x_0$.

We notice that the smoothness of $\SS_T$ and of $\gamma_T^i$ holds in $U\setminus\{x_0\}$, not in the whole $U$. We want to show that these curves are actually $C^1$ in $U$, that is, till the point $x_0$ and that their curvature is of order $o(1/r)$, where $r$ is the distance from $x_0$.

We consider one of the curves of $\NN_t$ (dropping the superscript by simplicity, from now on) $\gamma_t$, which converges (possibly, after reparametrization), as $t\to T$, to a limit $C^0$ curve $\gamma_T$ and such convergence is also in $C^\infty\loc(U\setminus\{x_0\})$. 

As the full rescaled sequence $\widetilde{\SS}_\tt$ converges to the blow--up limit $\widetilde{\SS}_\infty$, as $\tt\to+\infty$, also the full sequence of parabolically rescaled flows $\SS^\mu_\tt$ converges in $C^1\loc$ for every $\tt\in(-\infty,0)$, as $\mu\to+\infty$, to the limit self--similarly shrinking flow $\SS^\infty_\tt=\sqrt{-2\tt}\,\,\widetilde{\SS}_\infty$ (see Remark~\ref{relaz}). Then, the curves $\gamma^\mu_\tt$, which are the parabolic rescalings of the curves $\gamma_t$ converge to the halfline $H$, as $\mu\to+\infty$. We choose $\tt_0<0$ and $\mu_0>0$ such that the parabolic rescalings ${\mathbb{M}}^\mu_\tt$ of the subnetwork ${\mathbb{M}}_t$ of $\SS_t$ are contained in $B_{1/2}(0)$, for every $\mu>\mu_0$ and $\tt\in(\tt_0,0)$. Then, the rescaled curves $\gamma^\mu_\tt$ smoothly converge (by Lemma~\ref{boh}), as $\mu\to+\infty$, to the halfline $H$ (which has zero curvature) in $B_{4}(0)\setminus B_{1}(0)$, for every $\tt\in[\tt_0,0)$. Moreover, repeating the above argument, we have that, as $\tt\to0$, the curves $\gamma^\mu_\tt$ locally smoothly converge in $B_{4}(0)\setminus\{0\}$ to some limit curves $\gamma^\mu_0$, smooth in $B_{3}(0)\setminus\{0\}$, for every fixed $\mu>\mu_0$.

We are now going to apply the following special case of the {\em pseudolocality theorem} for mean curvature flow (see~\cite[Theorem~1.5]{Ilnevsch}) and the subsequent remark.

\begin{thm}\label{thm:graph_local} 
Let $\gamma_t$, for $t\in[0,T)$, be a smooth curvature flow of an embedded curve in $\R^2$ with bounded length ratios by a constant $D$ (see Definition~\ref{ublr}) and let 
$$
Q_r(x_0,y_0)=\{(x,y)\in \R^2~|~|x-x_0|<r,~|y-y_0|<r\}\,.
$$
Then, for any $\varepsilon>0$, there exists $\eta\in(0,\varepsilon)$ and $\delta\in(0,1)$, depending only on $\varepsilon$ and $D$, such that if $(x_0,y_0) \in
\gamma_0$ and $\gamma_0\cap Q_1(x_0,y_0)$ can be written as the graph of a function $u:(x_0-1,x_0+1) \to\R$ with Lipschitz constant less than
$\eta$, then
$$
\gamma_t\cap Q_{\delta}(x_0,y_0),\qquad\text{ for every $t\in [0,\delta^2)\cap [0,T)$,}
$$
is a graph over $(x_0-\delta,x_0+\delta)$ of a function with Lipschitz constant less than $\varepsilon$ and ``height'' bounded by $\varepsilon\delta$.
\end{thm}
 
\begin{rem}\label{rmks:graph_local}
Then, the local estimates of Ecker and Huisken~\cite{eckhui2} imply that, for every $m>0$ there is a constant $\sigma=\sigma(\delta,\varepsilon,m)>0$ and a constant $\eta =\eta(\delta,\varepsilon,m) >0$ such that if the curvature of $\gamma_0\cap Q_{\delta}(x_0,y_0)$ is bounded by $\sigma$, then the curvature of $\gamma_t\cap Q_{\delta/2}(x_0,y_0)$ is bounded by $m$, for every 
$t\in [0,\eta)\cap[0,T)$.
\end{rem}

By a rotation, we can assume that $H=\{(x,0)~\vert~x\geqslant a\}$ and let $\overline{H}=\{(x,0)~\vert~x\geqslant 0\}$. Taken any $\varepsilon>0$, let $\eta$ and $\delta$ be given by this theorem, we consider $\tt_1\in(\tt_0,0)$ such that $\tt_1+\delta^2/8>0$, then if $\mu$ is large enough, say larger than some $\mu_1>0$, the curve $\gamma^\mu_{\tt_1}$ in $B_{3}(0)\setminus B_{1}(0)$ is a graph of a function $u$ over the interval $[1,3]\times\{0\}\subseteq \overline{H}$ (with a small ``error'' at the borders), with gradient smaller than $\eta>0$. Hence, its evolution in the smaller annulus $B_{2+\delta}(0)\setminus B_{2-\delta}(0)$ is still a graph over $\overline{H}$ of a function with gradient smaller than $\varepsilon$, for every $\tt\in[\tt_1,\min\{\tt_1+\delta^2,0\})$, hence for every $\tt\in[\tt_1,0)$, by the assumption on $\tt_1$. Notice that, it follows that also $\gamma^\mu_0$ in $B_{2+\delta}(0)\setminus B_{2-\delta}(0)$ is a graph of a function over $\overline{H}$ with gradient smaller than $\varepsilon$, when $\mu>\mu_1$.

Rescaling back, since the $C^1$--norm is scaling invariant, we see that $\gamma_t$, for $t\in[T+\mu^{-2}\tt_1,T]$, can be written as a graph with $C^1$--norm less than $\varepsilon$ over $x_0+\overline{H}$ in $B_{(2+\delta)/\mu}(x_0)\setminus B_{(2-\delta)/\mu}(x_0)$, for every $\mu>\mu_1$. Hence, this conclusion holds for every pair
 $(\gamma_t,t)$ in 
$$
\bigcup_{\mu>\mu_2}\bigl(B_{(2+\delta)/\mu}(x_0)\setminus B_{(2-\delta)/\mu}(x_0)\bigr)\times[T+\mu^{-2}\tt_1,T]
\,\subseteq\,\R^2\times[0,T]\,,
$$
for every $\mu_2\geqslant\mu_1$ and this union contains the set 
$$
{\mathcal{A}}=B_{(2+\delta)/\mu_2}(x_0)\times[T+\mu_2^{-2}\tt_1,T]\,\,\,\, \text{\Large{$\setminus$}}\,\,\,\Bigl\{\,(x,t)\in\R^2\times[0,T]\,\,\Bigl\vert\,\, \vert x -x_0\vert\leqslant \frac{2-\delta}{\sqrt{-2\tt_1}}\sqrt{2(T-t)}\,\,\Bigr\}\,.
$$
Choosing now $\mu_2\geqslant\mu_1$ large enough, we know that there exists some $\tt_2>\tt_1$ such that for every $\tt>\tt_2$, the rescaled curves $\gamma^{\mu_2}_\tt$ can be written as graphs with $C^1$--norm less than $\varepsilon$ over $\overline{H}$ in the ball centered at the origin with radius $2\frac{2-\delta}{\sqrt{-2\tt_1}}$. That is, for $t\in[T+\mu^{-2}_2\tt_2,T]$, the curve $\gamma_t$ can be written as a graph with $C^1$--norm less than $\varepsilon$ over $x_0+\overline{H}$ in the ball of center $x_0$ and radius $2\frac{2-\delta}{\sqrt{-2\tt_1}}\sqrt{2(T-t)}$, hence, for every $(\gamma_t,t)$ in
$$
{\mathcal{B}}=\Bigl\{\,(x,t)\in\R^2\times[T+\mu_2^{-2}\tt_2,T)\,\,\Bigl\vert\,\, \vert x -x_0\vert< 2\frac{2-\delta}{\sqrt{-2\tt_1}}\sqrt{2(T-t)}\,\,\Bigr\}\,,
$$
The union of the sets ${\mathcal{A}}$ and ${\mathcal{B}}$ clearly contains the set
$$
B_{(2+\delta)/\mu_2}(x_0)\times[T+\mu_2^{-2}\tt_2,T]\,\,\,\text{\large{$\setminus$}}\,\,\bigl\{(x_0,T)\bigr\}\,,
$$
hence, in other words, for every $\varepsilon>0$ there exists a radius $R_\varepsilon>0$ and a time $t_\varepsilon<T$ such that the curve $\gamma_t$ in the ball $B_{R_\varepsilon}(x_0)$ can be written as a graph with $C^1$--norm less than $\varepsilon$, for every $t\in[t_\varepsilon,T)$. Moreover, this also holds for the limit curve $\gamma_T$ on the union
$$
\bigcup_{\mu>\mu_2}\bigl(B_{(2+\delta)/\mu}(x_0)\setminus B_{(2-\delta)/\mu}(x_0)\bigr)=B_{(2+\delta)/\mu_2}(x_0)\setminus\{x_0\}\,.
$$
This fact, recalling that the inner unit tangent vector of the curve $\gamma_t$ at its end--point (the one going to $x_0$) converges to the direction of $H$, as $t\to T$, clearly shows that, locally around $x_0$, we can write $\gamma_T$ as a graph of a function over $x_0+\overline{H}$ whose $C^1$--norm decays like $o(1)$, as the distance from $x_0$ goes to zero.\\
In particular, we conclude that all the curves $\gamma^i_T$, hence the limit network $\SS_T$, are of class $C^1$ and that all the sequences of curves $\gamma^i_t$ converge in $C^1$ to $\gamma_T^i$ (possibly after reparametrization in arclength).

Arguing similarly for the curvature by means of Remark~\ref{rmks:graph_local}, we have that the curvature of the curve $\gamma_0^\mu$ in $B_{2+\delta/2}(0)\setminus B_{2-\delta/2}(0)$ is smaller than any $m>0$, if we choose $\mu$ large enough, say $\mu>\mu_3\geqslant\mu_2$. It follows, rescaling back, that 
$$
\mu^{-2}\sup_{\SS_T\cap B_{(2+\delta/2)/\mu}(x_0)\setminus B_{(2-\delta/2)/\mu}(x_0)} k^2<m\,,
$$
for every $\mu>\mu_3$. This implies that the curvature of $\SS_T$ is of order $o(1/r)$, as $r\to 0$, where $r$ is the distance from the multi--point $x_0\in\SS_T$.

Finally, $\SS_T$ cannot have two concurring curve
at a multi--point with the same unit tangent, since this would imply
that the limit shrinker $\widetilde\SS_\infty$ had halflines of
multiplicity larger than one.
\end{proof}

It follows by this proposition that the networks $\SS_t$ converge in $C^1(U)$ to a degenerate regular network $\widehat{\SS}_T$ having $\SS_T$ as non--collapsed part, with underlying graph homeomorphic to $\SS_t$ and core given by the collapsing subnetwork ${\mathbb{M}_t}$.

\begin{rem}\label{gjump} Notice that the limit Gaussian density $\widehat{\Theta}(x_0)=\widehat{\Theta}(x_0,T)$ (see Definition~\ref{Gaussiandensities}) at $x_0$ (and time $T$) of the flow $\SS_t$ is the Gaussian density of the blow--up limit shrinker $\widetilde{\SS}_\infty=\SS^\infty_{-1/2}$ and can be different from the number of curves of $\SS_T$ concurring at $x_0$, divided by two. This does not happen when the network $\SS_t$ is a tree in a neighborhood of $x_0$, for $t$ close enough to $T$ and the singularity is given by the collapsing of a single curve producing a $4$--point with angles of $60/120$ degrees between the four concurring curves, as described in Proposition~\ref{bdcurvcollapse} (after applying Proposition~\ref{prop999}), in such case the blow--up limit shrinker is a standard cross and the limit Gaussian density $\widehat{\Theta}(x_0,T)$ is clearly equal to two.
\end{rem}

We actually expect that the curvature of the curves in $\NN_t$ and of $\SS_T$ is bounded, not only of order $o(1/r)$, close to the non--regular multi--points.

\begin{oprob}\label{ooo1115}\ 
\begin{itemize}
\item The curvature of $\SS_T$ is bounded?
\item The curvature of the subnetwork $\NN_t$ is locally uniformly bounded around $x_0$, as $t\to T$?
\end{itemize}
\end{oprob}

We can finally describe the local behavior of the whole network $\SS_t$, as $t\to T$, around a point $x_0\in\Omega$ where $\SS_t$ is not a tree for $t$ close enough to $T$.

\begin{thm}\label{loopsing2} 
Let $x_i=\lim_{t\to T}O^i(t)\in\overline{\Omega}$, for $i\in\{1,2,\dots,m\}$, 
and let $x_0$ one of such points such that
$x_0\in \Omega$ and the blow--up limit at $x_0$, as $t\to T$, is not a line, a standard triod or a standard cross. Then, under the uniqueness assumption {\bf{U}} 
and the multiplicity--one conjecture {\bf{M1}}, there exists a $C^1$, possibly non--regular network $\SS_T$ in a neighborhood $U$ of $x_0$, which is smooth in $U\setminus\{x_0\}$ and whose curvature is of order $o(1/r)$, as $r\to 0$, where $r$ is the distance from $x_0$, such that 
$$
\NN_t\to\SS_T \,\,\textrm{in}\,\, C^1\loc(U)\quad\text{ and }\quad
\SS_t\to\SS_T \,\,\textrm{in}\,\, C^\infty\loc(U\setminus\{x_0\})\,,
$$
where $\NN_t$ is the subnetwork of the non--collapsing curves of $\SS_t$.\\ 
Moreover, at the multi--point $x_0$ of $\SS_T$ any two concurring curves cannot have the same exterior unit tangent vectors.

The network $\SS_T$ is the non--collapsed part of a $C^1$ degenerate regular network $\widehat{\SS}_T$ in $U$ with underlying graph homeomorphic to $\SS_t$ and core given by the collapsed subnetwork ${\mathbb{M}_t}$, which is the $C^1$--limit of $\SS_t$, as $t\to T$.
\end{thm}

\begin{rem}\label{are}
It is easy to see that, thanks to the uniformly bounded length ratios of $\SS_t$, the one--dimensional Hausdorff measures associated to $\SS_t$ weakly--converge (as measures) to the one--dimensional Hausdorff measure associated to $\SS_T$ (see Remark~\ref{remreg}).
\end{rem}

\subsection{Continuing the flow}\label{resum}
We summarize in the following two theorems the behavior of the evolving regular network at a singular time, worked out in the previous sections, assuming the multiplicity--one conjecture~\ref{ooo9} and the uniqueness assumption~\ref{ooo12}.

\begin{thm}\label{ppp2}
If {\bf{M1}} is true and the uniqueness assumption {\bf{U}} holds, 
then the (possibly simultaneous) singularities, as $t\to T$, of the curvature flow of a regular network $\SS_t$ in a strictly convex, open subset $\Omega\subseteq\R^2$ are locally given by: 
\begin{itemize} 
\item the ``isolated'' collapse with bounded curvature of a ``boundary curve'' getting to a fixed end--point on $\partial\Omega$ (regions cannot collapse to boundary point); indeed, around any end--point $P^r$ either the flow is smooth, or the curve of $\SS_t$ getting to $P^r$ collapses letting two concurring curves forming an angle of $120$ degrees at such end--point;

\item the collapse with bounded curvature of an ''isolated'' curve with the formation of a regular $4$--point, locally around a point $x_0\in\Omega$;

\item the collapse with unbounded curvature locally around a point $x_0\in\Omega$ of a group of bounded regions (each one of them with less than six edges), producing a possibly non--regular multi--point.
\end{itemize}
If $\{y_1,y_2,\dots,y_n,z_1,z_2,\dots,z_m\}$ are the points of $\overline{\Omega}$ where such singularities occur (which are a subset of the limits, as $t\to T$, of the $3$--points of $\SS_t$), where we denoted with $y_i$ 
the ``cross'' or ``boundary'' singularities and with $z_j$ the other singularities, then there exists a possibly non--regular $C^1$ limit network $\SS_T$ such that:

\begin{itemize}

\item the network $\SS_t$ converges locally in $C^1$ to $\widehat{\SS}_T$ in $\overline{\Omega}$, as $t\to T$, where $\widehat{\SS}_T$ is a degenerate regular network having $\SS_T$ as non--collapsed part, moreover, the network $\SS_t$ converges locally smoothly to $\SS_T$ in $\overline{\Omega}\setminus\{z_1,z_2,\dots,z_m\})$;

\item the non--collapsing subnetwork $\NN_t$ of $\SS_t$ converges locally in $C^1$ to $\SS_T$ in $\overline{\Omega}$, as $t\to T$, moreover, the convergence is locally smooth in $\overline{\Omega}\setminus\{z_1,z_2,\dots,z_m\})$;

\item the network $\SS_T$ is smooth in $\overline{\Omega}\setminus\{z_1,z_2,\dots,z_m\})$;

\item every two concurring curves at a multi--point of $\SS_T$ have distinct exterior unit tangent vectors;

\item the curvature of $\SS_T$ is of order $o(1/r)$, as $r\to 0$, where $r$ is the distance from the set of points $\{z_i\}$.
\end{itemize}
\end{thm}

The case of a tree is special (for instance, the uniqueness assumption {\bf{U}} is not needed
in this case).

\begin{thm}\label{ppp1} 
If {\bf{M1}} is true and the evolving regular network $\SS_t$ is a tree (or no regions are collapsing, as $t\to T$), then the curvature is uniformly bounded and the only possible singularities, as $t\to T$, are given by the collapses of ``isolated'' curves in $\Omega$, producing a regular $4$--point or the collapse of some ``boundary curves'' getting to the fixed end--points of the network, letting two concurring curves forming at such end--point an angle of $120$ degrees.
The network $\SS_t$ converges locally smoothly with uniformly bounded curvature to a degenerate regular network $\widehat{\SS}_T$ in $\overline{\Omega}$, as $t\to T$, having a network $\SS_T$ as non--collapsed part, composed of smooth curves with distinct exterior unit tangents at the multi--points. Such multi--points can be only regular $3$--points and regular $4$--points in $\Omega$ and end--points on $\partial\Omega$ with two concurring curves forming an angle of 120 degrees between their exterior unit tangents. Clearly. the non--collapsing subnetwork of $\SS_t$ converges locally smoothly to $\SS_T$, as $t\to T$.
\end{thm}

The next step, after this description, is to understand how the flow can continue after a singular time. There are clear situations where the flow simply ends, for instance if all the network collapses to a single point, like a circle shrinks down to a point in the evolution of a closed embedded single curve, see for instance the following example.

\begin{figure}[H]
\begin{center}
\begin{tikzpicture}[scale=1]
\draw[color=black!40!white,rotate=90,shift={(0,-5.1)}]
(-0.05,2.65)to[out= -90,in=150, looseness=1] (0.17,2.3)
(0.17,2.3)to[out= -30,in=100, looseness=1] (-0.12,2)
(-0.12,2)to[out= -80,in=40, looseness=1] (0.15,1.7)
(0.15,1.7)to[out= -140,in=90, looseness=1.3](0,1.1)
(0,1.1)--(-.2,1.35)
(0,1.1)--(+.2,1.35);
\draw[shift={(-0.75,0)}]
(0,0)
to[out= 180,in=0, looseness=1] (-1,0) 
(0,0)
to[out= 60,in=-120, looseness=1] (0.5,0.86) 
(0,0)
to[out= -60,in=120, looseness=1] (0.5,-0.86)
(0.5,0.86) 
to[out= 0,in=90, looseness=1](1.5,0)
to[out= -90,in=0, looseness=1] (0.5,-0.86)
(0.5,0.86) 
to[out= 120,in=30, looseness=1](-0.75,1.29)
to[out= -150,in=120, looseness=1] (-1,0)
(0.5,-0.86) 
to[out= -120,in=-30, looseness=1](-0.75,-1.29)
to[out= 150,in=-120, looseness=1] (-1,0);
\draw[color=black,scale=1,domain=-3.141: 3.141,
smooth,variable=\t,shift={(-0.75,0)},rotate=0]plot({2.*sin(\t r)},
{2.*cos(\t r)}); 
\draw[color=black,scale=1,domain=-3.141: 3.141,
smooth,variable=\t,shift={(7.31,0)},rotate=0]plot({2.*sin(\t r)},
{2.*cos(\t r)}); 
\path[font=\small] 
(-0.75,0)node[right]{$O^1$}
(-0.25,1.1)node[right]{$O^2$}
 (-1.75,0) node[left] {$O^3$}
(-0.25,-1)node[right]{$O^4$};
 \path[font=\small] 
(7.31,0.05)node[above]{$O^1=O^2=O^3=O^4$};
\path[font=\small,rotate=90]
(-1,-3.28) node[above]{$t\to T$}
(2,-1.47) node[below]{$\SS_t$}
(2,-9.53) node[below]{$\SS_T$};
\fill(7.31,0) circle (2pt);
\end{tikzpicture}
\end{center}
\begin{caption}{A {\em Mercedes--Benz shrinker} (see the Appendix)
collapsing to a single point.}
\end{caption}
\end{figure}

In other situations how the flow should continue is easy to guess or define. For instance, the case when a part of the network collapses forming a 2--point, that can be also seen simply as an interior corner point of a single curve (see the following figure).

\begin{figure}[H]
\begin{center}
\begin{tikzpicture}[scale=1]
\draw[color=black!40!white,rotate=90,shift={(0,-5.1)}]
(-0.05,2.65)to[out= -90,in=150, looseness=1] (0.17,2.3)
(0.17,2.3)to[out= -30,in=100, looseness=1] (-0.12,2)
(-0.12,2)to[out= -80,in=40, looseness=1] (0.15,1.7)
(0.15,1.7)to[out= -140,in=90, looseness=1.3](0,1.1)
(0,1.1)--(-.2,1.35)
(0,1.1)--(+.2,1.35);
\draw[shift={(1.2,0.2)}] 
(-1.73,-1.8) 
to[out= 180,in=180, looseness=1] (-2.8,0) 
to[out= 60,in=150, looseness=1.5] (-1.5,1) 
(-2.8,0)
to[out=-60,in=180, looseness=0.9] (-1.25,-0.75)
(-1.5,1)
to[out= -30,in=90, looseness=0.9] (-1,0)
to[out= -90,in=60, looseness=0.9] (-1.25,-0.75)
to[out= -60,in=0, looseness=0.9](-1.73,-1.8);
\draw[color=black,scale=1,domain=-3.141: 3.141,
smooth,variable=\t,shift={(-0.75,0)},rotate=0]plot({2.*sin(\t r)},
{2.*cos(\t r)}); 
\draw[color=black,scale=1,domain=-3.141: 3.141,
smooth,variable=\t,shift={(7.31,0)},rotate=0]plot({2.*sin(\t r)},
{2.*cos(\t r)}); 
\path[font=\small, shift={(1.2,0.2)}] 
(-1.25,-0.75)node[right]{$O^2$}
 (-1.5,-0.3)[left] node{$\gamma^2$}
 (-0.72,1.1)[left] node{$\gamma^1$}
 (-0.7,-1.45)[left] node{$\gamma^3$}
 (-3,0.65) node[below] {$O^1$}; 
\draw[shift={(9,0)}] 
(-1.73,-1.8) to[out= 180,in=180, looseness=1.5] (-2,.2) 
(-2,.2)to[out= -30,in=0, looseness=0.9](-1.73,-1.8);
\path[font=\small, shift={(9,0)}] 
(-2,.2)node[above]{$O^1=O^2$}
 (-0.7,-1.45)[left] node{$\gamma^3$};
\path[font=\small,rotate=90]
(-1,-3.28) node[above]{$t\to T$}
(2,-1.47) node[below]{$\SS_t$}
(2,-9.53) node[below]{$\SS_T$};
\fill(7,.2) circle (2pt);
\end{tikzpicture}
\end{center}
\begin{caption}{Collapse of both curves $\gamma^1$, $\gamma^2$ and
 the region they enclose to the point $O^1=O^2$, leaving a
 closed curve $\gamma^3$, possibly with a corner at $O^1=O^2$.\label{angfig}}
\end{caption}
\end{figure}

Here, we can restart the flow by means of the work of Angenent~\cite{angen3} where the evolution of curves 
with corners is treated (see Remark~\ref{gremh}). In general, one would need an analogue of the short--time existence Theorem~\ref{c2shorttime} for networks with 2--points or with curves with corners. This will be actually a particular case of Theorem~\ref{evolnonreg} in the next section (see the beginning of Section~\ref{stesecsub}).

Instead, a situation that really needs a ``decision'' about whether and how the flow should continue after the singularity is depicted in the following figures.

\begin{figure}[H]
\begin{center}
\begin{tikzpicture}
\draw[color=black!40!white,rotate=90,shift={(0,-5.1)}]
(-0.05,2.65)to[out= -90,in=150, looseness=1] (0.17,2.3)
(0.17,2.3)to[out= -30,in=100, looseness=1] (-0.12,2)
(-0.12,2)to[out= -80,in=40, looseness=1] (0.15,1.7)
(0.15,1.7)to[out= -140,in=90, looseness=1.3](0,1.1)
(0,1.1)--(-.2,1.35)
(0,1.1)--(+.2,1.35);
\draw[color=black,scale=1,domain=-3.141: 3.141,
smooth,variable=\t,shift={(-0.75,0)},rotate=0]plot({2.*sin(\t r)},
{2.*cos(\t r)}); 
\draw[color=black,scale=1,domain=-3.141: 3.141,
smooth,variable=\t,shift={(7.31,0)},rotate=0]plot({2.*sin(\t r)},
{2.*cos(\t r)}); 
\draw[rotate=90,scale=0.985, black, shift={(0,-2.5)}]
(-0.05,2.65)to[out= 30,in=140, looseness=1] (1.25,1.66)
(-0.05,2.65)to[out= 150,in=60, looseness=1] (-1.35,1.77);
\draw[rotate=90,scale=0.985, black, shift={(0,-2.5)}]
(-0.05,2.65)to[out= -90,in=150, looseness=1] (0.17,2.3)
(0.17,2.3)to[out= -30,in=100, looseness=1] (-0.12,2)
(-0.12,2)to[out= -80,in=40, looseness=1] (0.15,1.7)
(0.15,1.7)to[out= -140,in=90, looseness=1](0,1.24);
\path[font=\small]
(1.2,0) node[right]{$P^r$};
\draw[rotate=90,dashed, scale=0.75,color=black!40!white,shift={(0,-15.08)}]
(0,2.65)--(1.73,3.65)
(0,2.65)--(-1.73,3.65);
\draw[rotate=90,scale=0.985, black, shift={(0,-10.7)}]
(0,1.25)to[out= 30,in=140, looseness=1] (1.25,1.66)
(0,1.25)to[out= 150,in=60, looseness=1] (-1.35,1.77);
\fill(9.3,0) circle (2pt);
\path[font=\small]
(9.35,0) node[right]{$P^r$};
\path[font=\small,rotate=90]
(-1,-3.28) node[above]{$t\to T$}
(2,-1.47) node[below]{$\SS_t$}
(2,-9.53) node[below]{$\SS_T$};
\end{tikzpicture}
\end{center}
\begin{caption}{A limit network with two curves arriving at the same
 end--point on $\partial\Omega$.\label{boundarysing}}
\end{caption}
\end{figure}

\begin{figure}[H]
\begin{center}
\begin{tikzpicture}
\draw[color=black!40!white,rotate=90,shift={(0,-5.1)}]
(-0.05,2.65)to[out= -90,in=150, looseness=1] (0.17,2.3)
(0.17,2.3)to[out= -30,in=100, looseness=1] (-0.12,2)
(-0.12,2)to[out= -80,in=40, looseness=1] (0.15,1.7)
(0.15,1.7)to[out= -140,in=90, looseness=1.3](0,1.1)
(0,1.1)--(-.2,1.35)
(0,1.1)--(+.2,1.35);
\draw[color=black,scale=1,domain=-3.141: 3.141,
smooth,variable=\t,shift={(-0.75,0)},rotate=0]plot({2.*sin(\t r)},
{2.*cos(\t r)}); 
\draw[color=black,scale=1,domain=-3.141: 3.141,
smooth,variable=\t,shift={(8.31,0)},rotate=0]plot({2.*sin(\t r)},
{2.*cos(\t r)}); 
\draw[shift={(1,0)}] 
(-3.73,0) 
to[out= 50,in=180, looseness=1] (-3.1,0) 
to[out= 60,in=150, looseness=1.5] (-1,1) 
(-3.1,0)
to[out= -60,in=-120, looseness=0.9] (-0.75,-0.75)
(-1,1)
to[out= -30,in=90, looseness=0.9] (-0.5,0)
to[out= -90,in=60, looseness=0.9] (-0.75,-0.75);
\path[font=\small, shift={(1,0)}]
 (-3.7,0) node[left]{$P^1$}
 (-3.4,0.0) node[below] {$\gamma^1$}
 (-0.5,-0.1)[left] node{$\gamma^2$}
 (-2.7,-0.2) node[above] {$O^1$}; 
\draw[shift={(10.03,0)}] 
(-3.7,0) 
to[out= 60,in=150, looseness=1.5] (-1.75,0.8) 
(-3.7,0)
to[out= -60,in=-120, looseness=0.9] (-1.5,-0.75)
(-1.75,0.8)
to[out= -30,in=90, looseness=0.9] (-1.25,0)
to[out= -90,in=60, looseness=0.9] (-1.5,-0.75);
\draw[dashed, color=black!40!white,shift={(1,0)}]
(5.33, 0)--(6.08,1.297)
(5.33, 0)--(6.08,-1.297);
\path[font=\small, shift={(10.03,0)}]
(-3.7,0) node[left]{$P^1=O^1$}
(-1.3,-0.1)[left] node {$\gamma^2$};
\fill(6.33,0) circle (2pt);
\path[font=\small,rotate=90]
(-1,-3.28) node[above]{$t\to T$}
(2,-1.47) node[below]{$\SS_t$}
(2,-10.53) node[below]{$\SS_T$};
\end{tikzpicture}
\end{center}
\begin{caption}{Collapse of the curve $\gamma^1$ leaving a
 closed curve $\gamma^2$ with an angle of 120 degrees at an end--point.\label{spoonPcollap}}
\end{caption}
\end{figure}
\noindent One can decide that the flow stops at $t=T$ or that the curves become extremal curves of a new network that must have, for every $t>T$, a fixed end in the end--point $P^r$ (this would require some analogs of the short--time existence Theorem~\ref{c2shorttime} for this class of non--regular networks, which are actually possible to be worked out). Anyway, the subsequent analysis becomes more troublesome because of such concurrency at the same end--point,
indeed, it should be allowed that, at some time $t>T$, a new curve and
a new $3$--point possibly ``emerges'' from such end--point (it would be needed a ``boundary'' extension of Theorem~\ref{evolnonreg} in the next section).

Another situation that also needs a decision, which in this
case is easier, is described in the following figures. 
\begin{figure}[H]
\begin{center}
\begin{tikzpicture}[scale=1]
\draw[color=black!40!white,rotate=90,shift={(0,-5.1)}]
(-0.05,2.65)to[out= -90,in=150, looseness=1] (0.17,2.3)
(0.17,2.3)to[out= -30,in=100, looseness=1] (-0.12,2)
(-0.12,2)to[out= -80,in=40, looseness=1] (0.15,1.7)
(0.15,1.7)to[out= -140,in=90, looseness=1.3](0,1.1)
(0,1.1)--(-.2,1.35)
(0,1.1)--(+.2,1.35);
\draw[color=black,scale=1,domain=-3.141: 3.141,
smooth,variable=\t,shift={(-0.75,0)},rotate=0]plot({2.*sin(\t r)},
{2.*cos(\t r)}); 
\draw[color=black,scale=1,domain=-3.141: 3.141,
smooth,variable=\t,shift={(7.31,0)},rotate=0]plot({2.*sin(\t r)},
{2.*cos(\t r)}); 
\draw[shift={(0.97,0)}]
(-2,0) 
to[out= 170,in=40, looseness=1] (-2.9,1.2) 
to[out= -140,in=90, looseness=1] (-3.2,0)
(-2,0)
to[out= -70,in=0, looseness=1] (-3,-0.9) 
to[out= -180,in=-90, looseness=1] (-3.2,0)
(-2,0) 
to[out= 50,in=180, looseness=1] (-1.3,0) 
to[out= 60,in=150, looseness=1.5] (-0.75,1) 
(-1.3,0)
to[out= -60,in=-120, looseness=0.9] (-0.5,-0.75)
(-0.75,1)
to[out= -30,in=90, looseness=0.9] (-0.25,0)
to[out= -90,in=60, looseness=0.9] (-0.5,-0.75);
\path[font=\small, shift={(0.97,0)}]
(-1.90,-0.2) node[left]{$O^1$}
(-1.3,0)node[right]{$O^2$}
(-2.8,0.8) node[below] {$\gamma^1$}
(-1.3,0.35)[left] node{$\gamma^2$}
(0,-1)[left] node{$\gamma^3$};
\draw[shift={(9.03,0)}] 
(-1.85,0) to[out= 170,in=40, looseness=1] (-2.9,1.2) 
to[out= -140,in=90, looseness=1] (-3.2,0)
(-1.85,0)
to[out= -70,in=0, looseness=1] (-3,-0.9) 
to[out= -180,in=-90, looseness=1] (-3.2,0)
(-1.85,0) to[out= 50,in=160, looseness=2] (-.9,0);
\path[font=\small,shift={(9.03,0)}]
(-.6,.1) node[below]{$O^2$}
(-1.75,-.25) node[left]{$O^1$}
(-1.3,0.25)[above] node{$\gamma^2$}
(-2.8,0.8) node[below] {$\gamma^1$};
\path[font=\small,rotate=90]
(-1,-3.28) node[above]{$t\to T$}
(2,-1.47) node[below]{$\SS_t$}
(2,-9.53) node[below]{$\SS_T$};
\fill(8.13,0) circle (2pt);
\end{tikzpicture}
\end{center}
\begin{caption}{Collapse of the curves $\gamma^3$ and
 the region enclosed to the point $O^3$ leaving a
 curve $\gamma^2$ with a 1--point as an end--point.\label{f21}}
\end{caption}
\end{figure}
\begin{figure}[H]
\begin{center}
\begin{tikzpicture}[scale=1]
\draw[color=black!40!white,rotate=90,shift={(0,-5.1)}]
(-0.05,2.65)to[out= -90,in=150, looseness=1] (0.17,2.3)
(0.17,2.3)to[out= -30,in=100, looseness=1] (-0.12,2)
(-0.12,2)to[out= -80,in=40, looseness=1] (0.15,1.7)
(0.15,1.7)to[out= -140,in=90, looseness=1.3](0,1.1)
(0,1.1)--(-.2,1.35)
(0,1.1)--(+.2,1.35);
\draw[color=black,scale=1,domain=-3.141: 3.141,
smooth,variable=\t,shift={(-0.75,0)},rotate=0]plot({2.*sin(\t r)},
{2.*cos(\t r)}); 
\draw[color=black,scale=1,domain=-3.141: 3.141,
smooth,variable=\t,shift={(7.31,0)},rotate=0]plot({2.*sin(\t r)},
{2.*cos(\t r)}); 
\draw[shift={(1,0)}] 
(-3.73,0) 
to[out= 50,in=180, looseness=1] (-2.3,0) 
to[out= 60,in=150, looseness=1.5] (-1,1) 
(-2.3,0)
to[out= -60,in=-120, looseness=0.9] (-0.75,-0.75)
(-1,1)
to[out= -30,in=90, looseness=0.9] (-0.5,0)
to[out= -90,in=60, looseness=0.9] (-0.75,-0.75);
\path[font=\small, shift={(1,0)}]
 (-3.7,0) node[left]{$P^1$}
 (-2.9,0.75) node[below] {$\gamma^1$}
 (-0.8,-0.5)[left] node{$\gamma^2$}
 (-2.5,0) node[below] {$O^1$}; 
\draw[shift={(9.03,0)}] 
(-3.73,0) to[out= 50,in=180, looseness=1] (-1.8,0); 
\path[font=\small, shift={(9.03,0)}]
 (-3.7,0) node[left]{$P^1$}
 (-2.9,0.83) node[below] {$\gamma^1$}
 (-1.4,0) node[below] {$O^1$}; 
\path[font=\small,rotate=90]
(-1,-3.28) node[above]{$t\to T$}
(2,-1.47) node[below]{$\SS_t$}
(2,-9.53) node[below]{$\SS_T$};
\fill(7.23,0) circle (2pt);
\end{tikzpicture}
\end{center}
\begin{caption}{Collapse of the curves $\gamma^2$ and
 the region enclosed to the point $O^1$ leaving a
 curve $\gamma^1$ with a 1--point as an end--point.\label{f22}}
\end{caption}
\end{figure}
\noindent If the limit network $\SS_T$ contains a curve (or curves) which ends in a 1--point, it is actually natural to impose that such curve vanishes for every future time, so considering only the evolution of the network of the rest of the network $\SS_T$ according to the above discussion (cutting away such a curve will produce a 2--point or the empty set, in the figures above, for instance).

\medskip

Theorem~\ref{evolnonreg} in the next section will give a way to restart the flow in the ``nice'' singularity situation described in Theorem~\ref{ppp1}, when the curvature remains bounded and a single curve collapses to an interior point of $\Omega$ 
forming a non--regular network with a regular $4$--point..
\begin{figure}[H]
\begin{center}
\begin{tikzpicture}
\draw[color=black!40!white,rotate=90,shift={(0,-5.1)}]
(-0.05,2.65)to[out= -90,in=150, looseness=1] (0.17,2.3)
(0.17,2.3)to[out= -30,in=100, looseness=1] (-0.12,2)
(-0.12,2)to[out= -80,in=40, looseness=1] (0.15,1.7)
(0.15,1.7)to[out= -140,in=90, looseness=1.3](0,1.1)
(0,1.1)--(-.2,1.35)
(0,1.1)--(+.2,1.35);
\draw[color=black,scale=1,domain=-3.141: 3.141,
smooth,variable=\t,shift={(-0.75,0)},rotate=0]plot({2.*sin(\t r)},
{2.*cos(\t r)}); 
\draw[color=black,scale=1,domain=-3.141: 3.141,
smooth,variable=\t,shift={(7.31,0)},rotate=0]plot({2.*sin(\t r)},
{2.*cos(\t r)}); 
\path[font=\small,rotate=90]
(-1,-3.28) node[above]{$t\to T$}
(2,-1.47) node[below]{$\SS_t$}
(2,-9.53) node[below]{$\SS_T$}
(.25,-7.38) node[right]{$O^1=O^2$}
(-.45,2) node[right]{$O^1$}
(.5,.13) node[right]{$O^2$};
\path[rotate=90,font=\small,scale=1.5]
(.2,.7) node[right]{$\gamma$};
\draw[black, shift={(1.17,0)}, scale=1.1]
 (-3.03,1.25) 
to[out= -50,in=180, looseness=1]
(-2.4,-0.2)
to[out= -60,in=50, looseness=1.5] (-3.03,-1.25)
(-2.4,-0.2)
to[out= 60,in=-130, looseness=0.9] (-1.2,0.3)
to[out= 110,in=-100, looseness=0.9](-0.43,1.25) 
(-1.2,0.3)
to[out=-10,in=150, looseness=0.9](-0.43,-1.25);
\draw[black, shift={(9.18,0)}, scale=1.1]
 (-3.03,1.25) 
to[out= -50,in=120, looseness=1] (-1.75,0.15)
to[out= 60,in=-100, looseness=1.5] (-0.43,1.25) 
(-1.75,0.15)
to[out= -60,in=150, looseness=0.9](-0.43,-1.25) 
(-1.75,0.15)
to[out=-120,in=50, looseness=0.9] (-3.03,-1.25) ;
\draw[rotate=90,scale=0.75,color=black!40!white,shift={(0.2,-12.32)},dashed]
(0,2.65)--(1.73,3.65)
(0,2.65)--(1.73,1.65)
(0,2.65)--(-1.73,3.65)
(0,2.65)--(-1.73,1.65);
\fill(7.25,.18) circle (1.85pt);
\end{tikzpicture}
\end{center}
\begin{caption}{A limit ``nice'' collapse of a single curve $\gamma$ producing a non--regular network $\SS_T$.\label{nice}}
\end{caption}
\end{figure}

Finally if we are in the situation of a non--regular limit network $\SS_T$ described by Theorem~\ref{ppp2}, after the collapse of a region of $\SS_t$, as $t\to T$ (see for instance the following figures), in order to restart the flow one will need either an
extension of Theorem~\ref{evolnonreg} (mentioned in Remark~\ref{felix})
or an improvement of Proposition~\ref{loopsing} 
(the curvature of the non--degenerate limit curves is bounded).
\begin{figure}[H]
\begin{center}
\begin{tikzpicture}
\draw[color=black!40!white,rotate=90,shift={(0,-5.1)}]
(-0.05,2.65)to[out= -90,in=150, looseness=1] (0.17,2.3)
(0.17,2.3)to[out= -30,in=100, looseness=1] (-0.12,2)
(-0.12,2)to[out= -80,in=40, looseness=1] (0.15,1.7)
(0.15,1.7)to[out= -140,in=90, looseness=1.3](0,1.1)
(0,1.1)--(-.2,1.35)
(0,1.1)--(+.2,1.35);
\draw[color=black,scale=1,domain=-3.141: 3.141,
smooth,variable=\t,shift={(-0.75,0)},rotate=0]plot({2.*sin(\t r)},
{2.*cos(\t r)}); 
\draw[color=black,scale=1,domain=-3.141: 3.141,
smooth,variable=\t,shift={(7.31,0)},rotate=0]plot({2.*sin(\t r)},
{2.*cos(\t r)}); 
\path[font=\small,rotate=90]
(-1,-3.28) node[above]{$t\to T$}
(2,-1.47) node[below]{$\SS_t$}
(2,-9.53) node[below]{$\SS_T$};
\draw[color=black,shift={(-1.2,-0.1)},scale=2]
(-.150,0.13)to[out= -101,in=90, looseness=1](-.165,0)
(-.165,0)to[out= -90,in=101, looseness=1](-.150,-0.13)
(0.725,0)to[out= 0,in=100, looseness=1.5] (1.12,-.4)
(-.150,0.13)to[out= 150,in=-90, looseness=1.5](-.5,.73)
(-.150,-0.13)to[out= -150,in=90, looseness=1](-.6,-.52);
\draw[color=black,shift={(-1.2,-0.1)},scale=2]
(-.150,0.13)to[out= 20,in=180, looseness=1](0.325,0.2)
(0.325,0.2)to[out= 0,in=120, looseness=1] (0.725,0)
(-.150,-0.13)to[out= -20,in=180, looseness=1](0.325,-0.2)
(0.325,-0.2)to[out= 0,in=-120, looseness=1] (0.725,0);
\draw[color=black,shift={(6.86,-0.1)},scale=2]
(.19,.135)to[out= -20,in=100, looseness=1.5] (1.12,-.4)
(.19,.135)to[out= 130,in=-90, looseness=1.5](-.5,.73)
(.19,.135)to[out= -170,in=90, looseness=1](-.6,-.52);
\fill(7.25,.18) circle (2pt);
\end{tikzpicture}

\bigskip

\bigskip

\bigskip

\begin{tikzpicture}
\draw[color=black!40!white,rotate=90,shift={(0,-5.1)}]
(-0.05,2.65)to[out= -90,in=150, looseness=1] (0.17,2.3)
(0.17,2.3)to[out= -30,in=100, looseness=1] (-0.12,2)
(-0.12,2)to[out= -80,in=40, looseness=1] (0.15,1.7)
(0.15,1.7)to[out= -140,in=90, looseness=1.3](0,1.1)
(0,1.1)--(-.2,1.35)
(0,1.1)--(+.2,1.35);
\draw[color=black,scale=1,domain=-3.141: 3.141,
smooth,variable=\t,shift={(-0.75,0)},rotate=0]plot({2.*sin(\t r)},
{2.*cos(\t r)}); 
\draw[color=black,scale=1,domain=-3.141: 3.141,
smooth,variable=\t,shift={(7.31,0)},rotate=0]plot({2.*sin(\t r)},
{2.*cos(\t r)}); 
\path[font=\small,rotate=90]
(-1,-3.28) node[above]{$t\to T$}
(2,-1.47) node[below]{$\SS_t$}
(2,-9.53) node[below]{$\SS_T$};
\draw[color=black, scale=.8, shift={(-1.3,.3)}, rotate=-30]
(0.76,0)to[out= 120,in=-48, looseness=1](0.23,0.72)
to[out= -162,in=24, looseness=1](-0.61,0.44)
to[out= -96,in=96, looseness=1](-0.61,-0.44)
(0.76,0)to[out= -120,in=48, looseness=1](0.23,-0.72)
to[out= 162,in=-24, looseness=1](-0.61,-0.44);
\draw[color=black,shift={(-1.2,-0.1)},scale=2]
(.347,.017)to[out= -30,in=100, looseness=1.5] (1.12,-.4)
(.305,.375)to[out= 42,in=-170, looseness=1.5](.88,.8)
(-.045,.443)to[out= 114,in=-100, looseness=1.5](-.5,.73)
(-.216,.138)to[out= 186,in=20, looseness=1.5](-.6,-.52)
(.018,-.128)to[out= -102,in=130, looseness=1.5](.4,-.93);
\draw[color=black,shift={(6.86,-0.1)},scale=2]
(.19,.135)to[out= -20,in=100, looseness=1.5] (1.12,-.4)
(.19,.135)to[out= 52,in=-170, looseness=1.5](.88,.8)
(.19,.135)to[out= 124,in=-100, looseness=1.5](-.5,.73)
(.19,.135)to[out= 196,in=20, looseness=1.5](-.6,-.52)
(.19,.135)to[out= -92,in=130, looseness=1.5](.4,-.93);
\fill(7.25,.18) circle (2pt);
\end{tikzpicture}
\end{center}
\begin{caption}{Less ``nice'' examples of collapse and convergence to non--regular networks.\label{lessnice}}
\end{caption}
\end{figure}

We conclude this section by discussing the (conjectural) ``generic'' situation of singularity formation, in the sense that it should happen {\em for a dense set} of initial networks.

By numerical evidence (computing the lowest relevant eigenvalue of the Jacobi--field operator of the candidates -- Dominic Descombes and Tom Ilmanen, {\em personal communication}) the {\em dynamically stable} shrinkers (meaning that ``perturbing'' the flow, the blow--up limit network remains the same) should be only the line, the unit circle, the standard triod, the standard cross, the Brakke spoon, the lens and the ``three--ray star'' (see the figure below).
\begin{figure}[H]
\begin{center}
\begin{tikzpicture}[rotate=90,scale=.4]
\draw[color=black,dashed]
(6.25,0)to[out= 0,in=180, looseness=1](7.25,0)
(-3.25,5.48)to[out=120,in=-60, looseness=1](-3.75,6.34)
(-3.25,-5.48)to[out=-120,in=60, looseness=1](-3.75,-6.34);
\draw[color=black]
(2.75,0)to[out= 0,in=180, looseness=1](6.25,0)
(-1.5,2.45)to[out= 120,in=-60, looseness=1](-3.25,5.48)
(-1.5,-2.45)to[out= -120,in=60, looseness=1](-3.25,-5.48)
(2.75,0)to[out= 120,in=0, looseness=1](-1.5,2.45)
(-1.5,-2.45)to[out= 120,in=-120, looseness=1](-1.5,2.45)
(2.75,0)to[out= -120,in=0, looseness=1](-1.5,-2.45);
\fill(0,0) circle (2.5pt);
\path
(-.1,-.4) node[above]{$O$};
\end{tikzpicture}
\end{center}
\begin{caption}{A ``three--ray star'' regular shrinker.\label{threeray}}
\end{caption}
\end{figure}

\begin{conge} The ``generic'' singularities of the curvature flow of a network are (locally) asymptotically described by one of the above shrinkers.
\end{conge}

We remark that if rescaling around a singular point $x_0$ we get one of the 
listed above shrinkers, the limit network $\SS_T$ is locally quite ``nice''.
If the shrinker is a line or a standard triod, there is no singularity. If it is a circle, it means that the flow ends at the singularity. If it is a Brakke spoon, locally the flow produces a curve with an end--point in $\Omega$ (see Figures~\ref{f21} and~\ref{f22}), which we can reasonably ``assume'' it disappears at subsequent times and we have to deal with an empty network or with a curve containing an angle (as in Figure~\ref{angfig}) that has a ``natural'' unique evolution, immediately smooth.
In the case of a standard cross, we can deal with the ``new'' $4$--point by means of Theorem~\ref{evolnonreg}. If we get a lens, $\SS_T$ will be (locally) given by two $C^1$ curves (smooth outside $x_0$) concurring at the singular point without forming an angle (even if their curvature could be unbounded, getting to $x_0$, if Problem~\ref{ooo1115} has a negative answer). Finally, if the shrinker is a three--ray star, the limit network $\SS_T$ is locally a triod at $x_0$ with angles of $120$ degrees, by Proposition~\ref{loopsing} (also, in this case, the curvature could be unbounded getting close to $x_0$). Notice that in these last two cases, even if apparently ``nice'', we have to use Theorem~\ref{evolnonreg} (and possibly its extension mentioned in Remark~\ref{felix}) in order to restart the flow, since the curves are not necessarily $C^2$ up to $x_0$.

However, we remark that in all these cases (and in particular in the most ``delicate'' ones: cross, lens and three--ray star, when we need to apply Theorem~\ref{evolnonreg}, or its extension mentioned in Remark~\ref{felix}) the associated limit network $\SS_T$ (if not empty and ``cutting'' away a curve if it ends in a $1$--point in $\Omega$) has either a regular $4$--point (with angles of $120/60$ degrees) or a regular $3$--point, or a $2$--point with no angle. In particular, the cone generated by inner unit tangent vectors of the concurring curves at such point form, respectively, is either a standard cross, a regular triod, or a line. Since, as we will see in the next section, the curvature flow produced by Theorem~\ref{evolnonreg} is associated with a regular self--similarly expanding network (see Definition~\ref{expanders}) originating from such cone, which in these special cases it is unique (see the end of Section~\ref{expsec} and Problems~\ref{ooo9999},~\ref{ooo9999st},~\ref{ooo9999s}), it is natural to expect that also the flow produced by such theorem is unique, which would give a unique ``canonical'' way to continue the flow in the (conjectural) generic situation.

\section{Short time existence III -- Non--regular networks}\label{smtm3}

 In this section we consider the problem of defining and finding a curvature flow (as smooth as
possible) starting from an initial possibly non--regular network, that is, having multiple points of order greater than three
or triple junctions where the $120$ degrees condition is not satisfied. As we have seen in the previous sections, this is naturally related to the ``restarting'' of the flow after a singularity. To deal with such problem, we clearly need a definition of solution slightly different from Definitions~\ref{probdef} and~\ref{probdef-open} in a positive time interval $[0,T)$, asking anyway that Definition~\ref{d1} still holds for every positive time.

We are going to present two short--time existence results for general networks, the first by T.~Ilmanen, A.~Neves and the last author in~\cite{Ilnevsch}, Theorem~\ref{evolnonreg}, the second by J.~Lira, M.~Mazzeo, M.~Saez and the third author in~\cite{LiMazPlSa}, Theorem~\ref{evononreg2}. Both theorems are based on the existence and the properties of the {\em self--similarly expanding} networks and provide a ``nice'' motion by curvature if the initial datum belongs to the class of non--regular networks with bounded curvature, such that at every multiple point the exterior unit tangent vectors are mutually distinct. Notice that the 
second assumption is not restrictive for the ``restarting'' problem, taking into account the conclusions of 
Theorems~\ref{ppp2} and~\ref{ppp1}.

\subsection{Self--similarly expanding networks}\label{expsec}
\begin{defn}\label{expanders} A regular $C^2$ open network $\epsi$ is called a {\em regular expander} if at every point $x\in\epsi$ there holds
\begin{equation}\label{expeq}
\underline{k}=x^\perp\,. 
\end{equation}
This relation is called the {\em expanders equation}.
\end{defn}
The name comes from the fact that if $\epsi$ is a regular expander, then $\epsi_t = \sqrt{2t}\,\epsi$ describes a {\em self--similarly expanding} curvature flow of regular networks in $(0,+\infty)$, with $\epsi=\epsi_{1/2}$. Viceversa, if $\epsi_t$ is a self--similarly expanding curvature flow of regular networks in the time interval $(0,+\infty)$, then $\epsi_{1/2}$ is a regular expander, that is, $\epsi_{1/2}$ satisfies equation~\eqref{expeq}.
\begin{figure}[H]
\begin{center}
\begin{tikzpicture}[scale=1.22]
\path[font=\small] 
(0,0)node[right]{$O$};
\draw[color=black]
(0,0.5)to[out=30,in=-120, looseness=1](0.9,1.73)
(0,0.5)to[out= 150,in=-60, looseness=1](-0.9,1.73);
\draw[color=black]
(0,-1.75)to[out=90,in=-90, looseness=1](0,0.5);
\draw[color=black,dashed]
(0,-2.25)to[out=90,in=-90, looseness=1](0,-1.75);
\draw[color=black,dashed]
(0.9,1.73)to[out=60,in=-120, looseness=1](1.2,2.24)
(-0.9,1.73)to[out=120,in=-60, looseness=1](-1.2,2.24);
\draw[color=black!40!white]
(0,0)to[out=60,in=-120, looseness=1](1,1.73)
(0,0)to[out=120,in=-60, looseness=1](-1,1.73);
\draw[color=black!40!white,dashed]
(1,1.73)to[out=60,in=-120, looseness=1](1.3,2.24)
(-1,1.73)to[out=120,in=-60, looseness=1](-1.3,2.24);
\fill[color=black]
(0,0) circle (1.23pt);
\end{tikzpicture}\qquad\qquad\qquad
\begin{tikzpicture}[scale=1.22]

\path[font=\small] 
(0,0)node[right]{$O$};
\draw[color=black]
(0,0.5)to[out=30,in=-120, looseness=1](0.9,1.73)
(0,0.5)to[out= 150,in=-60, looseness=1](-0.9,1.73)
(0,-0.5)to[out=-150,in=60, looseness=1](-0.9,-1.73)
(0,-0.5)to[out=-30,in=120, looseness=1](0.9,-1.73);
\draw[color=black]
(0,-0.5)to[out=90,in=-90, looseness=1](0,0.5);
\draw[color=black,dashed]
(0.9,1.73)to[out=60,in=-120, looseness=1](1.2,2.24)
(0.9,-1.73)to[out=-60,in=120, looseness=1](1.2,-2.24)
(-0.9,-1.73)to[out=-120,in=60, looseness=1](-1.2,-2.24)
(-0.9,1.73)to[out=120,in=-60, looseness=1](-1.2,2.24);
\draw[color=black!40!white]
(0,0)to[out=60,in=-120, looseness=1](1,1.73)
(0,0)to[out=-60,in=120, looseness=1](1,-1.73)
(0,0)to[out=-120,in=60, looseness=1](-1,-1.73)
(0,0)to[out=120,in=-60, looseness=1](-1,1.73);
\draw[color=black!40!white,dashed]
(1,1.73)to[out=60,in=-120, looseness=1](1.3,2.24)
(1,-1.73)to[out=-60,in=120, looseness=1](1.3,-2.24)
(-1,-1.73)to[out=-120,in=60, looseness=1](-1.3,-2.24)
(-1,1.73)to[out=120,in=-60, looseness=1](-1.3,2.24);
\fill[color=black]
(0,0) circle (1.23pt);
\end{tikzpicture}\qquad\qquad
\begin{tikzpicture}[rotate=126]

\path[font=\small] 
(-0.18,-0.19)node[below]{$O$};
\draw[color=black]
(0.8,0.584)to[out=96,in=-118, looseness=1](0.82,2.37)
(0.8,0.584)to[out= -24,in=170, looseness=1](2.5,0.1)
(-0.308,-0.948)to[out=-48,in=108, looseness=1](0.67,-2.37)
(-0.308,-0.948)to[out=-168,in=36, looseness=1](-1.9,-1.46);
\draw[color=black]
(-0.308,-0.948)to[out=72,in=-108, looseness=1](0,0)
(0.8,0.584)to[out=-144,in=36, looseness=1](0,0);
\draw[color=black,dashed]
(0.67,-2.37)to[out=-62,in=108, looseness=1](0.87,-2.85)
(2.5,0.1)to[out= -10,in=180, looseness=1](3,0.05)
(0.82,2.37)to[out=72,in=-108, looseness=1](0.92,2.85)
(-1.9,-1.46)to[out=-144,in=36, looseness=1](-2.32,-1.76);
\draw[color=black!40!white]
(0,0)to[out= 0,in=180, looseness=1](2.5,0)
(0,0)to[out=-72,in=108, looseness=1](0.77,-2.37)
(0,0)to[out=72,in=-108, looseness=1](0.77,2.37)
(0,0)to[out=-144,in=36, looseness=1](-2,-1.46);
\draw[color=black]
(0,0)to[out=144,in=-36, looseness=1](-1.72,1.26);
\draw[color=black,dashed]
(-1.72,1.26)to[out=144,in=-36, looseness=1](-2.14,1.56);
\draw[color=black!40!white,dashed]
(2.5,0)to[out= 0,in=180, looseness=1](3,0)
(0.77,-2.37)to[out=-72,in=108, looseness=1](0.92,-2.85)
(0.77,2.37)to[out=72,in=-108, looseness=1](0.92,2.85)
(-2,-1.46)to[out=-144,in=36, looseness=1](-2.42,-1.76);
\fill[color=black]
(0,0) circle (1.5pt);
\end{tikzpicture}
\end{center}
\begin{caption}{Examples of tree--like regular expanders with 3, 4, 5 asymptotic halflines (in gray).}
\end{caption}
\end{figure}

By studying the ODE satisfied along each curve, one
can easily show that an expander cannot be compact, all its curves are smooth and each noncompact curve must be asymptotic to a halfline. Moreover, it is trivial that the family of the asymptotic halflines of the open networks of a self--similarly expanding curvature flow $\epsi_t$ is the same for all $t\in(0,+\infty)$ and, by a direct maximum principle argument, one can prove exponential decay of the functions representing the network as graphs on such halflines, outside a large ball.

\begin{lem}\label{lem:expander-asymptotic}
Let $P$ be a finite union of distinct halflines meeting at the origin
and $\epsi$ a regular expander, such that each noncompact curve of
$\epsi$ is asymptotic in Hausdorff distance to one of the halflines of
$P$. Then, there exists an $r_0>0$ large enough such that each noncompact curve $\sigma$ of $\epsi$ corresponds to a connected component of $\epsi\setminus B_{r_0}(0)$ and can be parametrized as 
$$
\sigma(\ell)=\ell e^{i\omega}+u(\ell)e^{i(\omega+\pi/2)}\quad\text{for }\ell\geqslant r_0.
$$
where $\bigl\{\,\ell e^{i\omega}~|~\ell\geqslant 0\,\bigl\}$ is a halfline of $P$ and
$\lim_{\ell\to+\infty}u(\ell) = 0$. Moreover, the decay of $u$ is given by
$$
|u(\ell)|\leqslant C_0e^{-\ell^2/2}, \quad |u'(\ell)|\leqslant C_1 \ell^{-1}e^{-\ell^2/2},\quad |u''(\ell)|\leqslant C_2 e^{-\ell^2/2}
$$
and
$$
|u'''(\ell)|\leqslant C_3 \ell e^{-\ell^2/2}, \quad |u''''(\ell)|\leqslant C_4 \ell^2e^{-\ell^2/2},
$$
where each constant $C_i$ depends only on $r_0$, $u(r_0)$ and $u'(r_0)$.
\end{lem}

Then, it is easy to see that for every smooth self--similarly expanding curvature flow $\epsi_t$, letting $P$ be the network given by the finite union of the distinct (common) asymptotic halflines of $\epsi_t$, meeting at the origin, we have $\epsi_t \to P$, as $t\to0$, in $C^\infty\loc(\R^2\setminus\{0\})$. We say that $P$ is the {\em generator} of the flow $\epsi_t$ or that $\epsi_t$ is a (possibly not unique) curvature flow of $P$ in the time interval $[0,+\infty)$.\\
Conversely, if we consider a network $P$ given by a finite number of distinct halflines meeting at the origin and we assume that we have a smooth curvature flow $\SS_t$ for $t\in(0,T)$, such that $\SS_t \to P$ in $C^\infty\loc(\R^2\setminus\{0\})$, as $t\to0$, then the parabolically rescaled flows
$$
\SS^\mu_\tt = \mu\,\SS_{\mu^{-2}\tt}
$$
also satisfy $\SS^\mu_\tt \to P$, as $\tt\to 0$, for any
$\mu>0$, since $P$ is invariant under rescalings. Thus, supposing that the flow $\SS_t$ is unique in some ``appropriate class'' with initial condition $P$, we obtain that $T =+\infty$ and $\SS_t=\SS^\mu_t$, for any $\mu,t>0$. This is like to say that $\SS_t = \sqrt{2t}\,\SS_{1/2}$, that is, $\SS_t$ is a self--similarly expanding curvature flow of regular networks, for $t\in(0,+\infty)$ and $P$ is its generator. As we said, the family of distinct (common) asymptotic halflines of all $\SS_t$ coincides with the family of halflines of $P$.

\begin{rem}\label{uniqexpand} Notice that the generator of a self--similarly expanding curvature flow of networks is uniquely defined, while, for a network $P$ composed of a finite number of halflines for the origin, there could be several self--similarly expanding curvature flows of regular networks having $P$ as a generator, as in the following figure.
\begin{figure}[H]
\begin{center}
\begin{tikzpicture}[scale=1.3]
\path[color=black]
(0.1,0.05)node[right]{$O$};
\draw[color=black]
(0,-0.5)to[out=90,in=-90, looseness=1](0,0.5);
\draw[color=black]
(0,0.5)to[out=30,in=-135, looseness=1](1.34,1.44)
(0,0.5)to[out= 150,in=-45, looseness=1](-1.34,1.44)
(0,-0.5)to[out=-150,in=45, looseness=1](-1.34,-1.44)
(0,-0.5)to[out=-30,in=135, looseness=1](1.34,-1.44);
\draw[color=black,dashed]
(1.34,1.44)to[out=45,in=-135, looseness=1](1.74,1.84)
(1.34,-1.44)to[out=-45,in=135, looseness=1](1.74,-1.84)
(-1.34,-1.44)to[out=-135,in=45, looseness=1](-1.74,-1.84)
(-1.34,1.44)to[out=135,in=-45, looseness=1](-1.74,1.84);
\draw[color=black!40!white]
(0,0)to[out=45,in=-135, looseness=1](1.44,1.44)
(0,0)to[out=-45,in=135, looseness=1](1.44,-1.44)
(0,0)to[out=-135,in=45, looseness=1](-1.44,-1.44)
(0,0)to[out=135,in=-45, looseness=1](-1.44,1.44);
\draw[color=black!40!white,dashed]
(1.44,1.44)to[out=45,in=-135, looseness=1](1.84,1.84)
(1.44,-1.44)to[out=-45,in=135, looseness=1](1.84,-1.84)
(-1.44,-1.44)to[out=-135,in=45, looseness=1](-1.84,-1.84)
(-1.44,1.44)to[out=135,in=-45, looseness=1](-1.84,1.84);
\fill[color=black]
(0,0) circle (1.25pt);
\end{tikzpicture}\qquad\qquad\qquad
\begin{tikzpicture}[rotate=90,scale=1.3]
\path[color=black]
(0.08,-0.07)node[above]{$O$};
\draw[color=black]
(0,-0.5)to[out=90,in=-90, looseness=1](0,0.5);
\draw[color=black]
(0,0.5)to[out=30,in=-135, looseness=1](1.34,1.44)
(0,0.5)to[out= 150,in=-45, looseness=1](-1.34,1.44)
(0,-0.5)to[out=-150,in=45, looseness=1](-1.34,-1.44)
(0,-0.5)to[out=-30,in=135, looseness=1](1.34,-1.44);
\draw[color=black,dashed]
(1.34,1.44)to[out=45,in=-135, looseness=1](1.74,1.84)
(1.34,-1.44)to[out=-45,in=135, looseness=1](1.74,-1.84)
(-1.34,-1.44)to[out=-135,in=45, looseness=1](-1.74,-1.84)
(-1.34,1.44)to[out=135,in=-45, looseness=1](-1.74,1.84);
\draw[color=black!40!white]
(0,0)to[out=45,in=-135, looseness=1](1.44,1.44)
(0,0)to[out=-45,in=135, looseness=1](1.44,-1.44)
(0,0)to[out=-135,in=45, looseness=1](-1.44,-1.44)
(0,0)to[out=135,in=-45, looseness=1](-1.44,1.44);
\draw[color=black!40!white,dashed]
(1.44,1.44)to[out=45,in=-135, looseness=1](1.84,1.84)
(1.44,-1.44)to[out=-45,in=135, looseness=1](1.84,-1.84)
(-1.44,-1.44)to[out=-135,in=45, looseness=1](-1.84,-1.84)
(-1.44,1.44)to[out=135,in=-45, looseness=1](-1.84,1.84);
\fill[color=black]
(0,0) circle (1.25pt);
\end{tikzpicture}
\end{center}
\begin{caption}{An example of two different tree--like regular expanders (not in the same ``topological class'' -- see below) with the same asymptotic halflines 
(in gray).\label{nonuniqexp}}
\end{caption}
\end{figure}
\end{rem}

Given $P=\bigcup_{j=1}^nP_j$, where $P_j$ are halflines from the origin, in~\cite{schn-schu} it was shown that for $n=3$ there exists a {\em unique} tree--like, regular expander $\epsi$ asymptotic to $P$ (if $P$ is a standard triod such an expander $\epsi$ is $P$ itself), in the case $n>3$ the existence of such tree--like, connected, regular expanders was shown by Mazzeo--Saez~\cite{mazsae}. This result is based on the following simple lemma.

\begin{lem}\label{rem:self-expander} A regular expander is a critical
point of the length functional with respect to the negatively curved
metric
$$
g = e^{{|x|^{2}}}\bigl(dx_1^2+ dx_2^2\bigr)\,.
$$
\end{lem}
\begin{proof}
See~\cite[Proposition~2.3]{mazsae} or~\cite[Lemma~4.1]{Ilnevsch}.
\end{proof}

To be precise, such a network is a {\em stable} critical point of the length functional in $(\mathbb{R}^2,g)$ (where, as usual, it suffices to look at the length of the networks in any large ball $B_R(0)$). 

The geodesic arcs and rays for the metric $g$ are qualitatively similar to the geodesics in the hyperbolic space, as one can expect, since the curvature of $g$ is everywhere negative. For instance, if $P_i$ and $P'_i$ are any two 
halflines emanating from the origin, then there is a unique complete geodesic for the metric $g$ which is asymptotic to these halflines 
along its two ends. A way to see this is to consider the ``geodesic compactification'' of $(\mathbb{R}^2, g)$ as a closed ball $B$. A limiting direction (i.e., the asymptotic limit of any halfline $P_i$) then corresponds to a point $q_i \in \partial B$. Thus, any $P=\bigcup_{j=1}^nP_j$ is uniquely determined by the choice of $n$ distinct points $q_1, \ldots, q_n \in \partial B$.\\
We remind the reader that, given a collection of points
$q_1,\ldots,q_n$ a solution of the so called {\em Steiner problem} in $(\mathbb{R}^2,g)$
is a connected set that contains the points $q_1,\ldots,q_n$
and minimize the length functional (with respect to the metric $g$).
One can prove that for any collection of points $q_1,\ldots,q_n$ 
there exists a solution to the Steiner problem and it is a ``geodesic'' and regular network. In particular, a minimizer of the Steiner problem is an expander.\\
This observation leads to the following result of Mazzeo--Saez~\cite[Main Theorem]{mazsae}.

\begin{prop}
Let $P=\bigcup_{j=1}^nP_j$ be a set of halflines from the origin in $\mathbb{R}^2$ and $q_1, \dots, q_n$ the corresponding points (listed in cyclic order) on $\partial B$, as above. Then, the set of expanding self--similar solutions of the network flow with initial datum $P$ is in one--to--one correspondence with the set of (possibly disconnected) regular networks on $B$ with end--points $\{q_1, \ldots, q_n\}$, whose arcs are geodesics for the metric $g$.\\
Moreover, for each choice of $P=\bigcup_{j=1}^nP_j$ there exists at least one
self--similar expanding solution whose non--compact branches are asymptotic
to the halflines $P_j$.
\end{prop}

Another key fact is that two regular expanders with the same ``topological structure'' and which are asymptotic to the same family of halflines, have to be identical. 

\begin{defn}\label{sametopologicalclass} 
We say that two regular expanders $\epsi_0$ and $\epsi_1$ are 
{\em asymptotic one to each other} if their ends are asymptotic to the 
same halflines.\\
We say that two regular expanders $\epsi_0$ and $\epsi_1$ are 
in the same {\em topological class}, if there is a smooth family of maps
$$
F_\theta:\epsi_0 \to\R^2,\quad 0\leqslant \theta\leqslant 1
$$
such that $F_0$ is the identity, $F_1(\epsi_0)=\epsi_1$, the distance between any two triple junctions of $F_\theta(\epsi_0)$ is uniformly bounded below and
$$
\lim_{r_0\to+\infty}\sup\,\bigl\{\left|{\partial F_\theta(x)/\partial\theta}\right|~\bigl|~x\in \epsi_0\setminus B_{r_0}(0)\bigr\}=0,\quad \text{ for every $0\leqslant \theta\leqslant 1$}.
$$
\end{defn}

Notice that two regular expanders in the same topological class are asymptotic to each other.

\begin{thm}\label{thm:unique} 
If $\epsi_0$ and $\epsi_1$ are two regular expanders in the same topological class, then they coincide.
\end{thm}
\begin{proof}
We work in the negatively curved metric in the plane
$$
g = e^{|x|^{2}}(dx_1^2+ dx_2^2)\,,
$$
such that each curve of a regular expander is a geodesic in this metric.

Let $\{x^0_i\}$ and $\{x^1_i\}$ denote the triple junctions (a finite set) of $\epsi_0$ and $\epsi_1$, respectively. 
As the networks are in the same topological class, we can rearrange the elements of $\{x^0_i\}$ so that each $x_i^0$ is 
connected to $x_i^1$ by the existing deformation $F_\theta$ of $\epsi_0$ into $\epsi_1$. Denote by $x^\xi_i$, for $\xi\in[0,1]$, the unique geodesic 
connecting these points.\\ 
For each $\xi$, we consider the network $\epsi_\xi$ such that if $x_i^0$ is connected to $x_k^0$ by a geodesic, then $x_i^\xi$ is
connected to $x_k^\xi$ through a geodesic as well. To handle the noncompact curves we proceed as follows. Let $P_j$ denote a common asymptotic halfline to $\epsi_0$ and $\epsi_1$, which means that there are geodesics $\psi_0\subseteq \epsi_0$, $\psi_1\subseteq \epsi_1$ asymptotic to $P_j$ at infinity and starting at some points $x^0_i$ and $x^1_i$ respectively. Define then, for every $\xi\in(0,1)$, the curve $\psi_\xi\subseteq\epsi_\xi$ to be the unique geodesic starting at $x^\xi_i$ and asymptotic to $P_j$. This gives a deformation of the curve $\psi_0$ to $\psi_1$.\\
Hence, we have constructed a smooth family of networks with only triple junctions $\epsi_\xi$, for $\xi\in[0,1]$, ``connecting'' $\epsi_0$ and $\epsi_1$ and such that:
\begin{enumerate}
\item The triple junctions $\{x^\xi_i\}$ of $\epsi_\xi$ connect the triple junctions of $\epsi_0$ to the ones of $\epsi_1$ and, for each index $i$ fixed, the path $x_i^\xi$, with $\xi\in[0,1]$, is a geodesic with respect to the metric $g$.
\item Each curve of $\epsi_\xi$ is a geodesic of $(\R^2,g)$.
\item There is $r_0>0$ large enough so that $\epsi_\xi\setminus B_{r_0}(0)$ has $n$ connected components, each asymptotic to a halfline $P_j$, for $j=1, 2, \dots, n$. We can find angles $\omega_{j}$ such that each end of $\epsi_\xi$ becomes parametrized as 
$$
\epsi_\xi(\ell)=\ell e^{i\omega_{j}}+u_{j,\xi}(\ell)e^{i(\omega_{j}+\pi/2)}\quad\text {for } \ell\geqslant r_0.
$$
This follows from Lemma~\ref{lem:expander-asymptotic}.
\item The vector field along $\epsi_\xi$,
$$
X_\sss(\ell)=\frac{\partial\,\,}{\partial \xi}\epsi_\xi(\ell)
$$
is continuous, smooth when restricted to each curve and 
$$
|X_\sss(\ell)|=O(e^{-\ell^2/2}), \quad |\nabla X_\sss(\ell)|=O(\ell^{-1}e^{-\ell^2/2}),
$$
uniformly in $\sss\in[0,1]$, where the gradient is computed along $\epsi_\sss$ with respect to the metric $g$.\\
Moreover,
$$
\alpha_{j,\xi}(\ell)=\frac{\partial u_{j,\xi}(\ell)}{\partial\xi}
$$
satisfies
$$
|\alpha_{j,\xi}(\ell)|=O(e^{-\ell^2/2})\quad |\alpha'_{j,\xi}(\ell)|=O(\ell^{-1}e^{-\ell^2/2}).
$$
It is enough to provide justification for the second set of estimates. 
For ease of notation we omit the indices $j$ and $\xi$ on $\alpha_{j,\xi}$ and $u_{j,\xi}$. By linearizing the equation for an expanding graph, see~\cite[equation~(2.3)]{schn-schu}, we have
$$
\alpha''=(1+[u']^2)(\alpha-\ell\alpha')+2u'\alpha'(u-\ell u').
$$
We can assume without loss of generality that $\alpha(r_0)\geqslant 0$. Moreover, it follows from our construction that 
$$
\lim_{\ell\to+\infty}|\alpha(\ell)|+|\alpha'(\ell)|=0.
$$
A simple application of the maximum principle shows that $\alpha$ can not have a negative local minimum or a positive local maximum. Hence, $\alpha\geqslant 0$ and $\alpha'\leqslant 0$. We can assume that $u'\leqslant 0$ (see the proof of Lemma~\ref{lem:expander-asymptotic}). The function $\beta=\alpha-\ell\alpha'$ thus satisfies
$$
\beta'=-\ell(1+[u']^2)\beta-2\ell u'\alpha'\leqslant -x\beta
$$
and integration of this inequality gives the conclusion.
\end{enumerate}
Denote by $L$ the length functional with respect to the metric $g$ and
consider the family of functions
\begin{equation*}
W_r(\xi)=L(\epsi_\xi\cap B_{2r_0}(0))+\sum_{j=1}^n\int_{2r_0}^r e^{[\ell^2+u^2_{j,\xi}(\ell)]/{2}}\sqrt{1+[u'_{j,\xi}(\ell)]^2}\,d\ell -n\int_{2r_0}^r e^{\ell^2/2}\,d\ell\,.
\end{equation*}
The decays given in Lemma~\ref{lem:expander-asymptotic} imply the existence of a constant $C$ such that for every $r\leqslant\overline{r}$
\begin{equation}\label{C^3}
\Vert W_r-W_{\overline{r}}\Vert_{C^3}\leqslant Ce^{-r}\,,
\end{equation}
so, when $r\to+\infty$, the sequence of functions $W_r:[0,1]\to\R$ converges uniformly in $C^2$ to a function $W:[0,1]\to\R$. Furthermore, if $\xi=0$ or $\xi=1$, we have, combining Lemma~\ref{lem:expander-asymptotic} with point~4 above, that
$$
\lim_{r\to+\infty} \frac{d W_r(\xi)}{d\xi}=0\,,
$$
thus, $W$ has a critical point when $\xi=0$ or $\xi=1$.

A standard computation shows that on each compact curve of $\epsi_\xi$, we have (after reparametrization proportional to arclength)
\begin{equation*}\begin{split}
\frac{d^2}{d\xi^2}\int_a^b \sqrt{g(\epsi'_\xi,\epsi'_\xi)}\,dl &=\int_a^b|\epsi'_\xi |^{-1}\bigl(|(\nabla _{\epsi_\xi'} X_\sss)^{\bot}|^2-\text{Riem}(X_\sss,\epsi'_\xi,\epsi'_\xi,X_\sss)\bigr)\,dl+|\epsi'_\xi|^{-1} g(\nabla_{X_\sss} X_\sss, \epsi'_\xi)\Bigr|_a^b\\
&=\int_a^b|\epsi'_\xi |^{-1}\bigl(|(\nabla _{\epsi'_\xi} X_\sss)^{\bot}|^2-\text{Riem}(X_\sss,\epsi'_\xi,\epsi'_\xi,X_\sss)\bigr)\,dl\,,
\end{split}
\end{equation*}
where $\epsi_\xi'=d\epsi_\xi/dl$, we used property~1 above and all the
geometric quantities are computed with respect to the metric $g$ (${\mathrm{Riem}}$ is the Riemann tensor of $(\R^2,g)$). Combining this identity with property~4, we have
\begin{equation*}
\frac{d ^2 W_r(\xi)}{d\xi^2}=\int_{\epsi_\xi\cap B_{r}(0)}|\epsi'_\xi |^{-2}\bigl(|(\nabla _{\epsi'_\xi} X_\sss)^{\bot}|^2-\text{Riem}(X_\sss,\epsi'_\xi,\epsi'_\xi,X_\sss)\bigr)\,dl+O(e^{-r})\,.
\end{equation*}
As $(\R^2,g)$ is negatively curved, more precisely, its Gaussian curvature is equal to $-e^{-|x|^2}$, the integrals above are bounded independently of $r>2r_0$. Therefore, by means of estimate~\eqref{C^3}, we obtain
$$
\frac{d^2 W(\xi)}{d\xi^2}= \int_{\epsi_\xi}|\epsi'_\xi |^{-2}\bigl(|(\nabla _{\epsi'_\xi} X_\sss)^{\bot}|^2-\text{Riem}(X_\sss,\epsi'_\xi,\epsi'_\xi,X_\sss)\bigr)\,dl\geqslant 0\,,
$$
where the last inequality comes form the fact that $(\R^2,g)$ is negatively curved. It follows that $W:[0,1]\to\R$ is a convex function with two critical points at $\xi=0$ and $\xi=1$, hence, it is identically constant.
The last formula above then implies that the vector field $X_\sss$ must be a constant multiple of $\epsi'_\xi$, hence, it must vanish at all triple junctions. The fact that $X_\sss$ is continuous implies that $X_\sss$ is identically zero and this proves that all the networks $\epsi_\xi$ coincide, for $\xi\in[0,1]$, in particular $\epsi_0=\epsi_1$, which is the desired result. 
\end{proof}

\begin{cor}\label{crossunique} 
If $P=\bigcup_{j=1}^4P_j$ is a standard cross, then there exists a unique, connected, tree--like, regular expander asymptotic to $P$.
\end{cor}
\begin{proof}
In this case, it is easy to see that there are only two possible topological classes of connected regular expanders asymptotic to $P$ (analogous to the two situations depicted in 
Figure~\ref{nonuniqexp}), but since every unbounded curve cannot change its convexity (as for the shrinkers, by analyzing the expanders equation~\eqref{expeq}), if two such curves are contained in the angle of $120$ degrees of the standard cross, when they concur at a $3$--point they must form an angle larger than $120$ degrees, which is a contradiction, hence such topological class is forbidden.\\
Thus, only one topological class is allowed and it contains only one regular expander (with two symmetry axes), by Theorem~\ref{thm:unique}.
\end{proof}

We recall that the same conclusion of this corollary also holds when $P$ is composed of three halflines from the origin.

\subsection{A short--time existence theorem for non--regular networks}
The first result we present (\!\cite[Theorem 1.1]{Ilnevsch}) requires the notion of 
convergence {\em in the sense of varifolds} and can be stated as follows.

\begin{thm}\label{evolnonreg}
Let $\SS_0$ be a possibly non--regular, embedded, $C^1$ network with bounded curvature, which is $C^2$ away from its multi--points and 
such that the exterior unit tangent vectors of the concurring curves at every multi--point are mutually distinct. Then, there exist $T>0$ and a
smooth curvature flow of connected regular networks $\SS_t$, locally
tree--like, for $t\in(0,T)$, such that $\SS_{t}$ for $t\in[0,T)$ is a regular
Brakke flow. Moreover, away from the multi--points of $\SS_{0}$ the
convergence of $\SS_t$ to $\SS_0$, as $t\to 0$, is in $C^2\loc$ (or as
smooth as $\SS_0$).\\
Furthermore, there exists a constant $C>0$ such that
$\sup_{\SS_t}\left|k\right|\leqslant C/\sqrt{t}$
and the length of the shortest curve of $\SS_t$ is bounded
from below by $C\sqrt{t}$.
\end{thm}

\begin{rem}\label{var} To be more precise, we define the sets $G_t$ as 
$$
G_t=\{(x,\tau(x,t))~\vert~x\in\SS_t\}\cup\{(x,-\tau(x,t))~\vert~
x\in\SS_t\}\subseteq\R^2\times\SS^1\,,
$$
for every $t\in[0,T)$, where $\tau(x,t)$ is the unit tangent vector at
$x\in\SS_t$.
The convergence of $\SS_t\to\SS_0$ in the previous theorem is 
in the sense of {\em varifolds},
that is, as $t\to 0$, the Hausdorff measures $\H^1\res G_t$ converge to $\H^1\res G_0$,
as measures on $\R^2\times\SS^1$ (see~\cite{simon} for the general definition).
It is easy to see that this implies that $\H^1\res\SS_t\to\HH^1\res\SS_0$,
as $t\to0$, as measures on $\R^2$, hence there is no instantaneous loss 
of mass of the network at the starting time.\\
Around a non--regular multi--point the $C^1$--convergence is not
possible: for every $t>0$, the networks $\SS_t$ are regular, so they
satisfy the $120$ degrees condition and that would pass to the limit.
Varifold--convergence is anyway a sort of ``weak'' $C^1$--convergence,
slightly stronger than simply asking that
$\H^1\res\SS_t\to\HH^1\res\SS_0$, as $t\to0$.
\end{rem}

We aim to present now an outline of the proof of
Theorem~\ref{evolnonreg} which depends crucially on an expander monotonicity formula implying that self--similarly expanding flows are ``dynamically stable''. 
The monotone integral quantity we will consider
has been applied previously by A.~Neves in the setting of Lagrangian mean
curvature flow~\cite{neves1,neves2,nevestian}. Other main ingredients
are the local regularity Theorem~\ref{thm:locreg.2} and the pseudolocality Theorem~\ref{thm:graph_local} (see~\cite[Theorem~1.5]{Ilnevsch}). 
We underline that for curves moving in the plane, this latter can be replaced by S.~Angenent's {\em intersection
counting theorem}, see~\cite[Proposition~1.2]{angen2},~\cite[Section~2]{angen1}
and~\cite{angen5} for the proof.

By the assumptions at any multi--point of an initial network $\SS_0$, the
cone generated (at such point) by the interior unit normal vectors of the
concurring curves consist of a finite number of distinct halflines. The natural evolution of such a cone is a self--similarly expanding curvature flow, due to the scaling invariance of this particular initial network. The strategy is then as follows: we ``glue in'', around each possibly non--regular multi--point of the initial network $\SS_0$, a (piece of a) smooth, self--similarly expanding, tree--like, connected regular network at the scale $\sqrt{\sss}$ (in a ball of radius proportional to $\sqrt{\sss}$), corresponding to the cone generated by the interior unit tangent vectors of the concurring curves of $\SS_0$ at the multi--point, to obtain an approximating $C^2$ regular network $\SS^\sss_0$ (satisfying the compatibility conditions of every order, see Definition~\ref{ncompcond}).
The curvature of $\SS^\sss_0$ is thus of order $1/\sqrt{\sss}$ and the
shortest curve has length proportional to $\sqrt{\sss}$. Then, the
standard short--time existence result yields a smooth curvature flow
$\SS_t^\sss$ up to a positive time $T_\sss$.\\
To prove that these approximating flows exist for a time
$T>0$, independent of $\sss$, we make use of the expander monotonicity
formula to show that the flows $\SS^\sss_t$ stay close to the
corresponding self--similarly expanding flows, in an integral sense, around each
multi--point. This gives that the curvature is bounded by $C/\sqrt{t}$ up
to a fixed time $T>0$, together with a lower bound on the length of the
shortest curve. Thus, we can pass to the limit, as $\sss\to 0$, to obtain
the desired curvature flow.

\begin{rem}\label{noeq}
The Brakke flow provided by the above theorem is not necessarily {\em with equality} (see Definition~\ref{brk}). Indeed, assume for instance that $\SS_0$ is a standard cross (see Figure~\ref{crossfig}) and $\varphi$ a test function such that $0\leqslant \varphi \leqslant 1, \varphi = 1$ on $B_1(0)$ and $\varphi = 0$ outside of $B_2(0)$. Let $\SS_t=\sqrt{2t}\,\SS_0$ be the regular expander ``exiting'' from $\SS_0$ (which is the curvature flow given by Theorem~\ref{evolnonreg}).
Suppose by contradiction that $\SS_t$ is a regular Brakke flow with equality.
Since $\SS_0$ has no curvature, by using equation~\eqref{brakkeqqq} we have
\begin{equation*}
\frac{d\,\,}{dt}\int_{\SS_t}\varphi\,ds\,\Bigl\vert_{t=0}=
-\int_{\SS_0}\varphi k^2\,ds
+\int_{\SS_0}\langle\nabla\varphi\,\vert\,\underline{k}\rangle\,ds=0\,.
\end{equation*}
Anyway, by the mean value theorem for any $t >0$ there holds
$$
\frac{\int_{\SS_t} \varphi\,ds - \int_{\SS_0} \varphi\,ds}{t} =-\int_{\SS_\theta} \varphi k^2\, ds + \int_{\SS_\theta} \langle \nabla \varphi, \underline{k}\rangle \,ds\,,
$$
for some $0<\theta<t$. By the self--similarity property of $\SS_t=\sqrt{2t}\,\SS_0$, 
it is then easy to see that the first term on the right-hand side of this formula goes to $-\infty$ and the second one stays bounded, hence,
$$
\frac{\overline{d\,\,}}{dt}\int_{\SS_t}\varphi\,ds\,\Bigl\vert_{t=0}=\limsup_{t\to0} \frac{\int_{\SS_t} \varphi\,ds - \int_{\SS_0} \varphi\,ds}{t}=-\infty\,,
$$
which is a contradiction.
\end{rem}

\begin{rem}\label{ilrem1}
In writing this paper, we got informed that the hypothesis on the
non--coincidence of two (but no more than two) exterior unit tangent vectors can actually
be removed (Tom Ilmanen, {\em personal communication}).
\end{rem}

\begin{rem}\label{connrem}
The a priori choice of gluing in only {\em connected} regular self--similarly expanding networks, hence obtaining a connected network flows, has a physical meaning: 
it ensures that initially separated regions remain separated during the flow while using only {\em tree--like} self--similarly expanding networks excludes the formation of new bounded regions.\\
Indeed, from a $7$--point one could try (this is only conjectural, the line of Theorem~\ref{evolnonreg} does not work in this case) to get a flow with a new heptagonal region, by gluing in a symmetric self--similarly expanding network with a heptagonal region, following the construction of Theorem~\ref{evolnonreg} described above.\\
Anyway, it can be seen that all the connected, regular self--similarly expanding networks containing a bounded region must have at least seven unbounded halflines. This because, by means of the same arguments of Section~\ref{geopropsub} (Remark~\ref{schreg}), every bounded region of a regular self--similarly expanding network is bounded by at least seven curves. This clearly implies that from a multi--point of order less than six, the flow produced by Theorem~\ref{evolnonreg} is always locally tree--like, even if the line of proof (and at the moment it is not) could be adapted to ``glue in'' {\em any} self--similarly expanding network (that is, possibly also a non tree--like one, in general). It is then a natural question if a multi--point with more than five (or possibly more than six) concurring curve can appear in the limit network $\SS_T$, as $t\to T$, described in Theorem~\ref{ppp2} of the previous section. This is related to finding a regular (possibly degenerate) shrinker with more than five (or maybe six) unbounded halflines.
\end{rem}

\begin{oprob}\label{ooo4000}
Do there exist (possibly degenerate) regular shrinkers with more that five (or six) unbounded halflines?
\end{oprob}

\subsection{The expander monotonicity formula}\label{expmon}
Let $\SS_t$ be a curvature flow of tree--like regular networks.
The tangent vector of $\SS_t$
makes with the $x$--axis an angle $\overline{\theta}_t$ which, away from the triple
junctions, is a well defined function up to a multiple of $\pi$, since we
do not care about orientation. Because at the triple junctions, the angle jumps by
$2\pi/3$,
there is a well defined function $\theta_t$ which is continuous on
$\SS_t$ and coincides with $\overline{\theta}_t$ up to a multiple of
$\pi/3$.
We identify the plane $\mathbb{R}^2$ with $\mathbb{C}$, thus
$$
\underline{k} = J\tau\,\partial_s\theta_t = \nu \,\partial_s\theta_t\,\,,
$$
where $J$ is the complex structure.

Let $\mathscr{L} = xdy - ydx$ be the Liouville form on $\R^2$. Since we
assumed that $\SS_t$ has no loops, we can find a function $\beta_t$,
unique up to a time--dependent constant, such that
$$
d\beta_t = \mathscr{L}|_{\SS_t}\,.
$$
We can modify the time--dependent constant so that the following
evolution equations hold, see~\cite[Lemma~3.1]{Ilnevsch}.
\begin{lem} \label{thm:evoeq}
The following evolution equations hold away from the triple junctions:
\begin{equation*}\frac{d\theta_t}{dt} = \partial^2_{s}\theta_t +
\partial_s\theta_t\,\langle\tau\,\vert\,X\rangle\,,
\end{equation*}
\begin{equation*}\frac{d\beta_t}{dt} = \partial^2_{s}\beta_t +
\partial_s\beta_t\,\langle\tau\,\vert\,X\rangle - 2 \theta_t\,,
\end{equation*}
where $X = \underline{k}+\lambda \tau$ is the velocity of the evolution.
\end{lem}
Notice that this implies that the function $\alpha_t =
\beta_t+2t\theta_t$ satisfies the evolution equation
$$
\frac{d\alpha_t}{dt} = \partial_s^2\alpha_t +
\partial_s\alpha_t\,\langle\tau\,\vert\,X\rangle\,.
$$
Furthermore, $J\tau\,\partial_s\alpha_t = \nu\, \partial_s\alpha_t =
-x^\perp + 2t\underline{k}$, which
exactly vanishes on a self--similarly expanding network. With a
computation similar to the one leading to Huisken's monotonicity
formula~\eqref{eqmonfor}, we arrive at the following result,
see~\cite[Lemma~3.2]{Ilnevsch}.

\begin{lem}[Expander monotonicity formula]\label{thm:expmon}
 The following identity holds
\begin{equation*}
\frac{d}{dt}\int_{\SS_t} \alpha_t^2\,\rho_{x_0,t_0}(x, t)\,ds
=-\int_{\SS_t}2\,\bigl|\,x^{\perp}-2t\underline{k}\,
\bigr|^2\rho_{x_0,t_0}(x, t)\, ds
 -\int_{\SS_t}
\alpha_t^2\left|\,\underline{k}+\frac{(x-x_0)^{\perp}}{2(t_0-t)}\,\right|^2\rho_{x_0,t_0}(x,
t)\, ds\,,
\end{equation*}
for some constant $C$.
\end{lem}

In the later applications, the evolving networks will be only locally tree--like, that is, only locally without loops. In order
to apply the above monotonicity formula, it will need to be localized. We assume that $\SS_t\cap B_4(x_0)$ does not contain any closed loop for all
$0\leqslant t<T$. We define $\beta_t$ locally on $\SS_t\cap B_4(x_0)$ and we let $\varphi:\R^2\to\R$ be a smooth cut--off function such that $\varphi = 1 $ on
$B_2(x_0)$, $\varphi = 0$ on $\R^2\setminus B_3(x_0)$ and $0\leqslant \varphi\leqslant 1$.
Then, we have the following localized version of Lemma~\ref{thm:expmon},
see~\cite[Lemma~3.3]{Ilnevsch}.

\begin{lem}[Localized expander monotonicity formula]\label{thm:evol.local} The following estimate holds,
\begin{equation*}
\frac{d}{dt}\int_{\SS_t}\varphi\, \alpha_t^2\,\rho_{x_0,t_0}(x,t)\, ds
\leqslant - \int_{\SS_t} \varphi\, |x^\perp - 2t\underline{k}|^2 \rho_{x_0,t_0}(x,t)\,ds
+ C \int_{\SS_t\cap(B_3(x_0)\setminus B_2(x_0))} \alpha_t^2\,\rho_{x_0,t_0}(x,t)\, ds\,.
\end{equation*} 
\end{lem}

\subsection{Outline of the proof of Theorem~\ref{evolnonreg}}\label{stesecsub}
Now let $\SS_0$ be a non--regular initial network with bounded
curvature. For simplicity, let us assume that $\SS$ has only one non--regular multi--point at the origin.

If the multi--point consists of only two curves meeting at an
angle different from $\pi$ (remember that a zero angle is not allowed), then, by the work of Angenent~\cite{angen1,angen2,angen3}, there exists a curvature flow starting at $\SS_0$, satisfying the statement of Theorem~\ref{evolnonreg}: actually the angle is immediately smoothed and the two curves become a single smooth one.

So we can assume that at the origin at least three curves meet and let
$\tau_j$, for $j=1, 2, \dots, n$, be the exterior unit tangent vectors. We denote with 
$$
P_j =\bigl\{-\ell \tau_j\, | \ \ell\geqslant 0\,\bigr\}
$$
the corresponding halflines and $P=\bigcup_{j=1}^n P_j$. Since $\SS_0$ has bounded curvature,
we can assume, by scaling $\SS_0$ if necessary, that $\SS_0 \cap
B_5(0)$ consists of $n$ curves $\sigma_j$ corresponding to the tangents $\tau_j$ and if $\omega_j$ is the angle that $P_j$ makes with the $x$--axis, there
is a function $u_j$ such that $\sigma_j$ can be parametrized (with a small error at the boundary of the ball $B_5(0)$) as
$$
\sigma_j =\bigl\{ \ell e^{i\omega_j}+u_j(\ell) e^{i(\omega_j + \pi/2)}\, |\, 0\leqslant
\ell \leqslant 5\bigr\}\,.
$$
Notice that the assumption that $\SS_0$ has bounded curvature implies
\begin{equation}
 \label{eq:shorttime.1}
 |u_j(\ell)|\leqslant C \ell^2\qquad \text{and}\qquad |u_j'(\ell)| \leqslant C \ell\,,
\end{equation}
for some constant $C$.\\
As already mentioned, in~\cite{schn-schu} it was shown that for $n=3$ there exists a unique tree--like regular expander $\epsi$ asymptotic to $P=\bigcup_{j=1}^nP_j$. In 
the case $n>3$, the existence of tree--like, connected, regular expanders was shown by Mazzeo--Saez~\cite{mazsae}.\\
We remind that, thanks to Lemma~\ref{lem:expander-asymptotic},
there exists $r_0>0$ such that outside the ball $B_{r_0}(0)$
the $n$ noncompact curves $\gamma_j$ of the regular expander $\epsi$ 
can be parametrized as 
$$
\gamma_j=\bigl\{ \ell e^{i\omega_j}+v_j(\ell) e^{i(\omega_j + \pi/2)}\, |\, 
\ell \geqslant r_0\bigr\}\,,
$$
where the functions $v_j$ have the following decay:
\begin{equation}\label{eq:shorttime.v_j}
|v_{j}(\ell)|\leqslant C_0\, e^{-\ell^2/2}\,,\quad
|v'_{j}(\ell)|\leqslant C_1\ell^{-1}\, e^{-\ell^2/2}\,,\quad
|v''_{j}(\ell)|\leqslant C_2\, e^{-\ell^2/2}\,.
\end{equation}

Consider now the rescaled expander $\epsi_\sss=\sqrt{2\sss}\,\epsi$,
call $\sigma_{j,\sss}$ be the curve of $\epsi_\sss$ asymptotic to $P_j$, for every $j=1, 2, \dots, n$,
then 
$$
\sigma_{j,\sss}=\bigl\{ \ell e^{i\omega_j}+v_{j,\sss}(\ell) e^{i(\omega_j + \pi/2)}\, |\, 
\ell\geqslant r_0{\sqrt{2\sss}}\bigr\}\,,
$$
and we have the estimates
\begin{equation}\label{eq:shorttime.2}
|v_{j,\sss}(\ell)|\leqslant C\sqrt{2\sss}\, e^{-\ell^2/4\sss}\,,\quad
|v'_{j,\sss}(\ell)|\leqslant C\ell^{-1}\sqrt{2\sss}\, e^{-\ell^2/4\sss}\,,\quad
|v''_{j,\sss}(\ell)|\leqslant C\, e^{-\ell^2/4\sss}/\sqrt{2\sss}\,.
\end{equation}
In particular, choosing $\sss$ small enough, we have $r_0 \sqrt{2\sss}<4$
and this holds in the annulus $A(r_0 \sqrt{2\sss},4)=B_4(0)\setminus B_{r_0 \sqrt{2\sss}}(0)$.

We now aim to construct the network $\SS_0^\sss$ by gluing $\epsi_\sss=\sqrt{2\sss}\,\epsi$ into $\SS_0$ (more precisely $\epsi_\sss\cap B_{r_0\sqrt{2\sss}}(0)$, for $\sss$ small enough). We define the network $\SS_0^\sss$ that coincides with $\epsi_\sss$ in $B_{r_0\sqrt{2\sss}}(0)$ and with $\SS_0$ outside $B_4(0)$, while in the ``gluing'' annulus 
$A(r_0 \sqrt{2\sss},4)$, in a way we ``interpolate'' between the two networks. 
Precisely, letting $\varphi:\R^+ \to [0,1]$ be a cut--off function such that $\varphi =1$ 
on $(0,1]$ and $\varphi = 0$ on $[2,+\infty)$, we define $\SS_0^\sss$ in $A(r_0 \sqrt{2\sss},4)$ via the graph function $u_{j,\sss}$ as follows, for $\ell\in[r_0\sqrt{2\sss},4)$, 
$$
u_{j,\sss}(\ell)= \varphi(\sss^{-1/4}\ell)v_{j,\sss}(\ell) + \bigl(1-\varphi(\sss^{-1/4}\ell)\bigr)u_j(\ell)\,.
$$
That is, 
$$
\SS_0^\sss\cap A(r_0 \sqrt{2\sss},4)=\bigl\{ \ell e^{i\omega_j}+u_{j,\sss}(\ell) e^{i(\omega_j + \pi/2)}\, |\, r_0\sqrt{2\sss}\leqslant\ell \leqslant 4\bigr\}
$$
(with a small error at the borders of the annulus $A(r_0 \sqrt{2\sss},4)$).

By construction, every network $\SS_0^\sss$ has the same regularity of $\SS_0$, it is regular and satisfies all the compatibility conditions of every order (see Definition~\ref{ncompcond}), it is locally a tree and it can be checked easily that it satisfies the following properties, for every $\sss$ smaller than some $\sss_0>0$:

\begin{enumerate}
\item[${\mathcal{P}}1$.] There is a constant $D_1$, independent of $\sss$, such that
$$
\H^1(\SS^\sss_0\cap B_r(x))\leqslant D_1r\,,
$$
for all $x\in\R^2$ and $r>0$. 

\item[${\mathcal{P}}2$.] There is a constant $D_2$ independent of $\sss$, such that for every $x\in\SS^\sss_0$,
$$
\bigl|\theta^\sss_0(x)\bigr|+\bigl|\beta^\sss_0(x)\bigr|\leqslant D_2(|x|^2+1)\,,
$$
where $\theta^\sss_0$ and $\beta^\sss_0$ are the ``angle function'' and a primitive for the Liouville form of the network $\SS^\sss_0$, as defined in Section~\ref{expmon}.

\item[${\mathcal{P}}3$.] The curvature of $\SS^\sss_0$ is bounded by $C/\sqrt{\sss}$ and $\SS^\sss_0\to \SS_0$ in $C^1\loc (\R^2\setminus\{0\})$, as $\sss\to 0$.

\item[${\mathcal{P}}4$.] The connected components of $P\cap A(r_0\sqrt{2\sss},4)$ are in one--to--one correspondence with the connected components of 
$\SS^\sss_0\cap A(r_0\sqrt{2\sss},4)$ and there is a constant $D_3$, independent of $\sss$, such that the functions $u_{j,\sss}$ satisfy 
\begin{equation}\label{H4}
|u_{j,\sss}(\ell)|+ \ell|u'_{j,\sss}(\ell)|+\ell^2|u''_{j,\sss}(\ell)| \leqslant D_3 \Bigl(\ell^2+\sqrt{2\sss}\,e^{-{\ell^2}/{4\sss}}\Bigr)\,,
\end{equation}
for every $\ell\in[r_0\sqrt{2\sss},4]$.

\item[${\mathcal{P}}5$.] The sequence of rescaled networks $\widetilde\SS^\sss_0=\SS^\sss_0/\sqrt{2\sss}$ converges in $C^{1,\alpha}\loc(B_{r_0}(0))$ to $\epsi$, for $\alpha\in(0,1)$, as $\sss\to 0$.\\
Without loss of generality we can also assume that locally 
\begin{equation*}\label{zero}
\lim_{\sss\to 0} (\widetilde\theta^\sss_0+\widetilde\beta^\sss_0)=0\,,
\end{equation*}
where $\widetilde\theta^\sss_0$ and $\widetilde\beta^\sss_0$ are relative to $\widetilde \SS^\sss_0$.
\end{enumerate}

Let $\SS^\sss_t$, for $t\in[0,T_\sss)$, be a maximal smooth curvature flow starting at the initial network $\SS^\sss_0$ and let
$$
\Theta^\sss_{x_0,t_0}(t)=\int_{\SS^\sss_t}\rho_{x_0,t_0}(\cdot,t)\,ds
$$
be the Gaussian density function with respect to the flow $\SS^\sss_t$.

We fix $\varepsilon_0>0$ such that $3/2+\varepsilon_0 < \Theta_{S^1}$. The main estimate, which will imply short--time existence, is given by the following proposition.
\begin{prop}\label{main}
There are constants $\sss_1$, $\delta_1$ and $\eta_1$ depending on $D_1$, $D_2$, $D_3$, $\epsi$, $r_0$ and $\varepsilon_0$, such that if 
$$
t\leqslant\delta_1,\, r^2\leqslant \eta_1^2t,\,\text{ and }\,\sss\leqslant \sss_1\,,
$$
then,
$$
\Theta^\sss_{x,t+r^2}(t)\leqslant 3/2+\varepsilon_0\,,
$$
for every $x\in B_1(0)$.
\end{prop}
We will sketch the proof after showing how this implies Theorem~\ref{evolnonreg}.
\begin{proof}[Proof of Theorem~\ref{evolnonreg}]
Considering the smooth curvature flows $\SS^\sss_t$ in the time interval $[0,T_\sss)$, for some $T_\sss>0$, discussed above, we now aim to show 
that there exists $T>0$ such that $T_\sss\geqslant T$, for all $\sss\in(0,\sss_1)$ and that there are interior estimates on the curvature and all its higher derivatives 
for all positive times, independent of $\sss\in(0,\sss_1)$. 

By~\cite[Theorem~1.5]{Ilnevsch}, there exists $\varepsilon > 0$ such that if $\SS^\sss_0$ can be written with respect to suitably chosen coordinate system as a graph with a small gradient in a ball $B_R(x)$, then $\SS^\sss_t$ remains a graph in this coordinate system in $B_{\varepsilon R}(x)$ with small gradient, for $t\in [0,\varepsilon R^2]$. Combining this fact with the interior estimates of Ecker--Huisken in~\cite{eckhui2} for the curvature and its higher
derivatives, we can choose a parametrization of the evolving network and a smooth family of points $\overline{P}_j^\sss\in
\SS_t^\sss$ in the annulus $B_{1/2}(0)\setminus B_{1/3}(0)$ along each curve corresponding to $P_j$, for $j=1,\ldots,n$, such that 
$$
\partial_s^l\lambda(\overline{P}^\sss_j,t) = 0~\text{ and }~\bigl|\partial^l_sk(\overline{P}_j^\sss,t)\bigr| \leqslant C_l\,, 
$$
for all $l\geqslant 0$ with constants $C_l$ independent of $\sss$ for $0\leqslant t< \min\{T_\sss,\delta\}$, where $\delta>0$ does not depend on $\sss$. Then, Corollary~\ref{topolino7} gives
estimates on the curvature and its derivatives, independent of $\sss$ and $t$, on $\SS_t^\sss\setminus B_{1/2}(0)$, for $t\in(0,\min\{T_\sss,\delta\})$ (possibly taking a smaller $\delta>0$).\\
To get the desired estimates on $\SS_t^\sss\cap B_{1/2}(0)$ we now apply Proposition~\ref{main} and Theorem~\ref{thm:locreg.2}. Let $\sss_1,\delta_1,\eta_1$ be given by Proposition~\ref{main}. 
If we choose $0<t_0<\min\{T_\sss,\delta_1,\delta,1/2 \}$ and $x_0 \in B_{1/2}(0)$, 
Proposition~\ref{main} implies that if $\sss<\sss_1$, we have
$$
\Theta^\sss_{x,t+r^2}(t)\leqslant 3/2 + \varepsilon_0\,,
$$
for all $x\in B_1(0)$, $t\in(0,t_0)$ and $r^2\leqslant\eta_1^2t$. In particular, we see that if $\overline{t}\in(t_0/2,t_0)$, choosing $r^2\leqslant \frac{\eta^2_1t_0}{2(1+\eta_1^2)}$ and setting $t=\overline{t}-r^2$, we have $t<t_0\leqslant \delta_1$ and $r^2\leqslant \eta_1^2 t$. Hence, the above estimate holds and it can be equivalently written as
$$
\Theta^\sss_{x,\overline{t}}(\overline{t}-r^2)\leqslant 3/2 + \varepsilon_0\,,
$$
for such pairs $(\overline{t},r)$. Letting $\rho=\sqrt{t_0/2}$ (notice that $B_\rho(x_0)\subseteq B_1(0)$), such estimate holds for all $(x,\overline{t}) \in B_\rho(x_0)\times (t_0-\rho^2,t_0)$ and $r\leqslant\frac{\eta^2_1}{\sqrt{1+\eta_1}}\rho$. Hence, by Theorem~\ref{thm:locreg.2} with $\sigma=1/2$, there exists a constant $C$, depending only on $\varepsilon_0$ and $\eta_1$ (by property~${\mathcal{P}}1$ above, the length ratios are uniformly bounded) such that
$$
\bigl|k^\sss(x,\overline{t})\bigr| \leqslant {C}/{\sqrt{t_0}}\,,
$$
for every $\overline{t}\in(t_0/8,t_0)$ and $x\in\SS_{\overline{t}}^\sss\cap B_{\sqrt{t_0/8}}(0)$. Sending $\overline{t}\to t_0$, we get
$$
\bigl|k^\sss(x_0,t_0)\bigr| \leqslant {C}/{\sqrt{t_0}}\,.
$$
Hence, by the arbitrariness of $x_0$, this estimates holds for all $x_0\in\SS_{t_0}^\sss\cap B_{1/2}(0)$ and $t_0$ small enough, together with the corresponding estimates on all higher
derivatives. Moreover, by the second point of Remark~\ref{rem:locreg.1}, there is a constant $C_1>0$, depending only on $\varepsilon_0$ and $\eta_1$, such that the length of the shortest curve of $\SS^\sss_{t_0}$ is bounded from below by $C_1\,\sqrt{t_0}$. By the arbitrariness of the choice, these estimates hold for every $t_0>0$ small enough. 

Together with the estimates on $\SS^\sss_t\setminus B_{1/2}(0)$ for every $t\in(0,\min\{T_\sss,\delta\})$, this implies that $T_\sss\geqslant T$, for some $T>0$, for every $\sss\leqslant\sss_1$. By the estimates on the curvature, which are independent of $\sss$, we can then take a subsequential limit of the flows $\SS^\sss_t$ on $[0,T)$, as $\sss\to 0$, to obtain a smooth limit curvature flow $\SS_t$ in a positive time interval, starting from the non--regular network $\SS_0$.

Notice that, by~\cite[Theorem~1.5]{Ilnevsch} and the interior estimates of
Ecker--Huisken, away from any multi--point, the flow $\SS_t$ attains the initial network $\SS_0$ in $C^2$ (or in the class of regularity of $\SS_0$, if it is better than $C^2$ away from the multi--point).\\
Furthermore, by the above estimate on the curvature and Theorem~\ref{thm:locreg.2}, we have 
$$
\bigl|k(x,t)\bigr| \leqslant {C}/{\sqrt{t}}\,,
$$
for every $x\in\SS_t$.
The estimate on the length of the shortest curve passes to the limit as well.
\end{proof}

\begin{rem}\label{felix} The conclusions of Theorem~\ref{evolnonreg} also hold if the initial network $\SS_0$ is a $C^1$ non--regular network, smooth away from the multi--points where the exterior unit tangent vectors of the concurring curves are mutually distinct and the curvature is of order $o(1/r)$, where $r$ is the distance from the set of the multi--points of $\SS_0$. The modifications in the proof are not completely trivial, the details of such result will appear elsewhere.
\end{rem}

We will now give a sketch of the proof of Proposition~\ref{main}. Since the estimates are rather technical we only outline it and refer the interested reader to~\cite{Ilnevsch}. 
However we want to underline the main three steps of the proof.

\bigskip

\noindent {\bf Step 1.} {\em Estimates far from the origin and for a short time.}

\smallskip

The following estimates are a direct consequence of Huisken's monotonicity formula~\eqref{eqmonfor}: the first one says that the flow is well controlled at a point $x$ away from the origin up to a time proportional to $|x|^2$. This follows by observing that in the annulus $A(K_0 \sqrt{2\sss},1)$, where $K_0$ is sufficiently large, the initial network $\SS^\sss_0$ is close to the collection of halflines $P$ for all $0<\sss\leqslant \sss_1$. Even more, for $1\geqslant |x| \geqslant K_0 \sqrt{2(\sss+t)}$ we see that in $B_{(K_0/2) \sqrt{2(\sss+t)}}(x)$ the initial network is 
$C^1$--close to a unit density line. By the monotonicity formula, this gives a control up to time $t$.\\
The second one shows that if we ``glue in'' the regular expander at scale $\sss$, then we get control in $t$ up to a time proportional to $\sss$. This estimate follows from observing that scaling the initial network $\SS^\sss_0$ by $1/\sqrt{2\sss}$, each point on the network is uniformly $C^1$--close, in a ball of fixed size, either to a unit density line or to a standard triod. The estimate then follows from the monotonicity formula.\\
For details of the proof see~\cite[Lemma 5.2]{Ilnevsch}.
\begin{lem}\label{lem:far and short estimates}\ 
\begin{itemize}
\item {\rm{(Far from origin estimate)}}
There are $\delta_1,K_0>0$ such that if $r^2\leqslant t\leqslant \delta_1$, then 
$$
\Theta^\sss_{x,t+r^2}(t) \leqslant 3/2 + \varepsilon_0\,,
$$
for every $x$ with $1\geqslant |x| \geqslant K_0 \sqrt{2(\sss+t)}$.

\item {\rm{(Short time estimate)}} 
There are $\sss_1,q_1>0$ such that if $\sss\leqslant\sss_1,\ r^2,t \leqslant q_1\sss$, then 
$$
\Theta^\sss_{x,t+r^2}(t)\leqslant 3/2 + \varepsilon_0\,,
$$
for every $x \in B_1(0)$.
\end{itemize}
\end{lem}

It is convenient to introduce a rescaling of the flow which makes the expander ``stationary''. We set (see property~${\mathcal{P}}5$ above)
$$
\widetilde{\SS}^\sss_t = \frac{ \SS^\sss_t}{\sqrt{2(\sss+t)}}\,,
$$
and let 
$$
\widetilde{\Theta}^\sss_{x_0,t_0}(t) = \int_{\widetilde{\SS}^\sss_t} \rho_{x_0,t_0}(\cdot,t)\, ds\,.
$$ 
Notice that 
\begin{equation}\label{eq:thetarescaled}
\Theta^\sss_{x_0,t+r^2}(t) = \widetilde{\Theta}^\sss_{\frac{x_0}{\sqrt{2(\sss+t)}}\,,\,t+\frac{r^2}{2(\sss+t)}}(t)\,.
\end{equation}

\begin{rem}\label{rem:rescaled}\ 
\begin{enumerate}
\item It follows from the second estimate in Lemma~\ref{lem:far and short estimates} that we need only to prove Proposition~\ref{main} when $t\geqslant q_1\sss$.
\item By formula~\eqref{eq:thetarescaled} and the previous point, it suffices to find $\sss_1, \delta_1$ and $\eta_1$ such that for every 
$\sss\leqslant \sss_1,\ q_1\sss \leqslant t\leqslant \delta_1,\ r^2 \leqslant \eta_1^2$ and $y$ with $|y|\leqslant1/\sqrt{2(\sss+t)}$, we have
$$
\widetilde{\Theta}^\sss_{y,t+r^2}(t) \leqslant 3/2 +\varepsilon_0\,.
$$
\item We set $\eta_1^2 = q_1/(2(q_1+1))$. The second estimate in Lemma~\ref{lem:far and short estimates} implies that for $\sss\leqslant \sss_1,\ t\leqslant q_1\sss$ and $r^2\leqslant \eta_1^2$ we have
$$
\widetilde{\Theta}^\sss_{y,t+r^2}(t) \leqslant 3/2 +\varepsilon_0\,,
$$
for every $|y| \leqslant 1/\sqrt{2(\sss+t)}$.\\
The first estimate in Lemma~\ref{lem:far and short estimates} implies that for $r^2 \leqslant \eta_1^2, \sss\leqslant \sss_1$ and $q_1\sss\leqslant t \leqslant\delta_1$,
$$
\widetilde{\Theta}^\sss_{y,t+r^2}(t) \leqslant 3/2 +\varepsilon_0\,,
$$
for every $y$ with $K_0 \leqslant |y| \leqslant 1/\sqrt{2(\sss+t)}$.
\end{enumerate}
\end{rem}

\bigskip

\noindent {\bf Step 2.} {\em Controlling the asymptotic behavior of $\widetilde{\SS}^\sss_t$.}

\smallskip

\noindent By some rather delicate estimates, but which only use the asymptotics ${\mathcal{P}}4$ and again the monotonicity formula, one can show that the following holds (see Lemma~\cite[Lemma 5.4]{Ilnevsch}). It is important here that $r_1$ does not depend on $\nu$.

\begin{lem}[Proximity to $P$]\label{lem:proximity expander}
There are constants $C_1$ and $r_1$ such that, for every $\nu>0$, we can find $\sss_2,\delta_2 >0$ such that the following holds. If $\sss\leqslant \sss_2, t\leqslant \delta_2$ and $r\leqslant 2$, then
$$
\text{dist}(y,P)\leqslant \nu + C_1 e^{-|y|^2/C_1}\ \ \text{if}\ \ y\in \widetilde{\SS}_t^\sss \cap A\bigl(r_1,(\sss+t)^{-1/8}\bigr)\,, 
$$
and
$$
\widetilde{\Theta}^\sss_{y,t+r^2}(t)\leqslant 1 + \varepsilon_0/2 + \nu \ \ \text{if}\ \ y\in A\bigl(r_1,(\sss+t)^{-1/8}\bigr)\,,
$$
where $A\bigl(r_1,(\sss+t)^{-1/8}\bigr)$ is the annulus $B_{(\sss+t)^{-1/8}}(0)\setminus B_{r_1}(0)$.
\end{lem}

The next step is to combine these estimates with the uniqueness of the regular expander in its topological class, given by Theorem~\ref{thm:unique} and a compactness argument
(see~\cite[Corollary 4.6]{Ilnevsch}) to show the following:
\begin{lem}
Let $C_1$ and $r_1$ be the constants given by Lemma~\ref{lem:proximity expander}
and let $\epsi$ be a regular expander. Set $r_2=\max \{r_0,r_1,1\}$, $R=\sqrt{1+2q_1}K_0 + r_2$.
Then there exist $R_1\geqslant R$, $\varrho, \nu>0$ such that if
$\SS$ is a regular network with controlled length ratios such that:
\begin{enumerate}
\item $\int_{\SS\cap B_{R_1}(0)} |\underline{k} - x^\perp|^2\, ds\leqslant \varrho$,
\item $\SS$ and $\epsi$ are in the same topological class 
(see Definition~\ref{sametopologicalclass}),
\end{enumerate}
then $\SS$ must be $\varepsilon$--close in $C^{1,\alpha}(B_{R_1}(0))$ to $\epsi$, for a fixed $\alpha\in(0,1/2)$ 
and a suitably small $\varepsilon>0$, depending on $\epsi$.
\end{lem}

Notice that $\varepsilon$ has to be chosen sufficiently small,
so that the monotonicity formula guarantees a control of the Gaussian 
densities for a network $C^{1,\alpha}$--close to $\epsi$.

\bigskip

\noindent{\bf Step 3.} {\em Application of the expander monotonicity formula.}

\smallskip

\noindent The next lemma is essential to prove Proposition~\ref{main}. Its content is that the proximity of $\widetilde \SS^\sss_t$ to the self--similarly expanding curvature flow generated by $\epsi$ can be controlled in an integral sense. This is the only point where the expander monotonicity formula is used.

We notice that by property~${\mathcal{P}}5$ above, we have that $\widetilde{\SS}^\sss_0=\sqrt{2\sss}\,\SS^\sss_0 \to \epsi$ in $C^{1,\alpha}_\text{loc}(B_{r_0}(0))$, as $\sss\to0$ and recall that the rescaled quantity
$$
\widetilde{\alpha}_t^\sss=\widetilde{\beta}_t^\sss+\widetilde{\theta}_t^\sss\,,
$$
of the expander monotonicity formula, converges locally to zero along this limit. Localizing the expander monotonicity formula (Lemma~\ref{thm:expmon}), choosing $(x_0,t_0)$ appropriately and estimating carefully, one arrives at the following (see~\cite[Lemma 5.6]{Ilnevsch}). Choose $a>1$ such $(1+2q_1)/a >1$ and set $q=q_1/a$.
\begin{lem}\label{integral}
There are constants $\delta_0$ and $\sss_0$ such that for every $\sss\leqslant\sss_0$ and $T_0\in[q\sss,\delta_0]$, we have
$$
\frac{1}{(a-1)T_0}\int^{aT_0}_{T_0}\int_{\widetilde \SS^\sss_t\cap B_{R_1}(0)}|\underline{k}-x^{\perp}|^2\,ds\,dt\leqslant \varrho\,.
$$
\end{lem}

Take $\delta_0, \sss_0$ for which this lemma holds, consider also
$\delta_1, \sss_1$ for which Lemma~\ref{lem:far and short estimates} holds
and $\sss_2=\sss_2(\nu)$, $\delta_2=\delta_2(\nu)$ given by Lemma~\ref{lem:proximity expander}.
Set $\sss_3=\min\{\sss_0, \sss_1,\sss_2\}$, $\delta_3=\min\{\delta_0, \delta_1,\delta_2\}$ and
then, decrease $\sss_3$ and $\delta_3$, if necessary, so that 
$(\sss_3+\delta_3)^{-1/8}\geqslant 2R_1$, $q_1\sss_3\leqslant \delta_3$.\\
Having all the constants properly defined, we can now finish the proof. Set
$$
T_1=\sup\bigl\{\,\widetilde{T}\,\,|\,\; \widetilde\Theta^\sss_{x,t+r^2}(t)\leqslant 3/2+\varepsilon_0\quad\text{for all } x\in B_{K_0}(0),\; r^2\leqslant
 \eta_1^2, \; t\leqslant \widetilde{T}\,\bigr\}\,.
$$
It suffices to show that $T_1 \geqslant \delta_3$, for every $\sss\leqslant\sss_3$. The first point of Remark~\ref{rem:rescaled} implies that $T_1\geqslant q_1\sss$. 
Suppose that $T_1<\delta_3$ and set $T_2=T_1/a$. Lemma~\ref{integral} implies the existence of $t_1\in[T_2,T_1]$ such that
$$
\int_{\widetilde \SS^\sss_{t_1}\cap B_{R_1}(0)}|\underline{k}-x^{\perp}|^2\,ds\leqslant \varrho\,.
$$
One can now check that all the conditions for the previous step are met with $\SS$ being $\widetilde \SS^\sss_{t_1}$.
Therefore, we obtain that $\widetilde \SS^\sss_{t_1}$ is $\varepsilon$--close in $C^{1,\alpha}(B_{R_1}(0))$ to $\epsi$.
Denote by $\widehat{\SS}^\sss_l$, for $l\geqslant 0$, the curvature flow with initial condition $\widetilde \SS^\sss_{t_1}$. A simple
computation shows that
$$
\widehat{\SS}^\sss_l=\sqrt{1+2l}\,\, \widetilde \SS^\sss_{t_1+l\mu^2},
$$
where $\mu^2=2(\sss+t_1).$ Since $\widetilde \SS^\sss_{t_1}$ is $\varepsilon$--close in $C^{1,\alpha}(B_{R_1}(0))$ to $\epsi$, 
we again use the monotonicity formula to conclude that for every $l\leqslant q_1$, we have
$$
\widetilde\Theta^\sss_{x,t_1+l\mu^2+r^2}(t_1+l\mu^2)=\widehat{\Theta}^\sss_{x\sqrt{1+2l}\,,\, l+r^2(1+2l)}(l)\leqslant 3/2+\varepsilon_0\,,
$$
provided 
$$
\sqrt{1+2l}\,|x|\leqslant R_1-1 \quad\text{and}\quad(1+2l)r^2\leqslant q_1\,.
$$
Hence, for all $t_1\leqslant t\leqslant t_1(1+2q_1)$, there holds 
$$
\widetilde\Theta^\sss_{x,t+r^2}(t)\leqslant 3/2+\varepsilon_0\,,
$$
for every $x$ in $B_{K_0}(0)$ and $r^2\leqslant \eta_1^2$, which implies that
$T_1\geqslant t_1(1+2q_1)$. This is a contradiction because
$$
t_1(1+2q_1)\geqslant T_2(1+2q_1)=T_1(1+2q_1)/a>T_0\,.
$$
This concludes the proof of Proposition~\ref{main}.

\begin{rem}\ 
\begin{itemize}
\item Combining Theorem~\ref{evolnonreg} and Theorem~\ref{c2shorttime} (or Theorem~\ref{geosmoothexist}, if $\SS_0$ is geometrically smooth) we have a curvature flow (in the sense of Brakke) smooth for every positive time for every initial $C^2$ network $\SS_0$ (satisfying the hypothesis that at every multi--point the exterior unit tangent vectors of the concurring curves are mutually distinct -- see anyway Remark~\ref{ilrem1}).

\item Notice that in the above proof, we do not perform the ``gluing in'' construction at the regular $3$--points of the initial network. Hence, since the approximating flows are obtained from Theorem~\ref{c2shorttime} (or Theorem~\ref{geosmoothexist} if $\SS_0$ is smooth), the convergence of $\SS_t$ to $\SS_0$, as $t\to0$, locally around a regular $3$--point of $\SS_0$ is the one given by such theorems.\\
Clearly, one could apply the ``gluing in'' procedure also at the regular $3$--points (in such case the regular expander $\epsi$ to be ``glued in'' is simply a standard triod). Then, a natural question is if the convergence of $\SS_t\to\SS_0$ locally around such regular $3$--point is at least $C^1$ or better (depending on the regularity of $\SS_0$ and the level of compatibility conditions it satisfies) and what is the relation between this curvature flow and the one instead obtained by Theorem~\ref{c2shorttime}.

\item In the special situation when we want to use Theorem~\ref{evolnonreg} to ``restart'' a limit non--regular network $\SS_T$, after a singularity at time $T$ (if possible), far from its multi--points $O^1, O^2, \dots, O^m$ such network is smooth, hence, $\SS_t\to\SS_T$ in $C^\infty\loc\bigl(\R^2\setminus\{O^1, O^2, \dots, O^m\}\bigr)$, as $t\to T$.
\end{itemize}
\end{rem}

\subsection{Another approach to short--time existence of the flow for non--regular networks}

One may wonder if it is possible to define the motion by curvature of a non--regular initial network without introducing the notions of varifolds and of Brakke flow. The answer is actually positive, as it was shown in~\cite[Theorem~1.1]{LiMazPlSa}.

\begin{thm}\label{evononreg2}
Let $\SS_0$ be an initial network where all curves are of class $C^2$.
Then, there exists a time $T>0$ and an evolving family of regular networks $\SS_t$ for $t\in(0,T)$, such that $\SS_t\to \SS_0$, as $t\to 0$, in a certain ``strong'' sense.\\
Moreover, the set of the possible flows is classified by the collection of all (appropriate) self--similarly expanding networks coming out from each junction. 
\end{thm}

\begin{rem}
The convergence toward the initial datum as $t\to 0$, which we are going to describe in detail below, in particular, implies that the {\em set} $\SS_t$ converges to $\SS_0$ in Hausdorff distance or that the collection of maps $(\gamma^1_t,\ldots,\gamma^N_t)$ composing the networks $\SS_t$ 
converges uniformly to the family of maps $(\gamma^1_0,\ldots,\gamma^N_0)$ that describes $\SS_0$ (we underline that some of the $\gamma^i_0$ could be constant maps).
\end{rem}

The method used to prove the previous result relies on a central tool in geometric microlocal analysis: the blow--up of the domain and range spaces. 
One interprets the ``non--regular'' junctions as ``singularities of the space'' and ``desingularise'' them by blowing--up the domain $[0,1]$ of each curve and the ambient $\mathbb{R}^2$. We are going to try to describe the ideas and give an outline of the proof, addressing the interested reader to the original paper~\cite{LiMazPlSa} for the full detail.\\

For simplicity, we consider the special case of an initial network $\SS_0=(\gamma_0^1, \ldots,\gamma_0^4)$ composed by only four curves, each one given by a smooth map $[0,1]\to\mathbb{R}^2$, meeting at a non--regular junction
$\gamma^1_0(0)=\gamma^2_0(0)=\gamma^3_0(0)=\gamma^4_0(0)$.
The eventual solution will be an evolving network with {\em five} curves.\\
We first define the {\em blow--up of the domain}. We may regard the entire network as a collection of mappings from a disjoint union of regions in the $(x,t)$ plane. For any $j = 1, \ldots, 4$, let $Q^{j}=\{ (x,t) \in \mathbb{R}^2\,\,\vert\,\, 0 \leqslant t,\ 0\leqslant x\leqslant 1\}$ be the domain parametrizing the evolution of the initial curve $\gamma^{j}_0$. We introduce parabolic polar coordinates defined near $(0,0)$ as 
$$\rho = \sqrt{t + x^2} \geqslant 0,\qquad\qquad \omega = \arcsin(t/\rho^2) = \arccos(x/\rho) \in [0, \pi/2],$$
hence,
$$
(t,x) = (\rho \cos \omega, \rho^2 \sin \omega).
$$
We define
$$
Q_h^{j}=[ Q^{j}, (0,0); dt]\subseteq\R\times \R^+ 
$$
as the set obtained by replacing the corner point $(0,0)$ with the corresponding ``faces'' $\{\rho = 0, 0 \leqslant \omega \leqslant \pi/2\}$. These are called the {\em front faces} of $Q_h^{j}$ and are denoted with $\ff$. Each $Q^{j}_h$ has then as a boundary: a front face $\ff$, two {\em side faces} $\lf$, $\rf$ and the {\em bottom face} $\bff$ (see the following figure).
\begin{figure}[H]
\begin{center}
\begin{tikzpicture}[scale=1.5]
\draw[black, scale=0.65]
(2.5,0)--(2.5,2.5)
(0,0)--(0,2.5)
(0,0)--(2.5,0);
\draw[black, scale=0.65, shift={(2.5,2.5)}, rotate=0]
(0,0)to[out= -45,in=135, looseness=1] (0.1,-0.1)
(0,0)to[out= -135,in=45, looseness=1] (-0.1,-0.1);
\draw[black, scale=0.65, shift={(0,2.5)}, rotate=0]
(0,0)to[out= -45,in=135, looseness=1] (0.1,-0.1)
(0,0)to[out= -135,in=45, looseness=1] (-0.1,-0.1);
\draw[white]
(2.5,-0.335)--(2.55,-0.335);
\path[font=\small]
(0,0)[left] node{$(0,0)$}
(1.63,0)[right] node{$(0,1)$}
(0.83,0)[above] node{$x$}
(1.83,1.3)[above] node{$t$}
(-0.2,1.3)[above] node{$t$};
\end{tikzpicture}\qquad\qquad
\begin{tikzpicture}
\draw[white]
(3.95,0)--(4,0);
\draw[black]
(2.5,0)--(2.5,2.5)
(0,1)--(0,2.5)
(1,0)--(2.5,0);
\draw[color=black,scale=1,domain=0: 1.57,
smooth,variable=\t,shift={(0,0)},rotate=0]plot({1.*sin(\t r)},
{1.*cos(\t r)}) ; 
\path[font=\small]
(1.8,-0.05)[below]node{$\bff$}
(1.05,0.65)[above] node{$\ff$}
(2.60,1.5)[right] node{$\rf$}
(-0.05,1.5)[left] node{$\lf$};
\draw[black, scale=0.65, shift={(3.84,3.84)}, rotate=0]
(0,0)to[out= -45,in=135, looseness=1] (0.1,-0.1)
(0,0)to[out= -135,in=45, looseness=1] (-0.1,-0.1);
\draw[black, scale=0.65, shift={(0,3.84)}, rotate=0]
(0,0)to[out= -45,in=135, looseness=1] (0.1,-0.1)
(0,0)to[out= -135,in=45, looseness=1] (-0.1,-0.1);
\end{tikzpicture}
\end{center}
\caption{The spaces $Q^j$ and $Q_h^j$.}
\label{Fig1}
\end{figure}
The front face $\ff$ is given by $\rho = 0$ in local parabolic polar coordinates; the left and right faces are the vertical sides ``above'' the corresponding front face where $\omega = \pi/2$ and the bottom face is the initial face $t=0$, at $\omega = 0$.\\
The solution $\gamma^{j}(t,x)$ will be defined on $Q^{j}_h$ rather than $Q^{j}$. It has initial condition $\gamma^{j}_0(x)$ on $\bff$ and satisfies the ``matching'' (Herring) conditions along the left and right faces. Its behavior on the front face is the key issue to address.\\
The evolving network $\SS_t$ will also include a new curve $\gamma^5$ which is defined on the set $P^5=\{ (t,x) \in \mathbb{R}^2\,\,\vert\,\, 0 \leqslant t,\ 0\leqslant x\leqslant\sqrt{t}\}$. The fact that this region shrinks to a point at $t=0$ corresponds to the fact that indeed the curve $\gamma^5(t,\cdot)$ disappears, as $t\to 0$.\\
We also blow--up the region $P^5$ parabolically at $(0,0)$, obtaining 
$$
P^5_h = [ P^5, (0,0); dt].
$$
In parabolic polar coordinates defined exactly as above, this space has coordinates $(\rho, \omega)$, with $\rho \geqslant 0$ and $0 \leqslant \omega \leqslant \arcsin (1/2) = \pi/6$.\\
We then define
$$
Q_h = \bigsqcup_{j=1}^4 Q^{j}_h, \qquad Q = \bigsqcup_{j=1}^4 Q^{j}. 
$$
The space $Q_h \bigsqcup P_h^5$ is the ``desingularized domain'' of the evolving network $\SS_t$.\\
Similarly, we now consider the {\em blow--up of the range}. The dilation properties of the self--similar expanding networks suggest that the
homogeneity in the range should also be emphasized. In other words, 
we have to introduce the change of variable $z \in \mathbb{R}^2$ to $w = z/\sqrt{2t}$ in the range. We formalize this as follows. We define $Z = \mathbb{R}^+_t \times \mathbb{R}^2_z$ and we consider the space $Z_h = [Z, (0,O) ; dt]$
obtained from $Z$ by taking the parabolic blow--up of $Z$ at the junction
$O$ at $t=0$. This parabolic blow--up is defined exactly as before, by replacing each $(0,O)$ with the inward--pointing spherical parabolic normal bundle. As above, this becomes more tangible in locally defined parabolic polar coordinates. Suppose that $O = (0,0)$ and define
$$
R = \sqrt{t + |z|^2}, \qquad \ \Theta = (t/ R^2, z/R) = (\Theta_0, \Theta'),
$$
then $Z_h$ has a front face $\ff = \{R=0\}$ and a bottom face $\bff = \{\Theta_0 = 0\}$. There is a codimension-two corner where these two faces intersect. 
\begin{figure}[H]
\begin{center}
\begin{tikzpicture}[scale=1.5]
\draw[black, scale=0.65]
(0,0)--(-1.25,-1.25)
(0,0)--(0,2.5)
(0,0)--(2.5,0);
\draw[black, scale=0.65,shift={(-1.25,-1.25)}, rotate=135]
(0,0)to[out= -45,in=135, looseness=1] (0.1,-0.1)
(0,0)to[out= -135,in=45, looseness=1] (-0.1,-0.1);
\draw[black, scale=0.65, shift={(2.5,0)}, rotate=-90]
(0,0)to[out= -45,in=135, looseness=1] (0.1,-0.1)
(0,0)to[out= -135,in=45, looseness=1] (-0.1,-0.1);
\draw[black, scale=0.65, shift={(0,2.5)}, rotate=0]
(0,0)to[out= -45,in=135, looseness=1] (0.1,-0.1)
(0,0)to[out= -135,in=45, looseness=1] (-0.1,-0.1);
\path[font=\small]
(-1,-1)[above] node{$x_1$}
(1.6,0.05)[above] node{$x_2$}
(-0.16,1.5)[above] node{$t$};
\end{tikzpicture}\qquad\qquad
\begin{tikzpicture}
\draw[black]
(-0.5,-0.5)to[out= 180,in=-135, looseness=1] (-1,0)
(-0.5,-0.5)to[out= 0,in=-135, looseness=1] (1,0);
\draw[dashed, black]
(-1,0)--(-2,0)
(0.5,0.5)to[out= 180,in=45, looseness=1] (-1,0)
(0.5,0.5)to[out= 0,in=45, looseness=1] (1,0);
\draw[black, shift={(-0.5,-0.5)}]
(0,0)--(-0.75,-0.75);
\draw[black]
(0,1)--(0,2.5)
(1,0)--(2.5,0);
\draw[black, shift={(0,1)}] plot[smooth,domain=-1:0] (\x, {\x^2});
\draw[black, shift={(0,1)}] plot[smooth,domain=0:1] (\x, {-\x^2});
\draw[black, shift={(-1.25,-1.25)}, scale=1, rotate=135]
(0,0)to[out= -45,in=135, looseness=1] (0.1,-0.1)
(0,0)to[out= -135,in=45, looseness=1] (-0.1,-0.1);
\draw[black, shift={(2.5,0)}, scale=1, rotate=-90]
(0,0)to[out= -45,in=135, looseness=1] (0.1,-0.1)
(0,0)to[out= -135,in=45, looseness=1] (-0.1,-0.1);
\draw[black, shift={(0,2.5)}, scale=1, rotate=0]
(0,0)to[out= -45,in=135, looseness=1] (0.1,-0.1)
(0,0)to[out= -135,in=45, looseness=1] (-0.1,-0.1);
\path[font=\small]
(-1.5,-1.5)[above] node{$x_1$}
(2.5,0.05)[above] node{$x_2$}
(-0.24,2.25)[above] node{$t$};
\end{tikzpicture}
\end{center}
\caption{From $Z=\mathbb{R}^+_t\times\mathbb{R}^2_z$ to $Z_h$}
\label{fig2}
\end{figure}
Now we regard each $\gamma^{j}$ as a map into $\mathbb{R}^+ \times \R^2$ via $(t,x) \mapsto (t, \gamma^{j}(t,x))$. We ``lift'' this map by blowing--up both the domain and the range. In other words, each ``lift'' should be regarded as a map
$$
\gamma^{j}(t,x): Q_h^{j} \longrightarrow Z_h.
$$
We then want to write the equation satisfied by the ``lifted'' map.
In the computations it is usually simpler to work in coordinate systems different than the parabolic polar coordinates above, in particular, we introduce two sets of projective coordinates near the front face at $(0,0)$. We define $\tau=\sqrt{2t}$ and $s={x}/{\sqrt{2t}}$. Then, $(\tau, s)$ is a nondegenerate coordinate system on $Q_h^{j}$ near $\ff$ away from $\bff$, moreover,
we may also use these coordinates in $P^5_h$ near its front face where $\rho = 0$. Notice that in each $Q^{j}_h$, we have $s \in [0,+\infty)$, while in $P^5_h$, $0 \leqslant s \leqslant 1$. The variable $\tau$ is a defining function for the front face in each of these cases, in the sense that it vanishes exactly on $\ff$ and is ``comparable'' with $\rho$ on any compact set in $Q_h^{j}$ which does not intersect $\rf$ and on the entirety of $P^5_h$. The variable $s$ is a defining function for $\lf$ or $\rf$ and identifies the interior of $\ff$ with $\mathbb{R}^+$.\\
The $(\tau,s)$ coordinates are not valid near the bottom face where $t=0$ and in particular, near the intersection of $\ff$ and $\bff$. Near this corner we introduce an alternate set of projective coordinates $y=x$ and $T={t}/{x^2}$. 
These are singular along the positive $t$--axis; the variable $y$ is now the defining function for $\ff$ and $T$ is the defining function for $\bff$.\\
There are useful projective coordinates for $Z_h$ too, namely
$$
\tau = \sqrt{2t}, \qquad w = z/\sqrt{2t},
$$
thus, $\tau = 0$ is a defining function for $\ff$, while $w$ is a projectively natural linear coordinate for $\ff$. These coordinates are valid away from $\bff$. Thus $(t, \gamma^{j}(t,x))$ ``lifts'' to $(\tau^2/2, \tau q\eta^{j}(\tau,s))$.\\
We now consider the evolution equation in terms of these blow--ups: we ``lift'' the maps $\gamma^{j}$ to maps between $Q_h^{j}$ and $Z_h$ by simply writing 
\begin{equation}\label{motion}
\partial_t\gamma=\frac{\partial^2_{x} \gamma}{\vert\partial_x\gamma\vert^2},
\end{equation}
using the coordinate systems $(\tau,s)$ on $Q_h^{j}$ and $(\tau, w)$ on $Z_h$.\\
We set $\gamma^{j} = \tau \eta^{j}$, which corresponds to the introduction of the projective coordinate on $Z_h$ and, for simplicity, we drop the superscript $j$ for the time being. As we noticed earlier, if $\gamma$ is an arc in an expanding soliton, then $\eta$ depends only on $s$.\\
Since $\partial_t = \tau^{-2}( \tau\partial_\tau - s\partial_s)$ and $\partial_x = \tau^{-1} \partial_s$, equation~\eqref{motion} becomes
$$
\frac{1}{\tau^2} (\tau \partial_\tau - s\partial_s) (\tau \eta) = \frac{ \tau^{-2} \partial_s^2 (\tau \eta)}{ | \tau^{-1} \partial_s (\tau \eta)|^2},
$$
or finally,
\begin{equation}
(\tau \partial_\tau + 1 - s\partial_s) \eta = \frac{ \partial_s^2 \eta}{|\partial_s\eta|^2}.
\label{liftedeq1}
\end{equation}
In particular, if $\gamma$ is an expander, so $\eta = \eta(s)$, this yields the dimensionally reduced expander equation
\begin{equation}
\frac{\partial_s^2 \eta}{|\partial_s\eta|^2} + (s\partial_s - 1) \eta = 0.
\label{stationary}
\end{equation}
Clearly, the equations are complemented with suitable boundary conditions which are naturally specified along all of the side faces.\\
Finally, it is not immediate how to specify an ``initial'' condition along any of the front faces. To determine this, we remark that we expect the ``lifted'' map $\eta^j(\tau,s)$ to be bounded as $\tau \to 0$, hence, we assume that $\eta^j$ is actually smooth up to the front face. This means that it has a boundary value $\eta^j_0(s)$. Noticing that $\tau\partial_\tau \eta^j |_{\tau=0} = 0$, we deduce that 
$$
\frac{\partial_s^2 \eta_0}{|\partial_s \eta_0|^2} + (s\partial_s - 1) \eta_0 = 0,
$$
which is precisely the expander equation. In other words, expander arise naturally as the initial conditions for the flow along the front faces.

\begin{rem}
We can now specify in which sense $\SS_t\to\SS_0$ in Theorem~\ref{evononreg2}. In the blown--up spaces, the number of curves of the initial datum and of the evolving network $\SS_t$ for $t>0$, coincides. Hence, we can consider a suitable convergence of the maps $\gamma^i_t$ to the initial ones $\gamma^i_0$, for each $i$, for instance, we can require that the convergence is in $C^2$, or even smooth, as $t\to 0$.
\end{rem}

Our last step is then solving the ``lifted'' PDE's system.
A rather delicate part of the proof is the construction
of approximate solutions, i.e., a family of networks $\widehat{\SS}_t$ which converges to the initial datum $\SS_0$ and which satisfies the flow equations up to an error that vanishes for all orders at $t=0$. 
To do so, we proceed at follow: by blowing--up the non--regular junctions we can determine the entire Taylor series of the solution whose first term satisfies the expanders equation. We are then able to prove that,
once the first term of the series is determined, all the other term (up to an error) can be obtained with a recursion argument. Thus, to determine the entire series one choose a specific expander at the non--regular junction which actually captures the geometry of the evolving network, in particular how the non--regular junction breaks apart. Then, we still need to get rid of the rapidly vanishing error term to get an exact solution. This is accomplished by an existence proof using a priori estimates. 

\begin{rem}\label{uniqueness-nonreg}
It is clear from the strategy of the proof, that with this alternative approach
(and this quite strong definition of solutions) we have as many different flows as choices, for every non--regular junction of the initial network, of a self--similar expanders compatible with the junction.
In particular, when at every junction there exists a unique expander coming out from the cone $P$ generated by the inner unit tangent vectors of the concurring curves, the produced solution is unique.
As a remarkable example, there is a unique tree--like, connected, regular expander, asymptotic to a standard cross, see Corollary~\ref{crossunique} (composed of four halflines from the origin with opposite directions pairs and forming angles of $120/60$ degrees between them), generated by the exterior unit tangents of the four concurring curves at the $4$--point which arises as the collapse with bounded curvature of a curve in the ``interior'' of $\SS_t$, as $t\to T$, described in Proposition~\ref{bdcurvcollapse}. The same conclusion holds also when $P$ is composed of three halflines from the origin~\cite{schn-schu}.\\
Hence, if all the junctions of the network are of these types, by means of this theorem, the flow can be started (or restarted, for instance in the situation of the collapse of single isolated curves with bounded curvature, as we said above) in a unique way. 
\end{rem}

\medskip

\begin{rem}\ 
\begin{itemize} 
\item One could apply the procedure of Theorem~\ref{evononreg2} also at any regular $3$--point and in such case the associated regular expander is simply a standard triod, hence the resulting flow is unique, moreover, it must coincide with the one obtained by means of Theorem~\ref{c2shorttime}, as it can be shown that it is among the flows of the class ${\mathcal N}$ defined in such theorem.
\item When we use Theorems~\ref{evolnonreg} or~\ref{evononreg2} to ``restart'' a limit non--regular network $\SS_T$ after a singularity at time $T$ (if possible), if such network is smooth far from its multi--points $O^1, O^2, \dots, O^m$, there holds $\SS_t\to\SS_T$ in $C^\infty\loc\bigl(\R^2\setminus\{O^1, O^2, \dots, O^m\}\bigr)$, as $t\to T$, by the local estimates for the motion by curvature (see~\cite{eckhui2}).
\end{itemize}
\end{rem}

Differently from Theorem~\ref{evononreg2}, it is not clear if Theorem~\ref{evolnonreg} produces a unique solution when the expander associated to every junction is unique. This is related to the use of the varifold convergence to the initial network in place of a stronger one.
 
\begin{oprob}\label{ooo9999} If there is a unique regular expander asymptotic to the family of halflines generated by the inner unit tangent vectors of the concurring curves to a multi--point of $\SS_0$, then does Theorem~\ref{evolnonreg} produce a {\em unique} curvature flow?
\end{oprob}

We can also state the open problem in the specific case of a triod and of a standard cross.

\begin{oprob}\label{ooo9999st} In the case of a single triple junction (possibly non regular), Theorem~\ref{evolnonreg} produces a {\em unique} curvature flow?
\end{oprob}

\begin{oprob}\label{ooo9999s} If the inner unit tangent vectors of the concurring curves to a $4$--point of $\SS_0$ generate a standard cross, Theorem~\ref{evolnonreg} produces a {\em unique} curvature flow?
\end{oprob}

We underline here that whatever procedure one decides to apply to have a curvature flow of a general network such that the networks of the flow are regular for every (small) positive time, uniqueness can be impossible, as is shown in the following figure:
\begin{figure}[H]
\begin{center}
\begin{minipage}[c]{.35\textwidth}
\begin{tikzpicture}
\draw[color=black,scale=1,domain=-3.141: 3.141,
smooth,variable=\t,shift={(-5,2)},rotate=0]plot({1.42*sin(\t r)},
{1.42*cos(\t r)});
\draw[color=black,thick,scale=1,domain=0: 1.5708,
smooth,variable=\t,shift={(-6,2)},rotate=0]plot({1*sin(\t r)},
{1*cos(\t r)});
\draw[color=black,thick,scale=1,domain=0: 1.5708,
smooth,variable=\t,shift={(-4,2)},rotate=180]plot({1*sin(\t r)},
{1*cos(\t r)});
\draw[color=black,thick,scale=1,domain=0: 1.5708,
smooth,variable=\t,shift={(-5,1)},rotate=90]plot({1*sin(\t r)},
{1*cos(\t r)});
\draw[color=black,thick,scale=1,domain=0: 1.5708,
smooth,variable=\t,shift={(-5,3)},rotate=-90]plot({1*sin(\t r)},
{1*cos(\t r)});
\path[font=\large]
(-6,3.1) node[left] {$P^4$}
(-4,3.1) node[right] {$P^3$}
(-6,.9) node[left] {$P^1$}
(-4,.9) node[right] {$P^2$}
(-5.3,2.47) node[below] {$O$};
\draw[color=black!40!white,rotate=69,shift={(.2,.3)}]
(-0.05,2.65)to[out= -90,in=150, looseness=1] (0.17,2.3)
(0.17,2.3)to[out= -30,in=100, looseness=1] (-0.12,2)
(-0.12,2)to[out= -80,in=40, looseness=1] (0.15,1.7)
(0.15,1.7)to[out= -140,in=90, looseness=1.3](0,1.1)
(0,1.1)--(-.2,1.35)
(0,1.1)--(+.2,1.35);
\draw[color=black!40!white,rotate=111,shift={(3.35,-1.05)}]
(-0.05,2.65)to[out= -90,in=150, looseness=1] (0.17,2.3)
(0.17,2.3)to[out= -30,in=100, looseness=1] (-0.12,2)
(-0.12,2)to[out= -80,in=40, looseness=1] (0.15,1.7)
(0.15,1.7)to[out= -140,in=90, looseness=1.3](0,1.1)
(0,1.1)--(-.2,1.35)
(0,1.1)--(+.2,1.35);
\end{tikzpicture}
\end{minipage}
\quad 
\begin{minipage}[c]{.25\textwidth}
\begin{tikzpicture}
\draw[color=black,scale=1,domain=-3.141: 3.141,
smooth,variable=\t,shift={(1,0)},rotate=0]plot({1.42*sin(\t r)},
{1.42*cos(\t r)});
\draw[color=black,thick,scale=1,domain=0: 1.5708,
smooth,variable=\t,shift={(0,0)},rotate=0]plot({1*sin(\t r)},
{1*cos(\t r)});
\draw[color=black,thick,scale=1,domain=0: 1.5708,
smooth,variable=\t,shift={(2,0)},rotate=180]plot({1*sin(\t r)},
{1*cos(\t r)});
\draw[color=black,thick,scale=1,domain=0: 1.5708,
smooth,variable=\t,shift={(1,-1)},rotate=90]plot({1*sin(\t r)},
{1*cos(\t r)});
\draw[color=black,thick,scale=1,domain=0: 1.5708,
smooth,variable=\t,shift={(1,1)},rotate=-90]plot({1*sin(\t r)},
{1*cos(\t r)});
\filldraw[color=white,scale=0.2,domain=-3.141: 3.141,
smooth,variable=\t,shift={(5,0)},rotate=0]plot({1.42*sin(\t r)},
{1.42*cos(\t r)});
\draw[color=black, thick, shift={(1,0)}]
(-0.05,-0.05)to[out=45,in=-135, looseness=1](0.05,0.05);
\draw[color=black, thick, shift={(1,0)}]
(-0.05,-0.05)to[out=165,in=10, looseness=1](-0.29,-0.045);
\draw[color=black, thick, shift={(1,0)}]
(-0.05,-0.05)to[out=-75,in=110, looseness=1](0.045,-0.29);
\draw[color=black, thick, shift={(1,0)},rotate=180]
(-0.05,-0.05)to[out=165,in=10, looseness=1](-0.29,-0.045);
\draw[color=black, thick, shift={(1,0)}, rotate=180]
(-0.05,-0.05)to[out=-75,in=110, looseness=1](0.045,-0.29);
\path[font=\large]
(0,1.1) node[left] {$P^4$}
(2,1.1) node[right] {$P^3$}
(0,-1.1) node[left] {$P^1$}
(2,-1.1) node[right] {$P^2$}
(1.4,0.67) node[below] {$O^1$}
(0.7,-0.13) node[below] {$O^2$};
\end{tikzpicture}

\quad

\begin{tikzpicture}[rotate=90]
\draw[color=black,scale=1,domain=-3.141: 3.141,
smooth,variable=\t,shift={(1,0)},rotate=0]plot({1.42*sin(\t r)},
{1.42*cos(\t r)});
\draw[color=black,thick,scale=1,domain=0: 1.5708,
smooth,variable=\t,shift={(0,0)},rotate=0]plot({1*sin(\t r)},
{1*cos(\t r)});
\draw[color=black,thick,scale=1,domain=0: 1.5708,
smooth,variable=\t,shift={(2,0)},rotate=180]plot({1*sin(\t r)},
{1*cos(\t r)});
\draw[color=black,thick,scale=1,domain=0: 1.5708,
smooth,variable=\t,shift={(1,-1)},rotate=90]plot({1*sin(\t r)},
{1*cos(\t r)});
\draw[color=black,thick,scale=1,domain=0: 1.5708,
smooth,variable=\t,shift={(1,1)},rotate=-90]plot({1*sin(\t r)},
{1*cos(\t r)});
\filldraw[color=white,scale=0.2,domain=-3.141: 3.141,
smooth,variable=\t,shift={(5,0)},rotate=0]plot({1.42*sin(\t r)},
{1.42*cos(\t r)});
\draw[color=black, thick, shift={(1,0)}]
(-0.05,-0.05)to[out=45,in=-135, looseness=1](0.05,0.05);
\draw[color=black, thick, shift={(1,0)}]
(-0.05,-0.05)to[out=165,in=10, looseness=1](-0.29,-0.045);
\draw[color=black, thick, shift={(1,0)}]
(-0.05,-0.05)to[out=-75,in=110, looseness=1](0.045,-0.29);
\draw[color=black, thick, shift={(1,0)},rotate=180]
(-0.05,-0.05)to[out=165,in=10, looseness=1](-0.29,-0.045);
\draw[color=black, thick, shift={(1,0)}, rotate=180]
(-0.05,-0.05)to[out=-75,in=110, looseness=1](0.045,-0.29);
\path[font=\large,rotate=-90, shift={(-1,1)}]
(0,1.1) node[left] {$P^4$}
(2,1.1) node[right] {$P^3$}
(0,-1.1) node[left] {$P^1$}
(2,-1.1) node[right] {$P^2$};
\path[font=\large]
(1.6,0.5) node[below] {$O^1$}
(1,-0.43) node[below] {$O^2$};
\end{tikzpicture}
\end{minipage}
\end{center}
\begin{caption}{An example of non--uniqueness of the flow.\label{nonuniqnet}}
\end{caption}
\end{figure}
\noindent indeed, by the symmetry of the initial network with respect to rotations of $90$ degrees, the rotation of any admissible evolution must still be a solution.

\begin{rem}\label{nouniq2}\ 
\begin{itemize}
\item In general, given the set $P$ composed of a finite union of $n$ halflines for the origin, with $n>3$, there are {\em many} 
regular expander asymptotic to $P$, even restricting ourselves to the class of the tree--like ones (see Figure~\ref{nonuniqexp}, for instance). 
One would like to have, at least for the ``generic'' family of halflines $P$, a sort of ``selection principle'' to choose the ``best'' regular expander $\epsi$ at a multi--point with more than $3$ concurring curves, in both procedures.
\item A simple uniqueness statement (which can hold, by what we said, only for a ``generic'' initial network) for the curvature flow obtained by Theorem~\ref{evolnonreg} or by Theorem~\ref{evononreg2} is missing at the moment.
\end{itemize}
\end{rem}

\begin{oprob}\label{ooo15}
For a ``generic'' family of networks $P$ given by $n$ halflines for the origin, 
does there exist a ``selection principle'' to choose the ``best'' regular expander $\epsi$ asymptotic to $P$, to use in performing the procedure of Theorem~\ref{evolnonreg}
or Theorem~\ref{evononreg2}?
\end{oprob}

\begin{oprob}\label{ooo16}
In what class of curvature flows, for a ``generic'' initial non--regular network $\SS_0$, is the flow given by Theorem~\ref{evolnonreg} or Theorem~\ref{evononreg2} unique?
\end{oprob}

\section{Restarting the flow after a singular time}\label{restart}

By means of the analysis of Section~\ref{resum} and the description of the limit network $\SS_T$ at a singular time in Theorems~\ref{ppp2} and~\ref{ppp1}, we can continue the flow by applying the ``restarting'' Theorem~\ref{evolnonreg} (or possibly its extension, see Remark~\ref{felix}). 
We then have an ``extended'' curvature flow for some positive time $T^\prime>T$ (if we are 
not in some of the situations, discussed in Section~\ref{resum}, when the flow ``naturally ends'' 
-- for instance, if the whole network collapses and vanishes, as $t\to T$) 
which is a Brakke flow (possibly without equality, see Remark~\ref{noeq}) in the time interval $(0,T')$ and a smooth curvature flow in $(0,T)\cup(T,T')$.

The passage through a singularity when (locally) a single curve vanishes and two triple junctions collapse forming a regular $4$--point in $\Omega$ is particularly interesting, 
as this type of singularities with bounded curvature, that we called of Type--0 (see Remark~\ref{type0}), is the only possible one for the motion of a tree--like network, assuming that \textbf{M1} holds.
We call this change in the structure of the network a ``standard transition''
(see Figures~\ref{figstandard},~\ref{figstatheta}).\\ 
We recall that while the curvature stays uniformly bounded for $t\leqslant T$, it is of order $1/\sqrt{T-t}$ as $t>T$ (and the ``new'' curve has length of order $\sqrt{T-t}$).
\begin{figure}[H]
\begin{center}
\begin{tikzpicture}[rotate=90,scale=0.6]
\draw[color=black!40!white, shift={(0,-4.8)}]
(-0.05,2.65)to[out=-90,in=150,looseness=1] (0.17,2.3)
(0.17,2.3)to[out=-30,in=100,looseness=1] (-0.12,2)
(-0.12,2)to[out=-80,in=40,looseness=1] (0.15,1.7)
(0.15,1.7)to[out=-140,in=90,looseness=1.3](0,1.1)
(0,1.1)--(-.2,1.35)
(0,1.1)--(+.2,1.35);
\draw[color=black!40!white, shift={(0,-13.8)}]
(-0.05,2.65)to[out=-90,in=150,looseness=1] (0.17,2.3)
(0.17,2.3)to[out=-30,in=100,looseness=1] (-0.12,2)
(-0.12,2)to[out=-80,in=40,looseness=1] (0.15,1.7)
(0.15,1.7)to[out=-140,in=90,looseness=1.3](0,1.1)
(0,1.1)--(-.2,1.35)
(0,1.1)--(+.2,1.35);
\path[font=\small]
(-1,-2.9) node[above]{$t\to T$}
(-1,-11.9) node[above]{$t>T$}
(2.8,-.2) node[below]{$\SS_t$}
(2.8,-18.4) node[below]{$\SS_t$}
(2.8,-9) node[below]{$\SS_T$};

\draw[color=black]
(-0.05,2.65)to[out=30,in=180,looseness=1] (2,3)
(-0.05,2.65)to[out=-90,in=150,looseness=1] (0.17,2.3)
(-0.05,2.65)to[out=150,in=-20,looseness=1] (-2,3.3)
(0.17,2.3)to[out=-30,in=100,looseness=1] (-0.12,2)
(-0.12,2)to[out=-80,in=40,looseness=1] (0.15,1.7)
(0.15,1.7)to[out=-140,in=90,looseness=1](0,1.25)
(0,1.25)to[out=-30,in=180,looseness=1] (1.9,0.7)
(0,1.25)to[out=-150,in=-15,looseness=1] (-1.9,1.2);
\draw[color=black,dashed]
(-2,3.3)to[out=160,in=-20,looseness=1](-2.7,3.5)
 (-1.9,1.2)to[out=165,in=-15,looseness=1](-2.7,1.3)
(1.9,0.7)to[out=0,in=-160,looseness=1] (2.6,0.9)
(2,3)to[out=0,in=160,looseness=1] (2.8,2.9);

\draw[color=black!40!white,shift={(0,-10)}]
(0,2.65)--(1.73,3.65)
(0,2.65)--(1.73,1.65)
(0,2.65)--(-1.73,3.65)
(0,2.65)--(-1.73,1.65);
\draw[color=black,shift={(0,-10)}]
(0,2.65)to[out=-30,in=180,looseness=1] (1.9,2)
(0,2.65)to[out=-150,in=-15,looseness=1] (-1.9,2.3)
(0,2.65)to[out=30,in=180,looseness=1] (2.2,3.3)
(0,2.65)to[out=150,in=-20,looseness=1] (-2.2,3.1);
\draw[color=black,dashed,shift={(0,-10)}]
(-2.2,3.1)to[out=160,in=-20,looseness=1](-3,3.3)
 (-1.9,2.3)to[out=165,in=-15,looseness=1](-2.7,2.4)
(1.9,2)to[out=0,in=-160,looseness=1] (2.6,2.2)
(2.2,3.3)to[out=0,in=160,looseness=1] (3,3.2);

\draw[color=black,scale=0.75,shift={(0,-24)},rotate=-30]
(0,2.65)to[out=-30,in=180,looseness=1] (1.9,2.1);
\draw[color=black,scale=0.75,dashed,shift={(0,-24)},rotate=-30]
(1.9,2.1)to[out=0,in=-160,looseness=1] (2.8,2.4);
\draw[color=black,scale=0.75,shift={(2.65,-24)},rotate=30]
(0,2.65)to[out=30,in=180,looseness=1] (2,3);
\draw[color=black,dashed,scale=0.75,shift={(2.65,-24)},rotate=30]
(2,3)to[out=0,in=160,looseness=1] (2.9,2.9);
\draw[color=black,scale=0.75,shift={(0,-24)},rotate=30]
(0,2.65)to[out=-150,in=-15,looseness=1] (-1.9,2.6);
\draw[color=black,dashed,scale=0.75,shift={(0,-24)},rotate=30]
(-1.9,2.6)to[out=165,in=-15,looseness=1](-2.8,3.0);
\draw[color=black,scale=0.75,shift={(-2.65,-24)},rotate=-30]
(0,2.65)to[out=150,in=-20,looseness=1] (-2,3.3);
\draw[color=black,dashed,scale=0.75,shift={(-2.65,-24)},rotate=-30]
(-2,3.3)to[out=160,in=-20,looseness=1](-3,3.5);
\draw[color=black,scale=0.75,shift={(0,-14)}]
(-1.32,-7.7)to[out=0,in=-150,looseness=1]
(-0.65,-7.4)to[out=30,in=150,looseness=1]
(0.65,-8)to[out=-30,in=180,looseness=1](1.32,-7.7);
\end{tikzpicture}
\end{center}
\begin{caption}{The local description of a ``standard'' transition.\label{figstandard}}
\end{caption}
\end{figure}

\begin{figure}[H]
\begin{center}
\begin{tikzpicture}[scale=0.6,rotate=90]
\draw[color=black!40!white, shift={(0,-4.8)}]
(-0.05,2.65)to[out=-90,in=150,looseness=1] (0.17,2.3)
(0.17,2.3)to[out=-30,in=100,looseness=1] (-0.12,2)
(-0.12,2)to[out=-80,in=40,looseness=1] (0.15,1.7)
(0.15,1.7)to[out=-140,in=90,looseness=1.3](0,1.1)
(0,1.1)--(-.2,1.35)
(0,1.1)--(+.2,1.35);
\draw[color=black!40!white, shift={(0,-13.8)}]
(-0.05,2.65)to[out=-90,in=150,looseness=1] (0.17,2.3)
(0.17,2.3)to[out=-30,in=100,looseness=1] (-0.12,2)
(-0.12,2)to[out=-80,in=40,looseness=1] (0.15,1.7)
(0.15,1.7)to[out=-140,in=90,looseness=1.3](0,1.1)
(0,1.1)--(-.2,1.35)
(0,1.1)--(+.2,1.35);
\path[font=\small]
(-1,-2.9) node[above]{$t\to T$}
(-1,-11.9) node[above]{$t>T$}
(2.8,-.2) node[below]{$\SS_t$}
(2.8,-18.4) node[below]{$\SS_t$}
(2.8,-9) node[below]{$\SS_T$};
\draw[color=black, scale=0.6, shift={(0,3.5)}]
(0,1.5)to[out=30,in=180,looseness=1] (3.3,2.5)
(3.3,2.5)to[out=0,in=90,looseness=1] (5.43,0)
(0,-1.5)to[out=-30,in=180,looseness=1] (3.3,-2.5)
(3.3,-2.5)to[out=0,in=-90,looseness=1] (5.43,0);
\draw[color=black, scale=0.6,rotate=180, shift={(0,-3.5)} ]
(0,1.5)to[out=30,in=180,looseness=1] (3.3,2.5)
(3.3,2.5)to[out=0,in=90,looseness=1] (5.43,0)
(0,-1.5)to[out=-30,in=180,looseness=1] (3.3,-2.5)
(3.3,-2.5)to[out=0,in=-90,looseness=1] (5.43,0);
\draw[color=black, scale=0.6, shift={(0,3.5)}]
(0,1.5)to[out=-90,in=50,looseness=1](0,0.8)
to[out=-130, in=110,looseness=1](0,-0.8)
to[out=-70, in=90,looseness=1](0,-1.5);
\draw[color=black,scale=0.6, shift={(0,-11)}]
(0,0)to[out=30,in=180,looseness=1] (2.7,1.5)
(2.7,1.5)to[out=0,in=90,looseness=1] (4.73,0)
(0,0)to[out=-30,in=180,looseness=1] (2.7,-1.5)
(2.7,-1.5)to[out=0,in=-90,looseness=1] (4.73,0);
\draw[color=black,scale=0.6, shift={(0,-11)}, rotate=180]
(0,0)to[out=30,in=180,looseness=1] (2.7,1.5)
(2.7,1.5)to[out=0,in=90,looseness=1] (4.73,0)
(0,0)to[out=-30,in=180,looseness=1] (2.7,-1.5)
(2.7,-1.5)to[out=0,in=-90,looseness=1] (4.73,0);
\draw[color=black, scale=0.6, shift={(0,-26)}]
(2,0)to[out=60,in=180,looseness=1] (3.7,1)
(3.7,1)to[out=0,in=90,looseness=1] (5,0)
(2,0)to[out=-60,in=180,looseness=1] (3.7,-1)
(3.7,-1)to[out=0,in=-90,looseness=1] (5,0);
\draw[color=black,scale=0.6, shift={(0,-26)}, rotate=180]
(2,0)to[out=60,in=180,looseness=1] (3.7,1)
(3.7,1)to[out=0,in=90,looseness=1] (5,0)
(2,0)to[out=-60,in=180,looseness=1] (3.7,-1)
(3.7,-1)to[out=0,in=-90,looseness=1] (5,0);
\draw[color=black,scale=0.6, shift={(0,-26)}]
(2,0)to[out=180,in=-80,looseness=1](1,0)
(1,0)to[out=100,in=0,looseness=1](0,0);
\draw[color=black,scale=0.6, shift={(0,-26)}, rotate=180]
(2,0)to[out=180,in=-80,looseness=1](1,0)
(1,0)to[out=100,in=0,looseness=1](0,0);
\end{tikzpicture}
\end{center}
\begin{caption}{A ``standard'' transition for a $\Theta$--shaped
 network (double cell).\label{figstatheta}}
\end{caption}
\end{figure}
\noindent We remark that such transition, passing by $\SS_T$, is not symmetric: when $\SS_t\to\SS_T$, as $t\to T^-$, the exterior unit tangent vectors, hence the four angles between the curves, are continuous, while when $\SS_t\to\SS_T$, as $t\to T^+$, there is a ``jump'' in such angles, precisely there is an instantaneous ``switch'' between the angles of $60$ degrees and the angles of $120$ degrees at time $T$.

\begin{rem}\label{unirem999} Since there is a single expander ``coming out'' from the cone of the inner unit tangent vectors generated by the four concurring curves, we expect that by restarting the flow by means of Theorem~\ref{evolnonreg}, we get a {\em unique} evolution (see Problem~\ref{ooo9999s}).
\end{rem}

Coming back to the general situation, we list a series of facts 
when passing through a singularity.

\begin{itemize}
\item The total length of the evolving network $\SS_t$ is non increasing and continuous for 
every $t\in(0,T')$. Hence as a Brakke flow in the time interval $[0,T')$ it does not suffer from the 
phenomenon of ``sudden mass loss'' (see~\cite{brakke} and the recent work~\cite{kimton}).
\item For every $x_0\in\R^2$ and $t_0\in(0,+\infty)$, the Gaussian density function 
$\Theta_{x_0,t_0}(t):[0,\min\{t_0,T'\})\to\R$ 
is still non-increasing. The same for the entropy of $\SS_t$, see formula~\eqref{entropydef}.
 \item The uniform bound on length ratios survives the ``restarting'' procedure 
 with the same constant.
 \end{itemize}
 
These points follow easily by the (weak) continuity of the Hausdorff measures 
$\HH^1\res\SS_t$, see Remarks~\ref{are} and~\ref{var} 
(it is clear in the case of a standard transition).
 
\begin{itemize}
\item By the construction in the ``restarting'' Theorem~\ref{evolnonreg}, no new regions are created passing a singularity, their total number is non-increasing. In particular, a tree remains a tree after restarting (even if its ``structure'' changes).

\item The number of curves of the network is not increasing. To be more precise, if
at least a region vanishes the total number of curves decreases by at least three. In a standard transition, it remains the same.

\item The number of triple junctions of the network is non-increasing. To be more precise, if
at least a region vanishes the total number of triple junctions decreases by at least two. In a standard transition, it remains the same.
\end{itemize}
The fact that no new regions arise follows by the fact that we ``desingularise'' a multi--point, in Theorem~\ref{evolnonreg}, by gluing in a tree--like, connected, regular expander (which is an a priori choice, see Remark~\ref{connrem}). In doing that, by means of Euler's formula for trees, we can see that if the multi--point has order $n$, being the number of the regions equal to $n$, the number of triple junctions we will have in the restarted network in place of the single multiple junction is equal to $n-2$ and the number of curves is $2n-3$.\\
It is then easy to check the above statements if only one bounded region is collapsing since it must be bounded by $n$ curves. If instead, a group of regions is collapsing, we can get the conclusion by applying the same argument to the bounded ``macro--region'' that we obtain considering their union, which will be bounded by a piecewise smooth loop (in a way, we are ``forgetting'' the interior curves to such ``macro--region'' which will anyway be ``lost'' in the collapse).

\smallskip

Clearly, all these facts say that, in a sense, the ``topological complexity'' of the network is ``non-increasing'' passing through a singular time.

\smallskip

We finally mention here that also the bound on the ``embeddedness measure'' $E(t)$, which we will introduce in Section~\ref{dsuL}, survives the ``restarting'' procedure.

\section{Long time behavior}\label{llong}

Since we can repeat the restarting procedure at every singular time, 
either the flow naturally ends at some time $\widehat{T}$ (for instance, if the whole network collapses and vanishes, as $t\to \widehat{T}$) 
or we found ourselves in some of the situations described in Section~\ref{resum} 
where we have to decide how to continue the flow (related to the behavior at the boundary of $\Omega$), 
or we have an increasing sequence of singular--restarting times $T_i$ 
for the evolution of the network $\SS_t$. In this latter case it
follows by the ``topological'' conclusions in the previous section that among these times $T_i$, 
the number of the ones such that we have a non--standard transition is actually finite and depends only on $\SS_0$
(indeed, if a transition is non--standard, then at least one region vanishes during the transition and $\SS_0$ can have only a finite number of regions). 
Instead, we cannot conclude the same for the number of standard transitions that a priori could be infinite. Even worse, notice that Theorem~\ref{evolnonreg} does not give any estimate on the (short) time of existence of the restarted flow, which means that we are not able to say in general if and when another singularity could appear after the restarting time. In particular, we are also not able to exclude that the singular times (associated to standard transitions) ``accumulate'', not even for a tree--like network when all the possible singularities are standard transitions.

The following figures show some examples of these (maybe) possible situations.
\begin{figure}[H]
\begin{center}
\begin{tikzpicture}[scale=0.6]
\draw[color=black!50!white,rotate=90,shift={(0,-1.9)}, scale=0.6]
(-0.05,2.65)to[out=-90,in=150,looseness=1] (0.17,2.3)
(0.17,2.3)to[out=-30,in=100,looseness=1] (-0.12,2)
(-0.12,2)to[out=-80,in=40,looseness=1] (0.15,1.7)
(0.15,1.7)to[out=-140,in=90,looseness=1.3](0,1.1)
(0,1.1)--(-.2,1.35)
(0,1.1)--(+.2,1.35);
\draw[color=black!50!white,rotate=90,shift={(0,-6.9)}, scale=0.6]
(-0.05,2.65)to[out=-90,in=150,looseness=1] (0.17,2.3)
(0.17,2.3)to[out=-30,in=100,looseness=1] (-0.12,2)
(-0.12,2)to[out=-80,in=40,looseness=1] (0.15,1.7)
(0.15,1.7)to[out=-140,in=90,looseness=1.3](0,1.1)
(0,1.1)--(-.2,1.35)
(0,1.1)--(+.2,1.35);
\draw[color=black!50!white,rotate=90,shift={(0,-11.9)}, scale=0.6]
(-0.05,2.65)to[out=-90,in=150,looseness=1] (0.17,2.3)
(0.17,2.3)to[out=-30,in=100,looseness=1] (-0.12,2)
(-0.12,2)to[out=-80,in=40,looseness=1] (0.15,1.7)
(0.15,1.7)to[out=-140,in=90,looseness=1.3](0,1.1)
(0,1.1)--(-.2,1.35)
(0,1.1)--(+.2,1.35);
\draw[color=black!50!white,rotate=90,shift={(0,-16.9)}, scale=0.6]
(-0.05,2.65)to[out=-90,in=150,looseness=1] (0.17,2.3)
(0.17,2.3)to[out=-30,in=100,looseness=1] (-0.12,2)
(-0.12,2)to[out=-80,in=40,looseness=1] (0.15,1.7)
(0.15,1.7)to[out=-140,in=90,looseness=1.3](0,1.1)
(0,1.1)--(-.2,1.35)
(0,1.1)--(+.2,1.35);
\draw[black]
 (-3.03,1.25) 
to[out=-50,in=180,looseness=1] (-2,0.6) 
to[out=60,in=180,looseness=1.5] (-0.43,1.25) 
(-2,0.6)
to[out=-60,in=130,looseness=0.9] (-1.5,-0.3)
to[out=10,in=100,looseness=0.9](-0.43,-1.25)
(-1.5,-0.3)
to[out=-110,in=50,looseness=0.9](-3.03,-1.25);
\draw[black, shift={(5,0)}]
 (-3.03,1.25) 
to[out=-50,in=180,looseness=1] (-1.65,0.25)
to[out=60,in=180,looseness=1.5] (-0.43,1.25) 
(-1.65,0.25)
to[out=0,in=100,looseness=0.9](-0.43,-1.25)
(-1.65,0.25)
to[out=-120,in=50,looseness=0.9](-3.03,-1.25);
\draw[black, shift={(10,0)}]
 (-3.03,1.25) 
to[out=-50,in=180,looseness=1]
(-2.4,-0.2)
to[out=-60,in=50,looseness=1.5] (-3.03,-1.25)
(-2.4,-0.2)
to[out=60,in=-130,looseness=0.9] (-1.2,0.3)
to[out=110,in=100,looseness=0.9](-0.43,1.25) 
(-1.2,0.3)
to[out=-10,in=50,looseness=0.9](-0.43,-1.25);

\draw[black, shift={(20,0)}]
 (-3.03,1.25) 
to[out=-40,in=150,looseness=1] (-1.8,0.5) 
to[out=30,in=180,looseness=1.5] (-0.43,1.25) 
(-1.8,0.5)
to[out=-90,in=90,looseness=0.9] (-1.8,-0.5)
to[out=-30,in=110,looseness=0.9](-0.43,-1.25)
(-1.8,-0.5)
to[out=-150,in=10,looseness=0.9](-3.03,-1.25);

\draw[black, shift={(15,0)}]
 (-3.03,1.25) 
to[out=-50,in=120,looseness=1] (-1.75,0.15)
to[out=60,in=180,looseness=1.5] (-0.43,1.25) 
(-1.75,0.15)
to[out=-60,in=100,looseness=0.9](-0.43,-1.25) 
(-1.75,0.15)
to[out=-120,in=50,looseness=0.9] (-3.03,-1.25) ;
\draw[color=black, domain=-3.141: 3.141,
smooth,variable=\t,shift={(-1.72,0)},rotate=0, scale=0.9]plot({2.*sin(\t r)},
{2.*cos(\t r)}) ;
\draw[color=black, domain=-3.141: 3.141,
smooth,variable=\t,shift={(3.28,0)},rotate=0, scale=0.9]plot({2.*sin(\t r)},
{2.*cos(\t r)}) ;
\draw[color=black, domain=-3.141: 3.141,
smooth,variable=\t,shift={(8.28,0)},rotate=0, scale=0.9]plot({2.*sin(\t r)},
{2.*cos(\t r)}) ;
\draw[color=black, domain=-3.141: 3.141,
smooth,variable=\t,shift={(13.28,0)},rotate=0, scale=0.9]plot({2.*sin(\t r)},
{2.*cos(\t r)}) ;
\draw[color=black, domain=-3.141: 3.141,
smooth,variable=\t,shift={(18.28,0)},rotate=0, scale=0.9]plot({2.*sin(\t r)},
{2.*cos(\t r)}) ;
\end{tikzpicture}
\end{center}
\begin{caption}{A tree--like network with four fixed end--points switching between its two possible topological classes.\label{switch2}}
\end{caption}
\end{figure}
\begin{figure}[H]
\begin{center}
\begin{tikzpicture}[scale=0.65]
\draw[color=black!50!white,rotate=90,shift={(0,-1.9)}, scale=0.6]
(-0.05,2.65)to[out=-90,in=150,looseness=1] (0.17,2.3)
(0.17,2.3)to[out=-30,in=100,looseness=1] (-0.12,2)
(-0.12,2)to[out=-80,in=40,looseness=1] (0.15,1.7)
(0.15,1.7)to[out=-140,in=90,looseness=1.3](0,1.1)
(0,1.1)--(-.2,1.35)
(0,1.1)--(+.2,1.35);
\draw[color=black!50!white,rotate=90,shift={(0,-6.9)}, scale=0.6]
(-0.05,2.65)to[out=-90,in=150,looseness=1] (0.17,2.3)
(0.17,2.3)to[out=-30,in=100,looseness=1] (-0.12,2)
(-0.12,2)to[out=-80,in=40,looseness=1] (0.15,1.7)
(0.15,1.7)to[out=-140,in=90,looseness=1.3](0,1.1)
(0,1.1)--(-.2,1.35)
(0,1.1)--(+.2,1.35);
\draw[color=black!50!white,rotate=90,shift={(0,-11.9)}, scale=0.6]
(-0.05,2.65)to[out=-90,in=150,looseness=1] (0.17,2.3)
(0.17,2.3)to[out=-30,in=100,looseness=1] (-0.12,2)
(-0.12,2)to[out=-80,in=40,looseness=1] (0.15,1.7)
(0.15,1.7)to[out=-140,in=90,looseness=1.3](0,1.1)
(0,1.1)--(-.2,1.35)
(0,1.1)--(+.2,1.35);
\draw[color=black!50!white,rotate=90,shift={(0,-16.9)}, scale=0.6]
(-0.05,2.65)to[out=-90,in=150,looseness=1] (0.17,2.3)
(0.17,2.3)to[out=-30,in=100,looseness=1] (-0.12,2)
(-0.12,2)to[out=-80,in=40,looseness=1] (0.15,1.7)
(0.15,1.7)to[out=-140,in=90,looseness=1.3](0,1.1)
(0,1.1)--(-.2,1.35)
(0,1.1)--(+.2,1.35);
\draw[color=black, domain=-3.141: 3.141,
smooth,variable=\t,shift={(-1.72,0)},rotate=0, scale=0.9]plot({2.*sin(\t r)},
{2.*cos(\t r)}) ;
\draw[color=black, domain=-3.141: 3.141,
smooth,variable=\t,shift={(3.28,0)},rotate=0, scale=0.9]plot({2.*sin(\t r)},
{2.*cos(\t r)}) ;
\draw[color=black, domain=-3.141: 3.141,
smooth,variable=\t,shift={(8.28,0)},rotate=0, scale=0.9]plot({2.*sin(\t r)},
{2.*cos(\t r)}) ;
\draw[color=black,domain=-3.141: 3.141,
smooth,variable=\t,shift={(13.28,0)},rotate=0, scale=0.9]plot({2.*sin(\t r)},
{2.*cos(\t r)}) ;
\draw[color=black, domain=-3.141: 3.141,
smooth,variable=\t,shift={(18.28,0)},rotate=0, scale=0.9]plot({2.*sin(\t r)},
{2.*cos(\t r)}) ;
\draw[color=black] 
(-1.6,0.5) 
to[out=60,in=-150,looseness=1](-0.43,1.25)
(-1.6,0.5) 
to[out=-60,in=150,looseness=1.5] (-1.5,0) 
(-1.6,0.5)
to[out=180,in=90, looseness=1] (-3.25,-0.625)
to[out=-90,in=180,looseness=0.9] (-1.25,-0.75)
(-1.5,0)
to[out=-30,in=90,looseness=0.9] (-1,0)
to[out=-90,in=60,looseness=0.9] (-1.25,-0.75)
to[out=-60,in=150,looseness=0.9](-0.43,-1.25);
\draw[color=black,shift={(5,0)}] 
(-1.5,0)
to[out=150,in=40,looseness=1] (-2.9,0.9) 
to[out=-140,in=90,looseness=1] (-3.2,0)
(-1.5,0) 
to[out=-150,in=0,looseness=1] (-3,-0.7) 
to[out=-180,in=-90,looseness=1] (-3.2,0)
(-1.5,0) 
to[out=30,in=150,looseness=1.5] (-0.43,1.25) 
(-1.5,0)
to[out=-30,in=-120,looseness=0.9](-0.43,-1.25);
\draw[color=black,shift={(10,0)}] 
(-2,0)
to[out=170,in=40,looseness=1] (-2.5,0.7) 
to[out=-140,in=90,looseness=1] (-2.8,0)
(-2,0) 
to[out=-70,in=0,looseness=1] (-2.65,-0.5) 
to[out=-180,in=-90,looseness=1] (-2.8,0)
(-2,0) 
to[out=50,in=180,looseness=1] (-1.3,0) 
to[out=60,in=150,looseness=1.5] (-0.75,1) 
(-1.3,0)
to[out=-60,in=-120,looseness=0.9] (-0.5,-0.75)
(-0.75,1)
to[out=-30,in=90,looseness=0.9](-0.43,1.25)
(-0.43,-1.25)
to[out=-90,in=60,looseness=0.9] (-0.5,-0.75);
\draw[color=black,shift={(15,0)}] 
(-1,0)
to[out=120,in=40,looseness=1] (-1.9,0.6) 
to[out=-140,in=90,looseness=1] (-2.2,0)
(-1,0) 
to[out=-120,in=0,looseness=1] (-2,-0.4) 
to[out=-180,in=-90,looseness=1] (-2.2,0)
(-1,0) 
to[out=60,in=-150,looseness=1.5] (-0.43,1.25) 
(-1,0)
to[out=-60,in=-120,looseness=0.9](-0.43,-1.25);
\draw[black, shift={(20,0)}]
(-1,0.65) 
to[out=45,in=180,looseness=1.5] (-0.43,1.25) 
(-1,0.65)
to[out=-75,in=75,looseness=0.9] (-1,-0.65)
(-1,0.65)
to[out=-195,in=90,looseness=0.9] (-1.5,0)
to[out=-90,in=195,looseness=0.9] (-1,-0.65)
to[out=-45,in=110,looseness=0.9](-0.43,-1.25);
\end{tikzpicture}
\end{center}
\begin{caption}{Standard transitions switching a lens--shaped network to an ``island--shaped'' (with a bridge) one and viceversa.\label{switch1}}
\end{caption}
\end{figure}
\begin{figure}[H]
\begin{center}
\begin{tikzpicture}[scale=0.65]
\draw[color=black!50!white,rotate=90,shift={(0,-1.9)}, scale=0.6]
(-0.05,2.65)to[out=-90,in=150,looseness=1] (0.17,2.3)
(0.17,2.3)to[out=-30,in=100,looseness=1] (-0.12,2)
(-0.12,2)to[out=-80,in=40,looseness=1] (0.15,1.7)
(0.15,1.7)to[out=-140,in=90,looseness=1.3](0,1.1)
(0,1.1)--(-.2,1.35)
(0,1.1)--(+.2,1.35);
\draw[color=black!50!white,rotate=90,shift={(0,-6.9)}, scale=0.6]
(-0.05,2.65)to[out=-90,in=150,looseness=1] (0.17,2.3)
(0.17,2.3)to[out=-30,in=100,looseness=1] (-0.12,2)
(-0.12,2)to[out=-80,in=40,looseness=1] (0.15,1.7)
(0.15,1.7)to[out=-140,in=90,looseness=1.3](0,1.1)
(0,1.1)--(-.2,1.35)
(0,1.1)--(+.2,1.35);
\draw[color=black!50!white,rotate=90,shift={(0,-11.9)}, scale=0.6]
(-0.05,2.65)to[out=-90,in=150,looseness=1] (0.17,2.3)
(0.17,2.3)to[out=-30,in=100,looseness=1] (-0.12,2)
(-0.12,2)to[out=-80,in=40,looseness=1] (0.15,1.7)
(0.15,1.7)to[out=-140,in=90,looseness=1.3](0,1.1)
(0,1.1)--(-.2,1.35)
(0,1.1)--(+.2,1.35);
\draw[color=black!50!white,rotate=90,shift={(0,-16.9)}, scale=0.6]
(-0.05,2.65)to[out=-90,in=150,looseness=1] (0.17,2.3)
(0.17,2.3)to[out=-30,in=100,looseness=1] (-0.12,2)
(-0.12,2)to[out=-80,in=40,looseness=1] (0.15,1.7)
(0.15,1.7)to[out=-140,in=90,looseness=1.3](0,1.1)
(0,1.1)--(-.2,1.35)
(0,1.1)--(+.2,1.35);
\draw[color=black, domain=-3.141: 3.141,
smooth,variable=\t,shift={(-1.72,0)},rotate=0, scale=0.9]plot({2.*sin(\t r)},
{2.*cos(\t r)}) ;
\draw[color=black, domain=-3.141: 3.141,
smooth,variable=\t,shift={(3.28,0)},rotate=0, scale=0.9]plot({2.*sin(\t r)},
{2.*cos(\t r)}) ;
\draw[color=black, domain=-3.141: 3.141,
smooth,variable=\t,shift={(8.28,0)},rotate=0, scale=0.9]plot({2.*sin(\t r)},
{2.*cos(\t r)}) ;
\draw[color=black, domain=-3.141: 3.141,
smooth,variable=\t,shift={(13.28,0)},rotate=0, scale=0.9]plot({2.*sin(\t r)},
{2.*cos(\t r)}) ;
\draw[color=black, domain=-3.141: 3.141,
smooth,variable=\t,shift={(18.28,0)},rotate=0, scale=0.9]plot({2.*sin(\t r)},
{2.*cos(\t r)}) ;
\draw[color=black,shift={(0,0.3)}] 
(-1.73,-1.8) 
to[out=180,in=180,looseness=1] (-2.8,0) 
to[out=60,in=150,looseness=1.5] (-1.5,1) 
(-2.8,0)
to[out=-60,in=180,looseness=0.9] (-1.25,-0.75)
(-1.5,1)
to[out=-30,in=90,looseness=0.9] (-1,0)
to[out=-90,in=60,looseness=0.9] (-1.25,-0.75)
to[out=-60,in=0,looseness=0.9](-1.73,-1.8);
\draw[scale=0.33,rotate=90,color=black,shift={(0,-9.7)}]
(0,0)to[out=30,in=180,looseness=1] (2.7,1.5)
(2.7,1.5)to[out=0,in=90,looseness=1] (4.73,0)
(0,0)to[out=-30,in=180,looseness=1] (2.7,-1.5)
(2.7,-1.5)to[out=0,in=-90,looseness=1] (4.73,0);
\draw[scale=0.33,rotate=-90,color=black,shift={(0,9.7)}]
(0,0)to[out=30,in=180,looseness=1] (2.7,1.5)
(2.7,1.5)to[out=0,in=90,looseness=1] (4.73,0)
(0,0)to[out=-30,in=180,looseness=1] (2.7,-1.5)
(2.7,-1.5)to[out=0,in=-90,looseness=1] (4.73,0);
\draw[black, shift={(8.5,2)}, rotate=90]
(-2.25,-0.12) 
to[out=20,in=180,looseness=1] (-1.55,0.17);
\draw[scale=0.43,color=black,shift={(19.34,-1)}, rotate=90]
(2,0)to[out=60,in=180,looseness=1] (3.7,1)
(3.7,1)to[out=0,in=90,looseness=1] (5,0)
(2,0)to[out=-60,in=180,looseness=1] (3.7,-1)
(3.7,-1)to[out=0,in=-90,looseness=1] (5,0);
\draw[scale=0.43,color=black,shift={(19.34,1.3)}, rotate=-70]
(2,0)to[out=60,in=180,looseness=1] (3.7,1)
(3.7,1)to[out=0,in=90,looseness=1] (5,0)
(2,0)to[out=-60,in=180,looseness=1] (3.7,-1)
(3.7,-1)to[out=0,in=-90,looseness=1] (5,0);
\draw[scale=0.35,rotate=90,color=black,shift={(0,-38)}]
(0,0)to[out=60,in=180,looseness=1] (2,1.5)
(2,1.5)to[out=0,in=90,looseness=1] (3.7,0)
(0,0)to[out=-60,in=180,looseness=1] (2,-1.5)
(2,-1.5)to[out=0,in=-90,looseness=1] (3.7,0);
\draw[scale=0.35,rotate=-90,color=black,shift={(0,38)}]
(0,0)to[out=60,in=180,looseness=1] (2,1.5)
(2,1.5)to[out=0,in=90,looseness=1] (3.7,0)
(0,0)to[out=-60,in=180,looseness=1] (2,-1.5)
(2,-1.5)to[out=0,in=-90,looseness=1] (3.7,0);
\draw[color=black,shift={(20,0)}] 
(-1.73,-1) 
to[out=180,in=-120,looseness=1](-2.2,0)
to[out=120,in=150,looseness=1.5] (-1.5,1) 
(-2.2,0)
to[out=0,in=180,looseness=0.9] (-1.4,0)
(-1.5,1)
to[out=-30,in=90,looseness=0.9] (-1.3,0.5)
to[out=-90,in=60,looseness=0.9] (-1.4,0)
to[out=-60,in=0,looseness=0.9](-1.73,-1);
\end{tikzpicture}
\end{center}
\begin{caption}{Switching by standard transitions of a $\Theta$--shaped network to an ``eyeglasses--shaped'' one and {\em viceversa}.\label{switch}}
\end{caption}
\end{figure}
\noindent In all these examples (where there is a sort of ``duality'' between the two involved networks: lens--island, theta--eyeglasses and between the only two possible trees connecting four points) we do not know if this kind of ``oscillatory phenomenon'' can happen infinitely many times.

\begin{oprob}\label{ooo120} Let us assume that the ``boundary'' 
curves do not collapse during the flow.
\begin{itemize}
\item The set of singular times is finite?
\item The set of singular times is discrete (i.e. it has no accumulation points)?
\item Can the flow be defined for every positive time?
\end{itemize}
\end{oprob}

\begin{rem}
The last question concerns the possibility that the other two have a negative answer. In such case, we could still hope to be able to find a ``well--behaved'' limit network $\SS_{\widehat{T}}$, as $t\to\widehat{T}$, even when the singular times $T_i$ accumulate at $\widehat{T}$, to possibly restart again the flow with Theorem~\ref{evolnonreg} or some extension. Indeed, by restarting the flow at every singularity we can define an {\em extended} curvature flow of networks on some maximal time interval $[0,\widehat{T})$. Then, either the whole network vanishes or there is an accumulation of singular times at $\widehat{T}$, if it is finite. This extended curvature flow is a Brakke flow, by Theorem~\ref{evolnonreg} and actually, it is easy to see that only singular times when a standard transition happens can accumulate at $\widehat{T}$ (the number of regions is non increasing, hence the number of singular times when at least one of them collapses is finite).
We also mention that it would be quite interesting to compare this extended curvature flow with the globally defined one introduced by L.~Kim and Y.~Tonegawa in~\cite{kimton} .
\end{rem}

\begin{rem}
In the recent paper~\cite{NoSc23}, it is shown that the previous questions have positive answers in the special case of axially symmetric networks with only two triple junctions. More precisely, it is proved that the number of singular times is necessarily finite.
We point out that, under these conditions, there are only four possible topological types of networks: the {\em tree}, the {\em lens}, the {\em theta} and the {\em eyeglasses (of ``type A'')} shapes, as in the following figure (see the discussion at the beginning of Section~\ref{globsec2}).
\begin{figure}[H]
\centering
\begin{tikzpicture}[x=1cm,y=1cm,scale=1]
\draw[line width=.5pt, dashed, shift={(-3 cm, 4 cm)}] (-2.5,0) -- (2.5,0);
\draw[line width=.5pt, dashed, shift={(-3 cm, 4 cm)}] (0,1.5) -- (0,-1.5);
\draw[line width=1pt, smooth, shift={(-3 cm, 4 cm)}] plot[samples=200,domain=.5:2] (\x, {1.73*(\x-.5)*((\x-.5)^3-1.75*(\x-.5)^2+1)});
\draw[line width=1pt, smooth, shift={(-3 cm, 4 cm)}] plot[samples=200,domain=.5:2] (\x, {-1.73*(\x-.5)*((\x-.5)^3-1.75*(\x-.5)^2+1)});
\draw[line width=1pt, smooth, shift={(-3 cm, 4 cm)}] (-.5,0) -- (.5,0);
\draw[line width=1pt, smooth, shift={(-3 cm, 4 cm)}] plot[samples=200,domain=-2:-.5] (\x,{1.73*(-\x-.5)*((-\x-.5)^3-1.75*(-\x-.5)^2+1)});
\draw[line width=1pt, smooth, shift={(-3 cm, 4 cm)}] plot[samples=200,domain=-2:-.5] (\x,{-1.73*(-\x-.5)*((-\x-.5)^3-1.75*(-\x-.5)^2+1)});
\draw[line width=.5pt, dashed, shift={(3 cm, 4 cm)}] (-2.5,0) -- (2.5,0);
\draw[line width=.5pt, dashed, shift={(3 cm, 4 cm)}] (0,1.5) -- (0,-1.5);
\draw[line width=1pt, smooth, shift={(3 cm, 4 cm)}] plot[samples=200,domain=-1:1] (\x,{.87*(\x*\x-1)});
\draw[line width=1pt, smooth, shift={(3 cm, 4 cm)}] plot[samples=200,domain=-1:1] (\x,{-.87*(\x*\x-1)});
\draw[line width=1pt, smooth, shift={(3 cm, 4 cm)}] (-2,0) -- (-1,0);
\draw[line width=1pt, smooth, shift={(3 cm, 4 cm)}] (1,0) -- (2,0);
\draw[line width=.5pt, dashed, shift={(-3 cm, 0 cm)}] (-2.5,0) -- (2.5,0);
\draw[line width=.5pt, dashed, shift={(-3 cm, 0 cm)}] (0,1.5) -- (0,-1.5);
\draw[line width=1pt, smooth, shift={(-3 cm, 0 cm)}] plot[samples=2000,domain=0:2] (\x,{.54*(\x+1)*sqrt(2-\x)});
\draw[line width=1pt, smooth, shift={(-3 cm, 0 cm)}] plot[samples=2000,domain=0:2] (\x,{-.54*(\x+1)*sqrt(2-\x)});
\draw[line width=1pt, smooth, shift={(-3 cm, 0 cm)}] (0,.77) -- (0,-.77);
\draw[line width=1pt, smooth, shift={(-3 cm, 0 cm)}] plot[samples=200,domain=-2:0] (\x,{.54*(-\x+1)*sqrt(2+\x)});
\draw[line width=1pt, smooth, shift={(-3 cm, 0 cm)}] plot[samples=200,domain=-2:0] (\x,{-.54*(-\x+1)*sqrt(2+\x)});
\draw[line width=.5pt, dashed, shift={(3 cm, 0 cm)}] (-2.5,0) -- (2.5,0);
\draw[line width=.5pt, dashed, shift={(3 cm, 0 cm)}] (0,1.5) -- (0,-1.5);
\draw[line width=1pt, smooth, shift={(3 cm, 0 cm)}] plot[samples=2000,domain=.5:2] (\x,{(\x-.5)*sqrt(4-2*\x)});
\draw[line width=1pt, smooth, shift={(3 cm, 0 cm)}] plot[samples=2000,domain=.5:2] (\x,{-(\x-.5)*sqrt(4-2*\x)});
\draw[line width=1pt, smooth, shift={(3 cm, 0 cm)}] (-.5,0) -- (.5,0);
\draw[line width=1pt, smooth, shift={(3 cm, 0 cm)}] plot[samples=200,domain=-2:-.5] (\x,{(-\x-.5)*sqrt(4+2*\x)});
\draw[line width=1pt, smooth, shift={(3 cm, 0 cm)}] plot[samples=200,domain=-2:-.5] (\x,{-(-\x-.5)*sqrt(4+2*\x)});
\end{tikzpicture}
\caption{The four possible types of axially symmetric networks with two triple junctions: the {\em tree,} the {\em lens,} the {\em theta} and the {\em eyeglasses}.\label{fig:net-types}}
\end{figure}
\end{rem}

We now discuss the long--time behavior of the curvature flow of a regular network, assuming that there is no accumulation of the singular times or, even better, that the flow definitely does not have singularities after some time. We see in the following proposition that this latter case can only happen for networks without regions with less than six edges.

\begin{prop}\label{loop}
Let $[0,T)$ be the maximal time interval of existence of a smooth curvature flow $\mathbb{S}_t$ of a network 
that has at least one loop $\ell$ of length $L(t)$, enclosing a region of area $A(t)$
composed of $m$ curves with $m<6$.
Then, $T\leqslant\frac{3A(0)}{(6-m)\pi}$ and the equality holds if and only if $\lim_{t\to T}A(t)=0$.
Moreover, if $\lim_{t\to T}L(t)=0$, then $\lim_{t\to T}\int_{\mathbb{S}_t}k^2\,ds=+\infty$.
\end{prop}

\begin{proof}
Integrating in time the equation~\eqref{areaevolreg}, we have
\begin{equation}\label{intareaevolreg}
A(t)-A(0)=\left(-2\pi+m\left(\frac{\pi}{3}\right)\right)t\,.
\end{equation}
Therefore, $T\leqslant\frac{3A(0)}{(6-m)\pi}$, with equality 
if and only if $\lim_{t\to T}A(t)=0$.

Suppose now that $\lim_{t\to T}L(t)=0$. Then we necessarily have $\lim_{t\to T}A(t)=0$, hence $T=\dfrac{3A(0)}{(6-m)\pi}$. Combining equation~\eqref{areaevolreg} and H\"{o}lder inequality, we get
$$
\Big\vert -2\pi+m\left(\frac{\pi}{3}\right)\Big\vert=\Big\vert \frac{dA(t)}{dt}\Big\vert =\Big\vert \int_{\ell_t} k\,ds\Big\vert
\leqslant \left(L(t)\right)^\frac12\left(\int_{\ell_t} k^2\,ds \right)^\frac12\,,
$$
which gives
$$
\int_{\mathbb{S}_t} k^2\,ds\geqslant\int_{\ell_t} k^2\,ds\geqslant \frac{\left(6-m\right)^2\pi^2}{9L(t)}\,.
$$
Since $\lim_{t\to T}L(t)=0$, it follows that
$\lim_{t\to T}\int_{\mathbb{S}_t}k^2\,ds=+\infty$.
\end{proof}

\begin{rem}\hspace{.5truecm}
\begin{enumerate}
\item If a loop is composed of six or more curves, then by equation~\eqref{areaevolreg}, either the enclosed area remains constant or increases during the evolution.
\item The previous proposition does not exclude the possibility that a singularity appears at a time $T<\frac{3A(0)}{(6-m)\pi}$.
\item We expect that, if $T=\frac{3A(0)}{(6-m)\pi}$, then the region is collapsing, hence, by Corollary~\ref{regioncor} the curvature cannot be bounded and we expect that $\lim_{t\to T}L(t)=0$ and $\lim_{t\to T}\int_{\mathbb{S}_t}k^2\,ds=+\infty$.
\end{enumerate}
\end{rem}

For a general network, even assuming that there is no accumulation of the singular times, if the boundary curves do not collapse, we cannot anyway exclude that there could be an infinite sequence of standard transitions with some loops present and regions (with more that five edges) never collapsing. We now deal with tree--like networks that after some time have no more singularities.
 
\begin{prop}\label{prolong}
Suppose that $\mathbb{S}_t$ is a smooth curvature flow in $[0,+\infty)$ of a tree--like network.
Then for every sequence of times $t_i\to\infty$, there exists a (non relabeled) subsequence such that the evolving networks $\SS_{t_i}$ converge
in $C^{1,\alpha}\cap W^{2,2}$, for every $\alpha\in(0,1/2)$, to a possibly degenerate (and non--embedded) regular network with zero curvature, that is, ``stationary'' for the length functional, as $i\to\infty$.
\end{prop}

\begin{proof}
From equation~\eqref{evolength} we have the estimate
\begin{equation}\label{1}
\int_0^{+\infty}\int_{\mathbb{S}_t}k^2\,ds\,dt\leqslant L(0)<+\infty\,.
\end{equation}
Suppose by contradiction that for a sequence of times $t_j\nearrow+\infty$
we have $\int_{\mathbb{S}_{t_j}}k^2\,ds\geqslant\delta$ for some $\delta>0$. 
By the following estimate, which is inequality~\eqref{kappatt} in Lemma~\ref{kappa2}, 
$$
\frac{d}{dt}\int_{\mathbb{S}_t}k^2\,ds\leqslant C\Bigl( 1+\Bigl( \int_{\mathbb{S}_t}k^2\Bigr)\Bigr)^3\,,
$$
holding (in the case of fixed end--points) with a uniform constant $C$ independent of time, 
we would have $\int_{\mathbb{S}_{\widetilde{t}}}k^2\,ds\geqslant{\delta}/{2}$, 
for every $\widetilde{t}$ in a uniform neighborhood of every $t_j$. This is clearly in contradiction with the estimate~\eqref{1}.
Hence, $\lim_{t\to+\infty}\int_{\mathbb{S}_t}k^2\,ds=0$ and, 
consequently, for every sequence of times $t_i\to+\infty$, there exists a subsequence (not relabeled) 
such that the evolving networks $\SS_{t_i}$ converge in $C^{1,\alpha}\cap W^{2,2}$, for every $\alpha\in(0,1/2)$,
to a possibly degenerate regular network with zero curvature, as $i\to\infty$.
\end{proof}

\begin{rem} 
The previous proposition shows that, up to subsequences, the sequence of evolving networks $\SS_{t_i}$ converge, as $t_i\to +\infty$, to a ``stationary'' network for the length functional (which is not necessarily a global minimum).
We do not know if such a stationary network can be non--embedded, that is, some segments have multiplicity greater than one and we underline that actually it can be degenerate, that is, taking the limit of $\SS_{t_i}$ when $t_i\to+\infty$, one or more curves collapse, as shown in the following example. Suppose that $\SS_0$ is 
the regular network in Figure~\ref{family}. It is a smooth regular network composed of five curves, symmetric with respect to the horizontal and vertical axes, the middle curve $\gamma^0$ is a vertical segment and the remaining four curves are convex, i.e., their oriented curvature has a sign. 
The network has four end--points located at the vertices of a rectangle of sides of length $2$ and $2\sqrt{3}$. 
\begin{figure}[H]
\begin{center}
\begin{tikzpicture}[scale=2.2]
\draw[gray, shift={(1.3,0)}, scale=0.5, rotate=-90]
(0,0)to[out= -45,in=135, looseness=1] (0.1,-0.1)
(0,0)to[out= -135,in=45, looseness=1] (-0.1,-0.1);		
\draw[gray, shift={(0,0.6)}, scale=0.5, rotate=0]
(0,0)to[out= -45,in=135, looseness=1] (0.1,-0.1)
(0,0)to[out= -135,in=45, looseness=1] (-0.1,-0.1);	
\draw[gray,dashed]
(-1,-0.5)--(1,0.5)
(1,-0.5)--(-1,0.5);
		 \draw[gray]
 (-1.3,0)--(1.3,0);
 \draw[gray]
 (0,-0.6)--(0,0.6);
		\draw[thick]
		(-1,-0.5)to[out=5, in=210, looseness=0.8]
		(0,-0.15)to[out=90, in=270, looseness=0]
		(0,0.15);
		\draw[rotate=180,thick]
		(-1,-0.5)to[out=5, in=210, looseness=0.8]
		(0,-0.15)to[out=90, in=270, looseness=0]
		(0,0.15);
		\draw[xscale=-1,thick]
		(-1,-0.5)to[out=5, in=210, looseness=0.8]
		(0,-0.15);
		\draw[xscale=-1, rotate=180,thick]
		(-1,-0.5)to[out=5, in=210, looseness=0.8]
		(0,-0.15);
\fill[black](-1,-0.5) circle (0.7pt); 
\fill[black](1,-0.5)circle (0.7pt); 
\fill[black](-1,0.5) circle (0.7pt); 
\fill[black](1,0.5)circle (0.7pt); 
\draw [decorate,decoration={brace,amplitude=5pt,mirror,raise=1ex},gray]
 (-1,-0.5) -- (0,-0.5)
 node[midway,yshift=-1.5em]{$\sqrt{3}$};
 \draw [decorate,decoration={brace,amplitude=5pt,raise=1ex},gray]
 (1,0.5) -- (1,0)
 node[midway,xshift=1.5em]{$1$};
		\path[font=\normalsize]
		(-0.3,0.3)node[above]{$\gamma^1$};
		\path[font=\normalsize]
		(0,0.07)node[right]{$\gamma^0$};
		\end{tikzpicture}
\end{center}
\begin{caption}{The initial network $\SS_0$.\label{family}}
\end{caption}
\end{figure}

Thanks to the symmetries, we can reduce to study the flow of $\SS_0$ to the evolution of a single curve, for instance, $\gamma^1$.
The flow $\SS_t$ starting from $\SS_0$
exists for every time with no singularities, the length of each curve $\gamma^i$ is strictly positive for any time, the curvature of each curve $\gamma^i$ is uniformly bounded and as $t\to+\infty$, the flow smoothly converges to the degenerate network composed of the two segments joining the opposite pairs of end--points and a core at the origin, given by the collapse of the vertical curve $\gamma^0$, whose length goes to zero, as $t\to+\infty$ (see~\cite{PlPo22A}).
\end{rem}

\begin{oprob}\label{ooo3001}\ \begin{itemize}
\item Can the tree--like hypothesis be removed in Proposition~\ref{prolong}? 
\item Is the limit network embedded?
\item What are the possible degeneracies of the limit network? We conjecture that it belongs to the class of networks described in Proposition~\ref{bdcurvcollapse}, in particular, it is embedded and it can only have as degeneracies some regular $4$--points, hence each one with a core given by a single isolated collapsed curve (as in the previous example).
\end{itemize}
\end{oprob}

\begin{rem}\label{singlim} 
If we do not assume that the number of singularities is finite and/or that the network becomes a tree, but only that the flow exists for every $t\in[0,+\infty)$, being globally a Brakke flow (see the previous section), inequality~\eqref{1} still holds (by the defining formula~\eqref{brakkeqqqineq}) and we can always find a sequence of networks $\SS_{t_i}$ converging in $C^{1,\alpha}\cap W^{2,2}$, for every $\alpha\in(0,1/2)$, to a possibly degenerate regular network with zero curvature, as $i\to\infty$. As said before, such limit network could be non--embedded.
\end{rem}

It is natural to ask ourselves if actually, the {\em full} flow of networks $\SS_t$ converges to a limit network, as $t\to+\infty$ (moreover, as we said, we expect that such a limit network is embedded and that the tree--like hypothesis in Proposition~\ref{prolong} is actually superfluous). We are able to show the full convergence assuming that the limit network is not degenerate. A key result to get such convergence is the following {\L}ojasiewicz--Simon inequality for regular networks, proved in~\cite{PlPo22A}.

\begin{thm}\label{cor:LojaNetworkMinimali}
Let $\SS_\ast=(\gamma^1_\ast, \ldots,\gamma^n_\ast)$ be a regular network composed of 
straight segments. Then, there exist $C_{\mathrm LS},\varepsilon >0$ and $\theta \in (0,1/2]$ such that 
if $\SS=(\gamma^1, \ldots,\gamma^n)$ is a regular network of class $W^{2,2}$ with the same topological structure, the same end--points of $\SS_\ast$ and such that
\begin{equation*}
\sum_{i=1}^n \| \gamma^i - \gamma^i _\ast \|_{W^{2,2}} \leqslant \varepsilon\,,
\end{equation*}
then,
\begin{equation}\label{eq:LojaNetworkMinimali}
\left|{L}(\SS)-{L}(\SS_\ast ) \right|^{1-\theta}
\leqslant C_{\mathrm LS} \left( \int_\SS k^2 \, ds \right)^{1/2}.
\end{equation}
\end{thm}

We state now the convergence result.

\begin{thm}[Theorem~5.3 in~\cite{PlPo22A}]\label{prop:Convergence}
Suppose that $\SS_t$ is a smooth curvature flow in $[0,+\infty)$ and let $\SS_\infty$ be a regular (non--degenerate) network with zero curvature, composed of straight segments such that $\SS_{t_n} \to \SS_\infty$ in $W^{2,2}$, for some sequence $t_n\nearrow +\infty$, as $n\to\infty$. Then, up to reparametrization, $\SS_{t} \to \SS_\infty$ smoothly, as $t\to+\infty$.
\end{thm}

We refer the reader to the original paper~\cite{PlPo22A} for the proofs of these two results. We just give here an idea of the application of the {\L}ojasiewicz--Simon inequality in order to get the full convergence of the sequence of networks. 
Let $\SS_t=(\gamma^1_t,\ldots,\gamma^n_t)$ be a smooth network flow defined on $[0,+\infty)$ and 
let $\SS_\infty=(\gamma^1_\infty,\ldots,\gamma^n_\infty)$ be the regular $C^{1,\alpha}\cap W^{2,2}$--limit network along a sequence of times $t_n\to+\infty$, given by Proposition~\ref{prolong}, which we assume to be non--degenerate. Then, by the evolution equation of the length, we have
\begin{equation}
\frac{d}{dt}\big({L}(\SS_t)-{L}(\SS_\infty)\big)
= -\int_{\SS_t} k^2\,ds
\end{equation}
and for all times for which
$\sum_{i=1}^n \| \gamma^i_t - \gamma^i _\infty \|_{W^{2,2}}\leqslant \varepsilon$, 
we get
\begin{align*}
-\frac{d}{dt}\big({L}(\SS_t)-{L}(\SS_\infty)\big)^\theta 
&=\theta\big ({L}(\SS_t)-{L}(\SS_\infty)\big)^{\theta-1} 
\int_{\SS_t} k^2\,ds\\
&\geqslant \frac{\theta}{C_{\mathrm LS}}{\bigg(\int_{\SS_t}k^2\,ds\bigg)^{-{1}/{2}}}\int_{\SS_t}k^2\,ds\\
& =\frac{\theta}{C_{\mathrm LS}}\bigg(\int_{\SS_t}k^2\,ds\bigg)^{{1}/{2}}\,,
\end{align*}
where we used the {\L}ojasiewicz--Simon inequality~\eqref{eq:LojaNetworkMinimali}.
Then, we can take $\widetilde{t}\in [0,+\infty)$ and $t_2>t_1\geqslant \widetilde{t}$ such that for every
$t\in [\widetilde{t}, t_2]$, there holds $ \sum_{i=1}^n \| \gamma^i_t - \gamma^i _\infty \|_{W^{2,2}}\leqslant {\varepsilon}/{4}$ and $\left|{L}(\SS_{t})-{L}(\SS_\infty) \right|^{\theta} \leqslant{\varepsilon}/{4}$. 
We get 
\begin{align*}
\bigg(\int_{0}^1\bigg(\gamma^i(x,t_2)-\gamma^i(x,t_1)\bigg)^2\,dx\bigg)^{1/2}
&= \bigg(\int_{0}^1\bigg(\int_{t_1}^{t_2}\gamma^i_t(x,t)\,dt\bigg)^2\,dx\bigg)^{1/2}\\
&\leqslant \int_{t_1}^{t_2}\bigg(\int_0^1 (\gamma^i_t(x,t))^2\,dx\bigg)^{1/2}\,dt\\
&=\int_{t_1}^{t_2}\bigg(\int_{\gamma^i_t} k^2\,ds\bigg)^{1/2}\,dt\\
&\leqslant \int_{t_1}^{t_2}\bigg(\int_{\SS_t} k^2\,ds\bigg)^{1/2}\,dt\\
& \leqslant \frac{C_{\mathrm LS}}{\theta} \left|{L}(\SS_t)-{L}(\SS_\infty) \right|^{\theta}<\frac{\varepsilon C_{\mathrm LS}}{4\theta}\,.
\end{align*}
This implies that $\gamma^i(\cdot,t):[0,1]\to\mathbb{R}^2$
is a Cauchy sequence and from it we can deduce the desired convergence.

\bigskip

After all this discussion, the following questions are rather natural.

\begin{oprob}\label{ooo3000}\ \begin{itemize}
\item In the hypotheses of Theorem~\ref{prop:Convergence}, does the whole sequence of networks $\SS_t$ converge
in $C^{1,\alpha}\cap W^{2,2}$, for every $\alpha\in(0,1/2)$, also if the limit network is a degenerate (embedded) regular network with zero curvature, as $t\to+\infty$?
\item The conclusions can be extended to the general situation described in Remark~\ref{singlim}? For instance, if the flow of networks has an infinite sequence of singular times going to $+\infty$?
\end{itemize}
\end{oprob}

\subsection{Stability}
Exploiting the {\L}ojasiewicz--Simon inequality~\eqref{eq:LojaNetworkMinimali}, it is also possible to prove a stability result: if a flow starts sufficiently close to a regular network with zero curvature composed of straight segments, then it exists for every time and smoothly converges to a (possibly different) network with zero curvature.

\begin{thm}[Theorem~5.3 in~\cite{PlPo22A}]\label{stability}
Let $\SS_\ast=(\gamma^1_\ast, \ldots,\gamma^n_\ast)$ be a regular network with zero curvature, composed of straight segments.
Then, there exists $\delta>0$ such that if $\SS_0=(\gamma^1_0, \ldots,\gamma^n_0)$ is a smooth regular network with the same topological structure and the same end--points of $\SS_\ast$ such that
\begin{equation*}
\sum_{i=1}^n \| \gamma^i_0- \gamma^i_\ast \|_{W^{2,2}} \leqslant \delta\,,
\end{equation*}
the flow by curvature of the network $\SS_0$ exists smooth for all times and smoothly converges, as $t\to+\infty$, to a regular network $\SS_\infty=(\gamma^1_\infty, \ldots,\gamma^n_\infty)$ with zero curvature (that is, composed of straight segments) satisfying ${L}(\SS_\infty)={L}(\SS_\ast)$.
\end{thm}

\begin{rem}
The special case in which $\SS_\ast$ is a triod was first considered in~\cite{kinderliu} and one can actually adapt such proof to the case in which $\SS_\ast$ is an isolated critical point of the length functional.
\end{rem}

\begin{rem}
It is not necessarily true that $\SS_\ast=\SS_\infty$, but there are some cases in which we are able to determine $\SS_\infty$:
\begin{itemize}
\item If the network $\SS_\ast$ is an isolated critical point of the length functional, then $\SS_\infty$ must coincide with $\SS_\ast$ and this is always the case if $\SS_\ast$ is a tree. 
\item Suppose that $\SS_\ast$ is a network composed of a regular hexagon $H$ with area $A_\ast$ and six straight segments connecting the vertices of a bigger regular hexagon. Then, $\SS_\ast$ is not an isolated critical point of the length functional, indeed, all the networks composed of concentric hexagons and straight segments connecting the end--points give a one--parameter family of critical points with the same length, see Figure~\ref{Fig:Ragnatela}. We underline that there are no other critical points of the length functional with this topology and with the same end--points. 
\begin{figure}[H]
\begin{center}
\begin{tikzpicture}[scale=2]
\draw[thick]
(-0.5,-0.86)--(-0.375,-0.649)
(-1,0)-- (-0.75,0)
(-0.5,0.86)--(-0.375,0.649)
(0.5,-0.86)--(0.375,-0.649)
(1,0)-- (0.75,0)
(0.375,0.649)--(0.5,0.86) ;
\draw[thick]
(-0.375,-0.649)--(-0.75,0)--(-0.375,0.649)--
(0.375,0.649)--(0.75,0)--(0.375,-0.649)--(-0.375,-0.649);
\draw[thick,dashed]
(-0.25,-0.43)--(-0.375,-0.649)
(-0.5,0)-- (-0.75,0)
(-0.25,0.43)--(-0.375,0.649)
(0.25,-0.43)--(0.375,-0.649)
(0.5,0)-- (0.75,0)
(0.375,0.649)--(0.25,0.43);
\draw[thick,dashed]
(-0.25,-0.43)--(-0.5,0)--(-0.25,0.43)--
(0.25,0.43)--(0.5,0)--(0.25,-0.43)--(-0.25,-0.43);
\draw[thick,dotted]
(-0.16,-0.28)--(-0.33,0)--(-0.16,0.28)--
(0.16,0.28)--(0.33,0)--(0.16,-0.28)--(-0.16,-0.28);
\draw[thick,dotted]
(-0.16,-0.28)--(-0.25,-0.43)
(-0.33,0)-- (-0.5,0)
(-0.16,0.28)--(-0.25,0.43)
(0.16,-0.28)--(0.25,-0.43)
(0.33,0)-- (0.5,0)
(0.25,0.43)--(0.16,0.28) ;
\fill[black](1,0) circle (1pt); 
\fill[black](-1,0) circle (1pt); 
\fill[black](0.5,-0.86) circle (1pt); 
\fill[black](-0.5,-0.86) circle (1pt); 
\fill[black](0.5,0.86) circle (1pt); 
\fill[black](-0.5,0.86) circle (1pt); 
\end{tikzpicture}
\end{center}
\begin{caption}{Three different networks with zero curvature with the same end--points and topology. They all have the same length.\label{Fig:Ragnatela}}
\end{caption}
\end{figure}

Suppose now that $\SS_0$ is regular network with the same end--points and the same topology of $\SS_\ast$, sufficiently close to $\SS_\ast$ and such that the area enclosed by the loop is equal to $A_0$. Then, $\SS_\infty$ coincides with $\SS_\ast$ if and only if $A_0=A_\ast$, as the area enclosed by any loop of six curves is preserved during the evolution and $\SS_\ast$ is the unique network with zero curvature and area $A_\ast$ among the possible limit critical networks. We remark that if $A_0\not=A_\ast$, we then have an example where the limit network $\SS_\infty$ is different by $\SS_\ast$, indeed $\SS_\infty$ must be the unique network of such family with a central regular hexagon of area $A_0$.
\end{itemize}
\end{rem}

We conclude this section with a couple of open problems.

\begin{oprob}\label{open-stab1}
Is it possible to replace the $W^{2,2}$--closedness condition in the stability Theorem~\ref{stability} 
with some ``small distance'' condition between the networks that allows also topological changes, for instance, the Hausdorff distance?
\end{oprob}

\begin{oprob}\label{open-stab2}
It is possible to ``identify'' the limit network $\SS_\infty$ in the stability Theorem~\ref{stability}, in general?
This question is relevant in the non--trivial case when $\SS_\ast$ belongs to a continuous family
of critical points for the length functional.
\end{oprob}

\section{An isoperimetric estimate}\label{dsuL} 

Given the smooth flow $\SS_t=F(\SS,t)$, we take two points $p=F(x,t)$ and
$q=F(y,t)$ belonging to $\SS_t$. A couple $\left(p,q\right)$ is ``admissible'' 
if the segment joining $p$ and $q$ does not intersect the network $\SS_t$ in other points.
We call $\mathfrak{A}$ the
class of the admissible couple. Given an admissible pair $\left( p,q\right)$ we consider the set 
of the embedded curves $\Gamma_{p,q}$ contained in $\SS_t$ connecting $p$ and $q$,
forming with the segment $\overline{pq}$ a Jordan curve.
Thus, it is well defined the area of the open region $\mathcal{A}_{p,q}$ enclosed by any Jordan curve constructed in this way and, for any pair $(p,q)$, 
we call $A_{p,q}$ the smallest area of all such possible regions $\mathcal{A}_{p,q}$. 
If $p$ and $q$ are both points of a set of curves forming a loop,
we define $\psi(A_{p,q})$ as 
$$
\psi(A_{p,q})=\frac{A}{\pi}\sin\left( \frac{\pi}{A}A_{p,q}\right)\,,
$$
where $A=A(t)$ is the area of the connected component
of $\Omega\setminus\SS_t$ which contains the open segment
joining $p$ and $q$.

We consider the function 
$\Phi_t: \SS\times\SS\to \R\cup\{+\infty\}$ as 
\begin{equation*}
\Phi_t(x,y)= 
\begin{cases}
\frac{\vert p-q\vert^2}{\psi(A_{p,q})}\qquad &\,\text{if $x\not=y$ and $x, y$ are points of a loop;}\\
\frac{\vert p-q\vert^2}{A_{p,q}}\qquad &\,\text{if $x\not=y$ and $x, y$ are not both points of a loop;}\\
4\sqrt{3}\;\; &\,\text{if $x$ and $y$ coincide with one of the 3--points $O^i$ of $\SS$;}\\
+\infty\;\; &\,\text{if $x=y\not=O^i$;}\\
\end{cases}
\end{equation*} 
where $p=F(x,t)$ and $q=F(y,t)$.

\begin{rem}
Following the argument of Huisken in~\cite{huisk2},
in the definition of the function $\Phi_t$ we introduce the function $\psi(A_{p,q})$, when
the two points belong to a loop because we want to maintain the function smooth
also when $A_{p,q}$ is equal to $A/2$.
\end{rem}

In the following, with a little abuse of notation, we consider the function $\Phi_t$ defined on 
$\SS_t\times\SS_t$ and we speak of admissible pair for the couples of points
$(p,q)\in\SS_t\times\SS_t$ instead of $(x,y)\in\SS\times\SS$.

We define $E(t)$ as the infimum of $\Phi_t$ between all admissible couple of points
$p=F(x,t)$ and $q=F(y,t)$:
\begin{equation}\label{equant}
E(t) = \inf_{(p,q)\in\mathfrak{A}}\Phi_t
\end{equation}
for every $t\in[0,T)$.\\
We call $E(t)$ ``embeddedness measure''. We underline that similar geometric quantities have already been applied to analogous problems in~\cite{chzh,hamilton3,huisk2}.

The following lemma holds, for its proof in the case of a compact network see~\cite[Theorem~2.1]{chzh}.

\begin{lem}\label{lemet1}
The infimum of the function $\Phi_t$ between all admissible couples $(p,q)$ is actually a minimum.
Moreover, assuming that $0<E(t)<4\sqrt{3}$, for any minimizing pair 
$(p,q)$ we have $p\ne q$ and
neither $p$ nor $q$ coincides with one of the 3--points $O^i(t)$ of $\SS_t$.
\end{lem}

\begin{rem}\label{opendl}
In the case of an open network without end--points, since the network is asymptotically $C^1$--close to a family of halflines (and during its curvature motion such halflines are fixed), there holds that if the infimum of $\Phi_t$ is less than a ``structural'' constant depending only on such halflines, then it is a minimum. By means of such modification to this lemma, all the rest of the analysis of this chapter also holds for the evolution of open networks, we let the details and the easy modifications of the arguments to the reader.
\end{rem}

Notice that it follows that the network $\SS_t$ is embedded {\em if and only if} $E(t)>0$. Moreover, $E(t)\leqslant 4\sqrt{3}$ always holds, thus when
$E(t)>0$ the two points $(p,q)$ of a minimizing pair can coincide if
and only if $p=q=O^i(t)$.\\
Finally, since the evolution is smooth, it is easy to see that the
function $E:[0,T)\to\R$ is locally Lipschitz, in particular, $\frac{dE(t)}{dt}>0$ exists for almost every time $t\in[0,t)$.

If the curvature flow $\mathbb{S}_t$ has fixed end--points $\{P^1, P^2,\dots, P^l\}$ on the boundary of a strictly convex set
$\Omega$, we consider the flows ${\mathbb{H}}^i_t$ each obtained as the union of $\mathbb{S}_t$ with its reflection $\mathbb{S}^{R_i}_t$ with respect to the end--point $P^i$, as we described at the end of Section~\ref{geopropsub}.\\
We underline that this is still a smooth curvature flow (as the compatibility conditions of every order in Definition~\ref{ncompcond} are satisfied by $\SS_t$ at its end--points) without self--intersections, where $P^i$ is no more an end--point and the
number of triple junctions of ${\mathbb{H}}^i_t$ is exactly twice the number of the ones of $\SS_t$.
\begin{figure}[H]
\begin{center}
\begin{tikzpicture}[rotate=25,scale=1.7]
\draw[color=black,scale=1,domain=-3.15: 3.15,
smooth,variable=\t,rotate=0]plot({1*sin(\t r)},
{1*cos(\t r)}); 
\draw[scale=0.5]
(-2,0) to [out=45, in=-160,looseness=1] (-0.85,0.25)
(-0.85,0.25) to [out=-40, in=150,looseness=1] (0.75,-0.35)
(-0.85,0.25)to [out=80, in=-90,looseness=1] (0,2)
(0.75,-0.35)to [out=30, in=-120,looseness=1](2,0)
(0.75,-0.35)to [out=-90, in=90,looseness=1](0,-2);
\draw[color=black,scale=1,domain=-3.15: 3.15,
smooth,variable=\t,rotate=0,shift={(2,0)}]plot({1*sin(\t r)},
{1*cos(\t r)}); 
\draw[scale=0.5]
(6,0) to [out=-135, in=20,looseness=1] (4.85,-0.25)
(4.85,-0.25) to [out=140, in=-30,looseness=1] (3.25,0.35)
(4.85,-0.25)to [out=-100, in=90,looseness=1] (4,-2)
(3.25,0.35)to [out=-150, in=90,looseness=1](2,0)
(3.25,0.35)to [out=90, in=-90,looseness=1](4,2);
\draw[color=black,scale=1,domain=-3.15: 3.15,
smooth,variable=\t,rotate=0,shift={(-2,0)}]plot({1*sin(\t r)},
{1*cos(\t r)}); 
\draw[scale=0.5]
(-2,0) to [out=-135, in=40,looseness=1] (-3.15,-0.25)
(-3.15,-0.25) to [out=140, in=-30,looseness=1] (-4.75,0.35)
(-3.15,-0.25)to [out=-100, in=90,looseness=1] (-4,-2)
(-4.75,0.35)to [out=-150, in=60,looseness=1](-6,0)
(-4.75,0.35)to [out=90, in=-90,looseness=1](-4,2);
\draw[color=black,scale=1,domain=-3.15: 3.15,
smooth,variable=\t,rotate=0,shift={(0,2)}]plot({1*sin(\t r)},
{1*cos(\t r)});
\draw[scale=0.5]
(2,4) to [out=-135, in=20,looseness=1] (0.85,3.75)
(0.85,3.75) to [out=140, in=-30,looseness=1] (-0.75,4.35)
(0.85,3.75)to [out=-100, in=90,looseness=1] (0,2)
(-0.75,4.35)to [out=-150, in=60,looseness=1](-2,4)
(-0.75,4.35)to [out=90, in=-90,looseness=1](0,6);
\draw[color=black,scale=1,domain=-3.15: 3.15,
smooth,variable=\t,rotate=0,shift={(0,-2)}]plot({1*sin(\t r)},
{1*cos(\t r)}); 
\draw[scale=0.5]
(2,-4) to [out=-135, in=20,looseness=1] (0.85,-4.25)
(0.85,-4.25) to [out=140, in=-30,looseness=1] (-0.75,-3.65)
(0.85,-4.25)to [out=-100, in=90,looseness=1] (0,-6)
(-0.75,-3.65)to [out=-150, in=60,looseness=1](-2,-4)
(-0.75,-3.65)to [out=90, in=-90,looseness=1](0,-2);
\path[font=\footnotesize,rotate=-25]
(1.35,-0.4) node[left]{$\SS_t$}
(-2.95,-0.4) node[left]{${\mathbb{H}}^1_t$}
(0.05,-2.77) node[left]{${\mathbb{H}}^2_t$}
(3.4,0.4) node[left]{${\mathbb{H}}^3_t$}
(0.1,3) node[left]{${\mathbb{H}}^4_t$}
(-0.88,-0.45) node[left]{$P^1$}
(0.43,-0.81) node[left]{$P^2$}
(.9,0.4) node[left]{$P^3$}
(-.08,1.11) node[left]{$P^4$}
(-.05,.05) node[left]{$O^1$}
(0.48,-0.15) node[left]{$O^2$};
\end{tikzpicture}
\end{center}
\begin{caption}{A tree--like network $\mathbb{S}_t$ with the associated networks ${\mathbb{H}}^i_t$.}
\end{caption}
\end{figure}
We define for the networks ${\mathbb{H}}^i_t$ the functions $E^i:[0,T)\to\mathbb{R}$, analogous to the function
$E:[0,T)\to\mathbb{R}$ of $\mathbb{S}_t$ and, for every $t\in[0,T)$, we call $\Pi(t)$ the minimum of the values $E^i(t)$. The function $\Pi:[0,T)\to\R$ is still a locally Lipschitz function (hence, differentiable for almost every time), clearly satisfying $\Pi(t)\leqslant E^i(t)\leqslant E(t)$ for all $t\in[0,T)$.
Moreover, as there are no self--intersections, by construction, we have $\Pi(0)>0$. If we prove that $\Pi(t)\geqslant C>0$ for all $t\in[0,T)$, form some constant $C\in\R$,
then, we can conclude that also $E(t)\geqslant C>0$, for all $t\in[0,T)$.

\begin{thm}\label{dlteo} 
Let $\Omega$ be an open, bounded, strictly convex subset of $\mathbb{R}^2$.
Let $\mathbb{S}_0$ be an initial regular network
with at most two triple junctions and let the $\mathbb{S}_t$ be a 
smooth evolution by curvature of $\mathbb{S}_0$, defined in a maximal time interval $[0,T)$.\\
Then, there exists a constant $C>0$ depending only on $\SS_0$ 
such that $E(t)\geqslant C>0$, for every $t\in[0,T)$.
In particular, the networks $\SS_t$ remain ({\em uniformly}, in a sense) embedded during the flow.
\end{thm}

To prove this theorem we first show the next proposition and lemma.
\begin{prop}\label{lemet2} 
Let $t\in[0,T)$ such that
\begin{itemize}
\item $0<E(t)<1/4$,
\item for at least one minimizing pair $(p,q)$ of $\Phi_t$, the curve $\Gamma_{p,q}$ contains at most two triple junctions with neither $p$ nor $q$ coinciding with one of the end--points $P^i$.
\end{itemize}
Then, if the derivative $\frac{dE(t)}{dt}$ exists, it is positive.
\end{prop}
\begin{proof}
By simplicity, we consider in detail only the case shown in Figure~\ref{disuelle}. The computations in the other situations are analogous. 
\begin{figure}[H]
\begin{center}
\begin{tikzpicture}[scale=0.85]
\draw[color=black,scale=1,domain=-3.15: 3.15,
smooth,variable=\t,shift={(-1,0)},rotate=0]plot({3.25*sin(\t r)},
{2.5*cos(\t r)}) ;
\filldraw[fill=black!10!white,shift={(-8,0)}]
(6,0.5) 
to[out=-165,in=60, looseness=1] (5.19,0)--(5.19,0)
to[out=-60,in=150, looseness=1] (7.62,-0.81)--(7.62,-0.81) 
to[out=30,in=-167.7,looseness=1.5] (8.483,-0.59) -- (6,0.5) ;
\draw[color=black!40!white]
(0,1) to[out=-90, in=90,looseness=1] (0,-0.67)
(0.15,0.98) to[out=-90, in=90,looseness=1] (0.15,-0.61)
(0.3,0.92) to[out=-90, in=90,looseness=1] (0.3,-0.585)
(0.45,0.87) to[out=-90, in=90,looseness=1] (0.45,-0.565)
(0.6,0.8) to[out=-90, in=90,looseness=1] (0.6,-0.55)
(0.75,0.7) to[out=-90, in=90,looseness=1] (0.75,-0.52)
(0.9,0.65) to[out=-90, in=90,looseness=1] (0.9,-0.5)
(1.05,0.57) to[out=-90, in=90,looseness=1] (1.05,-0.45)
(1.2,0.52) to[out=-90, in=90,looseness=1] (1.2,-0.4)
(1.35,0.4) to[out=-90, in=90,looseness=1] (1.35,-0.265)
(-0.15,1.025) to[out=-90, in=90,looseness=1] (-0.15,-0.68)
(-0.3,1.035) to[out=-90, in=90,looseness=1] (-0.3,0.15)
(-0.3,0) to[out=-90, in=90,looseness=1] (-0.3,-0.75)
(-0.45,1.02) to[out=-90, in=90,looseness=1] (-0.45,0.35)
(-0.45,0.1) to[out=-90, in=90,looseness=1] (-0.45,-0.74)
(-0.6,0.98) to[out=-90, in=90,looseness=1] (-0.6,0.1)
(-0.6,0) to[out=-90, in=90,looseness=1] (-0.6,-0.67)
(-0.75,0.93) to[out=-90, in=90,looseness=1] (-0.75,-0.63)
(-0.9,0.84) to[out=-90, in=90,looseness=1] (-0.9,-0.6)
(-1.05,0.75) to[out=-90, in=90,looseness=1] (-1.05,-0.573)
(-1.2,0.655) to[out=-90, in=90,looseness=1] (-1.2,-0.56)
(-1.35,0.595) to[out=-90, in=90,looseness=1] (-1.35,-0.545)
(-1.5,0.559) to[out=-90, in=90,looseness=1] (-1.5,-0.535)
(-1.65,0.526) to[out=-90, in=90,looseness=1] (-1.65,-0.527)
(-1.8,0.51) to[out=-90, in=90,looseness=1] (-1.8,-0.517)
(-1.95,0.47) to[out=-90, in=90,looseness=1] (-1.95,0)
(-1.95,-0.2) to[out=-90, in=90,looseness=1] (-1.95,-0.5)
(-2.1,0.445) to[out=-90, in=90,looseness=1] (-2.1,0.18)
(-2.1,-0.17) to[out=-90, in=90,looseness=1] (-2.1,-0.47)
(-2.25,0.405) to[out=-90, in=90,looseness=1] (-2.25,0)
(-2.25,-0.1) to[out=-90, in=90,looseness=1] (-2.25,-0.43)
(-2.4,0.34) to[out=-90, in=90,looseness=1] (-2.4,-0.36)
(-2.55,0.24) to[out=-90, in=90,looseness=1] (-2.55,-0.255)
(-2.7,0.12) to[out=-90, in=90,looseness=1] (-2.7,-0.12);
\draw[color=white,shift={(-8,0)}]
(7.62,-0.81) 
to[out=30,in=-167.7,looseness=1.5] (8.483,-0.59);
\draw[shift={(-8,0)}]
 (3.75,0) node[left] {$P^1$} to[out=30,in=110,looseness=1] (4.37,0) 
to[out=-50,in=180,looseness=1] (5.19,0)node[above] {$O^1$}
to[out=60, in=-145,looseness=1] (7,0.81)
to[out=35,in=150, looseness=1] (8.62,0.81) 
to[out=-30,in=90,looseness=1.5] (9.44,0)
(7.62,-0.81) 
to[out=30,in=-90, looseness=1](9.44,0)
(7.62,-0.81) 
to[out=150,in=-60, looseness=1] (5.19,0)
(7.62,-0.81) 
to[out=-90,in=60, looseness=1] (9.1,-1.9)node[below]{$P^2$};
\draw
 (0.438,-0.58) -- (-1.98,0.49); 
\path[font= \footnotesize] 
(-2,0.3) node[below] {$\mathcal{A}_{p,q}$};
\path[font= \Large]
(-3.75,-1.8) node[below] {$\Omega$};
\path[font=\large]
(-0.45,0.53) node[below] {$A$}
(-0.3,-1) node[left]{$O^2$}
(-2.03,1) node[below] {$p$}
(0.7,-0.52) node[below] {$q$};
\end{tikzpicture}
\end{center}
\begin{caption}{The situation considered in the computations of Proposition~\ref{lemet2}.\label{disuelle}}
\end{caption}
\end{figure}
Let $0<E(t)<1/4$ and let $(p,q)$ a minimizing pair for $\Phi_t$ such that the two points are both distinct
from the end--points $P^i$. We choose a value $\eps>0$ smaller than the ``geodesic'' distances of $p$ and $q$ from
the 3--points of $\SS_t$ and between them.\\
Possibly taking a smaller $\eps>0$, we fix an arclength
coordinate $s\in (-\eps,\eps)$ and a local parametrization $p(s)$ of
the curve containing $p$ such that $p(0)=p$, with the 
same orientation as the original one.
Let $\eta(s)=\vert p(s)-q\vert$, since 
\begin{equation*}
E(t)=\min_{s\in(-\varepsilon,\varepsilon)}\frac{\eta^2(s)}{\psi(A_{p(s),q})}=
\frac{\eta^2(0)}{\psi(A_{p,q})}\,,
\end{equation*}
if we differentiate in $s$ we obtain 
\begin{equation}\label{eqdsul}
\frac{d\eta^2(0)}{ds}\psi(A_{p(0),q})=
\frac{d\psi(A_{p(0),q})}{ds}\eta^2(0)\,.
\end{equation}

We underline that we are considering the function $\psi$ because we are doing all
the computation for the case shown in Figure~\ref{disuelle}, where there is a loop.
For a network without loops the computations are simpler:
instead of formula~\eqref{eqdsul}, one has
$$
\frac{d\eta^2(0)}{ds}A_{p(0),q}=
\frac{dA_{p(0),q}}{ds}\eta^2(0)\,,
$$
see~\cite[Page~281]{mannovtor}, for instance.

As the intersection of the segment $\overline{pq}$ with the network is transversal, 
we have an angle $\alpha(p)\in(0,\pi)$ 
determined by the unit tangent $\tau(p)$ and the vector $q-p$.\\
We compute
\begin{align*} 
\left.\frac{{ d} \eta^2(s)}{{ d} s}\right\vert_{s=0}
 &=\,
-2 \langle \tau(p)\,\vert\, q-p\rangle = -2 \vert
p-q\vert \cos\alpha(p)\\ 
\left.\frac{{ d} A(s)}{{ d} s}\right\vert_{s=0} 
&=\, 0\\
\left.\frac{{ d} A_{p(s),q}}{{ d} s}\right\vert_{s=0} 
&=\, 
\frac 12 \vert \tau(p) \wedge (q-p)\vert = 
\frac 12 \langle \nu(p)\,\vert\, q-p\rangle = 
\frac 12 \vert p-q\vert \sin\alpha(p)\\
\left.\frac{{ d} \psi(A_{p(s),q})}{{ d} s} \right\vert_{s=0} 
&=\, 
\frac{{ d} A_{p,q}}{{ d} s}\cos\left(\frac{\pi}{A}A_{p,q}\right)\\
\, &=\,
\frac 12 \vert p-q\vert \sin\alpha(p)\cos\left(\frac{\pi}{A}A_{p,q}\right)\,. \\
\end{align*}
Putting these derivatives in equation~\eqref{eqdsul} and
recalling that $\eta^2(0)/\psi(A_{p,q})=E(t)$, we get
\begin{equation}\label{topolino2}
\cot\alpha(p) = -\frac{\vert p-q\vert^2}{4\psi(A_{p,q})}\cos\left(\frac{\pi}{A}A_{p,q}\right)
= -\frac{E(t)}{4}\cos\left(\frac{\pi}{A}A_{p,q}\right)\,.
\end{equation}
Since $0<E(t)<\frac{1}{4}<4(2-\sqrt{3})$,
we have $\sqrt{3}-2<\cot\alpha(p)<0$, which implies
\begin{equation}\label{alpha}
\frac{\pi}{2} <\alpha(p) < \frac {7\pi}{12} \,.
\end{equation}
The same argument clearly holds for the point $q$, hence 
defining $\alpha(q)\in(0,\pi)$ to be the angle determined by the
unit tangent $\tau(q)$ and the vector $p-q$, by equation~\eqref{topolino2}
it follows that $\alpha(p)=\alpha(q)$ and we simply write $\alpha$
for both.\\
We consider now a different variation, moving at the same time the
points $p$ and $q$, in such a way that $\frac{dp(s)}{ds}=\tau(p(s))$ and 
$\frac{dq(s)}{ds}=\tau(q(s))$.\\
As above, letting $\eta(s)=\vert p(s)-q(s)\vert$, by minimality we have
\begin{align}\label{eqdersec}
\frac{{ d} \eta^2(0)}{{ d} s}\left.\psi(A_{p(s),q(s)})\right\vert_{s=0}
&=
\left( \left.\frac{{ d} \psi(A_{p(s),q(s)})}{{ d} s}\right\vert_{s=0}\right) \eta^2(0) \;\;\text{ { and} }\;\;\nonumber\\
\frac{{ d}^2 \eta^2(0)}{{ d} s^2}\left.\psi(A_{p(s),q(s)})\right\vert_{s=0}
&\geqslant 
\left( \left.\frac{{ d}^2\psi( A_{p(s),q(s)})}{{ d} s^2}\right\vert_{s=0}\right) \eta^2(0)\,.
\end{align}
Computing as before,
\begin{align*}
\left.\frac{{ d} \eta^2(s)}{{ d} s}\right\vert_{s=0}
=&\,2 \langle p-q \,\vert\,\tau(p)-\tau(q)\rangle = -4\vert p-q\vert\cos\alpha \\
\left.\frac{{ d} A_{p(s),q(s)}}{{ d} s}\right\vert_{s=0} 
=&\, -\frac12\langle p-q\,\vert\,\nu(p)+\nu(q)\rangle
=+\vert p-q\vert\sin\alpha\\
\left.\frac{{ d}^2 \eta^2(s)}{{ d} s^2}\right\vert_{s=0} 
=&\,2 \langle \tau(p)-\tau(q) \,\vert\,\tau(p)-\tau(q)\rangle +
2 \langle p-q \,\vert\, k(p)\nu(p) -k(q)\nu(q)\rangle\\
=&\, 2\vert \tau(p)-\tau(q)\vert^2+
2 \langle p-q \,\vert\, k(p)\nu(p) -k(q)\nu(q)\rangle\\
=&\, 8\cos^2\alpha+
2 \langle p-q \,\vert\, k(p)\nu(p) -k(q)\nu(q)\rangle\\
\left.\frac{{ d}^2 A_{p(s),q(s)}}{{ d} s^2}\right\vert_{s=0} 
=&\,-\frac12\langle \tau(p)-\tau(q)\,\vert\,\nu(p)+\nu(q)\rangle
+\frac12\langle p-q\,\vert\,k(p)\tau(p)+k(q)\tau(q)\rangle\\
=&\,-\frac12\langle\tau(p)\,\vert\,\nu(q)\rangle
+\frac12\langle \tau(q)\,\vert\,\nu(p)\rangle+\frac12\langle p-q\,\vert\,k(p)\tau(p)+k(q)\tau(q)\rangle\\
=&\,-2\sin\alpha\cos\alpha
-1/2\vert p-q\vert(k(p)-k(q))\cos\alpha\\
\left.\frac{{ d}^2 \psi(A_{p(s),q(s)})}{{ d} s^2}\right\vert_{s=0} 
=&\,\left.\frac{{ d}}{{ d} s}\left\{
\frac{{ d} A_{p(s),q(s)}}{{ d} s}
\cos\left(\frac{\pi}{A}A_{p(s),q(s)}\right) 
\right\}\right\vert_{s=0}\\
=&\, (-2\sin\alpha\cos\alpha-\frac12\vert p-q\vert(k(p)-k(q))\cos\alpha)
\cos\left( \frac{\pi}{A}A_{p,q}\right) \\
&\,-\frac{\pi}{A}\vert p-q\vert^2 \sin^2\alpha\sin\left(\frac{\pi}{A}A_{p,q}\right)\,. \\
\end{align*} 
Substituting the last two relations in inequality~\eqref{eqdersec}, we get 
\begin{align*}
(8\cos^2\alpha+&\, 
2\langle p-q \,\vert\, k(p)\nu(p)-k(q)\nu(q)\rangle)\psi(A_{p,q})\\
\geqslant&\, \vert p-q\vert^2
\left\{ 
(-2\sin\alpha\cos\alpha-\frac12\vert p-q\vert(k(p)-k(q))\cos\alpha)
\cos\left( \frac{\pi}{A}A_{p,q}\right)\right.\\
&\,\left.-\frac{\pi}{A}\vert p-q\vert^2 \sin^2\alpha\sin\left(\frac{\pi}{A}A_{p,q}\right) 
\right\} \,,
\end{align*}
hence, keeping in mind that 
$\tan\alpha=
\frac{-4}{E(t)\cos\left(\frac{\pi}{A}A_{p(s),q(s)} \right) }$, we obtain
\begin{align}
2\psi(A_{p,q})\langle p-q \,\vert\, &\,k(p)\nu(p) -k(q)\nu(q)\rangle+1/2\vert p-q\vert^3(k(p)-k(q))\cos\alpha 
\cos\left(\frac{\pi}{A}A_{p,q}\right)\nonumber\\
\geqslant&\,-2\sin\alpha\cos\alpha\vert p-q\vert^2\cos\left(\frac{\pi}{A}A_{p,q}\right)\nonumber\\
&\,-8\psi(A_{p,q})\cos^2\alpha
+\vert p-q\vert^4 \sin^2\alpha
\left[-\frac{\pi}{A}\sin\left(\frac{\pi}{A}A_{p,q}\right) \right]\nonumber\\
=&\,-2\psi(A_{p,q})\cos^2\alpha
\left(\tan\alpha\frac{\vert p-q\vert^2}{\psi(A_{p,q})}
\cos\left(\frac{\pi}{A}A_{p,q}\right)+ 4\right)\nonumber\\
&\,+\vert p-q\vert^4\sin^2\alpha
\left[ -\frac{\pi}{A}\sin\left(\frac{\pi}{A}A_{p,q}\right) \right] \nonumber\\
=&\,\vert p-q\vert^4\sin^2\alpha
\left[ -\frac{\pi}{A}\sin\left(\frac{\pi}{A}A_{p,q}\right) \right]\,.\label{eqfin}
\end{align}

We now compute the derivative $\frac{dE(t)}{dt}$ by means of the Hamilton's trick (see~\cite{hamilton2} or~\cite[Lemma~2.1.3]{Manlib}), that is,
$$
\frac{dE(t)}{dt} = \frac{\partial\,}{\partial
 t}\Phi_t(\overline{p},\overline{q})\,,
$$
for {\em any} minimizing pair $(\overline{p},\overline{q})$ for $\Phi_t$. In particular, 
$\frac{dE(t)}{dt} = \frac{\partial}{\partial t}\Phi_t(p,q)$ and, we recall, $\frac{\vert
 p-q\vert^2}{\psi(A_{p,q})}=E(t)$.\\
Notice that by minimality of the pair $(p,q)$, we are free to choose the 
``motion'' of the points $p(s)$, $q(s)$ ``inside'' the networks $\Gamma_s$
in computing such partial derivative, that is,
$$
\frac{dE(t)}{dt} = \frac{\partial\,}{\partial t}\Phi_t(p,q) = \left.\frac{d\,}{ds}\Phi_t(p(s),q(s))\right\vert_{s=t}\,.
$$
Since locally the networks are moving by curvature and we know that
neither $p$ nor $q$ coincides with the 3--point, 
we can find $\varepsilon>0$ and two smooth curves $p(s), q(s)\in\Gamma_s$ for
every $s\in(t-\varepsilon,t+\varepsilon)$ such that
\begin{align*}
p(t) &=\, p \qquad \text{ and }\qquad \frac{dp(s)}{ds} = k(p(s), s)
~\nu(p(s),s)\,, \\
q(t) &=\, q \qquad \text{ and }\qquad \frac{dq(s)}{ds}= k(q(s),s)
~\nu(q(s),s)\,. 
\end{align*}
Then,
\begin{equation}\label{eqderE}
\frac{dE(t)}{dt} =\frac{\partial\,}{\partial t}\Phi_{t}(p,q)=\frac{1}{[\psi(A_{p,q})]^2}\left.
\left(\psi(A_{p,q})
\frac{d\vert p(s)-q(s)\vert^2}{ds} - 
\vert p-q\vert^2 \frac{d\psi(A_{p(s),q(s)})}{ds}\right)\right\vert_{s=t}\,.
\end{equation}
With a straightforward computation, we get the following equalities, 
\begin{align*}
\left.\frac{{ d} \vert p(s)-q(s)\vert^2}{{ d} s} \right\vert_{s=t}
=&\, 2 \langle p-q\,\vert\, k(p)\nu(p) -k(q)\nu(q)\rangle\\
\left.\frac{{ d}A(s)}{{ d} s}\right\vert_{s=t} 
=&\,-\frac{4\pi}{3}\\
\left.\frac{{ d} A_{p(s),q(s)}}{{ d} s}\right\vert_{s=t} 
=&\,\int_{\Gamma_{p,q}} \langle\underline{k}(s)\,\vert\nu_{\xi_{p,q}}\rangle\,ds
+ \frac12\vert p-q\vert\langle {\nu_{[p,q]}}\,\vert\, k(p)\nu(p)+k(q)\nu(q)\rangle\\
=&\,2\alpha -\frac{4\pi}{3}-\frac 12 \vert p-q\vert(k(p)-k(q))\cos\alpha\\
\left.\frac{{ d} \psi(A_{p(s),q(s)})}{{ d} s} \right\vert_{s=t}
=&\,-\frac{4\pi}{3}\left[\frac{1}{\pi}\sin\left(\frac{\pi}{A}A_{p,q}\right)
-\frac{A_{p,q}}{A}\cos\left(\frac{\pi}{A}A_{p,q} \right)\right] \\
&\,+\left(2\alpha -\frac{4\pi}{3}-\frac 12 \vert p-q\vert(k(p)-k(q))\cos\alpha\right)
\cos\left(\frac{\pi}{A}A_{p,q} \right)\\
\end{align*}
where we wrote $\nu_{\xi_{p,q}}$ and $\nu_{[p,q]}$ for the exterior
unit normal vectors to the region $A_{p,q}$, respectively at the
points of the geodesic $\xi_{p,q}$ and of the segment $\overline{pq}$.\\
We remind that in general $\frac{{ d}A(t)}{{ d} t}=-(2-m/3)\pi$
where $m$ is the number of triple junctions of the loop
(see formula~\eqref{areaevolreg}), hence, we have $\frac{{ d}A(t)}{{ d} t}=-\frac{4\pi}{3}$, since we are referring to the situation in Figure~\ref{disuelle}, where there is a loop with exactly two triple junctions.\\
Substituting these derivatives in equation~\eqref{eqderE} we get
\begin{align*}
\frac{dE(t)}{dt} 
=&\,\frac{2\langle p-q\,\vert\, k(p)\nu(p) -k(q)\nu(q)\rangle}{\psi(A_{p,q})}\\
&\,-\frac{\vert p-q\vert^2}{[\psi(A_{p,q})]^2}
\left\{ 
-\frac{4\pi}{3}\left[\frac{1}{\pi}\sin\left(\frac{\pi}{A}A_{p,q}\right)
-\frac{A_{p,q}}{A}\cos\left(\frac{\pi}{A}A_{p,q}\right)\right]
\right. \\
&\,\left. +\left(2\alpha -\frac{4\pi}{3}-\frac 12 \vert p-q\vert(k(p)-k(q))\cos\alpha\right)
\cos\left(\frac{\pi}{A}A_{p,q} \right)
 \right\} \\
\end{align*}
and, by equation~\eqref{eqfin}, \\
\begin{align*}
\frac{dE(t)}{dt} 
\geqslant 
&\,-\frac{\vert p-q\vert^2}{[\psi(A_{p,q})]^2}
\left\{ 
-\frac{4}{3}\sin\left(\frac{\pi}{A}A_{p,q}\right)
+\frac{4\pi}{3}\frac{A_{p,q}}{A}\cos\left(\frac{\pi}{A}A_{p,q}\right) \right.\\
&\,\left. +\left(2\alpha -\frac{4\pi}{3}\right)
\cos\left(\frac{\pi}{A}A_{p,q} \right)+\frac{\pi}{A}\vert p-q \vert^2 \sin^2(\alpha) 
\sin\left(\frac{\pi}{A}A_{p,q}\right)
\right\}\,.\\
\end{align*}
It remains to prove that the quantity 
\begin{align*}
&\frac{4}{3}\sin\left(\frac{\pi}{A}A_{p,q}\right)
-\frac{4\pi}{3}\frac{A_{p,q}}{A}\cos\left(\frac{\pi}{A}A_{p,q} \right) 
 +\left(\frac{4\pi}{3}-2\alpha\right)
\cos\left(\frac{\pi}{A}A_{p,q} \right)\\
&-\frac{\pi}{A}\vert p-q \vert^2 \sin^2(\alpha) 
\sin\left(\frac{\pi}{A}A_{p,q}\right)\\
\end{align*}
is positive.\\
As $E(t)=\frac{\vert p-q\vert^2}{\psi(A_{p,q})}=\frac{\vert p-q\vert^2}{\frac{A}{\pi}\sin(\frac{\pi}{A}A_{p,q})}$,
we can write
\begin{align*}
&\,\frac{4}{3}\sin\left(\frac{\pi}{A}A_{p,q}\right)
-\frac{4\pi}{3}\frac{A_{p,q}}{A}\cos\left(\frac{\pi}{A}A_{p,q} \right) 
 +\left(\frac{4\pi}{3}-2\alpha\right)
\cos\left(\frac{\pi}{A}A_{p,q} \right)\\
&\,-\frac{\pi}{A}\vert p-q \vert^2 \sin^2(\alpha) 
\sin\left(\frac{\pi}{A}A_{p,q}\right)\\
=&\,\frac{4}{3}\sin\left(\frac{\pi}{A}A_{p,q}\right)
-\frac{4\pi}{3}\frac{A_{p,q}}{A}\cos\left(\frac{\pi}{A}A_{p,q} \right) 
 +\left(\frac{4\pi}{3}-2\alpha\right)
\cos\left(\frac{\pi}{A}A_{p,q} \right)\\
&\,-E(t)\sin^2(\alpha) 
\sin^2\left(\frac{\pi}{A}A_{p,q}\right)\,.
\end{align*}
Notice that using inequality~\eqref{alpha}, we can evaluate $\frac{4\pi}{3}-2\alpha\in(\pi/6,\pi/3)$, in particular, it is positive.\\
We finally conclude the estimate of $\frac{dE(t)}{dt}$ and the proof of this proposition by separating the analysis into two cases, depending on the value of $\frac{A_{p,q}}{A}$.\\
If $0\leqslant \frac{A_{p,q}}{A}\leqslant \frac 13$,
we have
\begin{align*}
\frac{dE(t)}{dt} 
\geqslant&\,\frac{4}{3}\sin\left(\frac{\pi}{A}A_{p,q}\right)
-\frac{4\pi}{3}\frac{A_{p,q}}{A}\cos\left(\frac{\pi}{A}A_{p,q} \right) \\
& +\left(\frac{4\pi}{3}-2\alpha\right)
\cos\left(\frac{\pi}{A}A_{p,q} \right)-E(t)\sin^2(\alpha) 
\sin^2\left(\frac{\pi}{A}A_{p,q}\right)\\
\geqslant&\,\left(\frac{4\pi}{3}-2\alpha\right)
\cos\left(\frac{\pi}{A}A_{p,q} \right)-E(t)\sin^2(\alpha) 
\sin^2\left(\frac{\pi}{A}A_{p,q}\right)\\
\geqslant&\,\left(\frac{\pi}{6}\right)
\cos\left(\frac{\pi}{3} \right)-E(t) 
\sin^2\left(\frac{\pi}{3}\right)>0\,.
\end{align*}
If $\frac 13\leqslant \frac{A_{p,q}}{A}\leqslant \frac 12$,
we get
\begin{align*}
\frac{dE(t)}{dt} 
\geqslant&\,\frac{4}{3}\sin\left(\frac{\pi}{A}A_{p,q}\right)
-\frac{4\pi}{3}\frac{A_{p,q}}{A}\cos\left(\frac{\pi}{A}A_{p,q}\right)\\ 
&+\left(\frac{4\pi}{3}-2\alpha\right)
\cos\left(\frac{\pi}{A}A_{p,q} \right)-E(t)\sin^2(\alpha) 
\sin^2\left(\frac{\pi}{A}A_{p,q}\right)\\
\geqslant&\,\frac{4}{3}\sin\left(\frac{\pi}{A}A_{p,q}\right)
-\frac{4\pi}{3}\frac{A_{p,q}}{A}\cos\left(\frac{\pi}{A}A_{p,q} \right) 
 -E(t)\sin^2(\alpha) 
\sin^2\left(\frac{\pi}{A}A_{p,q}\right)\\
\geqslant&\,\frac{4}{3} \left( 
\sin\left(\frac{\pi}{3}\right)
-\frac{\pi}{3}\cos\left(\frac{\pi}{3}\right) 
 \right) 
 -E(t)>0\,.
\end{align*}
\end{proof}

\begin{rem}
We want to stress here the reason why we are able to prove Proposition~\ref{lemet2} only when $\Gamma_{p,q}$ contains at most two triple junctions and so Theorem~\ref{dlteo} only for networks with at most two 3--points. If we try to repeat the computations of the final part of this proof
considering a situation such that $\Gamma_{p,q}$ contains more than two triple junctions, as the value of $\frac{dA(t)}{dt}$ changes 
according to $\frac{{ d}A(t)}{{ d} t}=-(2-m/3)\pi$, when $m\geqslant 3$, 
we only have $\frac{{ d}A(t)}{{ d} t}\geqslant -\pi$ (instead of being equal to $-4\pi/3$), which is not sufficient to get to the inequality $\frac{dE(t)}{dt}>0$.
\end{rem}

\begin{lem}\label{trinoncoll}
Let $\Omega$ be an open, bounded, strictly convex subset of $\mathbb{R}^2$.
Let $\mathbb{S}_0$ be an initial regular network
with two triple junctions and let the $\mathbb{S}_t$ be the 
evolution by curvature of $\mathbb{S}_0$
defined in a maximal time interval $[0,T)$.
Then, there cannot be a sequence of times $t_j\to T$ such that, along such sequence, the two triple junctions converge to the same end--point of the network.
\end{lem}
\begin{proof}
Let $O^1(t)$ and $O^2(t)$ be the two triple junctions of $\SS_t$
and $P^i$ the end--points on $\partial\Omega$. Suppose, by contradiction, that $\lim_{i\to\infty} O^j(t_i)=P^1$, for $j\in\{1, 2\}$. Notice that if $\SS_t$ is not a tree, then it has the structure either of a ``lens/fish--shaped'' network (see Figure~\ref{fishshape}) or of an ``island--shaped'' network.
\begin{figure}[H]
\begin{center}
\begin{tikzpicture}[scale=1]
\draw[shift={(10,0)}] 
(-2,0)
to[out=170,in=40,looseness=1] (-2.9,1.2) 
to[out=-140,in=90,looseness=1] (-3.2,0)
(-2,0) 
to[out=-70,in=0,looseness=1] (-3,-0.9) 
to[out=-180,in=-90,looseness=1] (-3.2,0)
(-2,0) 
to[out=50,in=180,looseness=1] (-1.3,0) 
to[out=60,in=150,looseness=1.5] (-0.75,1) 
(-1.3,0)
to[out=-60,in=-120,looseness=0.9] (-0.5,-0.75)
(-0.75,1)
to[out=-30,in=-150,looseness=0.9] (1.75,1)
(0,-2.32)
to[out=90,in=60,looseness=0.9] (-0.5,-0.75);
\draw[color=black,scale=1,domain=-3.141: 3.141,
smooth,variable=\t,shift={(8.78,0)},rotate=0]plot({3.25*sin(\t r)},
{2.5*cos(\t r)});
\path[font=\small,shift={(10,0)}]
(-1.95,-0.2) node[left]{$O^1$}
(-1.25,0)node[right]{$O^2$}
(-2.9,0.8) node[below] {$\gamma^1$}
(-1.3,0.35)[left] node{$\gamma^2$}
(-.15,-1.15)[left] node{$\gamma^3$}
(.8,1.3) node[below] {$\gamma^4$}
(0.5,-2.6)[left] node{$P^2$}
(2.43,1.1)[left] node{$P^1$};
\end{tikzpicture}
\end{center}
\begin{caption}{An island--shaped network.\label{island}}
\end{caption}
\end{figure}
If we consider the sequence of rescaled networks $\widetilde{{\mathbb{H}}}^1_{P^1,\tt_{j}}$ 
obtained via Huisken's dynamical procedure applied to ${\mathbb{H}}^1_t$, as in Proposition~\ref{resclimit-general}, centered in $P^1$, it converges in $C^{1,\alpha}\loc\cap W^{2,2}\loc$, for any $\alpha \in (0,1/2)$
to a (not empty) limit degenerate regular shrinker $\widetilde{{\mathbb{H}}}_\infty$.
We analyze the possible $\widetilde{{\mathbb{H}}}_\infty$ without using the multiplicity--one conjecture {\bf{M1}}, to avoid a ``circular argument''. Moreover, we consider among all the possible blow--up limits $\widetilde{{\mathbb{H}}}_\infty$, one with the maximum number of $3$--points (which can only be $0$, $2$ or $4$).

We first consider the case when $\widetilde{{\mathbb{H}}}_\infty$ (hence, also the underlying graph) is a tree, then it is a symmetric family of halflines from the origin, by Lemma~\ref{lemmatree}.\\
If $\widetilde{{\mathbb{H}}}_\infty$ has no $3$--points, then it is a line through the origin, which means that in the rescaling procedure all the $3$--points go to infinity, hence it must be that the curves $\gamma^i$ of $\SS_t$ not going to infinity, in the sequence of rescalings, satisfy
\begin{equation*}
\lim_{j\to\infty}\frac{L^i(t_j)}{\sqrt{T-t_j}}=+\infty\,,
\end{equation*}
then, repeating the argument of Proposition~\ref{resclimit} (leading to Proposition~\ref{resclimit2}), such a line must have multiplicity one, being composed of the ``reflection'' of two halflines with unit multiplicity.\\
If $\widetilde{{\mathbb{H}}}_\infty$ contains only two $3$--points (hidden in its core at the origin), recalling the argument in the proof of Lemma~\ref{lemmatree}, it is given by four halflines forming angles of $120/60$ degrees.\\
In both these two cases the curvature of the non--rescaled networks ${{\mathbb{H}}}_t$ (hence, of $\SS_t$) 
is locally uniformly bounded around $P^1$ (by White's regularity theorem in~\cite{white1} and Proposition~\ref{cross}, which are both independent of {\bf{M1}}), then (in the second case, by arguing as in Lemma~\ref{boh}) the presence of another $3$--point of $\SS_t$ in a space--time neighborhood of $(P^1,T)$ is ``forbidden'', clearly contradicting the hypotheses.\\
The remaining case of four $3$--points in $\widetilde{{\mathbb{H}}}_\infty$, is when the (symmetric) core of $\widetilde{{\mathbb{H}}}_\infty$ is given by three degenerate curves (and four $3$--points) at the origin. In this case it is straightforward to see that $\widetilde{\SS}_\infty$ contains a straight line through the origin, which is not possible since $\widetilde{\SS}_\infty$ must be contained in an angle with opening less than $\pi$, by the strict convexity of $\Omega$, as it is shown in Proposition~\ref{omegaok2}.

If instead $\widetilde{{\mathbb{H}}}_\infty$ contains a loop (actually, two symmetric ones coming from a collapsing loop in $\SS_t$, as $t\to T$ and its ``reflection''), pushing a little the analysis in Section~\ref{sshnet} (see also the Appendix), it could only have the structure of a Brakke spoon (see Figure~\ref{brakspfig}) or of a shrinking lens/fish (see Figure~\ref{fishfig}). Then, it would contains the origin of $\R^2$ in its inside, which is clearly not possible in our situation of blow--up around an end--point of the network $\SS_t$.
\end{proof}

\begin{rem} As before, we remark that the strictly convexity hypothesis on $\Omega$
can actually be weakened by asking that $\Omega$ is convex and that there does not exist three
aligned end--points of the initial network $\mathbb{S}_0$ on $\partial\Omega$.
\end{rem}

\begin{proof}[Proof of Theorem~\ref{dlteo}]
If $\SS_t$ is the evolution of a network with only one triple junction, any of the evolving networks ${\mathbb{H}}^i_t$ has exactly two 3--points. Let $t\in[0,T)$ a time such that $0<\Pi(t)<1/4$ and $\Pi$ and all embeddedness measures $E^i$, associated to the networks ${\mathbb{H}}^i_t$, are differentiable at $t$ (this clearly holds for almost every time).\\
Let $E^i(t)=\Pi(t)<1/4$ and $E^i(t)$ is realized by a pair of points $p$ and $q$ in ${\mathbb{H}}^i_t$, we separate the analysis in the following cases:
\begin{itemize}
\item If the points $p$ and $q$ of the minimizing pair are both end--points of ${\mathbb{H}}^i_t$, by construction
$\vert p-q\vert\geqslant \varepsilon>0$.
Moreover, the area enclosed in the Jordan curve formed by the segment $\overline{pq}$ and 
by the geodesic curve $\Gamma_{p,q}$ can be uniformly bounded by above by a constant $\overline{C}>0$, for instance, 
the area of a ball containing all the networks ${\mathbb{H}}^i_t$.
Since $\varepsilon>0$ and $\overline{C}$ depend only on $\Omega$ and on the structure of the initial network $\SS_0$
(more precisely on the position of the end--points on the boundary of $\Omega$, that stay fixed during the evolution
and that do not coincide), the ratio $\frac{\vert p-q\vert^2}{\psi(A_{p,q})}$ (or $\frac{\vert p-q\vert^2}{A_{p,q}}$, if $p, q$ do not belong to a loop) is greater or equal than some constant $C_\eps=\frac{\varepsilon^2}{\overline{C}}>0$ uniformly, hence the same holds for $\Pi(t)$.
\item If one point is internal and the other is an end--point of ${\mathbb{H}}^i_t$, we consider the following two situations.
If one of the two point $p$ and $q$ is in $\mathbb{S}_t\subseteq{\mathbb{H}}^i_t$ and the other is
in the reflected network $\mathbb{S}^{R_i}_t$, then, we obtain, by construction, a uniform bound from below on $\Pi(t)$ as in the case in which 
$p$ and $q$ are both boundary points of ${\mathbb{H}}^i_t$.\\
Otherwise, if $p$ and $q$ are both in $\SS_t$ and one of them coincides with $P^j$ with $j\neq i$, either the other point coincides with $P^i$ and we have again a uniform bound from below on $\Pi(t)$, as before, or both $p$ and $q$ are points of ${\mathbb{H}}^j_t$ both not coinciding with its end--points and $E^j(t)=E^i(t)=\Pi(t)<1/4$, so we can apply the argument at the next point.

\item If $p$ and $q$ are both ``inside'' ${\mathbb{H}}^i_t$, by Hamilton's trick (see~\cite{hamilton2} or~\cite[Lemma~2.1.3]{Manlib}), we have $\frac{d\Pi(t)}{dt}=\frac{dE^i(t)}{dt}$ and, by Proposition~\ref{lemet2}, $\frac{dE^i(t)}{dt}>0$, hence $\frac{d\Pi(t)}{dt}>0$.
\end{itemize}
All this discussion implies that at almost every point $t\in[0,T)$ such that $\Pi(t)$ is smaller than some uniform constant depending only on $\Omega$ and on the structure of the initial network $\SS_0$, then $\frac{d\Pi(t)}{dt}>0$, which clearly proves the theorem in the case a network with a single triple junction (see also~\cite[Section~4]{mannovtor}).

Let now $\SS_t$ be a flow of regular networks with {\em two} triple junctions. If there are no end--points, the conclusion follows immediately from Proposition~\ref{lemet2}. Hence, we assume that $\SS_t$ has two or four end--points (in the first case there is a loop, and in the second $\SS_t$ is a tree), which are the only possibilities.\\
The analysis is the same as above, with only a delicate point to be addressed, that is, in the last case, when the two points $p$ and $q$ of the minimizing pair are ``inside'' ${\mathbb{H}}^i_t$ and we apply Proposition~\ref{lemet2}. Indeed, since ${\mathbb{H}}^i_t$ has {\em four} 3--points it can happen that the geodesic curve $\Gamma_{p,q}$ contains more than two 3--points, hence this case requires special treatment. Notice that if the points $p$ and $q$ are both ``inside'' $\mathbb{S}_t\subseteq{\mathbb{H}}^i_t$, then Proposition~\ref{lemet2} applies and we are done. We then assume that $p\in\mathbb{S}_t$, $q\in\mathbb{S}^{R_i}_t$,
and $\Gamma_{p,q}$ contains more than two triple junctions.\\
We want to show that there exists a uniform positive constant $\varepsilon$
such that $\vert p-q\vert\geqslant\varepsilon>0$, which implies a uniform positive estimate from below on $E^i(t)$, as above. This will conclude the proof.\\
Assume by contradiction that such a bound is not possible, then, for a sequence of times $t_j\to T$, the Euclidean distance between the two points $p_j$ and $q_j$ of the associated minimizing pair of $\Phi_{t_j}$ goes to zero, as $j\to\infty$ and this can happen only if $p_i,q_i\to P^i$. It follows, by the maximum principle that the two 3--points $O^1(t)$ and $O^2(t)$ converge to $P^i$ on some sequence of times $t_k\to T$ (possibly different from $t_j$), which is forbidden by Lemma~\ref{trinoncoll} and we are done.
\end{proof}

\begin{rem}\label{dl1O}
Notice, by inspecting the previous proof, that in the case that $\SS_t$ has a single 3--point, the strict convexity of $\Omega$ is not necessary, convexity is sufficient.
\end{rem}

\subsection{Consequences for the multiplicity--one conjecture}\label{consequences}
The quantity $E(t)$ considered in the previous section is clearly, by
definition, dilation and translation invariant, moreover it is
continuous under $C^1\loc$--convergence of networks. Hence, if
$E(t)\geqslant C>0$ for every $t\in[0,T)$, the same holds for every
$C^1\loc$--limit of rescalings of networks of the flow $\SS_t$. This
clearly implies the strong multiplicity--one conjecture {\bf{SM1}}.

\begin{cor}\label{Elemma1}
If $\Omega$ is strictly convex and the initial network $\SS_0$ has at most two triple junctions,
then the strong multiplicity--one conjecture {\bf{SM1}} is true for the flow $\SS_t$.
\end{cor}

A by--product of the proofs of Proposition~\ref{lemet2} and Theorem~\ref{dlteo} is actually that also the function $\Pi(t)$ is positively uniformly bounded from below during the flow.
 
\begin{cor}\label{Hlemma1}
If $\Omega$ is strictly convex and the initial network $\SS_0$ has at most two triple junctions,
then the strong multiplicity--one conjecture {\bf{SM1}} is true for all the ``symmetrized'' flows ${\mathbb{H}}^i_t$.
\end{cor}

\begin{rem} Actually, in general, if we are able to show the (strong) multiplicity--one conjecture for a curvature flow $\SS_t$ in a strictly convex open set $\Omega$, then, by construction and Proposition~\ref{omegaok2}, it also holds for all the ``symmetrized'' flows ${\mathbb{H}}^i_t$. This remark is in order since in the analysis of the flow $\SS_t$ in the previous sections, we used the ``reflection'' argument at the end--points of the network $\SS_t$, then we argued applying {\bf{M1}} to the resulting networks ${\mathbb{H}}^i_t$ (to be precise, in Section~\ref{novan} and in the proofs of Proposition~\ref{bdcurvcollapse} and of Proposition~\ref{prop999}).
\end{rem}

Another situation that can be analyzed by means of the ideas of this section is the following.

\begin{prop}
If during the curvature flow of a network $\SS_t$ the triple junctions stay uniformly far from each other and from the end--points, then {\bf{SM1}} is true for the flows $\SS_t$ and all ${\mathbb{H}}^i_t$. 
As a consequence, the evolution of $\SS_t$ does not develop singularities.
\end{prop}

\begin{proof} We divide all the pairs of curves of the evolving network
 $\SS_t$ in two families, depending on the curve of a pair have a
 common 3--point or not. In the second case, by means of maximum
 principle and the assumption on the 3--points, there is a uniform
 constant $C>0$ such that any couple of points, one on each curve of
 such pair, have distance bounded below by $C$. Then, if the pair of
 points of $\SS_t$ realizing the quantity $E(t)$ stay on such curves
 it follows $E(t)\geqslant C'>0$ for some uniform constant $C'$. In case
 $E(t)< C'$, it follows that such a pair of points either stay on the
 same curve or on two curves with a common 3--point. Hence, {\em the}
 ``geodesic'' curve $\Gamma_{p,q}$ contains at most one 3--point,
 since otherwise the distance between $p$ and $q$ would be at least $C$,
 contradicting the fact that $E(t)< C'$.
 This implies that $\frac{dE(t)}{dt}>0$ by
 Proposition~\ref{lemet1}. Then, the strong multiplicity--one conjecture follows for $\SS_t$ and all the ``symmetrized'' flows ${\mathbb{H}}^i_t$, by the same argument in the proof of Theorem~\ref{dlteo}, taking into account the hypothesis that the triple junctions stay uniformly far also from the end--points.\\
It follows that the only possible singularities of the flow are given by the collapse of a curve of the network, but this is excluded by the assumption.
\end{proof}

\section{The flow of networks with at most two triple junctions}\label{globsec}

In what follows we present, up to the best of our knowledge, the description of the evolution of the networks with at most two triple junctions.
For simplicity, we let them evolve in a strictly convex, open and smooth subset $\Omega\subseteq\R^2$. These are not only simple examples of a complete analysis of the flow, but they are interesting since most of the 
relevant phenomena of the motion by curvature of networks 
are already present. 

All the results are based on the content of the previous sections.
We underline that in 
the current situations the strong multiplicity--one conjecture {\bf{SM1}} holds (see Section~\ref{consequences}), hence it is not necessary to assume it. 
We require instead the {\em uniqueness of blow--up assumption} {\bf{U}}, stated in Problem~\ref{ooo12}, to hold, which is still conjectural, even if some positive partial results were recently obtained in~\cite{PlPo22A}. 

We recall that if the maximal time of smooth existence is finite, either a curve is vanishing with bounded curvature, or there exists at least a point $x_0\in\Omega$ where the curvature is not bounded, that is, at least a region of the network collapses at such point and we have there a blow--up limit network which cannot have zero curvature. 

Since the multiplicity--one conjecture holds for these networks, when a region collapses, also the loop that encloses the region must collapse, with its length going to zero.

\subsection{Networks with only one triple junction}\label{globsec1}
If we consider the possible (topological) structures of regular networks with only one triple junction, we see that there are only two cases: the {\em triod} and the {\em spoon}--shaped network. The motion of a triod can be regarded as the simplest example of 
the evolution by curvature of a tree--like 
configuration of an ``essentially'' singular one--dimensional set,
the motion of a spoon is the simplest one with a loop.
\begin{figure}[H]
\begin{center}
\begin{tikzpicture}[scale=0.85]
\draw[color=black]
 (-3.73,0) 
to[out=50,in=180,looseness=1] (-2,0) 
to[out=60,in=180,looseness=1.5] (-0.45,1.55) 
(-2,0)
to[out=-60,in=180,looseness=0.9] (-0.75,-1.75);
\draw[color=black,scale=1,domain=-3.141: 3.141,
smooth,variable=\t,shift={(-1.72,0)},rotate=0]plot({2.*sin(\t r)},
{2.*cos(\t r)}) ;
\path
 (-3.3,-1.75)node[left]{$\Omega$}
 (-3.7,0) node[left]{$P^1$}
 (-2.9,0.85) node[below] {$\gamma^1$}
 (-1.5,1) node[right] {$\gamma^3$}
 (-0.8,-1)[left] node{$\gamma^2$}
 (-2.2,0) node[below] {$O$}
 (-0.07,1.32)node[above]{$P^3$}
 (-0.75,-1.75) node[below] {$P^2$}; 
\draw[shift={(7,0)},color=black] 
(-3.73,0) 
to[out=50,in=180,looseness=1] (-2.3,0) 
to[out=60,in=150,looseness=1.5] (-1,1) 
(-2.3,0)
to[out=-60,in=-120,looseness=0.9] (-0.75,-0.75)
(-1,1)
to[out=-30,in=90,looseness=0.9] (-0.5,0)
to[out=-90,in=60,looseness=0.9] (-0.75,-0.75);
\draw[color=black,scale=1,domain=-3.141: 3.141,
smooth,variable=\t,shift={(5.28,0)},rotate=0]plot({2.*sin(\t r)},
{2.*cos(\t r)}) ;
\path[shift={(7,0)}]
 (-3.3,-1.75)node[left]{$\Omega$}
 (-1.3,-0.3)node[above]{$A$}
 (-3.7,0) node[left]{$P$}
 (-2.9,0.8) node[below] {$\gamma^2$}
 (-0.8,-1.3)[left] node{$\gamma^1$}
 (-2.5,0) node[below] {$O$}; 
\end{tikzpicture}
\end{center}
\begin{caption}{Networks with only one triple junction: triod and spoon network.\label{untripuntobis}} 
\end{caption}
\end{figure}
In what follows we present a complete description of the evolution of networks with these two shapes (from~\cite{mannovtor,MMN13,pluda}). We will see that in the case of the triod, we can exclude the presence of singularities till the lengths of the three curves stay positively bounded from below, while in the case of the spoon instead, a singularity develops.

As defined in Section~\ref{smtm}, fixed a smooth, open, strictly convex set $\Omega\subseteq\mathbb{R}^2$, 
a triod is a network (a tree) $\TT$ composed only of three regular, embedded $C^1$
curves $\gamma^i:[0,1]\to\overline{\Omega}$. 
These curves intersects each other only at a single 3--point $O$, that is, $\gamma^1(0)=\gamma^1(0)=\gamma^1(0)=O$ 
and have the other three end--points $P^1, P^2, P^3$ on the boundary of $\Omega$ with $\gamma^i(1)=P^i$, for $i\in\{1, 2, 3\}$. The triod is regular if the three concurring curves form angles of $120$ degrees.

A spoon $\Gamma=\gamma^1([0,1])\cup \gamma^2([0,1])$ is the union of two regular, 
embedded $C^1$ curves $\gamma^1, \gamma^2:[0,1]\to\overline{\Omega}$ which intersect 
each other only at a triple junction $O$, with angles of $120$ degrees,
that is, $\gamma^1(0)=\gamma^1(1)=\gamma^2(0)=O\in\Omega$ 
and $\gamma^2(1)=P\in \partial\Omega$. 
We call $\gamma^1$ the ``closed'' curve and $\gamma^2$ the ``open'' curve of the spoon 
and we denote with $A$ the area of the region enclosed by the loop. A spoon is regular if $\tau^1(0)+\tau^2(0)-\tau^1(1)=0$, which means that the three angles at $O$ are of $120$ degrees.

For simplicity, we will assume in the following that all the initial networks are smooth, hence Theorem~\ref{smoothexist} applies and gives a smooth curvature flow in a maximal time interval $[0,T)$. As we discussed in the previous sections, to start the flow 
if the curves of the initial network are only $C^2$ but the Herring condition is still satisfied, we need Theorem~\ref{c2shorttime}. 
If the initial network is not regular, we need to apply Theorem~\ref{evolnonreg} to have a curvature flow. Anyway, in all these cases, the flow is smooth for every positive time.
If the network is regular, thanks to Theorem~\ref{c2shorttime}, we have uniqueness (geometric uniqueness to be more precise, see Definition~\ref{uniqdef}). If these networks with only triple junctions are not regular but their curves are smooth, we still get geometric uniqueness (see Remark~\ref{uniqueness-nonreg}).

Collecting and specializing the results for a smooth initial network to the cases of a triod or of a spoon (Theorem~\ref{smoothexist}), we have the following proposition.

\begin{prop}
Let $\Omega\subseteq\R^2$ be a smooth, open, strictly convex set, then, for any smooth regular initial triod $\TT_0$ or any smooth regular initial spoon $\Gamma_0$ in $\Omega$, there exists a geometrically unique smooth (and special) curvature flow in a maximal time interval $[0,T)$.
\end{prop}

Before proceeding, we also recall that during the flow the evolving networks stay embedded and intersect the boundary of $\Omega$ only at the fixed end--points (transversally), see Section~\ref{geopropsub}.

\subsubsection{The triod}

Suppose that $T<+\infty$, then, by Proposition~\ref{regnocollapse}, the lengths of the three curves 
cannot be uniformly positively bounded from below. 
Hence, as $\Omega$ is strictly convex, Corollary~\ref{tree-bdcurb} and Theorem~\ref{prop999b}
imply that the curvature of $\TT_t$ is uniformly bounded and there must be a collapse of a curve to a fixed end--point on $\partial\Omega$, when $t\to T$, as depicted in the right side of Figure~\ref{Pcollapse} or Figure~\ref{boundarysing}.

Suppose instead that $T=+\infty$. Then, by Proposition~\ref{prolong}, 
for every sequence of times $t_i\to+\infty$, there exists a (not relabeled) 
subsequence such that the evolving triod $\mathbb{T}_{t_i}$
converge in $C^1$ to a possibly degenerate regular triod, embedded (by Theorem~\ref{dlteo}) and with
zero curvature, as $i\to\infty$, that is, a Steiner configuration connecting the three fixed points $P^i$ on $\partial\Omega$ (which possibly have a zero--length degenerate curve, for instance if the three end--points are the vertices of a triangle with an angle of $120$ degrees).
Moreover, as the Steiner configuration (which is length minimizing) connecting three points is unique (if it exists), for every subsequence of times, we have the same limit triod, hence, the {\em full} sequence of triods $\TTT_t$ converge to such limit, as $t\to+\infty$.

We notice that there is an obvious example where the length of 
one curve goes to zero in finite time: the case of an initial triod $\TT_0$
with the boundary points $P^i$ on $\partial\Omega$ such that 
one angle of the triangle with vertices $P^1, P^2, P^3$ is greater than $120$ degrees. 
In this case the Steiner triod does not exist, hence the maximal time of a smooth evolution must be finite.\\
Instead, if the angles of the triangle with vertices $P^1, P^2, P^3$ are all smaller than $120$ degrees and the initial triod $\TT_0$ is contained in the convex envelope of $P^1, P^2, P^3$, then no length can go to zero during the evolution, by Remark~\ref{angolostretto}, the maximal time of existence is $+\infty$ and the triods $\TT_t$ tend, as $t\to+\infty$, to the unique Steiner triod.

When the maximal time $T$ is finite and a curve collapses to an end--point
(see Figures~\ref{Pcollapse},~\ref{boundarysing} and the above discussion), 
it is not clear how to continue/restart the flow. Indeed, although the curvature is bounded, Theorem~\ref{evolnonreg} does not apply and we need some ``boundary'' extension (see the discussion in Section~\ref{resum}, after Figure~\ref{boundarysing}).

\subsubsection{The spoon}

In Section~\ref{geopropsub} we discussed the behavior of the area $A$ of a bounded region enclosed by a loop of an evolving regular network.
In the case of the spoon, the loop is composed of one curve and there is only one triple junction. Then, equation~\eqref{areaevolreg} gives $A'(t)=-{5\pi}/{3}$, which implies that the maximal time $T$ of existence of a smooth flow of a spoon is finite and $T\leqslant\frac{3A(0)}{5\pi}$, where $A(0)$ is the initial area enclosed in the loop (see Proposition~\ref{loop}).\\
As $t\to T$, the only possible limit regular shrinkers $\widetilde{\Gamma}_\infty$ arising from Huisken's rescaling procedure at a reachable point $x_0\in\overline{\Omega}$ are given by
\begin{itemize}
\item a halfline from the origin,
\item a straight line through the origin,
\item a standard triod,
\item a Brakke spoon (see Figure~\ref{brakspfig}). 
\end{itemize}
This follows by the simple topological structure of $\Gamma_t$ and the uniqueness (up to rotation) of the Brakke spoon among the shrinkers in its topological class (see Section~\ref{shrinkers}). We remind that all the possible blow--up limits are non--degenerate networks with multiplicity one, thanks to Corollary~\ref{Elemma1}.

We first notice that, if the curve $\gamma^1$ collapses, then the curvature clearly 
cannot be bounded. Moreover, by Proposition~\ref{prop999b}, it is not possible that both lengths of $\gamma^1$ and $\gamma^2$ go to zero, as $t\to T$.

Suppose that the length of the ``open'' curve $\gamma^2$ is uniformly positively bounded from below for all $t\in[0,T]$, then the curve $\gamma^1$ must collapse and the maximum of the curvature goes to $+\infty$ as $t\to T$ (indeed, $\lim_{t\to T}\int_{\mathbb{S}_t}k^2\,ds=+\infty$, by Proposition~\ref{loop}). Then, if $x_0=\lim_{t\to T}O(t)$, taking a blow--up limit $\widetilde{\Gamma}_\infty$ at $x_0\in\Omega$, we can only get a Brakke spoon, since in the other cases (a halfline is obviously excluded) the curvature would be locally bounded and the flow regular. Hence, as $t\to T$, the length of the closed curve $\gamma^1$ goes to zero and the area $A(t)$ enclosed in the loop goes to zero 
since (as {\bf{U}} holds) we have a limit network $\Gamma_T$, as $t\to T$, composed only by a $C^1$ curve $\gamma^2_T$ connecting $P$ with $x_0$ (and curvature going as $o(1/d_{x_0})$), as in Figure~\ref{f22}.
Moreover, from the evolution law $A(t)=A(0)-5\pi t/3$, we obtain that 
$T=\frac{3A(0)}{5\pi}$.

If instead the length of the curve $\gamma^2$ is not positively bounded from below then, as $t\to T$, by Proposition~\ref{prop999b} such curve collapses to the end--point $P$, the curvature stays bounded and the network $\Gamma_t$ is locally a tree around every point, uniformly in $t\in[0,T)$. Hence, the region enclosed by the curve $\gamma^1$ does not vanishes and the triple junction $O$ has collapsed onto the boundary point $P$, maintaining the $120$ degrees condition and with bounded curvature (see Proposition~\ref{cross}). The networks $\Gamma_t$ converge in $C^1$, as $t\to T$, to a limit network $\Gamma_T$, as in Figure~\ref{spoonPcollap}.

We actually do not have a natural way to restart the flow in the first situation. In the second one, a natural ``choice'' is to assume that the flow ends and the whole network vanishes for $t>T$.

We conclude this example with a couple of open questions.

\begin{oprob}[Special case of Problem~\ref{ooo12}]
Is the limit Brakke spoon obtained in the previous theorem (in the second situation) independent of the chosen sequence of times $\tt_k\to+\infty$?
That is, is the direction of its unbounded halfline unique?
\end{oprob}

\begin{oprob}
Having in mind the ``convexification'' result for simple closed curves
by Grayson~\cite{gray1} (see Remark~\ref{gremh}),
a natural question is: if we consider an initial spoon moving by curvature with the length of the non--closed curve uniformly positively 
bounded below during the evolution, does
the closed curve become eventually convex and then remain convex?
\end{oprob}

These two open problems are related: the uniqueness of the blow--up limit (which is a Brakke spoon, hence with a convex region) would imply that the region at some time becomes convex and then remains so, by the smooth convergence of the rescaled networks to the Brakke spoon (this follows from the argument of Lemma~8.6 in~\cite{Ilnevsch}, see the discussion just after the proof of Lemma~\ref{boh}).

\subsection{Networks with two triple junctions}\label{globsec2}
We consider now regular networks with exactly two triple junctions and we focus on their topological classification. We parametrize the curves composing the network by $\gamma^i:[0,1]\to\mathbb{R}^2$.
At each $3$--point either three different not closed curves concur (for instance $O^1=\gamma^1(0)=\gamma^2(0)=\gamma^3(0)$)
or two curves, one of which closed (that is $O^1=\gamma^1(0)=\gamma^1(1)=\gamma^2(0)$).
As we do not consider here open networks (with branches
that go to infinity asymptotic to halflines, see Definition~\ref{Cinftyopen}),
if a curve is not closed (hence $\gamma^1(0)\neq \gamma^1(1)$), there are only two possibilities for its end--point not concurring in $O^1$: either it is an end--point on the boundary of $\Omega$, or it belongs to the other triple junction $O^2$.
If we repeat the above reasoning for every end--point, we obtain all the cases shown in Figure~\ref{clas}.

When we say that a network has a loop $\ell$, we mean that there is a Jordan curve in $\mathbb{S}$ that encloses
an area $A$. For networks with two triple junctions, there are two cases (see Figure~\ref{clas}):
\begin{itemize}
\item the loop $\ell$ is composed of a single curve $\gamma:[0,1]\to\mathbb{R}^2$,
$\gamma(0)=\gamma(1)$ forming an angle of $120$ degrees.
The length $L$ of $\ell$ coincides with the length of $\gamma$.
\item the loop $\ell$ is composed of two curves $\gamma^1,\gamma^2:[0,1]\to\mathbb{R}^2$,
that meet each other at their end--points and at both junctions
there is an angle of $120$ degrees.
The length $L$ of $\ell$ is the sum of the lengths of the two curves.
\end{itemize}

\begin{figure}[H]
\begin{center}
\begin{tikzpicture}[scale=1.4]
\draw[color=black,scale=0.45,shift={(8,-4.5)}] 
(-1.73,-1.8) 
to[out= 180,in=180, looseness=1] (-2.8,0) 
to[out= 60,in=150, looseness=1.5] (-1.5,1) 
(-2.8,0)
to[out=-60,in=180, looseness=0.9] (-1.25,-0.75)
(-1.5,1)
to[out= -30,in=90, looseness=0.9] (-1,0)
to[out= -90,in=60, looseness=0.9] (-1.25,-0.75)
to[out= -60,in=0, looseness=0.9](-1.73,-1.8);
\draw[color=black!50!white,scale=0.45,domain=-3.15: 3.15,
smooth,variable=\t,shift={(6.28,-4.5)},rotate=0]plot({2.*sin(\t r)},
{2.*cos(\t r)}); 
\draw[color=black,scale=0.33,shift={(26.8,-5.5)}] 
(-2,0) 
to[out= 170,in=40, looseness=1] (-2.9,1.2) 
to[out= -140,in=90, looseness=1] (-3.2,0)
(-2,0)
to[out= -70,in=0, looseness=1] (-3,-0.9) 
to[out= -180,in=-90, looseness=1] (-3.2,0)
(-2,0) 
to[out= 50,in=180, looseness=1] (-1.3,0) 
to[out= 60,in=150, looseness=1.5] (-0.75,1) 
(-1.3,0)
to[out= -60,in=-120, looseness=0.9] (-0.5,-0.75)
(-0.75,1)
to[out= -30,in=90, looseness=0.9] (-0.25,0)
to[out= -90,in=60, looseness=0.9] (-0.5,-0.75);
\draw[color=black!50!white,scale=0.33,domain=-3.15: 3.15,
smooth,variable=\t,shift={(25,-5.5)},rotate=0]plot({2.*sin(\t r)},
{2.*cos(\t r)}) ; 
\draw[color=black,scale=0.33,shift={(31.1,-7)}] 
(-3,0) 
to[out= 170,in=140, looseness=1] (-2.1,1.4) 
to[out= -40,in=90, looseness=1] (0,0)
(-3,0) 
to[out= -70,in=-180, looseness=1] (-1,-1.3) 
to[out= 0,in=-90, looseness=1] (0,0)
(-3,0) 
to[out= 50,in=180, looseness=1] (-2.3,0) 
to[out= 60,in=150, looseness=1.5] (-1.75,1) 
(-2.3,0)
to[out= -60,in=-120, looseness=0.9] (-1.5,-0.75)
(-1.75,1)
to[out= -30,in=90, looseness=0.9] (-1.25,0)
to[out= -90,in=60, looseness=0.9] (-1.5,-0.75);
\draw[color=black!50!white,scale=0.33,domain=-3.15: 3.15,
smooth,variable=\t,shift={(29.58,-7)},rotate=0]plot({2.*sin(\t r)},
{2.*cos(\t r)}) ; 
\draw[color=black,scale=0.45, shift={(8,-9.72)}] 
(-3.73,0)
to[out= 50,in=180, looseness=1] (-2.8,0) 
to[out= 60,in=150, looseness=1.5] (-1.5,1) 
(-2.8,0)
to[out=-60,in=180, looseness=0.9] (-1.25,-0.75)
(-1.5,1)
to[out= -30,in=90, looseness=0.9] (-1,0)
to[out= -90,in=60, looseness=0.9] (-1.25,-0.75)
to[out= -60,in=150, looseness=0.9](-0.3,-1.3);
\draw[color=black!50!white,scale=0.45,domain=-3.15: 3.15,
smooth,variable=\t,shift={(6.28,-9.72)},rotate=0]plot({2.*sin(\t r)},
{2.*cos(\t r)}) ; 
\draw[color=black,scale=0.45, shift={(8,-14.94)}] 
 (-3.73,0) 
to[out= 50,in=180, looseness=1] (-2.3,0.7) 
to[out= 60,in=180, looseness=1.5] (-0.45,1.55) 
(-2.3,0.7)
to[out= -60,in=130, looseness=0.9] (-1,-0.3)
to[out= 10,in=100, looseness=0.9](0.1,-0.8)
(-1,-0.3)
to[out=-110,in=50, looseness=0.9](-2.7,-1.7);
\draw[color=black!50!white,scale=0.45,domain=-3.15: 3.15,
smooth,variable=\t,shift={(6.28,-14.94)},rotate=0]plot({2.*sin(\t r)},
{2.*cos(\t r)}) ; 
\draw[color=black,scale=0.45, shift={(14.65,-9.72)}] 
(-2,0)
to[out= 170,in=40, looseness=1] (-2.9,1.2) 
to[out= -140,in=90, looseness=1] (-3.2,0)
(-2,0) 
to[out= -70,in=0, looseness=1] (-3,-0.9) 
to[out= -180,in=-90, looseness=1] (-3.2,0)
(-2,0) 
to[out= 50,in=180, looseness=1] (-1.3,0) 
to[out= 60,in=150, looseness=1.5] (-0.75,1) 
(-1.3,0)
to[out= -60,in=-120, looseness=0.9] (-0.5,-0.75)
(-0.75,1)
to[out= -30,in=90, looseness=0.9] (0,1)
(0,-1)
to[out= -90,in=60, looseness=0.9] (-0.5,-0.75);
\draw[color=black!50!white,scale=0.45,domain=-3.15: 3.15,
smooth,variable=\t,shift={(12.93,-9.72)},rotate=0]plot({2.*sin(\t r)},
{2.*cos(\t r)}) ; 
\draw
(0,0)--(0,-8.05)
(10.5,0)--(10.5,-8.05)
(7.5,0)--(7.5,-8.05)
(4.5,0)--(4.5,-8.05)
(1.5,0)--(1.5,-8.05)
(0,0)--(10.5,0)
(0,-1)--(10.5,-1)
(0,-3.35)--(10.5,-3.35)
(0,-5.7)--(10.5,-5.7)
(0,-8.05)--(10.5,-8.05);
\path[color=black, font=\small]
(4.35,-3.11) node[left]{Theta}
(10.35,-3.16) node[left]{Eyeglasses}
(4.35,-5.46) node[left]{Lens}
(7.35,-5.46) node[left]{Island}
(4.35,-7.81) node[left]{Tree};
\path[font=\small]
(0.75,-2.2) node[above,rotate=90]{$0$ end--points}
(1.1,-2.2) node[above,rotate=90]{on $\partial\Omega$}
(0.75,-4.55) node[above,rotate=90]{$2$ end--points}
(1.1,-4.55) node[above,rotate=90]{on $\partial\Omega$}
(0.75,-6.9) node[above,rotate=90]{$4$ end--points}
(1.1,-6.9) node[above,rotate=90]{on $\partial\Omega$}
(2.1,-0.5) node[right]{$0$ closed curves}
(5.1,-0.5)node[right]{$1$ closed curve}
(8.1,-0.5) node[right] {$2$ closed curves};
\end{tikzpicture}
\end{center}
\begin{caption}
{Networks with two triple junctions.\label{clas}}
\end{caption}
\end{figure}

\begin{prop}
Let $\Omega\subseteq\R^2$ be a smooth, open, strictly convex set,
then, for any smooth regular initial network in the above family, there exists a geometrically unique smooth (and special) curvature flow in a maximal time interval $[0,T)$. During the flow, the evolving networks stay 
embedded and intersect the boundary of $\Omega$ 
only at the fixed end--points (transversally).
\end{prop}

We first analyze the possible blow--up limits at a singular time of the evolution of networks with two triple junctions
of general topological type, then we discuss in detail all the possible topologies, case by case.

It is crucial that all the possible blow--up limits $\widetilde\SS_\infty$, arising from Huisken's rescaling procedure, 
are embedded networks with multiplicity--one, by Corollary~\ref{Elemma1} in Section~\ref{dsuL}.

\begin{prop}\label{possiblelimit}
If the rescaling point $x_0$ belongs to $\Omega$, then the blow--up limit network $\widetilde\SS_\infty$ (if not empty) is one of the following (see Section~\ref{sshnet}):
\begin{itemize}
\item a straight line through the origin;
\item a standard triod centered at the origin;
\item a standard cross;
\item a Brakke spoon;
\item a shrinking lens;
\item a shrinking fish.
\end{itemize}
If the rescaling point $x_0$ is a fixed end--point of the evolving network
(on the boundary of $\Omega$), then the blow--up limit network 
$\widetilde\SS_\infty$ (if not empty) is one of the following:
\begin{itemize}
\item a halfline from the origin;
\item two halflines from the origin forming an angle of $120$ degrees (``half'' of a standard cross).
\end{itemize}
\end{prop}

\begin{proof}
The limit (possibly degenerate) network $\widetilde\SS_\infty$ has to satisfy the shrinkers equation $k_\infty+x^\perp=0$
for all $x\in\widetilde\SS_\infty$ (see the proof of Proposition~\ref{resclimit-general}).

If we assume that $\widetilde\SS_\infty$ is a degenerate regular shrinkers, it must be a standard cross, if $x_0\in\Omega$, or two halflines from the origin forming an angle of $120$ degrees, when $x_0\in\partial\Omega$ (``half'' of a standard cross). Then, its core is composed of a single curve 
(connecting the two triple junctions or a triple junction with an end--point, by Lemma~\ref{trinoncoll}) ``collapsed'' in the limit.

If $\widetilde\SS_\infty$ is not degenerate and the curvature $\widetilde{k}_\infty$ is constantly zero, the network is composed only of halflines or straight lines. Then, the possible regular shrinkers are either a straight line through the origin or a standard triod, if $x_0\in\Omega$, 
or a halfline, if $x_0\in\partial\Omega$.\\
If instead the curvature is not constantly zero and the network $\widetilde\SS_\infty$ is not degenerate, by the classification of regular shrinkers with two triple junctions, we can only have either the Brakke spoon, the shrinking lens, or the shrinking fish.
In all these three cases, the center of the homothety is inside the enclosed region, hence $x_0$ cannot be an end--point on the boundary of $\Omega$.
\end{proof}

\begin{prop}\label{loop2trip}
Let $\mathbb{S}_0$ be a network with two triple junctions and with a loop $\ell$ of length $L$, enclosing a region of area $A$ and let $\mathbb{S}_t$ be a smooth evolution by curvature of such network in the maximal time interval $[0,T)$. Then, $T$ is finite and if $\lim_{t\to T}L(t)=0$, there holds $\lim_{t\to T}\int_{\mathbb{S}_t}k^2\,ds=+\infty$.
\end{prop}
\begin{proof}
If a loop is present, by the above classification of the possible topological structures of the networks with two triple junctions, it must be composed of $m$ curves, with $m<6$, hence, Proposition~\ref{loop} applies.
\end{proof}

\begin{thm}\label{main2trip}
Let $\Omega\subseteq\mathbb{R}^2$ be a smooth, strictly convex, open set. 
Let $\mathbb{S}_0$ be a compact initial network with two triple junctions and with possibly
fixed end--points on $\partial\Omega$ and
let $\mathbb{S}_t$ be the smooth evolution by curvature
of $\mathbb{S}_0$ in a maximal time interval $[0,T)$.\\
If the network $\mathbb{S}_0$ has at least one loop, then the maximal time of existence $T$ is finite and one of the following situations occurs:
\begin{enumerate}
\item the limit of the length of a curve that connects the two $3$--points goes to zero, as $t\to T$, 
and the curvature remains bounded;
\item the limit of the length of a curve that connects the $3$--point with an end--point goes to zero, as $t\to T$ and the curvature remains bounded;
\item a region enclosed by a loop collapses with the length of the loop going to zero (since {\bf{SM1}} holds), as $t\to T$ and $\displaystyle{\lim_{t\to T}}\int_{\mathbb{S}_t}k^2\,ds=+\infty$.
\end{enumerate}
If the network is a tree and $T$ is finite, the curvature is uniformly bounded and only one of the first two situations listed above can happen. If instead $T=+\infty$, for every sequence of times $t_i\to+\infty$, there exists a subsequence (not relabeled) such that the evolving networks $\SS_{t_i}$ converge in $C^{1,\alpha}\cap W^{2,2}$, for every $\alpha\in(0,1/2)$,
to a possibly degenerate, regular network with zero curvature (hence, ``stationary'' for the length functional), as $i\to\infty$.
\end{thm}
\begin{proof}
If a loop is present, by Proposition~\ref{loop}, the maximal time of smooth existence $T>0$ is finite. If such time is smaller than the ``natural'' time at which the loop shrinks (depending on the number of curves composing the loop, as in Proposition~\ref{loop}), the network is locally a tree, uniformly for $t\in[0,T)$. Hence, every blow--up limit at any point $x_0\in\overline{\Omega}$ cannot contain loops, then Proposition~\ref{possiblelimit} shows that it must have zero curvature, thus, by Proposition~\ref{prop999} the curvature of $\SS_t$ is uniformly bounded along the flow and converges, as $t\to T$, to a degenerate regular network $\SS_T$ with vertices that are either a regular triple junction, an end--point, or 
\begin{itemize}
\item a $4$--point where the four concurring curves have opposite unit tangents in pairs and form 
angles of $120/60$ degrees between them
(collapse of the curve joining the two triple junctions of $\SS_t$);
\item a $2$--point at an end--point of the network $\SS_t$ where the
 two concurring curves form an angle of $120$
degrees among them (collapse of the curve joining a triple junction to such end--point of $\SS_t$).
\end{itemize}
The same conclusion clearly holds if $\SS_0$ is a tree and $T$ is finite.

If instead the time $T$ coincides with the vanishing time of a loop of the network, by Proposition~\ref{loop}, the curvature is unbounded and there must exist a reachable point for the flow $x_0\in\Omega$ and a sequence of times $t_j\to T$ such that, the associate sequence of rescaled networks $\widetilde{\SS}_{x_0,\tt_{j}}$, as in Proposition~\ref{resclimit-general}, 
converges in $C^{1,\alpha}\loc\cap W^{2,2}\loc$, for any $\alpha \in (0,1/2)$,
to a limit degenerate regular shrinker $\widetilde\SS_\infty$ that is either a Brakke spoon, a shrinking lens, or a shrinking fish.

If $T=+\infty$, hence $\SS_0$ must be a tree, then $\mathbb{S}_t$ converges, as $t\to+\infty$, to a (possibly degenerate) regular network with zero curvature (a stationary point for the length functional), thanks to Proposition~\ref{prolong}
\end{proof}

To now proceed with a more detailed analysis of the behavior of the flow of these networks, we consider each topological type separately.

\subsubsection{The theta}

We call $A_1$ the area enclosed by the curves $\gamma^1$ and $\gamma^2$ and $A_2$ the area enclosed by $\gamma^2$ and
$\gamma^3$, as in the following figure.
\begin{figure}[H]
\begin{center}
\begin{tikzpicture}[scale=0.9]
\draw[color=black,shift={(5,0)}] 
(-1.73,-1.8) 
to[out= 180,in=180, looseness=1] (-2.8,0) 
to[out= 60,in=150, looseness=1.5] (-1.5,1) 
(-2.8,0)
to[out=-60,in=180, looseness=0.9] (-1.25,-0.75)
(-1.5,1)
to[out= -30,in=90, looseness=0.9] (-1,0)
to[out= -90,in=60, looseness=0.9] (-1.25,-0.75)
to[out= -60,in=0, looseness=0.9](-1.73,-1.8);
\draw[color=black!50!white,scale=1,domain=-3.15: 3.15,
smooth,variable=\t,shift={(3.3,-0.3)},rotate=0]plot({3.25*sin(\t r)},
{2.5*cos(\t r)}) ; 
\path[shift={(5,0)}] 
(-5.5,-0.75)node[right]{$\Omega$}
(-1.25,-0.75)node[right]{$O^2$}
 (-1.5,-0.3)[left] node{$\gamma^2$}
 (-1.5,0.5)[left] node{$A_1$}
 (-1.5,-1.1)[left] node{$A_2$}
 (-0.8,1.1)[left] node{$\gamma^1$}
 (-0.6,-1.45)[left] node{$\gamma^3$}
 (-3,0.65) node[below] {$O^1$}; 
 \end{tikzpicture}
\end{center}
\begin{caption}{Theta.\label{theta}} 
\end{caption}
\end{figure}
Let $x_0\in\Omega$ be a reachable point of the flow, from Proposition~\ref{resclimit-general}, we know that the sequence of rescaled networks $\widetilde{\SS}_{x_0,\tt_{j}}$
converges in $C^{1,\alpha}\loc\cap W^{2,2}\loc$, for any $\alpha \in (0,1/2)$, to a blow--up limit shrinker $\widetilde\SS_\infty$. By Proposition~\ref{possiblelimit}, the possible $\widetilde\SS_\infty$ are: 
\begin{itemize}
\item a straight line through the origin;
\item a standard triod;
\item a standard cross;
\item a shrinking lens;
\item a shrinking fish,
\end{itemize}
where we excluded the Brakke spoon for topological reasons.\\
We know from Proposition~\ref{loop} that the maximal time $T$ of existence of a smooth flow is finite and bounded by $\frac{3}{4\pi}\min\{A_1(0),A_2(0)\}$. Indeed, from equation~\eqref{areaevolreg}, we know that the areas enclosed in the two loops are linearly decreasing in time, precisely, $A_1'(t)=A_1'(t)=-{4\pi}/{3}$.

\medskip

If $T<\frac{3}{4\pi}\min\{A_1(0) , A_2(0)\}$, then the evolving network is locally a tree for all times, hence the analysis of Sections~\ref{van0} and~\ref{van} (in particular, Theorem~\ref{tree-bdcurb2}) applies and the curvature stays bounded while the length of only one curve is going to zero, as $t\to T$, forming a regular $4$--point, where the two triple junctions converge.

\medskip

Suppose that $T=\frac{3}{4\pi}\min\{A_1(0) , A_2(0)\}$.
Suppose by contradiction that
$A_1(0)=A_2(0)$.
Clearly the two regions should collapse both at $T$. 
Taking a blow--up limit $\widetilde{\mathbb{S}}_\infty$ at a hypothetical vanishing point $x_0\in\Omega$, such limit must contain two contiguous regions with a common edge and with equal finite area. Indeed, 
every rescaled network of the sequence $\widetilde{\SS}_{x_0,\tt}$
contain two contiguous regions and the two loops cannot vanish in the limit (neither collapsing to a core because of the enclosed constant area), since at least one is present in the possible blow--up limit shrinker. Then, since there are no possible limit shrinkers with two bounded regions, by Proposition~\ref{possiblelimit}, this situation is not possible.\\
So, in the case $T=\frac{3}{4\pi}\min\{A_1(0) , A_2(0)\}$, the two areas $A_1(0)$ and $A_2(0)$ must be different. The curvature cannot stay bounded, hence there must exist a singular point $x_0\in\Omega$ where $\widetilde\SS_\infty$ is a non--straight shrinker, thus, a shrinking lens or a shrinking fish. The resulting possible limit network $\SS_T$, as $t\to T$, will then be given by a $C^1$ curve, ``closing'' at $x_0$, possibly forming an angle. As we supposed that the {\em uniqueness of blow--up assumption} {\bf{U}} in Problem~\ref{ooo12} holds, such angle is either the one between the two ``halfilnes'' of the shrinking fish, if this is the blow--up limit shrinker, or the curve is $C^1$ (no angle), if the blow--up limit shrinker is a shrinking lens (see Figure~\ref{angfig}).

\bigskip

In the first case, we ``pass through'' the (topological) Type-0 singularity by a standard transition, as described in Section~\ref{restart} (see Figure~\ref{figstandard}) and we actually conjecture that this can be done in a unique way, see Remark~\ref{unirem999}. After the transition, the network becomes eyeglasses--shaped (of ``type A'' or of ``type B'', depending on whether the collapsed curve was the central one, or one of the other two, respectively), as in Figure~\ref{figstatheta} or in the left side of Figure~\ref{switch}.\\
In the case $\SS_T$ is a $C^1$ closed curve with possibly an angle, by the results of Angenent in~\cite{angen3} (see also~\cite{eckhui2}), we can (uniquely) restart the evolution by means of the ``classical'' curve shortening flow, obtaining an evolving closed embedded curve, which becomes immediately smooth. After some time it becomes convex and then shrinks in finite time to a ``round'' point of $\Omega$, by the well--known works of Gage, Grayson and Hamilton~\cite{gage,gage0,gaha1,gray1}

\subsubsection{The eyeglasses}

We analyze the two different ``types'' of these networks, as in the following Figure~\ref{eyeglasses}. 

From equation~\eqref{areaevolreg} we know that the area enclosed by any loop is linearly decreasing in time. Hence, being present some regions, by Proposition~\ref{loop} it follows that the maximal time $T>0$ of existence of a smooth flow is finite and bounded by 
$\frac{3}{5\pi}\min\{A_1(0),A_2(0)\}$, where $A_1$ and $A_2$ are the areas of the regions respectively enclosed by the curves $\gamma^1$ and $\gamma^2$, in the ``type A'' case, by
$\frac{3A_1(0)}{5\pi}$, where $A_1$ is the area enclosed by the ``internal'' loop, in the ``type B'' case (in the case of collapse of a region also its boundary loop must vanish, hence the ``internal'' region is forced to collapse).
\begin{figure}[H]
\begin{center}
\begin{tikzpicture}[scale=0.9]
\draw[black]
(-2,0) 
to[out= 170,in=40, looseness=1] (-2.9,1.2) 
to[out= -140,in=90, looseness=1] (-3.2,0)
(-2,0)
to[out= -70,in=0, looseness=1] (-3,-0.9) 
to[out= -180,in=-90, looseness=1] (-3.2,0)
(-2,0) 
to[out= 50,in=180, looseness=1] (-1.3,0) 
to[out= 60,in=150, looseness=1.5] (-0.75,1) 
(-1.3,0)
to[out= -60,in=-120, looseness=0.9] (-0.5,-0.75)
(-0.75,1)
to[out= -30,in=90, looseness=0.9] (-0.25,0)
to[out= -90,in=60, looseness=0.9] (-0.5,-0.75);
\draw[color=black!50!white,scale=1,domain=-4.141: 4.141,
smooth,variable=\t,shift={(-1.72,0)},rotate=0]plot({2.3*sin(\t r)},
{2.3*cos(\t r)}) ;

\path
(-3,-2.1) node[left]{$\Omega$}
(-1.9,-0.2) node[left]{$O^1$}
(-1.3,0)node[right]{$O^2$}
 (-2.8,0.8) node[below] {$\gamma^1$}
 (-1.3,0.4)[left] node{$\gamma^3$}
 (0,-0.95)[left] node{$\gamma^2$};

\draw[color=black,shift={(7,0)}] 
(-3,0) 
to[out= 170,in=140, looseness=1] (-2.1,1.4) 
to[out= -40,in=90, looseness=1] (0,0)
(-3,0) 
to[out= -70,in=-180, looseness=1] (-1,-1.3) 
to[out= 0,in=-90, looseness=1] (0,0)
(-3,0) 
to[out= 50,in=180, looseness=1] (-2.3,0) 
to[out= 60,in=150, looseness=1.5] (-1.75,1) 
(-2.3,0)
to[out= -60,in=-120, looseness=0.9] (-1.5,-0.75)
(-1.75,1)
to[out= -30,in=90, looseness=0.9] (-1.25,0)
to[out= -90,in=60, looseness=0.9] (-1.5,-0.75);
\draw[color=black!50!white,scale=1,domain=-4.141: 4.141,
smooth,variable=\t,shift={(5.28,0)},rotate=0]plot({2.3*sin(\t r)},
{2.3*cos(\t r)}) ;
\path[shift={(7,0)}] 
(-3,-2.1) node[left]{$\Omega$}
(-2.9,-0.2) node[left]{$O^2$}
(-2.3,0)node[right]{$O^1$}
 (-2.5,0.72) node[below] {$\gamma^3$}
 (-0.7,0.5)[left] node{$\gamma^1$}
 (0,-0.5)[left] node{$\gamma^2$};
\end{tikzpicture}
\end{center}
\begin{caption}{Eyeglasses -- ``type A'' and ``type B''.\label{eyeglasses}}
\end{caption}
\end{figure}
Considering a reachable point for the flow $x_0\in\Omega$, the possible blow--up limit shrinkers $\widetilde\SS_\infty$, as $t\to T$, by Proposition~\ref{possiblelimit} are:
\begin{itemize}
\item a straight line through the origin;
\item a standard triod;
\item a standard cross;
\item a Brakke spoon,
\end{itemize}
where we excluded the shrinking lens and fish, since are not topological compatible with the possible limits of eyeglasses--shaped network (limit regions cannot ``increase'' the number of edges).

\smallskip

We first analyze the behavior of a ``type A'' eyeglasses--shaped network.\\
If $T<\frac{3}{5\pi}\min\{A_1(0) , A_2(0)\}$, no region has collapsed, then the evolving network is locally a tree for all times, hence (as in the analogous case for a $\Theta$--shaped network), the curvature stays bounded while only the length of the single ``open'' curve is going to zero, as $t\to T$, forming a regular $4$--point, where the two triple junctions converge.

\medskip

If $T=\frac{3}{5\pi}\min\{A_1(0) , A_2(0)\}$, then at last one of the two region collapses (with unbounded curvature) at some point $x_0\in\Omega$ and the blow--up limit shrinker $\widetilde\SS_\infty$ must be a Brakke spoon. We underline that, differently from the case of the $\Theta$--shaped network, if $A_1(0)=A_2(0)$, we do not have an argument to exclude that both regions collapse to a single common point, as $t\to T$ (even if it seems quite implausible).\\
Hence, we have the following possibilities, in the case of a collapse of a region, as $t\to T$:
\begin{itemize}
\item if $A_1(0)\not=A_2(0)$, then the limit network $\SS_T$ is a spoon with an ``open'' $C^1$ curve ending at the collapse point, see Figure~\ref{f21};
\item if $A_1(0)=A_2(0)$ and the two regions collapse at two different points of $\Omega$, the limit network $\SS_T$ is a $C^1$ curve connecting such two points;
\item if $A_1(0)=A_2(0)$ and the two regions collapse at a common point of $\Omega$, then the limit network $\SS_T$ is a closed $C^1$ curve, starting and ending at the collapse point and there possibly forming an angle, if the length of the ``open'' curve does not go to zero, otherwise, all the network collapses at such point, if also the length of the ``open'' curve goes to zero.
\end{itemize}
Anyway, we conjecture that this last situation and in particular, a complete ``vanishing'' of the network, as $t\to T$, is not possible.

\medskip

We now deal with a ``type B'' eyeglasses--shaped network.\\
From what we said above, if $T<\frac{3A_1(0)}{5\pi}$, no region has collapsed, then the evolving network is locally a tree for all times, hence (as above), the curvature stays bounded while only the length of the single ``open'' curve is going to zero, as $t\to T$, forming a regular $4$--point, where the two triple junctions converge.

\medskip

If $T=\frac{3A_1(0)}{5\pi}$, then the ``internal'' region collapses (with unbounded curvature) at some point $x_0\in\Omega$ and the blow--up limit shrinker $\widetilde\SS_\infty$ must be a Brakke spoon. Since, arguing as in Proposition~\ref{loop}, we have that the area $A_2$ of the region between the inner and the outer closed curve of the network satisfies
$A_2'=-2\pi/3$, while $A_1'=-5\pi/3$, if $A_1(0)/A_2(0)>5/2$, we have a contradiction since $A_2$ would go to zero before $A_1$ and this cannot happen (it would contradict what we said at the beginning of this section). Hence, in this case, the initial areas must satisfy $A_1(0)/A_2(0)\leqslant5/2$.\\
If we have the equality $A_1(0)/A_2(0)=5/2$, at time $T$, both regions are collapsing, as $t\to T$ and they cannot ``disappear'' in the blow--up limit shrinker $\widetilde\SS_\infty$, since in the rescaled sequence they have constant area and, being one contained in the other, no one of them can ``go to infinity''. Hence, the blow--up limit shrinker would have two regions, which is impossible, as it must a Brakke spoon, by Proposition~\ref{possiblelimit}.\\
Thus, it must be $A_1(0)/A_2(0)<5/2$ and the network cannot completely ``vanish'' at a single point of $\Omega$. We have instead that only the ``interior'' region collapses at a point and the limit network $\SS_T$ is a closed $C^1$ curve, starting and ending at the collapse point and there forming an angle of $120$ degrees, if the length of the ``open'' curve goes to zero, otherwise, there is also a $C^1$ curve connecting the collapse point with the angle of the limit of the curve $\gamma^2$ (see Figure~\ref{eyeglasses}), if the length of the ``open'' curve does not go to zero.

\bigskip

In case of collapse of the ``open'' curve, for both types, we ``pass through'' the singularity as before, with a standard transition, getting a $\Theta$--shaped network after the time $T$ (see the right side of Figure~\ref{switch}).\\
In the other cases, imposing that after the time $T$, all the ``open'' curves with a ``free'' end--point vanish in the subsequent evolution, we have only to deal with the remaining part of the network (if present) and we can restart the flow with the same arguments discussed above for the limits at a singular time of a $\Theta$--shaped network.

\subsubsection{The lens}

The main difference between this case (and also the next ones) with the theta and the eyeglasses cases, is that boundary points are present. 
\begin{figure}[H]
\begin{center}
\begin{tikzpicture}[scale=0.9]
\draw[color=black,shift={(5,0)}] 
(-4.7,1)
to[out= -50,in=180, looseness=1] (-2.8,0) 
to[out= 60,in=150, looseness=1.5] (-1.5,1) 
(-2.8,0)
to[out=-60,in=180, looseness=0.9] (-1.25,-0.75)
(-1.5,1)
to[out= -30,in=90, looseness=0.9] (-1,0)
to[out= -90,in=60, looseness=0.9] (-1.25,-0.75)
to[out= -60,in=150, looseness=0.9](1,-1.3);
\draw[color=black!50!white,scale=1,domain=-3.15: 3.15,
smooth,variable=\t,shift={(3.28,0)},rotate=0]plot({3.25*sin(\t r)},
{2.5*cos(\t r)}) ; 
\path[shift={(5,0)}]
(-5.5,-0.75)node[right]{$\Omega$}
(-4.7,1) node[left]{$P^1$}
(1,-1.3)node[right]{$P^2$}
(-2.85,0.05) node[below] {$O_1$}
(-2,0.5) node[below] {$A$}
(-1.25,-0.75)node[right]{$O_2$}
(-1.5,-1)[left] node{$\gamma^2$}
(-3.6,0.89) node[below] {$\gamma^1$}
(-0.6,0.9)[left] node{$\gamma^4$}
(0.7,-0.75)[left] node{$\gamma^3$};
\end{tikzpicture}
\end{center}
\begin{caption}{Lens.\label{lens}}
\end{caption}
\end{figure}
This increases the list of the possible blow--up limit networks $\widetilde\SS_\infty$. Indeed, by Proposition~\ref{possiblelimit}, if the blow--up point $x_0\in\Omega$, they can be:
\begin{itemize}
\item a straight line through the origin;
\item a standard triod;
\item a standard cross;
\item a shrinking lens;
\item a shrinking fish,
\end{itemize}
where we excluded the Brakke spoon for topological reasons and, if $x_0\in\partial\Omega$,
\begin{itemize}
\item a halfline from the origin;
\item two halflines from the origin that form an angle of $120$ degrees.
\end{itemize}
As we have a region with two edges in this network, the maximal time of existence $T$ is finite and bounded by $\frac{3A(0)}{4\pi}$, by Proposition~\ref{loop}.

\medskip

If $T<\frac{3A(0)}{4\pi}$, no region collapses, the evolving network is locally a tree for all times and the curvature stays bounded. Then, we can have two cases: either the length of one of the two ``central'' curves goes to zero, or this happens for one or both the ``boundary'' curves. In the first case, the limit network $\SS_T$ has a regular $4$--point connected with the two end--points and with a closed $C^1$ curve, starting and ending at such point, forming an angle of $60$ degrees. In the second case, $\SS_T$ can have either two curves between the two end--points bounding a region, or a curve from an end--point with a triple junction at its other end, which is connected with the other end--point by two curves bounding a region. At the end points, the curves form an angle of $120$ degrees.

\medskip

If $T=\frac{3A(0)}{4\pi}$ the central region is collapsing (with unbounded curvature) and the sequence $\widetilde{\SS}_{x_0,\tt_{j}}$ converges to a shrinking lens or to a shrinking fish, hence giving as a limit network $\SS_T$, either a $C^1$ curve connecting the two end--points (if the blow--up limit shrinker is a shrinking lens), or two curves from the two end--points to the collapse point in $\Omega$, where they form and angle like the one between the two ``halfilnes'' of the shrinking fish (if this is the blow--up limit shrinker). We remind that 
the collapse at the same time of both triple junctions (and the central region) to an end--point on $\partial\Omega$ is excluded by Lemma~\ref{trinoncoll}.

\bigskip

In the first case, we ``pass through'' the singularity as before, with a standard transition, getting an island--shaped network, after the time $T$ (see the left side of Figure~\ref{switch1}).\\
If one or both the boundary curves collapses to an end--point, we actually do not have a natural way to restart the flow (as in the case of the spoon, when the ``open'' curve collapses).\\
If the central region collapses, hence $\SS_T$ is a piecewise $C^1$ curve with possibly a single angle and (fixed) end--points on $\partial\Omega$, by the results in~\cite{eckhui2} and~\cite{angen3}, we can (uniquely) restart the evolution by means of the curve shortening flow with fixed end--points, obtaining an evolving embedded curve, which becomes immediately smooth and converges as $t\to+\infty$, to the segment connecting such end--points.

\subsubsection{The island}

As for the previous networks with a closed curve, for an island--shaped network, the maximal time of existence $T$ of a smooth flow in bounded by $\frac{3A(0)}{5\pi}$.
\begin{figure}[H]
\begin{center}
\begin{tikzpicture}[scale=0.9]
\draw[color=black, shift={(10,0)}] 
(-2,0)
to[out= 170,in=40, looseness=1] (-2.9,1.2) 
to[out= -140,in=90, looseness=1] (-3.2,0)
(-2,0) 
to[out= -70,in=0, looseness=1] (-3,-0.9) 
to[out= -180,in=-90, looseness=1] (-3.2,0)
(-2,0) 
to[out= 50,in=180, looseness=1] (-1.3,0) 
to[out= 60,in=150, looseness=1.5] (-0.75,1) 
(-1.3,0)
to[out= -60,in=-120, looseness=0.9] (-0.5,-0.75)
(-0.75,1)
to[out= -30,in=-150, looseness=0.9] (1.75,1)
(0,-2.32)
to[out= 90,in=60, looseness=0.9] (-0.5,-0.75);
\draw[color=black!50!white,scale=1,domain=-3.141: 3.141,
smooth,variable=\t,shift={(8.78,0)},rotate=0]plot({3.25*sin(\t r)},
{2.5*cos(\t r)}) ;
\path[shift={(10,0)}]
(-2,-0.2) node[left]{$O^1$}
(-1.3,0)node[right]{$O^2$}
(-2.85,0.8) node[below] {$\gamma^1$}
(-1.2,0.4)[left] node{$\gamma^2$}
(-0.5,-1.2)[left] node{$\gamma^3$}
(1,1.5) node[below] {$\gamma^4$}
(0.5,-2.6)[left] node{$P^2$}
(2.5,1.35)[left] node{$P^1$};
\end{tikzpicture}
\end{center}
\begin{caption}{Island.\label{island2}}
\end{caption}
\end{figure}
By Theorem~\ref{curvexplod-general}, if the blow--up point $x_0\in\Omega$, the blow--up limit networks $\widetilde\SS_\infty$ can be:
\begin{itemize}
\item a straight line through the origin;
\item a standard triod;
\item a standard cross;
\item a Brakke spoon,
\end{itemize}
where we excluded the standard lens and fish (as the limit cannot have a region with more than one edge) and, if $x_0\in\partial\Omega$,
\begin{itemize}
\item a halfline from the origin;
\item two halflines from the origin that form an angle of $120$ degrees.
\end{itemize}

If $T<\frac{3A(0)}{4\pi}$, no region collapses, the evolving network is locally a tree for all times and the curvature stays bounded. Then, we can have two cases: either the curve $\gamma^2$ in the figure collapses with $O^1$ and $O^2$ forming a $4$--point, or the length of one of the two ``boundary'' curves goes to zero. In the first case, the limit network $\SS_T$ has a regular $4$--point connected with the two end--points and with a closed $C^1$ curve, starting and ending at such point, forming an angle of $120$ degrees. In the second case, $\SS_T$ is formed by the union of a spoon and a $C^1$ curve connecting the two end--point. The ``open'' curve of the spoon form an angle of $120$ degrees with such connecting curve at the end--point where they concur.

\medskip

If $T=\frac{3A(0)}{4\pi}$, the region is collapsing (with unbounded curvature) and the sequence $\widetilde{\SS}_{x_0,\tt_{j}}$ converges to a Brakke spoon. Hence, since the collapse at the same time of both triple junctions (and the central region) to an end--point on $\partial\Omega$ is excluded by Lemma~\ref{trinoncoll}, the limit network $\SS_T$ must be either a triod composed by two curves connecting the two end--points to the triple junctions and an ``open'' curve, or (if the curve $\gamma^2$ also collapses -- see Figure~\ref{island2}) two curves from the two end--points to the collapse point in $\Omega$. In both cases the two ``boundary'' curves form an angle of $120$ degrees.

\bigskip

In the first case, no region collapses, we ``pass through'' the singularity with a standard transition, getting an lens--shaped network, after the time $T$ (see the right side of Figure~\ref{switch1}), if the ``open'' curve collapses in $\Omega$. If instead, it is a ``boundary'' curve which collapses, we do not have a natural way to continue the flow.\\
In the second case, as before, we ``forget'' the possibly present ``open'' curve, imposing that after the time $T$ it vanish and we can (uniquely) restart the evolution of the piecewise $C^1$ curve with a single angle by means of the curve shortening flow with fixed end--points (as in the case of a lens--shaped network when the central region collapses), obtaining an evolving embedded curve, which becomes immediately smooth and converges, as $t\to+\infty$, to the segment connecting such end--points.

\subsubsection{The tree}

This is the only network with two triple junctions which does not present loops. Consequently, it is the only case where we could have the global existence of the flow.
\begin{figure}[H]
\begin{center}
\begin{tikzpicture}[scale=0.9]
\draw[black]
 (-3.73,0) 
to[out= 50,in=180, looseness=1] (-2.3,0.7) 
to[out= 60,in=180, looseness=1.5] (-0.45,1.55) 
(-2.3,0.7)
to[out= -60,in=130, looseness=0.9] (-1,-0.3)
to[out= 10,in=100, looseness=0.9](0.1,-0.8)
(-1,-0.3)
to[out=-110,in=50, looseness=0.9](-2.7,-1.7);
\draw[color=black!50!white,scale=1,domain=-4.141: 4.141,
smooth,variable=\t,shift={(-1.72,0)},rotate=0]plot({2.*sin(\t r)},
{2.*cos(\t r)}) ;
\path
 (-3.73,0) node[left]{$P^1$}
 (-2.75,-1.75)node[below]{$P^2$}
 (0.1,-0.8)node[right]{$P^3$} 
 (-0.45,1.55) node[right]{$P^4$}
 (-3,0.6) node[below] {$\gamma^1$}
 (-1.5,1.3) node[right] {$\gamma^4$}
 (-1.1,-1.2)[left] node{$\gamma^2$}
 (0,-0.8)[left] node{$\gamma^3$}
 (-1.3,0.5)[left] node{$\gamma^5$}
 (-2.5,1.38) node[below] {$O_1$}
 (-1.4,0) node[below] {$O_2$}; 
\end{tikzpicture}
\end{center}
\begin{caption}{Tree.\label{tree2}}
\end{caption}
\end{figure}
Being a tree, by the analysis of Sections~\ref{van0} and~\ref{van} the curvature stays bounded till a possible singular time and we can only have a formation of a $4$--point or one or two non--concurrent ``boundary'' curve collapse to their respective end--point, forming an angle of $120$ degrees. In this latter case, as we said, we do not have a natural way to continue the flow, while in the first case, we have a standard transition, getting another tree, with the only other possible structure with the same end--points (see Figure~\ref{switch2}).\\
If $T=+\infty$ or the number of standard transition during the ``extended'' flow is finite, $\mathbb{S}_t$ tends, as $t\to\infty$, to the Steiner configuration of minimal length, connecting the four fixed end--points.

\subsubsection{The symmetric tree}

Following~\cite{PlPo22A,NoSc23}, if we add a symmetry assumption, we get a complete description of the evolution of a tree with two triple junctions.
Suppose $\SS_0$ is the smooth regular network in Figure~\ref{family2}.
The network has four end--points located at the vertices of a rectangle,
it is composed of five curves, symmetric with respect to the horizontal and vertical axes, the middle curve $\gamma^0$ is a segment.
\begin{figure}[H]
\begin{center}
\begin{tikzpicture}[scale=2.2]
\draw[gray, shift={(1.3,0)}, scale=0.5, rotate=-90]
(0,0)to[out= -45,in=135, looseness=1] (0.1,-0.1)
(0,0)to[out= -135,in=45, looseness=1] (-0.1,-0.1);		
\draw[gray, shift={(0,0.6)}, scale=0.5, rotate=0]
(0,0)to[out= -45,in=135, looseness=1] (0.1,-0.1)
(0,0)to[out= -135,in=45, looseness=1] (-0.1,-0.1);	
\draw[gray,dashed]
(-1,-0.5)--(1,0.5)
(1,-0.5)--(-1,0.5);
\draw[gray]
(-1.3,0)--(1.3,0);
\draw[gray]
(0,-0.6)--(0,0.6);
\draw[thick]
(-1,-0.5)to[out=5, in=210, looseness=0.8]
(0,-0.15)to[out=90, in=270, looseness=0]
(0,0.15);
\draw[rotate=180,thick]
(-1,-0.5)to[out=5, in=210, looseness=0.8]
(0,-0.15)to[out=90, in=270, looseness=0]
(0,0.15);
\draw[xscale=-1,thick]
(-1,-0.5)to[out=5, in=210, looseness=0.8]
(0,-0.15);
\draw[xscale=-1, rotate=180,thick]
(-1,-0.5)to[out=5, in=210, looseness=0.8]
(0,-0.15);
\fill[black](-1,-0.5) circle (0.7pt); 
\fill[black](1,-0.5)circle (0.7pt); 
\fill[black](-1,0.5) circle (0.7pt); 
\fill[black](1,0.5)circle (0.7pt); 
\path[font=\normalsize]
(-0.3,0.3)node[above]{$\gamma^1$};
\path[font=\normalsize]
(0,0.07)node[right]{$\gamma^0$};
\end{tikzpicture}
\end{center}
\begin{caption}{A symmetric tree network.\label{family2}}
\end{caption}
\end{figure}
Thanks to the symmetries, we can reduce to study the evolution of a single curve, for instance, $\gamma^1$. In this case one can prove (see~\cite{NoSc23}) that the network flow encounters only a finite number of standard transitions, so that it is eventually regular and globally defined. The limit, as $t\to +\infty$, is therefore a Steiner tree or a standard cross (only when the ratio between the longer and shorter side of the rectangle is equal to $\sqrt{3})$. In the latter case, the length of $\gamma^0$ goes to zero and the curvature of the network remains bounded.

\begin{rem} Taking into account the discussion at the beginning of Section~\ref{llong}, one should actually consider the flow of theta--eyeglasses and lens--island coupled, as a standard transition ``switches'' the shape/topology of two networks
from one to the other (like for the only two possible trees connecting four points, as we said above), as in the Figures~\ref{switch2},~\ref{switch1} and~\ref{switch}). 

Let us assume that 
\begin{itemize}
\item[i)] singular times are finite;
\item[ii)] there is no collapse of ``boundary'' curves;
\item[iii)] any ``open'' curve generated by a singularity, immediately disappears when we restart the flow;
\end{itemize}
then, at some time at least one region must collapse and
\begin{itemize}
\item in the case theta--eyeglasses, either we get a closed curve with possibly an angle that evolves smoothly by curve shortening flow and shrinks in finite time to a ``round'' point of $\Omega$, or the network completely vanishes (we actually think this last scenario is not possible),
\item in the case lens--island, we get a piecewise $C^1$ curve with possibly a single angle connecting the two end--points, which then evolves smoothly by curve shortening flow with fixed end--points and converges, as $t\to+\infty$, to the segment connecting such end--points.
\end{itemize}
We observe that in both cases, these are the last singular times of the flows (before the ``vanishing'' in the first case).
\end{rem}

\section{Open problems}\label{open}

In this section we recall some problems that we find the most important among the several open questions scattered in the text.

\begin{enumerate}

\item {\em Definition of the flow.}\\ 
Our ``parametric'' approach gives a good definition for the curvature flow of a network, compared with the existing notions of generalized evolutions for singular objects, more general but allowing weaker conclusions. The only unsatisfactory point is that we {\em impose} the presence of only triple junctions and the $120$ degrees angle condition. Thanks to them, we have the well--posedness of the system of PDE's~\eqref{problema-nogauge-general}, hence the short--time existence Theorems~\ref{wellposednessSobolev} and~\ref{2compexist0}, in Sobolev and H\"{o}lder spaces, respectively.\\
Nevertheless, one may wonder if these two conditions are automatically satisfied instantaneously, for every positive time, by choosing a different suitable definition of the curvature flow of a network.

\item {\em Multiplicity--one conjecture.}\\
Maybe the main open problem in the subject is the multiplicity--one
conjecture, that is, whether every blow--up limit shrinker is an
embedded network with multiplicity one (see Problem~\ref{ooo9}).\\
Several of the arguments and results in this work depend on such conjecture, we mention its fundamental role in the description of the limit network at a singular time and, consequently, in the possibility to implement the restarting procedure in order to continue the evolution, moreover, it is also a key ingredient in showing that the curvature of a tree--like network is uniformly bounded during the flow for all times and that one has only to deal with ``standard transitions'' at the singular times (see Section~\ref{behavsing}).\\
At the moment, we are able to prove the (strong) multiplicity--one conjecture only for networks with at most two triple junctions (see Section~\ref{dsuL}).

\item {\em Uniqueness of blow--up limits.}\\
According to Proposition~\ref{resclimit-general},
the sequence of rescaled networks $\widetilde{\SS}_{x_0,\tt_{j}}$ associated to a sequence of rescaled times
$\tt_j\to+\infty$, converges to a degenerate regular shrinker $\widetilde{\SS}_\infty$, only up to a subsequence. Analogously, in Proposition~\ref{thm:shrinkingnetworks.1}, the sequence of rescaled curvature flows $\SS^{\mu_i}_\tt$ converges to a degenerate regular self--similarly shrinking flow $\SS^\infty_\tt$, up to a subsequence.\\
One would like to prove that the limit degenerate regular shrinker $\widetilde\SS_\infty$ (and/or the degenerate regular self--similarly shrinking flow $\SS^\infty_\tt$) is actually independent of the chosen converging subsequences, that is, the full family $\widetilde{\SS}_{x_0,\tt}$ $C^1\loc$--converges to $\widetilde\SS_\infty$, as $\tt\to+\infty$. This is what we called {\em uniqueness assumption} in Problem~\ref{ooo12} and it is fundamental for the conclusions of Proposition~\ref{loopsing} and Theorem~\ref{ppp2}, necessary to restart the flow when a region collapses at a singular time.\\
Some positive partial results were recently obtained in~\cite{PlPo22A-2}, in particular, uniqueness holds if the blow--up limit shrinker is compact (some examples are given in the Appendix).

\item {\em Behavior when a region collapses.}\\
The singularities when a whole region collapses and then vanishes are the most difficult to deal with, in particular because the curvature is unbounded. We are not able, at the moment, to give a complete picture of the behavior of the evolving network, getting close to the singular time. A couple of conjectures are stated in Problems~\ref{ooo112} and~\ref{ooo113}, in particular, we expect that the non--collapsing curves ``exiting'' from the collapsing regions (and converging to the concurring curves at the new multi--point of the limit network) have locally uniformly bounded curvature during the flow and that, anyway, such singularities are actually all Type~I singularities, see Remark~\ref{T1} (in other words, the curvature flow of embedded networks does not develop Type~II singularities).\\
Anyway, hypothetically admitting the possibility of Type~II singularities, one is led to consider and try to analyze/classify also Type~II blow--up limits (see~\cite[Section~7]{mannovtor}), which are actually ``eternal'' curvature flows of regular networks (for instance, the ``translating'' ones, see~\cite[Section~5.2]{mannovtor}, that possibly coincide with them).

\item {\em Classification of shrinkers.}\\
Several questions (also of independent interest) arise in trying to classify the (embedded) regular shrinkers. Such a classification is complete for shrinkers with at most two triple junctions~\cite{balhausman,balhausman2,balhausman3}, or for the shrinkers with a single bounded region~\cite{chen_guo,chenguo,schnurerlens,balhausman}, see the following figure.

\begin{figure}[H]
\begin{center}
\begin{tikzpicture}[scale=0.5]
\draw[color=black, scale=0.4, shift={(9.875,0)}]
(-3.035,0)to[out=60,in=180,looseness=1] (2.2,3)
(2.2,3)to[out=0,in=90,looseness=1] (5.43,0)
(-3.035,0)to[out=-60,in=180,looseness=1] (2.2,-3)
(2.2,-3)to[out=0,in=-90,looseness=1] (5.43,0);
\draw[color=black, scale=0.4, shift={(9.875,0)}]
(-7,0)to[out=0,in=180,looseness=1](-3,0);
\draw[color=black, scale=0.4, shift={(9.875,0)}, dashed]
(-9,0)to[out=0,in=180,looseness=1](-7,0);
\fill(4,0) circle (2pt);
\fill[white](4,-2) circle (2pt);
\end{tikzpicture}\qquad
\begin{tikzpicture}[scale=0.5]
\draw[color=black,scale=0.4, shift={(-23,-15)}]
(-3.035,0)to[out=60,in=180,looseness=1] (2.2,2.8)
(2.2,2.8)to[out=0,in=120,looseness=1] (7.435,0)
(2.2,-2.7)to[out=0,in=-120,looseness=1] (7.435,0)
(-3.035,0)to[out=-60,in=180,looseness=1] (2.2,-2.7);
\draw[color=black,scale=0.4, shift={(-23,-15)}]
(-7,0)to[out=0,in=180,looseness=1](-3.035,0)
(7.435,0)to[out=0,in=180,looseness=1](11.4,0);
\draw[color=black,dashed,scale=0.4, shift={(-23,-15)}]
(-9,0)to[out=0,in=180,looseness=1](-7,0)
(11.4,0)to[out=0,in=180,looseness=1](13.4,0);
\fill(-8.33,-6) circle (2pt);
\fill[white](-4,-8) circle (2pt);
\end{tikzpicture}\qquad
\begin{tikzpicture}[scale=0.5]
\draw[color=black,shift={(2,-6)},scale=1.5]
(-0.47,0)to[out=20,in=180,looseness=1](1.5,0.65)
(1.5,0.65)to[out=0,in=90,looseness=1] (2.37,0)
(-0.47,0)to[out=-20,in=180,looseness=1](1.5,-0.65)
(1.5,-0.65)to[out=0,in=-90,looseness=1] (2.37,0);
\draw[white, very thick,shift={(2,-6)},scale=1.5]
(-0.47,0)--(-.150,-0.13)
(-0.47,0)--(-.150,0.13);
\draw[color=black,shift={(2,-6)},scale=1.5]
(-.150,0.13)to[out=-101,in=90,looseness=1](-.18,0)
(-.18,0)to[out=-90,in=101,looseness=1](-.150,-0.13);
\draw[black,shift={(2,-6)},scale=1.5]
(-.150,0.13)--(-1.13,0.98)
(-.150,-0.13)--(-1.13,-0.98);
\draw[black, dashed,shift={(2,-6)},scale=1.5]
(-1.13,0.98)--(-1.50,1.31)
(-1.13,-0.98)--(-1.50,-1.31);
\fill(2,-6) circle (2pt);
\end{tikzpicture}\qquad
\begin{tikzpicture}[scale=0.25, rotate=90]
\fill[color=black](0,0) circle (4pt); 
\draw[color=black, scale=1.32]
(-1.37,2.38)to[out=120,in=-60,looseness=1](-2.5,4.33)
(-1.37,-2.38)to[out=-120,in=60,looseness=1](-2.5,-4.33);
\draw[color=black]
(2.75,0)to[out=0,in=180,looseness=1](5,0)
(-1.37,2.38)to[out=120,in=-60,looseness=1](-2.5,4.33)
(-1.37,-2.38)to[out=-120,in=60,looseness=1](-2.5,-4.33)
(-1.37,-2.38)to[out=120,in=-120,looseness=1](-1.37,2.38)
(2.75,0)to[out=120,in=0,looseness=1](-1.37,2.38)
(2.75,0)to[out=-120,in=0,looseness=1](-1.37,-2.38);
\draw[color=black,dashed]
(5,0)to[out=0,in=180,looseness=1](6,0);
\draw[color=black,dashed, scale=1.32]
(-2.5,4.33)to[out=120,in=-60,looseness=1](-3,5.19)
(-2.5,-4.33)to[out=-120,in=60,looseness=1](-3,-5.19);
\end{tikzpicture}
\ \\

\bigskip

\ \\
\begin{tikzpicture}[scale=0.5]
 \fill[color=black](0,0) circle (2pt); 
 \draw[color=black,shift={(0,0)}]
(2.8,0)to[out=0,in=180,looseness=1](4,0);
 \draw[color=black,dashed,shift={(0,0)}]
(4,0)to[out=0,in=180,looseness=1](5,0);
\draw[color=black,shift={(-0.05,0)},scale=4]
(-.150,0.13)to[out=-101,in=90,looseness=1](-.18,0)
(-.18,0)to[out=-90,in=101,looseness=1](-.150,-0.13);
\draw[black,shift={(-0.05,0)},scale=4]
(-.150,0.13)--(-0.63,0.45);
\draw[black,shift={(-0.05,0)},scale=4]
(-.150,-0.13)--(-0.63,-0.45);
\draw[color=black,shift={(-0.05,0)},scale=4]
(-.150,0.13)to[out=20,in=180,looseness=1](0.325,0.2)
(0.325,0.2)to[out=0,in=120,looseness=1] (0.725,0)
(-.150,-0.13)to[out=-20,in=180,looseness=1](0.325,-0.2)
(0.325,-0.2)to[out=0,in=-120,looseness=1] (0.725,0);
\draw[black,dashed,shift={(-0.05,0)},scale=4]
(-0.63,0.45)--(-0.7,0.5);
\draw[black,dashed,shift={(-0.05,0)},scale=4]
(-0.63,-0.45)--(-0.7,-0.5);
 \fill[color=white](0,-3.225) circle (2pt); 
\end{tikzpicture}\qquad\quad
\begin{tikzpicture}[scale=0.55, rotate=90]
\draw[color=black]
(-2.5,0)to[out=0,in=180,looseness=1](-1.15,0)
(1.15,0)to[out=0,in=180,looseness=1](2.5,0)
(0,2.5)to[out=-90,in=90,looseness=1](0,1.15)
(0,-2.5)to[out=90,in=-90,looseness=1](0,-1.15)
(-1.15,0)to[out=60,in=-150,looseness=1](0,1.15)
(-1.15,0)to[out=-60,in=150,looseness=1](0,-1.15)
(1.15,0)to[out=1200,in=-30,looseness=1](0,1.15)
(1.15,0)to[out=-120,in=30,looseness=1](0,-1.15);
\draw[color=black,dashed]
(-3,0)to[out=0,in=180,looseness=1](-2.5,0)
(2.5,0)to[out=0,in=180,looseness=1](3,0)
(0,3)to[out=-90,in=90,looseness=1](0,2.5)
(0,-3)to[out=90,in=-90,looseness=1](0,-2.5);
\fill(0,0) circle (1.5pt);
\end{tikzpicture}\quad
\begin{tikzpicture}[scale=0.6, rotate=90]
\draw[color=black]
(0.76,0)to[out=0,in=180,looseness=1](2.5,0)
(0.23,-0.72)to[out=-72,in=108,looseness=1](0.77,-2.37)
(0.23,0.72)to[out=72,in=-108,looseness=1](0.77,2.37)
(-0.61,-0.44)to[out=-144,in=36,looseness=1](-2,-1.46)
(-0.61,0.44)to[out=144,in=-36,looseness=1](-2,1.46)
(0.76,0)to[out=120,in=-48,looseness=1](0.23,0.72)
to[out=-162,in=24,looseness=1](-0.61,0.44)
to[out=-96,in=96,looseness=1](-0.61,-0.44)
(0.76,0)to[out=-120,in=48,looseness=1](0.23,-0.72)
to[out=162,in=-24,looseness=1](-0.61,-0.44);
\draw[color=black,dashed]
(2.5,0)to[out=0,in=180,looseness=1](3,0)
(0.77,-2.37)to[out=-72,in=108,looseness=1](0.92,-2.85)
(0.77,2.37)to[out=72,in=-108,looseness=1](0.92,2.85)
(-2,-1.46)to[out=-144,in=36,looseness=1](-2.42,-1.76)
(-2,1.46)to[out=144,in=-36,looseness=1](-2.42,1.76);
\fill(0,0) circle (1.5pt);
 \fill[color=white](-2.6875,0) circle (2pt); 
\end{tikzpicture}
\end{center}
\begin{caption}{The regular shrinkers with a single bounded region.\label{chsh}}
\end{caption}
\end{figure}
A lot of numerical computations, partial results and conjectures can be found in~\cite{haettenschweiler}. We mention the very natural question whether there exist regular shrinkers with more than five halflines. Moreover, interesting stability/instability results were recently obtained in~\cite{Chang}.

\item {\em The set of singular times.}\\
An important point to be understood, in order to define a curvature flow in a maximal time interval, passing through singular times by means of the restarting procedure described in Section~\ref{restart}, is whether the set of singular times is discrete or even finite (as it happens for symmetric networks with two triple junctions, see~\cite{NoSc23}), or if they can accumulate in some particular situation (see Problem~\ref{ooo120}). In this latter case, at the moment we actually do not know how to continue the flow.

\item {\em Asymptotic convergence.}\\
In the case of global existence in time of an ``extended'' curvature flow (see Section~\ref{llong}), we would like to show the convergence of the evolving network, as $t\to+\infty$, to a stationary network for the length functional (Problem~\ref{ooo3000}). At the moment, we are able to face this problem only under the assumption that the structure of the network stops changing after some time, that is, there are no singularities of the flow for large times, see~\cite{PlPo22A}.
\end{enumerate}

\newpage

\appendix
\setcounter{secnumdepth}{0}

\section{Appendix -- A regular shrinkers gallery ({\em courtesy of Tom Ilmanen)}}\label{appshrink}

The following figures of regular shrinkers with their Gaussian density are
based on numerical computations due to J. H\"attenschweiler
(see~\cite{haettenschweiler} where one can also find other positive and
negative examples and several conjectures) and T. Ilmanen. We remark that
this is not an exhaustive list, only the shrinkers with at most one
bounded region are completely classified, by the work of Chen and
Guo~\cite{chenguo} (and actually they are the only ones in this gallery
whose existence is rigorously proved). Moreover, all the shrinkers shown
below have at least one symmetry axis, we do not know of examples without
any symmetries at all.

\bigskip

\medskip

\noindent {\Large{No regions:}}

\medskip

\noindent
\begin{tikzpicture}[scale=0.9]
\draw
(-2,0)to[out=0,in=180,looseness=1](2,0);
\draw[color=black, scale=1]
(-2,1.6)to[out=-90,in=90,looseness=1](-2,-1.6)
(2,1.6)to[out=-90,in=90,looseness=1](2,-1.6)
(-2,1.6)to[out=0,in=180,looseness=1](2,1.6)
(-2,-1.6)to[out=0,in=180,looseness=1](2,-1.6);
\path
(-2,-2) node[right]{\text{Line}}
(-2,-2.5) node[right]{$\Theta=1$};
\end{tikzpicture}\qquad
\begin{tikzpicture}[scale=0.9]
\draw
(-2,0)to[out=0,in=180,looseness=1](0,0)
(0,0)to[out=60,in=-120,looseness=1](0.92,1.6)
(0,0)to[out=-60,in=120,looseness=1](0.92,-1.6);
\draw[color=black, scale=1]
(-2,1.6)to[out=-90,in=90,looseness=1](-2,-1.6)
(2,1.6)to[out=-90,in=90,looseness=1](2,-1.6)
(-2,1.6)to[out=0,in=180,looseness=1](2,1.6)
(-2,-1.6)to[out=0,in=180,looseness=1](2,-1.6);
\path
(-2,-2) node[right]{\text{Triod}}
(-2,-2.5) node[right]{$\Theta=1.5$};
\end{tikzpicture}
\medskip

\noindent{\Large{1 region:}}

\medskip

\noindent
\begin{tikzpicture}[scale=0.9]
\draw[color=black,scale=1,domain=-3.141: 3.141,
smooth,variable=\t,rotate=0]plot({1.3*sin(\t r)},
{1.3*cos(\t r)});
\draw[color=black, scale=1]
(-2,1.6)to[out=-90,in=90,looseness=1](-2,-1.6)
(2,1.6)to[out=-90,in=90,looseness=1](2,-1.6)
(-2,1.6)to[out=0,in=180,looseness=1](2,1.6)
(-2,-1.6)to[out=0,in=180,looseness=1](2,-1.6);
\path
(-2,-2) node[right]{\text{Circle}}
(-2,-2.5) node[right]{$\Theta=\sqrt{2\pi/e}\approx1.520$};
\end{tikzpicture}\qquad
\begin{tikzpicture}[scale=0.9]
\draw[color=black, scale=0.355]
(-3.035,0)to[out=60,in=180,looseness=1] (2.2,3)
(2.2,3)to[out=0,in=90,looseness=1] (5.43,0)
(-3.035,0)to[out=-60,in=180,looseness=1] (2.2,-3)
(2.2,-3)to[out=0,in=-90,looseness=1] (5.43,0);
\draw[color=black, scale=0.355]
(-5.6,0)to[out=0,in=180,looseness=1](-3,0);
\draw[color=black, scale=1]
(-2,1.6)to[out=-90,in=90,looseness=1](-2,-1.6)
(2,1.6)to[out=-90,in=90,looseness=1](2,-1.6)
(-2,1.6)to[out=0,in=180,looseness=1](2,1.6)
(-2,-1.6)to[out=0,in=180,looseness=1](2,-1.6);
\path
(-2,-2) node[right]{\text{Spoon}}
(-2,-2.5) node[right]{$\Theta\approx1.699$};
\end{tikzpicture}\qquad
\begin{tikzpicture}[scale=0.9]
\draw[color=black,shift={(-0.65,0.07)},scale=0.3]
(-3.035,0)to[out=60,in=180,looseness=1] (2.2,2.8)
(2.2,2.8)to[out=0,in=120,looseness=1] (7.435,0)
(2.2,-2.7)to[out=0,in=-120,looseness=1] (7.435,0)
(-3.035,0)to[out=-60,in=180,looseness=1] (2.2,-2.7);
\draw[color=black,shift={(-0.65,0.07)},scale=0.3]
(-4.5,0)to[out=0,in=180,looseness=1](-3.035,0)
(7.435,0)to[out=0,in=180,looseness=1](8.9,0);
\draw[color=black, scale=1]
(-2,1.6)to[out=-90,in=90,looseness=1](-2,-1.6)
(2,1.6)to[out=-90,in=90,looseness=1](2,-1.6)
(-2,1.6)to[out=0,in=180,looseness=1](2,1.6)
(-2,-1.6)to[out=0,in=180,looseness=1](2,-1.6);
\path
(-2,-2) node[right]{\text{Lens}}
(-2,-2.5) node[right]{$\Theta\approx1.789$};
\end{tikzpicture}\qquad
\begin{tikzpicture}[scale=0.9]
\draw[color=black,shift={(-1.05,-0.05)},scale=1.1]
(-0.47,0)to[out=20,in=180,looseness=1](1.5,0.65)
(1.5,0.65)to[out=0,in=90,looseness=1] (2.37,0)
(-0.47,0)to[out=-20,in=180,looseness=1](1.5,-0.65)
(1.5,-0.65)to[out=0,in=-90,looseness=1] (2.37,0);
\draw[white, very thick,shift={(-1.05,-0.05)},scale=1.1]
(-0.47,0)--(-.150,-0.13)
(-0.47,0)--(-.150,0.13);
\draw[color=black,shift={(-1.05,-0.05)},scale=1.1]
(-.150,0.13)to[out=-101,in=90,looseness=1](-.15,0)
(-.15,0)to[out=-90,in=101,looseness=1](-.150,-0.13);
\draw[black,shift={(-1.05,-0.05)},scale=1.1]
(-.150,0.13)--(-0.87,0.585)
(-.150,-0.13)--(-0.87,-0.585);
\draw[color=black, scale=1]
(-2,1.6)to[out=-90,in=90,looseness=1](-2,-1.6)
(2,1.6)to[out=-90,in=90,looseness=1](2,-1.6)
(-2,1.6)to[out=0,in=180,looseness=1](2,1.6)
(-2,-1.6)to[out=0,in=180,looseness=1](2,-1.6);
\path
(-2,-2) node[right]{\text{Fish}}
(-2,-2.5) node[right]{$\Theta\approx2.026$};
\end{tikzpicture}
\medskip

\noindent
\begin{tikzpicture}[scale=0.9]
\draw[color=black, shift={(-0.22,0)}, scale=0.485]
(2.75,0)to[out=0,in=180,looseness=1](4.6,0)
(-1.5,2.45)to[out=120,in=-60,looseness=1](-1.95,3.3)
(-1.5,-2.45)to[out=-120,in=60,looseness=1](-1.95,-3.3)
(2.75,0)to[out=120,in=0,looseness=1](-1.5,2.45)
(-1.5,-2.45)to[out=120,in=-120,looseness=1](-1.5,2.45)
(2.75,0)to[out=-120,in=0,looseness=1](-1.5,-2.45);
\draw[color=black, scale=1]
(-2,1.6)to[out=-90,in=90,looseness=1](-2,-1.6)
(2,1.6)to[out=-90,in=90,looseness=1](2,-1.6)
(-2,1.6)to[out=0,in=180,looseness=1](2,1.6)
(-2,-1.6)to[out=0,in=180,looseness=1](2,-1.6);
\path
(-2,-2) node[right]{\text{$3$--ray star}}
(-2,-2.5) node[right]{$\Theta\approx2.031$};
\end{tikzpicture}\qquad
\begin{tikzpicture}[scale=0.9]
\draw[color=black,shift={(-0.8,-0.1)}]
(2.17,0)to[out=0,in=180,looseness=1](2.8,0);
\draw[color=black,shift={(-0.8,-0.1)},scale=3]
(-.150,0.13)to[out=-101,in=90,looseness=1](-.165,0)
(-.165,0)to[out=-90,in=101,looseness=1](-.150,-0.13);
\draw[black,shift={(-0.8,-0.1)},scale=3]
(-.150,-0.13)--(-0.4,-0.3);
\draw[black,shift={(-0.8,-0.1)},scale=3]
(-.150,0.13)--(-0.4,0.3);
\draw[color=black,shift={(-0.8,-0.1)},scale=3]
(-.150,0.13)to[out=20,in=180,looseness=1](0.325,0.2)
(0.325,0.2)to[out=0,in=120,looseness=1] (0.725,0)
(-.150,-0.13)to[out=-20,in=180,looseness=1](0.325,-0.2)
(0.325,-0.2)to[out=0,in=-120,looseness=1] (0.725,0);
\draw[color=black, scale=1]
(-2,1.6)to[out=-90,in=90,looseness=1](-2,-1.6)
(2,1.6)to[out=-90,in=90,looseness=1](2,-1.6)
(-2,1.6)to[out=0,in=180,looseness=1](2,1.6)
(-2,-1.6)to[out=0,in=180,looseness=1](2,-1.6);
\path
(-2,-2) node[right]{\text{Rocket}}
(-2,-2.5) node[right]{$\Theta=\,?$};
\end{tikzpicture}\qquad
\begin{tikzpicture}[scale=0.9]
\draw[color=black,rotate=45,scale=1.15]
(-1.95,0)to[out=0,in=180,looseness=1](-1.15,0)
(1.15,0)to[out=0,in=180,looseness=1](1.95,0)
(0,1.95)to[out=-90,in=90,looseness=1](0,1.15)
(0,-1.95)to[out=90,in=-90,looseness=1](0,-1.15);
\draw[color=black,rotate=45,scale=1.15]
(-1.15,0)to[out=60,in=-150,looseness=1](0,1.15)
(-1.15,0)to[out=-60,in=150,looseness=1](0,-1.15)
(1.15,0)to[out=120,in=-30,looseness=1](0,1.15)
(1.15,0)to[out=-120,in=30,looseness=1](0,-1.15);
\draw[color=black, scale=1]
(-2,1.6)to[out=-90,in=90,looseness=1](-2,-1.6)
(2,1.6)to[out=-90,in=90,looseness=1](2,-1.6)
(-2,1.6)to[out=0,in=180,looseness=1](2,1.6)
(-2,-1.6)to[out=0,in=180,looseness=1](2,-1.6);
\path
(-2,-2) node[right]{\text{$4$--ray star}}
(-2,-2.5) node[right]{$\Theta\approx2.295$};
\end{tikzpicture}\qquad
\begin{tikzpicture}[scale=0.9]
\draw[color=black,rotate=36, scale=1.2]
(0.76,0)to[out=0,in=180,looseness=1](2.05,0)
(0.23,-0.72)to[out=-72,in=108,looseness=1](0.638,-1.97)
(0.23,0.72)to[out=72,in=-108,looseness=1](0.44,1.34)
(-0.61,-0.44)to[out=-144,in=36,looseness=1](-1.15,-0.82)
(-0.61,0.44)to[out=144,in=-36,looseness=1](-1.355,0.995);
\draw[color=black,rotate=36, scale=1.2]
(0.76,0)to[out=120,in=-48,looseness=1](0.23,0.72)
to[out=-162,in=24,looseness=1](-0.61,0.44)
to[out=-96,in=96,looseness=1](-0.61,-0.44)
(0.76,0)to[out=-120,in=48,looseness=1](0.23,-0.72)
to[out=162,in=-24,looseness=1](-0.61,-0.44);
\draw[color=black, scale=1]
(-2,1.6)to[out=-90,in=90,looseness=1](-2,-1.6)
(2,1.6)to[out=-90,in=90,looseness=1](2,-1.6)
(-2,1.6)to[out=0,in=180,looseness=1](2,1.6)
(-2,-1.6)to[out=0,in=180,looseness=1](2,-1.6);
\path
(-2,-2) node[right]{\text{$5$--ray star}}
(-2,-2.5) node[right]{$\Theta\approx2.606$};
\end{tikzpicture}
\medskip

\noindent{\Large{2 regions:}}

\medskip

\noindent
\begin{tikzpicture}[scale=0.9]
\draw[black,scale=0.95]
(-1.35,0.295)--(-2.07,0.445)
(-1.35,-0.295)--(-2.07,-0.445)
(1.35,0.295)--(2.07,0.445)
(1.35,-0.295)--(2.07,-0.445);
\draw[color=black, scale=0.95]
(-1.2,0)to[out=0,in=180,looseness=1](1.2,0)
(0,0.85)to[out=180,in=48,looseness=1](-1.35,0.295)
(0,-0.85)to[out=180,in=-48,looseness=1](-1.35,-0.295)
(-1.2,0)to[out=120,in=-72,looseness=1](-1.35,0.295)
(-1.2,0)to[out=-120,in=72,looseness=1](-1.35,-0.295)
(0,0.85)to[out=0,in=132,looseness=1](1.35,0.295)
(0,-0.85)to[out=0,in=-132,looseness=1](1.35,-0.295)
(1.2,0)to[out=60,in=-108,looseness=1](1.35,0.295)
(1.2,0)to[out=-60,in=108,looseness=1](1.35,-0.295);
\draw[color=black]
(-2,1.6)to[out=-90,in=90,looseness=1](-2,-1.6)
(2,1.6)to[out=-90,in=90,looseness=1](2,-1.6)
(-2,1.6)to[out=0,in=180,looseness=1](2,1.6)
(-2,-1.6)to[out=0,in=180,looseness=1](2,-1.6);
\path
(-2,-2) node[right]{\text{Cisgeminate eye}}
(-2,-2.5) node[right]{$\Theta=\,?$};
\end{tikzpicture}\qquad
\begin{tikzpicture}[scale=0.9]
\draw[color=black, scale=1]
(-2,1.6)to[out=-90,in=90,looseness=1](-2,-1.6)
(2,1.6)to[out=-90,in=90,looseness=1](2,-1.6)
(-2,1.6)to[out=0,in=180,looseness=1](2,1.6)
(-2,-1.6)to[out=0,in=180,looseness=1](2,-1.6);
\draw[black,scale=1.15]
(0,1.4)to[out=-90,in=90,looseness=1](0,1.15)
(0,-1.4)to[out=90,in=-90,looseness=1](0,-1.15)
(-1.08,0)to[out=0,in=180,looseness=1](1.08,0)
(1.08,0)to[out=60,in=-117,looseness=1](1.12,0.05)
(1.08,0)to[out=-60,in=117,looseness=1](1.12,-0.05)
(-1.08,0)to[out=120,in=-63,looseness=1](-1.12,0.05)
(-1.08,0)to[out=-120,in=63,looseness=1](-1.12,-0.05)
(1.12,0.05)to[out=3,in=-177,looseness=1](1.75,0.078)
(1.12,-0.05)to[out=-3,in=177,looseness=1](1.75,-0.078)
(-1.12,0.05)to[out=177,in=-3,looseness=1](-1.75,0.078)
(-1.12,-0.05)to[out=-177,in=3,looseness=1](-1.75,-0.078)
(-1.12,0.05)to[out=57,in=-150,looseness=1](0,1.15)
(-1.12,-0.05)to[out=-57,in=150,looseness=1](0,-1.15)
(1.12,0.05)to[out=123,in=-30,looseness=1](0,1.15)
(1.12,-0.05)to[out=-123,in=30,looseness=1](0,-1.15);
\path
(-2,-2) node[right]{\text{Cisgeminate $4$--ray star}}
(-2,-2.5) node[right]{$\Theta=\,?$};
\end{tikzpicture}

\eject

\noindent{\Large{3 regions:}}

\medskip

\noindent
\begin{tikzpicture}[scale=0.9]
\draw
(0,0)
to[out=180,in=0,looseness=1] (-1,0)
(0,0)
to[out=60,in=-120,looseness=1] (0.5,0.86)
(0,0)
to[out=-60,in=120,looseness=1] (0.5,-0.86)
(0.5,0.86)
to[out=0,in=90,looseness=1](1.5,0)
to[out=-90,in=0,looseness=1] (0.5,-0.86)
(0.5,0.86)
to[out=120,in=30,looseness=1](-0.75,1.29)
to[out=-150,in=120,looseness=1] (-1,0)
(0.5,-0.86)
to[out=-120,in=-30,looseness=1](-0.75,-1.29)
to[out=150,in=-120,looseness=1] (-1,0);
\draw[color=black, scale=1]
(-2,1.6)to[out=-90,in=90,looseness=1](-2,-1.6)
(2,1.6)to[out=-90,in=90,looseness=1](2,-1.6)
(-2,1.6)to[out=0,in=180,looseness=1](2,1.6)
(-2,-1.6)to[out=0,in=180,looseness=1](2,-1.6);
\path
(-2,-2) node[right]{\text{Mercedes--Benz}}
(-2,-2.5) node[right]{$\Theta\approx2.532$};
\end{tikzpicture}\qquad
\begin{tikzpicture}[scale=0.9]
\draw[color=black, scale=1]
(-0.12,0)to[out=120,in=-63,looseness=1](-0.37,0.51)
(-0.12,0)to[out=-120,in=63,looseness=1](-0.37,-0.51)
(-1.62,0)to[out=-60,in=-177,looseness=1](-0.37,-0.51)
(-1.62,0)to[out=60,in=177,looseness=1](-0.37,0.51)
(-0.37,-0.51)to[out=-57,in=180,looseness=1](0.72,-1.14)
to[out=0,in=-90,looseness=1](1.4,-0.51)
to[out=90,in=-60,looseness=1](1.28,0)
(-0.37,0.51)to[out=57,in=180,looseness=1](0.72,1.14)
to[out=0,in=90,looseness=1](1.4,0.51)
to[out=-90,in=60,looseness=1](1.28,0)
(-2,0)--(-1.62,0)
(-0.12,0)--(1.28,0);
\draw[color=black, scale=1]
(-2,1.6)to[out=-90,in=90,looseness=1](-2,-1.6)
(2,1.6)to[out=-90,in=90,looseness=1](2,-1.6)
(-2,1.6)to[out=0,in=180,looseness=1](2,1.6)
(-2,-1.6)to[out=0,in=180,looseness=1](2,-1.6);
\path
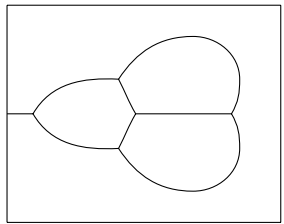
(-2,-2) node[right]{\text{$1$--ray Mercedes--Benz}}
(-2,-2.5) node[right]{$\Theta\approx2.598$};
\end{tikzpicture}\qquad
\begin{tikzpicture}[scale=0.9]
\draw[color=black, scale=1]
(-1.35,0)to[out=-60,in=180,looseness=1](-0.07,-0.6)
(-0.07,-0.6)to[out=-60,in=180,looseness=1](1.13,-1.47)
(1.13,-1.47)to[out=60,in=-60,looseness=1](0.98,0)
(-1.35,0)to[out=60,in=180,looseness=1](-0.07,0.6)
(-0.07,0.6)to[out=60,in=180,looseness=1](1.13,1.47)
(1.13,1.47)to[out=-60,in=60,looseness=1](0.98,0)
(1.13,1.47)--(1.2,1.6)
(1.13,-1.47)--(1.2,-1.6)
(0.28,0)--(-0.07,0.6)
(0.28,0)--(-0.07,-0.6)
(-2,0)--(-1.35,0)
(0.28,0)--(0.98,0);
\draw[color=black, scale=1]
(-2,1.6)to[out=-90,in=90,looseness=1](-2,-1.6)
(2,1.6)to[out=-90,in=90,looseness=1](2,-1.6)
(-2,1.6)to[out=0,in=180,looseness=1](2,1.6)
(-2,-1.6)to[out=0,in=180,looseness=1](2,-1.6);
\path
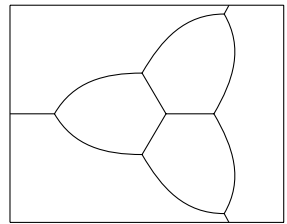
(-2,-2) node[right]{\text{$3$--ray Mercedes--Benz}}
(-2,-2.5) node[right]{$\Theta\approx2.762$};
\end{tikzpicture}\qquad
\begin{tikzpicture}[scale=0.9]
\draw[color=black, scale=1]
(-2,1.6)to[out=-90,in=90,looseness=1](-2,-1.6)
(2,1.6)to[out=-90,in=90,looseness=1](2,-1.6)
(-2,1.6)to[out=0,in=180,looseness=1](2,1.6)
(-2,-1.6)to[out=0,in=180,looseness=1](2,-1.6);
\draw[color=black, scale=0.52]
(0,0)to[out=0,in=180,looseness=1](2.75,0)
(2.75,0)to[out=120,in=-61,looseness=1](2.79,0.12)
(2.75,0)to[out=-120,in=61,looseness=1](2.79,-0.12)
(2.79,0.12)to[out=2,in=-178,looseness=1](3.85,0.16)
(2.79,-0.12)to[out=2,in=178,looseness=1](3.85,-0.16)
(2.79,0.12)to[out=120,in=-30,looseness=1](1.1,1.9)
(2.79,-0.12)to[out=-120,in=30,looseness=1](1.1,-1.9);
\draw[color=black, scale=0.52, rotate=120]
(0,0)to[out=0,in=180,looseness=1](2.75,0)
(2.75,0)to[out=120,in=-61,looseness=1](2.79,0.12)
(2.75,0)to[out=-120,in=61,looseness=1](2.79,-0.12)
(2.79,0.12)to[out=2,in=-178,looseness=1](3.65,0.15)
(2.79,-0.12)to[out=2,in=178,looseness=1](3.5,-0.14)
(2.79,0.12)to[out=120,in=-30,looseness=1](1.1,1.9)
(2.79,-0.12)to[out=-120,in=30,looseness=1](1.1,-1.9);
\draw[color=black, scale=0.52, rotate=-120]
(0,0)to[out=0,in=180,looseness=1](2.75,0)
(2.75,0)to[out=120,in=-61,looseness=1](2.79,0.12)
(2.75,0)to[out=-120,in=61,looseness=1](2.79,-0.12)
(2.79,0.12)to[out=2,in=-178,looseness=1](3.5,0.14)
(2.79,-0.12)to[out=2,in=178,looseness=1](3.65,-0.15)
(2.79,0.12)to[out=120,in=-30,looseness=1](1.1,1.9)
(2.79,-0.12)to[out=-120,in=30,looseness=1](1.1,-1.9);
\path
(-2,-2) node[right]{\text{Cisgeminate $3$--ray star}}
(-2,-2.5) node[right]{$\Theta=\,?$};
\end{tikzpicture}
\medskip

\noindent{\Large{4 regions:}}

\medskip

\noindent
\begin{tikzpicture}[scale=0.9]
\draw[color=black, shift={(0.19,0.08)}, rotate=120]
(-1.5,0.02)to[out=90,in=180,looseness=1](-0.51,0.85)
(-1.5,0.02)to[out=-90,in=180,looseness=1](-0.51,-0.85);
\draw[color=black, shift={(0.19,0.02)}, rotate=-120]
(-1.5,0.02)to[out=90,in=180,looseness=1](-0.51,0.85)
(-1.5,0.02)to[out=-90,in=180,looseness=1](-0.51,-0.85);
\draw[color=black, shift={(0.22,0.05)}]
(-1.5,0.02)to[out=90,in=180,looseness=1](-0.51,0.85)
(-1.5,0.02)to[out=-90,in=180,looseness=1](-0.51,-0.85)
(0.88,0)--(0.96,0)
(-0.48,0.77)--(-0.51,0.85)
(-0.48,-0.77)--(-0.51,-0.85);
\draw[color=black, shift={(0.22,0.05)}, scale=0.32]
(2.75,0)to[out=120,in=0,looseness=1](-1.5,2.45)
(-1.5,-2.45)to[out=120,in=-120,looseness=1](-1.5,2.45)
(2.75,0)to[out=-120,in=0,looseness=1](-1.5,-2.45);
\draw[color=black, scale=1]
(-2,1.6)to[out=-90,in=90,looseness=1](-2,-1.6)
(2,1.6)to[out=-90,in=90,looseness=1](2,-1.6)
(-2,1.6)to[out=0,in=180,looseness=1](2,1.6)
(-2,-1.6)to[out=0,in=180,looseness=1](2,-1.6);
\path
(-2,-2) node[right]{\text{$3$--leaf clover}}
(-2,-2.5) node[right]{$\Theta\approx3.064$};
\end{tikzpicture}\qquad
\begin{tikzpicture}[scale=0.9]
\draw[color=black, scale=1]
(0,0.19)to[out=150,in=-30,looseness=1](-0.39,0.4)
(0,0.19)to[out=30,in=-150,looseness=1](0.39,0.4)
(0,-0.19)to[out=-150,in=30,looseness=1](-0.39,-0.4)
(0,-0.19)to[out=-30,in=150,looseness=1](0.39,-0.4)
(0.39,0.4)to[out=-30,in=120,looseness=1](0.7,0)
(-0.39,0.4)to[out=-150,in=60,looseness=1](-0.7,0)
(0.39,-0.4)to[out=30,in=-120,looseness=1](0.7,0)
(-0.39,-0.4)to[out=150,in=-60,looseness=1](-0.7,0)
(0.39,0.4)to[out=90,in=-30,looseness=1](0,1.5)
(-0.39,0.4)to[out=90,in=-150,looseness=1](0,1.5)
(0.39,-0.4)to[out=-90,in=30,looseness=1](0,-1.5)
(-0.39,-0.4)to[out=-90,in=150,looseness=1](0,-1.5)
(0,-1.5)to[out=-90,in=90,looseness=1](0,-1.6)
(0,1.5)to[out=90,in=-90,looseness=1](0,1.6)
(0,0.19)to[out=-90,in=90,looseness=1](0,0)
(2,0)to[out=180,in=0,looseness=1](0.7,0)
(0,-0.19)to[out=90,in=-90,looseness=1](0,0)
(-2,0)to[out=0,in=180,looseness=1](-0.7,0);
\draw[color=black, scale=1]
(-2,1.6)to[out=-90,in=90,looseness=1](-2,-1.6)
(2,1.6)to[out=-90,in=90,looseness=1](2,-1.6)
(-2,1.6)to[out=0,in=180,looseness=1](2,1.6)
(-2,-1.6)to[out=0,in=180,looseness=1](2,-1.6);
\path
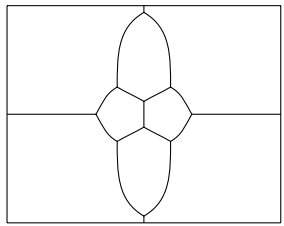
(-2,-2) node[right]{\text{$2$--ray $2$--floc}}
(-2,-2.5) node[right]{$\Theta\approx3.249$};
\end{tikzpicture}
\medskip

\noindent{\Large{5 regions:}}

\medskip

\noindent
\begin{tikzpicture}[scale=0.9]
\draw[color=black]
(-1.42,0)to[out=90,in=-165,looseness=1](-0.78,0.78)
(-1.42,0)to[out=-90,in=165,looseness=1](-0.78,-0.78)
(1.42,0)to[out=90,in=-15,looseness=1](0.78,0.78)
(1.42,0)to[out=-90,in=15,looseness=1](0.78,-0.78);
\draw[color=black,rotate=90]
(-1.42,0)to[out=90,in=-165,looseness=1](-0.78,0.78)
(-1.42,0)to[out=-90,in=165,looseness=1](-0.78,-0.78)
(1.42,0)to[out=90,in=-15,looseness=1](0.78,0.78)
(1.42,0)to[out=-90,in=15,looseness=1](0.78,-0.78);
\draw[color=black,rotate=45,scale=0.57]
(-1.95,0)to[out=0,in=180,looseness=1](-1.15,0)
(1.15,0)to[out=0,in=180,looseness=1](1.95,0)
(0,1.95)to[out=-90,in=90,looseness=1](0,1.15)
(0,-1.95)to[out=90,in=-90,looseness=1](0,-1.15);
\draw[,color=black,rotate=45,scale=0.57]
(-1.15,0)to[out=60,in=-150,looseness=1](0,1.15)
(-1.15,0)to[out=-60,in=150,looseness=1](0,-1.15)
(1.15,0)to[out=1200,in=-30,looseness=1](0,1.15)
(1.15,0)to[out=-120,in=30,looseness=1](0,-1.15);
\draw[color=black, scale=1]
(-2,1.6)to[out=-90,in=90,looseness=1](-2,-1.6)
(2,1.6)to[out=-90,in=90,looseness=1](2,-1.6)
(-2,1.6)to[out=0,in=180,looseness=1](2,1.6)
(-2,-1.6)to[out=0,in=180,looseness=1](2,-1.6);
\path
(-2,-2) node[right]{\text{$4$--leaf clover}}
(-2,-2.5) node[right]{$\Theta\approx3.234$};
\end{tikzpicture}\qquad
\begin{tikzpicture}[scale=0.9]
\draw[color=black, scale=1]
(-2,0)to[out=0,in=180,looseness=1](-1.72,0)
(1.72,0)to[out=0,in=180,looseness=1](2,0)
(-0.75,0.35)to[out=15,in=165,looseness=1](0.75,0.35)
(-0.75,-0.35)to[out=-15,in=-165,looseness=1](0.75,-0.35)
(0.75,0.35)to[out=-75,in=75,looseness=1](0.75,-0.35)
(-0.75,0.35)to[out=-105,in=105,looseness=1](-0.75,-0.35)
(0.94,0.52)to[out=105,in=0,looseness=1](0,1.3)
(-0.94,0.52)to[out=75,in=180,looseness=1](0,1.3)
(-0.94,-0.52)to[out=-75,in=180,looseness=1](0,-1.3)
(0.94,-0.52)to[out=-105,in=0,looseness=1](0,-1.3)
(0.94,0.52)--(0.75,0.35)
(0.94,-0.52)--(0.75,-0.35)
(-0.94,0.52)--(-0.75,0.35)
(-0.94,-0.52)--(-0.75,-0.35)
(1.72,0)to[out=120,in=-15,looseness=1](0.94,0.52)
(1.72,0)to[out=-120,in=15,looseness=1](0.94,-0.52)
(-1.72,0)to[out=60,in=-175,looseness=1](-0.94,0.52)
(-1.72,0)to[out=-60,in=175,looseness=1](-0.94,-0.52);
\draw[color=black, scale=1]
(-2,1.6)to[out=-90,in=90,looseness=1](-2,-1.6)
(2,1.6)to[out=-90,in=90,looseness=1](2,-1.6)
(-2,1.6)to[out=0,in=180,looseness=1](2,1.6)
(-2,-1.6)to[out=0,in=180,looseness=1](2,-1.6);
\path
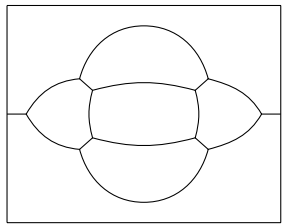
(-2,-2) node[right]{\text{$2$--ray $4$--leaf clover}}
(-2,-2.5) node[right]{$\Theta\approx3.365$};
\end{tikzpicture}\qquad
\begin{tikzpicture}[scale=0.9]
\draw[color=black, scale=1]
(0,1.42)--(0,1.6)
(0,-1.42)--(0,-1.6)
(2,0)--(1.42,0)
(-2,0)--(-1.42,0);
\draw[color=black]
(-0.615,0.615)to[out=75,in=-150,looseness=1](0,1.42)
(0.615,0.615)to[out=105,in=-30,looseness=1](0,1.42);
\draw[color=black,rotate=90]
(-0.615,0.615)to[out=75,in=-150,looseness=1](0,1.42)
(0.615,0.615)to[out=105,in=-30,looseness=1](0,1.42);
\draw[color=black,rotate=180]
(-0.615,0.615)to[out=75,in=-150,looseness=1](0,1.42)
(0.615,0.615)to[out=105,in=-30,looseness=1](0,1.42);
\draw[color=black,rotate=270]
(-0.615,0.615)to[out=75,in=-150,looseness=1](0,1.42)
(0.615,0.615)to[out=105,in=-30,looseness=1](0,1.42);
\draw[color=black,rotate=45,scale=0.65]
(-1.35,0)to[out=0,in=180,looseness=1](-1.15,0)
(1.15,0)to[out=0,in=180,looseness=1](1.35,0)
(0,1.35)to[out=-90,in=90,looseness=1](0,1.15)
(0,-1.35)to[out=90,in=-90,looseness=1](0,-1.15);
\draw[color=black,rotate=45,scale=0.65]
(-1.15,0)to[out=60,in=-150,looseness=1](0,1.15)
(-1.15,0)to[out=-60,in=150,looseness=1](0,-1.15)
(1.15,0)to[out=1200,in=-30,looseness=1](0,1.15)
(1.15,0)to[out=-120,in=30,looseness=1](0,-1.15);
\draw[color=black, scale=1]
(-2,1.6)to[out=-90,in=90,looseness=1](-2,-1.6)
(2,1.6)to[out=-90,in=90,looseness=1](2,-1.6)
(-2,1.6)to[out=0,in=180,looseness=1](2,1.6)
(-2,-1.6)to[out=0,in=180,looseness=1](2,-1.6);
\path
(-2,-2) node[right]{\text{$4$--petal flower}}
(-2,-2.5) node[right]{$\Theta\approx3.474$};
\end{tikzpicture}
\medskip

\noindent{\Large{6 regions:}}

\medskip

\noindent
\begin{tikzpicture}[scale=0.9]
\draw[rotate=36, scale=0.61]
(0.76,0)to[out=120,in=-48,looseness=1](0.23,0.72)
to[out=-162,in=24,looseness=1](-0.61,0.44)
to[out=-96,in=96,looseness=1](-0.61,-0.44)
(0.76,0)to[out=-120,in=48,looseness=1](0.23,-0.72)
to[out=162,in=-24,looseness=1](-0.61,-0.44);
\draw
(-1.2,-0.87)to[out=-54,in=-168,looseness=1](-0.37,-1.15)
(-1.22,0)to[out=-120,in=126,looseness=1](-1.2,-0.87);
\draw[rotate=72]
(-1.2,-0.87)to[out=-54,in=-168,looseness=1](-0.37,-1.15)
(-1.22,0)to[out=-120,in=126,looseness=1](-1.2,-0.87);
\draw[rotate=-72]
(-1.2,-0.87)to[out=-54,in=-168,looseness=1](-0.37,-1.15)
(-1.22,0)to[out=-120,in=126,looseness=1](-1.2,-0.87);
\draw[rotate=144]
(-1.2,-0.87)to[out=-54,in=-168,looseness=1](-0.37,-1.15)
(-1.22,0)to[out=-120,in=126,looseness=1](-1.2,-0.87);
\draw[rotate=-144]
(-1.2,-0.87)to[out=-54,in=-168,looseness=1](-0.37,-1.15)
(-1.22,0)to[out=-120,in=126,looseness=1](-1.2,-0.87);
\draw[ rotate=72]
(-1.22,0)to[out=0,in=180,looseness=1](-0.45,0);
\draw[rotate=-72]
(-1.22,0)to[out=0,in=180,looseness=1](-0.45,0);
\draw[rotate=144]
(-1.22,0)to[out=0,in=180,looseness=1](-0.45,0);
\draw[rotate=-144]
(-1.22,0)to[out=0,in=180,looseness=1](-0.45,0);
\draw
(-1.22,0)to[out=0,in=180,looseness=1](-0.45,0);
\draw[color=black, scale=1]
(-2,1.6)to[out=-90,in=90,looseness=1](-2,-1.6)
(2,1.6)to[out=-90,in=90,looseness=1](2,-1.6)
(-2,1.6)to[out=0,in=180,looseness=1](2,1.6)
(-2,-1.6)to[out=0,in=180,looseness=1](2,-1.6);
\path
(-2,-2) node[right]{\text{$5$--leaf clover}}
(-2,-2.5) node[right]{$\Theta\approx3.455$};
\end{tikzpicture}\qquad
\begin{tikzpicture}[scale=0.9]
\draw
(-0.54,0.43)to[out=45,in=-165,looseness=1](-0.1,0.685)
(-0.36,0)to[out=0,in=180,looseness=1](0,0)
(-0.36,0)to[out=-120,in=75,looseness=1](-0.54,-0.43)
(-1.67,0) to[out=-90,in=-165,looseness=1](-0.54,-0.43)
(-0.36,0)to[out=120,in=-75,looseness=1](-0.54,0.43)
(-1.67,0) to[out=90,in=165,looseness=1](-0.54,0.43);
\draw[rotate=-120]
(-0.54,0.43)to[out=45,in=-165,looseness=1](-0.1,0.685)
(-0.36,0)to[out=0,in=180,looseness=1](0,0)
(-0.36,0)to[out=-120,in=75,looseness=1](-0.54,-0.43)
(-1.67,0) to[out=-90,in=-165,looseness=1](-0.54,-0.43)
(-0.36,0)to[out=120,in=-75,looseness=1](-0.54,0.43)
(-1.67,0) to[out=90,in=165,looseness=1](-0.54,0.43);
\draw[rotate=120]
(-0.54,0.43)to[out=45,in=-165,looseness=1](-0.1,0.685)
(-0.36,0)to[out=0,in=180,looseness=1](0,0)
(-0.36,0)to[out=-120,in=75,looseness=1](-0.54,-0.43)
(-1.67,0) to[out=-90,in=-165,looseness=1](-0.54,-0.43)
(-0.36,0)to[out=120,in=-75,looseness=1](-0.54,0.43)
(-1.67,0) to[out=90,in=165,looseness=1](-0.54,0.43);
\draw[color=black, scale=1]
(-2,1.6)to[out=-90,in=90,looseness=1](-2,-1.6)
(2,1.6)to[out=-90,in=90,looseness=1](2,-1.6)
(-2,1.6)to[out=0,in=180,looseness=1](2,1.6)
(-2,-1.6)to[out=0,in=180,looseness=1](2,-1.6);
\path
(-2,-2) node[right]{\text{$3$--floc}}
(-2,-2.5) node[right]{$\Theta\approx3.477$};
\end{tikzpicture}\qquad
\begin{tikzpicture}[scale=0.9]
\draw[color=black, scale=1]
(-2,1.6)to[out=-90,in=90,looseness=1](-2,-1.6)
(2,1.6)to[out=-90,in=90,looseness=1](2,-1.6)
(-2,1.6)to[out=0,in=180,looseness=1](2,1.6)
(-2,-1.6)to[out=0,in=180,looseness=1](2,-1.6);
\draw[shift={(0.36,0)}]
(-0.38,0)to[out=0,in=180,looseness=1](0,0)
(-2.32,0)to[out=0,in=180,looseness=1](-1.62,0)
(-1.62,0)to[out=60,in=180,looseness=1](-1,0.26)
(-1.62,0)to[out=-60,in=180,looseness=1](-1,-0.26)
(-1,0.26)to[out=0,in=170,looseness=1](-0.47,0.21)
(-1,-0.26)to[out=0,in=-170,looseness=1](-0.47,-0.21)
(-0.38,0)to[out=-120,in=70,looseness=1](-0.47,-0.21)
(-0.38,0)to[out=120,in=-70,looseness=1](-0.47,0.21)
(-0.47,0.21)to[out=50,in=-170,looseness=1](0.065,0.515);
\draw[shift={(0.36,0)}, rotate=120]
(-0.38,0)to[out=0,in=180,looseness=1](0,0)
(-1.85,0)to[out=0,in=180,looseness=1](-1.62,0)
(-1.62,0)to[out=60,in=180,looseness=1](-1,0.26)
(-1.62,0)to[out=-60,in=180,looseness=1](-1,-0.26)
(-1,0.26)to[out=0,in=170,looseness=1](-0.47,0.21)
(-1,-0.26)to[out=0,in=-170,looseness=1](-0.47,-0.21)
(-0.38,0)to[out=-120,in=70,looseness=1](-0.47,-0.21)
(-0.38,0)to[out=120,in=-70,looseness=1](-0.47,0.21)
(-0.47,0.21)to[out=50,in=-170,looseness=1](0.065,0.515);
\draw[shift={(0.36,0)},rotate=-120]
(-0.38,0)to[out=0,in=180,looseness=1](0,0)
(-1.85,0)to[out=0,in=180,looseness=1](-1.62,0)
(-1.62,0)to[out=60,in=180,looseness=1](-1,0.26)
(-1.62,0)to[out=-60,in=180,looseness=1](-1,-0.26)
(-1,0.26)to[out=0,in=170,looseness=1](-0.47,0.21)
(-1,-0.26)to[out=0,in=-170,looseness=1](-0.47,-0.21)
(-0.38,0)to[out=-120,in=70,looseness=1](-0.47,-0.21)
(-0.38,0)to[out=120,in=-70,looseness=1](-0.47,0.21)
(-0.47,0.21)to[out=50,in=-170,looseness=1](0.065,0.515);
\path
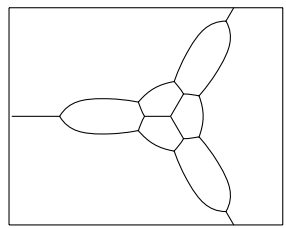
(-2,-2) node[right]{\text{$3$--ray three--floc}}
(-2,-2.5) node[right]{$\Theta\approx3.517$};
\end{tikzpicture}\qquad
\begin{tikzpicture}[scale=0.9]
\draw[rotate=36, scale=0.77]
(0.76,0)to[out=120,in=-48,looseness=1](0.23,0.72)
to[out=-162,in=24,looseness=1](-0.61,0.44)
to[out=-96,in=96,looseness=1](-0.61,-0.44)
(0.76,0)to[out=-120,in=48,looseness=1](0.23,-0.72)
to[out=162,in=-24,looseness=1](-0.61,-0.44);
\draw
(-1.01,0)to[out=0,in=180,looseness=1](-0.58,0);
\draw[rotate=72]
(-1.01,0)to[out=0,in=180,looseness=1](-0.58,0);
\draw[rotate=-72]
(-1.01,0)to[out=0,in=180,looseness=1](-0.58,0);
\draw[rotate=144]
(-1.01,0)to[out=0,in=180,looseness=1](-0.58,0);
\draw[rotate=-144]
(-1.01,0)to[out=0,in=180,looseness=1](-0.58,0);
\draw
(1.52,0)to[out=0,in=180,looseness=1](2.1,0);
\draw[rotate=72]
(1.52,0)to[out=0,in=180,looseness=1](1.67,0);
\draw[rotate=-72]
(1.52,0)to[out=0,in=180,looseness=1](1.67,0);
\draw[rotate=144]
(1.52,0)to[out=0,in=180,looseness=1](2.35,0);
\draw[rotate=-144]
(1.52,0)to[out=0,in=180,looseness=1](2.35,0);
\draw
(1.52,0)to[out=-120,in=24,looseness=1](0.81,-0.59)
(1.52,0)to[out=120,in=-24,looseness=1](0.81,0.59);
\draw[ rotate=72]
(1.52,0)to[out=-120,in=24,looseness=1](0.81,-0.59)
(1.52,0)to[out=120,in=-24,looseness=1](0.81,0.59);
\draw[rotate=-72]
(1.52,0)to[out=-120,in=24,looseness=1](0.81,-0.59)
(1.52,0)to[out=120,in=-24,looseness=1](0.81,0.59);
\draw[rotate=144]
(1.52,0)to[out=-120,in=24,looseness=1](0.81,-0.59)
(1.52,0)to[out=120,in=-24,looseness=1](0.81,0.59);
\draw[rotate=-144]
(1.52,0)to[out=-120,in=24,looseness=1](0.81,-0.59)
(1.52,0)to[out=120,in=-24,looseness=1](0.81,0.59);
\draw[color=black, scale=1, shift={(0.1,0)}]
(-2,1.6)to[out=-90,in=90,looseness=1](-2,-1.6)
(2,1.6)to[out=-90,in=90,looseness=1](2,-1.6)
(-2,1.6)to[out=0,in=180,looseness=1](2,1.6)
(-2,-1.6)to[out=0,in=180,looseness=1](2,-1.6);
\path
(-2,-2) node[right]{\text{$5$--petal flower}}
(-2,-2.5) node[right]{$\Theta\approx3.907$};
\end{tikzpicture}

\eject

\noindent{\Large{9 regions:}}

\medskip

\noindent
\begin{tikzpicture}[scale=0.9]
\draw[color=black, scale=1]
(-2,1.6)to[out=-90,in=90,looseness=1](-2,-1.6)
(2,1.6)to[out=-90,in=90,looseness=1](2,-1.6)
(-2,1.6)to[out=0,in=180,looseness=1](2,1.6)
(-2,-1.6)to[out=0,in=180,looseness=1](2,-1.6);
\draw
(-0.43,0)to[out=0,in=180,looseness=1](0,0)
(-0.6,0.365)to[out=50,in=-170,looseness=1](-0.02,0.705)
(-0.43,0)to[out=-120,in=70,looseness=1](-0.6,-0.365)
(-0.43,0)to[out=120,in=-70,looseness=1](-0.6,0.365)
(-0.975,-0.425) to[out=10,in=-170,looseness=1](-0.6,-0.365)
(-0.975,0.425) to[out=-10,in=170,looseness=1](-0.6,0.365)
(-1.1,0) to[out=-90,in=130,looseness=1](-0.975,-0.425)
(-1.1,0) to[out=90,in=-130,looseness=1](-0.975,0.425)
(-0.975,0.425) to[out=110,in=-150,looseness=1](-0.78,1.36)
(-0.975,-0.425) to[out=-110,in=150,looseness=1](-0.78,-1.36);
\draw[rotate=-120]
(-0.43,0)to[out=0,in=180,looseness=1](0,0)
(-0.6,0.365)to[out=50,in=-170,looseness=1](-0.02,0.705)
(-0.43,0)to[out=-120,in=70,looseness=1](-0.6,-0.365)
(-0.43,0)to[out=120,in=-70,looseness=1](-0.6,0.365)
(-0.975,-0.425) to[out=10,in=-170,looseness=1](-0.6,-0.365)
(-0.975,0.425) to[out=-10,in=170,looseness=1](-0.6,0.365)
(-1.1,0) to[out=-90,in=130,looseness=1](-0.975,-0.425)
(-1.1,0) to[out=90,in=-130,looseness=1](-0.975,0.425)
(-0.975,0.425) to[out=110,in=-150,looseness=1](-0.78,1.36)
(-0.975,-0.425) to[out=-110,in=150,looseness=1](-0.78,-1.36);
\draw[rotate=120]
(-0.43,0)to[out=0,in=180,looseness=1](0,0)
(-0.6,0.365)to[out=50,in=-170,looseness=1](-0.02,0.705)
(-0.43,0)to[out=-120,in=70,looseness=1](-0.6,-0.365)
(-0.43,0)to[out=120,in=-70,looseness=1](-0.6,0.365)
(-0.975,-0.425) to[out=10,in=-170,looseness=1](-0.6,-0.365)
(-0.975,0.425) to[out=-10,in=170,looseness=1](-0.6,0.365)
(-1.1,0) to[out=-90,in=130,looseness=1](-0.975,-0.425)
(-1.1,0) to[out=90,in=-130,looseness=1](-0.975,0.425)
(-0.975,0.425) to[out=110,in=-150,looseness=1](-0.78,1.36)
(-0.975,-0.425) to[out=-110,in=150,looseness=1](-0.78,-1.36);
\path
(-2,-2) node[right]{\text{$9$--floc}}
(-2,-2.5) node[right]{$\Theta\approx4.194$};
\end{tikzpicture}\qquad
\begin{tikzpicture}[scale=0.9]

\draw[color=black, scale=1, shift={(0.1,0)}]
(-2,1.6)to[out=-90,in=90,looseness=1](-2,-1.6)
(2,1.6)to[out=-90,in=90,looseness=1](2,-1.6)
(-2,1.6)to[out=0,in=180,looseness=1](2,1.6)
(-2,-1.6)to[out=0,in=180,looseness=1](2,-1.6);
\draw
(-0.78,1.36)--(-0.92,1.59)
(-0.4,0)to[out=0,in=180,looseness=1](0,0)
(-0.6,0.53)to[out=50,in=-150,looseness=1](-0.4,0.7)
(-0.6,-0.53)to[out=-50,in=150,looseness=1](-0.4,-0.7)
(-0.4,0)to[out=-120,in=70,looseness=1](-0.6,-0.53)
(-0.4,0)to[out=120,in=-70,looseness=1](-0.6,0.53)
(-0.8,-0.59) to[out=10,in=-170,looseness=1](-0.6,-0.53)
(-0.8,0.59) to[out=-10,in=170,looseness=1](-0.6,0.53)
(-0.8,0.59) to[out=110,in=-120,looseness=1](-0.78,1.36)
(-0.8,-0.59) to[out=-110,in=120,looseness=1](-0.78,-1.36)
(-1,0) to[out=-90,in=130,looseness=1](-0.8,-0.59)
(-1,0) to[out=90,in=-130,looseness=1](-0.8,0.59);
\draw[rotate=-120]
(-0.78,1.36)--(-1.07,1.85)
(-0.4,0)to[out=0,in=180,looseness=1](0,0)
(-0.6,0.53)to[out=50,in=-150,looseness=1](-0.4,0.7)
(-0.6,-0.53)to[out=-50,in=150,looseness=1](-0.4,-0.7)
(-0.4,0)to[out=-120,in=70,looseness=1](-0.6,-0.53)
(-0.4,0)to[out=120,in=-70,looseness=1](-0.6,0.53)
(-0.8,-0.59) to[out=10,in=-170,looseness=1](-0.6,-0.53)
(-0.8,0.59) to[out=-10,in=170,looseness=1](-0.6,0.53)
(-0.8,0.59) to[out=110,in=-120,looseness=1](-0.78,1.36)
(-0.8,-0.59) to[out=-110,in=120,looseness=1](-0.78,-1.36)
(-1,0) to[out=-90,in=130,looseness=1](-0.8,-0.59)
(-1,0) to[out=90,in=-130,looseness=1](-0.8,0.59);
\draw[rotate=120]
(-0.78,1.36)--(-0.92,1.59)
(-0.4,0)to[out=0,in=180,looseness=1](0,0)
(-0.6,0.53)to[out=50,in=-150,looseness=1](-0.4,0.7)
(-0.6,-0.53)to[out=-50,in=150,looseness=1](-0.4,-0.7)
(-0.4,0)to[out=-120,in=70,looseness=1](-0.6,-0.53)
(-0.4,0)to[out=120,in=-70,looseness=1](-0.6,0.53)
(-0.8,-0.59) to[out=10,in=-170,looseness=1](-0.6,-0.53)
(-0.8,0.59) to[out=-10,in=170,looseness=1](-0.6,0.53)
(-0.8,0.59) to[out=110,in=-120,looseness=1](-0.78,1.36)
(-0.8,-0.59) to[out=-110,in=120,looseness=1](-0.78,-1.36)
(-1,0) to[out=-90,in=130,looseness=1](-0.8,-0.59)
(-1,0) to[out=90,in=-130,looseness=1](-0.8,0.59);
\path
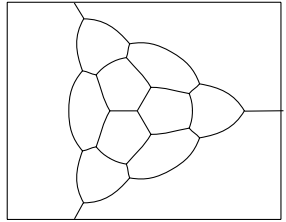
(-2,-2) node[right]{\text{$3$--ray $9$--floc}}
(-2,-2.5) node[right]{$\Theta\approx4.321$};
\end{tikzpicture}

\bigskip

\bigskip

\noindent{\Large{Non--embedded regular shrinkers:}}

\bigskip

\noindent
\begin{tikzpicture}[scale=0.25, rotate=-54]
\draw
(2.2,3)to[out=0,in=90,looseness=1] (5.43,0)
(-3.035,0)to[out=-60,in=180,looseness=1] (2.2,-3)
(2.2,-3)to[out=0,in=-90,looseness=1] (5.43,0);
\draw[rotate=108]
(-3.035,0)to[out=60,in=180,looseness=1] (2.2,3)
(2.2,3)to[out=0,in=90,looseness=1] (5.43,0)
(2.2,-3)to[out=0,in=-90,looseness=1] (5.43,0);
\draw[color=black, rotate=234, scale=0.43]
(-8.7,0)to[out=0,in=180,looseness=1](16.635,0);
\draw[color=black, scale=3.6, shift={(0,0)}, rotate=54]
(-2,1.6)to[out=-90,in=90,looseness=1](-2,-1.6)
(2,1.6)to[out=-90,in=90,looseness=1](2,-1.6)
(-2,1.6)to[out=0,in=180,looseness=1](2,1.6)
(-2,-1.6)to[out=0,in=180,looseness=1](2,-1.6);
\path[rotate=54, scale=3.6]
(-2,-2) node[right]{\text{Antispoon}}
(-2,-2.5) node[right]{$\Theta\approx2.365$};
\end{tikzpicture}\qquad
\begin{tikzpicture}[scale=0.9]
\draw[color=black, scale=1, shift={(0,0)}]
(-2,1.6)to[out=-90,in=90,looseness=1](-2,-1.6)
(2,1.6)to[out=-90,in=90,looseness=1](2,-1.6)
(-2,1.6)to[out=0,in=180,looseness=1](2,1.6)
(-2,-1.6)to[out=0,in=180,looseness=1](2,-1.6);
\draw[color=black, scale=1, shift={(0,0)}]
(0,0.05)to[out=0,in=160,looseness=1](1.6,-0.13)
(1.6,-0.13)to[out=-20,in=90,looseness=1](1.88,-0.43)
(0,0.05)to[out=180,in=20,looseness=1](-1.6,-0.13)
(-1.6,-0.13)to[out=-160,in=90,looseness=1](-1.88,-0.43)
(0,0)to[out=90,in=-90,looseness=1](0,1.6)
(0,0)to[out=-30,in=180,looseness=1](1.6,-0.7)
(1.6,-0.7)to[out=0,in=-90,looseness=1](1.88,-0.43)
(0,0)to[out=-150,in=0,looseness=1](-1.6,-0.7)
(-1.6,-0.7)to[out=180,in=-90,looseness=1](-1.88,-0.43);
\path
(-2,-2) node[right]{\text{Bowtie}}
(-2,-2.5) node[right]{$\Theta\approx2.503$};
\end{tikzpicture}

\bigskip

\bigskip

\noindent{\Large{Impossible regular shrinkers:}}

\bigskip

\noindent
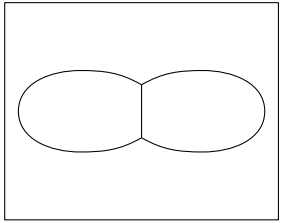
\begin{tikzpicture}[scale=0.9]
\draw[color=black, scale=1, shift={(0,0)}]
(-2,1.6)to[out=-90,in=90,looseness=1](-2,-1.6)
(2,1.6)to[out=-90,in=90,looseness=1](2,-1.6)
(-2,1.6)to[out=0,in=180,looseness=1](2,1.6)
(-2,-1.6)to[out=0,in=180,looseness=1](2,-1.6);
\draw[color=black, scale=0.3]
(0,1.3)to[out=30,in=180,looseness=1] (3,2)
(3,2)to[out=0,in=90,looseness=1] (6,0)
(0,-1.3)to[out=-30,in=180,looseness=1] (3,-2)
(3,-2)to[out=0,in=-90,looseness=1] (6,0);
\draw[color=black, rotate=180, scale=0.3]
(0,1.3)to[out=30,in=180,looseness=1] (3,2)
(3,2)to[out=0,in=90,looseness=1] (6,0)
(0,-1.3)to[out=-30,in=180,looseness=1] (3,-2)
(3,-2)to[out=0,in=-90,looseness=1] (6,0);
\draw[color=black, scale=0.3]
(0,1.3)to[out=-90,in=90,looseness=1](0,-1.3);
\end{tikzpicture}\qquad
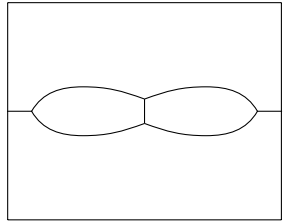
\begin{tikzpicture}[scale=0.9]
\draw[color=black, scale=1, shift={(0,0)}]
(-2,1.6)to[out=-90,in=90,looseness=1](-2,-1.6)
(2,1.6)to[out=-90,in=90,looseness=1](2,-1.6)
(-2,1.6)to[out=0,in=180,looseness=1](2,1.6)
(-2,-1.6)to[out=0,in=180,looseness=1](2,-1.6);
\draw[color=black, scale=0.3]
(0,0.6)to[out=20,in=180,looseness=1] (3,1.2)
(3,1.2)to[out=0,in=120,looseness=1] (5.5,0)
(0,-0.6)to[out=-20,in=180,looseness=1] (3,-1.2)
(3,-1.2)to[out=0,in=-120,looseness=1] (5.5,0)
(5.5,0)to[out=0,in=180,looseness=1] (6.7,0);
\draw[color=black, scale=0.3, rotate=180]
(0,0.6)to[out=20,in=180,looseness=1] (3,1.2)
(3,1.2)to[out=0,in=120,looseness=1] (5.5,0)
(0,-0.6)to[out=-20,in=180,looseness=1] (3,-1.2)
(3,-1.2)to[out=0,in=-120,looseness=1] (5.5,0)
(5.5,0)to[out=0,in=180,looseness=1] (6.7,0);
\draw[color=black, scale=0.3]
(0,0.6)to[out=-90,in=90,looseness=1](0,-0.6);
\end{tikzpicture}\qquad
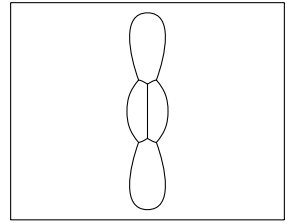
\begin{tikzpicture}[scale=0.9]
\draw[color=black, scale=1, shift={(0,0)}]
(-2,1.6)to[out=-90,in=90,looseness=1](-2,-1.6)
(2,1.6)to[out=-90,in=90,looseness=1](2,-1.6)
(-2,1.6)to[out=0,in=180,looseness=1](2,1.6)
(-2,-1.6)to[out=0,in=180,looseness=1](2,-1.6);
\draw[color=black, scale=1]
(0,0.4)to[out=150,in=-10,looseness=1](-0.13,0.46)
(0,0.4)to[out=30,in=-170,looseness=1](0.13,0.46)
(0,-0.4)to[out=-150,in=10,looseness=1](-0.13,-0.46)
(0,-0.4)to[out=-30,in=170,looseness=1](0.13,-0.46)
(0.13,0.46)to[out=-50,in=90,looseness=1](0.3,0)
(-0.13,0.46)to[out=-130,in=90,looseness=1](-0.3,0)
(0.13,-0.46)to[out=50,in=-90,looseness=1](0.3,0)
(-0.13,-0.46)to[out=130,in=-90,looseness=1](-0.3,0)
(0.13,0.46)to[out=70,in=0,looseness=1](0,1.45)
(-0.13,0.46)to[out=110,in=180,looseness=1](0,1.45)
(0.13,-0.46)to[out=-70,in=0,looseness=1](0,-1.45)
(-0.13,-0.46)to[out=-110,in=180,looseness=1](0,-1.45)
(0,0.4)to[out=-90,in=90,looseness=1](0,0)
(0,-0.4)to[out=90,in=-90,looseness=1](0,0);
\end{tikzpicture}\qquad
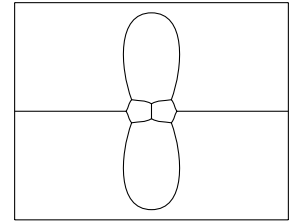
\begin{tikzpicture}[scale=0.9]
\draw[color=black, scale=1, shift={(0,0)}]
(-2,1.6)to[out=-90,in=90,looseness=1](-2,-1.6)
(2,1.6)to[out=-90,in=90,looseness=1](2,-1.6)
(-2,1.6)to[out=0,in=180,looseness=1](2,1.6)
(-2,-1.6)to[out=0,in=180,looseness=1](2,-1.6);
\draw[color=black, scale=1]
(0,0.11)to[out=150,in=-10,looseness=1](-0.29,0.17)
(0,0.11)to[out=30,in=-170,looseness=1](0.29,0.17)
(0,-0.11)to[out=-150,in=10,looseness=1](-0.29,-0.17)
(0,-0.11)to[out=-30,in=170,looseness=1](0.29,-0.17)
(0.29,0.17)to[out=-50,in=120,looseness=1](0.37,0)
(-0.29,0.17)to[out=-130,in=60,looseness=1](-0.37,0)
(0.29,-0.17)to[out=50,in=-120,looseness=1](0.37,0)
(-0.29,-0.17)to[out=130,in=-60,looseness=1](-0.37,0)
(0.37,0)to[out=0,in=180,looseness=1](2,0)
(-0.37,0)to[out=180,in=0,looseness=1](-2,0)
(0.29,0.17)to[out=70,in=0,looseness=1](0,1.45)
(-0.29,0.17)to[out=110,in=180,looseness=1](0,1.45)
(0.29,-0.17)to[out=-70,in=0,looseness=1](0,-1.45)
(-0.29,-0.17)to[out=-110,in=180,looseness=1](0,-1.45)
(0,0.11)to[out=-90,in=90,looseness=1](0,0)
(0,-0.11)to[out=90,in=-90,looseness=1](0,0);
\end{tikzpicture}

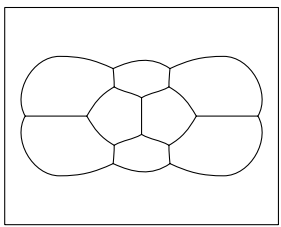
\bigskip

\noindent
\begin{tikzpicture}[scale=0.9]
\draw[color=black, scale=1, shift={(0,0)}]
(-2,1.6)to[out=-90,in=90,looseness=1](-2,-1.6)
(2,1.6)to[out=-90,in=90,looseness=1](2,-1.6)
(-2,1.6)to[out=0,in=180,looseness=1](2,1.6)
(-2,-1.6)to[out=0,in=180,looseness=1](2,-1.6);
\draw[color=black, scale=1, shift={(0,0)}]
(0,0.27)to[out=30,in=-150,looseness=1](0.4,0.43)
(0,-0.27)to[out=-30,in=150,looseness=1](0.4,-0.43)
(0,0.27)to[out=150,in=-30,looseness=1](-0.4,0.43)
(0,-0.27)to[out=-150,in=30,looseness=1](-0.4,-0.43)
(0.4,0.43)to[out=-30,in=120,looseness=1](0.8,0)
(0.4,-0.43)to[out=30,in=-120,looseness=1](0.8,0)
(-0.4,0.43)to[out=-150,in=60,looseness=1](-0.8,0)
(-0.4,-0.43)to[out=150,in=-60,looseness=1](-0.8,0)
(0.4,0.43)to[out=90,in=-95,looseness=1](0.415,0.7)
(0.4,-0.43)to[out=-90,in=95,looseness=1](0.415,-0.7)
(-0.4,0.43)to[out=90,in=-85,looseness=1](-0.415,0.7)
(-0.4,-0.43)to[out=-90,in=85,looseness=1](-0.415,-0.7)
(0.415,0.7)to[out=25,in=180,looseness=1](1.2,0.88)
(1.2,0.88)to[out=0,in=60,looseness=1](1.7,0)
(0.415,-0.7)to[out=- 25,in=180,looseness=1](1.2,-0.88)
(1.2,-0.88)to[out=0,in=-60,looseness=1](1.7,0)
(-0.415,0.7)to[out=155,in=0,looseness=1](-1.2,0.88)
(-1.2,0.88)to[out=180,in=120,looseness=1](-1.7,0)
(-0.415,-0.7)to[out=-155,in=0,looseness=1](-1.2,-0.88)
(-1.2,-0.88)to[out=180,in=-120,looseness=1](-1.7,0)
(-0.415,0.7) to[out=25,in=145,looseness=1] (0.415,0.7)
(-0.415,-0.7) to[out=-25,in=-145,looseness=1] (0.415,-0.7)
(0,0.27)to[out=-90,in=90,looseness=1](0,-0.27)
(-1.7,0)to[out=0,in=180,looseness=1](-0.8,0)
(0.8,0)to[out=0,in=180,looseness=1](1.7,0);
\end{tikzpicture}\qquad
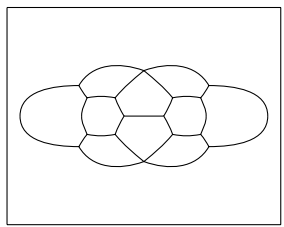
\begin{tikzpicture}[scale=0.9]
\draw[color=black, scale=1, shift={(0,0)}]
(-2,1.6)to[out=-90,in=90,looseness=1](-2,-1.6)
(2,1.6)to[out=-90,in=90,looseness=1](2,-1.6)
(-2,1.6)to[out=0,in=180,looseness=1](2,1.6)
(-2,-1.6)to[out=0,in=180,looseness=1](2,-1.6);
\draw[color=black, scale=1, shift={(0,0)}]
(0.95,0.45)to[out=0,in=90,looseness=1](1.81,0)
(-0.95,0.45)to[out=180,in=90,looseness=1](-1.81,0)
(0.95,0.45)to[out=120,in=20,looseness=1](0,0.67)
(-0.95,0.45)to[out=60,in=160,looseness=1](0,0.67)
(0.83,0.27)to[out=50,in=-120,looseness=1](0.95,0.45)
(-0.83,0.27)to[out=130,in=-60,looseness=1](-0.95,0.45)
(0.83,0.27)to[out=-70,in=90,looseness=1](0.91,0)
(-0.83,0.27)to[out=-110,in=90,looseness=1](-0.91,0)
(-0.29,0)to[out=120,in=-70,looseness=1](-0.42,0.27)
(0.29,0)to[out=60,in=-110,looseness=1](0.42,0.27)
(0.42,0.27) to[out=10,in=170,looseness=1](0.83,0.27)
(-0.42,0.27) to[out=170,in=10,looseness=1](-0.83,0.27)
(0.42,0.27) to[out=130,in=-40,looseness=1](0,0.67)
(-0.42,0.27) to[out=50,in=-140,looseness=1](0,0.67)
(-0.29,0)to[out=0,in=180,looseness=1](0.29,0);
\draw[color=black, scale=1, shift={(0,0)}, rotate=180]
(0.95,0.45)to[out=0,in=90,looseness=1](1.81,0)
(-0.95,0.45)to[out=180,in=90,looseness=1](-1.81,0)
(0.95,0.45)to[out=120,in=20,looseness=1](0,0.67)
(-0.95,0.45)to[out=60,in=160,looseness=1](0,0.67)
(0.83,0.27)to[out=50,in=-120,looseness=1](0.95,0.45)
(-0.83,0.27)to[out=130,in=-60,looseness=1](-0.95,0.45)
(0.83,0.27)to[out=-70,in=90,looseness=1](0.91,0)
(-0.83,0.27)to[out=-110,in=90,looseness=1](-0.91,0)
(-0.29,0)to[out=120,in=-70,looseness=1](-0.42,0.27)
(0.29,0)to[out=60,in=-110,looseness=1](0.42,0.27)
(0.42,0.27) to[out=10,in=170,looseness=1](0.83,0.27)
(-0.42,0.27) to[out=170,in=10,looseness=1](-0.83,0.27)
(0.42,0.27) to[out=130,in=-40,looseness=1](0,0.67)
(-0.42,0.27) to[out=50,in=-140,looseness=1](0,0.67);
\end{tikzpicture}

\bigskip

\bigskip

\noindent Conjecturally, by numerical evidence
in~\cite{haettenschweiler}, there are no regular shrinkers with these
topological shapes. The only one whose non--existence is rigorously proved
is the first one, the $\Theta$--shaped (double cell) shrinker,
in~\cite{balhausman}.

\eject

\bibliographystyle{amsplain}
\bibliography{biblio}

@article {BeChKh,
    AUTHOR = {G. Bellettini and A. Chambolle and S. Kholmatov},
     TITLE = {Minimizing movements for forced anisotropic mean curvature
              flow of partitions with mobilities},
   JOURNAL = {Proc. Roy. Soc. Edinburgh Sect. A},
    VOLUME = {151},
      YEAR = {2021},
    NUMBER = {4},
     PAGES = {1135-1170},
}

@Unpublished{applicato4,
    author =      {T. Kagaya and M. Mizuno and K. Takasao},
    title =      {Long time behavior for a curvature flow of networks related to grain bundary motion with the effect of lattice misoriantations},
    note = {ArXiv Preprint Server -- http:/$\!\!$/$\!$arxiv.org},
    OPTkey =      {},
    OPTmonth =      {},
    year =      {2021},
    OPTannote =      {}
}

@article {applicato2,
    AUTHOR = {Y. Epshteyn and C. Liu and M. Mizuno},
     TITLE = {Motion of grain boundaries with dynamic lattice
              misorientations and with triple junctions drag},
   JOURNAL = {SIAM J. Math. Anal.},
    VOLUME = {53},
      YEAR = {2021},
    NUMBER = {3},
     PAGES = {3072-3097},
}

@Unpublished{Chang,
author = {J. Chang},
title ={Stability of regular shrinkers in the network flow},
note = {ArXiv Preprint Server -- http:/$\!\!$/$\!$arxiv.org},
year = {2021},
}

@article {applicato3,
    AUTHOR = {Y. Epshteyn and C. Liu and M. Mizuno},
     TITLE = {Large time asymptotic behavior of grain boundaries motion with
              dynamic lattice misorientations and with triple junctions
              drag},
   JOURNAL = {Commun. Math. Sci.},
    VOLUME = {19},
      YEAR = {2021},
    NUMBER = {5},
     PAGES = {1403-428},
}

@Unpublished{gronovpoz,
author = {M. G{\"{o}}{\ss}wein and M. Novaga and P. Pozzi},
title ={Stability analysis for the anisotropic curve shortening flow of planar networks},
note = {ArXiv Preprint Server -- http:/$\!\!$/$\!$arxiv.org},
year = {2023}
}

@article {kronovpoz,
    AUTHOR = {H. Kr\"{o}ner and M. Novaga and P. Pozzi},
     TITLE = {Anisotropic curvature flow of immersed networks},
   JOURNAL = {Milan J. Math.},
    VOLUME = {89},
      YEAR = {2021},
    NUMBER = {1},
     PAGES = {147-186},
}

@Article{BeNo,
 author = "G. Bellettini and M. Novaga",
 title = "Curvature evolution of nonconvex lens--shaped domains",
 OPTcrossref = "",
 OPTkey = "",
 journal = "J. Reine Angew. Math.",
 year = "2011",
 volume = "656",
 OPTnumber = "",
 pages = "17-46",
 OPTmonth = "",
 OPTnote = "",
 OPTannote = ""
}

@article{EsOt:14,
 AUTHOR =	 {S. Esedoglu and F. Otto},
 TITLE =	 {Threshold Dynamics for Networks with Arbitrary Surface Tensions},
 JOURNAL =	 {Comm. Pure Appl. Math.},
 VOLUME =	 {68},
 YEAR =	 {2015},
 NUMBER =	 {5},
 PAGES =	 {808-864}
}

@Unpublished{FiHeLaSi,
	author = 	 {J. Fischer and S. Hensel and T. Laux and T. Simon},
	title = 	 {The local structure of the energy landscape in multiphase mean curvature flow:
	weak{--}strong uniqueness and stability of evolutions},
	note = 	 {ArXiv Preprint Server -- http:/$\!\!$/$\!$arxiv.org},
	OPTkey = 	 {},
	OPTmonth = 	 {},
	year = 	 {2020},
	OPTannote = 	 {}
}

@article {GoMePl,
 AUTHOR = {M. G{\"{o}}{\ss}wein and J. Menzel and A. Pluda},
 TITLE = {Existence and uniqueness of the motion by curvature of regular
 networks},
 JOURNAL = {Interfaces Free Bound.},
 VOLUME = {25},
 YEAR = {2023},
 NUMBER = {1},
 PAGES = {109-154},
}

@Article{Il:93g,
 author =	 "T. Ilmanen",
 title =	 "Convergence of the {A}llen--{C}ahn equation to
 {B}rakke's motion by mean curvature",
 OPTcrossref =	 "",
 OPTkey =	 "",
 journal =	 "J. Differential Geom.",
 year =	 "1993",
 volume =	 "38",
 OPTnumber =	 "",
 pages =	 "417-461",
 OPTmonth =	 "",
 OPTnote =	 "",
 OPTannote =	 ""
}

@Article{Ilnevsch,
 author = {T. Ilmanen and A. Neves and F. Schulze},
 TITLE = {On short time existence for the planar network flow},
 JOURNAL = {J. Differential Geom.},
 VOLUME = {111},
 YEAR = {2019},
 NUMBER = {1},
 PAGES = {39-89},
}

@article {KimTonegawa2,
 AUTHOR = {L. Kim and Y. Tonegawa},
 TITLE = {Existence and regularity theorems of one{--}dimensional {B}rakke
 flows},
 JOURNAL = {Interfaces Free Bound.},
 VOLUME = {22},
 YEAR = {2020},
 NUMBER = {4},
 PAGES = {505-550},
}

@article {LiMazPlSa,
 AUTHOR = {J. Lira and R. Mazzeo and A. Pluda and M. S{\'{a}}ez},
 TITLE = {Short{--}time existence for the network flow},
 JOURNAL = {Comm. Pure Appl. Math.},
 VOLUME = {76},
 YEAR = {2023},
 NUMBER = {12},
 PAGES = {3968-4021},
}

@article{MMN13,
	author = {A. Magni and C. Mantegazza and M. Novaga},
	title = {Motion by curvature of planar networks {II}},
	JOURNAL = {Ann. Sc. Norm. Sup. Pisa},
 VOLUME = {15},
 YEAR = {2016},
 PAGES = {117-144},
}

@book{Manlib,
	author = {C. Mantegazza},
	TITLE = {Lecture notes on mean curvature flow},
 SERIES = {Progress in Mathematics},
 VOLUME = {290},
 PUBLISHER = {Birkh\"auser/Springer Basel AG, Basel},
 YEAR = {2011}
}

@Unpublished{NoSc23,
 author = {M. Novaga and L. Sciaraffia},
 title = {Singularities of the network flow with symmetric initial data},
 OPTkey = {},
 OPTmonth = {},
 year = {2023},
 OPTannote = {},
 note = {ArXiv Preprint Server -- http:/$\!\!$/$\!$arxiv.org},
}

@Unpublished{PlPo22A,
 AUTHOR = {A. Pluda and M. Pozzetta},
 TITLE = {Lojasiewicz{--S}imon inequalities for minimal networks: stability and convergence},
 YEAR = {2023},
 note = {ArXiv Preprint Server -- http:/$\!\!$/$\!$arxiv.org, to appear on Math. Ann.},
}

@Unpublished{PlPo22A-2,
 AUTHOR = {A. Pluda and M. Pozzetta},
 TITLE = {On the uniqueness of nondegenerate blowups for the motion by curvature of networks},
 YEAR = {2022},
 note = {ArXiv Preprint Server -- http:/$\!\!$/$\!$arxiv.org},
}

@book {Prusssimonett,
 AUTHOR = {J. Pr\"{u}ss and G. Simonett},
 TITLE = {Moving interfaces and quasilinear parabolic evolution
 equations},
 SERIES = {Monographs in Mathematics},
 VOLUME = {105},
 PUBLISHER = {Birkh\"{a}user/Springer},
 YEAR = {2016},
}

@Unpublished{SalvoToni,
	author = {S. Stuvard and Y. Tonegawa},
	title ={On the existence of canonical multi{--}phase Brakke flows},
	YEAR = {2022},
	note = {ArXiv Preprint Server -- http:/$\!\!$/$\!$arxiv.org, to appear on Adv. Calc. Var}, 
}

@book {Triebel,
 AUTHOR = {Triebel, H.},
 TITLE = {Interpolation {T}heory, {F}unction {S}paces, {D}ifferential {O}perators},
 SERIES = {North{--}Holland Mathematical Library},
 VOLUME = {18},
 PUBLISHER = {North{--}Holland Publishing Co., Amsterdam{--}New York},
 YEAR = {1978},
 PAGES = {528}
}

@article{ablang1,
 AUTHOR = {U. Abresch and J. Langer},
 TITLE = {The normalized curve shortening flow and homothetic solutions},
 JOURNAL = {J. Differential Geom.},
 VOLUME = {23},
 YEAR = {1986},
 NUMBER = {2},
 PAGES = {175-196}
}

@Article{altawa,
 author = 	 "F. J. Almgren and J. E. Taylor and L. Wang",
 title = 	 "Curvature driven flows: a variational approach",
 OPTcrossref = "",
 OPTkey = 	 "",
 journal = 	 "SIAM J. Control Opt.",
 year = 	 "1993",
 volume = 	 "31",
 OPTnumber = 	 "",
 pages = 	 "387-438",
 OPTmonth = 	 "",
 OPTnote = 	 "",
 OPTannote = 	 ""
}

@Article{altgra,
 author = 	 "S. J. Altschuler and M. Grayson",
 title = 	 "Shortening space curves and flow through singularities",
 OPTcrossref = "",
 OPTkey = 	 "",
 journal = 	 "J. Differential Geom.",
 year = 	 "1992",
 volume = 	 "35",
 OPTnumber = 	 "",
 pages = 	 "283-298",
 OPTmonth = 	 "",
 OPTnote = 	 "",
 OPTannote = 	 ""
}

@article{altsch,
 AUTHOR = {S. J. Altschuler},
 TITLE = {Singularities of the curve shrinking flow for space curves},
 JOURNAL = {J. Differential Geom.},
 VOLUME = {34},
 YEAR = {1991},
 NUMBER = {2},
 PAGES = {491-514}
}

@Article{angen1,
 author = 	 {S. Angenent},
 title = 	 {On the formation of singularities in the curve shortening flow},
 journal = 	 {J. Differential Geom.},
 year = 	 {1991},
 OPTkey = 	 {},
 volume = 	 {33},
 OPTnumber = 	 {},
 pages = 	 {601-633},
 OPTmonth = 	 {},
 OPTnote = 	 {},
 OPTannote = 	 {}
}

@article{angen2,
 AUTHOR = {S. Angenent},
 TITLE = {Parabolic equations for curves on surfaces. {II}.
 {I}ntersections, blow--up and generalized solutions},
 JOURNAL = {Ann. of Math. (2)},
 VOLUME = {133},
 YEAR = {1991},
 NUMBER = {1},
 PAGES = {171-215},
}

@article{angen3,
 AUTHOR = {S. Angenent},
 TITLE = {Parabolic equations for curves on surfaces. {I}. {C}urves with
 {$p$--}integrable curvature},
 JOURNAL = {Ann. of Math. (2)},
 VOLUME = {132},
 YEAR = {1990},
 NUMBER = {3},
 PAGES = {451-483},
}

@article{angen5,
 AUTHOR = {S. Angenent},
 TITLE = {The zero set of a solution of a parabolic equation},
 JOURNAL = {J. Reine Angew. Math.},
 VOLUME = {390},
 YEAR = {1988},
 PAGES = {79-96},
}

@Article{balhausman,
 author = {P. Baldi and E. Haus and C. Mantegazza},
 title ={{Non--existence of Theta--shaped self--similarly} shrinking
	networks moving by curvature},
 JOURNAL = {Comm. Partial Differential Equations},
 VOLUME = {43},
 YEAR = {2018},
 NUMBER = {3},
 PAGES = {403-427},
}

@Article{balhausman2,
 author = {P. Baldi and E. Haus and C. Mantegazza},
 title ={Networks {self--similarly} moving by curvature with two triple
	junctions},
 JOURNAL = {Atti Accad. Naz. Lincei Rend. Lincei Mat. Appl.},
 VOLUME = {28},
 YEAR = {2017},
 NUMBER = {2},
 PAGES = {323-338},
 }

@article{balhausman3,
 author = {P. Baldi and E. Haus and C. Mantegazza},
 TITLE = {On the classification of networks {self--similarly} moving by
 curvature},
 JOURNAL = {Geom. Flows},
 VOLUME = {2},
 YEAR = {2017},
 PAGES = {125-137},
}

@Article{bellettinikho,
	author = {G. Bellettini and S. Y. Kholmatov},
	TITLE = {Minimizing movements for mean curvature flow of partitions},
 JOURNAL = {SIAM J. Math. Anal.},
 VOLUME = {50},
 YEAR = {2018},
 NUMBER = {4},
 PAGES = {4117-4148},
}

@Book{brakke,
 author = 	 "K. A. Brakke",
 title = 	 "The motion of a surface by its mean curvature",
 publisher = 	 "Princeton University Press",
 year = 	 "1978",
 OPTcrossref = "",
 OPTkey = 	 "",
 OPTeditor = 	 "",
 OPTvolume = 	 "",
 OPTnumber = 	 "20",
 OPTseries = 	 "Math. Notes",
 address = 	 "NJ",
 OPTedition = 	 "",
 OPTmonth = 	 "",
 OPTnote = 	 ""
}

@article{bronsard,
 AUTHOR = {L. Bronsard and F. Reitich},
 TITLE = {On three{--}phase boundary motion and the singular limit of a
 vector{--}valued {G}inzburg{--L}andau equation},
 JOURNAL = {Arch. Rat. Mech. Anal.},
 VOLUME = {124},
 YEAR = {1993},
 NUMBER = {4},
 PAGES = {355-379},
}

@PhdThesis{caraballo1,
 author = 	 {D. G. Caraballo},
 title = 	 {A variational scheme for the evolution of polycrystals by curvature},
 school = 	 {Princeton University},
 year = 	 {1996},
 OPTkey = 	 {},
 OPTtype = 	 {},
 OPTaddress = 	 {},
 OPTmonth = 	 {},
 OPTnote = 	 {},
 OPTannote = 	 {}
}

@Article{cgg,
 author = 	 "Y. G. Chen and Y. Giga and S. Goto",
 title = 	 "Uniqueness and existence of viscosity solutions of
		 generalized mean curvature flow equations",
 OPTcrossref = "",
 OPTkey = 	 "",
 journal = 	 "J. Differential Geom.",
 year = 	 "1991",
 volume = 	 "33",
 OPTnumber = 	 "",
 pages = 	 "749-786",
 OPTmonth = 	 "",
 OPTnote = 	 "",
 OPTannote = 	 ""
}

@article{chen_guo,
 AUTHOR = {X. Chen and J. Guo},
 TITLE = {Motion by curvature of planar curves with end 
points moving freely on a line},
 JOURNAL = {Math. Ann.},
 VOLUME = {350},
 YEAR = {2011},
 NUMBER = {2},
 PAGES = {277-311},
}

@article{chenguo,
 AUTHOR = {X. Chen and J.-S. Guo},
 TITLE = {Self--similar solutions of a 2--{D} multiple--phase curvature
 flow},
 JOURNAL = {Phys. D},
 VOLUME = {229},
 YEAR = {2007},
 NUMBER = {1},
 PAGES = {22-34},
}

@article{chzh,
 AUTHOR = {K.-S. Chou and X.-P. Zhu},
 TITLE = {Shortening complete plane curves},
 JOURNAL = {J. Differential Geom.},
 VOLUME = {50},
 YEAR = {1998},
 NUMBER = {3},
 PAGES = {471-504},
}

@article{coldmin6,
 author = 	 {T. H. Colding and W. P. Minicozzi II},
 TITLE = {Generic mean curvature flow {I}: generic singularities},
 JOURNAL = {Ann. of Math. (2)},
 VOLUME = {175},
 YEAR = {2012},
 NUMBER = {2},
 PAGES = {755-833},
}

@article{degako,
	author = {D. Depner and H. Garcke and Y. Kohsaka},
	TITLE = {Mean curvature flow with triple junctions in higher space dimensions},
 JOURNAL = {Arch. Rat. Mech. Anal.},
 VOLUME = {211},
 YEAR = {2014},
 NUMBER = {1},
 PAGES = {301-334},
}

@Book{degio4,
 author = 	 "E. {De~Giorgi}",
 title = 	 "Barriers, boundaries and motion of manifolds",
 OPTcrossref = "",
 OPTkey = 	 "",
 publisher = 	 "Scuola Normale Superiore di Pisa",
 year = 	 "1995",
 OPTvolume = 	 "",
 OPTnumber = 	 "",
 OPTpages = 	 "",
 OPTmonth = 	 "",
 OPTnote = 	 "",
 OPTannote = 	 ""
}

@incollection {degio7,
 AUTHOR = {E. {De~Giorgi}},
 TITLE = {Motions of partitions},
 BOOKTITLE = {Variational methods for discontinuous structures (Como, 1994)},
 SERIES = {Progr. Nonlinear Differential Equations Appl.},
 VOLUME = {25},
 PAGES = {1-5},
 PUBLISHER = {Birkh\"auser},
 ADDRESS = {Basel},
 YEAR = {1996}
}

@book{dellame,
 AUTHOR = {C. Dellacherie and P.-A. Meyer},
 TITLE = {Probabilities and potential},
 SERIES = {North{--}Holland Mathematics Studies},
 VOLUME = {29},
 PUBLISHER = {North{--}Holland Publishing Co.},
 ADDRESS = {Amsterdam},
 YEAR = {1978},
}

@Book{eck1,
 author = {K. Ecker},
 title = {Regularity theory for mean curvature flow},
 SERIES = {Progress in Nonlinear Differential Equations and their
 Applications, 57},
 PUBLISHER = {Birkh\"auser Boston Inc.},
 ADDRESS = {Boston, MA},
 YEAR = {2004},
}

@Article{eckhui2,
 author = 	 "K. Ecker and G. Huisken",
 title = 	 "Interior estimates for hypersurfaces moving by mean curvature",
 OPTcrossref = "",
 OPTkey = 	 "",
 journal = 	 "Invent. Math.",
 year = 	 "1991",
 volume = 	 "105",
 number = 	 "3",
 pages = 	 "547-569",
 OPTmonth = 	 "",
 OPTnote = 	 "",
 OPTannote = 	 ""
}

@book{eidelman2,
 AUTHOR = {S. D. Eidelman and N. V. Zhitarashu},
 TITLE = {Parabolic boundary value problems},
 SERIES = {Operator Theory: Advances and Applications},
 VOLUME = {101},
 PUBLISHER = {Birkh\"auser Verlag, Basel},
 YEAR = {1998},
}

@Article{es,
 author = 	 "L. C. Evans and J. Spruck",
 title = 	 "Motion of level sets by mean curvature {I}",
 OPTcrossref = "",
 OPTkey = 	 "",
 journal = 	 "J. Differential Geom.",
 year = 	 "1991",
 volume = 	 "33",
 OPTnumber = 	 "",
 pages = 	 "635-681",
 OPTmonth = 	 "",
 OPTnote = 	 "",
 OPTannote = 	 ""
}

@article{freire2,
 AUTHOR = {A. Freire},
 TITLE = {Mean curvature motion of triple junctions of graphs in two
 dimensions},
 JOURNAL = {Comm. Partial Differential Equations},
 VOLUME = {35},
 YEAR = {2010},
 NUMBER = {2},
 PAGES = {302-327},
}

@article{freire3,
 AUTHOR = {A. Freire},
 TITLE = {Mean curvature motion of graphs with constant contact angle at a free boundary},
 JOURNAL = {Anal. P.D.E.},
 VOLUME = {3},
 YEAR = {2010},
 NUMBER = {4},
 PAGES = {359-407},
}

@Article{gage,
 author = 	 "M. Gage",
 title = 	 "Curve shortening makes convex curves circular",
 OPTcrossref = "",
 OPTkey = 	 "",
 journal = 	 "Invent. Math.",
 year = 	 "1984",
 volume = 	 "76",
 OPTnumber = 	 "",
 pages = 	 "357-364",
 OPTmonth = 	 "",
 OPTnote = 	 "",
 OPTannote = 	 ""
}

@article{gage0,
 AUTHOR = {M. Gage},
 TITLE = {An isoperimetric inequality with applications to curve
 shortening},
 JOURNAL = {Duke Math. J.},
 VOLUME = {50},
 YEAR = {1983},
 NUMBER = {4},
 PAGES = {1225-1229},
}

@Article{gaha1,
 author = 	 "M. Gage and R. S. Hamilton",
 title = 	 "The heat equation shrinking convex plane curves",
 OPTcrossref = "",
 OPTkey = 	 "",
 journal = 	 "J. Differential Geom.",
 year = 	 "1986",
 volume = 	 "23",
 OPTnumber = 	 "",
 pages = 	 "69-95",
 OPTmonth = 	 "",
 OPTnote = 	 "",
 OPTannote = 	 ""
}

@article{garkoh,
 AUTHOR = {H. Garcke and Y. Kohsaka and D. {\v{S}}ev{\v{c}}ovi{\v{c}}},
 TITLE = {Nonlinear stability of stationary solutions for curvature flow
 with triple function},
 JOURNAL = {Hokkaido Math. J.},
 VOLUME = {38},
 YEAR = {2009},
 NUMBER = {4},
 PAGES = {721-769},
}

@phdthesis{goesswein2019Dissertation,
	title={Surface Diffusion Flow of Triple Junction Clusters in Higher Space Dimensions},
	author={M. G{\"{o}}{\ss}wein},
	 school = {Universit\"{a}t Regensburg},
	year={2019}
}

@Article{gray1,
 author = 	 "M. A. Grayson",
 title = 	 "The heat equation shrinks embedded plane curves to round points",
 OPTcrossref = "",
 OPTkey = 	 "",
 journal = 	 "J. Differential Geom.",
 year = 	 "1987",
 volume = 	 "26",
 OPTnumber = 	 "",
 pages = 	 "285-314",
 OPTmonth = 	 "",
 OPTnote = 	 "",
 OPTannote = 	 ""
}

@Book{gurtin2,
 AUTHOR = "M. E. Gurtin",
 title =	 "Thermomechanics of evolving phase boundaries in the plane",
 publisher =	 "Oxford Science Publication",
 year =	 "1993",
 OPTcrossref =	 "",
 OPTkey =	 "",
 OPTeditor =	 "",
 OPTvolume =	 "",
 OPTnumber =	 "",
 OPTseries =	 "",
 address =	 "New York",
 OPTedition =	 "",
 OPTmonth =	 "",
 OPTnote =	 "",
 OPTannote =	 ""
}

@MastersThesis{haettenschweiler,
 author = 	 {J. H{\"a}ttenschweiler},
 title = 	 {Mean Curvature Flow of Networks with Triple
Junctions in the Plane},
 school = 	 {ETH Z\"urich},
 year = 	 {2007},
 OPTkey = 	 {},
 OPTtype = 	 {},
 OPTaddress = 	 {},
 OPTmonth = 	 {},
 OPTnote = 	 {},
 OPTannote = 	 {}
}

@Article{hamilton2,
 author = 	 {R. S. Hamilton},
 title = 	 {Four--manifolds with positive curvature operator},
 journal = 	 {J. Differential Geom.},
 year = 	 {1986},
 OPTkey = 	 {},
 volume = 	 {24},
 number = 	 {2},
 pages = 	 {153-179},
 OPTmonth = 	 {},
 OPTnote = 	 {},
 OPTannote = 	 {}
}

@incollection{hamilton3,
 AUTHOR = {R. S. Hamilton},
 TITLE = {Isoperimetric estimates for the curve shrinking flow in the
 plane},
 BOOKTITLE = {Modern methods in complex analysis (Princeton, NJ, 1992)},
 PAGES = {201-222},
 PUBLISHER = {Princeton University Press},
 ADDRESS = {NJ},
 YEAR = {1995},
}

@article{hamilton4,
 AUTHOR = {R. S. Hamilton},
 TITLE = {The {H}arnack estimate for the mean curvature flow},
 JOURNAL = {J. Differential Geom.},
 VOLUME = {41},
 YEAR = {1995},
 NUMBER = {1},
 PAGES = {215-226},
}

@book{hermul,
 TITLE = {Fundamental contributions to the continuum theory of evolving
 phase interfaces in solids},
 EDITOR = {Ball, J. M. and Kinderlehrer, D. and Podio-Guidugli, P. and
 Slemrod, M.},
 PUBLISHER = {Springer--Verlag},
 ADDRESS = {Berlin},
 YEAR = {1999},
}

@Article{huisk1,
 author = 	 "G. Huisken",
 title = 	 "Flow by mean curvature of convex surfaces into spheres",
 OPTcrossref = "",
 OPTkey = 	 "",
 journal = 	 "J. Differential Geom.",
 year = 	 "1984",
 volume = 	 "20",
 OPTnumber = 	 "",
 pages = 	 "237-266",
 OPTmonth = 	 "",
 OPTnote = 	 "",
 OPTannote = 	 ""
}

@Article{huisk2,
 author = 	 "G. Huisken",
 title = 	 "A distance comparison principle for evolving curves",
 OPTcrossref = "",
 OPTkey = 	 "",
 journal = 	 "Asian J. Math.",
 year = 	 "1998",
 volume = 	 "2",
 OPTnumber = 	 "1",
 pages = 	 "127-133",
 OPTmonth = 	 "",
 OPTnote = 	 "",
 OPTannote = 	 ""
}

@article{huisk3,
 AUTHOR = {G. Huisken},
 TITLE = {Asymptotic behavior for singularities of the mean curvature
 flow},
 JOURNAL = {J. Differential Geom.},
 VOLUME = {31},
 YEAR = {1990},
 PAGES = {285-299}
}

@Book{ilman1,
 author = 	 "T. Ilmanen",
 title = 	 "Elliptic regularization and partial regularity for
		 motion by mean curvature ",
 publisher = 	 "Amer. Math. Soc.",
 series = 	 "Mem. Amer. Math. Soc.",
 VOLUME = {108},
 YEAR = {1994},
 NUMBER = {520},
}

@Unpublished{ilman3,
 author = 	 {T. Ilmanen},
 title = 	 {Singularities of Mean Curvature Flow of Surfaces},
 year = 	 {1995},
 OPTkey = 	 {},
 OPTvolume = 	 {},
 OPTnumber = 	 {},
 OPTpages = 	 {},
 OPTmonth = 	 {},
 note = 	 {http:/$\!\!$/$\!$www.math.ethz.ch/{$\sim$}ilmanen/papers/sing.ps},
 OPTannote = 	 {}
}

@Article{kaston,
 author = {K. Kasai and Y. Tonegawa},
 TITLE = {A general regularity theory for weak mean curvature flow},
 JOURNAL = {Calc. Var. Partial Differential Equations},
 VOLUME = {50},
 YEAR = {2014},
 NUMBER = {1--2},
 PAGES = {1-68},
}

@article{kimton,
 AUTHOR = {L. Kim and Y. Tonegawa},
 TITLE = {On the mean curvature flow of grain boundaries},
 JOURNAL = {Annales de l'Institute Fourier (Grenoble)},
 VOLUME = {67},
 YEAR = {2017},
 NUMBER = {1},
 PAGES = {43-142},
}

@article{kinderliu,
 AUTHOR = {D. Kinderlehrer and C. Liu},
 TITLE = {Evolution of grain boundaries},
 JOURNAL = {Math. Models Methods Appl. Sci.},
 VOLUME = {11},
 YEAR = {2001},
 NUMBER = {4},
 PAGES = {713-729},
}

@book{krylov1,
 AUTHOR = {N. V. Krylov},
 TITLE = {Lectures on elliptic and parabolic equations in {H}\"older
 spaces},
 SERIES = {Graduate Studies in Mathematics},
 VOLUME = {12},
 PUBLISHER = {Amer. Math. Soc.},
 ADDRESS = {Providence, RI},
 YEAR = {1996},
}

@Article{langer2,
 author = 	 {J. Langer},
 title = 	 {A Compactness Theorem for Surfaces with
{$L_p$}--Bounded Second Fundamental Form},
 journal = 	 {Math. Ann.},
 year = 	 {1985},
 OPTkey = 	 {},
 volume = 	 {270},
 OPTnumber = 	 {},
 pages = 	 {223-234},
 OPTmonth = 	 {},
 OPTnote = 	 {},
 OPTannote = 	 {}
}

@book{lasolura,
 AUTHOR = {O.~A. Ladyzhenskaya and V.~A. Solonnikov and N.~N. Ural'tseva},
 TITLE = { Linear and Quasilinear Equations of Parabolic Type},
 PUBLISHER = {Amer. Math. Soc.},
 ADDRESS = {Providence, RI},
 YEAR = {1975},
}

@article{lauxotto,
 AUTHOR = {T. Laux and F. Otto},
 TITLE = {Convergence of the thresholding scheme for {multi--phase}
 {mean--curvature} flow},
 JOURNAL = {Calc. Var. Partial Differential Equations},
 VOLUME = {55},
 YEAR = {2016},
 NUMBER = {5},
 PAGES = {Art. 129, 74},
}

@Article{lauxotto2,
 AUTHOR = {T. Laux and F. Otto},
 TITLE = {Brakke's inequality for the thresholding scheme},
 JOURNAL = {Calc. Var. Partial Differential Equations},
 VOLUME = {59},
 YEAR = {2020},
 NUMBER = {1},
 PAGES = {Paper No. 39, 26},
}

@article{luckstur,
 AUTHOR = {S. Luckhaus and T. Sturzenhecker},
 TITLE = {Implicit time discretization for the mean curvature flow
 equation},
 JOURNAL = {Calc. Var. Partial Differential Equations},
 VOLUME = {3},
 YEAR = {1995},
 NUMBER = {2},
 PAGES = {253-271},
}

@book{lunardi1,
 AUTHOR = {A. Lunardi},
 TITLE = {Analytic semigroups and optimal regularity in parabolic
 problems},
 PUBLISHER = {Birkh\"auser},
 ADDRESS = {Basel},
 YEAR = {1995},
}

@incollection {lunardi2,
 AUTHOR = {A. Lunardi},
 TITLE = {Nonlinear parabolic equations and systems},
 BOOKTITLE = {Evolutionary equations. {V}ol. {I}},
 SERIES = {Handb. Differ. Equ.},
 PAGES = {385-436},
 PUBLISHER = {North{--}Holland},
 ADDRESS = {Amsterdam},
 YEAR = {2004},
}

@article{lusiw,
 AUTHOR = {A. Lunardi and E. Sinestrari and W. von~Wahl},
 TITLE = {A semigroup approach to the time dependent parabolic
 initial--boundary value problem},
 JOURNAL = {Differential Integral Equations},
 VOLUME = {5},
 YEAR = {1992},
 NUMBER = {6},
 PAGES = {1275-1306},
}

@incollection {manmag,
 AUTHOR = {A. Magni and C. Mantegazza},
 TITLE = {Some remarks on {H}uisken's monotonicity formula for mean
 curvature flow},
 BOOKTITLE = {Singularities in nonlinear evolution phenomena and
 applications},
 SERIES = {CRM Series},
 VOLUME = {9},
 PAGES = {157-169},
 PUBLISHER = {Ed. Norm., Pisa},
 YEAR = {2009},
}

@article{mannovplu,
 AUTHOR = 	 {C. Mantegazza and M. Novaga and A. Pluda},
 TITLE= 	 {Motion by curvature of networks with two triple junctions},
 JOURNAL = {Geom. Flows},
 VOLUME = {2},
 YEAR = {2017},
 PAGES = {18-48},
}

@article{mannovplu2,
 AUTHOR = 	 {C. Mantegazza and M. Novaga and A. Pluda},
 TITLE = {Type{--}0 singularities in the network flow {--} Evolution of trees},
 JOURNAL = {J. Reine Angew. Math.},
 VOLUME = {792},
 YEAR = {2022},
 PAGES = {189-221},
}

@incollection {mannovplunotes,
 author = 	 {C. Mantegazza and M. Novaga and A. Pluda},
 TITLE = {Lectures on curvature flow of networks},
 BOOKTITLE = {Contemporary research in elliptic {PDE}s and related topics},
 SERIES = {Springer INdAM Ser.},
 VOLUME = {33},
 PAGES = {369-417},
 PUBLISHER = {Springer, Cham},
 YEAR = {2019},
}

@article{mannovtor,
 AUTHOR = {C. Mantegazza and M. Novaga and V.~M. Tortorelli},
 TITLE = {Motion by curvature of planar networks},
 JOURNAL = {Ann. Sc. Norm. Sup. Pisa},
 VOLUME = {3 (5)},
 YEAR = {2004},
 PAGES = {235-324},
}

@incollection {mannovtor2,
 AUTHOR = {C. Mantegazza},
 TITLE = {Evolution by curvature of networks of curves in the plane},
 BOOKTITLE = {Variational problems in {R}iemannian geometry},
 SERIES = {Progr. Nonlinear Differential Equations Appl.},
 VOLUME = {59},
 PAGES = {95-109},
 OPTNOTE = {Joint project with Matteo Novaga and Vincenzo Maria
 Tortorelli},
 PUBLISHER = {Birkh\"auser},
 ADDRESS = {Basel},
 YEAR = {2004},
}

@incollection {mazsae,
 AUTHOR = {R. Mazzeo and M. S{\'a}ez},
 TITLE = {Self--similar expanding solutions for the planar network flow},
 BOOKTITLE = {Analytic aspects of problems in {R}iemannian geometry:
 elliptic {PDE}s, solitons and computer imaging},
 SERIES = {S\'emin. Congr.},
 VOLUME = {22},
 PAGES = {159-173},
 PUBLISHER = {Soc. Math. France, Paris},
 YEAR = {2011},
}

@article{neves1,
 AUTHOR = {A. Neves},
 TITLE = {Singularities of {L}agrangian mean curvature flow: zero--{M}aslov class case},
 JOURNAL = {Invent. Math.},
 VOLUME = {168},
 YEAR = {2007},
 NUMBER = {3},
 PAGES = {449-484},
}

@article{neves2,
 AUTHOR = {A. Neves},
 TITLE = {Finite time singularities for {L}agrangian mean curvature flow},
 JOURNAL = {Ann. of Math. (2)},
 VOLUME = {177},
 YEAR = {2013},
 NUMBER = {3},
 PAGES = {1029-1076},
}

@article{nevestian,
 AUTHOR = {A. Neves and G. Tian},
 TITLE = {Translating solutions to {L}agrangian mean curvature flow},
 JOURNAL = {Trans. Amer. Math. Soc},
 VOLUME = {365},
 YEAR = {2013},
 NUMBER = {11},
 PAGES = {5655-5680},
}

@Article{nirenberg1,
 author = 	 {L. Nirenberg},
 title = 	 {On elliptic partial differential equations},
 journal = 	 {Ann. Sc. Norm. Sup. Pisa},
 year = 	 {1959},
 OPTkey = 	 {},
 volume = 	 {13},
 OPTnumber = 	 {},
 pages = 	 {116-162},
 OPTmonth = 	 {},
 OPTnote = 	 {},
 OPTannote = 	 {}
}

@article{pluda,
 AUTHOR = {A. Pluda},
 TITLE = {Evolution of spoon--shaped networks},
 JOURNAL = {Network and Heterogeneus Media},
 VOLUME = {11},
 YEAR = {2016},
 NUMBER = {3},
 PAGES = {509--526},
}

@book{prowein,
 AUTHOR = {M. H. Protter and H. F. Weinberger},
 TITLE = {Maximum principles in differential equations},
 PUBLISHER = {Springer--Verlag},
 ADDRESS = {New York},
 YEAR = {1984},
}

@article{saez1,
 AUTHOR = {M. {S{\'a}ez Trumper}},
 TITLE = {Uniqueness of self--similar solutions to the network flow in a
 given topological class},
 JOURNAL = {Comm. Partial Differential Equations},
 VOLUME = {36},
 YEAR = {2011},
 NUMBER = {2},
 PAGES = {185-204},
}

@article{saez2,
 AUTHOR = {M. {S{\'a}ez Trumper}},
 TITLE = {Relaxation of the flow of triods by curve shortening flow via
 the vector--valued parabolic {A}llen--{C}ahn equation},
 JOURNAL = {J. Reine Angew. Math.},
 VOLUME = {634},
 YEAR = {2009},
 PAGES = {143-168},
}

@article{schn-schu,
 AUTHOR = {O. C. Schn\"urer and F. Schulze},
 TITLE = {Self--similarly expanding networks to curve shortening flow},
 JOURNAL = {Ann. Sc. Norm. Super. Pisa},
 VOLUME = {6},
 YEAR = {2007},
 NUMBER = {4}, 
 PAGES = {511-528},
}

@article{schnurerlens,
 AUTHOR = {O. C. Schn{\"u}rer and A. Azouani and M. Georgi and J. Hell and J. Nihar and A. Koeller and T. Marxen and S. Ritthaler and M. S{\'a}ez and F. Schulze and B. Smith},
 TITLE = {Evolution of convex lens--shaped networks under the curve
 shortening flow},
 JOURNAL = {Trans. Amer. Math. Soc.},
 VOLUME = {363},
 YEAR = {2011},
 NUMBER = {5},
 PAGES = {2265-2294},
}

@Article{schulzewhite,
author = {F. Schulze and B. White},
TITLE = {A local regularity theorem for mean curvature flow with triple edges},
 JOURNAL = {J. Reine Angew. Math.},
 VOLUME = {758},
 YEAR = {2020},
 PAGES = {281-305},
}

@Book{simon,
 author = 	 "L. Simon",
 title = 	 "Lectures on geometric measure theory",
 publisher = 	 "Australian National University",
 year = 	 1983,
 volume =	 3,
 series =	 "Proc. Center Math. Anal.",
 address = "Canberra"
}

@book{solonnikov1,
 author = {V. A. Solonnikov},
 TITLE = {Boundary value problems of mathematical physics. {VIII}},
 BOOKTITLE = {Proceedings of the Steklov Institute of Mathematics, No. 125
 (1973)},
 PUBLISHER = {Amer. Math. Soc.},
 ADDRESS = {Providence, R.I.},
 YEAR = {1975},
}

@Article{soner1,
 author = 	 "H. M. Soner",
 title = 	 "Motion of a set by the curvature of its boundary",
 OPTcrossref = "",
 OPTkey = 	 "",
 journal = 	 "J. Differential Equations",
 year = 	 "1993",
 volume = 	 "101",
 number = 	 "2",
 pages = 	 "313-372",
 OPTmonth = 	 "",
 OPTnote = 	 "",
 OPTannote = 	 ""
}

@article{stahl1,
 AUTHOR = {A. Stahl},
 TITLE = {Convergence of solutions to the mean curvature flow with a
 {N}eumann boundary condition},
 JOURNAL = {Calc. Var. Partial Differential Equations},
 VOLUME = {4},
 YEAR = {1996},
 NUMBER = {5},
 PAGES = {421-441},
}

@article{stahl2,
 AUTHOR = {A. Stahl},
 TITLE = {Regularity estimates for solutions to the mean curvature flow
 with a {N}eumann boundary condition},
 JOURNAL = {Calc. Var. Partial Differential Equations},
 VOLUME = {4},
 YEAR = {1996},
 NUMBER = {4},
 PAGES = {385-407},
}

@Article{stone1,
 author = 	 {A. Stone},
 title = 	 {A density function and the structure of
 singularities of the mean curvature flow},
 journal = 	 {Calc. Var. Partial Differential Equations},
 year = 	 {1994},
 OPTkey = 	 {},
 volume = 	 {2},
 OPTnumber = 	 {},
 pages = 	 {443-480},
 OPTmonth = 	 {},
 OPTnote = 	 {}, 
 OPTannote = 	 {}
}

@Article{tonwic,
 author = {Y. Tonegawa and N. Wickramasekera},
 TITLE = {The blow up method for {B}rakke flows: networks near triple junctions},
 JOURNAL = {Arch. Rat. Mech. Anal.},
 VOLUME = {221},
 YEAR = {2016},
 NUMBER = {3},
 PAGES = {1161-1222},
}

@article {vB,
 AUTHOR = {J. {von Below}},
 TITLE = {Classical solvability of linear parabolic equations on
 networks},
 JOURNAL = {J. Differential Equations},
 VOLUME = {72},
 YEAR = {1988},
 NUMBER = {2},
 PAGES = {316-337},
}

@InCollection{vn,
 author = 	 "J. {von Neumann}",
 title = 	 "{Discussion and remarks concerning the paper of C.~S.~Smith ``Grain shapes and other metallurgical applications of topology''}",
 OPTcrossref = "",
 OPTkey = 	 "",
 booktitle = "Metal Interfaces",
 publisher = "American Society for Metals",
 year = 	 "1952",
 OPTeditor = 	 "",
 OPTvolume = 	 "",
 OPTnumber = 	 "",
 OPTseries = 	 "",
 OPTtype = 	 "",
 OPTchapter = 	 "",
 OPTpages = 	 "108-110",
 OPTaddress = 	 "",
 OPTedition = 	 "",
 OPTmonth = 	 "",
 OPTnote = 	 "",
 OPTannote = 	 ""
}

@article{white1,
 AUTHOR = {B. White},
 TITLE = {A local regularity theorem for mean curvature flow},
 JOURNAL = {Ann. of Math. (2)},
 VOLUME = {161},
 YEAR = {2005},
 NUMBER = {3},
 PAGES = {1487-1519},
}

@article{white5,
 AUTHOR = {B. White},
 TITLE = {Stratification of minimal surfaces, mean curvature flows, and
 harmonic maps},
 JOURNAL = {J. Reine Angew. Math.},
 VOLUME = {488},
 YEAR = {1997},
 PAGES = {1-35},
}

\end{document}